\documentclass{book}
\usepackage[lang = british]{ems-book} 
\usepackage{mathtools,braket,amsfonts,tikz-cd,transparent}
\usepackage{enumitem}
\usepackage[bb=libus]{mathalpha}
\makeindex

\newcommand\gobbleone[1]{}
\newcommand*{\seeonly}[2]{\ \emph{\seename} #1}
\definecolor{grey}{rgb}{0.5,0.5,0.5}
\definecolor{darkred}{rgb}{0.5,0,0}
\definecolor{darkpurple}{rgb}{0.5,0,.5}
\definecolor{darkpink}{rgb}{0,.5,0.5}
\definecolor{darkgreen}{rgb}{0,0.5,0}
\definecolor{darkblue}{rgb}{0,0,0.5}
\hypersetup{ colorlinks,
linkcolor=darkblue,
filecolor=darkgreen,
urlcolor=darkred,
citecolor=darkblue }

\renewcommand\theenumi{\alph{enumi}}
\renewcommand\theenumii{\roman{enumii}}
\makeatletter
\newtheorem*{rep@theorem}{\rep@title}
\newcommand{\newreptheorem}[2]{%
\newenvironment{rep#1}[1]{%
 \def\rep@title{#2 \ref{##1}}%
 \begin{rep@theorem}}%
 {\end{rep@theorem}}}
\makeatother

\newenvironment{subproof}[1][\proofname]{%
  \begin{proof}[#1]%
}{%
  \end{proof}%
}

\newtheorem{theorem}{Theorem}[chapter]

\newtheorem{corollary}[theorem]{Corollary}
\newreptheorem{corollary}{Corollary}

\newtheorem{proposition}[theorem]{Proposition}
\newreptheorem{proposition}{Proposition}
\newtheorem{lemma}[theorem]{Lemma}

\theoremstyle{definition}
\newtheorem{definition}[theorem]{Definition}
\theoremstyle{remark}
\newtheorem{remark}[theorem]{Remark}
\newtheorem{example}[theorem]{Example}
\newtheorem{notation}[theorem]{Notation}
\newtheorem{claim}[theorem]{Claim}
\newtheorem*{claim*}{Claim}

\numberwithin{section}{chapter}
\numberwithin{equation}{chapter}

\newcounter{mparcnt}

\newcommand\M{\mathcal{M}}
\newcommand\D{\mathbb{D}}

\newcommand\mL{\mathcal{L}}

\renewcommand\S{\mathcal{S}}

\newcommand{\W}{\mathcal{W}}

\newcommand{\T}{\mathcal{T}}
\newcommand{\J}{\mathcal{J}}

\newcommand{\U}{\mathcal{U}}
\newcommand{\F}{\mathcal{F}}
\newcommand{\N}{\mathbb{N}}
\newcommand{\R}{\mathbb{R}}
\renewcommand{\H}{\mathbb{H}}

\newcommand{\C}{\mathbb{C}}
\newcommand{\CC}{C\kern-1.3ex|}

\newcommand{\bB}{\mathcal{B}}

\newcommand{\Z}{\mathbb{Z}}
\newcommand{\Q}{\mathbb{Q}}

\newcommand{\ddt}{\frac{d}{dt}}

\newcommand{\dds}{\frac{d}{ds}}

\renewcommand{\P}{\mathbb{P}}
\newcommand{\e}{{\mathbb{e}}}
\newcommand{\PP}{\mathcal{P}}
\newcommand{\QQ}{\mathcal{Q}}
\newcommand{\pp}{\mathfrak{p}}

\newcommand{\zz}{\mathfrak{z}}
\newcommand{\Pe}{\mathfrak{p}}
\newcommand{\PPe}{\mathscr{P}}

\newcommand\lie[1]{\mathfrak{#1}}
\renewcommand{\k}{\lie{k}}
\newcommand{\li}{\lie{i}}

\newcommand{\g}{\lie{g}}

\newcommand{\m}{\lie{m}}

\renewcommand{\t}{\lie{t}}

\renewcommand{\u}{\lie{u}}

\newcommand{\on}{\operatorname}

\newcommand{\Crit}{\on{Crit}} 
\newcommand{\Fl}{\on{Fl}} 
\newcommand{\Diff}{\on{Diff}}
\newcommand{\ann}{\on{ann}}
\newcommand{\aug}{{\on{aug}} }
 
\newcommand{\univ}{\on{univ}}

\newcommand{\GC}{\on{GC}}
\newcommand{\ainfty}{{$A_\infty$\ }}
 \newcommand{\pre}{{\on{pre}}}
\newcommand{\Free}{{\on{Free}}}
\newcommand{\fr}{{\on{fr}}}
\newcommand{\red}{{\on{red}}}
 \newcommand{\adj}{{\on{adj}}}
\newcommand{\dual}{\vee}
\newcommand\Mod[1]{\lVert #1 \rVert}

\newcommand\abs[1]{\lvert #1 \rvert}

\newcommand{\ver}{{\on{vert}}}

\newcommand{\Edge}{\on{Edge}}

\newcommand{\loc}{{\on{loc}}}
\newcommand{\Area}{{\on{Area}}}

\newcommand{\Ver}{\on{Vert}}

\newcommand{\st}{\on{st}}
\newcommand{\nl}{\delta}
\newcommand{\internal}{{\on{int}}}

\newcommand\B{\mathcal{B}}

\newcommand{\End}{\on{End}}

\newcommand{\Aut}{ \on{Aut} }

\newcommand{\Ad}{ \on{Ad} } 
\newcommand{\ad}{ \on{ad} } 
 
\newcommand{\Hol}{ \on{Hol} }

\newcommand{\Hom}{ \on{Hom}}

\renewcommand{\ker}{ \on{ker}}
\newcommand{\coker}{ \on{coker}}
\newcommand{\im}{ \on{im}}
\newcommand{\ind}{ \on{ind}}

\newcommand{\Vol}{  \on{Vol}}

\newcommand{\mult}{  \on{mult}}

\newcommand{\supp}{\on{supp}}

\newcommand{\glue}{{\on{glue}}}

\newcommand{\neck}{{\on{neck}}}
\newcommand{\prim}{{\on{prim}}}
\newcommand{\Facets}{{\on{Facets}}}
\newcommand{\cu}{u^0}

\newcommand{\dvol}{  \on{dvol}}
\newcommand{\Neck}{  \on{Neck}}
\newcommand{\codim}{\on{codim}}

\newcommand{\ssm}{-}
\newcommand\qu{/\kern-.7ex/} 
\newcommand\lqu{\backslash \kern-.7ex \backslash} 
\newcommand\bs{\backslash}

\newcommand\lldots{\hbox to 1em{.\hss.\hss.}}
\newcommand{\triv}{{\on{triv}}}
\newcommand{\lag}{L}

\newcommand{\refcomment}[1]   {{}}
\newcommand{\labell}\label

\newcommand{\hra}{\hookrightarrow}

\renewcommand{\d}{{\on{d}}}

\newcommand{\ol}{\overline}
\newcommand{\olp}{\ol{\partial}}

\newcommand\Phinv{\Phi^{-1}}

\newcommand\pinv{\pi^{-1}}

\newcommand\lam{\lambda}
\newcommand\Lam{\Lambda}

\newcommand\tiu{\tilde u}

\newcommand\sig{\sigma}
\newcommand\eps{\epsilon}
\newcommand\veps{\varepsilon}
\newcommand\Om{\Omega}
\newcommand\om{\omega}

\newcommand{\lan}{\langle}
\newcommand{\ran}{\rangle}
\newcommand{\hh}{{\tfrac{1}{2}}}

\newcommand{\sqq}{{\tfrac{1}{4}}}

\newcommand{\tM}{\M}
\newcommand{\tP}{{\tilde P}}
\newcommand{\tQ}{{\tilde Q}}
\newcommand{\tR}{{\tilde R}}

\newcommand\cE{\mathcal{E}}
\newcommand\E{\mathcal{E}}
\newcommand\cF{\mathcal{F}}

\newcommand\cG{\mathcal{G}}

\newcommand\cP{\mathcal{P}}
\newcommand\cT{\mathcal{T}}
\newcommand\pT{\mathbb{T}}
\newcommand\cW{\mathcal{W}}
\newcommand\TT{\mathcal{T}}

\newcommand\cI{\mathcal{I}}

\newcommand\cQ{\mathcal{Q}}

\newcommand\mE{\mathcal{E}}

\newcommand\curv{\on{curv}}

\newcommand\Gr{\on{Gr}}
\newcommand\Map{\on{Map}}

\newcommand\ev{\on{ev}}
\newcommand\Vect{\on{Vect}}
\newcommand\ul{\underline}
\newcommand\mO{\mathcal{O}}
\newcommand\G{\mathcal{G}}

\newcommand\grad{\on{grad}}
\newcommand\reg{{\on{reg}}}
\newcommand\bran[1]{ \lan {#1} \ran}

 \newcommand\cwl[1]{{}}

 \newcommand{\dge}{\rotatebox[origin=c]{45}{$\ge$}}
 \newcommand{\uge}{\rotatebox[origin=c]{315}{$\ge$}}
 \newcommand{\deq}{\rotatebox[origin=c]{45}{$=$}}
 \newcommand{\ueq}{\rotatebox[origin=c]{315}{$=$}}
 \newcommand{\updots}{\hbox to1.65em{\rotatebox[origin=c]{45}{$\cdots$}}}
 \newcommand{\dndots}{\hbox to1.65em{\rotatebox[origin=c]{315}{$\cdots$}}}

\newcommand{\Phid}[2]{\hbox to1.65em{$ \Phi_{#1, #2}$}}

\newcommand\Cone{\on{Cone}}
\newcommand\NCone{\on{NCone}}

\newcommand\Id{\on{Id}}

\newcommand\XX{\mathfrak{X}} 
\newcommand\XC{\XX} 
\newcommand\XB{X} 
\newcommand\oX{X^{\whitesq}} 
\newcommand\oZ{Z} 
\newcommand\wX{X^{\blacksq}} 
\newcommand\wP{P^{\blacksq}} 
\newcommand\wQ{Q^{\blacksq}} 
\newcommand\wR{R^{\blacksq}} 

\newcommand\bD{\mathfrak{D}} 
\newcommand\DD{\mathfrak{D}} 
\newcommand\cD{D} 

\newcommand\JJ{\mathfrak{J}}

\newcommand\BB{\mathbb{B}}
\newcommand\fj{\Delta j}

\newcommand\crit{{\on{crit}}}

\newcommand\sx{*\kern-.5ex_X}

\newcommand\white{{\includegraphics[width=.05in]{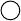}}}
\newcommand\black{{\includegraphics[width=.05in]{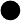}}}
\newcommand\whitet{{\includegraphics[width=.07in]{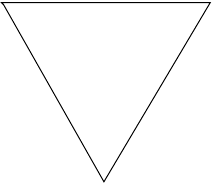}}}
\newcommand\greyt{\includegraphics[width=.07in]{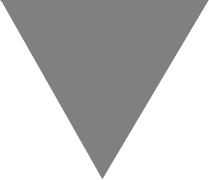}}
\newcommand\blackt{{\includegraphics[width=.07in]{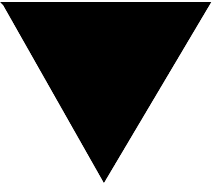}}}
\newcommand\whitesq{{\includegraphics[width=.05in]{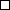}}}
\newcommand\blacksq{{\includegraphics[width=.05in]{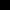}}}
\newcommand\ldiag{{\includegraphics[width=.1in]{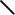}}}
\newcommand\rdiag{{\includegraphics[width=.1in]{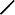}}}

\newcommand{\tJJ}{\tilde \JJ}
\newcommand{\tX}{\tilde X}

\newcommand{\tC}{\tilde C}
\newcommand{\tS}{\tilde S}

\newcommand{\bT}{\mathbb{T}}

\newcommand{\ssum}{{\textstyle \sum}}

\newcommand{\Bl}{ \on{Bl}}

\newcommand\mN{\mathcal{N}}

\newcommand{\tensor}{\otimes}
\renewcommand{\rm}{\rmfamily}

\renewcommand{\em}{\textit}

\newcommand{\cyl}{{\on{cyl}}}
\newcommand{\Cyl}{{\on{Cyl}}}

\newcommand{\RR}{\mathcal{R}}
\newcommand{\Hof}{{\on{Hof}}}
\newcommand{\co}{c} 
\newcommand\br{{\on{brok}}}
\newcommand\quilt{{q}}
\newcommand\qq{{qq}}
\newcommand{\tr}{{\on{tr}}}
\newcommand{\trop}{{\on{trop}}}
\newcommand{\Morse}{{\on{Morse}}}
\newcommand\tGam{{\tilde \Gamma}}

\newcommand\bGam{\Gamma}
\newcommand\Gam{\Gamma}

\newcommand\fw{w}

\newcommand\delbar{\ol{\partial}}
\renewcommand\em{\textit}
\newcommand\fillblack{\blacksq}

\def\mathunderaccent#1{\let\theaccent#1\mathpalette\putaccentunder}
\def\putaccentunder#1#2{\oalign{$#1#2$\crcr\hidewidth \vbox
to.2ex{\hbox{$#1\theaccent{}$}\vss}\hidewidth}}

\counterwithin{figure}{chapter}

\begin{document}

\frontmatter

\title{Tropical Fukaya algebras}


\author{Sushmita Venugopalan and Chris T. Woodward}

\maketitle

\tableofcontents
\mainmatter

\thanks{We thank Jonny Evans, Mohammad F. Tehrani, and Nick Sheridan for discussions and helpful emails. We gratefully thank the referees for their careful reading and detailed suggestions,
which dramatically improved the exposition. We also thank Paul Seidel and the rest of
the EMS Monographs in Mathematics editorial board, whose advice at various stages
was extremely useful.
 S.V. acknowledges support from the Matric grant MTR/2017/000920 administered by the Science and Engineering Research Board.
  C.W.  was partially supported by NSF grant DMS 2105417.  Any   opinions, findings, and conclusions or recommendations expressed in  this material are those of the author(s) and do not necessarily  reflect the views of the National Science Foundation.}

\vskip .1in

\chapter{Statement of results}\label{chap:state}

In this book, we study the behavior of holomorphic disks under a
multiple symplectic cut.  In particular, we study the behavior of the
Fukaya algebra associated to a Lagrangian submanifold of a symplectic
manifold.  The Fukaya algebra is an \ainfty algebra defined using
counts of holomorphic disks whose boundaries lie on the Lagrangian
submanifold.  The main result of the book is a \ainfty homotopy
equivalence between two versions of the Fukaya algebra: the standard
version defined on the symplectic manifold and the \em{broken} version
obtained by cutting the manifold.  The Lagrangian submanifold is
assumed to be disjoint from the cuts.

We use the homotopy equivalence to compute the disk potentials of
Lagrangians in various examples.  An important application is the
weak unobstructedness of almost toric Lagrangians, 
which we carry out in a
sequel \cite{vw:split} which was split off from the original
manuscript.  In this introductory chapter, we give a low-tech tour of
the book, interwoven with motivations, context, and also limitations,
of our results.

\section{Multiple cuts}

The \em{cut} \index{Cut} operation introduced by Lerman \cite{le:sy2}
cuts a symplectic manifold along a regular level set of a moment map
for a circle action, and then quotients the boundary by the circle
action to produce a smooth symplectic manifold. The inverse operation
is called a \em{symplectic sum}.  The \em{multiple cut}
\index{Cut!Multiple cut} is a generalization where the symplectic
manifold is cut along intersecting hypersurfaces. The set-up for a
multiple cut is that of a tropical Hamiltonian action, which we now define.

\begin{definition}\label{def:polydec}
  \index{Polyhedral decomposition $\PP$} {\rm(Polyhedral
    decomposition)} Let $n >0$ be an integer and $T \simeq (S^1)^n$ a
  torus with Lie algebra $\t \cong \R^n$.  A \em{simplicial
    polyhedral decomposition} of $\t^\dual$ is a collection
  \[ \PP = \{ P \subset \t^\dual \} \]
  of simple polytopes\footnote{By a polytope we mean a finite
    intersection of half-planes, as in Definition \ref{def:delz}.  As
    such, our polytopes are closed but not necessarily compact.} such
  that
  \begin{enumerate}
  \item {\rm (Covering property)} the interiors $P^\circ$ of the
    polytopes $P \in \PP$ cover $\t^\dual$; that is,
    \[ \t^\dual= \cup_{P \in \PP} P^\circ; \]
  \item {\rm (Face property)} and for any $\sig_1, \dots, \sig_k \in \PP$,
    the intersection
    $\sig_1 \cap \dots \cap \sig_k$ is a polytope in $\PP$ that is a
    face of each of the polytopes $\sig_1,\dots,\sig_k$.
  \end{enumerate}
\end{definition}
Any polytope $P$ in a simplicial polyhedral decomposition $\PP$
corresponds to a sub-torus
\begin{equation}
  \label{eq:deftp}
  T_P \subseteq T, \quad \text{defined by} \quad \t_P:=\ann(TP).  
\end{equation}
Thus $P$ and $T_P$ have complementary dimensions, and $T_P=\{\Id\}$
for top-dimensional polytopes $P \in \PP$.
  
\begin{definition} \label{def:tropmanifold}
  {\rm(Tropical Hamiltonian action)} \index{Tropical Hamiltonian action} A \em{tropical
    Hamiltonian action} $(X,\PP,\Phi)$ consists of 
  \begin{enumerate}
  \item \label{part:trop1} a simplicial polyhedral decomposition $\PP$ of $\t^\dual$, and
  \item \label{part:trop2} a compact symplectic manifold $(X,\om_X)$
    with a 
    \index{Tropical moment map} 
    \em{tropical moment map}
    \[\Phi:X \to \t^\dual, \]
    such for any $P \in \PP$ there is a neighborhood $U_P \subset X$
    of $\Phinv(P)$ on which the projection
    \[U_P \to \t^\dual \to \t_P^\dual\]
    is a moment map for a free Hamiltonian action of a torus $T_P$.
  \end{enumerate}
\end{definition}

Given a tropical Hamiltonian action $X$ the output of a multiple cut is a
collection of
\index{Cut space} 
\em{cut spaces}
\begin{equation}
  \label{eq:cutspace}
  X_P:=\Phinv(P^\circ),
\end{equation}
where $P^\circ$ is the complement of the faces of $P$; the cut space $X_P$ compactifies to 
\[\ol X_P=\Phinv(P)/\sim,\]
where $\Phinv(P)$ is a manifold with corners and the equivalence
relation $\sim$ quotients any codimension one boundary
$\Phinv(Q) \subset \Phinv(P)$, $\codim_P(Q)=1$ by the action of
$S^1 \simeq T_Q/T_P$. The spaces $\ol X_P$, $P \in \PP$ are orbifolds,
whose local structure is given by an iterative application of Lerman's
construction \cite{le:sy2}.  In a multiple cut with polyhedral
decomposition $\PP$, for any pair of facets $Q_1$, $Q_2 \in \PP$ of a
polytope $P \in \PP$, $\ol X_{Q_1}$ and $\ol X_{Q_2}$ are embedded as
divisors in the compactified cut space $\ol X_P$, called
\index{Relative divisor} 
\em{relative
  divisors}, that intersect each other normally along
$\ol X_{Q_1 \cap Q_2}$.  See Figure \ref{fig:break}.  The space $X_P$
is the complement of the relative divisors in $\ol X_P$.
\begin{figure}[ht]
  {
\begingroup%
  \makeatletter%
  \providecommand\color[2][]{%
    \errmessage{(Inkscape) Color is used for the text in Inkscape, but the package 'color.sty' is not loaded}%
    \renewcommand\color[2][]{}%
  }%
  \providecommand\transparent[1]{%
    \errmessage{(Inkscape) Transparency is used (non-zero) for the text in Inkscape, but the package 'transparent.sty' is not loaded}%
    \renewcommand\transparent[1]{}%
  }%
  \providecommand\rotatebox[2]{#2}%
  \newcommand*\fsize{\dimexpr\f@size pt\relax}%
  \newcommand*\lineheight[1]{\fontsize{\fsize}{#1\fsize}\selectfont}%
  \ifx\svgwidth\undefined%
    \setlength{\unitlength}{185.03764235bp}%
    \ifx\svgscale\undefined%
      \relax%
    \else%
      \setlength{\unitlength}{\unitlength * \real{\svgscale}}%
    \fi%
  \else%
    \setlength{\unitlength}{\svgwidth}%
  \fi%
  \global\let\svgwidth\undefined%
  \global\let\svgscale\undefined%
  \makeatother%
  \begin{picture}(1,0.11876923)%
    \lineheight{1}%
    \setlength\tabcolsep{0pt}%
    \put(0,0){\includegraphics[width=\unitlength,page=1]{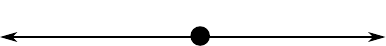}}%
    \put(0.07653387,0.04580504){\color[rgb]{0,0,0}\makebox(0,0)[lt]{\lineheight{1.25}\smash{\begin{tabular}[t]{l}$P_+$\end{tabular}}}}%
    \put(0.83399432,0.05348175){\color[rgb]{0,0,0}\makebox(0,0)[lt]{\lineheight{1.25}\smash{\begin{tabular}[t]{l}$P_-$\end{tabular}}}}%
    \put(0.49155663,0.07770917){\color[rgb]{0,0,0}\makebox(0,0)[lt]{\lineheight{1.25}\smash{\begin{tabular}[t]{l}$P_0$\end{tabular}}}}%
  \end{picture}%
\endgroup%
}
  \caption{A single cut.}
  \label{fig:poly-single}
\end{figure}
\begin{example}
  For a single cut, the polyhedral decomposition $\PP$ partitions
  $\t^\dual \simeq \R$ into two semi-infinite lines intersecting at a
  point as in Figure \ref{fig:poly-single}.

  The cut spaces are
  $X_{P_+}$, $X_{P_-}$ and $X_{P_0}$, of which the first two are
  top-dimensional.  The simplest example of a multiple cut consists of
  two single cuts along hypersurfaces
  \[Z_1:=\Phinv(P_{12} \cup P_{30}), \quad Z_2:=\Phinv(P_{01} \cup
    P_{23}).\]
  %
  %
See Figure \ref{fig:break}. The cut spaces are
\[X_{P_i}, \quad i=1,\dots,4; \quad X_{P_{ij}}, \quad j=(i+1) \mod 4; \quad X_{P_\cap}.\]
The hypersurface $\Phinv(P_{ij})$ is an $S^1$-bundle over $X_{P_{ij}}$ and $\Phinv(P_{\cap})$ is an $(S^1)^2$ bundle over $X_{P_\cap}$. This finishes the example.
\end{example}
\begin{figure}[h]
\centering
    \scalebox{.8}{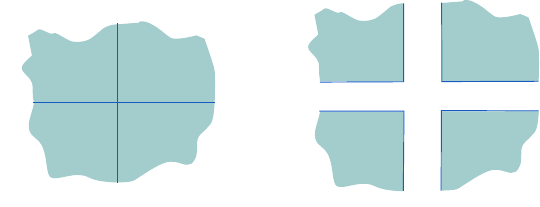} 
  \caption{A multiple cut of $\R^2$.}
  \label{fig:break}
\end{figure}
\begin{example}
  {\rm(On orbifolds)} Although, the definition of a tropical Hamiltonian action
  requires torus actions to be free in neighborhoods of cut loci, cut
  spaces may be orbifolds.  For example, the multiple cut on the right
  side of Figure \ref{fig:orb-cut} (which is a part of the multiple
  cut of a cubic surface in Figure \ref{fig:27polys}) is allowed,
  although two of the points $X_{Q_1}$, $X_{Q_2}$ in the cut space
  $\ol X_{P_2}$ are $A_1$-singularities. However, the single cut on
  the left side of Figure \ref{fig:orb-cut} is not allowed since the
  circle action corresponding to the cut is not free.  Requiring a
  free torus action in the neighborhoods of cut loci is needed to
  ensure the transversality of moduli spaces of broken maps as
  explained
  %
  at the end of Section \ref{sec:bfa-intro}. 
\end{example}

\begin{figure}[h]
  \centering \scalebox{.8}{
\begingroup%
  \makeatletter%
  \providecommand\color[2][]{%
    \errmessage{(Inkscape) Color is used for the text in Inkscape, but the package 'color.sty' is not loaded}%
    \renewcommand\color[2][]{}%
  }%
  \providecommand\transparent[1]{%
    \errmessage{(Inkscape) Transparency is used (non-zero) for the text in Inkscape, but the package 'transparent.sty' is not loaded}%
    \renewcommand\transparent[1]{}%
  }%
  \providecommand\rotatebox[2]{#2}%
  \newcommand*\fsize{\dimexpr\f@size pt\relax}%
  \newcommand*\lineheight[1]{\fontsize{\fsize}{#1\fsize}\selectfont}%
  \ifx\svgwidth\undefined%
    \setlength{\unitlength}{318.12933856bp}%
    \ifx\svgscale\undefined%
      \relax%
    \else%
      \setlength{\unitlength}{\unitlength * \real{\svgscale}}%
    \fi%
  \else%
    \setlength{\unitlength}{\svgwidth}%
  \fi%
  \global\let\svgwidth\undefined%
  \global\let\svgscale\undefined%
  \makeatother%
  \begin{picture}(1,0.41524144)%
    \lineheight{1}%
    \setlength\tabcolsep{0pt}%
    \put(0,0){\includegraphics[width=\unitlength,page=1]{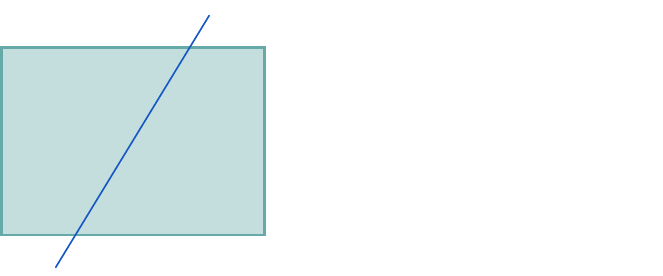}}%
    \put(0.32405375,0.37229678){\makebox(0,0)[lt]{\lineheight{1.25}\smash{\begin{tabular}[t]{l}$y=2x$\end{tabular}}}}%
    \put(0,0){\includegraphics[width=\unitlength,page=2]{orb-cut.pdf}}%
    \put(0.73629152,0.09441867){\color[rgb]{0,0,0}\makebox(0,0)[lt]{\lineheight{1.25}\smash{\begin{tabular}[t]{l}$P_2$\end{tabular}}}}%
    \put(0,0){\includegraphics[width=\unitlength,page=3]{orb-cut.pdf}}%
    \put(0.82350388,0.13644917){\color[rgb]{0,0,0}\makebox(0,0)[lt]{\lineheight{1.25}\smash{\begin{tabular}[t]{l}$Q_2$\end{tabular}}}}%
    \put(0.77471098,0.18553012){\color[rgb]{0,0,0}\makebox(0,0)[lt]{\lineheight{1.25}\smash{\begin{tabular}[t]{l}$Q_1$\end{tabular}}}}%
  \end{picture}%
\endgroup%
}
  \caption{The cut in the left is disallowed, since the torus action is not free; the cut on the right is allowed.}
  \label{fig:orb-cut}
\end{figure}

Our definition of tropical Hamiltonian actions is similar to Gross-Siebert
\cite{gross:book}. A limitation of our set-up is that we require
integral affine singularities to lie away from the cut locus.  We
expect that this requirement can be weakened by replacing a single cut
passing through a singularity by two parallel cuts straddling the
singularity, as explained below in Section \ref{sec:degen}.  On the
other hand, our definitions include manifolds that are not integral
affine, because we require a toric structure only in the neighborhood
of the cut locus, rather than globally.

\section{Neck-stretching}

To build a correspondence between curves in the manifold and curves in
the cut spaces, we construct a sequence of ``large almost complex
structures'', in the language of Kontsevich-Soibelman
\cite{KS:mirror}, which degenerate to almost complex structures in the
cut spaces.  We call this limiting process \em{neck stretching},
because it amounts to modifying the complex structure in the
neighborhoods of the cut loci, which we call \em{necks}.

As an example of neck-stretching, consider a single cut.  Let
$\Phi: X \to \R$ be a moment map for a Hamiltonian circle action, so
that the zero level set
\[Z:=\Phinv(0)\]
is a separating hypersurface.  Choose a tubular neighborhood
\[U_Z \subset X, \quad U_Z \cong Z \times I \]
where $I \subset \R$ is an interval.  For a sequence $\{J^\nu\}_\nu$ of
\em{neck-stretching} almost complex structures on $X$,
\index{Almost complex structure! Neck-stretching}
the fibers of
the projection $U_Z \to Z/S^1$ are $J_\nu$-holomorphic cylinders
$[\frac {-\nu} 2, \frac {\nu} 2] \times S^1$ for any $\nu \in \N$, as
shown in Figure \ref{fig:neck1}.
\begin{figure}[h]
  {\centering\scalebox{.8}{
\begingroup%
  \makeatletter%
  \providecommand\color[2][]{%
    \errmessage{(Inkscape) Color is used for the text in Inkscape, but the package 'color.sty' is not loaded}%
    \renewcommand\color[2][]{}%
  }%
  \providecommand\transparent[1]{%
    \errmessage{(Inkscape) Transparency is used (non-zero) for the text in Inkscape, but the package 'transparent.sty' is not loaded}%
    \renewcommand\transparent[1]{}%
  }%
  \providecommand\rotatebox[2]{#2}%
  \newcommand*\fsize{\dimexpr\f@size pt\relax}%
  \newcommand*\lineheight[1]{\fontsize{\fsize}{#1\fsize}\selectfont}%
  \ifx\svgwidth\undefined%
    \setlength{\unitlength}{339.99604208bp}%
    \ifx\svgscale\undefined%
      \relax%
    \else%
      \setlength{\unitlength}{\unitlength * \real{\svgscale}}%
    \fi%
  \else%
    \setlength{\unitlength}{\svgwidth}%
  \fi%
  \global\let\svgwidth\undefined%
  \global\let\svgscale\undefined%
  \makeatother%
  \begin{picture}(1,0.56513354)%
    \lineheight{1}%
    \setlength\tabcolsep{0pt}%
    \put(0,0){\includegraphics[width=\unitlength,page=1]{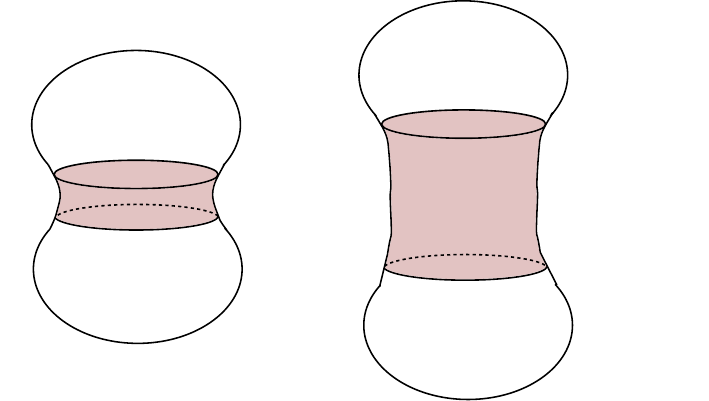}}%
    \put(-0.00255297,0.27604604){\color[rgb]{0,0,0}\makebox(0,0)[lt]{\lineheight{1.25}\smash{\begin{tabular}[t]{l}$(X,\om)$\end{tabular}}}}%
    \put(0.31166642,0.28050258){\color[rgb]{0,0,0}\makebox(0,0)[lt]{\lineheight{1.25}\smash{\begin{tabular}[t]{l}$Z \times [-\eps,\eps]$\end{tabular}}}}%
    \put(0.64811159,0.26490451){\color[rgb]{0,0,0}\makebox(0,0)[lt]{\lineheight{1.25}\smash{\begin{tabular}[t]{l}$J^\nu$\end{tabular}}}}%
    \put(0.77119708,0.27359936){\color[rgb]{0,0,0}\makebox(0,0)[lt]{\lineheight{1.25}\smash{\begin{tabular}[t]{l}$Z \times [-\frac \nu 2, \frac \nu 2]$\end{tabular}}}}%
    \put(0.58946584,0.20532983){\color[rgb]{0,0,0}\rotatebox{90}{\makebox(0,0)[lt]{\lineheight{1.25}\smash{\begin{tabular}[t]{l}Neck region\end{tabular}}}}}%
  \end{picture}%
\endgroup%
}}
  \caption{Neck-stretching for a single cut.}
  \label{fig:neck1}
\end{figure}

To define neck-stretching for a multiple cut, we need an additional
datum of a \em{dual complex}, \index{Dual complex} which encodes the
proportion in which the neck is stretched in different directions.
The dual complex $B^\dual$ consists of a complementary dimensional polytope
denoted $P^\dual$ for every $P \in \PP$.  Additionally, we need a
\em{cutting datum} corresponding to the dual complex, that consists of
a thickening of the cut locus; that is, for each polytope $P \in \PP$,
there is a polytope $\tP \subset \t^\dual$ that has a projection
$\tP \to P$ with fibers isomorphic to the dual complex $P^\dual$.  See
Definition \ref{def:cut-datum} for details and Figure
\ref{fig:broken-a3} for an example.  Curve counts will depend on the
choice of dual complex, but as with other choices (such as tamed
almost complex structures), the curve counts corresponding to any two
dual complexes will produce Fukaya algebras that are \ainfty homotopy
equivalent.

\begin{figure}[h]
  {
\begingroup%
  \makeatletter%
  \providecommand\color[2][]{%
    \errmessage{(Inkscape) Color is used for the text in Inkscape, but the package 'color.sty' is not loaded}%
    \renewcommand\color[2][]{}%
  }%
  \providecommand\transparent[1]{%
    \errmessage{(Inkscape) Transparency is used (non-zero) for the text in Inkscape, but the package 'transparent.sty' is not loaded}%
    \renewcommand\transparent[1]{}%
  }%
  \providecommand\rotatebox[2]{#2}%
  \newcommand*\fsize{\dimexpr\f@size pt\relax}%
  \newcommand*\lineheight[1]{\fontsize{\fsize}{#1\fsize}\selectfont}%
  \ifx\svgwidth\undefined%
    \setlength{\unitlength}{130.83621408bp}%
    \ifx\svgscale\undefined%
      \relax%
    \else%
      \setlength{\unitlength}{\unitlength * \real{\svgscale}}%
    \fi%
  \else%
    \setlength{\unitlength}{\svgwidth}%
  \fi%
  \global\let\svgwidth\undefined%
  \global\let\svgscale\undefined%
  \makeatother%
  \begin{picture}(1,0.12242)%
    \lineheight{1}%
    \setlength\tabcolsep{0pt}%
    \put(0,0){\includegraphics[width=\unitlength,page=1]{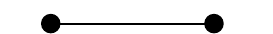}}%
    \put(-0.00317219,0.09316004){\color[rgb]{0,0,0}\makebox(0,0)[lt]{\lineheight{1.25}\smash{\begin{tabular}[t]{l}$P_+^\dual$\end{tabular}}}}%
    \put(0.36789575,0.08548239){\color[rgb]{0,0,0}\makebox(0,0)[lt]{\lineheight{1.25}\smash{\begin{tabular}[t]{l}$P_0^\dual$\end{tabular}}}}%
    \put(0.77119077,0.09338484){\color[rgb]{0,0,0}\makebox(0,0)[lt]{\lineheight{1.25}\smash{\begin{tabular}[t]{l}$P_-^\dual$\end{tabular}}}}%
  \end{picture}%
\endgroup%
}
\caption{The dual complex for a single cut.} 
\label{fig:dualcomplex} 
\end{figure}

\begin{example} For the single cut in Figure \ref{fig:poly-single},
  the dual complex is a line segment and is shown in Figure
  \ref{fig:dualcomplex}.  For the multiple cut consisting of two single
  cuts in Figure \ref{fig:break}, the dual complex is a rectangle
  shown in Figure \ref{fig:sqdual} with side lengths, say $l_1$,
  $l_2$.  In the neck-stretched almost complex manifold $(X,J^\nu)$, a
  neighborhood $U_{P_\cap}$ of $\Phinv(P_{\cap})$ fibers over
  $X_{P_\cap}$.  The fibers of $U_{P_\cap} \to X_{P_\cap}$ are
  biholomorphic to the product of cylinders
  $((\frac {-\nu l_1} 2, \frac {\nu l_1} 2) \times S^1) \times ((\frac
  {-\nu l_2} 2, \frac {\nu l_2} 2) \times S^1)$, as shown in Figure
  \ref{fig:pcoord}. The ratio $\frac {l_1} {l_2}$ is not required to
  be rational.
\end{example}

\begin{figure}[ht]
  \centering \scalebox{.8}{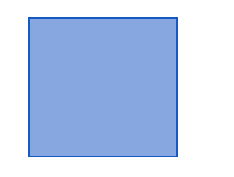}
  \caption{Dual complex $B^\dual$ for the cut in Figure
    \ref{fig:break}.}
  \label{fig:sqdual}
\end{figure}
\begin{figure}[ht]
  \centering \scalebox{.8}{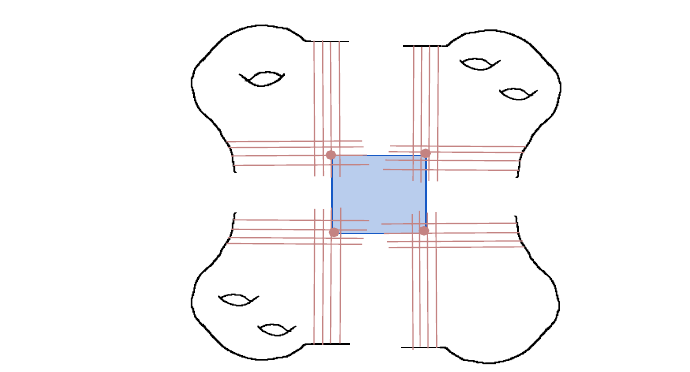}
  \caption{Neck-stretched manifold $(X,J_\nu)$ for the double cut in
    Figure \ref{fig:break} using the dual complex $B^\dual$.}
  \label{fig:pcoord}
\end{figure}

A \em{broken manifold} \index{Broken manifold} corresponding to a
multiple cut $(X,\PP)$ is a disjoint union of thickenings 
of cut spaces
\begin{equation} \label{eq:xx} 
\XX=\bigsqcup_{P \in \PP}\XC_P.\end{equation}
For top-dimensional polytopes $P \in \PP$, $\XX_P:=X_P$ is the cut space
from \eqref{eq:cutspace}, and for
lower dimensional polytopes $P$, $\XC_P$ is a
toric fibration over the 
cut space $X_P$ and is called a \em{neck piece}.\index{Neck piece} 
The manifold $\XX_P$ has a compactification $\ol \XX_P$ obtained by adding relative divisors to $\XX_P$.

For example, in the case of a single cut as in Figure
\ref{fig:poly-single}, the neck piece denoted by $\XC_{P_0}$ arises as
the limit of the neighborhoods of the separating hypersurface and
corresponds to a zero-dimensional polytope $P_0 \in \PP$.  The space
$\XC_{P_0}$ is a $(\R \times S^1)$-bundle over the relative divisor
$X_{P_0}$, and its compactification $X_{P_0}$ is a $\P^1$-bundle over
$X_{P_0}$.  The analog for multiple cuts is the following:  For a polytope $P \in \PP$,
the neck piece $\XC_P$ is a torus bundle
\[T_{P,\C} \to \XC_P \to \XB_P \]
over the cut space $\XB_P$. The fiber  of the bundle is isomorphic to
the complexification $T_{P,\C}$ of the compact torus $T_P$.  This
torus is part of the data of the tropical Hamiltonian action $X$, and was
defined in \eqref{eq:deftp}; see
Figure \ref{fig:introxx}.  The compactification of $\XC_P$ is a
fibration $\ol \XC_P \to \ol X_P$ with toric orbifold fibers.

\begin{figure}[ht]
  \centering \scalebox{.8}{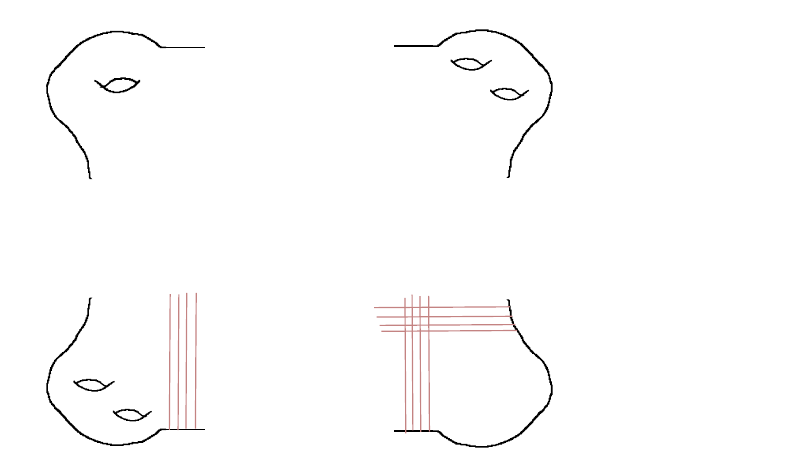}
  \caption{Broken manifold arising from the neck-stretching in Figure
    \ref{fig:pcoord}.}
  \label{fig:introxx}
\end{figure}

We will study the behavior of holomorphic disks bounding a Lagrangian
submanifold under a family of neck-stretching almost complex
submanifolds.  To simplify the situation somewhat, we assume that the
Lagrangian submanifold $L \subset (X,\om)$ is disjoint from the cuts.
The Lagrangian in $X$ therefore descends to a Lagrangian submanifold, also
denoted by $L$, in a cut space $\XB_P$ for a top-dimensional polytope
$P$.

Broken manifolds have orbifold components, and the reader may be
concerned that these present additional technical issues.  However, we
consider maps on punctured curves whose images lie in the
uncompactified spaces $\XX_P$, which are manifolds with cylindrical
ends.  Broken maps have extensions over punctures, with puncture
points mapping to relative divisors, but the compactification is just
a useful technical tool in some places and
no analysis on orbifolds is actually used.

\section{Broken maps}
A broken map is modelled on a graph and consists of a ``map part'' and
a ``tropical part''.  For the map part, 
in the case without boundary 
the domain is the normalization
of a nodal curve $C$ whose irreducible components denoted
\[ C_v, v \in \Ver(\Gamma) \] 
correspond to vertices of a graph $\Gamma$, and whose nodes denoted
$w_e \in C$ correspond to edges $e \in \Edge(\Gamma)$ of $\Gamma$.
The broken map is a collection of holomorphic maps on punctured curves
\[ u_v:C_v^\circ \to \XC_{P(v)}, \quad v \in \Ver(\Gamma), \]
satisfying certain matching conditions (explained later in the
paragraph) on the lifts of nodal points.  Each of the domain components
$C_v^\circ$ is an irreducible curve component $C_v \subset C$
(possibly with boundary) punctured at interior nodal points, that is,
\[C_v^\circ :=C_v \bs \{\text{interior nodes}\},\]
the target space $\XC_{P(v)} $ is a piece of the broken manifold
corresponding to the polytope
\[P(v) \in \PP,\]
\ and the punctures in the domain are removable singularities when
$u_v$ is viewed as a map to the compactification $\ol \XC_{
  P(v)}$.\footnote{Here for the sake of exposition we assume that
  $\XX_{P(v)}$ has a manifold compactification.  The general
  case is treated in Chapter \ref{chap:brokendisks}, where the quantity $\cT(e)$ is defined using the asymptotic behavior of the map near the nodal point.} Thus
a broken map can be evaluated at nodal lifts. In particular, in a broken map $u$ modelled on a graph $\Gamma$, consider a node $w$ corresponding to an edge $e=(v_+,v_-) \in \Edge(\Gamma)$ which 
lifts to $w^\pm \in C_{v_\pm}$ in the normalization $\tC$ of $C$. The map $u$ can be
evaluated at the nodal lift $w^+ \in C_{v_+}$, and the point $u(w^+)$
lies on the intersection
\[Y_+=\cap_{i=1}^k D_i^+\]
of a collection of relative divisors $D_1^+, \dots, D_k^+$ in
$\ol X_{P(v_+)}$. Assume that each $D_i^+$ is the fixed point set of a
one-dimensional torus generated by $\mu_i \in \t$, and $n_i$ is the
intersection multiplicity of the map $u_{v_+}$ at $w^+$.  Then the sum
\[\cT(w^+) :=\sum n_i \mu_i \in \t_\Z\]
lies in the integer lattice $\t_\Z \subset \t$. 
Define
$\cT(w^-) \in \t_\Z$ for the lift $w^-$ similarly.  For any holomorphic
coordinate $z_\pm$ on a neighborhood of $w^\pm$ in $C_{v_\pm}$, the
limit
\[\lim_{z_\pm \to 0}z_\pm^{-\cT(w^\pm)}u \]
exists.  (See Section \ref{sec:mbrokdisks} for more details.)  The
\em{matching condition} at the node $w$ says that
\index{Matching condition at nodes}
\begin{itemize}
\item $\cT(w^+)=\cT(w^-)$, and
\item there exist holomorphic coordinates $z_+$, $z_-$ in neighborhoods of the nodal lifts $w_+$, $w_-$ such that
  \begin{equation} \label{eq:nodematch-intro} \lim_{z_+ \to
      0}z_+^{-\cT(w^+)}u=\lim_{z_- \to 0}z_-^{-\cT(w^-)}u
    .\end{equation}
\end{itemize}
The quantity $\cT(w^+)=\cT(w^-)$ is called the \em{{direction} of the node
  $w$} or the \em{{direction} of the edge $e$} in $\Gamma$ corresponding to
the node $w$; and is denoted by
\index{Direction! of an edge}
\index{Direction! of a node} 
\begin{equation}
  \label{eq:nodeslope}
  \cT(e).   
\end{equation}
The quantities in the left-hand side and right-hand side of
\eqref{eq:nodematch-intro} are called the \em{tropical evaluations}
\index{Tropical evaluation map} 
at $w^+$ and $w^-$ respectively.  An equivalent formulation of the
matching condition states that, for any node, the evaluations of the
lifts are equal modulo the action of the torus generated by the {direction}
of the node: For a node $w \in C$ corresponding to an edge $e$, the
evaluations of the lifts $w^\pm$ lie in $\XC_{P(e)}$ which has the
structure of a $T_{P(e),\C}$-bundle, where
\[P(e) := P(v_+) \cap P(v_-) \in \PP,\]
is the polytope assigned to the edge $e=(v_+,v_-)$.  The {direction} of the
node $\cT(e)$ lies in the lattice $\t_{P(e),\Z}$ and generates a
one-dimensional torus $T_{\cT(e),\C}$. The matching condition is then
that the images of the evaluations match in the base of the
$T_{\cT(e),\C}$-fibration:
\begin{equation}
  \label{eq:match-intro-proj}
  (\pi_{\cT(e)}^\perp \circ u)(z_+)=(\pi_{\cT(e)}^\perp \circ u)(z_-) \in \XC_{P(e)}/T_{\cT(e),\C},
\end{equation}
where $\pi_{\cT(e)}^\perp : \XC_{P(e)} \to \XC_{P(e)}/T_{\cT(e),\C}$
is the projection to the quotient. The quantities in the left-hand
side and right-hand side of \eqref{eq:match-intro-proj} are called
\em{projected tropical evaluations}.
\index{Projected tropical evaluation map} 
Thus the space
$\XC_{P(e),\C}/T_{\cT(e),\C}$ in which the matching condition of an
edge $e$ is defined is dependent on $e$. The matching condition is
simpler in the special case of a single cut, because the space
$\ol \XC_{P(e),\C}/T_{\cT(e),\C}$ is the relative divisor $X_{P_0}$
(using notation from Figure \ref{fig:poly-single}).  The matching
condition at a node $w$ is
\[u_{v_+}(w^+)=u_{v_-}(w^-) \in X_{P_0}. \]

A \em{broken map}
\index{Map! Broken maps} 
is a collection of maps
$u_v: C_v^\circ \to \XC_{P(v)} $ described above together with a
tropical structure on the graph $\Gamma$ underlying the domain nodal
curve.  
In the case with Lagrangian boundary condition, the components $C_v^\circ$
with boundary have additional gradient treed segments in the domain 
explained below in Section \ref{chap:brokendisks}, used to work in the Morse model for Floer cohomology.
A \em{tropical structure} on a graph $\Gamma$ is a collection
of edge {direction}s
  \index{Tropical structure} 
\[ \cT(e) \in \t_{P(e),\Z} \subset \t \simeq \t^\dual , \] 
and polytope assignments $P(v) \in \PP$ for vertices
$v \in \Ver(\Gamma)$ so that the graph is realizable in the dual
complex $B^\dual$ of the neck-stretching in the following sense: There
exist \em{tropical positions} of the vertices 
 \index{Tropical vertex positions|seeonly{Vertex positions}}
in the dual complex
\[\cT : \Ver(\Gamma) \to B^\dual \subset \t^\dual, \quad \cT(v) \in P(v)^\dual\]
that satisfy the following: For any edge $e=(v_+,v_-)$ the {direction} of
the line segment joining $\cT(v_+)$ to $\cT(v_-)$ is equal to the
{direction} $\cT(e)$ of the node $w$ corresponding to $e$.  
That is,
\index{Direction condition! for tropical graphs}
\begin{equation} \label{slopecond} \cT(e) \in \R_{> 0} ( \cT(v_+) -
  \cT(v_-) ).
\end{equation}
The image of $\cT(\Gamma)$ under the map to $B^\dual$ induced by
$v \mapsto \cT(v)$ is called a \em{realization of a tropical
  graph},   \index{Realizability of a tropical graph}
and the underlying graph equipped with just the edge {direction}s
$\{\cT(e)\}_e$ and vertex polytopes $\{P(v)\}_v$ is called a \em{
  tropical graph}.
\index{Tropical graph}
Thus, changing the tropical vertex positions
$\{\cT(v)\}_v$ produces a different realization of the same tropical
graph; an example is shown in Figure \ref{fig:movetrop}.  Broken maps
may also contain nodes corresponding to the standard nodal
degeneration encountered in Gromov-Witten theory. The edges
corresponding to such nodes do not appear in the tropical graph.

Broken maps arise naturally as limits of sequences of holomorphic maps
in neck-stretched manifolds.  A converging sequence of maps consists
of pockets of high symplectic area separated by long cylinders in neck
regions.  Each such sequence $u_\nu$ of long cylinders maps to a neck
piece
\[u_\nu : \left[\tfrac {-\nu} 2, \tfrac \nu 2 \right] \times S^1 \to
  \XC_P\]
for some $P \in \PP$ in the polytopal decomposition, and is
asymptotically close to a \em{trivial cylinder}.
\index{Trivial cylinder} 
A trivial cylinder
is a holomorphic cylinder which lies in a fiber of the projection
\[T_{P,\C} \to \XC_P \to \XB_P, \]
and is thus diffeomorphic to a subtorus $T_{\mu,\C} \subset T_{P,\C}$
generated by a rational element $\mu \in \t_{P,\Z}$.  The pockets of
high area converge to spheres or disks, and (roughly speaking) the
long cylinders connecting these pockets converge to trivial cylinders.
We drop the rational cylinders from our description of the ``map
part'' of a broken map, and instead encode them as edges of the
tropical graph, as shown in Figure \ref{fig:trop}.  The generator
\[ \cT(e) \in \t_{P(e),\Z} \]  
of the trivial cylinder corresponding to an edge $e$ is the {direction} of
the edge in the tropical graph.  In particular, each edge {direction} is
integral.  Every component $C_v$, $v \in \Ver(\Gamma)$ of the limit
curve has a natural position $\cT(v) \in B^\dual$ in the dual complex
as follows: For a limit curve component mapping to a neck piece
$\XC_P$, $P \in \PP$ with fibers $T_{P,\C} \simeq (\C^\times)^n$, the
convergence is modulo $\R^n$-translation in the target space $\XC_P$.  More
specifically, for a limit curve component $u:C \to \XC_P$ there is a
sequence $t_\nu \in \nu P^\dual \subset \R^n \subset \t_{P,\C}$ such
that $u$ is the limit of the translated maps $e^{-t_\nu} u_\nu$. Thus,
the component $u$ inherits a \em{tropical coordinate}
\[ \cT(v) := \lim_{\nu \to \infty} \frac {t_\nu} \nu \in P^\dual . \]
For an edge $e=(v_+,v_-)$, the line segment connecting $\cT(v_+)$ and
$\cT(v_-)$ has {direction} $\cT(e)$.

\begin{figure}[ht]
  \centering\scalebox{.8}{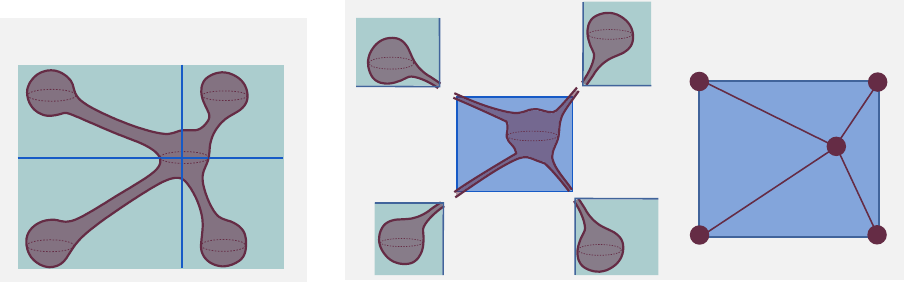}
  \caption{Left : Map in a neck-stretched manifold. Right : A broken map and its tropical graph.}
  \label{fig:trop}
\end{figure}

The realization of the broken map as a limit of maps in neck-stretched
manifolds also explains the matching condition at nodes.  A node $w$
corresponding to an edge $e$ in a limit broken map arises from the
convergence of a sequence of long cylinders
\[u_\nu: S^1 \times [-l_\nu, l_\nu] \to (X,J_\nu), \quad l_\nu \to \infty\]
with small area, each of which is asymptotically close to a trivial
cylinder
\[ u_{\triv}: S^1 \times [-l_\nu, l_\nu] \to (X,J_\nu), \quad (s,t)
  \mapsto e^{\cT(e)(s+it)}x_0, \quad x_0 \in X \]
generated by an integral element $\cT(e) \in \t_{P(e),\Z}$. As a
result, the evaluation of both lifts of the node $w_e^\pm$ are equal
modulo the action of the one-dimensional complex torus
$T_{\cT(e),\C}$.  This is a re-statement of the matching condition
\eqref{eq:match-intro-proj}.

\index{Tropical symmetry} The set of broken maps of a fixed type
$\Gamma$ has a free action of a \em{tropical symmetry} group
$T_{\on{trop}}(\Gamma)$ arising out of the torus action on neck
pieces. The tropical symmetry group $T_{\on{trop}}(\Gamma)$ is
generated by the degrees of freedom of the tropical graph $\cT$: Each
element of the real part of this group corresponds to ways of moving
the vertex positions $\{\cT(v)\}_v$ without changing edge {direction}s
$\{\cT(e)\}_e$, as shown in Figure \ref{fig:rigid}.  In particular,
the symmetry group $T_{\on{trop}}(\Gamma)$ is finite if in the
tropical graph vertex positions $\cT(v)$ are uniquely \index{Rigid!
  tropical graph} determined by the edge {direction}s $\cT(e)$.  Such
tropical graphs are called \em{rigid}.  In the second graph
$\Gamma_2$ in Figure \ref{fig:rigid}, there is one degree of freedom
moving the edges of the tropical graph, and this generates a
one-dimensional complex torus $T_\trop(\Gamma_2)$. The quotient of the
moduli space $\M_\Gamma^\br(\XX,L)$ of broken maps of type $\Gamma$
 \index{Moduli space!of broken maps $\M_\Gamma^\br(L)$}
by
the action of the tropical symmetry group $T_\trop(\Gamma)$ is called
the \em{reduced moduli space} (see \eqref{eq:reduced-def}), and is
denoted by
\index{Moduli space!of broken maps, reduced $\M_{\Gamma,\red}^\br(L)$} 
\[ \M_{\Gamma,\red}^\br(\XX,L) := \tM_\Gamma^\br(\XX,L) /
T_{\on{trop}}(\Gamma). \]

A feature of broken maps is that ``nodes do not lower the index of a
map'', which is in contrast to the behavior of stable maps in smooth
manifolds. Here, the \em{index} of a broken map $u$,\footnote{We use
  the notation $i(u)$ for the index of a map $u$, which is equal to
  the expected dimension of the moduli space containing $u$; and the
  notation $I(u)$ for Maslov index of a map $u$ with disk as domain
  with Lagrangian boundary conditions.}
\index{Index! of a broken map}
denoted by $i^\br(u)$, refers to the expected dimension of the
moduli space component containing $u$, and is given by the index of
the linearization of the Fredholm operator cutting out the moduli
space.  The index is equal to the actual dimension
\[ i^\br(u) = \dim T_u \M^\br(\XX,L)  \]
if the moduli space $\M^\br(\XX,L)$ is regular.  Nodes do not contribute
negatively to the expected dimension formula for broken maps because
the matching condition at an edge $e \in \Edge(\Gamma)$ has
codimension $(\dim X - 2)$, since the matching condition is a
condition on the quotient of $\XC_{P(e)}$ by a complex one-dimensional
torus $T_{\cT(e),\C}$ as in \eqref{eq:match-intro-proj}.  In contrast,
the codimension of the matching condition for ordinary stable maps is
$\dim(X)$.  Therefore, whereas in nodes for stable maps occur only in
strata whose expected codimension is at least two, broken maps $u$
with components in neck pieces may have index zero, $i^\br(u)=0$,
though that can happen only if the tropical graph is rigid.  Indeed,
if for a type $\Gamma$ the tropical graph is not rigid, then the
tropical symmetry group $T_\trop(\Gamma)$ is at least two-dimensional.
Since $T_\trop(\Gamma)$ has a free action on the moduli space
$\M_\Gamma^\br(\XX,L)$ of maps of type $\Gamma$, the dimension of the
moduli space $\M_\Gamma^\br(\XX,L)$ in this case must be at least $2$.

The Gromov limit of a sequence of broken disks of type $\Gamma$ may
have configurations with sphere and disk bubbling, and may also have
configurations with a different tropical graph $\Gamma'$.  Such a
graph $\Gamma$ can be recovered from $\Gamma'$ by collapsing a subset
of edges (Theorem \ref{thm:cpt-broken}). The tropical graph $\Gamma'$,
when it is distinct from $\Gamma$, necessarily has a larger tropical
symmetry group. By the previous paragraph, the dimensions of the
moduli spaces $\M^\br_{\Gamma'}(\XX,L)$ and $\M^\br_\Gamma(\XX,L)$ are
the same, and therefore, the dimension of the quotient
$\M^\br_{\red,\Gamma'}(\XX,L)$ is at least two lower than the
dimension of $\M^\br_{\red,\Gamma}(\XX,L)$. Thus, the moduli space
$\M^\br_{\red,\Gamma}(\XX,L)$ has a compactification consisting of
codimension one strata made up of configurations with broken treed
segments, and strata of codimension at least two with additional tropical
nodes, as in Remark \ref{rem:maslov4-moduli}.  In this book, we only
deal with moduli spaces of broken maps of expected dimension zero and one; and
therefore codimension-two strata do not occur.  Reduced moduli spaces of broken maps consisting of tropical symmetry orbits of broken maps are the usual objects
of study in the literature on relative and log Gromov-Witten theory
\cite{abram}, \cite{chen:log}, \cite{teh:deg}.  However, we choose to
use the unquotiented moduli space $\M^\br_{\Gamma}(\XX,L)$ because in
the zero-dimensional strata, it has a bijection to unbroken maps, as
we explain
%
following the statement of
Theorem \ref{thm:bfuk}.

Symplectic cuts have been used in the literature to give formulas for
Gromov-Witten invariants in works of many different authors.  The case
of a single cut was studied by Ionel-Parker \cite{io:rel} and, in the
algebro-geometric context, J.~Li \cite{li:degen}.  They obtained a
symplectic sum formula for Gromov-Witten invariants on $X$ in terms of
relative Gromov-Witten invariants of the cut spaces, relative to the
relative divisor.  Their work is a special case of symplectic field
theory of Bourgeois-Eliashberg-Hofer-Wysocki-Zehnder \cite{bo:com}, in
which the analogs of broken maps are known as ``holomorphic
buildings''.  Eleny Ionel \cite{ion:nc} first studied
compactifications of moduli spaces of maps relative to a normal
crossing divisor.  She used a generalization of holomorphic buildings
with levels and chambers to describe the maps appearing in the moduli
space. Brett Parker \cite{bp1,bp2,bp3,bp4,bp5,bp6,bp7,bp8,bp9} used
the tropical approach to study this problem. Parker defined a category
of exploded manifolds, which combined the map part and 
the 
tropical part
into a single space, in a way that convergence of broken maps is
continuous. The corresponding Gromov-Witten invariants in algebraic
geometry are studied in, for example, Abramovich-Chen-Gross-Siebert
\cite{abram}.  Tehrani gives an alternate compactification of
holomorphic curves relative to a normal crossing symplectic divisor
\cite{teh:relcpt, teh:reldef} and uses it to give a degeneration
formula \cite{teh:deg} for Gromov-Witten invariants in the almost
K\"ahler category.  Our approach is essentially a version of Parker's,
except that we use Morse chain models rather than de Rham theory and
work on the chain level to define counts of pseudoholomorphic maps
with Lagrangian boundary.

\section{Broken Fukaya algebras} \label{sec:bfa-intro}

Counts of disks in the broken symplectic manifold lead to a definition
of a broken Fukaya algebra.  We use Cieliebak-Mohnke perturbations
\cite{cm:trans} to regularize the various moduli spaces occurring in
the book.  These are domain-dependent perturbations; the domain of
the map is stabilized by treating intersection points of the map with
a Donaldson-type divisor as marked points.  In order to construct a
Donaldson-type divisor, which we call a \em{stabilizing divisor},
    \index{Stabilizing divisor}
we assume the Lagrangian $L \subset X$ and the symplectic manifold
$(X,\om)$ are compact, connected, and \em{rational}.  Rationality
means that $X$ admits a line bundle whose curvature is the symplectic
form and some tensor power of the line bundle is flat over $L$.  The
perturbation datum is a collection
\[ \ul{\Pe} = \{ \Pe_\Gamma = (J_\Gamma,F_\Gamma ) \}_\Gamma \] 
of domain-dependent perturbations $\Pe_\Gamma$ for each type $\Gamma$
of treed disk, each consisting of a domain-dependent almost complex
structure $J_\Gamma$ and Morse function $F_\Gamma$ on the Lagrangian
$L$.

Using these perturbations, we construct the geometric Fukaya algebra
$CF^{\on{geom}}(L)$ of the Lagrangian $L$ using the Morse model.  The
constructions are similar to those of Seidel \cite{se:bo} in the exact
case and Charest-Woodward \cite{cw:flips} in the rational case.  Let $\Lambda$ denote the universal Novikov field in a formal variable $q$, 
and let $\Lam_{\geq 0} \subset \Lam$ be the subring with only non-negative exponents of $q$, called the \em{Novikov ring}.  
The
underlying vector space of Floer cochains is generated by critical
points of the Morse function $f: L \to \R$ so that
\[ CF^{\on{geom}}(L, \ul {\Pe}) = \bigoplus_{ x  \in \crit(f)} \Lambda_{\geq 0} x , \] 
where we sometimes drop the perturbation data $\ul \Pe$ from the
notation.  The structure maps
\[ m^d: (CF^{\on{geom}}(L))^{\otimes d} \to CF^{\on{geom}}(L) \] 
are defined by counting zero-dimensional components in the moduli space $\M(X, L)$ of
holomorphic treed disks
\[ \M(X, L) = \{ u: C \to X: \enspace u(\partial C) \subset L , \text{ $u$ is a holomorphic treed disk} \}/ \sim .\]
Here $C$ is a union of disks, spheres, and line segments; the map $u$
is pseudoholomorphic on the disk and sphere components and satisfies a
gradient flow equation on the segments.  See Figure
\ref{fig:treed-disk} for a representation and Definition
\ref{def:treedholdisk} for the details.
\begin{figure}[ht]
  \centering\scalebox{.8}{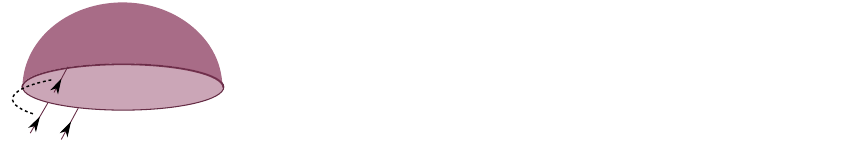}
  \caption{Holomorphic treed disks counted by the composition map $m^d$. Here, segments with arrows are gradient flow lines of the Morse function $f:L \to \R$, and $x_0,\dots,x_d \in \crit(f)$.}
  \label{fig:treed-disk}
\end{figure}

The geometric Fukaya algebra $CF^{\on{geom}}(L)$ is enlarged to an
enhanced space of cochains $CF(L)$ by a \em{homotopy unit
  construction}
\index{Homotopy units} 
so that a strict unit $1_L \in CF(L)$ exists, similar
to the construction in Fukaya-Oh-Ohta-Ono \cite[(3.3.5.2)]{fooo} (see
Section \ref{sec:units} for an exposition on the lines of
\cite{cw:flips}).  Analogously, counts of zero-dimensional components
of the moduli space of broken disks in $\XX$ with boundary in $L$ lead to a
broken Fukaya algebra, denoted $CF_{\br}(L)$, with \ainfty structure
maps denoted
\[ m^d_{\br}
: CF_{\br}(L)^{\otimes d} \to CF_{\br}(L), \quad d \ge 0 
  ;\]
see Chapter \ref{chap:bfa}.  In defining the composition maps
$m^d_\br$, we only need to consider broken maps with rigid tropical
graphs, since otherwise tropical symmetry orbits are at least
two-dimensional.  The single cut version of this \ainfty algebra has
been constructed by Charest-Woodward \cite{cw:flips}. Our main result
is the following, proved in Chapter \ref{chap:bfa}:
\begin{theorem} \label{thm:bfuk} For a rational Lagrangian submanifold
  $L \subset X$ and a polyhedral decomposition $\cP$ equipped with a cutting datum
  as above, the unbroken Fukaya
  algebra $CF(L)$ admits a curved $A_\infty$ homotopy equivalence to
  the broken Fukaya algebra $CF_{\br}(L)$.
\end{theorem}

\noindent
The above result presumably holds in the more general situation where the tropical moment map $\Phi$ takes values not in $\t^\dual$, but in a tropical affine manifold with singularities as in Definition \ref{def:tropham}. 

Theorem \ref{thm:bfuk} implies that the broken Fukaya algebras associated
to different polyhedral decompositions are homotopy-equivalent.  The
proof of Theorem \ref{thm:bfuk} is completed in Proposition
\ref{prop:unbreak-break}. The ingredients of the proof are a
convergence result and its converse, which is a gluing result. The
convergence result is a generalization of compactness in symplectic
field theory (\cite{bo:com}, \cite{cm:com}) for a single cut and is
proved in Chapter \ref{chap:cpt}: Given a sequence of maps
$u_\nu : C \to X$ that are holomorphic with respect to almost
structures $J_\nu$ that are stretched along multiple necks, there is a
subsequence of $u_\nu$ that converges to a broken map. The limit is
unique up to the action of the identity component of the tropical symmetry group.  The gluing
result in Chapter \ref{chap:glue}, which is proved only for broken
maps of index zero, states that an index zero regular broken map gives
rise to a family of $J^\nu$-holomorphic maps $u_\nu:C \to X$ in the
unbroken manifold.

\label{page:unquot}
We emphasize that the bijection produced by the gluing construction
involves broken maps, and not tropical symmetry orbits of broken
maps. The distinction is significant even for broken maps of index zero, because 
such maps may have a finite non-trivial tropical symmetry
group, as in Example \ref{eg:rigidfr}. The distinction is relevant
even in case of a single cut: If at a node, a map intersects the
relative divisor with multiplicity $k$, then the node has $k$
different framings, where \em{framing} is part of the data of a broken
map in Definition \ref{def:bmap}. Thus there are $k$ broken maps,
which differ from each other only in their framing, but they each
produce a distinct unbroken map when glued at the neck. This set of
$k$ broken maps lies in the same orbit of the tropical symmetry group.

\label{page:orb-noreg}
To regularize the moduli spaces of broken maps, we require that the
torus actions in the neighborhood of cut loci (in Definition
\ref{def:tropmanifold} \eqref{part:trop2}) be free. In the absence of
this condition, we would need to allow orbifold domain components for
broken maps, and the node matching condition \eqref{eq:nodematch}
would not be transversally cut out, since the evaluation map at the
orbifold points would be constrained to lie on a stratum of
$\XX_{P(e)}/T_{P(e),\C}$ that has positive codimension. 

\section{Unobstructedness, disk potentials, and cohomology} 

Homotopy equivalence of \ainfty algebras preserves the Floer
cohomology and disk potentials, when each of these is interpreted as a
function on the Maurer-Cartan moduli space.

The cohomology $H(A)$ of an \ainfty algebra $A$ is well-defined if the
first order composition map $m^1$ satisfies $(m^1)^2=0$. The condition
$(m^1)^2=0$ may fail to hold if the \em{curvature} $m^0(1) \in A$ is
not a $\Lam$-multiple of the unit $1_L$, which is an ``obstruction''
to the definition of Floer cohomology.  \em{Weak unobstructedness} is
a more general condition under which the Floer cohomology of a
Lagrangian brane can be defined. A Lagrangian brane $L$ is \em{weakly
  unobstructed} if the projective Maurer-Cartan equation
\index{Maurer-Cartan equation} \index{Weakly bounding cochain|seeonly {Maurer-Cartan
    equation}}
\begin{equation} \label{eq:mc-intro} 
m^0(1) + m^1(b) + m^2(b,b) + \ldots  \in \Lam_{>0} 1_L
\end{equation}
has an odd solution $b \in CF(L)$; here, $\Lambda_{> 0}$ is the subspace
of $\Lambda$ with positive $q$-valuation.  Such a solution $b$ is called a
\em{weakly bounding cochain}, and the set of all the odd solutions is
denoted $MC(L)$.  Given a weakly bounding cochain, the Fukaya algebra
$CF(L)$ may be ``deformed'' by $b$ (see \eqref{eq:mnb}) to yield an
\ainfty-algebra $CF(L,b)$ with composition maps $(m^n_b)_{n \geq 0}$
satisfying $m^0_b \in \Lam_{>0}1_L$.  As a consequence, $(m^1_b)^2=0$, and
the Floer cohomology
 \[H(CF(L,b)):=\ker(m^1_b)/\im(m^1_b)\]
 is well-defined. 

 The homotopy equivalence of Theorem \ref{thm:bfuk} gives an
 isomorphism of Floer cohomology in the following sense.  The \ainfty
 homotopy equivalence in Theorem \ref{thm:bfuk} is given by curved
 \ainfty functors
\begin{equation}
  \label{eq:fghom}
  \F: CF(L) \to CF_\br(L), \quad \G: CF_\br(L) \to
  CF(L),   
\end{equation}
The functors $\F$, $\G$ induce maps on the spaces of Maurer-Cartan
solutions
\[\F : MC(CF(L)) \to MC(CF_\br(L)),\quad \G:MC(CF_\br(L)) \to MC(CF(L))\]
and maps on Floer cohomology, that are isomorphisms
\begin{align*}
  H(\F) : H(CF(L,b_0)) &\to H(CF_\br(L,\F(b_0))), \\
  H(\G) : H(CF_\br(L,b_1)) &\to H(CF(L,\G(b_1)))  
\end{align*}
 for any Maurer-Cartan solution $b_0 \in MC(CF(L))$, $b_1 \in MC(CF_\br(L))$.
See \cite[Section 5.1]{cw:flips} for details.

The \em{disk potential} of a Lagrangian brane $L$ is defined as a
count of disks with no input and a single output.  
It is a function on
the space of local systems on $L$ given by 
\begin{equation}
  \label{eq:wnaive}
W : \Hom(\pi_1(L),\C^\times) \to \Lam, \quad y \mapsto \sum_u \eps(u) y([\partial u]) q^{\om(u)},  
\end{equation}
where $\eps(u) \in \pm 1$ is an orientation sign, $q$ is the
Novikov formal variable,
$\om(u)$ is the area,
and the sum ranges over all disk maps $u$ with Maslov index $2$ that
pass through a fixed (output) point on the boundary.

The homotopy equivalence between Fukaya algebras does not preserve the
disk potential, but preserves a \em{generalized disk potential}
defined by Fukaya-Oh-Ohta-Ono \cite{fooo} as a function on the space
of solutions of the projective Maurer-Cartan equation
\[ W_{\on{gen}}: MC(L) \to \Lambda , \quad b \mapsto W_{\on{gen}}(b), \]
 \index{Potential}
\index{Disk potential|seeonly{Potential}}
 where the quantity $W_{\on{gen}}(b) \in \Lam_{\geq 0}$ is given by the Maurer-Cartan equation
 \[m^0(1) + m^1(b) + m^2(b,b) + \ldots =
W_{\on{gen}}(b) 1_{L}. \]
The potential is preserved by the homotopy equivalence
with the broken Fukaya algebra:
For any $b \in MC(CF(L))$, 
$\cF(b)$ is a solution of the Maurer-Cartan equation on $CF_\br(L)$, that is, $\F(b) \in MC(CF_\br(L))$, and
$W_{\on{gen}}(b)=W_{\on{gen}}(\F(b))$.  If 
 $b=0$ is a
solution of the Maurer-Cartan equation,
we say that $W_{\on{gen}}(0)$ is the \em{naive disk potential} as in \eqref{eq:wnaive}.
The naive potential 
is not preserved by neck-stretching because the \ainfty functors $\F$,
$\G$ occurring in the homotopy equivalence are curved, that is,
$\F^0=\G^0$ may not equal zero, 
and therefore, $\F(b=0) \neq 0$.  See example
\ref{ex:h2}.

For all our applications, we use the definition of the disk potential
as a naive disk count, in spite of the limitations of this
definition.  We consider various examples in Chapter \ref{chap:apps};
in all of these $b=0$ is a solution of the Maurer-Cartan equation,
\footnote{Technically speaking, once the Fukaya algebra is enhanced with homotopy units,
in the examples of Chapter \ref{chap:apps}, we actually obtain that the naive disk count is a multiple of $x^\blackt$, which implies weak unobstructedness by Lemma \ref{lem:unobs-cond}.}
and therefore the disk potential is meaningful. In the monotone case, the
\ainfty functors are not curved, and in the semi-Fano case, we work
with perturbations for which \ainfty functors are flat, so that in
both cases, disk potentials are preserved by neck-stretching. In
the sequel \cite{vw:split}, we use neck-stretching and a further
degeneration to show weak unobstructedness of toric Lagrangians and to
compute disk potentials for almost toric manifolds.

\section{Almost toric manifolds and toric degenerations}\label{sec:degen}

We relate our results to those for toric degenerations as in
Gross-Siebert (\cite{grsi:toric}, \cite{gross:invite}) and to
computation of disk potentials in almost toric manifolds. Both of
these involve generalizing broken manifolds by allowing the moment map
to take values in singular tropical affine manifolds. A \em{tropical
  affine manifold} $B$ is a topological manifold with coordinate
charts whose transition functions take values in
$\R^n \ltimes GL(n,\Z)$.

\em{Almost toric manifolds} in the sense of Symington \cite{sym:2to4}
are a class of examples in dimension four; these are fibrations over
tropical affine two-manifolds with dimension zero singular loci whose
fibers are Lagrangian tori. Thus, they are generalizations of toric
manifolds, where the completely integrable system is allowed to have
isolated focus-focus singularities, with the monodromy of the
Lagrangian torus fibration around the singular point given by a shear
map; see \cite{evans:almosttoric} for recent exposition.  An almost
toric manifold is equipped with a \em{base diagram} which is the
image of the tuple of Hamiltonians, with the additional data of the
locations of the singularities on the base diagram, and the directions
of the eigenvector of the monodromy indicated by a dotted line
emanating from the singular value, as in Figure \ref{fig:Ad-sing}.
The complement of the singularities and boundary has a torus fibration
that projects to the complement of the dotted line.  The affine
structure is glued along the dotted line by the map of tori induced by the
transpose of the shear above.  The fiber above the singular point is a
pinched torus.

As an example, for integers $d>1$ the resolution of an $A_d$-singularity $\C^2/\Z_d$ has an almost toric structure whose base diagram is as in Figure
\ref{fig:Ad-sing}.  The fibration may be deformed by a ``nodal slide'' operation
(without changing the total space, as a symplectic manifold) that
pushes the singular values in the base diagram to a vertex; see
\cite[Theorem 6.5]{sym:2to4}.  In the limit when the singular value in
the base diagram coincides with a vertex $v$, the resulting fibration
has a non-toric singularity at the vertex.  The fiber over the vertex
$v$ is a path of $d-1$ Lagrangian spheres.\footnote{For $t \in (0,1]$
  the hypersurface $\PP_t:=\{z_1z_2 + P_t(z_3) =0\} \subset \C^3$
  (where $P_t$ is a polynomial of degree $d$ with distinct real roots
  for $t >0$ and $P_0(z)=z^d$) has an almost toric structure with $d$
  focus-focus singularities (\cite[Section
  7.3]{evans:almosttoric}). There is a path of $d-1$ Lagrangian
  spheres, each passing through a pair of the singular points
  (\cite[Remark 7.7]{evans:almosttoric}). The path of spheres gets
  collapsed to an $A_{d-1}$-singularity in the limit $t \to 0$.}
Via a deformation of the symplectic form, the Lagrangian spheres may
be deformed to symplectic spheres with self-intersection number $-2$, as in
the rightmost polytope in Figure \ref{fig:Ad-sing}.  We compute the
disk potentials of resolutions of $A_1$ and $A_2$-singularities using
multiple cuts in Sections \ref{sec:a1} and \ref{sec:cubic-intro}.
$A_2$-singularities occur in toric degenerations of cubic surfaces.
Similar multiple cuts to the one in Section \ref{sec:cubic-intro} can
presumably be used to analyze spheres and disks in resolutions of
$A_d$-singularities for $d \ge 2$, which would give alternate proofs
of results in Chan-Lau \cite{chanlau}.  In upcoming work
\cite{vw:torus}, we compute disk potentials of resolutions of cyclic
quotient $T$-singularities, which are quotients of $A_d$-singularities
by finite groups, and which model del Pezzo surfaces locally.

\begin{figure}[ht]
   \scalebox{.8}{
\begingroup%
  \makeatletter%
  \providecommand\color[2][]{%
    \errmessage{(Inkscape) Color is used for the text in Inkscape, but the package 'color.sty' is not loaded}%
    \renewcommand\color[2][]{}%
  }%
  \providecommand\transparent[1]{%
    \errmessage{(Inkscape) Transparency is used (non-zero) for the text in Inkscape, but the package 'transparent.sty' is not loaded}%
    \renewcommand\transparent[1]{}%
  }%
  \providecommand\rotatebox[2]{#2}%
  \newcommand*\fsize{\dimexpr\f@size pt\relax}%
  \newcommand*\lineheight[1]{\fontsize{\fsize}{#1\fsize}\selectfont}%
  \ifx\svgwidth\undefined%
    \setlength{\unitlength}{334.71248561bp}%
    \ifx\svgscale\undefined%
      \relax%
    \else%
      \setlength{\unitlength}{\unitlength * \real{\svgscale}}%
    \fi%
  \else%
    \setlength{\unitlength}{\svgwidth}%
  \fi%
  \global\let\svgwidth\undefined%
  \global\let\svgscale\undefined%
  \makeatother%
  \begin{picture}(1,0.33794937)%
    \lineheight{1}%
    \setlength\tabcolsep{0pt}%
    \put(0,0){\includegraphics[width=\unitlength,page=1]{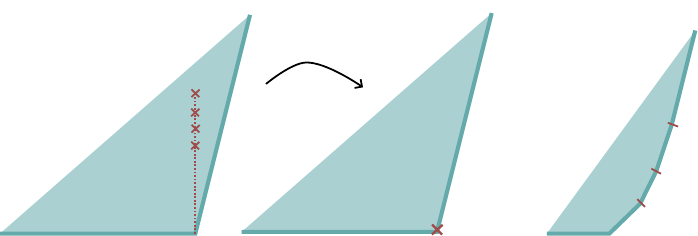}}%
    \put(0.40623785,0.30838149){\color[rgb]{0,0,0}\makebox(0,0)[lt]{\lineheight{1.25}\smash{\begin{tabular}[t]{l}Nodal \\slide\end{tabular}}}}%
    \put(0,0){\includegraphics[width=\unitlength,page=2]{Ad-sing.pdf}}%
    \put(0.74301214,0.28444306){\color[rgb]{0,0,0}\makebox(0,0)[lt]{\lineheight{1.25}\smash{\begin{tabular}[t]{l}Deform \\symplectic\\form\end{tabular}}}}%
  \end{picture}%
\endgroup%
}
   \caption{Resolution of the $A_{d-1}$-singularity $\C^2/\Z_d$ for $d=4$.  Left and middle polytope: $\{y \geq 0, y \geq dx\}$. In the middle figure $\Phinv(0)$ is a path of $(d-1)$ Lagrangian spheres.
The right figure is a toric smoothing of the $A_{d-1}$-singularity whose moment polytope has 
     additional facets $y=kx + \eps_k$, $k=1,\dots, (d-1)$.}
\label{fig:Ad-sing}
\end{figure}

Our results also apply to some toric degenerations. A \em{toric degeneration}, as in Gross and Siebert
(\cite{grsi:toric}, \cite{gross:invite}), is a flat family
\[ \pi : \mathcal{X} \to \C \]
whose fiber
\[X_t:=\pinv(t)\]
is regular and diffeomorphic to an irreducible projective variety $X$
for $t \in \C \bs \{0\}$, and whose central fiber $X_0$ is a union of
toric varieties glued pair-wise along torus-invariant divisors. The
central fiber $X_0$ has a tropical moment map $\Phi : X_0 \to B$ to a
singular integral affine manifold $B$, with a singular set
$\Delta \subset B$ of codimension $2$, known as the \em{discriminant
  locus}.  The integral structure of $B$ has a non-trivial monodromy
in $SL(n,\Z)$ around points in the discriminant locus.  In the
complement of the discriminant locus, the central fiber $X_0$
looks exactly like a union of cut spaces obtained from a multiple cut.
Because of the discriminant locus, the gluing of the toric varieties
in the degeneration along torus-invariant divisors is not toric.

Conjecturally, Gross-Siebert's toric degenerations may be viewed as
multiple cuts where the cuts are allowed to pass through singular
Lagrangian fibers in an almost toric manifold, as can be seen from the
toric degeneration of a cubic surface: Given a generic homogeneous
polynomial $f \in \C[x_0,x_1,x_2,x_3]$ of degree $3$, the toric
degeneration of 
 $\{f=0\} \subset \P^3$ 
is given by the family
  \[Y:=\{tf + x_0x_1x_2=0 : t \in \C\} \xmapsto{\pi} t \in \C. \]
  The fibers of $\pi$ are smooth except for the central fiber
  $X_0:=\pinv(0)$, which is the union of $3$ copies of $\P^2$, namely
  $\{x_i=0\}, i=0,1,2$, glued pairwise along lines $\{x_i=x_j=0\}$,
  $i \neq j$. The line $\{x_i=x_j=0\}$ has singular points at its
  intersection points with $\{f=0\}$. Thus, the central fiber $X_0$ has the structure of a tropical Hamiltonian action, where any two copies of $\P^2$ are glued along   a line containing three points in the discriminant locus, and the
  underlying integral affine manifold $B$ is as shown in the left side
  of Figure \ref{fig:cubic-deg}.
 
\begin{figure}[h]
  \centering\scalebox{.8}{
\begingroup%
  \makeatletter%
  \providecommand\color[2][]{%
    \errmessage{(Inkscape) Color is used for the text in Inkscape, but the package 'color.sty' is not loaded}%
    \renewcommand\color[2][]{}%
  }%
  \providecommand\transparent[1]{%
    \errmessage{(Inkscape) Transparency is used (non-zero) for the text in Inkscape, but the package 'transparent.sty' is not loaded}%
    \renewcommand\transparent[1]{}%
  }%
  \providecommand\rotatebox[2]{#2}%
  \newcommand*\fsize{\dimexpr\f@size pt\relax}%
  \newcommand*\lineheight[1]{\fontsize{\fsize}{#1\fsize}\selectfont}%
  \ifx\svgwidth\undefined%
    \setlength{\unitlength}{444.03389533bp}%
    \ifx\svgscale\undefined%
      \relax%
    \else%
      \setlength{\unitlength}{\unitlength * \real{\svgscale}}%
    \fi%
  \else%
    \setlength{\unitlength}{\svgwidth}%
  \fi%
  \global\let\svgwidth\undefined%
  \global\let\svgscale\undefined%
  \makeatother%
  \begin{picture}(1,0.25181261)%
    \lineheight{1}%
    \setlength\tabcolsep{0pt}%
    \put(0,0){\includegraphics[width=\unitlength,page=1]{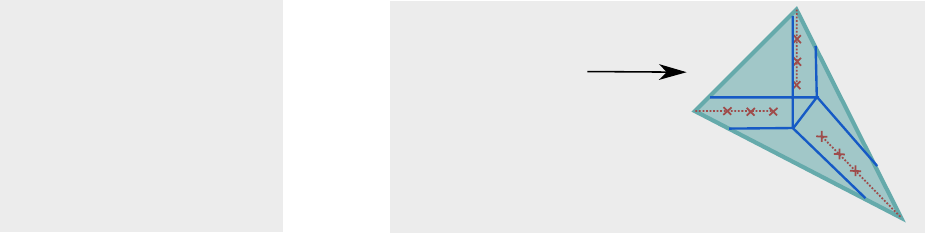}}%
    \put(0.62957034,0.19666175){\color[rgb]{0,0,0}\makebox(0,0)[lt]{\lineheight{1.25}\smash{\begin{tabular}[t]{l}Straddling cut\end{tabular}}}}%
    \put(0,0){\includegraphics[width=\unitlength,page=2]{cubic-deg.pdf}}%
    \put(0.06679883,0.15164604){\color[rgb]{0,0,0}\makebox(0,0)[lt]{\lineheight{1.25}\smash{\begin{tabular}[t]{l}$x_0=0$\end{tabular}}}}%
    \put(0.16093933,0.14398778){\color[rgb]{0,0,0}\makebox(0,0)[lt]{\lineheight{1.25}\smash{\begin{tabular}[t]{l}$x_1=0$\end{tabular}}}}%
    \put(0.07917518,0.09005625){\color[rgb]{0,0,0}\makebox(0,0)[lt]{\lineheight{1.25}\smash{\begin{tabular}[t]{l}$x_2=0$\end{tabular}}}}%
    \put(0.01488989,0.0485284){\color[rgb]{0,0,0}\makebox(0,0)[lt]{\lineheight{1.25}\smash{\begin{tabular}[t]{l}$B$\end{tabular}}}}%
  \end{picture}%
\endgroup%
}
  \caption{Left: The integral affine manifold $B$ underlying $X_0$, the central fiber in the toric degeneration of the cubic surface. Right: A straddling multiple cut on the almost toric base diagram of the cubic surface.}
  \label{fig:cubic-deg}
\end{figure}

\begin{figure}[h]
  \centering\scalebox{.8}{
\begingroup%
  \makeatletter%
  \providecommand\color[2][]{%
    \errmessage{(Inkscape) Color is used for the text in Inkscape, but the package 'color.sty' is not loaded}%
    \renewcommand\color[2][]{}%
  }%
  \providecommand\transparent[1]{%
    \errmessage{(Inkscape) Transparency is used (non-zero) for the text in Inkscape, but the package 'transparent.sty' is not loaded}%
    \renewcommand\transparent[1]{}%
  }%
  \providecommand\rotatebox[2]{#2}%
  \newcommand*\fsize{\dimexpr\f@size pt\relax}%
  \newcommand*\lineheight[1]{\fontsize{\fsize}{#1\fsize}\selectfont}%
  \ifx\svgwidth\undefined%
    \setlength{\unitlength}{285.89654156bp}%
    \ifx\svgscale\undefined%
      \relax%
    \else%
      \setlength{\unitlength}{\unitlength * \real{\svgscale}}%
    \fi%
  \else%
    \setlength{\unitlength}{\svgwidth}%
  \fi%
  \global\let\svgwidth\undefined%
  \global\let\svgscale\undefined%
  \makeatother%
  \begin{picture}(1,0.22301974)%
    \lineheight{1}%
    \setlength\tabcolsep{0pt}%
    \put(0,0){\includegraphics[width=\unitlength,page=1]{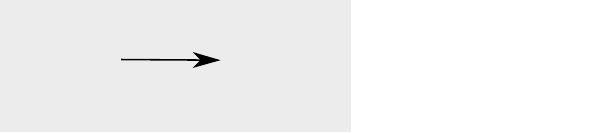}}%
    \put(0.20082604,0.18491624){\color[rgb]{0,0,0}\makebox(0,0)[lt]{\lineheight{1.25}\smash{\begin{tabular}[t]{l}Straddling \\cut\end{tabular}}}}%
    \put(0,0){\includegraphics[width=\unitlength,page=2]{strad.pdf}}%
    \put(0.92042503,0.11548436){\color[rgb]{0,0,0}\makebox(0,0)[lt]{\lineheight{1.25}\smash{\begin{tabular}[t]{l}$(-1)$\\sphere\end{tabular}}}}%
  \end{picture}%
\endgroup%
}
  \caption{Left: A straddling cut. Right: Part of a broken sphere. }
  \label{fig:strad}
\end{figure}

\begin{figure}[h]
  \centering\scalebox{.8}{
\begingroup%
  \makeatletter%
  \providecommand\color[2][]{%
    \errmessage{(Inkscape) Color is used for the text in Inkscape, but the package 'color.sty' is not loaded}%
    \renewcommand\color[2][]{}%
  }%
  \providecommand\transparent[1]{%
    \errmessage{(Inkscape) Transparency is used (non-zero) for the text in Inkscape, but the package 'transparent.sty' is not loaded}%
    \renewcommand\transparent[1]{}%
  }%
  \providecommand\rotatebox[2]{#2}%
  \newcommand*\fsize{\dimexpr\f@size pt\relax}%
  \newcommand*\lineheight[1]{\fontsize{\fsize}{#1\fsize}\selectfont}%
  \ifx\svgwidth\undefined%
    \setlength{\unitlength}{444.72584461bp}%
    \ifx\svgscale\undefined%
      \relax%
    \else%
      \setlength{\unitlength}{\unitlength * \real{\svgscale}}%
    \fi%
  \else%
    \setlength{\unitlength}{\svgwidth}%
  \fi%
  \global\let\svgwidth\undefined%
  \global\let\svgscale\undefined%
  \makeatother%
  \begin{picture}(1,0.34154923)%
    \lineheight{1}%
    \setlength\tabcolsep{0pt}%
    \put(0,0){\includegraphics[width=\unitlength,page=1]{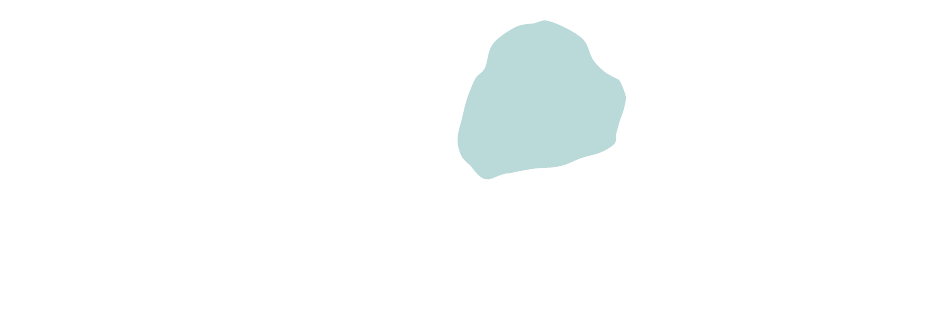}}%
    \put(0.58346741,0.20949621){\color[rgb]{0,0,0}\makebox(0,0)[lt]{\lineheight{1.25}\smash{\begin{tabular}[t]{l}$P$\end{tabular}}}}%
    \put(0,0){\includegraphics[width=\unitlength,page=2]{quartic.pdf}}%
    \put(0.18106877,0.31385792){\color[rgb]{0,0,0}\makebox(0,0)[lt]{\lineheight{1.25}\smash{\begin{tabular}[t]{l}$P$\end{tabular}}}}%
    \put(0.02850173,0.28726793){\color[rgb]{0,0,0}\makebox(0,0)[lt]{\lineheight{1.25}\smash{\begin{tabular}[t]{l}(A)\end{tabular}}}}%
    \put(0.41009779,0.30089648){\color[rgb]{0,0,0}\makebox(0,0)[lt]{\lineheight{1.25}\smash{\begin{tabular}[t]{l}(B)\end{tabular}}}}%
    \put(0.7644369,0.30771064){\color[rgb]{0,0,0}\makebox(0,0)[lt]{\lineheight{1.25}\smash{\begin{tabular}[t]{l}(C)\end{tabular}}}}%
    \put(0,0){\includegraphics[width=\unitlength,page=3]{quartic.pdf}}%
  \end{picture}%
\endgroup%
}
  \caption{(a) Integral affine manifold corresponding to the toric
    degeneration $X_0$ of a quartic surface. There are 24 singular
    points on the cut locus.  (b) Base diagram for the almost toric
    fibration on the gluing $X_0^\glue$ in a neighborhood of the
    vertex $P$.  (c) A part of the straddling cut near $P$.}
  \label{fig:quartic}
\end{figure}

In some cases, we may apply our results to these kinds of singular
toric degenerations if we replace a single cut by two parallel
non-singular cuts \em{straddling} the singular point as shown in
Figure \ref{fig:strad}. The ``non-toric gluing'' in toric
degenerations is thus replaced by an extra cut space that is
non-toric. Multiple singular cuts may also be replaced by straddling
cuts, as in Figure \ref{fig:cubic-deg} for the example of the cubic
surface.  Straddling cuts have the feature that (for a carefully chosen
multiple cut), a holomorphic curve component intersecting the singular
point is necessarily a cover of a fixed $(-1)$-curve.  We use such cuts
to compute disk potentials of resolutions of $T$-singularities in
\cite{vw:torus}.

\begin{example}\label{ex:quartic}
  A quartic surface in $\P^3$ can be degenerated in a similar way to
  the cubic surface, as explained in Gross \cite[Section
  2.1]{gross:invite}. The degenerated variety consists of four copies
  of $\P^2$ glued pairwise along toric divisors, and there are $4$
  singular points along each of the glued divisors.  See Figure
  \ref{fig:quartic}.  Analogously to the example of the cubic surface,
  via a straddling cut we obtain a multiply cut manifold whose
  (symplectic) sum is conjecturally symplectomorphic to the quartic.
  Our set-up assumes that the polyhedral decomposition in 
   the tropical Hamiltonian action is embedded in a vector space, while in the
  quartic example the topological manifold $B=\cup_{P \in \PP}P^\circ$
  underlying the polyhedral decomposition is a two-sphere, and has the
  structure of an affine manifold with singularities as defined in
  Definition \ref{def:affwsing} below.  It would be interesting to
  compute the disk potential of a Lagrangian using broken disks. This
  ends the Example.
\end{example}

\chapter{Applications to disk counting} 
\label{chap:apps}

In this Chapter, we describe several examples of disk potential and
sphere count computations that the reader might keep in mind during
the reading of the theory.
\label{page:monotone}  Computationally, the easy applications of the theory are in cases
where $(X,L)$ is monotone. A pair $(X,L)$ is \em{monotone} if there
is a constant $\tau>0$ for any disk $u$ in $X$ with boundary on $L$,
the Maslov index of $u$ is equal to the multiple $\tau \om(u)$
of the area $\om(u)$.  For a monotone pair $(X,L)$, the Fukaya algebra
is weakly unobstructed since any holomorphic disk has Maslov index $\geq
2$. The potential $m^0(1)$ is independent of the choice of almost
complex structure.  Indeed, for any two almost complex structures $J_0$
and $J_1$, and a generic path $\J:=\{J_t\}_t$ connecting them, the
one-dimensional component of the moduli space of $\J$-holomorphic
disks with a single boundary constraint does not have any disk
bubbling in its compactification, since a non-constant disk has Maslov
index $\geq 2$. As a consequence, the signed count of $J_0$ and
$J_1$-holomorphic disks is equal.  In other words, for a monotone pair
$(X,L)$ for any $J$, $b=0$ is a solution of the Maurer-Cartan
equation.

The disk potential in \eqref{eq:wnaive} is preserved by the
neck-stretching for a monotone pair $(X,L)$, because the functors
$\F$, $\G$ in the homotopy equivalence in \eqref{eq:fghom} are not
curved, that is, $\F^0=\G^0=0$. Indeed, $\F^0$, $\G^0$ are given by
the count of disks of Maslov index $0$ that are pseudoholomorphic with
respect to some $J_t$ in the family $\{J_t\}_{t \in [0,\infty]}$ of
neck-stretched almost complex structures, and there are no such disks
by monotonicity.  In Section \ref{sec:flag} we work with flag
varieties, which are monotone, and we compute the disk potential of a
monotone Lagrangian torus.

In all the curve counting examples that we consider, the problem is
reduced to curve counting in toric manifolds with a toric
Lagrangian.  For the convenience of the reader, we recall the
classification result for holomorphic disks to a toric manifold $X$
with boundary on a toric Lagrangian $L$ from Cho-Oh
\cite{chooh:toric}.  For the standard complex structure on the toric
manifold, such disks are regular by Cho-Oh \cite{chooh:toric}.  We view
the toric manifold as a geometric invariant theory (git) quotient of a
vector space $\C^N$ where each hyperplane $\{z_i=0\}$ projects to a
torus-invariant divisor in $X$.  The Lagrangian torus $L \subset X$ is
a quotient of the standard torus $\{|z_i|=1, 1 \leq i \leq N\}$ in
$\C^N$. Then the holomorphic disks in $X$ bounding $L$ are projections
of Blaschke disks in $\C^N$; each such
disk is a product, and is of the form
  \begin{equation} \label{eq:blaschke}
 u: \D \to \C^N, \quad z \mapsto
    \left( \zeta_i \prod_{j=1}^{d_i} \frac{ z - a_{i,j}}{1 - z
        \ol{a_{i,j}}} \right)_{i = 1,\dots,N} .\end{equation}
  for some constants
  $a_{i,j}, \zeta_i \in \C, |a_{i,j}| \leq 1, |\zeta_i|=1$.  The
  moduli space of Blaschke disks of Maslov index two whose boundary
  passes through a fixed point on the Lagrangian $L$ has dimension
  zero. There is one such disk intersecting the $i$-th torus-invariant
  divisor of $X$, corresponding to the index two Blaschke disk in
  $\C^N$ with single non-vanishing degree $d_i$ equal to $1$.

  Blaschke disks of Maslov index two account for all disks of Maslov
  index two in the case when $X$ is Fano (that is, all torus-invariant
  spheres have positive Chern class). 
  This includes the case when $X$ is monotone; and there are no disks
  of Maslov index less than two.  As a result the Fukaya algebras of
  toric Lagrangians in Fano toric manifolds are weakly unobstructed.  In
  general, one also needs to consider disk classes that are sums of
  Blaschke disk classes and sphere classes.

  In the first two sections of this Chapter, we consider semi-Fano toric surfaces.
  \begin{definition}{\rm(Semi-Fano toric)}\label{def:semiFano}
     \index{Semi-Fano toric surface}
    A toric manifold $X$ is \em{semi-Fano} if for any torus invariant
    sphere of positive symplectic area $S \subset X$, the first Chern
    number $c_1(TX|S)$ is non-negative.
  \end{definition}
  \noindent Toric resolutions of
  $A_n$-singularities $\C^2/\Z_{n+1}$ are semi-Fano toric surfaces,
  since the spheres added by the resolution have Chern number $0$.

  In a semi-Fano toric surface, the standard torus-invariant almost
  complex structure is not regular.  We show that for any regular almost
  complex structure arbitrarily close to the torus-invariant complex
  structure, the Fukaya algebra of a toric Lagrangian is weakly unobstructed,
  and the disk count is independent of the exact choice of the almost
  complex structure.  We prove a limited result (Proposition
  \ref{prop:semiFano-break}) describing multiple-cuts for which
  neck-stretching does not alter the disk count.  For an almost
  complex structure close to the standard one, the sum of an index two
  Blaschke disk class and a sphere class of Chern number zero has
  Maslov index two. The sphere class may be replaced by its multiples
  without affecting the expected dimension of the moduli space, so it
  is not clear at the outset which classes contribute to the disk
  potential.

  \section{Counting disks in the second Hirzebruch surface}
\label{sec:a1}
In this Section, we compute the potential of a toric Lagrangian in the
second Hirzebruch surface $F_2$ using a cut that splits up $F_2$ into
Fano surfaces.  Here, $F_2$ is the toric symplectic manifold with
moment map $\Phi : F_2 \to \R^2$ whose image is the polytope
\begin{equation}
  \label{eq:f2mom}
  \Delta_{F_2}:=\{(x,y) \in \R^2: -1 \leq y \leq 1, -1-t+y \leq x \leq 
  1+t-y\}.  
\end{equation}
for a fixed $t>0$. The neighborhood of the divisor $\Phinv(\{y=1\})$
in the second Hirzebruch surface is the toric smoothing of an
$A_1$-singularity, as in Figure \ref{fig:Ad-sing}.  The space $F_2$ is
a semi-Fano toric surface and the sphere $\Phinv(\{y=1\})$ has Chern
number $0$.  This example is a warm-up for the calculation of the disk
potential of the cubic surface which contains smoothings of
$A_2$-singularities.
\begin{proposition}\label{prop:A1}
  {\rm(Disk potential of $F_2$)} Let $X:=F_2$ be the second Hirzebruch
  surface with a Lagrangian
  $L$.  Let $\XX_\PP$ be the broken manifold with a single cut $\PP$
  that splits $X$ into two copies of $\Bl_{pt}\P^2$ as in Figure
  \ref{fig:h2-cut}.  In both $(X,L)$ and $(\XX_\PP, L)$, the
  Lagrangian is weakly unobstructed and the naive disk potential of $L$ (see
  \eqref{eq:wnaive}) is (taking $q=1$)
  \[ W(y_1,y_2) = y_2 y_1^{-1} + y_2 y_1 + 2 y_2 + y_2^{-1}.  \] 
\end{proposition}

\begin{figure}[h]
    \centering \scalebox{.8}{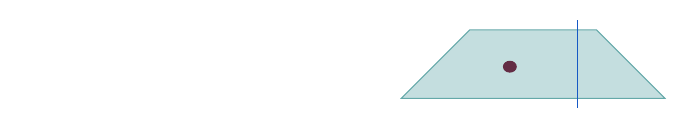}
\caption{A cut $\PP$ on the second Hirzebruch surface.}
\label{fig:h2-cut}
\end{figure}

\begin{figure}[h]
    \centering \scalebox{.8}{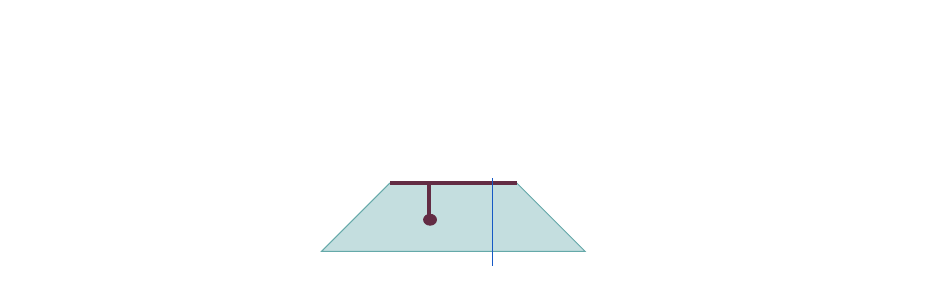}
\caption{Broken disks with Maslov index $2$ on the second Hirzebruch surface, along with their contributions to the potential.}
\label{fig:h2-disks}
\end{figure}

\begin{proof}
  [Prelude to the proof of Proposition \ref{prop:A1}] We describe the
  perturbation used for counting broken spheres and broken disks on
  the multiply cut cubic surface.  The perturbation depends both on
  the domain curve and the tropical graph.  Such perturbations, called
  \em{split perturbations}, are defined rigorously in Section
  \ref{sec:decouple}. In Section \ref{sec:decouple} we also show that
  split perturbations yield the same Gromov-Witten invariants and
  \ainfty homotopy equivalent Fukaya algebras as coherent domain
  dependent almost complex structures.  We recall that a
  domain-dependent perturbation has an underlying \em{background
    almost complex structure}
\label{re:balmost} 
(see Definition \ref{def:domdep} \eqref{part:domdepJ}) which is
perturbed in a compact subset of the complement of the nodal points in
the domain. In a split perturbation, the background almost complex
structure is allowed to vary across components $C_v$,
$v \in \Ver(\Gamma)$ of the domain.  For a tropical graph $\Gamma$ we
denote the background almost complex structures by
\[J_v^b, \quad v \in \Ver(\Gamma),\]
where $J_v^b$ is a compatible almost complex structure on $\ol \XX_{P(v)}$ 
that is cylindrical in the neighborhoods of relative divisors.
For our example of the second Hirzebruch surface, in cases where $P(v)$ is top-dimensional (that is, $P(v)=P_\pm$), $J_v^b$ is defined as 
\begin{equation}
  \label{eq:jvbdef}
  J_v^b:=\phi_v^*J_{\on{std}},   
\end{equation}
where $J_{\on{std}}$ is the standard almost complex structure on the toric
manifold $\ol \XX_{P(v)}$; and
$\phi_v: \ol \XX_{P(v)} \to \ol \XX_{P(v)}$ is a Hamiltonian
diffeomorphism that is supported away from the relative divisors.  If
$P(v)$ is not top-dimensional, define $J_v^b:=J_{\on{std}}$.  Observe
that the tropical graph provides a necessary condition for a map
component to be homologous to a multiple of a $(-1)$-curve, such as
$E'$, $E''$, in Figure \ref{fig:h2-cut}.  For example, a map component
$u_v$ with $P(v)=P_+$ (see Figure \ref{fig:27polys} for notation) is
homologous to $kE'$ only if the sum of multiplicities of the edges
incident on $v$ is $k$.  We call a map component $u_v$ \em{
  {univalent}}, if in the tropical graph there is at most one edge
incident on $v$, and the multiplicity of the edge is $1$.  Choose
perturbations that
\begin{enumerate}
\item \label{item:const-disk}
  {\rm(Constant on disks)}
  vanish on disk components $C_v$, $v \in \Ver_\white(\Gamma)$, and 
\item \label{item:cosc}
  {\rm(Constant on univalent components)} also vanish on components where the map is univalent.
\end{enumerate}
%
%
That is, on such components $C_v$ the perturbed almost complex
structure is equal to the background $J_v^b$. On other components,
where the map may be a multiple cover of a $(-1)$-sphere, the
background $J_v^b$ is perturbed in a domain-dependent way. The proof
of the disk count is carried out below.
\end{proof}

\begin{notation}{\rm(Pseudoholomorphic $(-1)$-spheres)}
  Suppose $\ol \XX_P$ is the cut space corresponding to one of the
  top-dimensional polytopes $P \in \PP$. (In the current case
  $\ol \XX_P$ is either $\XX_+$ or $\XX_-$.)  Suppose
  $E \subset \ol \XX_P$ is a $(-1)$-sphere that is
  $J_{\on{std}}$-holomorphic.  For any tamed almost complex structure $J$ on
  $\ol \XX_P$, there is a $J$-holomorphic sphere homologous to $E$
  which we denote by
  \begin{equation}
    \label{eq:ejsphere}
    E_J \subset \ol \XX_P.
  \end{equation}
  Observe that in this notation $E$ is the same as $E_{J_{\on{std}}}$. 
\end{notation}

\begin{remark}\label{rem:nominus2}
  {\rm(There are no broken $(-2)$-spheres)} A split perturbation as
  above ensures that, with respect to the background almost complex
  structure, there are no pseudoholomorphic spheres whose
  self-intersection number is $-2$ by the following reason: We need to rule
  out spheres that are homologous to the torus-invariant divisor $E$
  whose broken version has components $E' \subset \ol \XX_+$,
  $E'' \subset \ol \XX_-$.  For any tropical graph $\Gamma$, and
  $e =(v_0,v_1) \in \Edge(\Gamma)$, with $P(v_0)=P_+$, $P(v_1)=P_-$,
  for generic background almost complex structures $J_{v_0}^b$,
  $J_{v_1}^b$ the spheres $E'_{J_{v_0}}$ and $E''_{J_{v_1}}$ (using
  notation from \eqref{eq:ejsphere}) intersect $Y$ at different
  points. Therefore there is no broken pseudoholomorphic sphere
  homologous to $E' \cup E''$ for the background almost complex
  structures.  We may assume the domain-dependent perturbations of
  $J_{v_0}^b$, $J_{v_1}^b$ are small enough that the evaluations
  $\ev_{w_e^-}(u_{v_0})$, $\ev_{w_e^+}(u_{v_1}) \in Y$ lie in disjoint
  neighborhoods of the points $E'_{J_{v_0}} \cap Y$ and
  $E''_{J_{v_1}} \cap Y$, thereby ruling out perturbed $(-2)$-spheres
  also.
\end{remark}

\begin{proof}[Proof of Proposition \ref{prop:A1}]
  The cut $\PP$ satisfies the hypothesis of Proposition
  \ref{prop:semiFano-break} and therefore, the potential of $L$ in the
  broken manifold $\XX_\PP$ is the same as the potential in the
  unbroken manifold $X$. The set of disk classes in $(X,L)$ that have
  Maslov index $2$ is
  \[\delta_E, \quad \delta_{D_i}, \quad i=1,\dots,3, \quad \delta_{E}+ k[E], \enspace k \geq 1.  \]
  Indeed, any disk class in $(X,L)$ is a sum
  $\sum_D \delta_D + \sum_{D'} [D']$, where $D$, $D'$ range over
  torus-invariant divisors of $X$, the first summation is non-empty;
  and the Maslov index of such a disk class is
  $\sum_D 2 + \sum_{D'} c_1(D')$.  We will work in the broken manifold
  $\XX_\PP$ to show that these classes are the only relative homology
  classes that contain a regular broken holomorphic disk. The classes
  $\delta_E$, $\delta_{D_1}$, $\delta_{D_2}$ are each represented by a
  perturbation of a Blaschke disk in $\ol \XX_+$, see
  \eqref{eq:blaschke}. The class $\delta_{D_3}$ is represented by a
  broken disk as shown in Figure \ref{fig:h2-disks}.  Next, we
  consider a broken disk $u$ in $\XX_\PP$ whose gluing has homology
  class $\delta_{E} + k[E]$ for some $k \geq 1$.  We will show that
  $k=1$ is the only possibility by the following steps.

  \vskip .1in \noindent \textsc{ Step 1}: \em{The disk component
    $u_+ : C_{v_+} \to \ol \XX_+$ of $u$ has homology class $\delta_{E} + [E']$, and has a simple intersection with the relative divisor $Y$ at a node $w$.}
  The possibility that $[u_+]=\delta_{E} + m[E']$, $m >1$, is ruled
  out as follows: The condition
  \hyperref[item:const-disk]{(Constant on disks)}
   implies that $u_+$ is holomorphic
  with respect to a domain-independent almost complex structure
  $J_{v_+}$ on $\ol \XX_+$. Using notation in \eqref{eq:ejsphere} we
  denote by $E_{J_{v_+}}$ the $J_{v_+}$-holomorphic sphere in
  $\ol \XX_+$ that is homologous to the $(-1)$-sphere $E'$. If $m>1$,
  the intersection number $[u_+].[E_{J_{v_+}}]$ is negative.  This is
  not possible since $[u_+]$ can only have positive intersections with
  $E_{J_{v_+}}$.

  \vskip .1in \noindent  \textsc{ Step 2}: \em{The component $u_- : C_{v_-} \to \ol \XX_-$  that is incident on the node $w$ is homologous to $E''$.}  
  Since $[u]=\delta_{E}+ k[E]$ and $[u_+]=\delta_{E}+ [E']$, we have
  $[u_-]=m[E'']$ for some $m \geq 1$.  By the matching condition at
  the node $w$, $u_-$ has a simple intersection with $Y$ at the node
  $w$. If $m>1$, $u_-$ has another intersection with $Y$ at a node
  $w' \neq w$.  If $u' : C_{v'} \to \ol \XX_+$ is the other component
  incident at the node $w'$ then $[u']$ is a multiple of
  $[E']$. Indeed, that this is so is forced by homology requirements
  since $[u]=\delta_{E}+ k[E]$.  Since both $[u_-]$ and $[u']$ are
  multiples of $(-1)$-spheres the matching condition at the node $w'$
  is not satisfied for our choice of split perturbation.  A generic
  choice of background almost complex structures $J_{v_-}$, $J_{v'}$
  ensures that the $(-1)$-spheres $E_{J_{v_-}} \subset \ol \XX_-$,
  $E_{J_{v'}} \subset \ol \XX_+$ do not intersect $Y$ at the same
  point; \footnote{This is a point in the proof which requires the perturbation to depend on the tropical graph, since we need $J_{v_-}$, $J_{v'}$ to be distinct.}
  and for small enough domain-dependent perturbations the
  points $\ev_{w'}(u_-) \in Y$ and $\ev_{w'}(u') \in Y$ are
  distinct. Therefore we conclude that $k=1$.

  The five disks of Maslov index $2$ yield the potential claimed by
  the Proposition.
\end{proof}

\begin{example}
  \label{ex:h2}
  We give an example where a multiple cut changes the potential of a
  Lagrangian brane.  Consider the second Hirzebruch surface $X:=F_2$
  with two single cuts as in Figure \ref{fig:h2} and a toric
  Lagrangian $L$. The potential for the broken manifold is
  \[ W_2(y_1,y_2) = y_2 y_1^{-1} + y_2 + y_2^{-1}. \]
  It does not have disks in two out of the five classes of disks of
  Maslov index two, namely $\delta_{D_3}$ and $\delta_{E} +[E]$.  For
  example, if $\XX_{P_2}$ were to contain a regular broken disk $u$ in
  the class $\delta_{D_1} +[D_1]$, the homology classes of its pieces
  would be as shown in Figure \ref{fig:h2}, and $u$ would satisfy a
  matching condition along the relative divisor $Y_2$ between two
  curves of self-intersection number $-1$.  Indeed, for a generic
  perturbation the two $(-1)$-curves will not be incident at the same
  point on $Y_2$.  By the same argument, broken disks of class
  $\delta_{D_4}$ are also ruled out in the broken manifold
  $\XX_{P_2}$.  The discrepancy between the disk potentials of $(X,L)$
  and $(\XX_\PP,L)$ is explained by the fact that there is a
  $J_t$-holomorphic disk of Maslov index zero for some $J_t$ in the
  family of neck-stretched almost complex structures.  See Remark
  \ref{rem:ctrh2}.
\end{example}
  
  \begin{figure}[h]
  \centering \scalebox{.8}{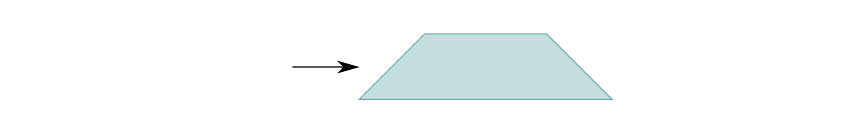}
  \caption{Moving a cut in the second Hirzebruch surface changes the disk count.}
  \label{fig:h2}
\end{figure}

\section{Counting disks in cubic surfaces}\label{sec:cubic-intro}
We use broken maps to give a proof of the classic result that there
are twenty seven lines on the cubic surface, originally proved by
Salmon and Cayley, see Mumford \cite[Section 8D]{mumford}.  We also
compute its ``open'' analog, namely the disk potential of a toric
Lagrangian in the cubic surface.  The number of disks is independent
of perturbations because the Lagrangian is monotone.  We show that
there are twenty one disks contributing to the potential as
conjectured by Sheridan \cite[Appendix B]{sh:hmsfano}. This example
also appears in Pascaleff-Tonkonog \cite{pasc:wall} and can also be
addressed using the method of Chan-Lau-Cho-Tseng \cite{grossfib}.  The
paper Bardwell-Evans-Hong-Lin \cite{bchl} describes the disks in terms
of scattering diagrams.

The cubic surface is a monotone almost toric manifold equipped with a
monotone Lagrangian torus.  Let
\[(X_0,\om_0,J_0) \subset \P^3\]
be a non-singular cubic surface with the Fubini-Study symplectic
form.  Since any cubic surface is biholomorphic to the del Pezzo
surface $\Bl_6\P^2$, we view $(X_0,\om_0)$ as $\Bl_6 \P^2$ with a
monotone symplectic form, which is $(\P^2,\om_{FS})$ blown up at six
points with a blow-up parameter of $1$. McDuff \cite[Corollary
1.5]{mcduff:sd} shows that up to symplectic isotopy, there is exactly
one way of performing such a blow-up.  By Vianna \cite{vianna:dp},
$(X_0,\om_0)$ has an almost toric structure whose base diagram
\begin{equation}
  \label{eq:x0toric}
\Delta_0:=\on{hull}\{ (-1,-1), (0,1), (1,0) \} 
\end{equation}
is given in Figure \ref{fig:defcubic}. Using a multiple cut, we count
the number of holomorphic spheres with Chern number $1$, which is the
same as the number of $(-1)$-curves in $\Bl_6 \P^2$, and which is the
same as the number of lines in the cubic surface.  Using the same
multiple cut, we will also compute the number of rigid disks
(Proposition \ref{twentyoneex}) bounding the monotone Lagrangian
\begin{equation}
  \label{eq:L0def}
  L_0:=\Phinv(0,0).  
\end{equation}

\begin{figure}[h]
  \centering \scalebox{.8}{
\begingroup%
  \makeatletter%
  \providecommand\color[2][]{%
    \errmessage{(Inkscape) Color is used for the text in Inkscape, but the package 'color.sty' is not loaded}%
    \renewcommand\color[2][]{}%
  }%
  \providecommand\transparent[1]{%
    \errmessage{(Inkscape) Transparency is used (non-zero) for the text in Inkscape, but the package 'transparent.sty' is not loaded}%
    \renewcommand\transparent[1]{}%
  }%
  \providecommand\rotatebox[2]{#2}%
  \newcommand*\fsize{\dimexpr\f@size pt\relax}%
  \newcommand*\lineheight[1]{\fontsize{\fsize}{#1\fsize}\selectfont}%
  \ifx\svgwidth\undefined%
    \setlength{\unitlength}{314.40896582bp}%
    \ifx\svgscale\undefined%
      \relax%
    \else%
      \setlength{\unitlength}{\unitlength * \real{\svgscale}}%
    \fi%
  \else%
    \setlength{\unitlength}{\svgwidth}%
  \fi%
  \global\let\svgwidth\undefined%
  \global\let\svgscale\undefined%
  \makeatother%
  \begin{picture}(1,0.43114605)%
    \lineheight{1}%
    \setlength\tabcolsep{0pt}%
    \put(0,0){\includegraphics[width=\unitlength,page=1]{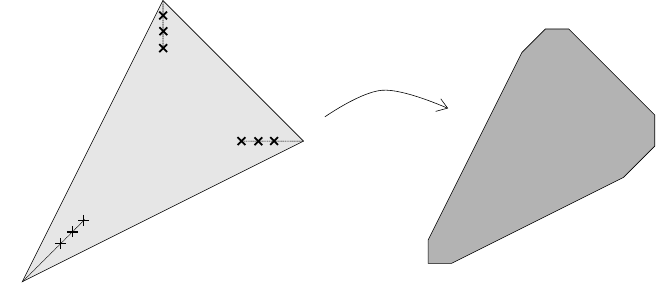}}%
    \put(0.27072993,0.21533652){\color[rgb]{0,0,0}\makebox(0,0)[lt]{\lineheight{1.25}\smash{\begin{tabular}[t]{l}$L$\end{tabular}}}}%
    \put(-0.00109606,0.25568795){\color[rgb]{0,0,0}\makebox(0,0)[lt]{\lineheight{1.25}\smash{\begin{tabular}[t]{l}$(X_0,\om_0)$\end{tabular}}}}%
    \put(0.85346051,0.06863277){\color[rgb]{0,0,0}\makebox(0,0)[lt]{\lineheight{1.25}\smash{\begin{tabular}[t]{l}$(X_\eps,\om_\eps)$\end{tabular}}}}%
    \put(0,0){\includegraphics[width=\unitlength,page=2]{defcubic.pdf}}%
  \end{picture}%
\endgroup%
}
    \caption{Deformation of the monotone symplectic form to
a toric symplectic form on the cubic surface.}
\label{fig:defcubic}
\end{figure}

To perform curve counts on the cubic surface, we symplectically deform
it to a semi-Fano toric surface
(Definition \ref{def:semiFano}).
In the following result (Proposition
\ref{prop:sympdef}) we show that the cubic surface $(X_0,\om_0)$ is
symplectically deformation equivalent to a toric manifold
$(X_\eps,\om_\eps)$ whose moment polytope is
\begin{multline}\label{eq:xepstoric}
    \Delta_\eps:=\{(x,y) \in \R^2 : 2x-y+1, -x+2y+1, -x-y+1 \geq \eps,\\
    -1 \leq x,y,y-x \leq 1\},
  \end{multline}
  for any small $\eps>0$.  Note that $(X_\eps,\om_\eps)$ is a
  semi-Fano toric surface.
  
  We observe that $(X_\eps,\om_\eps)$ is obtained by blowing up the
  toric monotone del Pezzo surface $M:=\Bl_3\P^2$ at three points with blow-up
  parameter $1-\eps$, see Figure \ref{fig:defcubic4}. Here, the moment
  polytope of $M$ is $\Delta_M:=\{-1 \leq x,y,y-x \leq 1\}$ and the
  blow-up is performed at the torus fixed points corresponding to
  $(1,1),(0,-1),(-1,0) \in \Delta_M$.  The torus-invariant divisors of
  $X_\eps$ corresponding to the facets
  \begin{equation}
    \label{eq:shortfacets}
    y=\pm 1, \quad x=\pm 1, \quad x-y=\pm 1
  \end{equation}
  of $\Delta_\eps$ are called \em{short divisors}, and those corresponding to
  \begin{equation}
    \label{eq:longfacets}
    x+y=1+\eps, \quad y-1-\eps=2x, \quad x-1-\eps=2y.
  \end{equation}
  are called \em{long divisors}. The short resp. long divisors are
  denoted by $E_i$ resp. $D_i$ in Figure \ref{fig:twentyseven}.  Note
  that the short facets of $\Delta_\eps$ are collapsed to points in
  the moment polytope $\Delta_0$ of the almost toric $X_0$. However
  this is not a proof of the fact that $X_0$ is symplectic deformation
  equivalent to $X_\eps$, $\eps>0$, because in Figure
  \ref{fig:defcubic4} the blow-up parameter 
  cannot 
  be equal to
  $1$.  So in Proposition \ref{prop:sympdef} we take a more indirect
  approach to show the symplectic deformation equivalence.

  \begin{figure}[h]
    \centering \scalebox{.8}{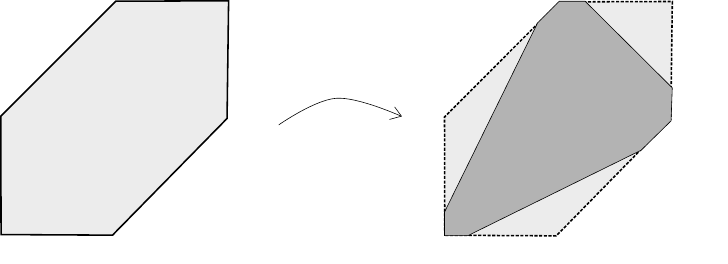}
    \caption{From three to six toric blowups.}
\label{fig:defcubic4}
\end{figure} 
  
\begin{proposition}\label{prop:sympdef}{\rm(Symplectic deformation equivalence)}
  Let $(X_0,\om_0)$ be the cubic surface with an almost toric
  structure \eqref{eq:x0toric}, and for any $\eps>0$, let
  $(X_\eps,\om_\eps)$ be the toric manifold \eqref{eq:xepstoric}.
  There is a continuous family of diffeomorphisms
  $\phi_\eps : X_0 \to X_\eps$ with $\phi_0=\Id_{X_0}$, so that
  $\phi_\eps^*\om_\eps$ is a continuous family of symplectic forms on
  $X_0$ and $L_\eps:=\phi_\eps(L_0)$ is a toric Lagrangian for all
  $\eps>0$.
 \end{proposition}

 \begin{proof}
   To prove that $X_0$ and $X_\eps$ are symplectic deformation
   equivalent, it is enough to show the same for the manifolds with
   moment polytopes in Figure \ref{fig:defcubic2}. The first space in
   Figure \ref{fig:defcubic2} is an almost toric manifold whose base
   diagram is a polytope $\Delta_0':=\{ -x+2y+1, -x-y+1 \geq 0\}$ and
   three focus-focus singularities along $y=0$. This space is the
   smoothing of the $A_2$ singularity, see \cite[Example
   7.6]{evans:almosttoric}. The right side $Y_\eps$ in Figure
   \ref{fig:defcubic2} is a toric manifold with moment polytope
 \[\Delta_\eps':=\{-x+2y, -x-y \geq 0, x,x-y \leq -\eps\}. \]
 \begin{figure}[h]
  \centering \scalebox{.8}{
\begingroup%
  \makeatletter%
  \providecommand\color[2][]{%
    \errmessage{(Inkscape) Color is used for the text in Inkscape, but the package 'color.sty' is not loaded}%
    \renewcommand\color[2][]{}%
  }%
  \providecommand\transparent[1]{%
    \errmessage{(Inkscape) Transparency is used (non-zero) for the text in Inkscape, but the package 'transparent.sty' is not loaded}%
    \renewcommand\transparent[1]{}%
  }%
  \providecommand\rotatebox[2]{#2}%
  \newcommand*\fsize{\dimexpr\f@size pt\relax}%
  \newcommand*\lineheight[1]{\fontsize{\fsize}{#1\fsize}\selectfont}%
  \ifx\svgwidth\undefined%
    \setlength{\unitlength}{193.41480142bp}%
    \ifx\svgscale\undefined%
      \relax%
    \else%
      \setlength{\unitlength}{\unitlength * \real{\svgscale}}%
    \fi%
  \else%
    \setlength{\unitlength}{\svgwidth}%
  \fi%
  \global\let\svgwidth\undefined%
  \global\let\svgscale\undefined%
  \makeatother%
  \begin{picture}(1,0.44842532)%
    \lineheight{1}%
    \setlength\tabcolsep{0pt}%
    \put(0,0){\includegraphics[width=\unitlength,page=1]{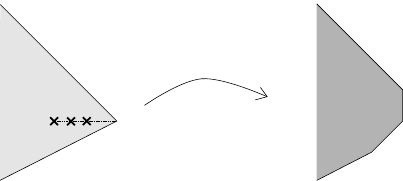}}%
    \put(0.45667177,0.43200503){\color[rgb]{0,0,0}\makebox(0,0)[lt]{\lineheight{1.25}\smash{\begin{tabular}[t]{l}Deform  the \\symplectic \\form\end{tabular}}}}%
  \end{picture}%
\endgroup%
}
  \caption{A deformation of the symplectic form}
\label{fig:defcubic2}
\end{figure}
\begin{figure}[h]
  \centering \scalebox{.8}{
\begingroup%
  \makeatletter%
  \providecommand\color[2][]{%
    \errmessage{(Inkscape) Color is used for the text in Inkscape, but the package 'color.sty' is not loaded}%
    \renewcommand\color[2][]{}%
  }%
  \providecommand\transparent[1]{%
    \errmessage{(Inkscape) Transparency is used (non-zero) for the text in Inkscape, but the package 'transparent.sty' is not loaded}%
    \renewcommand\transparent[1]{}%
  }%
  \providecommand\rotatebox[2]{#2}%
  \newcommand*\fsize{\dimexpr\f@size pt\relax}%
  \newcommand*\lineheight[1]{\fontsize{\fsize}{#1\fsize}\selectfont}%
  \ifx\svgwidth\undefined%
    \setlength{\unitlength}{409.79683789bp}%
    \ifx\svgscale\undefined%
      \relax%
    \else%
      \setlength{\unitlength}{\unitlength * \real{\svgscale}}%
    \fi%
  \else%
    \setlength{\unitlength}{\svgwidth}%
  \fi%
  \global\let\svgwidth\undefined%
  \global\let\svgscale\undefined%
  \makeatother%
  \begin{picture}(1,0.2989172)%
    \lineheight{1}%
    \setlength\tabcolsep{0pt}%
    \put(0,0){\includegraphics[width=\unitlength,page=1]{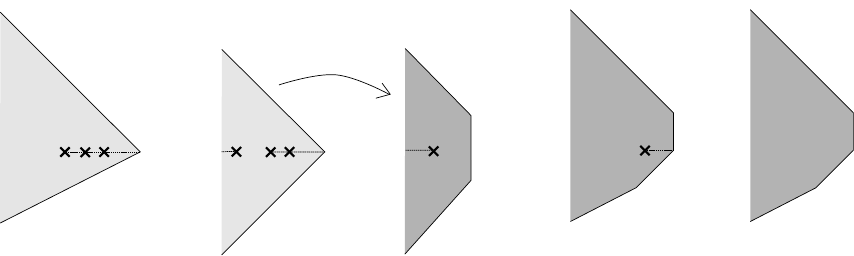}}%
    \put(0.34115309,0.28210361){\color[rgb]{0,0,0}\makebox(0,0)[lt]{\lineheight{1.25}\smash{\begin{tabular}[t]{l}Deform  the \\symplectic \\form\end{tabular}}}}%
    \put(0,0){\includegraphics[width=\unitlength,page=2]{defcubic3.pdf}}%
    \put(0.53560027,0.28574708){\color[rgb]{0,0,0}\makebox(0,0)[lt]{\lineheight{1.25}\smash{\begin{tabular}[t]{l}Transferring\\the cut\end{tabular}}}}%
    \put(0,0){\includegraphics[width=\unitlength,page=3]{defcubic3.pdf}}%
    \put(0.73564328,0.289678){\color[rgb]{0,0,0}\makebox(0,0)[lt]{\lineheight{1.25}\smash{\begin{tabular}[t]{l}Nodal trade\end{tabular}}}}%
    \put(0,0){\includegraphics[width=\unitlength,page=4]{defcubic3.pdf}}%
    \put(0.09924655,0.28539201){\color[rgb]{0,0,0}\makebox(0,0)[lt]{\lineheight{1.25}\smash{\begin{tabular}[t]{l}Transferring\\the cut\end{tabular}}}}%
  \end{picture}%
\endgroup%
}
\caption{Symplectic deformation from an almost toric to a toric manifold.}
\label{fig:defcubic3}
\end{figure} 
The deformation is carried out by the series of steps in Figure
\ref{fig:defcubic3}. All the steps except the second one leave the
symplectic form unchanged.  In the second step of ``deforming the
symplectic form'', the resolution of an $A_1$-singularity is replaced
by the total space of $\mO_{\P^1}(-2)$.
To see that such a deformation exists
recall that the $A_1$-singularity can be smoothed to $T^*\P^1$
\cite[p138]{evans:almosttoric}.  The symplectic form on the total
space $\on{Tot}(\mO_{\P^1}(-2))$ is given by
$\om_{T^*\P^1} + \eps \pi^*\om_{\P^1}$ where
$\pi : \on{Tot}(\mO_{\P^1}(-2)) \to \P^1$ is the projection map and
$\om_{P^1}$ is the Fubini-Study form on $\P^1$.  We leave the details
of constructing the map $\phi_\eps$ to the reader.
 \end{proof}

 As a second step in the curve count, we fix $\eps>0$ and apply a
 multiple cut on
 \begin{equation}
   \label{eq:Xeps-def}
  X:=(X_\eps,\om_\eps) 
 \end{equation}
 that splits it into Fano pieces as shown in Figure
 \ref{fig:cubiccut}.  The resulting broken manifold is called $\XX$.
 We count broken maps consisting of disks and spheres in $\XX$, and
 prove that the count is the same as the unbroken count in the cubic
 surface. Under the multiple cut in Figure \ref{fig:cubiccut}, the
 homology class of any short sphere degenerates to the homology class
 of a union of two spheres in the pieces $X_{P_i}$, each of which has
 self-intersection number equal to $-1$.

\begin{figure}[ht]
  \centering \scalebox{.8}{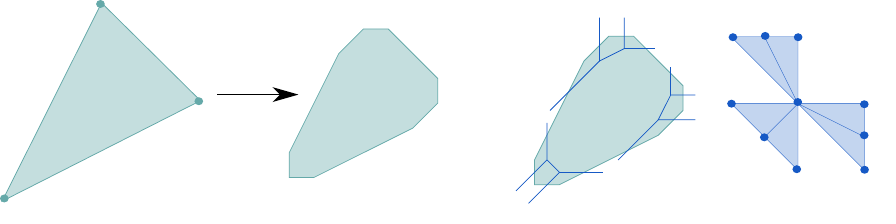}
\caption{Multiple cut on the cubic surface.}
\label{fig:cubiccut}
\end{figure}

\subsection{Spheres in the cubic surface}
\begin{figure}[ht]
  \centering \scalebox{.8}{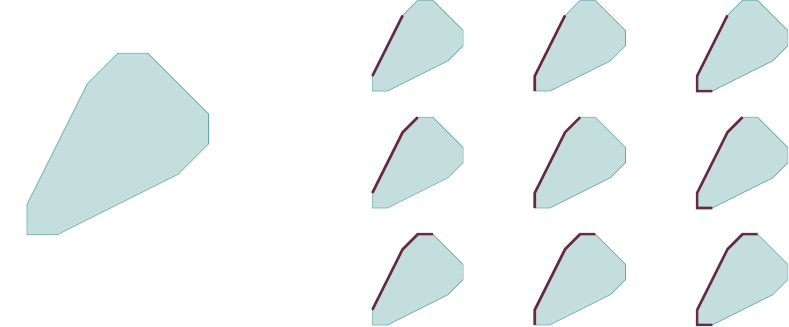}
\caption{Nine of the twenty-seven lines on a (deformation of) a cubic
  surface, in the homology classes $[D_1]+\sum_{i=1,2,5,6} n_i[E_i]$.}
\label{fig:twentyseven}
\end{figure} 

In the following result, we show that a holomorphic sphere in
$X:=X_\eps$ with Chern number one consists of one long divisor and a
non-self crossing path of short divisors. There is one configuration
corresponding to each of the pictures in Figure \ref{fig:twentyseven}.
By symmetry there are twenty-seven such configurations in total.

\begin{proposition}\label{prop:spherecount}
  {\rm(Twenty seven lines in the cubic)} 
  Any sphere of index zero in the manifold $X:=X_\eps$ is of
  the form shown in Figure \ref{fig:twentyseven}, and each such broken
  sphere has multiplicity one. Therefore, a cubic has $27$ lines.
\end{proposition}

\begin{remark}\label{rem:gw1}{\rm(Disks versus spheres)} 
  We point out that the results in this book are about Fukaya
  algebras (defined using counts of disks), whereas Proposition
  \ref{prop:spherecount} uses analogous results in Gromov-Witten
  theory (defined using counts of spheres).  Broken Gromov-Witten
  invariants are well-defined in cases when markings are constrained
  to lie on homology classes that are represented by cycles that are
  contained in a single top-dimensional cut space. In the case of
  lines in the cubic surface there are no markings.  Whenever the
  broken Gromov-Witten invariants are well-defined, they are equal to
  the ordinary GW invariants on the unbroken manifold. The proofs are
  easier than the \ainfty counterpart because the coherence condition
  on the perturbation datum is weaker.  We elaborate on this difference
  in Remark \ref{rem:gwpert}.
\end{remark}

\begin{figure}[h]
\centering
\scalebox{.8}{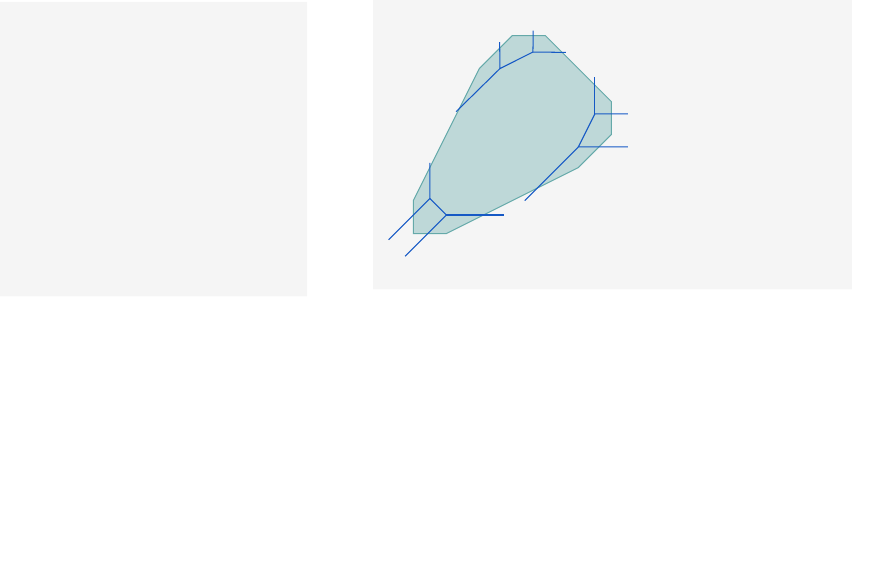} 
  \caption{Top left: Broken versions of some divisors in the cubic surface. Other figures: Broken maps $u^1$, $u^2$, $u^3$ whose gluing is homologous to $[D_1]+k_1[E_1]+k_2[E_2]$ for some $k_1,k_2 \geq 0$.}
  \label{fig:27polys}
\end{figure}

We now prove the result on spheres on the cubic surface.

\begin{proof}[Proof of Proposition \ref{prop:spherecount}]
  We recall from the discussion at the beginning of the section that
  the homology class of a sphere of index zero in $X$ is the sum of a
  long divisor class $[D_i]$ and an arbitrary number of short divisor
  classes.  Consider a broken sphere $u$ such that its gluing has
  homology class
  \[[u_{\glue}]=[D_1] + n_1 [E_1] + n_2 [E_2] + n_5 [E_5] + n_6 [E_6],\]
  where $D_1$, $E_i$ are defined in Figure \ref{fig:twentyseven}.  We
  aim to prove that $u$ is of the form shown in Figure
  \ref{fig:twentyseven}.  We first assume that that $n_5=n_6=0$, and
  so,
  \begin{equation}
    \label{eq:ugluehom}
  [u_{\glue}]=[D_1] + n_1 [E_1] + n_2 [E_2].
  \end{equation}
  We will show that a regular configuration is one of the types
  (a)-(c) in Figure \ref{fig:twentyseven}.  The conclusion for the
  general case ($n_5,n_6 \neq 0$) follows in a similar way. Indeed,
  the arguments given below for the ``$E_1$ -end'' of the long divisor
  $D_1$ can be applied to the ``$E_6$-end'' also.
  
  Given a broken map $u$ with homology class as in \eqref{eq:ugluehom}
  and a tropical graph $\Gamma$, we will show that $n_1,n_2 \leq 1$.
  The proof is carried out in steps, where in every step, we uncover
  the possible maps $u_v$ corresponding to a vertex $v$ of the
  tropical graph.  In the following analysis we refer to Figure
  \ref{fig:27polys} for notation.

   \vskip .1in \noindent  \textsc{ Step 1} ($v_0$): Since the long divisor is a
  summand in the homology class represented by the map $u$, there is a
  component of $u_{v_0}$ of the broken map that maps to $D_1'$ (see
  Figure \ref{fig:27polys} for notations). By the assumption that
  $n_5=n_6=0$, $v_0$ has
  \begin{itemize}
  \item  an edge to $v_{-1}$ with $P(v_{-1})=P_{9}$, and the homology class of the map $u_{v_{-1}}$ is $[D_1''']$; and
    \item an edge to $v_1$ with $P(v_1)=P_1$; 
  \end{itemize}
  and no other edges.

   \vskip .1in \noindent  \textsc{ Step 2} ($v_1$): So far, we have shown that the
  broken map has vertices $v_0$, $v_{-1}$ and further, $v_0$ has an
  edge of multiplicity one to a vertex $v_1$ with $P(v_1)=P_1$. Given
  that the homology class of the gluing of $u$ is given by
  \eqref{eq:ugluehom}, the homology class of $u_{v_1}$ is
  \begin{equation}
    \label{eq:dke}
    [D_1''] + k[E_1'] \quad \text{for some $k \geq 0$.}
  \end{equation}
  We proceed to show that $k=0$ or $1$: Recall that $E'_{1,J_{v_1}}$
  is a $J_{v_1}$-holomorphic sphere in $\ol X_{P_1}$ that is
  homologous to the $(-1)$-sphere $E_1' \subset \ol X_{P_1}$, using
  the notation in \eqref{eq:ejsphere}.  Since the map $u_{v_1}$
  intersects the relative divisor $\ol X_{P_{01}}$, the homology class
  $[u_{v_1}]$ 
  cannot
  be a multiple of the class $[E_1']$.  The
  property
  \hyperref[item:cosc]{(Constant on univalent components)} 
  then implies that the
  domain-dependent almost complex structure for $u_{v_1}$ is equal to
  the constant $J_{v_1}$.  In \eqref{eq:dke} if $k>1$, then
  $[u_{v_1}].[E'_{1,J_{v_1}}] <0$. This implies that the image of
  $u_{v_1}$ is contained in $E'_{1,J_{v_1}}$, contradicting the fact
  that $u_{v_1}$ intersects the relative divisor $\ol X_{P_{01}}$. In case
  $k=0$, the broken map $u$ is $u^{1}$ represented in Figure
  \ref{fig:twentyseven} (a) and Figure \ref{fig:27polys} (a),
  $[u_\glue^1]=[D_1]$ and there are no more vertices in the tropical
  graph.  If $k=1$, $v_1$ has an edge to a vertex $v_2$ with
  $P(v_2)=P_2$.

   \vskip .1in \noindent  \textsc{ Step 3} ($v_2$): So far, we have shown that for a
  broken map $u$ whose gluing has homology class \eqref{eq:ugluehom},
  either $[u_\glue]=[D_1]$ or the tropical graph of $u$ has vertices
  $v_{-1}$, $v_0$, $v_1$ and $v_1$ has an edge of multiplicity $1$ to
  a vertex $v_2$ with $P(v_2)=P_2$.  In the latter case, the homology
  class \eqref{eq:ugluehom} of the glued map $[u_\glue]$ implies that
  the homology of $u_{v_2}$ is
  \[[u_{v_2}]=k_1[E_1''] + k_2 [E_2'],\]
  and since $v_2$ has an edge to $v_1$, $k_1\geq 1$.

  First, consider the case $k_2=0$. If $k_1=1$, then there are no more
  vertices, $u=u^{2}$, and the map is represented in Figure
  \ref{fig:twentyseven} (b) and Figure \ref{fig:27polys} (b).  Now
  suppose $k_2=0$ and $k_1>1$. Then $u_{v_2}$ is a multiple cover of
  the sphere $E''_{1,J_{v_2}}$ (using notation from
  \eqref{eq:ejsphere}).  Besides the edge $e$ to $v_1$, $v_2$ has at
  least one more edge to a vertex, say $\ol v_3$, with
  $P(\ol v_3)=P_1$, see Figure \ref{fig:hypo1}.  Indeed, at the node
  $w_{e'}$ corresponding to the edge $e':=(v_1,v_2)$, the map
  $u_{v_2}$ has an intersection of order $1$ with the relative divisor
  $\ol X_{P_{12}}$.  More edges are required to account for the $k$
  intersections that $u_{v_2}$ has with $\ol X_{P_{12}}$.  The
  homology class of $u_\glue$ in \eqref{eq:ugluehom} implies that the
  homology class of the map $u_{\ol v_3}$ is a multiple of
  $[E_1']$.  For generic background almost complex structures
  $J_{v_2}^b$, $J_{\ol v_3}^b$ and small enough domain-dependent
  perturbations, the evaluations
\[ \ev_{w_e}(u_{v_2}), \quad \ev_{w_e}(u_{\ol v_3}) \in \ol X_{P_{12}} \] 
will not agree; see Remark \ref{rem:nominus2} which explains a similar
result. Therefore, matching conditions
cannot
be simultaneously
satisfied at the node $(v_1,v_2)$ and $(v_1,\ol v_3)$.
This argument rules out the possibility that $k_1>1$ when $k_2=0$.

\begin{figure}[ht]
  \centering \scalebox{.8}{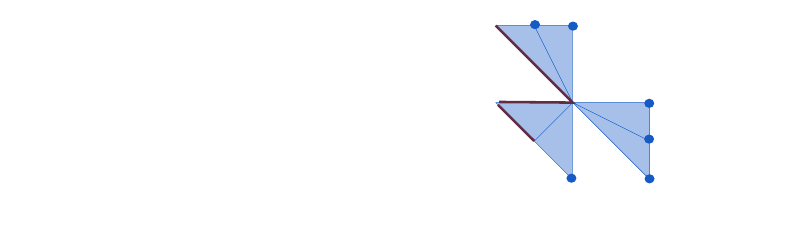}
  \caption{A hypothetical broken map (ruled out in proof of Proposition \ref{prop:spherecount}).}
  \label{fig:hypo1}
\end{figure}  

Next, assuming $k_2>0$, we show that $k_1=k_2$. The argument is
similar to the one used in Step 2.  For example, if $k_2>k_1$, then
$[u_{v_2}].[E_2']<0$.  It follows that the image of $u_{v_2}$ lies in
the $J_{v_2}$-holomorphic sphere homologous to $E_2'$, contradicting
the fact that $k_1 \geq 1$. The case $k_1>k_2$ is ruled out by a
similar argument.

\begin{figure}[h]
  \centering \scalebox{.8}{
\begingroup%
  \makeatletter%
  \providecommand\color[2][]{%
    \errmessage{(Inkscape) Color is used for the text in Inkscape, but the package 'color.sty' is not loaded}%
    \renewcommand\color[2][]{}%
  }%
  \providecommand\transparent[1]{%
    \errmessage{(Inkscape) Transparency is used (non-zero) for the text in Inkscape, but the package 'transparent.sty' is not loaded}%
    \renewcommand\transparent[1]{}%
  }%
  \providecommand\rotatebox[2]{#2}%
  \newcommand*\fsize{\dimexpr\f@size pt\relax}%
  \newcommand*\lineheight[1]{\fontsize{\fsize}{#1\fsize}\selectfont}%
  \ifx\svgwidth\undefined%
    \setlength{\unitlength}{122.53679204bp}%
    \ifx\svgscale\undefined%
      \relax%
    \else%
      \setlength{\unitlength}{\unitlength * \real{\svgscale}}%
    \fi%
  \else%
    \setlength{\unitlength}{\svgwidth}%
  \fi%
  \global\let\svgwidth\undefined%
  \global\let\svgscale\undefined%
  \makeatother%
  \begin{picture}(1,0.87522084)%
    \lineheight{1}%
    \setlength\tabcolsep{0pt}%
    \put(0,0){\includegraphics[width=\unitlength,page=1]{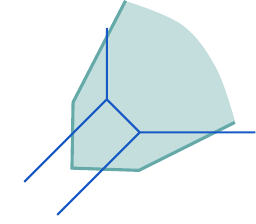}}%
    \put(0.29600085,0.41651948){\color[rgb]{0,0,0}\makebox(0,0)[lt]{\lineheight{1.25}\smash{\begin{tabular}[t]{l}$x$\end{tabular}}}}%
    \put(0.49895848,0.23440855){\color[rgb]{0,0,0}\makebox(0,0)[lt]{\lineheight{1.25}\smash{\begin{tabular}[t]{l}$f(x)$\end{tabular}}}}%
    \put(0,0){\includegraphics[width=\unitlength,page=2]{fdef.pdf}}%
    \put(-0.00337453,0.2389846){\color[rgb]{0.08627451,0.34901961,0.77254902}\makebox(0,0)[lt]{\lineheight{1.25}\smash{\begin{tabular}[t]{l}$P_{12}$\end{tabular}}}}%
    \put(0.2789244,0.0180517){\color[rgb]{0.08627451,0.34901961,0.77254902}\makebox(0,0)[lt]{\lineheight{1.25}\smash{\begin{tabular}[t]{l}$P_{23}$\end{tabular}}}}%
    \put(0.49751722,0.4301393){\color[rgb]{0.08627451,0.34901961,0.77254902}\makebox(0,0)[lt]{\lineheight{1.25}\smash{\begin{tabular}[t]{l}$P_{02}$\end{tabular}}}}%
    \put(0,0){\includegraphics[width=\unitlength,page=3]{fdef.pdf}}%
    \put(0.09405325,0.08867746){\color[rgb]{0,0,0}\makebox(0,0)[lt]{\lineheight{1.25}\smash{\begin{tabular}[t]{l}$P_2$\end{tabular}}}}%
  \end{picture}%
\endgroup%
}
  \caption{The space $\ol X_{P_2} \bs (\ol X_{P_{02}} \cup E''_{1,J_{v_2}}\cup E'_{2,J_{v_2}})$ is a $\P^1$-fibration. The map $f$ maps an end-point of a $\P^1$-fiber to its other end-point.}
  \label{fig:fdef}
\end{figure}  

Next, we rule out the case $k_1=k_2>1$ using the genericity of $J$.
Gluing the $(-1)$-spheres $E''_{1,J_{v_2}}$ and $E'_{2,J_{v_2}}$
produces a $J_{v_2}$-holomorphic sphere with a trivial normal bundle
in $\ol X_{P_2}$.  In this case, there is a fibration
  \[\pi: \ol X_{P_2} \bs (\ol X_{P_{02}} \cup E''_{1,J_{v_2}} \cup E'_{2,J_{v_2}}) \to \C^\times\]
  where the fibers are $J_{v_2}$-holomorphic spheres.  When $k_1=k_2$,
  the map $u_{v_2}$ is a cover of one of the fibers of $\pi$.  We
  define a diffeomorphism
  \[f: \ol X_{P_{12}}\bs \ol X_{P_{012}} \to \ol X_{P_{23}} \bs \ol X_{P_{023}}\]
  by the condition that $x$ and $f(x)$ lie on the same fiber of $\pi$,
  see Figure \ref{fig:fdef}.  As we argued in the previous paragraph,
  the inequality $k_1>1$ implies that $v_2$ has an edge to a vertex,
  say, $\ol v_3$ with $P(\ol v_3)=P_1$ and $u_{\ol v_3}$ is a cover of
  the sphere $E'_{1,J_{\ol v_3}}$ (see Figure \ref{fig:hypo2}).  Since
  $k_2>0$, $v_2$ also has an edge to a vertex, say, $\ol v_4$ with
  $P(\ol v_4)=P_3$ and $u_{\ol v_4}$ is a cover of the sphere
  $E''_{2, J_{\ol v_4}}$.  For generic background almost complex
  structures $J_{\ol v_3}$, $J_{\ol v_4}$,
  \[f(E'_{1,J_{\ol v_3}} \cap \XB_{P_{12}}) \neq E''_{2,J_{\ol v_4}} \cap \XB_{P_{23}}. \]
  The domain-dependent perturbations are small enough that the
  evaluations $\ev_{w_{e''}}(u_{v_2})$,
  $\ev_{w_{e''}}(u_{\ol v_4}) \in \ol X_{P_{23}}$ lie in disjoint
  neighborhoods of $f(E'_{1,J_{\ol v_3}} \cap \ol X_{P_{12}})$,
  $E''_{2,J_{\ol v_4}} \cap \ol X_{P_{23}}$. As a consequence, there is no
  broken map $u$ for which $[u_\glue]=D_1+k(E_1+E_2)$, $k>1$.
\begin{figure}[h]
  \centering \scalebox{.8}{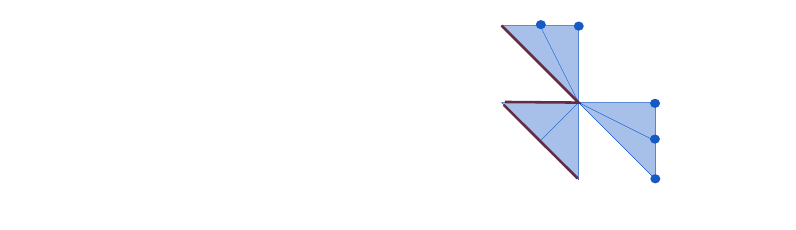}
  \caption{A hypothetical broken map (ruled out in proof of Proposition \ref{prop:spherecount}).}
  \label{fig:hypo2}
\end{figure}  
  
Finally, we are left with the case $k_1=k_2=1$. In this case $u_{v_2}$
is a simple map to a fiber of $\pi$, and $v_2$ has an edge to a vertex
$v_3$ with $P(v_3)=0$.

 \vskip .1in \noindent  \textsc{ Step 4} ($v_3$) : So far, we have shown that if a
broken map $u$ whose gluing has homology class \eqref{eq:ugluehom},
then either $[u_\glue]=[D_1]$ or $[u_\glue] = [D_1]+[E_1]$ or the
tropical graph of $u$ has vertices $v_{-1}$, $v_0$, $v_1$, $v_2$, and
$v_2$ has an edge of multiplicity $1$ to a vertex $v_3$ with
$P(v_3)=P_3$.  In this last case, the only possibility is that
$u_{v_3}$ is a simple map of class $[E_2'']$. The resulting broken
map, of class $[D_1]+[E_1]+[E_2]$ is shown in Figure
\ref{fig:twentyseven} (c) and Figure \ref{fig:27polys} (c).  Multiple
covers of $[E_2'']$ are ruled out using arguments similar to those in
Step 2: If $u_{v_3}$ is a multiple cover of $E''_{2}$ then there is a
path $P$ in $\Gamma$ emanating out of $v_3$ (which does not contain
$v_2$) whose end-point is a $(-1)$-curve.  Let $e \ni v_3$ be the first
edge in the path $P$.  Since the path ends in a $(-1)$-curve, there is
a point constraint on $u_{v_3}$ at the node $w_e$ corresponding to
$e$.  For a generic background $J$ and a small enough perturbation, the
point constraint will be distinct from
$E''_{2,J_{v_3}} \cap \ol X_{P_{23}}$, contradicting the existence of
such a multiple cover $m[E_2'']$ in the broken map.

Finally, we show that each tropical graph in Figure
\ref{fig:twentyseven} contributes $+1$ to the curve count.  Since the
sphere $u_v$ corresponding to each vertex $v \in \Ver(\Gamma)$ is
pseudoholomorphic with respect to an integrable almost complex
structure $J_v^b$ (see \eqref{eq:jvbdef}), and matching conditions are
cut out by diagonals on the relative divisors, each of the broken maps
has positive orientation in the moduli space.  In each of these cases,
there is no automorphism of the tropical graph.  Therefore, a broken
map with $n$ interior markings contributes $\frac 1 {n!}$ to the curve
count (see \eqref{eq:wtwu}). Assuming that the components of the
broken map have $n_1,\dots,n_m$ markings (that add up to $n$),
permuting the labellings of marked points within each component gives
$(n_1!\dots n_m!)$ curves.  Finally, there are
$\frac {n!} {n_1!\dots n_m!}$ ways of assigning interior markings to
components -- each of these would give rise to a different almost
complex structure.  Our proof shows that in all the cases, the number
of broken maps corresponding to each tropical graph is
unaffected. This finishes the proof of the curve count in broken
maps.  Since Gromov-Witten invariants are not altered by
neck-stretching and deformation of the symplectic form, we have shown
that there are $27$ lines on a cubic surface.
\end{proof}

\subsection{Disks in the cubic surface}

In the next result, we compute the disk potential of the monotone
Lagrangian torus $L_0$ (see \eqref{eq:L0def}) in the cubic surface.
The monotone Lagrangian, with its trivial local system split-generates
the sub-category of the Fukaya category corresponding to the ``small
eigenvalue'' in Sheridan's language in \cite{sh:hmsfano}.  To compute
the potential of the Lagrangian $L_0 \subset X_0$, we work in
$(X_\eps,\om_\eps)$ which is symplectic deformation equivalent to
$(X_0,\om_0)$, and compute the potential of the Lagrangian torus
$L_\eps \subset X_\eps$ (see Proposition \ref{prop:sympdef}).  We show
that a disk in the deformed cubic surface $X_\eps$ with boundary in
$L_\eps$ and of Maslov index $I(u) = 2$ either
  \begin{enumerate}
  \item intersects a long divisor $D_i$ (there are 3 such disks), or
  \item is a nodal disk consisting of a disk $u_0: S \to X$
    intersecting a short divisor $E_i$ and a path of spheres
    $u_1,\ldots, u_k: \P^1 \to X$, each mapping to a short divisor
    $E_j$, and thus having Maslov index zero.  The maps $u_i$,
    $u_{i+1}$ intersect at a nodal point; and the maps $u_i$, $u_j$ do
    not intersect if $j-i \geq 2$.  There are $18$ such nodal
    disks.  See Figure \ref{fig:twentyone}.
  \end{enumerate}

\begin{figure}[ht]
  \centering \scalebox{.8}{
\begingroup%
  \makeatletter%
  \providecommand\color[2][]{%
    \errmessage{(Inkscape) Color is used for the text in Inkscape, but the package 'color.sty' is not loaded}%
    \renewcommand\color[2][]{}%
  }%
  \providecommand\transparent[1]{%
    \errmessage{(Inkscape) Transparency is used (non-zero) for the text in Inkscape, but the package 'transparent.sty' is not loaded}%
    \renewcommand\transparent[1]{}%
  }%
  \providecommand\rotatebox[2]{#2}%
  \newcommand*\fsize{\dimexpr\f@size pt\relax}%
  \newcommand*\lineheight[1]{\fontsize{\fsize}{#1\fsize}\selectfont}%
  \ifx\svgwidth\undefined%
    \setlength{\unitlength}{351.91432202bp}%
    \ifx\svgscale\undefined%
      \relax%
    \else%
      \setlength{\unitlength}{\unitlength * \real{\svgscale}}%
    \fi%
  \else%
    \setlength{\unitlength}{\svgwidth}%
  \fi%
  \global\let\svgwidth\undefined%
  \global\let\svgscale\undefined%
  \makeatother%
  \begin{picture}(1,0.27849071)%
    \lineheight{1}%
    \setlength\tabcolsep{0pt}%
    \put(0,0){\includegraphics[width=\unitlength,page=1]{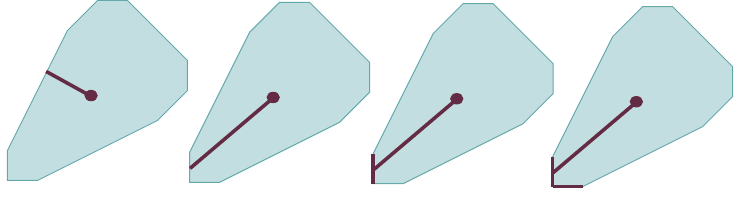}}%
    \put(0.04519704,0.24205185){\color[rgb]{0,0,0}\makebox(0,0)[lt]{\lineheight{1.25}\smash{\begin{tabular}[t]{l}(a)\end{tabular}}}}%
    \put(0.29911432,0.23657161){\color[rgb]{0,0,0}\makebox(0,0)[lt]{\lineheight{1.25}\smash{\begin{tabular}[t]{l}(b)\end{tabular}}}}%
    \put(0.55303172,0.24205185){\color[rgb]{0,0,0}\makebox(0,0)[lt]{\lineheight{1.25}\smash{\begin{tabular}[t]{l}(c)\end{tabular}}}}%
    \put(0.78502808,0.24205185){\color[rgb]{0,0,0}\makebox(0,0)[lt]{\lineheight{1.25}\smash{\begin{tabular}[t]{l}(d)\end{tabular}}}}%
    \put(0.21101843,0.04659032){\color[rgb]{0,0,0}\makebox(0,0)[lt]{\lineheight{1.25}\smash{\begin{tabular}[t]{l}{\small $E_1$}\end{tabular}}}}%
    \put(0.26513525,0.00562938){\color[rgb]{0,0,0}\makebox(0,0)[lt]{\lineheight{1.25}\smash{\begin{tabular}[t]{l}{\small $E_2$}\end{tabular}}}}%
    \put(-0.00156921,0.12619437){\color[rgb]{0,0,0}\makebox(0,0)[lt]{\lineheight{1.25}\smash{\begin{tabular}[t]{l}{\small $D_1$}\end{tabular}}}}%
  \end{picture}%
\endgroup%
}
\caption{There are twenty-one Maslov-index-two-disks
on a cubic surface: 3 of type (a), and 6 each of type (b), (c) and (d).} 
\label{fig:twentyone}\end{figure}

\begin{proposition}\label{prop:diskcount} \label{twentyoneex}
  {\rm(Twenty one disks in the cubic surface)} Any holomorphic disk in
  $X:=X_\eps$ bounding the Lagrangian $L_\eps$ with Maslov index two
  is of the form shown in Figure \ref{fig:twentyone}; there are twenty
  one such disks and each disk has a multiplicity of one.  
  
  The cubic surface $(X_0,\om_0)$ has twenty one disks of Maslov index
  two that bound the monotone Lagrangian $L_0 \subset X_0$ (see
  \eqref{eq:L0def}).  The disk potential of $(X_0,L_0)$ is
 \begin{equation}
   \label{eq:cubicpot}
   W(y_1,y_2) = y_1 y_2 + y_1/y_2^2 + y_2/y_1^2 + 3(  y_1 +  y_1/y_2 +
   y_2 +  y_2/y_1 + 1/y_1 + 1/y_2)     
 \end{equation}
\end{proposition}

\begin{remark} \label{21pot} The disk potential of $L_0$ in the cubic
  surface was originally obtained by Pascaleff-Tonkonog
  \cite{pasc:wall} using mutations.  This formula is equivalent to the
  one in \cite[Table 1]{pasc:wall} by the change of variables
\[ z_1= y_1^{-1} , z_2= y_2^{-1}  \] 
after which the formula becomes
\[ W(z_1,z_2) = (1 + z_1 + z_2)^3/z_1 z_2 - 6.  \]
The formula is also obtained in Galkin-Usnich \cite{gu};
c.f. \cite[Table 5.1]{barrott}.
\end{remark}

\begin{proof}[Proof of Proposition \ref{prop:diskcount}]
  As in the proof of the sphere count, we first count the number of
  broken disks of Maslov index $2$ in the broken manifold, and then
  argue that the disk count is the same for the cubic surface.  Let
  $\XX_\eps$ be the broken manifold obtained by first deforming the
  symplectic form on the cubic surface to produce a toric manifold
  $X^\eps$, and then by multiple cutting.

  The homology classes of the disks of Maslov index two in $X$
  (equipped with a torus-invariant almost complex structure) must
  consist of a single Blaschke product (i.e. a disk having exactly one
  intersection with one of the toric divisors) and a collection of
  short divisors, each of which have Chern number zero:
  \begin{equation}
    \label{eq:h2xl}
    H_2(X_\eps,L) \ni [u] =
    \begin{cases}
      \delta_{\on{long}} \quad \text{or}\\
      \delta_{\on{short}} + \sum_i a_i[E_i],
    \end{cases}
  \end{equation}
  where $\delta_{\on{long}}$ resp. $\delta_{\on{short}} \in H_2(X,L)$
  is the class of a disk which is a single Blaschke product
  intersecting a long resp. short divisor, $a_i \in \Z_+$, and any
  $E_i \in H_2(X)$ is the homology class of a short divisor.

  A semi-Fano toric surface has a well-defined disk potential in the
  following sense: The standard torus-invariant complex structure
  $J_\eps$ on $X_\eps$ is not regular, but for any (domain-dependent)
  almost complex structure $J_\eps'$ that is sufficiently close to
  $J_\eps$, the disk count is independent of $J_\eps'$, see
  Proposition \ref{prop:semiFanopot}.  Furthermore, the potential is
  unchanged under neck-stretching if the broken manifold, with the
  standard almost complex structure (componentwise torus-invariant)
  $\JJ_\eps$, does not have disks of Maslov index zero, as explained
  in Proposition \ref{prop:semiFano-break}. This condition is indeed
  satisfied in the multiple cut we consider on the cubic surface, as
  explained in Proposition \ref{prop:cubic0disks}. Therefore, for any
  regular domain-dependent almost complex structure $J_\eps'$
  resp. $\JJ_\eps'$ close enough to $J_\eps$ resp. $\JJ_\eps$, the
  potential on $(X_\eps,J_\eps')$ and $(\XX_\eps,\JJ_\eps')$ is the
  same. In particular, the contribution to the potential from any disk
  homology class $\beta \in H_2(X_\eps,L)$ is the same for
  $(X_\eps,J_\eps')$ and $(\XX_\eps,\JJ_\eps')$. Thus we may assume
  that the homology class of any broken disk in $\XX_\eps$ is as in
  \eqref{eq:h2xl}.

  Next, we show that out of all the homology classes, only twenty one
  classes, corresponding to $a_i=0,1$ are represented by
  $\JJ_\eps'$-holomorphic broken disks, with a single disk in each
  class.  Let $u$ be a broken disk in $\XX_\eps$ whose gluing has
  relative homology class
  \[[u_{\glue}]=[\delta_{E_1}] + \sum_i n_i[E_i] \in H_2(X_\eps,L),\]
  where $[\delta_{E_1}]$ is the class of a Maslov index two disk in
  $(X_\eps,L_\eps)$ that intersects the short divisor $E_1$ (using
  notation in Figure \ref{fig:twentyone}).  The other cases where
  $\delta_{E_1}$ is replaced by $\delta_{E_i}$, $i=2,\dots,6$, can be
  analyzed similarly.  Now we need to show that the gluing of $u$ is
  either (b), (c) or (d) in Figure \ref{fig:twentyone}.  Since the
  class $[\partial u] \in H_1(L_\eps)$ is preserved by neck
  stretching, we conclude that the disk component of $u$ is a Maslov
index two disk in $\ol X_{P_0}$ that intersects the relative divisor
  $\ol X_{P_{01}}$.  Indeed $\ol X_{P_0}$ is a toric manifold; and no
  other positive integral combination of outward normal vectors of
  facets adds up to the outward normal $\nu(\ol X_{P_{01}})=(-1,0)$,
  besides $(-1,0)$ itself.  The rest of the proof that $u$ is of the
  expected form is exactly same as Steps 2, 3, 4 in the proof of the
  sphere count (Proposition \ref{prop:spherecount}).  Each of the
  tropical graphs contributes $+1$ to the curve
  count: \label{rep:orientation-discuss} Using the standard spin
  structure, isolated Blaschke disks have a positive orientation sign
  (see \cite[p22]{chooh:toric}), the maps $u_v$ are pseudoholomorphic
  with respect to an integrable almost complex structure $J_v^b$ (see
  \eqref{eq:jvbdef}), matching conditions are cut out by diagonals on
  the relative divisors, and each of the broken maps has positive
  orientation in the moduli space. 
The combinatorics involving marked points carries over
  from the sphere case.

  So far, we have proved that for any regular perturbation $\ul \Pe_0$
  of the standard almost complex structure $J_0$ on the toric manifold
  $X$, the potential of $CF(L,\ul \Pe_0)$ has twenty one terms. It remains
  to show that the potential is the same for the cubic surface
  $(X_0,\om_0)$.  We take an indirect route to construct the bijection
  of disks in $X_\eps$ and $X_0$, namely, we use a different multiple
  cut $\PP_1$ (Figure \ref{fig:P1cut}) than the one above, and show
  that there is a bijection between the disks contributing to the
  potential in the broken manifolds $\XX_{\PP_1}$ and $\XX_{0,\PP_1}$.
  \begin{figure}[ht] \centering
      \scalebox{.8}{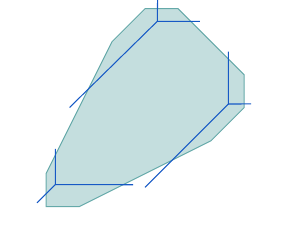} 
    \caption{The cut $\PP_1$.}
    \label{fig:P1cut}
  \end{figure}
  Thus we need to prove
  the three bijections of disks,
  corresponding to the degenerations
  \begin{equation}
    \label{eq:3bij}
  X \longleftrightarrow \XX_{\PP_1} \longleftrightarrow
  \XX_{0,\PP_1} \longleftrightarrow X_0.  
  \end{equation}
  The first of these bijections between the set of disks in $X$ and
  $\XX_{\PP_1}$ is a consequence of Proposition
  \ref{prop:semiFano-break}.  To prove the second bijection, denote by
  \[\M^{<E_0}_{\beta,\JJ}(\XX)\]
  the zero-dimensional component of the moduli space of broken disks
  with symplectic area $\leq E_0$ and $[\partial u]=\beta$ with a
  point constraint on the boundary, assuming that the broken disks are
  defined with respect to a perturbation $\pp$ of $\JJ$.  Let
  $\{\tJJ_t\}_{t \in [0,\eps]}$ be a generic path of almost complex
  structures with end-points $\tJJ_\eps:=\JJ_\eps'$ and $\tJJ_0$
  compatible with the monotone form $\om_0$. Also, let
  $\{\pp_t\}_{t \in [0,\eps]}$ be a family of perturbations of
  $\{\tJJ_t\}_t$.  Finally, the moduli space
  \[\M^{<E_0}_{\beta,\tJJ}(\XX):=\cup_t \M^{<E_0}_{\beta,\tJJ_t}(\XX)\]
  is one-dimensional whose end-points are
  $\M^{<E_0}_{\beta,\tJJ_0}(\XX)$ and
  $\M^{<E_0}_{\beta,\tJJ_\eps}(\XX)$.  Indeed,
  for $\M^{<E_0}_{\beta,\tJJ}(\XX)$ there is no disk bubbling, because
  each $t$, the torus-invariant divisors of $\ol X_{P-0}$ are
  $\tJJ_t$-holomorphic; and for any broken map in the moduli space,
  the disk component has Maslov index two and intersects the
  divisor $\ol X_{P_{01}}$. Therefore, the boundary strata of
  $\M^{<E_0}_{\beta,\tJJ}(\XX)$ consist only of
  $\M^{<E_0}_{\beta,\tJJ_0}(\XX)$ and
  $\M^{<E_0}_{\beta,\tJJ_\eps}(\XX)$.  It follows that these last two
  moduli spaces are bijective.
  This proves the second bijection in \eqref{eq:3bij}. The last
  bijection in \eqref{eq:3bij} follows from the fact that
  neck-stretching does not alter the potential of a monotone
  symplectic manifold, since there are no disks of Maslov index zero
  in the family of neck-stretched manifolds.
  %
  This finishes the proof of the Proposition.
\end{proof}

\section{Counting disks in flag varieties}
\label{sec:flag}
In this Section, we use our results to compute the disk potential of
partial flag varieties.  The Gelfand-Cetlin system \cite{gs:gc} is a
collection of functions on a partial flag variety that forms a
completely integrable system. The image of these functions is the
\em{Gelfand-Cetlin} polytope $\Delta$, and the toric variety
corresponding to this polytope is called the \em{Gelfand-Cetlin} toric
variety. In the complement of faces of $\Delta$ of codimension
$\geq 2$, the flag variety and Gelfand-Cetlin toric variety are
symplectomorphic.  We use this observation to compute the disk
potential of the monotone Lagrangian torus in Theorem
\ref{thm:flagpot}; we use a multiple cut consisting of a collection of
intersecting single cuts (Figure \ref{fig:fl3cut}) each of which is
parallel to, and very close to a facet of the polytope $\Delta$.  The
original approach of computing disk potentials in flag varieties by
Nishinou-Nohara-Ueda \cite{nnu:degen} and Nohara-Ueda \cite{no:degen}
uses small resolutions.

We start by defining partial flag varieties. 
Given a tuple $\ul n:=(n_1,\dots,n_r)$ with $\sum_i n_i=n$, a \em{
  flag variety} $\Fl(n_1,\dots,n_r)$ is the set of flags
\begin{multline}\label{eq:fldef}
 \Fl(\ul n):= \Fl(n_1,\dots,n_r) := \{ \{ 0 \} \subset V_{n_1}
 \subset V_{n_1+n_2} \subset \ldots V_{n_1+\dots+n_r}= \C^n \\
 : \dim(V_i)=i \enspace \forall i\}.  
\end{multline}
A flag variety may be identified with a coadjoint orbit of the action
of $U(n)$ on $u(n)^\dual$. In particular,
\[\Fl(\ul n)\simeq \mO_\Lam \subset \u(n)^\dual\]
is the
coadjoint orbit of an element $\Lam$ defined as 
\begin{equation}
  \label{eq:coad}
\Lam:=  \on{diag}(\sqrt{-1}(\underbrace{\Lam_1,\dots, \Lam_1}_{n_1 \text{ times}},
  \underbrace{\Lam_2,\dots, \Lam_2}_{n_2 \text{ times}},\dots,
  \underbrace{\Lam_r,\dots, \Lam_r}_{n_r \text{ times}})) \in
  \u(n)^\dual,
\end{equation}
where 
\[\Lam_i \in \R, \quad \Lam_1 > \dots > \Lam_r,\]
and $\u(n)$ is identified to its dual via the $\Ad_{U(n)}$-invariant
inner product
\[(\cdot,\cdot): \u(n) \times \u(n) \to \R, \quad (A,B) \mapsto
  \on{trace}(A^*B).\]
Indeed, an element of \eqref{eq:fldef} corresponds to an orthogonal decomposition
\[u(n)^\dual = W_{n_1} \oplus \dots \oplus W_{n_r}\]
into subspaces $W_{n_i}$ of dimension $n_i$ for all $i=1,\dots,r$; and
such a decomposition gives a unique element in the coadjoint orbit of \eqref{eq:coad} for which the $\sqrt{-1} \Lam_i$-eigenspace is $W_{n_i}$ for each $i$.

The symplectic form $\om_\Lam$ on $\Fl(\ul n)$ is given by the Kostant-Kirillov form: A tangent vector at a point
 $x \in \u(n)$ is of the form $\ad_\xi(x)$ for some $\xi \in \u(n)$ and the symplectic form at $x$ is given by $(\om_\Lam)_x(\ad_{\xi_1}(x), \ad_{\xi_2}(x)):=\bran{x,[\xi_1,\xi_2]}$.

We recall the construction of the 
Gelfand-Cetlin system on a flag
variety from Guillemin-Sternberg \cite{gs:gc}. The \em{Gelfand-Cetlin system} is a
collection of functions that forms a completely integrable system on an
open set of $\Fl(\ul n)$.  For any $m = 1,\ldots, n$, let 
\begin{equation}
  \label{eq:pimdef}
  \pi_m: \Fl(\ul n) \to \u(m)^\dual   
\end{equation}
be a map sending a matrix $A$ to its  
upper-left $m \times m$
submatrix $\pi_m(A)$. 
Let
\begin{equation}
  \label{eq:gcmap}
  \Phi: \Fl(\ul n) \to \t^\dual := (\sqrt{-1} \R)^{n(n-1)/2}
\end{equation}
be a collection of maps $\Phi=\{\sqrt{-1} \Phi_{i,k}\}_{1 \leq k \leq i \leq n}$ where for any $m \leq n$ 
and $A \in \Fl(\ul n)$, 
\[ \Phi_{m,1}(A) \geq \Phi_{m,2}(A) \geq \ldots \geq \Phi_{m,m}(A), \]
and $\sqrt{-1} \Phi_{m,k}$, $1 \leq k \leq m$ are the eigenvalues of $\pi_m(A)$. 
\label{rep:eigen}
For any $i$, the eigenvalue functions $\Phi_{i,j}$ are smooth
functions whenever they are distinct and are continuous everywhere.
By the min-max theorem for Hermitian matrices \cite[Theorem
8.10]{zhang:book}, the eigenvalues satisfy the interlacing inequalities
\[ \Phi_{i+1,k} \geq \Phi_{i,k} \geq \Phi_{i+1,k+1}, \quad \forall i, k
  < n \]
 and they fit into the following diagram:
\begin{equation}
  \begin{alignedat}{17}
    \Lam_1 &&&& \Lam_2 &&&& \Lam_3 && \cdots && \Lam_{n-1} &&&& \Lam_n  \\
    & \uge && \dge && \uge && \dge &&&&&& \uge && \dge & \\
    && \Phid {n-1}1 &&&& \Phid {n-1}2 &&&&&&&& \Phid{n-1}{n-1} && \\
    &&& \uge && \dge &&&&&&&& \dge &&& \\
    &&&& \Phid {n-2}1 &&&&&&&& \Phid{n-2}{n-2} &&&& \\
    &&&&& \uge &&&&&& \dge &&&&& \\
    &&&&&& \dndots &&&& \updots &&&&&& \\
    &&&&&&& \uge && \dge &&&&&&& \\
    &&&&&&&& \Phid 11 &&&&&&&&&
  \end{alignedat}
  \label{eq:GCineq}
\end{equation}

We give a proof of the min-max theorem for completeness.

\begin{lemma}{\rm(Interlacing inequality)}
  Let $A$ be an $(n+1) \times (n+1)$ Hermitian matrix with  
  eigenvalues
  $a_1 \geq \dots \geq a_{n+1}$. 
  Suppose the upper left $n \times n$ submatrix $\pi_n(A)$
  has eigenvalues $b_1 \geq \dots \geq b_n$. Then,
  \[a_1 \geq b_1 \geq a_2 \geq \dots \geq a_{n-1} \geq b_n \geq a_{n+1}.\]
\end{lemma}

\begin{proof}
  It is enough to prove the result assuming that $b_1 > \dots >b_n$,
  since the general result follows by continuity.  Since $b_1,\dots, b_n$ are the eigenvalues of $\pi_n(A)$, the  $\Ad_{U(n)}$-orbit of $A$ contains a matrix
  \[A':=
    \begin{pmatrix}
      b_1 & & 0 & \bar{z}_1\\
      & \ddots & & \vdots\\
      0 & & b_n & \bar{z}_n\\
      z_1 & \hdots & z_n & r_{n+1},
    \end{pmatrix}.
  \]
where $z_i \in \C$, $r_{n+1} \in \R$. Then,
  \begin{equation*}
    \det(A-t\Id)= \det(A'-t\Id)=  \prod_{i=1}^n (b_i - t) f(t), \quad f(t)= \left(r_{n+1} - t - \sum_{j=1}^n \frac {|z_j|^2} {b_j-t}\right).
  \end{equation*}
  The function 
  $f(t)$ is decreasing in $t$ and discontinuous at
  $b_1,\dots,b_n$.  Since the real numbers
  $b_j$'s are distinct, the function $f(t)$ has exactly
  one root each in the intervals $(-\infty,b_n)$, $(b_{i+1},b_i)$,
  $i=1,\dots,n-1$, $(b_1,\infty)$, and the roots are $a_{n+1}, \dots, a_1$.
\end{proof}

For a partial flag manifold, certain values in the Gelfand-Cetlin
system are fixed by the interlacing inequalities. In particular, for $k \leq i \leq n-1$,
\[\Phi_k=\Phi_{n+k-i}=\Lam_j \implies \Phi_{i,k}=\Lam_j. \]
The Gelfand-Cetlin system thus consists of exactly
\[N:=\sum_{1 \leq i<j \leq r} n_in_j\]
variables, which is half the dimension of $\Fl(\ul n)$.  See Figure
\ref{fig:gc-partial}.  We denote the set of variables in the
Gelfand-Cetlin system of the partial flag variety $\Fl(\ul n)$ by
\[\on{Free}(\ul n):=\{(i,j): \Phid i j  \text{ is not a
    constant on }\Fl(\ul n)\}. \]
The image $\Phi(\Fl(\ul n))$ is contained in a polytope
\[\Delta_{\Fl(\ul n)} \subset (\sqrt{-1} \R)^N\]
cut out by the Gelfand-Cetlin inequalities involving free variables
$\{\Phid i j : (i,j) \in \on{Free}(\ul n)\}$, which we call the \em{Gelfand-Cetlin polytope of $\Fl(\ul n)$}.

\begin{figure}[h]
  \centering\scalebox{.8}{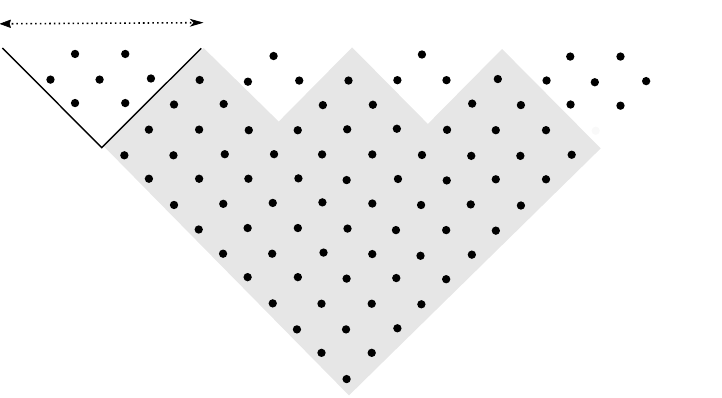}
  \caption{For a partial flag, the variables of the Gelfand-Cetlin system lie in the shaded region.  Each of the unshaded triangles is a block of non-free variables. In any block, all the entries are equal to $\Lam_i$ for some $i$.}
  \label{fig:gc-partial}
\end{figure}

\begin{remark}\label{rem:gclowest}
  In the Gelfand-Cetlin system corresponding to the flag variety
  $\Fl(n_1,\dots,n_r)$, the non-free variables can be partitioned into
  $r$ disjoint sets
  \[\{(i,k) : 1 \leq k \leq i \leq n-1\} \bs \on{Free}(\Fl(\ul n)) =
    N_1 \cup \dots N_r\]
  where each $N_j$ consists of the collection of indices for maps $\Phid i k$ which are fixed to be equal to $\Lam_j$ by the interlacing inequalities.  We call $N_j$ a \em{block} of non-free variables. A block $N_j$ is non-empty if $n_j > 1$.  Each non-empty block is triangular and has a \em{lowest} element whose location in the Gelfand-Cetlin diagram is the lowest. The lowest element in the block $N_j$ is
  \begin{equation}
    \label{eq:blocklowest}
  \Phi_{\mathfrak{i}_j, \mathfrak k_j}, \quad \text{ where } 
(\mathfrak i_j, \mathfrak k_j)=\left(n-n_j+1,  \sum_{\ell=1}^jn_\ell + 1\right).  
  \end{equation}
  In Figure \ref{fig:gc-partial} the lowest position in each block is in red. 
\end{remark}

The face structure of the Gelfand-Cetlin polytope can be read off from
ladder diagrams \cite{AnChoKim}.  For our purposes, we explicitly
describe the faces of codimension zero and one.
Let $\Delta_{\Fl(\ul n)}$ be the Gelfand-Cetlin polytope of the flag
variety $\Fl(\ul n)$.
\begin{enumerate}
\item {\rm(Interior)} A point $\Phi \in \Delta_{\Fl(\ul n)}$ is an
  interior point exactly if every Gelfand-Cetlin inequality involving at
  least one free variable $\Phid i j$ is strict.  We denote the inverse image of the interior by 
  \[F^{(0)}:=\Phinv(\Delta_{\Fl(\ul n)}^\circ) \subset \Fl(\ul n).\]
\item {\rm(Facet)} A subset of the polytope $\Delta_{\Fl(\ul n)}$
  \begin{equation}
    \label{eq:facetdef}
    \F_{i,k,\ldiag}:=\{\Phi_{i,k}=\Phi_{i+1,k}\} \quad \text{resp.} \quad \F_{i,k,\rdiag}:=\{\Phi_{i,k}=\Phi_{i+1,k+1}\}   
  \end{equation}
  is a facet of $ \Delta_{\Fl(\ul n)}$ if either both variables in the
  equality are free or the non-free variable is the lowest variable in
  a block of non-free variables (as defined in Remark
  \ref{rem:gclowest}).  We remark that if the non-free variable in
  \eqref{eq:facetdef} were not the lowest in the block, then the
  codimension of the subset would be more than $1$, because the
  interlacing inequalities would fix the values of some free variables
  not occurring in \eqref{eq:facetdef}.  We denote the set of indices of facets by
  \[\on{Facets}(\Delta_{\Fl(\ul n)}) \subset \Z_+^2 \times \{\ldiag, \rdiag\}.\]
  A point $\Phi \in \F_{i,k,\ldiag}$ resp. $\F_{i,k,\rdiag}$ is in the interior of the facet if all the other Gelfand-Cetlin inequalities are strict.  The set of points in $\Fl(\ul n)$ that map to the interiors of facets is denoted by
 \[F^{(1)} := \cup_{I \in \on{Facets}(\Delta_{\Fl(\ul n))}}\Phinv(\F_I).\]
\item {\rm(Regular points)}
  We denote by
  \begin{equation}
    \label{eq:fregdef}
    F^{\on{reg}} \subset \Fl(\ul n)    
  \end{equation}
the set of points $x$ for which for every $m <n$ the set of
eigenvalues of $\pi_m(x)$ has the maximum cardinality. In other words,
$x \in F^{\on{reg}}$ exactly when the following holds: For all positive
integers $i \leq n-1$, $k<i$,
 \[\Phi_{i,k}(x)=\Phi_{i,k+1}(x) \Leftrightarrow (i,k), (i,k+1) \notin \on{Free}(\ul n).\]
The Gelfand-Cetlin map is smooth on $F^{\reg}$ (Proposition \ref{prop:gcresult}). 
\end{enumerate}

\begin{remark} \label{rep:nonsimplicial}
  {\rm(Non-simplicial points of the Gelfand-Cetlin polytope)}
  The Gelfand-Cetlin polytope contains non-simplicial points.  A face $\F$ of a polytope $\Delta$ is \em{non-simplicial} if it is the intersection of $n_\F$ facets where $n_\F > \codim(\F)$.
  A point $A \in \Fl(\ul n)$ is mapped to a non-simplicial point the polytope $\Delta_{\Fl(\ul n)}$ if it satisfies a loop of equalities (see \cite[Example 3.8]{nnu:degen}),
  \begin{equation}
    \label{eq:loopeq}
    \begin{alignedat}{5}
      &&\Phid {i+1} {k+1}&&\\
      &\deq&&\ueq&\\
      \Phid i k &&&& \Phid i {k+1}\\
      &\ueq&&\deq&\\
       &&\Phid {i-1} k&&
     \end{alignedat}
\end{equation}
 and each inequality in the loop corresponds to a facet. 
Indeed, the loop of equalities in \eqref{eq:loopeq} gives a subset
  $S$ of codimension three, whereas $S$ is the intersection of four
  facets corresponding to each of the equalities in the loop.  
  Observe that $S$ is the intersection of four faces of
  codimension two, each given by a pair of equalities as in the diagram of inequalities below:
  \begin{equation*}
    \begin{alignedat}{5}
      &&\Phid {i+1} {k+1}&&\\
      &\dge&&\ueq&\\
      \Phid i k &&&& \Phid i {k+1}\\
      &\ueq&&\dge&\\
       &&\Phid {i-1} k&&
     \end{alignedat}
     \qquad
      \begin{alignedat}{5}
      &&\Phid {i+1} {k+1}&&\\
      &\deq&&\uge&\\
      \Phid i k &&&& \Phid i {k+1}\\
      &\uge&&\deq&\\
       &&\Phid {i-1} k&&
     \end{alignedat}
     \qquad
      \begin{alignedat}{5}
      &&\Phid {i+1} {k+1}&&\\
      &\dge&&\ueq&\\
      \Phid i k &&&& \Phid i {k+1}\\
      &\uge&&\deq&\\
       &&\Phid {i-1} k&&
     \end{alignedat}
     \qquad
      \begin{alignedat}{5}
      &&\Phid {i+1} {k+1}&&\\
      &\deq&&\uge&\\
      \Phid i k &&&& \Phid i {k+1}\\
      &\ueq&&\dge&\\
       &&\Phid {i-1} k&&
     \end{alignedat}
   \end{equation*}
\end{remark}

\begin{example}\label{ex:singstrat}
  The full flag $F^{(3)}=\Fl(1,1,1)$ resp. the Grassmannian
  $\Gr(2,4)=\Fl(2,2)$ has a singular point resp. a singular line given
  by a loop of equalities (see Figure \ref{fig:fl3}) 
  \begin{equation}
    \begin{alignedat}{5}
      &&\Lam_2&&\\
      &\deq&&\ueq&\\
      \Phid 2 1 &&&& \Phid 2 2\\
      &\ueq&&\deq&\\
      &&\Phid 1 1&&
    \end{alignedat}
    \quad \text{resp.} \quad 
    \begin{alignedat}{5}
      &&\Phid 3 2&&\\
      &\deq&&\ueq&\\
      \Phid 2 1 &&&& \Phid 2 2\\
      &\ueq&&\deq&\\
      &&\Phid 1 1&&
    \end{alignedat}.
  \end{equation}
\end{example}

\begin{figure}[h]
  \centering\scalebox{.8}{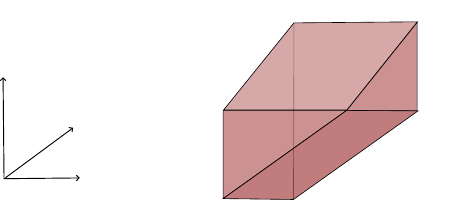}
  \caption{There is a non-simplicial point in $\Phi(\Fl(1,1,1))$
    corresponding to a loop of equalities.}
  \label{fig:fl3}
\end{figure}

\begin{remark}{\rm(Gelfand-Cetlin toric variety)}
  For any tuple $\ul n$, the toric variety $M_{\ul n}$ corresponding
  to the polytope $\Phi(\Fl(\ul n))$, called the \em{Gelfand-Cetlin
    toric variety}, is a toric degeneration of the flag manifold
  $\Fl(\ul n)$ \cite{bckv:mirror}.  The Gelfand-Cetlin toric variety
  has conifold singularities of codimension three (see Batyrev et
  al. \cite{bckv:mirror}) corresponding to the non-simplicial points
  in the polytope $\Delta_{\Fl(\ul n)}$.
  The fibers of $\Phi|\Fl(\ul n)$ over non-simplicial points has larger dimension
  than the corresponding fibers in the toric variety.  For example, the inverse image of the
  singular point in $\Phi(\Fl(1,1,1))$ in Figure \ref{fig:fl3} is a
  Lagrangian $S^3$ \cite[Example 3.8]{nnu:degen}. A detailed
  description of the singular fibers in partial flag varieties, and
  their comparison with the Gelfand-Cetlin toric variety is carried
  out in \cite{cko:lag}.
\end{remark}

Properties of the Gelfand-Cetlin map, such as smoothness and
submersivity, are applicable strata-wise. There are two types of
stratification on the flag manifold defined as follows. The first is based on all inequalities in the
Gelfand-Cetlin polytope, the second one is based on inequalities between row elements. We remark that
 the $k$-th
horizontal row $(\Phi_{k,i})_i$ is a projection of the $k \times k$
minor $\pi_{k-1}(A)$ to the positive Weyl chamber of $\u(k)$.

\begin{definition}{\rm(Stratifications of the flag manifold)} Let $\Fl(\ul n)$ be a flag manifold with Gelfand-Cetlin variables $(\Phi_{k,i})_{(k,i)}$. 
  \begin{enumerate}
  \item {\rm($\GC$-stratification)} Given a collection $I$ of
    Gelfand-Cetlin inequalities, we denote by
    \[\Delta_I \subseteq \Delta = \Phi(\Fl(\ul n))\]
    the set of points where the inequalities in $I$ are equalities,
    and the equalities not lying in $I$ are strict. The set
    \[F_I:=\Phinv(\Delta_I) \subset \Fl(\ul n)\]
    is a \em{Gelfand-Cetlin stratum}, or \em{$\GC$-stratum} for
    short. A variable $\Phi_{k,i}$ is \em{free} on the $\GC$-stratum
    $F_I$ if it is the lowest element (having the least value of $k$)
    in an equivalence class of variables that are constrained to be
    equal. The set of free variables on the stratum $F_I$ is denoted
    by $\Free_I(\ul n) \supseteq \Free(\ul n)$.
  \item {\rm($\W$-stratification)} For any $k \geq 0$, inequalities of
    the form $\Phi_{k,i} \geq \Phi_{k,i+1}$, are called
    \em{$\W_k$-inequalities}.  For a fixed $k$, given a set $I$ of
    $\W_k$-inequalities between the variables $\{\Phi_{k,i}\}_i$, we
    denote by $F_I \subset \Fl(\ul n)$ the subset on which the
    inequalities in $I$ are equalities, and the inequalities not lying
    in $I$ are strict. The set $F_I$ is called a \em{$\W_k$-stratum}
    of $\Fl(\ul n)$.
  \end{enumerate}
\end{definition}

\begin{proposition}{\rm(Smoothness)}
  \label{prop:smooth}
  Suppose $k \geq 0$, and $F_I$ is a $\W_k$-stratum corresponding 
  to a set $I$ of inequalities between the
  variables $\{\Phi_{k,i}\}_i$.  Suppose the tuple  
  $(\alpha_i \in \R)_{1 \leq i \leq k}$ satisfies
  \[\Phi_{k,i} \equiv \Phi_{k,i+1} \enspace \text{on $F_I$} \implies \alpha_i=\alpha_{i+1}.\]
  Then, the map $\sum_i \alpha_i \Phi_{k,i}$ is
  smooth in a neighborhood of $F_I$ in $\Fl(\ul n)$. 
  Therefore, the Gelfand-Cetlin map $\Phi$ is smooth on $F^\reg$. 
 \end{proposition}
\begin{proof}
  The proof follows from the following observation:
  Suppose the maps
  \[\lam_1 \geq \dots \geq \lam_k: \text{Hermitian $k \times k$ matrices} \to \R\]
  assign matrices to the imaginary part of their eigenvalues. If a Hermitian matrix $A$ and $i \geq 1$, $j \geq 0$ are such that
  \[\dots\geq \lam_{i-1}(A) > \lam_{i}(A) =\dots=\lam_{i+j}(A) > \lam_{i+j+1}(A) \geq \dots,\]
  then $\lam_{i} +\dots + \lam_{i+j}$ is smooth in a neighborhood
  of $A$ in the space of symmetric matrices. 
  This fact may also be viewed as a
  consequence of the symplectic cross-section theorem 
  \cite[Theorem 26.7]{gs:phys}. 
\end{proof}

\begin{remark}\label{rem:wsm}
  The above result can restated as follows: The map 
  $\sum_i \alpha_i \Phi_{k,i}$ is smooth at points lying in $F_I$ if
  the levels $\{\sum_i \alpha_i \Phi_{k,i} = constant\} \subset \R^k$ are perpendicular to
  the face
  $(\Phi_{k,i})_i(F_I) \subset (\Phi_{k,i})_i(\Fl(\ul n)) \subset
  \R^k$ with respect to the standard inner product on $\R^k$.
\end{remark}

\begin{proposition} {\rm(Submersivity)}
  \label{prop:subm}
  Let $F_I \subset \Fl(\ul n)$ be a $\GC$-stratum corresponding a
  collection $I$ of $\GC$-inequalities. Then, the restriction
  \[\Phi|F_I : F_I \to \Delta_I\]
  is a surjective submersion. Furthermore, if $F_I \subset F^\reg$,
  then a fiber of $\Phi|F_I$ is a torus $(S^1)^{|\Free_I(\ul
    n)|}$. (Here, we recall that $\Free_I(\ul n)$ is the set of free
  variables of $F_I$.)
\end{proposition}
\begin{proof}
  The surjectivity is a consequence of inductively applying Lemma
  \ref{lem:invphi}.  Indeed, given
  $(\lam_{k,i})_{(k,i)} \in \Delta_I$, and $A_{k-1} \in \u(k-1)^\dual$
  satisfying $\Phi_{\ell,i}(A_{k-1})=\lam_{\ell,i}$ for
  $\ell \leq k-1$, we apply Lemma \ref{lem:invphi} to the diagonal
  matrix $B$ with decreasing entries (that is, an element of the
  positive Weyl chamber) conjugate to $A_{k-1}$, and produce
  $A \in \u(k)^\dual$ of the form \eqref{eq:matAdef} with eigenvalues
  $(\lam_{k,i})_i$. In the case that $F_I \subset F^\reg$, there is a
  $(S^1)^{|\Free_{I,k}|}$-family of such matrices $A$, where
  $\Free_{I,k} \subset \Free_I(\ul n)$ is the set of free variables of
  $F_I$ in the $k$-th row. Indeed, for each free variable
  $\Phi_{k,i}$, in \eqref{eq:matAdef}, $r_i=|z_i|^2 \neq 0$ is
  uniquely determined, and for a non-free $\Phi_{k,i}$, $z_i=0$. Thus,
  we have proved the last sentence of the proposition.

  To prove submersivity, we define some
  notation. For any $k \in \N$ and
  $B \in \pi_{k-1}(F_I) \subset \u(k-1)^\dual$, define
  \[\F_B:=\{B_0 \in \pi_k(F_I) \subset \u(k)^\dual : \pi_{k-1}(B_0)=B\}.\]
  We also define
  \[\Phi_{I,k}:=(\Phi_{k,i})_{i:(k,i)\in \Free_{I,k}}.\]

  We prove submersivity at a point $A \in F_I$ by induction on the
  rows of the Gelfand-Cetlin diagram, starting with the lowest row.
  The induction step requires us to prove that for any $k \geq 1$, the
  restriction $\Phi_{I,k}|\F_{\pi_{k-1}(A)}$ has a surjective
  derivative at $\pi_k(A)$.  We may assume, possibly after conjugating
  by some $P \in U(k-1)$, that $\pi_{k-1}(A)$ lies in the positive
  Weyl chamber of $\u(k-1)^\dual$, that is, it is a diagonal matrix
  with decreasing entries. The result then follows from Lemma \ref{lem:invphi} \eqref{part:mapz}, which
  shows that $\Phi_{I,k}$ has a smooth right inverse whose image
  contains $\pi_k(A)$.
\end{proof}

The following result shows that on the regular locus
$F^{\on{reg}} \subset \Fl(\ul n)$ (defined in \eqref{eq:fregdef}), the
Gelfand-Cetlin system is a moment map for a torus
action. As a symplectic toric manifold, this set is isomorphic to the
corresponding subset in the Gelfand-Cetlin toric variety.  The regular
locus $F^{\on{reg}}$ includes $F^{(0)} \cup F^{(1)}$, that is, the inverse
image of the interior and the facets of the Gelfand-Cetlin polytope.

\begin{proposition}{\rm(\cite{gs:gc})}
  \label{prop:gcresult}
  Let $\Phi : \Fl(\ul n) \to \R^N$ be the Gelfand-Cetlin system on the
  flag variety.
  \begin{enumerate}
  \item The restriction $\Phi|F^{\on{reg}}$  is a moment map for the action of the torus $\bT^N$, where $F^{\on{reg}} \subset \Fl(\ul n)$ is defined in \eqref{eq:fregdef}. 
    %
            %
  \item \label{part:gc3} As a symplectic toric manifold,
    $F^{\on{reg}} \subset \Fl(\ul n)$ is isomorphic to the corresponding
    subset in the Gelfand-Cetlin toric variety.  
%
  \item Consequently, for any point $\lam$ in the interior of
    $\Delta_{\Fl(\ul n)}$, $\Phinv(\lam)$ is a Lagrangian torus.
  \end{enumerate}
\end{proposition}
\begin{proof}
  Firstly, the map  $\Phi|F^{\on{reg}}$ is smooth by Proposition \ref{prop:smooth}. 
  Any pair of functions in the Gelfand-Cetlin system Poisson commute :
  \[\{\Phid i j, \Phid k l\}=0 \]
  by \cite[Section 5]{gs:gc}, \cite[Lemma 3.3]{nnu:degen}.  So, the
  system $\{\Phid i k\}_{i,k}$ generates a Hamiltonian $\R^n$-action
  on $\Fl(\ul n)$.
  Any non-constant Hamiltonian trajectory generated by $\Phid i k$ has
  a period of $2\pi$ by \cite[Lemma 3.4]{gs:gc}.  Consequently, the
  system $\{\Phid i k\}_{i,k}$ generates an action of the torus
  $(\R/2\pi \Z)^N$ on $F^{\on{reg}}$.
  On any $\GC$-stratum $F_I \subset F^\reg$, by Proposition \ref{prop:subm}, 
  $\on{rank}(\d\Phi|F_I)=\hh \dim(F_I)$.  Therefore, on each stratum
  $F_I \subset F^{\on{reg}}$, $\{\Phi_{i,j}\}_{i,j}$ is a complete integrable
  system; proving \eqref{part:gc3}.
    By the Arnol'd-Liouville
  theorem (proved in, for example, Duistermaat \cite{Duist}),
  $\Phi|F^{(0)}$ is a Lagrangian torus fibration,
  and therefore, for any $\lam$ in the interior of $\Delta_{\Fl(\ul n)}$, $\Phinv(\lam)$ is a Lagrangian torus.   This finishes the
  proof.  As an aside, we remark that the rank of $\d\Phi$ is smaller
  than half the dimension of the stratum, in the case of a stratum mapping
  to a non-simplicial point of $\Delta_{\Fl(\ul n)}$; see Example
  \ref{ex:singstrat}.
\end{proof}

The following lemma was used in the proof of Proposition
\ref{prop:subm}. It constructs a $n \times n$ matrix in $\u(n)$ given its
eigenvalues and a $(n-1) \times (n-1)$ minor. 

\begin{lemma} \label{lem:invphi}
Given $n \in \N$, $b_1 \geq \dots \geq b_{n-1}$, and $I \subset \{1,\dots, 2n-2\}$, let $\Delta_I$ be the subset of 
%
\[ \{(a_1,\dots, a_n) \in \R^n: a_1 \geq b_1 \geq \dots \geq b_{n-1}\geq a_n\},\]
%
  where the $j$-th inequality is an equality if $j \in I$, and it is a strict inequality if $j \notin I$. Then, 
  \begin{enumerate}
  \item \label{part:existr}
    there is a smooth map
    \[r :\Delta_I \to (\R_{\geq 0})^{n-1} \times \R, \quad a \mapsto
      (r_1,\dots,r_n) \]
    such that for any $a \in \Delta_I$
    and a matrix
    \begin{equation}
      \label{eq:matAdef}
    A:=
    \begin{pmatrix}
      b_1 & & 0 & \bar{z}_1\\
      & \ddots & & \vdots\\
      0 & & b_{n-1} & \bar{z}_{n-1}\\
      z_1 & \hdots & z_{n-1} & r_n(a)
    \end{pmatrix}  
    \end{equation}
    with $|z_i|^2=r_i(a)$ for $i=1,\dots,n-1$, the eigenvalues are $a_1,\dots, a_n$.
  \item For any $i$, $r_i=0$ only if either $b_i=a_i$ or
    $b_i=a_{i+1}$.
  \item If $b_1 > \dots >b_{n-1}$, then the map $r$ is uniquely determined. 
  \item \label{part:mapz}
    There is a smooth map
    \[z=(z_1,\dots,z_{n-1},r_n) : \Delta_I \to \C^{n-1} \times \R\]
    such that given $a \in \Delta_I$, a matrix of the form
    \eqref{eq:matAdef} with bottom row $z(a)$ has eigenvalues
    $a_1,\dots,a_n$. Furthermore, given a
    matrix $A$ of the form \eqref{eq:matAdef} with eigenvalues $a \in \Delta_I$, $z$ can be chosen so that its image contains
    $(z_1(A), \dots,\\ z_{n-1}(A), r_n(A))$. 
  \end{enumerate}
\end{lemma}

\begin{proof}
{\sc Step 1}: We first show that for a given point in $\Delta_I$, $(r_1,\dots,r_n)$ satisfies a system of linear equations. 
  Observe that
  \begin{equation}
    \label{eq:detexp}
    \det(A-t\Id)=  \left(\prod_{i=1}^{n-1} (b_i - t)\right)
 \left(r_{n} - t - \sum_{j=1}^{n-1} \frac {|z_j|^2} {b_j-t}  \right). 
  \end{equation}
  We set
  \[r_{n}:= \sum_i a_i - \sum_i b_i,\]
The system of $n-1$ equations 
  \begin{equation}
    \label{eq:deteig}
    \det(A-a_i\Id)=0 \quad \text{for} \enspace i=1,\dots,n-1 
  \end{equation}
  with the substitution  $|z_i|^2=r_i$ is linear in $r_1,\dots,r_{n-1}$. The
  coefficient matrix for the variables $r_i$, $i=1,\dots,n-1$ is
  \begin{equation} \label{eq:cmat}
  \left(\prod_{1 \leq k \leq n-1:k \neq j}(b_k-a_i) \right)_{i,j} ,\end{equation} %
  and
  \[ \det( \eqref{eq:cmat}) = \pm \prod_{i<j} (a_i-a_j) \prod_{i<j} (b_i-b_j).\]

  \vskip .05in
{\sc Step 2}:   We prove the lemma in the case when 
  $I=\emptyset$, that is, 
  \[a_1 > b_1 > a_2 > \dots > a_{n-1} > b_{n-1} > a_{n}.\]
  We can uniquely solve \eqref{eq:detexp} for $(r_1,\dots,r_{n-1})$, since the determinant of \eqref{eq:cmat} is non-zero. 
  It remains to show that $r_1,\dots,r_{n-1}$ are non-negative.
  Firstly, we observe that
  $r^{-1}(\partial((\R_{\geq 0})^{n-1} \times \R)) \subset \partial
  \Delta_\emptyset$. Indeed, in the matrix $A$, if $z_i=0$ then $b_i$ is an
  eigenvalue, and so $b_i$ is equal to either $a_i$ or $a_{i+1}$.
  Secondly, by the interlacing property the image
  $r(\Delta_\emptyset)$ intersects the interior of
  $(\R_{\geq 0})^{n-1} \times \R$.  We may then conclude that
  $r(\Delta_\emptyset)$ is in fact contained in
  $(\R_+)^{n-1} \times \R$. The result is thus proved for the case
  $I=\emptyset$.

  \vskip .05in {\sc Step 3}: Next, we consider the case
  $b_1>\dots>b_{n-1}$, but with non-empty $I$.  If $b_i=a_i$
  resp. $b_i=a_{i+1}$, then we set $r_i=0$.  Define
  $\hat a:=(a_1,\dots \hat a_i,\dots, a_n)$ resp.
  $(a_1,\dots \hat a_{i+1},\dots, a_n)$, and
  $\hat b:=(b_1,\dots \hat b_i,\dots, b_{n-1})$.  The above construction
  applied to $(\hat a, \hat b)$ produces
  $\hat r:=(r_1,\dots,\hat r_i, \dots, r_n)$.  If
  $b_1>\dots>b_{n-1}$, the above deletions can be performed in a
  unique way, and a unique $(r_1,\dots,r_{n-1})$ exists for each
  $a \in \Delta_I$.  For a general $I$, a sequence of deletions can
  still be performed to produce the map $r$ and prove  
  \eqref{part:existr}, but there is a choice to be made whenever
  $b_i=b_{i+1}$. 

  \vskip .05in
{\sc Step 4}: 
  We now consider a general stratum $I$ and show that the map $z$ in
  \eqref{part:mapz} exists.  The term $r_n$ is uniquely determined, so we focus on
  $z_1,\dots, z_{n-1}$. For the given matrix $A$ if $z_i(A)=0$, we may
  set $z_i \equiv 0$ on $\Delta_I$ by the deletion algorithm of Step 3. Therefore,
  we may focus on the case that $z_i(A) \neq 0$ for all $i$. For any 
  $a \in \Delta_I$, the set of positive solutions $(r_i)_i$ of the
  system of linear equations \eqref{eq:detexp}  is a cone $C_a$ whose dimension is
  independent of $a \in \Delta_I$ and is non-empty by Steps 1 and 2 above. 
   Therefore, the map
  $(r_1,\dots,r_{n-1}) : \Delta_I \to (\R_{>0})^{n-1}$ can be chosen so
  that the image contains a given point in one of the cones $C_a$.
\end{proof}

In the rest of the Section, we prove the following result, which was originally proved by Nishinou-Nohara-Ueda \cite{nnu:degen}.

\begin{theorem}\label{thm:flagpot}
  Let $X:=\Fl(\ul n)$ be a flag variety. There is a symplectic form on
  $X$ for which the Gelfand-Cetlin system of $X$ has a unique monotone
  Lagrangian fiber $L_0$.  The \ainfty algebra $CF(L_0)$ is
  weakly unobstructed and the naive potential of $CF(L_0)$ consists of one term
  for every facet of the Gelfand-Cetlin polytope of $X$.
\end{theorem}

As part of the proof of Theorem \ref{thm:flagpot}, we compute the disk
potential (defined in \eqref{eq:wnaive}) of $\Fl(\ul n)$ as
    
\begin{multline}\label{eq:grpot}
  W(y_{i,k})_{i,k}= \left (\sum_{k \leq i : (i+1,k) \in
      \on{Free}(\Fl(\ul n))} \frac{y_{i+1,k}}{y_{i,k}}
    + \sum_{k \leq i : (i+1,k+1) \in \on{Free}(\Fl(\ul n))}
    \frac{y_{i,k}}{y_{i+1,k+1}} \right.\\ \left. + \sum_{1 \leq j \leq
      r-1} y_{\li_j, \k_j}
    + \sum_{2 \leq j \leq r} \frac 1 {y_{\li_j,\k_j}} \right)q.
\end{multline}
Here the index $(\li_j,\k_j)$ is the position of the lowest element in
the $j$-th block, see \eqref{eq:blocklowest}. In the above expression
\eqref{eq:grpot} for the potential, each term corresponds to a facet
\eqref{eq:facetdef} of the Gelfand-Cetlin polytope -- the first
resp. second summation corresponds to facets $\F_{i,k,\ldiag}$
resp. $\F_{i,k,\rdiag}$ where both variables in the equality defining
the facet are free, the third resp. fourth summation corresponds to
facets $\F_{i,k,\ldiag}$ resp. $\F_{i,k,\rdiag}$ where one of the
variables in the defining equality is non-free and is the lowest
element of the $j$-th block. The exponent of the Novikov formal
variable $q$ is $1$, because the symplectic form (given by a choice of
$(\Lam_i)_i$) and Lagrangian torus $L_0$ are chosen so that the pair
$(\Fl(\ul n),L_0)$ is monotone and the area of any disk of Maslov
index $2$ is $1$.  We will show for every facet, there is a single
holomorphic disk of Maslov index two intersecting the corresponding
divisor in $\Fl(\ul (n))$.

\begin{example}
  The Gelfand-Cetlin inequalities of the Grassmannian $\Gr(2,4)=\Fl(2,2)$ are given by
  \begin{equation}
    \begin{alignedat}{13}
      \Lam_1 &&&& \Lam_1 &&&& \Lam_2 &&&& \Lam_2  \\
      & \uge && \dge && \uge && \dge && \uge && \dge & \\
      && \Lam_1 &&&& \Phid 3 2 &&&& \Lam_2 && \\
      &&& {\color{grey} \uge} && {\color{grey} \dge} && {\color{grey} \uge} && {\color{grey} \dge} &&& \\
      &&&& \Phid 2 1 &&&& \Phid 2 2 &&&& \\
      &&&&& {\color{grey} \uge} && {\color{grey} \dge} &&&&& \\
      &&&&&& \Phid 11 &&&&&
    \end{alignedat}
    \label{eq:GCineq1}
  \end{equation}
  and the inequalities in grey correspond to facets.
  The potential is 
  \[W(y)=\left(\left(\frac{y_{2,1}}{y_{1,1}} + \frac{y_{3,2}}{y_{2,2}} \right) + \left(\frac{y_{1,1}}{y_{2,2}} + \frac{y_{2,1}}{y_{3,2}} \right) + \left(\frac 1 {y_{2,2}} \right) + \left(y_{2,1}\right)\right)q,\]
  where the terms are parenthesized to reflect the summation groupings from \eqref{eq:grpot}.
\end{example}
    
 \begin{remark}
   In the case of a full flag variety ($r=n$), the least order terms
   of the potential in \eqref{eq:grpot} coincide with the potential
   introduced by Givental in \cite{giv}. In the case of the
   Grassmannian ($r=2$), the analysis of the potential
   \eqref{eq:grpot} by Castronovo in \cite{Castr:Gr} implies that the
   Lagrangian $L_0$ has non-trivial Floer homology.
\end{remark}

\begin{figure}[h]
  \centering\scalebox{.8}{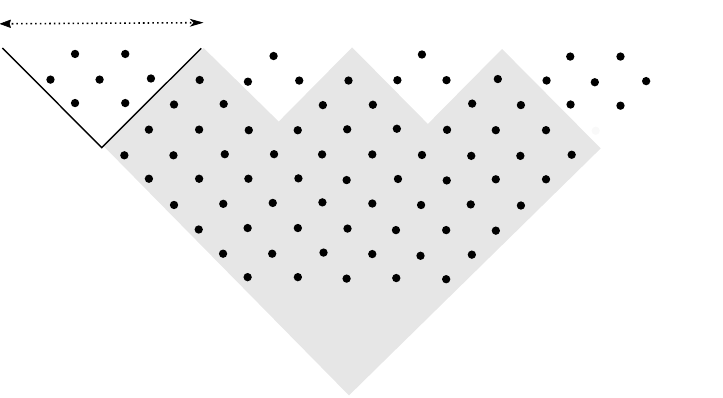}
  \caption{Center of the Gelfand-Cetlin polytope $\Delta_{\Fl(\ul n)}$.}
  \label{fig:gc-center}
\end{figure}

\begin{proof}[Proof of Theorem \ref{thm:flagpot}]
  First, we describe the monotone symplectic form and the monotone
  Lagrangian in the flag variety $\Fl(\ul n)$.  Let $\Fl(\ul n)$ be
  the coadjoint $U(n)$-orbit of the element
  \[\on{diag}(\sqrt{-1}(\underbrace{\Lam_1,\dots, \Lam_1}_{n_1 \text{
        times}}, \underbrace{\Lam_2,\dots, \Lam_2}_{n_2 \text{
        times}},\dots, \underbrace{\Lam_r,\dots, \Lam_r}_{n_r \text{
        times}})) \in \u(n)^\dual,\]
  where
  \[\Lam_j:=\li_j-1-2(\k_j-1),\]
  and $(\li_j,\k_j)$ is the index of the lowest element in the $j$-th
  block of non-free variables, see \eqref{eq:blocklowest}.  The
  Kostant-Kirillov form on $\Fl(\ul n)$ is monotone (see
  \cite[p653-654]{nnu:degen}) satisfying
  \begin{equation}
    \label{eq:c1om-monot}
    c_1(T\Fl(\ul n))=[\om_\Lam]. 
  \end{equation}
  The Gelfand-Cetlin polytope $\Delta_{\Fl(\ul n)}$ has a center
  $\lam$ (see Figure \ref{fig:gc-center}) given by
  \[\lam_{i,k}=(i-1)-2(k-1)\]
  with coordinates satisfying
  \begin{equation}
    \label{eq:lamik}
    \lam_{i,k} - \lam_{i-1,k} = 1, \quad \lam_{i-1,k}-\lam_{i,k+1}=1, \quad \forall i,k.  
  \end{equation}
  The fiber
  \[L_\lam:=\Phinv(\lam)\]
  is a monotone Lagrangian with monotonicity constant $\hh$, that is,
  \begin{equation}
    \label{eq:monot-constt}
    \forall u:(\mathbb D^2, \partial \mathbb D^2) \to (X,L_\lam) \enspace \om(u)=\hh I(u).   
  \end{equation}
  Indeed, using \eqref{eq:c1om-monot} and
  $\pi_2(X,L_\lam)/\pi_2(X) = \pi_1(L_\lam)$ from the homotopy long
  exact sequence, it is enough to verify \eqref{eq:monot-constt} on a
  set $S$ of disks $u$ bounding $L_\lam$ for which
  $\{[\partial u]: u \in S\}$ generates $\pi_1(L_\lam)$; and for a
  disk of Maslov index two that has a single intersection with a
  divisor corresponding to a facet of the Gelfand-Cetlin polytope, the
  symplectic area is $1$.  Cho-Kim \cite[Theorem B]{chokim} give an
  alternate proof of the monotonicity of the Lagrangian torus
  $L_\lam$.

\begin{figure}[ht]
  \centering \scalebox{.8}{
\begingroup%
  \makeatletter%
  \providecommand\color[2][]{%
    \errmessage{(Inkscape) Color is used for the text in Inkscape, but the package 'color.sty' is not loaded}%
    \renewcommand\color[2][]{}%
  }%
  \providecommand\transparent[1]{%
    \errmessage{(Inkscape) Transparency is used (non-zero) for the text in Inkscape, but the package 'transparent.sty' is not loaded}%
    \renewcommand\transparent[1]{}%
  }%
  \providecommand\rotatebox[2]{#2}%
  \newcommand*\fsize{\dimexpr\f@size pt\relax}%
  \newcommand*\lineheight[1]{\fontsize{\fsize}{#1\fsize}\selectfont}%
  \ifx\svgwidth\undefined%
    \setlength{\unitlength}{117.13438139bp}%
    \ifx\svgscale\undefined%
      \relax%
    \else%
      \setlength{\unitlength}{\unitlength * \real{\svgscale}}%
    \fi%
  \else%
    \setlength{\unitlength}{\svgwidth}%
  \fi%
  \global\let\svgwidth\undefined%
  \global\let\svgscale\undefined%
  \makeatother%
  \begin{picture}(1,0.74614156)%
    \lineheight{1}%
    \setlength\tabcolsep{0pt}%
    \put(0,0){\includegraphics[width=\unitlength,page=1]{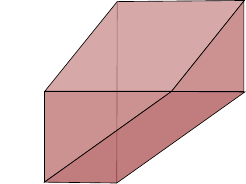}}%
    \put(-0.00644391,0.20392934){\color[rgb]{0,0,0}\makebox(0,0)[lt]{\lineheight{1.25}\smash{\begin{tabular}[t]{l}$X$\end{tabular}}}}%
    \put(0,0){\includegraphics[width=\unitlength,page=2]{fl3cut.pdf}}%
  \end{picture}%
\endgroup%
}
  \caption{ On
    $X:=\Fl(1,1,1)$ (from Figure \ref{fig:fl3}), $\PP$ consists of a single cut that separates the non-simplicial vertex.}
  \label{fig:fl3cut}
\end{figure}

We describe a multiple cut $\PP$, first assuming that $\Fl(\ul n)$ is the complete flag variety.
The cut 
$\PP$ is given by a
collection of transversely intersecting multiple cuts $\PP_k$ each
defined on the polytope
\[\Delta_k^\W:=\{(\Phi_{k,i})_i \in \R^k : \Phi_{k,i-1} \geq
  \Phi_{k,i} \enspace \forall i\},\]
and pulled back to $\Delta$.  The polyhedral decomposition $\PP_k$ has
a single top-dimensional polytope that does not intersect
$\partial \Delta_k^\W$ that is defined as
\begin{equation}
  \label{eq:Pk0}
P_{k,0}:=\{(\Phi_{k,i})_i \in \R^k : \Phi_{k,i-1} -
\Phi_{k,i} \geq  \eps_{k,i} \}   
\end{equation}
where the set of positive constants $(\eps_{k,i})_i$ is small, generic
and distinct.
Any face of $P_{k,0}$ is a translate of a face of $\Delta_k^\W$, and 
there is a natural bijective correspondence
\[f_k : \on{Faces}(P_{k,0}) \to \on{Faces}(\Delta_k^\W). \]
To describe $\PP_k$, it is enough to 
 describe facets $P \in \PP_k$ that are not contained in
$P_{k,0}$. Let $F \subset P_{k,0}$ be a face with codimension
$n_F \geq 2$, and let $(F_i \supset F)_i$ be faces of $P_{k,0}$ whose
dimension is $1$ higher than $F$.  The polyhedral decomposition
$\PP_k$ contains  a collection of facets $P_{F, F_j}$ in the exterior
of $P_{0,k}$, each of which contains $F$ and intersects the face
$f_k(F_j)$ orthogonally. Thus, corresponding to every proper face $F$
of $P_{0,k}$, there is a top-dimensional polytope $P_F \in \PP_k$
whose facets are
$\{P_{F,F'}: F' \subset F\} \cup \{P_{F', F}: F' \supset F\}$, and additionally, $F$ if $\codim(F)=1$. 
Finally, the multiple cut $\PP$ on the flag variety is the
intersection (in the sense of Example \ref{ex:singles} in the next
chapter) of the multiple cuts $\Phi_k^{-1}(\PP_k)$, where $\Phi_k$ is
the projection map
\[\Phi_k:=(\Phi_{k,i})_i : \Delta_{\Fl(\ul n)} \to \Delta_k^\W.\]
 The multiple cut $\PP$ is
well-defined by the following reasons: The cut locus $\PP_k$ intersects the boundary
$\partial \Delta^\W_k$ orthogonally and therefore the torus action
generating the cut is smooth by Proposition \ref{prop:smooth} and
Remark \ref{rem:wsm}; the torus action is also free since by
Proposition \ref{prop:subm} the moment map generating the action is
submersive on each $\GC$-stratum that the cut intersects. In the case
of a partial flag $\Fl(\ul n)$, the polytope $\Delta_k^\W$ is a subset
of $\R^{\ell_k}$, $\ell_k \leq k$, parametrized by Gelfand-Cetlin
variables $\Phi_{k,i}$ that are not constrained to be constants. The
rest of the construction is the same as the case of complete flags.

  The cut space $X_{P_0}$ corresponding to
  the polytope
  \begin{equation}
    \label{eq:P0def}
   P_0:=\cap_k \Phi_k^{-1}(P_{0,k}) 
  \end{equation}
  is contained in
$F^\reg$, and therefore, by Proposition \ref{prop:gcresult}, $X_{P_0}$
is a toric manifold whose moment map is the Gelfand-Cetlin map $\Phi$.
Assuming that the constants $\eps_{k,i}$ are small enough, $X_{P_0}$
contains the Lagrangian $L_\lam$. The cut loci are close to
codimension two corners of $\Delta_{\Fl(\ul n)}$.

\begin{figure}[ht]
  \centering \scalebox{.8}{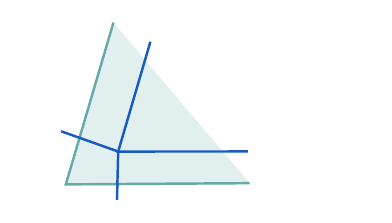}
  \caption{The cut $\PP_3$ on $\Fl(1,1,1,1)$, which is defined on the $3$-dimensional polytope $\Delta^\W_3$, but is projected to $2$ dimensions in this figure.}
  \label{fig:fl4}
\end{figure}

\begin{example}
  \label{ex:fl3cut}
  We explicitly write down the multiple cut $\PP$ for some low dimensional examples.
  \begin{enumerate}
  \item In $\Fl(1,1,1)$, $\PP$ consists of a single cut, which is a 
    pull back of the cut in $\PP_2$, namely $\{\Phi_{1,2}-\Phi_{2,2}= \eps_{1,2}\}$ on $\Delta_2^\W$. The cut separates the non-simplicial
    vertex as in Figure \ref{fig:fl3cut}. Similarly, in
    $\Gr(2,4)=\Fl(2,2)$, $\PP$ consists of a single cut pulled back
    from $\PP_2$.
  \item In $\Fl(1,1,1,1)$, $\PP_2$ consists of a single cut namely , and $\PP_3$, shown in Figure
    \ref{fig:fl4}, consists of the following codimension one polytopes 
    \begin{align*}
      F_1&=\{\Phi_{1,3}-\Phi_{2,3}=\eps_{1,3}\}, \quad F_2= \{\Phi_{2,3}-\Phi_{3,3}=\eps_{2,3}\},\\
      P_{F,F_1}&=\{\Phi_{1,3}+\Phi_{2,3}-2\Phi_{3,3}=\eps_{1,3} + 2\eps_{2,3}\}, \\
      P_{F,F_2}&=\{2\Phi_{1,3}-\Phi_{2,3}-\Phi_{3,3}=2\eps_{1,3} + \eps_{2,3}\}.
    \end{align*}
  \end{enumerate}
\end{example}

\begin{remark}
  Note that $\ol X_{P_{0\theta}}$ is not the inverse image of the
  facet $\theta$ of the Gelfand-Cetlin polytope. Rather, the space
  $\ol X_{P_{0\theta}}$ is the inverse image of a facet of the toric
  manifold $\ol X_{P_0}$, which is a toric smoothing of the
  Gelfand-Cetlin toric variety.  The subspace
  $\Phinv(\theta) \subset \Fl(\ul n)$ is, in general, not a smooth
  divisor.
\end{remark}

We continue the proof of Theorem \ref{thm:flagpot}.  Unobstructedness
is a consequence of the monotonicity of the Lagrangian: Let $J^\nu$ be
a family of neck-stretching almost complex structures on the flag
manifold $X=\Fl(\ul n)$.  Let $u_\nu$ be a $J^\nu$-holomorphic disk
with a single output $x \in \crit(F:L_\lam \to \R)$. By monotonicity
of $(X,L_\lam)$ the disk part of $u_\nu$ has Maslov index
$I(u_\nu) \geq 2$. The index of $u_\nu$ (including the treed
trajectory) is
\[\dim(L_\lam) + 1 - (\dim(L_\lam) - i(x)) + I(u_\nu) - \Aut(\mathbb
  D^2)=0.\]
Here $i(x)$ is the Morse index of the critical point $x$ of a Morse
function $F : L_\lam \to \R$ on the Lagrangian. The only possibility
then is $I(u_\nu)=2$ and $i(x)=\dim(L_\lam)$, that is, $x^\blackt$ is
the maximum point of $F$. Therefore
\[\m^0_{CF(L_\lam,J_\nu)}=W x^\blackt, \quad \text{for some
    $W \in \Lam_{>0}$}.\]
Weak unobstructedness now follows from Lemma \ref{lem:unobs-cond}.
Note that since $(X,L_\lam)$ is monotone, the count of disks
contributing to $\m^0$ is the same for all $J^\nu$.  The count also
stays the same in the limit broken almost complex structure $\JJ$ on
$\XX_\PP$.

We compute the potential $W$ by counting $\JJ$-holomorphic broken
disks in $\XX$ of Maslov index $2$. Since $(X,L_\lam)$ is monotone, the disks have $\om$-area $1$. 
In the Gelfand-Cetlin polytope, the cut locus $\PP$ lies near corners that have codimension at least $2$. Therefore, any disk in $(X_{P_0},L)$ that intersects a relative divisor has area
$2-\eps_{i,j}$, which is larger than $1$, if the constants $\eps_{i,j}$ are small enough.
Indeed, the disk area can be explicitly determined as follows: By \eqref{eq:Pk0}, \eqref{eq:P0def}, a relative divisor
of $X_{P_0}$ is of the form $D_{i,j}=\{\Phi_{i,j}-\Phi_{i+1,j}=\eps_{i,j}\}$, and the area of a disk bounding
$L_\lam$ (where $\lam$ satisfies \eqref{eq:lamik}) that intersects $D_{i,j}$ is at least $2-\eps_{i,j}$.
Therefore, the disks contributing to the potential are unbroken, and are incident on non-relative toric divisors of $X_{P_0}$, which are the Gelfand-Cetlin divisors.
For an appropriate choice of spin structure on $L_\lam$, 
the Maslov index $2$ disk intersecting each of these divisors has a positive orientation sign by \cite[p22]{chooh:toric}, and contributes $+1$ to the potential. 
Therefore, the disk potential of $CF_\br(L_\lam)$ on the broken manifold $\XX$
 is
 \begin{multline}
   \label{eq:grwmonot}
   W(y_{i,k})_{i,k}= \left (\sum_{k \leq i : (i+1,k) \in
       \on{Free}(\Fl(\ul n))} \frac{y_{i+1,k}}{y_{i,k}}
     + \sum_{k \leq i : (i+1,k+1) \in \on{Free}(\Fl(\ul n))}
     \frac{y_{i,k}}{y_{i+1,k+1}} \right.\\ \left. + \sum_{1 \leq j
       \leq r-1} y_{\li_j, \k_j}
     + \sum_{2 \leq j \leq r} \frac 1 {y_{\li_j,\k_j}} \right)q.
 \end{multline}
 Here, we recall that $\on{Free}(\Fl(\ul n))$ is the set of variables
 in the Gelfand-Cetlin system of $\Fl(\ul n)$ that are not fixed by
 the interlacing inequalities, and for any $j$, $(\li_j,\k_j)$ is the
 lowest element in the $j$-th block of non-free variables in the
 Gelfand-Cetlin system, as in \eqref{eq:blocklowest}.  This finishes
 the proof of Theorem \ref{thm:flagpot}.
\end{proof}
 
\section{Counting curves in the plane}
\label{sec:mikh} 

Mikhalkin's tropical curve counting \cite{Mikh:R2} is the first
example of applying tropical techniques to solving an enumerative
problem in algebraic geometry.  A tropical curve in \cite{Mikh:R2} is
a map from a graph to $\R^2$ that satisfies a balancing condition at
nodes; and these curves correspond to holomorphic curves in $\P^2$. In
this section, we show a correspondence between Mikhalkin's tropical
curves and broken maps.  We consider a degeneration of $\P^2$ by a
multiple cut so that Mikhalkin's tropical curves correspond to
tropical graphs of broken maps. In particular, Mikhalkin's curve
represents the part of the tropical graph that lies in the dual
polytope $P_0^\dual$ of a zero-dimensional polytope $P_0$ occurring in
the multiple cut.
In a sequel, we extend the formula to counts in almost toric
manifolds.

\begin{definition}\label{def:mikh-curve}
  A \em{tropical curve} is a map 
\[ h: \Gamma \to \R^2 \] %
from a graph $\Gamma$ (some of whose edges $e \in \Edge(\Gamma)$, called \em{{leaves}},  are incident on just one vertex
$v \in \Ver(\Gamma)$ instead of two) to $\R^2$ such that
\begin{itemize}
\item any edge $e \in \Edge(\Gamma)$ maps to a  line
    parallel  to a vector 
  $\cT(e) \in \Z^2$ which is called the \em{{direction} of the edge};  and
\item at any vertex $v \in \Ver(\Gamma)$ a \em{balancing
      condition} is satisfied, namely that the sum of the {direction}s
    $\cT(e)$ of the edges $e \ni v$ emanating out of $v$ is equal to $0$ :
  \begin{equation}
    \label{eq:orig-balance}
    \sum_{v \in e}\cT(e)=0.
  \end{equation}
\end{itemize}
The \em{multiplicity $\mu_e$ of an edge} $e$ is a positive integer
such that $\frac {\cT(e)}{\mu_e}$ is a primitive vector
$w_e \in \Z^2$, which is called the \em{primitive {direction}}.
\footnote{In Mikhalkin's notation in \cite{Mikh:R2}, the primitive
  {direction} $w_e$ is called \em{slope}.} \index{Primitive direction}
\index{Multiplicity of an edge} If $e \in \Edge(\Gamma)$ is a leaf,
$h(e)$ is a semi-infinite line, otherwise $h(e)$ is a line segment.
\end{definition}

We introduce basic terminology for tropical 
curves in the plane. Let $\Delta \subset \R^2$ be a simple polytope. A tropical curve
$h: \Gamma \to \R^2$ is \em{adapted to $\Delta$} if the {direction}
$\cT(e)$ of any leaf $e \in \Edge(\Gamma)$ is an outward normal of a
facet of $\Delta$. The discussion in Mikhalkin
focuses \label{rep:focuses} on $\P^2$ with moment polytope
$\Delta_{\P^2}$ whose outward normals are $\nu=(-1,0), (0,-1),
(1,1)$.  For a tropical curve generically adapted to $\Delta_{\P^2}$,
the \em{degree} of the curve is defined as the number of {leaves}
(counted with multiplicity) that are normal to a fixed facet.  By the
balancing condition, the number is the same for any of the three
facets. The \em{genus} of a tropical curve is the first Betti number
$b_1(\Gamma)$ of the domain graph $\Gamma$.

\begin{definition} \label{def:mi-generic} A set of points
  $x_1,\dots,x_k$ is \em{tropically generic} if any genus $g$
  tropical curve $h: \Gamma \to \R^2$ whose image $h(\Gamma)$ contains
  $x_1,\dots,x_k$,
  \begin{enumerate}
\item all the vertices of $\Gamma$ are trivalent, 
\item the images of vertices $h(\Ver(\Gamma))$ are disjoint from
  $x_1,\dots,x_k$, and
\item the {direction} $\cT(e)$ of any leaf $e \in \Edge(\Gamma)$ is
  primitive.
\end{enumerate}  
\end{definition}
\begin{definition}
  {\rm(Multiplicity of Mikhalkin's graphs)} Let $\Gamma$ be a
  Mikhalkin graph with only trivalent vertices, and let $v$ be a
  vertex in $\Gamma$ whose incident edges have {direction}s $\mu_1$,
  $\mu_2$, $\mu_3$. The multiplicity of the vertex $v$ is
  \[\mult(v):=|\det(\mu_1 \mu_2)|,\]
  which is the area of the parallelogram spanned by the vectors
  $\mu_1$, $\mu_2$. \footnote{Any two of the three vectors $\mu_1$,
    $\mu_2$, $\mu_3$ may be used for the definition. The result is the
    same because of the balancing condition $\sum_i \mu_i=0$.}  The
  multiplicity of the graph $\Gamma$ is
  \[\mult(\Gamma):=\prod_{v \in \Ver(\Gamma)} \mult(v).\]
\end{definition}

Mikhalkin \cite{Mikh:R2} shows that the number of curves in $\P^2$ of
degree $d$ and genus $g$ passing through $3d - 1 + g$ tropically
generic points can be computed by counting tropical curves with
multiplicity.

\begin{figure}[ht]
  \centering \scalebox{.8}{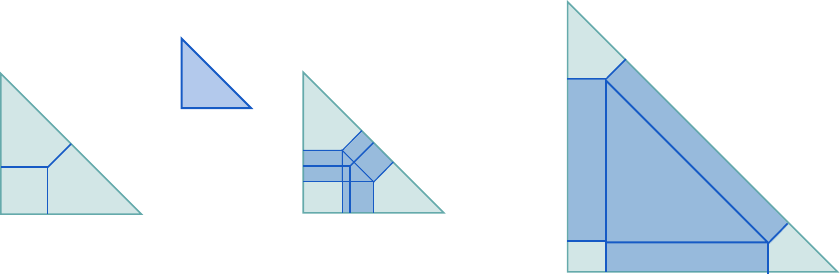}
  \caption{A multiple cut of $\P^2$.}
  \label{fig:stretchp2}
\end{figure}

The following result, Proposition 
\ref{prop:tropbrok} below,  says that in the genus zero case, there is a
bijective correspondence between Mikhalkin graphs adapted to
$\Delta_{\P^2}$ and broken maps in $\P^2$ with respect to the multiple
cut shown in Figure \ref{fig:stretchp2}.  This multiple cut is not
``allowed'' because the torus action corresponding to the cut $P_1$ is
not free at the $S^1$-orbit where $P_1$ intersects the toric divisor; a similar example was
given in Figure \ref{fig:orb-cut}.  However, the broken maps
corresponding to Mikhalkin graphs produced in our proof are not
affected by this issue, as we later explain in Remark
\ref{rem:orb-fix}.
  
Before stating the bijection, we describe it intuitively and 
introduce notation for point constraints in a broken manifold.   The degeneration corresponding to $\PP$ has the effect of enlarging the almost complex structure in the neighbourhood of $\Phinv(P_0)$ where $P_0 \in \PP$ is the point polytope. The dual polytope of $P_0$ is
\begin{equation}
  \label{eq:mikhpoly}
  \Delta_{\P^2} \simeq P_0^\dual.  
\end{equation}
For a broken map in $\P^2$, the Mikhalkin graph is the purely
tropical part, that is, it consists of the part of the subgraph of the
tropical graph lying in $P_0^\dual$.  For these components the map part
is fully determined by the graph since $P_0$ is $0$-dimensional.  For
each of the one-dimensional polytopes $P_i$, $i=1,2,3$, in the
multiple cut, the cut space $X_{P_i}$ is a toric divisor of $\P^2$
minus a neighborhood of fixed points. The thickening $\XX_{P_i}$ is a
trivial $\P^1$-fibration over $X_{P_i}$, and corresponding to a leaf
of the Mikhalkin graph normal to the $i$-th facet, the
broken map has a map component mapping to a fiber of $\XX_{P_i}$.
The point constraints for the broken map lie in the neck piece
$\XX_{P_0}$; these points in the broken manifold are written using the
following notation:

\begin{definition}
  {\rm(An $\XX$-point)} Given a broken manifold $\XX$, an \em{
    $\XX$-point} is a tuple $\ul x:=(P_x,\cT_x,x)$ consisting of a
  polytope $P_x \in \PP$, a tropical position $\cT_x \in P_x^\dual$,
  and a point $x \in \XX_{P_x}$.  We say that a broken map
  $u:(C,z) \to \XX$ satisfies the point constraint
  $u(z)=\ul x=(P_x,\cT_x,x)$ if the component $C_v \subset C$
  containing the marking $z$ satisfies
\[P(v)=P_x, \quad \cT(v)=\cT_x, \]
and $u(z)=x$.
\end{definition}

\begin{proposition} \label{prop:tropbrok} {\rm(Tropical curves as
    broken maps)} Let $\XX$ be the broken manifold obtained by
  applying the multiple cut in Figure \ref{fig:stretchp2} to $\P^2$.
  Consider a set of $\XX$-points $\ul x_1,\dots, \ul x_{3d-1} \in \XX$
  lying in the piece $\XC_{P_0} \subset \XX$. That is,
  $\ul x_i=(P_0,\cT_{x_i},x_i)$.  There is a bijective correspondence
  between the set of genus zero broken maps passing through the
  $\XX$-points $\{\ul x_i\}_i$ and the set of genus zero Mikhalkin
  graphs in $\R^2$ passing through the points
  $\{\cT_{x_i} \in \R^2\}_i$.
\end{proposition}

\begin{proof}
  Given a Mikhalkin graph $h: \Gamma \to \R^2$ adapted to
  $\Delta_{\P^2}$, we first construct a tropical graph (as in
  Definition \ref{def:tropgraph}) for the polyhedral decomposition
  $\PP$ of $\P^2$ from Figure \ref{fig:stretchp2}.  The tropical graph
  is an augmentation of $\Gamma$ and is denoted by
  \begin{equation}
    \label{eq:gamma-aug}
  \Gamma_{\aug} \supset \Gamma.  
  \end{equation}
  For all vertices $v \in \Ver(\Gamma)$, we assign $P(v):=P_0$. 
  For every leaf $e \in \Edge(\Gam)$, the
  augmented graph $\Gamma_{\aug}$ contains an extra vertex $v_e$ on
  which $e$ is incident. If $e$ in $\Gamma$ intersects the facet
  $F_i \subset \Delta_{\P^2}$, then the polytope $P(v_e)$ is $P_i$.
  The set of univalent vertices thus added is denoted by
  \[\Ver_1(\Gamma_\aug).\]
  Lastly, if an edge $e \in \Edge(\Gamma)$ of the Mikhalkin graph
  passes through a point constraint $x \in \R^2$ then in $\Gamma_\aug$
  we subdivide $e$ into two edges $e_1$ and $e_2$ by inserting a new
  vertex $v_x$ with $P(v_x):=P_0$.    The set of vertices
  $v \in \Ver(\Gamma_\aug)$ with markings is denoted by
  \[\Ver_\to(\Gamma_\aug).\]
  Thus,
  \[\Ver(\Gamma_\aug)= \Ver(\Gamma) \cup \Ver_\to(\Gamma_\aug) \cup \Ver_1(\Gamma_\aug).\]
  The {direction}s of edges in $\Gamma_\aug$ are the same as their {direction}s in
  $\Gamma$. The tropical positions for the vertices in $\Gamma_\aug$
  are given by the map $h:\Gamma \to \R^2$ on the Mikhalkin graph.  See
  Figure \ref{fig:mikh} for an example.

  We describe a way of orienting the edges in $\Gamma_\aug$ which is
  useful in the rest of the proof, called the \em{marking
    orientation},  that satisfies the
  following conditions:
  \begin{itemize}
  \item For a vertex $v \in \Ver_{\to}(\Gamma_\aug)$ containing a
    marking, both incident edges (corresponding to nodes) are
    outgoing.
  \item For a trivalent vertex $v \in \Ver(\Gamma)$ occurring in
    Mikhalkin's graph, there are two incoming and one outgoing edge.
  \item For a vertex $v \in \Ver_1(\Gamma_\aug)$ the only incident
    edge is incoming.
  \end{itemize}
  It is easy to verify that marking orientations can be assigned to
  all the edges of $\Gamma_\aug$, see Figure \ref{fig:mikh} for an
  example.

  The marking orientation has the following interpretation: Consider
  an ordering of the vertices of $\Gamma_\aug$ that respects the
  orientation, that is, an edge $(v_+,v_-)$ points towards the vertex
  $v_-$ that occurs later in the ordering.  For a vertex $v$, if the
  tropical positions of vertices prior to $v$ are fixed, then there is
  a unique possible tropical position for $v$.

  Next, we describe the map at each vertex.  Choose an ordering of the
  vertices $v \in \Ver(\Gamma_\aug)$ that respects the marking
  orientation as described in the previous paragraph, and define the
  maps $(u_v)_v$ in that order.  For a vertex $v \in \Ver(\Gamma)$
  whose incident edges have {direction}s $\mu_1$, $\mu_2$, $\mu_3$ the map
  is
  \[u_v : \P^1 \bs \{0,1,2\} \to \XC_{P_0} \simeq (\C^\times)^2, \quad z \mapsto c z^{\mu_1}(z-1)^{\mu_2}(z-2)^{\mu_3},\]
  where the domain is parametrized so that $0,1,2 \in \P^1$ are lifts
  of nodal points. The constant $c \in (\C^\times)^2$ is chosen so
  that the map satisfies the matching constraint at the nodal points
  corresponding to the two incoming edges. The constant $c$ is not
  unique; but, by Lemma \ref{lem:mikhsym}, two distinct values produce
  broken maps that are related to each other by an element of the
  tropical symmetry group.  For a vertex $v \in \Ver_\to(\Gamma)$ with
  a marking, and incident edges with {direction}s $\mu,-\mu$, the map is
  a trivial cylinder 
  \[u_v : \P^1 \bs \{0,\infty\} \to \XC_{P_0} \simeq (\C^\times)^2, \quad z \mapsto c z^\mu\]
  with direction $\mu \in \Z^2$, 
  and the constant $c$ is determined by the tropical point constraint
  at the marking. To define the map $u_v$ for a vertex
  $v \in \Ver_1(\Gam_\aug)$ corresponding to a leaf of the Mikhalkin
  graph, we first observe that the manifold $\XC_{P(v)}$ is a
  $\C^\times$-fibration 
  \begin{equation}
    \label{eq:xpifib}
    \C^\times \to \XC_{P(v)} \xrightarrow{\pi_{P(v)}} X_{P(v)}  
  \end{equation}
  where the manifold $X_{P(v)}$ is a subset of a torus-invariant
  divisor of $X$.  The map $u_{v_e}$ corresponding to the vertex $v_e$
  is an injective map to the fiber of \eqref{eq:xpifib}.

\begin{figure}[ht]
  \centering \scalebox{.8}{
\begingroup%
  \makeatletter%
  \providecommand\color[2][]{%
    \errmessage{(Inkscape) Color is used for the text in Inkscape, but the package 'color.sty' is not loaded}%
    \renewcommand\color[2][]{}%
  }%
  \providecommand\transparent[1]{%
    \errmessage{(Inkscape) Transparency is used (non-zero) for the text in Inkscape, but the package 'transparent.sty' is not loaded}%
    \renewcommand\transparent[1]{}%
  }%
  \providecommand\rotatebox[2]{#2}%
  \newcommand*\fsize{\dimexpr\f@size pt\relax}%
  \newcommand*\lineheight[1]{\fontsize{\fsize}{#1\fsize}\selectfont}%
  \ifx\svgwidth\undefined%
    \setlength{\unitlength}{310.17977016bp}%
    \ifx\svgscale\undefined%
      \relax%
    \else%
      \setlength{\unitlength}{\unitlength * \real{\svgscale}}%
    \fi%
  \else%
    \setlength{\unitlength}{\svgwidth}%
  \fi%
  \global\let\svgwidth\undefined%
  \global\let\svgscale\undefined%
  \makeatother%
  \begin{picture}(1,0.36219194)%
    \lineheight{1}%
    \setlength\tabcolsep{0pt}%
    \put(0,0){\includegraphics[width=\unitlength,page=1]{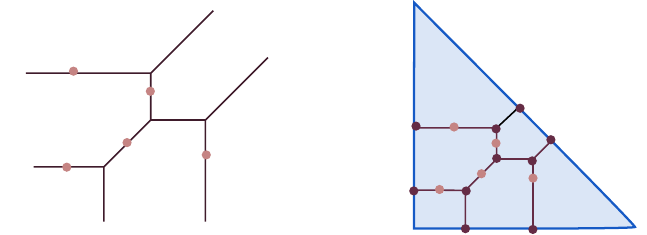}}%
    \put(-0.00374657,0.1708673){\color[rgb]{0,0,0}\makebox(0,0)[lt]{\lineheight{1.25}\smash{\begin{tabular}[t]{l}$\Gamma$\end{tabular}}}}%
    \put(0.89620863,0.16394636){\color[rgb]{0,0,0}\makebox(0,0)[lt]{\lineheight{1.25}\smash{\begin{tabular}[t]{l}$\Gamma_{\on{aug}}$\end{tabular}}}}%
    \put(0,0){\includegraphics[width=\unitlength,page=2]{mikh.pdf}}%
  \end{picture}%
\endgroup%
}
  \caption{A Mikhalkin graph $\Gamma$ of degree $2$ in $\R^2$ through
    5 generic points and the corresponding tropical
    graph $\Gamma_\aug$ of a broken map in the dual complex $B^\dual$
    from Figure \ref{fig:stretchp2}.  In $\Gamma_\aug$ the lighter vertices contain marked points mapping to the point constraints, and the arrows on edges indicate the marking orientation. The lighter vertices are stable since they have two incident edges and a marking.}
  \label{fig:mikh}
\end{figure}

Conversely, consider a broken map with tropical graph $\Gamma_u$
passing through the tropical point constraints. To prove that the
broken map arises from a Mikhalkin graph, it is enough to show that
\begin{enumerate}
\item there are no components mapping to $\XX_P$ with $\dim(P)=2$; and 
\item \label{part:p1map} if a component $u_v$ maps to $\XX_P$ with
  $\dim(P)=1$, then it is a simple map to a $\P^1$-fiber of
  \eqref{eq:xpifib}, that is, in the trivial $\P^1$-fibration
  $\ol \XX_P \simeq \ol X_P \times F$, the compactified map
  $u_v : \P^1 \to \ol \XX_P $ is homologous to
  $\ol X_P \times \{ \on{point}\}$.
\end{enumerate}
\label{rep:bothclaims} Both claims follow from considering dimensions
of moduli spaces
as follows: Markings lie on components mapping to $\ol \XC_{P_0}$.
Let $\Gamma_0 \subset \Gamma_u$ be the subgraph spanned by vertices
$v$ with $P(v)=P_0$.  Each connected component of $\Gamma_0$ is a
Mikhalkin graph (where vertices in $\Ver_{\to}$ are treated as point
constraints), and therefore, $\Gamma_0$ is necessarily connected.
Consider a vertex $v \notin \Gamma_0$ that has an edge $e$ incident on
$v' \in \Ver(\Gamma_0)$. Then, the map $u_v$ has a point constraint at
the node $w_e$ and no other point constraints.  Furthermore, the
intersections with toric divisors are simple, so the only possibility
is that $v$ does not have any incident edge besides $e$,
$\dim(P(v))=1$ and $v$ is of the form \eqref{part:p1map} above.  Thus,
we have shown that $\Gamma_0$ is the Mikhalkin graph corresponding to
the map $u$, vertices of form \eqref{part:p1map} lie in
$\Ver_1((\Gamma_0)_\aug)$, and $\Gamma_u=(\Gamma_0)_\aug$.

 Finally, by Lemma \ref{lem:mikhsym}, 
 the multiplicity \cite{Mikh:R2} of the Mikhalkin graph $\Gamma$ is equal 
 to the size of the tropical symmetry group of the broken
 map. 
\end{proof}

\begin{remark}
  \label{rem:orb-fix}
  The multiple cut of $\P^2$ used in Proposition \ref{prop:tropbrok}
  is not an ``allowed'' cut because at the shaded corner in Figure
  \ref{fig:fix-orb}, the $S^1$-action corresponding to the cut $P_1$
  is not free.
However, since none of the broken maps have a node at this orbifold point, our results are still valid. That is, the moduli space of all the broken maps we encountered are transversally cut out, and the set of broken maps (since they are index zero) correspond bijectively to unbroken maps. To avoid this issue altogether, one may alter the multiple cut to $\PP_1$ shown in Figure \ref{fig:fix-orb}. The broken maps corresponding to Mikhalkin graphs in $\XX_{\PP}$ naturally correspond to broken maps $\XX_{\PP_1}$; an example is shown in Figure \ref{fig:fix-orb2}. 
\end{remark}

\begin{figure}[ht]
  \centering \scalebox{.8}{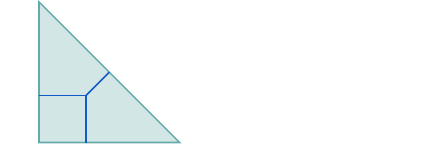}
  \caption{The shaded orbifold point in the multiple cut $\PP$ is avoided in the multiple cut $\PP_1$.}
  \label{fig:fix-orb}
\end{figure}
\begin{figure}[ht]
  \centering \scalebox{.8}{
\begingroup%
  \makeatletter%
  \providecommand\color[2][]{%
    \errmessage{(Inkscape) Color is used for the text in Inkscape, but the package 'color.sty' is not loaded}%
    \renewcommand\color[2][]{}%
  }%
  \providecommand\transparent[1]{%
    \errmessage{(Inkscape) Transparency is used (non-zero) for the text in Inkscape, but the package 'transparent.sty' is not loaded}%
    \renewcommand\transparent[1]{}%
  }%
  \providecommand\rotatebox[2]{#2}%
  \newcommand*\fsize{\dimexpr\f@size pt\relax}%
  \newcommand*\lineheight[1]{\fontsize{\fsize}{#1\fsize}\selectfont}%
  \ifx\svgwidth\undefined%
    \setlength{\unitlength}{407.70397757bp}%
    \ifx\svgscale\undefined%
      \relax%
    \else%
      \setlength{\unitlength}{\unitlength * \real{\svgscale}}%
    \fi%
  \else%
    \setlength{\unitlength}{\svgwidth}%
  \fi%
  \global\let\svgwidth\undefined%
  \global\let\svgscale\undefined%
  \makeatother%
  \begin{picture}(1,0.21773126)%
    \lineheight{1}%
    \setlength\tabcolsep{0pt}%
    \put(0,0){\includegraphics[width=\unitlength,page=1]{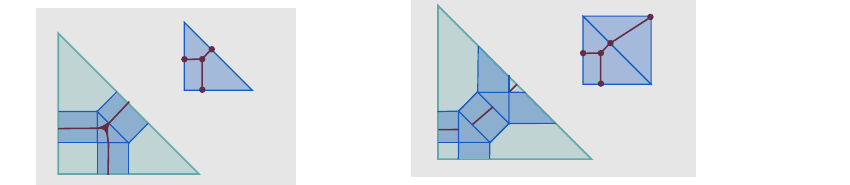}}%
    \put(0.24857699,0.01874322){\makebox(0,0)[lt]{\lineheight{1.25}\smash{\begin{tabular}[t]{l}$u$\end{tabular}}}}%
    \put(0.18175329,0.15499711){\makebox(0,0)[lt]{\lineheight{1.25}\smash{\begin{tabular}[t]{l}$P_0^\dual$\end{tabular}}}}%
    \put(0.65043542,0.160156){\makebox(0,0)[lt]{\lineheight{1.25}\smash{\begin{tabular}[t]{l}$P_0^\dual$\end{tabular}}}}%
    \put(0.7710661,0.1605165){\makebox(0,0)[lt]{\lineheight{1.25}\smash{\begin{tabular}[t]{l}$P_2^\dual$\end{tabular}}}}%
    \put(0.71128846,0.03177241){\makebox(0,0)[lt]{\lineheight{1.25}\smash{\begin{tabular}[t]{l}$u'$\end{tabular}}}}%
    \put(0,0){\includegraphics[width=\unitlength,page=2]{fix-orb2.pdf}}%
    \put(-0.00355101,0.09094677){\makebox(0,0)[lt]{\lineheight{1.25}\smash{\begin{tabular}[t]{l}$\PP$\end{tabular}}}}%
    \put(0.83450154,0.09167684){\makebox(0,0)[lt]{\lineheight{1.25}\smash{\begin{tabular}[t]{l}$\PP_1$\end{tabular}}}}%
  \end{picture}%
\endgroup%
}
  \caption{A broken map $u$ in $\XX_{\PP}$ corresponds to a broken map $u'$ in $\XX_{\PP_1}$. In both cases, the Figure shows a representation of the map and the tropical graph.}
  \label{fig:fix-orb2}
\end{figure}

\begin{remark}\label{rem:whytropical}
  {\rm(Origin of the word ``tropical'')} The connection between broken
  maps and Mikhalkin graphs explains why the objects of study in this
  book are called ``tropical'', as we elaborate in this remark.  The
  \em{tropical semi-ring} $\R_{\on{trop}}$ is the set $\R$ with the
  addition and multiplication operations defined as $\max$ and $+$
  respectively.  For example, a tropical quadratic polynomial in
  $\R_{\on{trop}}$ in variables $x$ and $y$ has the form
  \begin{equation}
    \label{eq:deg2poly}
    f(x,y)=\max\{a_{00}, a_{10} + x, a_{01}+y, a_{11} + x +y, a_{20}+2x, a_{02} +2y\}
  \end{equation}
for some real constants 
$a_{00}, a_{10}, a_{01}, a_{11}, a_{20}, a_{02}.$
  A \em{tropical hypersurface} is the zero set of a tropical
  polynomial in $\R^n$, and is defined to be the set of points where
  the function is not linear. Thus tropical hypersurfaces are
  complexes of polytopes.
  \begin{figure}[ht]
    \centering \scalebox{.8}{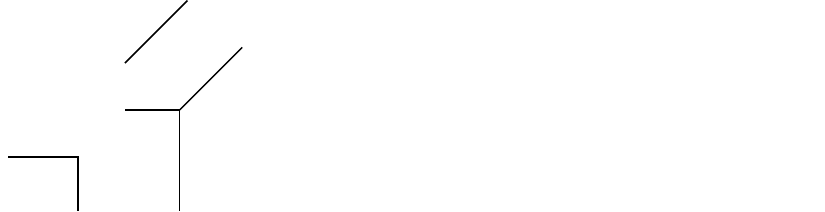}
    \caption{Tropical curves of degree $2$ that arise as zero sets of
      $f$ in \eqref{eq:deg2poly}, and the maximum monomial in regions
      of $\R^2$.}
    \label{fig:degree2}
  \end{figure}
  Mikhalkin's tropical curve in $\R^2$ is the parametrized version of
  a tropical hypersurface in $\R^2$. 
  \end{remark}

\chapter{Broken manifolds}\label{chap:bsymp}
In this Chapter, we review the multiple cut construction and the
associated degenerations of almost complex structures.  Our approach
is much less general than, for example, Parker \cite{bp1}, but we wish
to be completely explicit.

In the first half of the chapter we describe cut spaces and broken
manifolds as symplectic manifolds.  The multiple cut is a
generalization of the symplectic cut operation of Lerman \cite{le:sy2}, which we
call ``single cut''.  \label{rep:gompf} A multiple cut is defined
on a symplectic manifold equipped with a tropical moment map and a
polyhedral decomposition $\PP$, and yields a symplectic cut space
$X_P^\om$ \footnote{We use a superscript $\om$ to denote symplectic
  cut spaces and symplectic broken manifold to distinguish them from
  the corresponding almost complex cut spaces and broken manifolds
  defined in the second half of the chapter.} corresponding to every
polytope $P$ in the decomposition $\PP$.  The broken manifold consists
of top-dimensional cut spaces and thickenings of the lower-dimensional
cut spaces, denoted by $\XX_P^\om$ (see Section
\ref{sec:symp-broken}). The definition of the thickenings $\XX_P^\om$
requires the additional datum of a dual complex, Definition
\ref{def:dualcomplex} below.

In the second half of the chapter (starting from Section
\ref{sec:cylacs}), we describe broken manifolds and cut spaces as
almost complex manifolds with cylindrical almost complex
structures. The manifold $X$ is equipped with a family of
neck-stretched almost complex structures $J^\nu$. In the infinite neck
length limit $\nu \to \infty$, the almost complex manifolds
$X^\nu:=(X,J^\nu)$ degenerate into broken almost complex manifolds,
denoted by $\XC_P$. The broken almost complex manifold $\XC_P$ is
diffeomorphic to the symplectic broken manifold $\XX_P^\om$, but there
is no canonical embedding where the symplectic form tames the almost
complex structure. 

The definition of the broken manifold $\XC_P$ as the degenerate limit
gives a natural family of identifications 
preserving the almost complex structures and 
fibered structures 
between subsets of 
the neck-stretched manifold $X^\nu$ with a $P$-cylindrical  
 almost complex structure 
and the broken manifold
$\XX_P$.  
These identifications, called \em{translations},
are defined in Section \ref{sec:trans-def}, and are analogous to
`target rescalings' of Ionel \cite{ion:nc}.

Finally, in Section \ref{sec:sympcyl} we prove the existence of a
cylindrical structure on the symplectic form in the neighborhood of
cut loci in $(X,\om_X)$. The cylindrical structure underlying
neck-stretched almost complex structure is chosen to be the same as
the cylindrical structure on the symplectic form.  Choosing the
cylindrical structures in this manner will allow us (later in Chapter
\ref{chap:hof}) to construct families of diffeomorphisms
from $X$ to itself for which the pullback of $\om_X$ tames $J^\nu$.

\index{X@$X_P$, $\ol X_P$, $X_P^\om$, $\ol X_P^\om$|seeonly {Cut space}}
\index{X@$\XX$, $\XX_\PP$|seeonly {Broken manifold}}
\index{X@$\XX_P$, $\ol \XX_P$, $\XX_P^\om$, $\ol \XX_P^\om$|seeonly {Broken manifold}}
\index{Z@$Z_P$, $\ol Z_P^\om$|seeonly {Torus bundle}}

We point out that for most of the book, we treat the almost complex
manifold $X_P$ and the symplectic manifold $X_P^\om$ as distinct
spaces. This distinction is required in proving the main result
Theorem \ref{thm:bfuk}. Later, in applications where one is working
exclusively with broken manifolds, the distinction between $X_P$ and
$X_P^\om$ may be dropped, as we explain in Section
\ref{sec:brokenind}.

\section{Symplectic cut}\label{sec:singleneck}

A multiple cut is the generalization of the \em{symplectic cut}
\index{Cut!Symplectic cut} construction of Lerman \cite{le:sy2} which
we now review.  We call this construction a \em{single cut},
\index{Cut!Single cut} to distinguish it from a multiple cut.  The
construction of symplectic cuts uses Hamiltonian circle actions on
symplectic manifolds.

\begin{definition} \label{def:lerm} {\rm(Lerman's symplectic cut
    construction)}
  \begin{enumerate}
  \item {\rm(Hamiltonian circle actions)} Let $(X,\om_X)$ be a
    compact symplectic manifold.  Let
    %
    \[ S^1 = \{ z \in \C \ : \ | z | = 1 \} \]
    denote the circle group; we identify its Lie algebra
    $\on{Lie}(S^1)$ with $\R$ by division by $i$.  A \em{Hamiltonian
      action} of the circle group $S^1$ on $X$ is an action
    $S^1 \times X \to X$ generated by the Hamiltonian flow of a \em{
      moment map}
    \[ \Phi :X \to \R, \quad \omega_X(\xi_X, \cdot) = - \d \Phi \]
    where the generating vector field of an element $\xi \in \R$
    \begin{equation} \label{genvec} \xi_X \in \Vect(X) , \quad
      \xi_X(x) = \ddt |_{t = 0} \exp( it\xi ) x .\end{equation}
    In particular, the affine line $\C$ has symplectic form
    \[ \omega_\C = \frac{-i}{2} \d z \wedge \d \ol{z} \in \Omega^2(\C)
      .\]
    The Hamiltonian action of $S^1$ is given by scalar multiplication
    and has moment map
    \[ \Phi_\C : \C \to \R, z \mapsto \frac{-|z|^2}{2} . \]
  \item {\rm (Global symplectic cuts)} Let $X$ be a symplectic
    manifold with symplectic form $\om_X$ and a free Hamiltonian
    $S^1$-action with moment map $\Phi$.  The product
    $\hat{X} = X \times \C$ has product symplectic form
    $\hat{\om} = \pi_1^* \om_X + \pi_2^* \omega_\C$.  The diagonal
    action of $S^1$ has moment map
    \[ \hat{\Phi}: \hat{X} \to \R, \quad (x,z) \mapsto \Phi(x) -
      \frac{|z|^2}{2}. \]
    The zero level set is the union
    %
    \[ \hat{\Phi}^{-1}(0) = ( \Phi^{-1}(0) \times \{ 0 \} ) \sqcup
      \{ (x,z) : \Phi(x) = \tfrac{|z|^2}{2} > 0 \} \]
    of two pieces where both $\Phi$ and $z$ are zero and the piece
    where both $\Phi$ and $z$ are non-zero.  The action on $z \neq 0$
    has a natural slice given by $z \in \R_{> 0}$ so that
    %
    \[ \{ (x,z) : \Phi(x) = |z|^2/2 > 0 \} \cong \Phi^{-1}(\R_{> 0})
      .\]
    The symplectic quotient $\ol X_+:=\hat{\Phi}^{-1}(0)/ S^1 $ is
    called the \em{symplectic cut space}. Alternatively, the symplectic
    cut space is viewed as the compactification of
    \[X_+:=\{\Phi > 0\} \subset (X,\om_X),\]
    given by   \label{rep:phigeq}
    \begin{equation*}  \ol X_+ := \hat{\Phi}^{-1}(0)/
      S^1 \simeq \{\Phi \geq 0\}/\sim,\end{equation*}
where $\sim$ is the equivalence relation on the boundary
    $\Phinv(0)$ given by the $S^1$-action.  The cut space is the
    union of $\{\Phi>0\}$ and the symplectic quotient $\Phinv(0)/S^1$.
    One has a similar construction of a cut space
    \[\ol X_-:=\{\Phi \leq 0\}/\sim,\]
    which is the union of $\{\Phi<0\}$ and the symplectic quotient
    $\Phinv(0)/S^1$.  The symplectic manifolds $\ol X_-, \ol X_+$ both
    contain a copy of $\Phinv(0)/S^1$ via the embeddings
    \[ i_-: \Phinv(0)/S^1 \to \ol X_-, \quad i_+: \Phinv(0)/S^1 \to
      \ol X_+ \]
    with opposite normal bundles $N_\pm \to \Phinv(0)/S^1$. \label{rep:npm}
  \item {\rm (Local symplectic cuts)} Given an open subset
    $U \subset X$ with a free Hamiltonian $S^1$-action with moment map
    $\Phi: U \to \R$, such that $X \bs U$ is disconnected, gluing
    together the cut $U_+ \cup U_-$ with $X - \Phi^{-1}(0)$ produces
    cut spaces $\ol X_+$, $\ol X_-$.
  \end{enumerate}
\end{definition}

\section{Multiple cuts in a symplectic
  manifold}\label{sec:symp-multcut}
\index{Cut!Multiple cut}

We recall that the input datum for a single cut consists of a
hypersurface and a Hamiltonian $S^1$-action in the neighborhood of the
hypersurface. The input datum for a multiple cut consists of a
collection of intersecting hypersurfaces with Hamiltonian
$S^1$-actions in their neighborhoods. Neighborhoods of intersections
of hypersurfaces have Hamiltonian torus actions, whose restrictions
coincide with the $S^1$-actions corresponding to individual
hypersurfaces.  The various Hamiltonian actions are recorded by a \em{
  tropical moment map} with target space $\t^\dual$, and a \em{polyhedral decomposition} of $\t^\dual$.  These polyhedral
decompositions appeared in, for example, Meinrenken
\cite{mein}.

\begin{definition} \label{def:delz} {\rm(Simple resp. Delzant
    polytopes)} Let $T$ be a torus with Lie algebra $\t$.  Let
  $\t_\Z \subset \t$ denote the coweight lattice of points that map to
  the identity under the exponential map, so that
  $ T \cong \t/\t_\Z .$ A convex polytope $P$ in $\t^\dual$ is
  described by a collection of linear inequalities determined by
  constants $c_F \in \R$ and normal vectors $\nu_F \in \t$:
  %
  \[ P = \{ \lambda \in \t^\dual \ : \ \lan \lambda,\nu_F \ran \ge
  c_F, \quad \forall F \subset P \ \text{facets} \} .\]
We allow polytopes to be unbounded.  By the \em{interior} of a
polytope, we mean the complement of its proper faces, so that in
particular the interior of a zero-dimensional polytope (a point) is
itself.
  \label{page:intrem}
  The polytope $P$ is \em{simple} if \index{Polytope!Simple polytope}
  for each point $v \in P$, the normal primitive vectors
  $\nu_F \in \t_\Z$ to the facets $F \subset P$ containing $v$
 form a basis for the span of the vectors $\nu_F$ in $\t$.  \label{rep:delzant}
 For a simple polytope $P$, a face $F_0 \subset P$ which is the
 intersection of facets $F_1,\dots,F_m$ is \em{smooth} if the
 primitive normal vectors form a lattice basis:
  \begin{equation}
    \label{eq:zspan}
   \on{span}_\Z(\nu_{F_i}, i=1,\dots,m) = \on{span}(\nu_{F_i}, i=1,\dots,m) \cap \t_\Z.   
  \end{equation}
  A simple polytope is \em{Delzant} if all of its faces are
  smooth. \index{Polytope!Delzant polytope}
\end{definition}

By Delzant \cite{de:ha}, there is a bijection between compact,
connected Delzant polytopes with trivial generic stabilizer and
compact symplectic toric manifolds. A Delzant polytope $P$ corresponds
to a symplectic manifold $V_P$ with an effective Hamiltonian action of
a torus $T\simeq (S^1)^{(\dim(V_P)/2}$ and moment map and polytope
  \[ \Psi:V_P \to \t^\dual , \quad \Psi(V_P) = P . \]
  We remark that in a simple polytope $P$,
  if a face $F_0 \subset P$, which is the
 intersection of facets $F_1,\dots,F_m$, is not smooth,  then the
 primitive normal vectors span a sublattice 
\[ \on{span}_\Z(\nu_{F_i}, i=1,\dots,m) \subset  \on{span}(\nu_{F_i}, i=1,\dots,m) \cap \t_\Z. \]
 In this case $V_P$ is an
  \em{orbifold}, that is, it is covered by charts that are finite
  quotients of $\R^n$.

For later use, we define cones at faces of polytopes.
\begin{definition} \label{def:normcones} 
{\rm(Cones at faces of polytopes)} Let $P$ be a
 simple polytope in a vector space $V$. 
   For a face $Q \subset P$, the \em{cone} of $P$ at $Q$ is 
   \index{Cone!$\Cone_Q P$}
    \begin{equation}
      \label{eq:conedef}
      \Cone_Q(P):=\{\lam(p-q) : p \in P, q \in Q, \lam \in \R_{\geq 0}\} \subset V.
    \end{equation}
    For any interior, non-singular point $q$ in $Q$, let $V_Q:=T_qQ$ be the tangent space at $q$, independent up to isomorphism of the choice of $q$ using the affine structure.   
     The
    \em{normal cone} of $P$ at $Q$ is
       \index{Cone!Normal cone $\NCone_Q P$}
       \begin{equation}
      \label{eq:nconedef}
      \NCone_Q(P):= \Cone_Q(P)/V_Q \subset V/V_Q,
    \end{equation}
    which is the image of $\Cone_Q(P)$ under the projection
    $V \to V/V_Q$.  See Figure \ref{fig:ncone}.
\end{definition}
\begin{figure}[h]
  \centering \scalebox{.8}{
\begingroup%
  \makeatletter%
  \providecommand\color[2][]{%
    \errmessage{(Inkscape) Color is used for the text in Inkscape, but the package 'color.sty' is not loaded}%
    \renewcommand\color[2][]{}%
  }%
  \providecommand\transparent[1]{%
    \errmessage{(Inkscape) Transparency is used (non-zero) for the text in Inkscape, but the package 'transparent.sty' is not loaded}%
    \renewcommand\transparent[1]{}%
  }%
  \providecommand\rotatebox[2]{#2}%
  \newcommand*\fsize{\dimexpr\f@size pt\relax}%
  \newcommand*\lineheight[1]{\fontsize{\fsize}{#1\fsize}\selectfont}%
  \ifx\svgwidth\undefined%
    \setlength{\unitlength}{364.64935884bp}%
    \ifx\svgscale\undefined%
      \relax%
    \else%
      \setlength{\unitlength}{\unitlength * \real{\svgscale}}%
    \fi%
  \else%
    \setlength{\unitlength}{\svgwidth}%
  \fi%
  \global\let\svgwidth\undefined%
  \global\let\svgscale\undefined%
  \makeatother%
  \begin{picture}(1,0.22031348)%
    \lineheight{1}%
    \setlength\tabcolsep{0pt}%
    \put(0,0){\includegraphics[width=\unitlength,page=1]{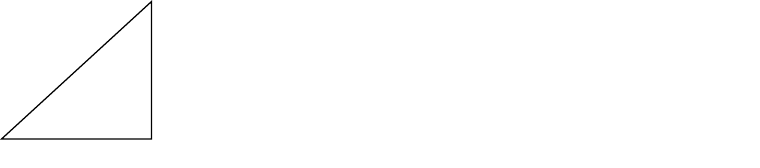}}%
    \put(0.08438585,0.00451418){\color[rgb]{0,0,0}\makebox(0,0)[lt]{\lineheight{1.25}\smash{\begin{tabular}[t]{l}$P$\end{tabular}}}}%
    \put(0.11729415,0.10323914){\color[rgb]{0,0,0}\makebox(0,0)[lt]{\lineheight{1.25}\smash{\begin{tabular}[t]{l}$Q$\end{tabular}}}}%
    \put(0,0){\includegraphics[width=\unitlength,page=2]{ncone.pdf}}%
    \put(0.38056075,0.18550991){\color[rgb]{0,0,0}\makebox(0,0)[lt]{\lineheight{1.25}\smash{\begin{tabular}[t]{l}$\Cone_P Q$\end{tabular}}}}%
    \put(0.78360421,0.19697771){\color[rgb]{0,0,0}\makebox(0,0)[lt]{\lineheight{1.25}\smash{\begin{tabular}[t]{l}$\NCone_P Q$\end{tabular}}}}%
    \put(0,0){\includegraphics[width=\unitlength,page=3]{ncone.pdf}}%
  \end{picture}%
\endgroup%
}
  \caption{The cone and normal cone of $Q$ at a face $P$.}
  \label{fig:ncone}
\end{figure}

Multiple cuts are defined on symplectic manifolds with a tropical Hamiltonian
action.
\index{Cut!Multiple cut}
\begin{definition} \label{def:tham}
  {\rm(Tropical Hamiltonian action)} A \em{tropical Hamiltonian
    action} \index{Tropical Hamiltonian action}
  of a torus $T$ with Lie algebra $\t$ is a triple
  $(X,\Phi, \PP)$ consisting of a
  \begin{enumerate}
  \item compact symplectic manifold $X$
  \item {\rm(Polyhedral decomposition)} \label{part:thamb} a
    decomposition
    \[ \t^\dual = \bigcup_{P \in \PP} P^\circ, \quad \PP = \{ P
      \subset \t^\dual \} \]
    of $\t^\dual$ into the disjoint union of the interiors $P^\circ$
    of simple polytopes $P \in \PP$ such that
    \begin{itemize}
    \item if $P_0,P_1 \in \PP$ have non-empty intersection, then
      $P_0 \cap P_1 \in \PP$ and $P_0 \cap P_1$ is a face of both
      $P_0$ and $P_1$,
    \item any polytope $P$ has at least one vertex $v \in \PP$; and
    \end{itemize}
  \item {\rm (Tropical moment map)}  a \em{tropical moment map} \index{Tropical moment map}
    compatible with $\PP$
    \[ \Phi:X \to \t^\dual \]
    in the following sense.  For any $P \in \PP$, we denote by
    \[ \t_P : = \on{ann} (TP) \subset \t \]
    the annihilator of the tangent space of $P$ at any point
    $p \in P$, and by
    \[ T_P = \exp(\t_P) \]
    the torus whose Lie algebra is $\t_P$.  Let $\t_{P,\Z} $ be the
    coweight lattice in $\t_P$ so that
    \[ T_P \cong \t_P/\t_{P,\Z} .\]
    For any $P \in \PP$, there exists an open neighbourhood $U_P$ of
    $\Phinv(P)$ such that the composition
    \[ \pi_{\t_P^\dual} \circ \Phi : U_P \to \t_P^\dual \]
    is a moment map for a free action of $T_P$ on $U_P$, where
    $\pi_{\t_P^\dual}: \t^\dual \to \t_P^\dual$ is the projection dual
    to the inclusion $\t_P \hra \t$.
  \end{enumerate}
\end{definition}
\begin{notation}
 For a polyhedral decomposition $\PP$ and any $k \in \Z_{\geq 0}$, we denote by 
 \begin{equation}
   \label{eq:pkk}
   \PP_{(k)} \subset \PP \quad \text{resp.} \quad \PP^{(k)}\subset \PP
 \end{equation}
 the set of polytopes of dimension resp. codimension $k$.
\end{notation}
\begin{remark}
  The definition of the tropical moment map implies that for any
  zero-dimensional polytope $R \in \PP$, $\t_R=\t$, and for any pair
  $Q \subset P$ in $\PP$, there is a canonical inclusion
  $\t_P \subset \t_Q$.
\end{remark}
\begin{remark}{\rm(Tropical manifold for a single cut)}
  \label{rem:singlecut}
  In the single breaking case,  a tropical Hamiltonian action consists of
  a map $\Phi:X \to \R$ and a decomposition
  $\R:=(-\infty,c] \cup [c,\infty)$, such that $\Phi$ generates a free
  $S^1$-action in the neighborhood of $\Phinv(c)$. Thus the set of
  polytopes is
\[ \PP=\{(-\infty,c],\{c\}, [c,\infty)\} . \]
In the case of a single cut, we denote the tropical Hamiltonian action by the
triple $(X,\Phi,c)$.
\end{remark}

A multiple cut operation on a tropical manifold produces cut spaces of various dimensions, each corresponding to a polytope in the polyhedral decomposition.
\begin{definition}
  {\rm(Cut space for a multiple cut)} \index{Cut space!Symplectic cut space $X_P^\om$, $\ol X_P^\om$}
\label{def:cutspace} 
Given a tropical symplectic manifold $(X,\Phi,\PP)$, for any polytope
$P \in \PP$, the \em{cut space} is a symplectic manifold (or
orbifold, if $P$ is not Delzant) that is a compactification of
\[X^\om_P:=\Phinv(P^\circ)/T_P\]
given by 
  \begin{equation}
    \label{eq:cutsp}
    \ol X^\om_P := \Phinv(P)/\sim,     
  \end{equation}
where the equivalence $\sim$ 
quotients
by the following torus actions: 
  \begin{equation} \label{appropriate} x \sim tx, \quad \forall x \in
    \Phinv(Q^\circ), t \in T_Q
  \end{equation} 
  for all polytopes $Q \subseteq P$ contained in $\PP$.  The space
  $\ol X^\om_P$ is a manifold (or orbifold if $P$ is not Delzant) by
  an iterative application of Lerman's cut. \label{smoothstr2} Cut
  spaces have natural inclusions
  \[\ol X^\om_Q \subset \ol X^\om_P, \quad Q \subset P.\]
  For a face $Q \subset P$ with $\codim_P(Q)=1$, the corresponding
  subset $\ol X^\om_Q$ is called a \em{relative divisor of
    $\ol X^\om_P$}. The relative divisors $\ol X^\om_Q$ intersect
  $\om$-orthogonally, and the intersections correspond to a cut space
  $\ol X^\om_R$ for some polytope $R \in \PP$. This ends the
  Definition.
\end{definition}
%

\begin{notation}{\rm($X_P^\om$ versus $X_P$)}
  \label{note:symp-vs-ac} We use the superscript $\om$ in the notation
  for symplectic cut spaces and symplectic broken manifolds in order
  to to distinguish them from the corresponding almost complex
  manifolds defined in Section \ref{sec:cylbrokenmfd}.  Almost complex
  cut spaces (and broken manifolds) are diffeomorphic to the
  symplectic cut spaces (and symplectic broken manifolds), but the
  diffeomorphisms are not canonical. The almost complex versions occur
  much more frequently in the text and they do not have any
  superscript.  See Lemma \ref{lem:immdiffeo} and Remark
  \ref{rem:nosymp} for related discussion.
  %
  The spaces $X_P$ and
  $X_P^\om$ need to be distinguished only in the course of the proof
  of Theorem \ref{thm:bfuk}, where the almost complex structures on
  broken manifolds are required to be gluable. Later, in applications,
  this distinction can be dropped as we explain in the beginning of
  Section \ref{sec:brokenind}.
\end{notation}

\begin{example}
  For a tropical Hamiltonian action $(X,\Phi,c)$ with a single cut (using
  notation as in Remark \ref{rem:singlecut}), the cut spaces are
  \[\{\Phi \geq c\}/\sim, \quad \Phinv(c)/\sim, \quad \{\Phi \leq c\}/\sim,\]
  and in all three spaces, 
  the relation $\sim$ quotients $\Phinv(c)$ by the Hamiltonian
  $S^1$-action.
\end{example}
\begin{example} 
  The multiple cut in Figure \ref{fig:break1} is made up of two
  simultaneous single cuts along hypersurfaces that intersect
  $\om$-orthogonally.  The set of polytopes is
\[ \PP=\{P_i, 0 \leq i \leq 3, P_{ij}, j=(i+1)\mod 4, P_\cap\}. \]  
The manifolds $\ol X_{P_{i(i+1)}}$, $\ol X_{P_{(i-1)i}}$ are relative
divisors of $\ol X_{P_i}$.
\end{example}
\begin{figure}[ht]
  \centering \input{break.pdf_tex}
  \caption{A multiple cut of $\R^2$.}
  \label{fig:break1}
\end{figure} 

\subsection{Some generalizations}

More generally, \label{integralaffine} we allow the moment map to take
values in affine manifolds as follows, continuing the discussion in Section \ref{sec:degen}.

\begin{definition}
  Let $B$ be a topological manifold of dimension $n$.  An \em{affine
    structure} on $B$ is an atlas for which the transition maps
  between coordinate charts  take values in
  the group of affine transformations, namely 
  $\R^n \ltimes GL(n,\R)$. A \em{tropical affine structure}
  in the language of Gross \cite{gross:book} is one for which
  transition maps take values in
  $\R^n \ltimes GL(n,\Z)$; a manifold with such a structure is a \em{
    tropical affine manifold} and the functions provided by the local
  coordinates are \em{affine coordinates}.  A map $B_1 \to B_2$
  between tropical affine manifolds is a \em{tropical affine
    morphism} if the map is given in local charts by affine linear
  combinations of the affine local coordinates.
\end{definition}
\noindent The transition maps preserve the trivial flat connection in affine
local coordinates, and these glue together to a flat connection on the
smooth locus.  We also introduce a definition which allows certain kinds of singularities in the affine structure.  Let $B$ be a topological manifold.

\begin{definition} \label{def:affwsing}
A \em{tropical affine
    structure with singularities} is a decomposition $B = \cup_i B_i$
  into topological submanifolds $B_i$ of dimension $d(i)$ and a tropical
  affine structure on each $B_i$ satisfying the following
  compatibility condition: Each $B_i$ has a topological tubular
  neighborhood $U_i$, homeomorphic to a topological vector bundle over
  $B_i$, equipped with a projection $\pi_i : U_i \to B_i$ so that if
  $\pi_{i,j}$ denotes the restriction of $\pi_i$ to $U_i \cap B_j$
  then each $\pi_{i,j}$ is a morphism of tropical affine manifolds,
  and the projections $\pi_{i,j}$ are compatible in the following sense: If $d(i) < d(j) < d(k)$ then
  \[\pi_{i,j} \circ \pi_{j,k}=\pi_{i,k}\]
  %
  whenever the maps are defined, which is to say on the intersection of $U_i, B_k, U_j$ and $\pi_j^{-1}(U_i)$.
    A topological manifold with the structure
  of a tropical affine structure with singularities is a \em{tropical
    affine manifold with singularities}.
\end{definition} 

\begin{example} Almost toric manifolds in
  the sense of Symington \cite{sym:2to4} 
  and Leung-Symington \cite{leungsym}
  have fibrations over tropical
  affine manifolds with singularities in codimension two.
\end{example}

In this setting where the moment map takes values in an affine manifold, the notion of polytope is generalized as follows:
\begin{definition} \label{def:poly} Given a tropical affine manifold
  with singularities $B$, a subset $P \subset B$ is a \em{polytope}
  if it is contractible and defined locally by some finite collection
  of inequalities satisfying the following property: Locally on each
  tubular neighborhood $U_i$ the set $P$ is defined by a finite
  collection of linear inequalities in pull-backs of affine
  coordinates on $B_i$. \end{definition}

The following generalizes the notion of a Hamiltonian torus action:

\begin{definition} \label{def:tropham} A tropical Hamiltonian action
  with moment map valued in $B$ is a tuple $(X,\Phi, \PP)$ where $\PP$
  is a polyhedral decomposition of $B$ and $\Phi: X \to B$ is a
  continuous map with the following property: For each polytope $P$,
  the collection of local functions $\lan \Phi, \nu \ran $ defining
  $P$ generate a locally free torus action.
\end{definition}
\noindent That is, for each polytope $P \subset B$, the pull-backs of the
functions $\lan \Phi, \nu \ran $ defining $P$ generate a locally free
torus action.

\begin{example} Any Hamiltonian action with a circle-valued moment map
  $(X,\omega, \Phi:X \to S^1)$ is a tropical Hamiltonian action for
  the polyhedral decomposition
  \[ \PP = \{ \{\theta_1\} ,[\theta_1, \theta_2 ], \{ \theta_2 \},
  [\theta_2,\theta_1] \} \] 
  for which $\theta_1,\theta_2$ are regular
  values.  For example, one can take $X = S^1 \times S^1$ with moment map given by
  projection and the decomposition corresponding to any two choices $\theta_1,\theta_2$ of regular value.
\end{example}

\begin{example} We given a example of an action mapping to an affine
  manifold with singularities.  Let $B = \R$ be equipped with the
  stratification into $B_0 =\{ 0\} $ and $B_1 = \R - \{ 0 \} $.  Let
  $X = T^* S^2$ equipped with the function
  $\Phi:X \to B, v \mapsto \Vert v \Vert$ using the standard metric on
  $S^2$.  Let $\cP = \{ (-\infty,c], \{ c \}, [c,\infty) \}$ where
  $c \neq 0$.  The map $\Phi$ is continuous and smooth away from $0$,
  and the triple $(X,\Phi,\cP)$ is a tropical Hamiltonian action.
\end{example}

With these Definitions, Theorem\ref{thm:bfuk} holds as in the case
that the codomain of the moment map is a vector space.

\section{Symplectic broken manifolds} \label{sec:symp-broken}
In this Section, we describe the broken manifold as a symplectic space.
A component $\XX_P^\om$ of the symplectic broken manifold
corresponding to a polytope $P$ is a thickening of the corresponding
cut space $X_P^\om$ into a toric fibrations.  The dimension of the
toric fiber is complementary to the dimension of $P$, therefore for a
top-dimensional polytope $P$, $X_P^\om=\XX_P^\om$.  The toric
fibrations are defined by considering neighborhoods of cut loci in $X$
and modding out the boundaries as in Lerman's construction.  This
construction requires the additional datum of a dual complex
associated to the polyhedral decomposition $\PP$. The \em{dual
  complex} $B^\dual$ (Definition \ref{def:dualcomplex}) is a union of
dual polytopes $\cup_{P \in \PP}P^\dual$, and $P^\dual$ is the moment
polytope of the fibers of the toric fibration $\XC_P^\om$ over the cut
space $X_P^\om$.

\begin{definition} \label{def:dualcomplex} {\rm(Dual complex)}
  \index{Dual complex $B^\dual$} \index{Dual polytope} For a
  polyhedral decomposition $\PP$ of $\t^\dual$, the \em{dual complex}
  is a topological space
  \begin{equation}
    \label{eq:bdef}
    B^\dual=\left(\cup_{P \in \PP}P^\dual\right)/\sim   
  \end{equation}
  defined by a dual polytope $P^\dual \subset \t_P$
  corresponding to every $P \in \PP$ satisfying
  \begin{enumerate}
  \item $P^\dual$ is top-dimensional in $\t_P$, that is,
    $\dim(P^\dual)=\dim(\t_P)=\dim(\t^\dual) - \dim(P)$;
  \item for any $Q \in \PP$, there is a bijection
    \[\{P \in \PP : P \supset Q\}  \to \text{Proper faces of $Q^\dual$}, \quad P \mapsto Q_P^\dual,\]
    with $\dim(Q_P^\dual)=\dim(P^\dual)$;
  \item \label{part:dualpair}
    for any pair $Q,P \in \PP$ with
    $Q \subset P$, the natural 
    inclusion $\t_P \to \t_Q$ maps $P^\dual$ isomorphically to a
    translate of $Q_P^\dual$; and 
  \item
    for any pair $Q,P \in \PP$ as in \eqref{part:dualpair},
    the equivalence relation $\sim$ in \eqref{eq:bdef} identifies
    $Q_P^\dual$ to $P^\dual$ via the isomorphism in \eqref{part:dualpair}. 
  \end{enumerate}
\end{definition}
As a consequence of the inclusions $\t_P \hra \t$, there is a natural inclusion $B^\dual \hra \t$ which is fixed up to translation in $\t$.

\begin{example}
\label{ex:singlecutdual}
  In the case of a single cut $(X,\Phi,c)$ (see Remark \ref{rem:singlecut} for notation), the set of polytopes is
  \[\PP = \{ P_-:=(-\infty,c],P_0:=\{ c \}, P_+:=[c, \infty) \},\]
  and the dual polytopes are
  \[P_-^\dual=\{ -\eps \}, \quad P_0^\dual=[-\eps,\eps], \quad P_+^\dual=\{ \eps \},\]
  for any $\eps>0$, with identifications $P_\pm^\dual \hra P_0^\dual$
  given by inclusion of the endpoints. Thus, the dual complex is an
  interval: $B^\dual \cong [-\eps,\eps]$.
\end{example} 

\begin{figure}[ht]
  \centering \scalebox{.8}{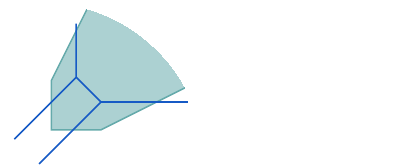}
  \caption{A polyhedral decomposition (left) and its dual complex (right).}
  \label{fig:mulcuteg1}
\end{figure}

\begin{example}
  Figure \ref{fig:mulcuteg1} shows the dual complex corresponding to a
  multiple cut on a toric surface.  We recall that we used three such
  multiple cuts to study disks in a cubic surface in Section
  \ref{sec:cubic-intro}.
\end{example}

We prove some properties of the dual complex:

\begin{lemma}\label{lem:dual-simple}
  For a polyhedral decomposition $\PP$ of $\t^\dual$, the dual
  polytope $P^\dual$ corresponding to any $P \in \PP$ is a compact
  simple polytope.
\end{lemma}
\begin{proof}
  Consider a top-dimensional dual polytope $Q^\dual$, which, we
  recall, corresponds to a point $Q \in \PP$.  Since the polytopes in
  $\PP$ cover $\t^\dual$,
  \begin{equation}
    \label{eq:span-t}
    \bigcup_{P \in \PP^{(0)}: P \ni Q}\Cone_QP=\t^\dual.  
  \end{equation}
  We observe that $\Cone_QP$ is the normal cone of
  $\Cone_{P^\dual}Q^\dual$, that is, $\Cone_QP$ is spanned by all the
  outward normal vectors to the facets of $\Cone_{P^\dual}Q^\dual$.
  By \eqref{eq:span-t}, we conclude that $Q^\dual$ is compact. Any
  dual polytope $P^\dual$ is a face of a top-dimensional dual
  polytope, and is therefore compact. To show that a top-dimensional
  dual polytope $Q^\dual$ is simple, consider a vertex $P_0^\dual$,
  which corresponds to a top-dimensional $P_0 \in \PP$.  The
  simplicity of $P_0$ implies the set of normal vectors
  $\nu_{F_i} \in \t$ to the facets $F_i \in \PP$ of $P_0$ span
  $\t$. The set of one-dimensional edges emanating from $P_0^\dual$ in
  the polytope $Q^\dual$ are precisely $F_i^\dual$, which are parallel
  to $\nu_{F_i}$, and hence span $\t$.  By a similar argument applied
  to every vertex $P_0^\dual$ of $Q^\dual$, we conclude that $Q^\dual$
  is simple.
\end{proof}

To define symplectic broken manifolds, dual polytopes $P^\dual \subset \t_P$ need to be viewed as subsets of $\t_P^\dual$. We fix inner products on $\t_P$ for this purpose:
\begin{definition}\label{def:x-idtt}
  An \em{$X$-inner product} is a collection of inner products  
  \begin{equation}
    \label{eq:idtt}
    g_P: \t_P \times \t_P \to \R
  \end{equation}
  for all $P \in \PP$ such that for any pair $Q \subset P$,
  $g_Q|\t_P = g_P$.  \index{Inner product! $X$-inner product} The
  $X$-inner product $(g_P)_P$ is \em{polytope-independent} if there is
  a inner product $g$ on $\t$ and for all $P \in \PP$, $g_P$ is the
  restriction of $g$ to $\t_P$.
\end{definition}

The broken manifold is a collection of manifolds $\ol \XC_P$ corresponding to polytopes $P \in \PP$, each of which is a toric fibration over the cut space $\ol X_P$.
The manifold $\ol \XC_P$ is modelled on a \em{fibered polytope} $\tP$
which is a thickening of the polytope $P \in \PP$, and is defined next. 
The fibers of the toric fibration are modelled on dual polytopes $P^\dual$ viewed as subsets of $\t_P^\dual$ via the $X$-inner product. 
\begin{definition}{\rm(Fibered polytope)} \label{def:fibpoly}
  Let
  $P \subset \t^\dual$ be a polytope in $\PP$, and let
  $P^\dual \subset \t_P$ be a complementary dimensional
  polytope.  A polytope
  \begin{equation}
    \label{eq:fibpoly}
   \tP \subset \t^\dual 
 \end{equation}
 is \em{fibered over $P$ with fiber $P^\dual$} if it is equipped with
 a diffeomorphism
   \[\tP \xrightarrow{(\pi_P, \pi_{P^\dual})} P \times P^\dual, \]
   where, viewing the fibers of $\pi_P$, $\pi_{P^\dual}$ and the
   polytope $P$ as subsets of $\t^\dual$, we have
   \begin{itemize}
    \item $\pi_{P^\dual}$ is the restriction of the projection $\t^\dual \to \t_P^\dual$ composed with the identification $\t_P \simeq \t_P^\dual$ from the $\t_P$-inner product from \eqref{eq:idtt}, and 
      %
%
  \item $P=\pinv_{P^\dual}(c_{P^\dual})$ for some interior point $c_{P^\dual} \in P^\dual$.
  \end{itemize}
  The facets of the fibered polytope $\tP$ are
  \begin{equation}
    \label{eq:facetsolp}
    \Facets(\tP)=\{\pi_P^{-1}(Q) : Q \in \Facets(P), Q \in \PP\} \cup \{\pi_{P^\dual}^{-1}(Q) : Q \in \Facets(P^\dual)\}.  
  \end{equation}
\end{definition}

\begin{figure}[ht]
  \centering \scalebox{.8}{
\begingroup%
  \makeatletter%
  \providecommand\color[2][]{%
    \errmessage{(Inkscape) Color is used for the text in Inkscape, but the package 'color.sty' is not loaded}%
    \renewcommand\color[2][]{}%
  }%
  \providecommand\transparent[1]{%
    \errmessage{(Inkscape) Transparency is used (non-zero) for the text in Inkscape, but the package 'transparent.sty' is not loaded}%
    \renewcommand\transparent[1]{}%
  }%
  \providecommand\rotatebox[2]{#2}%
  \newcommand*\fsize{\dimexpr\f@size pt\relax}%
  \newcommand*\lineheight[1]{\fontsize{\fsize}{#1\fsize}\selectfont}%
  \ifx\svgwidth\undefined%
    \setlength{\unitlength}{287.47522549bp}%
    \ifx\svgscale\undefined%
      \relax%
    \else%
      \setlength{\unitlength}{\unitlength * \real{\svgscale}}%
    \fi%
  \else%
    \setlength{\unitlength}{\svgwidth}%
  \fi%
  \global\let\svgwidth\undefined%
  \global\let\svgscale\undefined%
  \makeatother%
  \begin{picture}(1,0.30589578)%
    \lineheight{1}%
    \setlength\tabcolsep{0pt}%
    \put(0,0){\includegraphics[width=\unitlength,page=1]{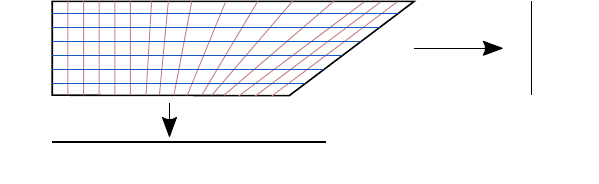}}%
    \put(0.30664938,0.10006531){\color[rgb]{0,0,0}\makebox(0,0)[lt]{\lineheight{1.25}\smash{\begin{tabular}[t]{l}$\pi_P$\end{tabular}}}}%
    \put(0.72907619,0.24605295){\color[rgb]{0,0,0}\makebox(0,0)[lt]{\lineheight{1.25}\smash{\begin{tabular}[t]{l}$\pi_{P^\dual}$\end{tabular}}}}%
    \put(-0.00148884,0.22523462){\color[rgb]{0,0,0}\makebox(0,0)[lt]{\lineheight{1.25}\smash{\begin{tabular}[t]{l}$\ol P$\end{tabular}}}}%
    \put(0.3160133,0.01560336){\color[rgb]{0,0,0}\makebox(0,0)[lt]{\lineheight{1.25}\smash{\begin{tabular}[t]{l}$P$\end{tabular}}}}%
    \put(0.08790434,0.01148529){\color[rgb]{0,0,0}\makebox(0,0)[lt]{\lineheight{1.25}\smash{\begin{tabular}[t]{l}$Q_0$\end{tabular}}}}%
    \put(0.90804833,0.22596065){\color[rgb]{0,0,0}\makebox(0,0)[lt]{\lineheight{1.25}\smash{\begin{tabular}[t]{l}$P^\dual$\end{tabular}}}}%
    \put(0,0){\includegraphics[width=\unitlength,page=2]{fibered-poly.pdf}}%
    \put(0.5433768,0.01460834){\color[rgb]{0,0,0}\makebox(0,0)[lt]{\lineheight{1.25}\smash{\begin{tabular}[t]{l}$Q_1$\end{tabular}}}}%
  \end{picture}%
\endgroup%
}
  \caption{A fibered polytope $\tP$. The fibers of $\pi_P$ are in red, and those of $\pi_{P^\dual}$ are in blue.}
  \label{fig:fibered-poly}
\end{figure}

The fibered polytopes corresponding to all the polytopes in $\PP$ fit into a \em{cutting datum} in $\t^\dual$ that is defined next.
\begin{definition}
  {\rm(Cutting datum)} \label{def:cut-datum} Given a polyhedral
  decomposition $\PP$ of $X$, and an $X$-inner product
  $g=(g_P)_{P \in \PP}$ on the spaces $\{\t_P\}_{P \in \PP}$
  (Definition \ref{def:x-idtt}), a \em{cutting datum} consists of a
  dual complex $B^\dual$, and inclusions of fibered polytopes (as in
  \eqref{eq:fibpoly})
  \begin{equation}
    \label{eq:ipdef}
    i_{\tP} : \tP \hra \t^\dual
  \end{equation}
  for all $P \in \PP$, such that for any $x \in P$ in a neighborhood
  of a face $Q \subset P$ the fiber $\pi_P^{-1}(x)$
  is a polytope
  contained in
  $\pi_Q^{-1}(\pi_Q(x))$ and is orthogonal to
  $P \cap \pi_Q^{-1}(\pi_Q(x))$ with respect to the inner product
  $g_Q$ on $\t_Q^\dual$. Here, the fibers of $\pi_P$, $\pi_Q$ are
  viewed as subsets of $\t^\dual$.
\end{definition}

\begin{figure}[h]
  \centering \scalebox{.8}{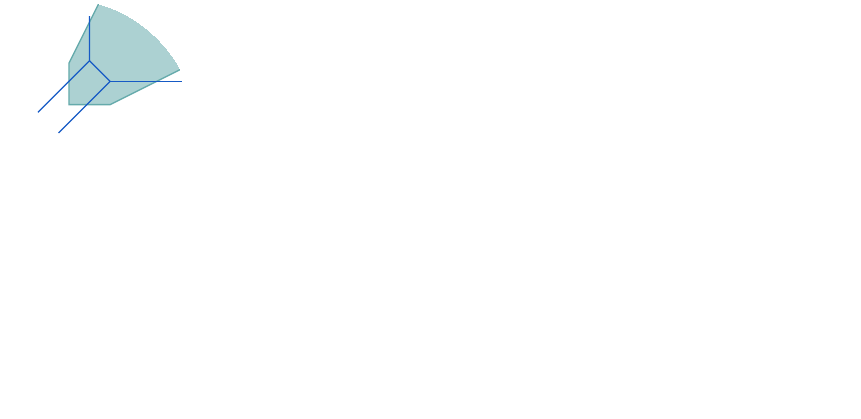}
  \caption{Top left: Polyhedral decomposition from Figure \ref{fig:mulcuteg1}. Bottom left: Cutting datum corresponding to the dual complex in Figure \ref{fig:mulcuteg1}. Right: The symplectic broken manifold is isomorphic to the inverse image of these polytopes under the tropical moment map.}
  \label{fig:broken-a3}
\end{figure}

\begin{remark}{\rm(The case of a polytope-independent $X$-inner product)}
  \label{rem:const-g}
  For many examples of multiple cuts, a dual complex exists when the
  $X$-inner product $g=(g_P)_P$ is polytope-independent as in
  Definition \ref{def:x-idtt}, that is, $g_P$ is the restriction of a
  $P$-independent inner product $g$ on $\t$. In this case, there is an
  embedding $i:B^\dual \hra \t^\dual$ such that any $P \in \PP$ is
  orthogonal to $P^\dual$.  For example, the dual complex in Figure
  \ref{fig:mulcuteg1} respects a polytope-independent inner product.
  Furthermore, in this case, the fibered polytopes are rectangles,
  that is, $\tP=P \times P^\dual$, where $P$ and $P^\dual$ are
  orthogonal in $\t^\dual$.  For the cutting datum in Figure
  \ref{fig:broken-a3}, the $X$-inner product is polytope-independent
  and is equal to the standard inner product on $\R^2$. The fibered
  polytopes $\tP_{01}$, $\tP_{12}$, $\tP_{23}$, $\tP_{30}$ are
  infinite rectangles.
\end{remark}

\begin{example}\label{ex:nonconst-metric}
  In Figure \ref{fig:fibered-poly}, suppose the vertices $Q_0, Q_1$ of
  $P$ are elements of $\PP$, and suppose that, in the $X$-inner
  product $g=(g_P)_P$, $g_{Q_0}$ is the standard inner product and
  $g_{Q_1}$ is such that $\pi_P^{-1}(Q_1)$ is orthogonal to $P$. Then,
  the fibered polytope in the figure respects the inner
  product. Indeed, in a neighborhood of $Q_i$, $i=0,1$, the fibers of
  $\pi_P$ are orthogonal to $P$ with respect to the inner product
  $g_{Q_i}$.
\end{example}

We expect that cutting data exist for a wide class of polyhedral
decompositions for suitable choices of $X$-inner products. In
Proposition \ref{prop:cutdat}, Examples \ref{ex:nodual} and \ref{ex:singles}, we list
some cases where the cutting data can be explicitly constructed.

\begin{proposition}\label{prop:cutdat}
  {\rm(Examples where a cutting datum exists)}
  Let $\PP$ be a polyhedral decomposition of $\t^\dual$. 
  \begin{enumerate}
  \item \label{part:cut1} Suppose the $1$-skeleton
    $\PP_{(\leq 1)}:=\cup_{P \in \PP, \dim(P) \leq 1}P$ of $\PP$ does
    not have a cycle. A cutting datum exists for the
    polytope-independent $X$-inner product. (A polytope-independent
    $X$-inner product is the restriction of an inner product on $\t$
    to $\t_P$ for all $P \in \PP$ as in Definition \ref{def:x-idtt}.)
  \item \label{part:cut2} Suppose $\dim(\t)=2$.  Suppose the edges in
    the $1$-skeleton can be oriented so that any vertex has at most
    two incoming edges and these two edges are not parallel. Then,
    there is an $X$-inner product for which a cutting datum exists.
 \end{enumerate}
\end{proposition}
\begin{proof}
  First, consider the case that there is a single vertex $v$ in
  $\PP$.  Fix any inner product on $\t$.  For each edge
  $e \in \PP_{(1)}$, choose any plane perpendicular to $e$ as the dual
  polytope $e^\dual$.  Since each of the polytopes in $\PP$ are
  simple, for any top dimensional polytope $P$ spanned by edges
  $\{e_i\}_i$, the planes $e_i^\dual$ intersect at a single point
  $P^\dual$. The other vertices of the polytope $v^\dual$ are
  similarly determined.

  Next to prove \eqref{part:cut1}, orient the edges and order the
  vertices in $\PP^{(0)}$ as $v_1,\dots,v_n$ so that for any $v_i$,
  there is at most one edge $(v_j, v_i)$ with $j<i$.
  It is enough to fix the fibered polytope  
  $\tilde v_i \subset \t^\dual$ so that the projections to faces
  match up, that is, if $P \in \PP$ contains vertices $v_i$, $v_j$, then
  the faces $v_{i,P}^\dual \subset v_i^\dual$, $v_{j,P}^\dual \subset v_j^\dual$
  are translates of each other. (Here, the notation $v_{i,P}^\dual$ is from Definition
  \ref{def:dualcomplex} \eqref{part:dualpair}, and note that since $v_i$ is $0$-dimensional $\tilde v_i \simeq v_i^\dual$.)
  The fibered polytopes $\tilde v_i$ are constructed as follows. 
  Assume  that $\tilde v_i \subset \t^\dual$ is fixed for
  every $i<k$.  Suppose for $v_k$, $e$ is an incoming edge. The facet
  $v_{k,e}^\dual$ of $v_k^\dual$ is fixed up to translation in the direction $e$, we
  translate $v_{k,e}^\dual$ so that it is close enough to $v_k$; this automatically fixes
  the facets of $\tilde v_k$ that intersect $v_{k,e}^\dual$, but we
  are free to choose the others, and fix $\tilde v_k$.

  To prove \eqref{part:cut2}, consider a vertex $v$ with incoming
  edges $e_1=(v_1,v)$, $e_2=(v_2,v)$. Assume that the fibered
  polytopes $\tilde v_1$, $\tilde v_2 \subset \t^\dual$ are fixed, and
  $\tilde w$ is not fixed for any other neighbor $w \in \PP^{(0)}$ of
  $v$.  We will choose a inner product $g_v$ for which $e_1$ and $e_2$
  are orthogonal, thus the facets $v^\dual_{e_1}$, $v^\dual_{e_2}$ are
  fixed, in particular, $v^\dual_{e_1}$ resp. $v^\dual_{e_2}$ is
  parallel to $e_2$ resp. $e_1$.  If there is only one other incident
  edge $e_3$ on $v$, then the face $v^\dual_{e_3}$ is also determined,
  and there is a unique choice of inner product $g_v$ for which
  $v^\dual_{e_3}$ is orthogonal to $e_3$. In case, there are four or
  more edges incident on $v$, then, we may fix any inner product for
  which $e_1$ and $e_2$ are orthogonal, and the polygon
  $\tilde v \subset \t$ can be completed with the edges $v^\dual_e$
  perpendicular to $e$ for all edges $e$ incident on $v$.
\end{proof}

\begin{figure}[ht]
  \centering \scalebox{.8}{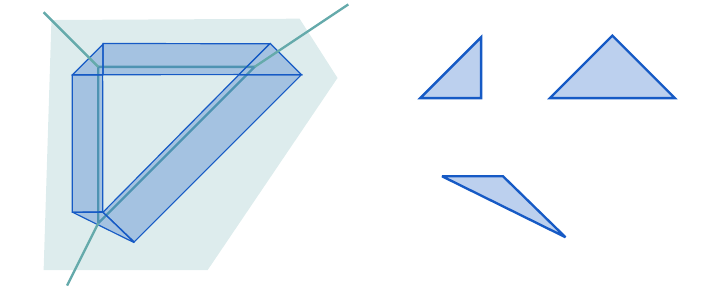}
  \caption{Dual complex with varying inner products.}
  \label{fig:dual-complex2}
\end{figure}

\begin{example}\label{ex:nodual}
  The cutting datum in Figure \ref{fig:dual-complex2} is given by
  applying the algorithm in the proof of Proposition \ref{prop:cutdat}
  \eqref{part:cut2}.  For vertices $P_1, P_2 \in \PP$, we take the
  inner products $g_{P_0}$, $g_{P_1}$ to be the standard ones, and for
  the vertex $P_3 \in \PP$, $g_{P_3}$ is the inner product for which
  $\{(1,0), (1,1)\}$ is an orthonormal basis. Note that in this
  example, it is not possible to construct a dual complex if the
  $X$-inner product $(g_P)_P$ is polytope-invariant (as in Definition
  \ref{def:x-idtt}); see Figure \ref{fig:dual-complex}.
\end{example}

\begin{figure}[ht]
  \centering \scalebox{.8}{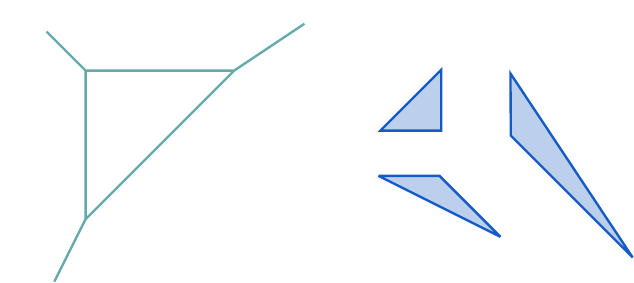}
  \caption{Failed attempt at constructing a dual complex assuming each of the inner products $g_{P_1}$, $g_{P_2}$, $g_{P_3}$ to be standard.}
  \label{fig:dual-complex}
\end{figure}

\begin{example}\label{ex:singles} 
  {\rm(Intersection of polyhedral decompositions)} Cutting data for a pair of
  transversely intersecting polyhedral decompositions can be combined
  to produce a cutting datum for the intersection. A pair of  polyhedral decompositions
  $\PP_0$, $\PP_1$ of $\t^\dual$ intersect \em{transvsersely} if any pair of polytopes $P_0 \in \PP_0$, $P_1 \in \PP_1$ intersect transversely. For such a pair,  we define a new polyhedral decomposition
  \[\PP:=\PP_0 \cap \PP_1:=\{P_{01}:=P_0 \cap P_1 : P_0 \in \PP_0, P_1 \in \PP_1\}. \]
Note that the cut locus of $\PP$ is the union of the cut loci of $\PP_0$ and $\PP_1$. 
Assuming both $\PP_0$, $\PP_1$ have a cutting datum, a cutting datum for $\PP$ is given as follows. For any minimal dimensional polytope $P_{01}$, the dual complex $P_{01}^\dual$ is the product $P_0^\dual \times P_1^\dual$. The inner product $g_{P_{01}}$ is defined as the orthogonal product of $g_{P_0}$ and $g_{P_1}$ so that $P_0^\dual$ and $P_1^\dual$ are orthogonal faces of $P_{01}^\dual$. The thickened fibered polytope is $\tilde P_{01}:= i_{P_0}(\tilde P_0) \cap i_{P_1}(\tilde P_1)$ where $i_{P_0}$, $i_{P_1}$ are from \eqref{eq:ipdef}. See Figure \ref{fig:singles2}.

  This method produces a family of dual complexes $B^\dual_\delta$ (with cutting data) for the polyhedral decomposition $\PP$, where $\delta >0$ is the ratio between the sizes of the dual complexes $B^\dual_{\PP_0}$ and $B^\dual_{\PP_1}$. Indeed, assuming $\delta<1$, we may perform the construction in the previous  paragraph by replacing the duals $P_0^\dual$ of the polytopes $P_0 \in \PP_0$ by $\delta P_0^\dual$.

  The simplest instance of this construction is a multiple cut consisting of two single cuts shown in Figure \ref{fig:break1}. The dual complex is a rectangle, and we obtain a family of possible complexes by varying the ratio of the sides of the rectangle. This ends the Example. 
\end{example}

\begin{figure}[ht]
  \centering \scalebox{.8}{
\begingroup%
  \makeatletter%
  \providecommand\color[2][]{%
    \errmessage{(Inkscape) Color is used for the text in Inkscape, but the package 'color.sty' is not loaded}%
    \renewcommand\color[2][]{}%
  }%
  \providecommand\transparent[1]{%
    \errmessage{(Inkscape) Transparency is used (non-zero) for the text in Inkscape, but the package 'transparent.sty' is not loaded}%
    \renewcommand\transparent[1]{}%
  }%
  \providecommand\rotatebox[2]{#2}%
  \newcommand*\fsize{\dimexpr\f@size pt\relax}%
  \newcommand*\lineheight[1]{\fontsize{\fsize}{#1\fsize}\selectfont}%
  \ifx\svgwidth\undefined%
    \setlength{\unitlength}{427.09990734bp}%
    \ifx\svgscale\undefined%
      \relax%
    \else%
      \setlength{\unitlength}{\unitlength * \real{\svgscale}}%
    \fi%
  \else%
    \setlength{\unitlength}{\svgwidth}%
  \fi%
  \global\let\svgwidth\undefined%
  \global\let\svgscale\undefined%
  \makeatother%
  \begin{picture}(1,0.35192442)%
    \lineheight{1}%
    \setlength\tabcolsep{0pt}%
    \put(0,0){\includegraphics[width=\unitlength,page=1]{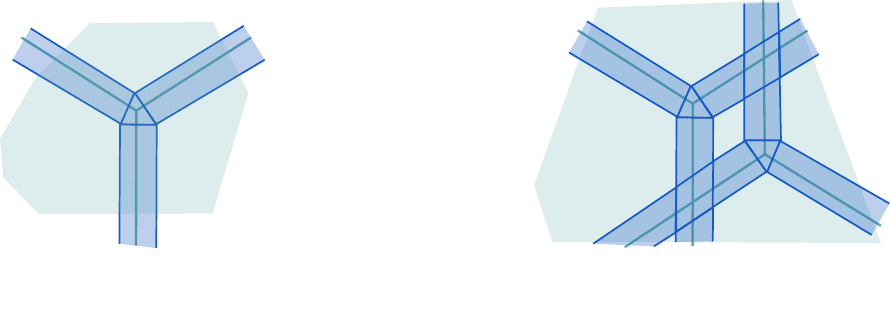}}%
    \put(0.05987592,0.08435534){\color[rgb]{0,0,0}\makebox(0,0)[lt]{\lineheight{1.25}\smash{\begin{tabular}[t]{l}$\PP_0$\end{tabular}}}}%
    \put(0,0){\includegraphics[width=\unitlength,page=2]{singles2.pdf}}%
    \put(0.48981407,0.25965001){\color[rgb]{0,0,0}\makebox(0,0)[lt]{\lineheight{1.25}\smash{\begin{tabular}[t]{l}$\PP_1$\end{tabular}}}}%
    \put(0.74684461,0.03706334){\color[rgb]{0,0,0}\makebox(0,0)[lt]{\lineheight{1.25}\smash{\begin{tabular}[t]{l}$\PP_0 \cap \PP_1$\end{tabular}}}}%
  \end{picture}%
\endgroup%
}
  \caption{Intersection of cuts $\PP_0$, $\PP_1$, and the cutting datum of $\PP_0 \cap \PP_1$.}
  \label{fig:singles2}
\end{figure}

With a cutting datum in hand, we define the symplectic broken manifold.
\begin{definition}
  \label{def:brokenxxsymp}
  \index{Broken manifold! Symplectic broken manifold $\ol \XX_P^\om$}
  {\rm(Symplectic broken manifold)} Suppose $(X,\PP,\Phi)$ is a
  tropical Hamiltonian manifold, and the polyhedral decomposition
  $\PP$ has a cutting datum (Definition \ref{def:cut-datum}) given by
  inclusions of fibered polytopes $\tP \subset \t^\dual$ for all
  $P \in \PP$.

  \begin{enumerate}
  \item The \em{symplectic broken manifold} corresponding to the
    tropical Hamiltonian action $(X,\Phi,\PP)$ is a disjoint union
  \[\XX_\PP\quad \text{or} \quad \XX:=\sqcup_{P \in \PP} \XX_P^\om,\]
  where we use the notation $\XX$ when $\PP$ is clear from the context.
  Here, 
  \begin{equation}
    \label{eq:XolP}
     \XX_P^\om :=\Phinv(\tP)
  \end{equation}
  is a symplectic manifold with corners.  In the case when the inner
  product \eqref{eq:idtt} for $\PP$ is rational, compactifications of
  the components of the broken manifold may be defined as
  \begin{equation}
    \label{eq:olXolP}
    \ol \XX_P^\om:=\XX_P^\om/\sim,  
  \end{equation}
  where the equivalence relation $\sim$ quotients any boundary
  component $\Phinv(Q)$, $Q \in \Facets(\tP)$ by the $S^1$-action
  generated by the vector $\nu_Q \in \t_\Z$ normal to the hyperplane
  $Q \in \t^\dual$.  The space $\ol \XC_P^\om$ is a smooth symplectic
  manifold if all faces $F \subset \tP$ are smooth (as in Definition
  \ref{def:delz}).  Otherwise $\ol \XC_P^\om$ is a symplectic
  orbifold.
  See Notation \ref{note:symp-vs-ac} regarding the superscript $\om$.
\item {\rm(Relative divisors of a symplectic broken manifold)}
  \label{part:bdrydiv}
  \index{Relative divisor} Suppose the inner product \eqref{eq:idtt}
  for $\PP$ is rational, and compactifications $\ol \XX_P^\om$ of the
  components of the broken manifold $\XX$ are orbifolds (see
  \eqref{eq:olXolP}).  Then, for any $P \in \cP$ and any facet $Q$ of the thickening $\tP$, the quotient
  \[Y_Q:=\Phinv(Q)/\exp(\R \nu_Q) \subset \ol \XX_P^\om \]
  is a \em{relative divisor} of $\ol \XC_P^\om$.  Here, the normal
  $\nu_Q \in \t_\Z$ to the facet $Q$ generates a subgroup
  $\exp(\R \nu_Q) \simeq S^1$.  The relative divisor $Y_Q$ is \em{
    horizontal} resp. \em{vertical} if the facet $Q \subset \tP$
  corresponds to a facet of $P$ resp. $P^\dual$ (see
  \eqref{eq:facetsolp}).  Thus
  \[  \XC_P^\om = \ol \XC_P^\om \ssm \bigcup_{ Q \subset \tP} Y_Q \]
  is the complement of all the relative divisors $Y_Q, Q \subset \tP$
  of $\ol \XC_P^\om$.
\end{enumerate} 
\end{definition}

\begin{example}\label{ex:1cut-symbr}
  In the case of a single cut $(X,\Phi,c)$ the symplectic broken
  manifold consists of the components
  \[\{\Phi>c\}/\sim, \quad \{\Phi<c\}/\sim, \quad \Phinv([c-\eps,c+\eps])/\sim,\]
  where $\sim$ mods out the boundaries $\Phinv(c-\eps)$ and $\Phinv(c+\eps)$ by $S^1$-actions.
\end{example}

\begin{remark} {\rm(Components of the symplectic broken manifold as thickenings)}
  \label{rem:fibxp} 
  Components of the symplectic broken manifold are fibrations over cut
  spaces whose fibers are compactifications of the complex tori. That
  is, there is a fibration
  \begin{equation}
    \label{eq:fibxp}
    V_{P^{\dual}} \to \XC_P^\om \xrightarrow{\pi_P} X_P^\om,
  \end{equation}
  whose fiber $V_{P^{\dual}}$ is a symplectic manifold with corners
  and a $T_P$-action whose moment map has image $P^\dual$.  Therefore,
  the symplectic broken manifold $\XX_P^\om$ they can be regarded as
  the thickening of the cut space $X_P^\om$ for any $P \in \PP$.  In
  the particular case when the $X$-inner product (Definition
  \ref{def:x-idtt}) is rational, the fiber $V_{P^{\dual}}$ has a
  compactification $\ol V_{P^{\dual}}$ obtained by modding out
  boundaries by $S^1$-actions. The compactification $\ol V_{P^\dual}$
  is a $T_P$-orbifold whose moment map is $P^\dual$.
\end{remark}

\begin{remark}\label{rem:symptorus}
  {\rm(Torus bundles on cut spaces)} \index{Torus bundle! on symplectic cut space $\ol Z_P^\om$}
  Torus bundles over cut spaces are used in the following sections to
  define neck-stretched almost complex structures. In the symplectic
  setting, for any $P \in \PP$,
  \[\Phinv(P^\circ) \to X_P^\om,\]
  is a principal $T_P$-bundle. It extends to a $T_P$-orbifold bundle
  \begin{equation}
    \label{eq:zpom}
    \ol Z_P^\om \to \ol X_P^\om  
  \end{equation}
  for which $\ol \XC_P^\om \to \ol X_P^\om$ is the associated
  $V_{P^\dual}$-bundle, that is there is a $T_P$-diffeomorphism
  \[\ol \XC_P^\om=\ol Z_P^\om \times_{T_P} V_{P^\dual},\]
  where $V_{P^\dual}$ is a toric variety whose moment polytope is
  $P^\dual$ (Remark \ref{rem:fibxp}).  We also observe that the
  complement of the vertical divisors of $\ol \XX_P^\om$ is
  $T_P$-diffeomorphic to $\ol Z_P^\om \times P^\dual$.
\end{remark}

\begin{example}
  For the multiple cut in Figure \ref{fig:break1}, the dual complex is
  a rectangle and the symplectic broken manifold $\XX$ is as in Figure
  \ref{fig:multpiece}. Relative submanifolds $\ol X_{P_{ij}}^\om$ and
  $\ol X_{P_\cap}^\om$ are thickened into toric fibrations
  $\ol \XC_{P_{ij}}^\om$ and $\ol \XC_{P_\cap}^\om$
  \[\P^1 \to \ol \XC^\om_{P_{ij}} \to \ol X^\om_{P_{ij}}, \quad
    (\P^1)^2 \to \ol \XC^\om_{P_\cap} \to \ol X^\om_{P_\cap}\]
  in the symplectic broken manifold.
\end{example}

\begin{figure}[ht]
  \centering \scalebox{.8}{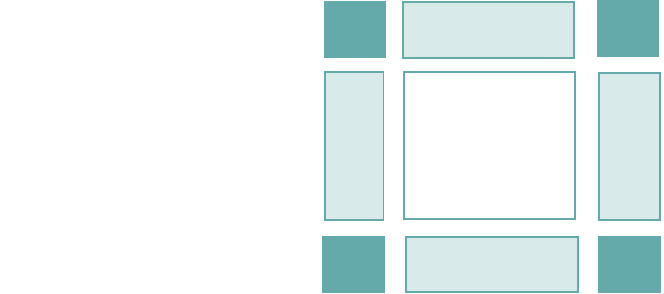}
  \caption{Dual complex and broken manifold for the multiple cut in Figure \ref{fig:break1}}.
  \label{fig:multpiece}
\end{figure}

\section{Neck-stretched almost complex structures}
\label{sec:cylacs}

We define a family of almost complex structures on the manifold $X$,
called neck-stretched almost complex structures. These almost complex
structures are ``cylindrical'' in the sense of the following
Definition:
\begin{definition}\label{def:cylp}
  {\rm(Cylindrical almost complex structure)}
  \begin{enumerate}
  \item {\rm($P$-cylinder)} Given a polytope $P \in \PP$, a $T_{P,\C}$-principal bundle
    \[Z_\C \to M\]
    on a manifold $M$ is called a $P$-\em{cylinder}.  By choosing a
    reduction of structure group to the maximal compact subgroup
    $T_P \subset T_{P,\C}$ we obtain a $T_{P,\C}$-equivariant
    diffeomorphism
    \[Z_{\C} \simeq Z \times i\t_P\]
    where $Z \to M$ is a principal $T_P$-bundle and the
    $T_{P,\C}$-action on $Z \times i\t_P$ is
    \[te^{is} (z,s_0) = (tz,s_0+s), \quad t \in T_P, s \in \t_P.\]
  \item Let $Z_\C \to M$ be a $P$-cylinder.  An almost complex
    structure $J$ is \em{$P$-cylindrical} if and only if
    \begin{enumerate}
    \item there exists an almost complex structure $J_M$ on the base
      manifold $M$ such that the projection
      \[ \pi :Z_\C \to M \]
      is $\left(J, J_{M}\right)$ holomorphic almost complex, that is,
      $d\pi \circ J = J_M \circ d\pi$, and we denote
      \begin{equation}
        \label{eq:dpij}
      D\pi(J):=J_M;  
      \end{equation}
       \label{rep:jm} 
    \item there exists a connection one-form
      $\alpha \in \Om^1(Z,\t_P)$ on the $T_P$-bundle $Z \to M$
      such that the horizontal sub-bundle
      \begin{equation}
        \label{eq:connJ}
      H:=\ker(\alpha) \subset TZ \cong TZ \times \{ 0 \} \subset TZ \times \R \cong TZ_{\C}  
      \end{equation}
      is $J$-invariant; and
    \item on any fiber $\pinv(m) \subset Z_\C$, $J$ is standard in the
      following sense:  For any point $z \in \pinv(m)$ the map
      \[T_{P,\C} \to \pinv(m), t \mapsto tz\]
      is a biholomorphism.
    \end{enumerate}
    As a result, $J$ is invariant under the $T_{P,\C}$-action on
    $Z_{\C}$. Denote by 
    \[ \J^{\cyl}(Z_{\C}):= \Set{ J \in \J(Z_{\C}) | (i) - (iii) } \]
    \index{J@$\J^\cyl$}the space of $P$-cylindrical almost complex
    structures on $Z_\C$. Note that a $P$-cylindrical almost complex
    structure $J$ is determined by its projection $d\pi(J)$ to the
    base $M$ called the \em{base almost complex structure}
    \index{Base almost complex structure} and its \em{associated
      connection one-form} $\alpha(J)$. \index{Connection one-form
      associated to a cylindrical almost complex structure}
  \end{enumerate}
\end{definition}

Given a tropical Hamiltonian action $(X,\PP,\Phi)$, we aim to define
``neck-stretched manifolds'' where the neck region associated to
$P \in \PP$ has a $P$-cylindrical almost complex structure. To achieve
this end, a subset of the manifold $X$ with the action of the torus
$T_P$ has to be identified with a subset of a $P$-cylinder. This
identification is made via a symplectic cylindrical structure defined
below, where a sufficiently small $T_P$-invariant neighborhood of $\Phinv(P)$ in $X$ 
is identified to a product $\Phinv(P) \times P^\dual$. In
what follows the identification between $\t_P$ and $\t_P^\dual$ from
\eqref{eq:idtt} is crucial.  For the symplectic structure, we view
$P^\dual$ as a subset of $\t_P^\dual$
 via the identification $\t_P \simeq \t_P^\dual$ from \eqref{eq:idtt} 
 and the $T_P$-moment map on
$\Phinv(P) \times P^\dual$ is given by the projection to
$P^\dual$.  For the $P$-cylindrical complex structure, we view
$P^\dual$ as a subset of $\t_P$ so that the fibers of the projection
\[ \Phinv(P) \times P^\dual \to \Phinv(P)/T_P \]
are subsets of $T_{P,\C}$ via the map 
\begin{equation}
  \label{eq:tp-id}
  T_P \times P^\dual \subset 
  T_P \times \t_P \xrightarrow{(t,\xi) \mapsto t \exp(i\xi)} T_{P,\C}. 
\end{equation}
   Consequently
  $\Phinv(P) \times P^\dual$ has a partial $T_{P,\C}$-action.  That
  is, there is an infinitesimal action
\begin{equation} \label{eq:pact} \t_{P,\C} \to \Vect(\Phinv(P) \times
  P^\dual) \end{equation} 
whose flows satisfy the axioms for an action of $T_{P,\C}$ wherever
they are defined.

\begin{definition}
  \label{def:sympcylstr}
  {\rm(Symplectic cylindrical structure on tropical Hamiltonian actions)}
  \index{Cylindrical! Symplectic cylindrical structure} Let
  $(X,\PP, \Phi)$ be a symplectic manifold with a tropical Hamiltonian
  action. A \em{symplectic cylindrical structure}
  $\ul \phi=(\phi_P)_{P \in \PP}$ consists of a $T_P$-equivariant
  symplectomorphism $\phi_P$
  \begin{equation}
    \label{eq:sympcylX}
    \Phinv(\tP) \xrightarrow{\phi_P} (\Phinv(P) \times
    P^\dual, \ol \om), \quad \ol \om := (\om_X|\Phinv(P)) +
    \d\bran{\alpha_P, \Pi_{P^\dual} - c_{P^\dual}} 
  \end{equation}
for each polytope $P \in \PP$. Here 
\begin{enumerate}
\item $\alpha_P \in \Om^1(\Phinv(P),\t_P)$ is a $T_P$-connection
  one-form;
\item \label{part:sympstr2}
  $\Pi_{P^\dual}: \Phinv(P) \times P^\dual \to P^\dual$ is the projection to the second factor,
  $c_{P^\dual}\in P^\dual$ is the constant that is the image of the
  polytope $P \subset \t^\dual$ under the projection
  $\pi_{P^\dual} : \tP \to P^\dual$ from Definition \ref{def:fibpoly},
  \label{rep:under}
  and $\bran{\cdot,\cdot}$ is the inner product on $\t_P$ from \eqref{eq:idtt}, 
  and so,
  $\bran{\alpha_P, \pi_{P^\dual} - c_{P^\dual}} $ is a one-form on
  $\Phinv(P) \times P^\dual$;
\item the dual polytopes $P^\dual$ are assumed to be small enough that
  the forms in the right hand side of \eqref{eq:sympcylX} are
  symplectic;
\item \label{part:sympstr4} $\phi_P$ satisfies
  $\Phi \circ \phi_P=\pi_P \circ \Phi$, where in the left hand side,
  $\Phi : \Phinv(P) \times P^\dual \to P$ is independent of the second
  domain component;
\end{enumerate}
and the maps $\ul \phi$ satisfy the following (Patching) condition:
  \begin{itemize}
  \item[] {\rm(Patching)}\label{item:patching} 
    For any pair $Q \subset P$, in the overlap
    region $\tQ \cap \tP$ the $T_{P,\C}$-action induced by $\phi_P$
    (see the explanation preceding this definition) is the restriction
    of the $T_{Q,\C}$-action induced by $\phi_Q$.
  \end{itemize}
  This ends the Definition.
\end{definition}
\begin{remark}\label{rem:moment}
  For any $P \in \PP$,
  viewing $P^\dual$ as a subset of $\t_P^\dual$ via the $\t_P$-inner product in \eqref{eq:idtt}, 
  the second component of the symplectic
  cylindrical structure map, namely $\Pi_P \circ \phi_P$, satisfies
  \[\Pi_{P^\dual} \circ \phi_P=\pi_{P^\dual} \circ \Phi, \]
  and therefore, it is a $T_P$-moment map. Here, the projection
  $\pi_{P^\dual} : \tP \to P^\dual$ is from Definition
  \ref{def:fibpoly}.
\end{remark}
In the above definition of the symplectic cylindrical structure maps,
the \hyperref[item:patching]{(Patching)} 
condition is equivalent to the following consistency
condition on connection one-forms: $(\alpha_P)_{P \in \PP}$.
\begin{lemma}\label{lem:conn-consis}
  {\rm(Consistency for connection one-forms)} The collection of
  symplectic cylindrical structure maps $(\phi_P)_P$ satisfy the
  \textup{\hyperref[item:patching]{(Patching)}}
  condition if and only if for any
  $x \in \Phinv(\tQ \cap \tP)$,
  \begin{equation}
    \label{eq:consis}
    \ker(\alpha_P(x))=\ker(\alpha_Q(x)) \oplus (\t_Q/\t_P)x,
  \end{equation}
  where we view the quotient $\t_Q/\t_P$ as a subspace of $\t_Q$ by
  the pairing \eqref{eq:idtt} on $\t_Q$.
\end{lemma}
\begin{proof}
  For any $x \in \Phinv(\tQ \cap \tP)$,
  \[T_xX = \t_{Q,\C} x \oplus \ker(\alpha_Q(x))= \t_{P,\C} x \oplus \ker(\alpha_P(x)),\]
  and in both decompositions, the summands are $\om$-complements. The
  (Patching) condition implies that $\t_{P,\C} x \subset \t_{Q,\C} x$,
  which is equivalent to \eqref{eq:consis}.
\end{proof}

\begin{remark}
  In Definition \ref{def:sympcylstr}, condition \eqref{part:sympstr4}
  exists in order to simplify exposition. This condition allows us to
  state the properties of fibered polytopes purely at the level of
  polytopes without using the map $\Phi$. In the construction of the
  symplectic cylindrical structure, we modify the non-essential
  components of the tropical moment map to achieve this condition, see
  \eqref{eq:modify}.
\end{remark}

For the moment, we assume the existence of symplectic cylindrical
structures and use them to define a family of neck-stretched almost
complex structures.  The proof of the existence of symplectic
cylindrical structures on tropical Hamiltonian actions is deferred to
Proposition \ref{prop:symcyl} at the end of Chapter \ref{chap:bsymp}.
The cylindrical structure maps $\{\phi_P\}_{P \in \PP}$ in
\eqref{eq:sympcylX} are fixed throughout the book.
\begin{remark}
  {\rm(A decomposition of polytopes)} We describe a polyhedral decomposition that is used to define neck-stretched manifolds.
  Let $(X,\PP,\Phi)$ be a tropical Hamiltonian action and let $B^\dual$ be a
  dual complex.  For any $P \in \PP$, let
  \begin{equation}
    \label{eq:pblack}
    P^\fillblack:=P \bs (\cup_{Q \subset P}\tQ),  
  \end{equation}
  be the complement of fibered neighborhoods of proper faces of
  $P$. Here, we recall from \eqref{eq:ipdef} that the cutting datum
  gives an embedding $\tQ \hra \t^\dual$ for any fibered polytope
  $\tQ$ corresponding to $Q \in \PP$.  Corresponding to any facet
  $Q \subset P$, $P^\fillblack$ has a facet $Q^\fillblack$.  Let
  \begin{equation}
    \label{eq:olpblack}
    \tP^\fillblack:=\pi_P^{-1}(P^\fillblack) \subset \tP
  \end{equation}
  be the thickening of $P^\fillblack$.  For a pair $Q \subset P$ with
  $\codim_P(Q)=1$, the fibered polytopes
  $\tP^\fillblack ,\tQ^\fillblack \subset \t^\dual$ share a facet,
  which is isomorphic to $Q^\fillblack \times P^\dual$.  The
  image of $\Phi$ is covered by the union of thickenings
  \begin{equation}
    \label{eq:phicover}
    \im(\Phi)=\cup_{P \in \PP} i_{\tP}(\tP^\fillblack); 
  \end{equation}
  see Figure \ref{fig:stretch}.  The partition of $\im(\Phi)$ pulls
  back to a partition of the symplectic manifold $(X,\om_X)$
  \begin{equation}
    \label{eq:xnupre}
    X:=\left( \bigsqcup_{P \in \PP} \Phinv( \tP^\fillblack) \right) / \sim
  \end{equation}
  into manifolds with corners, where the identifications are by the
  equivalence relation $\sim$ along the boundaries and are induced by
  the inclusions $\Phinv( \tP^\fillblack) \to X$.   The
  symplectic cylindrical structure map $\ul \phi=(\phi_P)_P$ may be
  used to rewrite the decomposition in \eqref{eq:xnupre} as
  \begin{equation}
    \label{eq:xnupre1}
    X:=\left( \bigsqcup_{P \in \PP} \Phinv( P^\fillblack)  \times P^\dual \right) / \sim
  \end{equation}
  as in Figure \ref{fig:stretch}.  In \eqref{eq:xnupre1}, the
  equivalence $\sim$ identifies the boundary components
  \begin{equation}
    \label{eq:bdryphicover}
    \Phinv( Q^\fillblack) \times Q^\dual \supset \Phinv(Q^\fillblack) \times P^\dual \xrightarrow{\sim} \Phinv(Q^\fillblack) \times P^\dual \subset \Phinv( P^\fillblack)
    \times P^\dual    
  \end{equation}
  for all pairs $Q \subset P$, $\codim_P(Q)=1$.
\end{remark}
\begin{definition}\label{def:neckstr}
  {\rm(Neck-stretched manifolds)} \index{Neck-stretched manifold} Let
  $(X,\PP,\Phi)$ be a tropical Hamiltonian action with a symplectic cylindrical
  structure.  For any $\nu \in \R_{\geq 1}$, define a \em{
    neck-stretched manifold} $X^\nu$ as
  \begin{equation}
    \label{eq:xnudef}
    X^\nu=   \left( \bigsqcup_{P \in \PP}( \Phinv(P^\fillblack) \times \nu P^\dual) \right) / \sim,
  \end{equation}
  where the equivalence relation $\sim$ is exactly as in
  \eqref{eq:bdryphicover}, that is, for all pairs $Q \subset P$,
  $\codim_P(Q)=1$, the boundary component
  $\Phinv(Q^\fillblack) \times \nu P^\dual$ in
  $\Phinv( P^\fillblack) \times \nu P^\dual$ is identified with the
  corresponding boundary component in
  $\Phinv( Q^\fillblack) \times \nu Q^\dual$ by the identity map.
This ends the Definition.
\end{definition}

\begin{remark}\label{rem:xnu-mani}
  In order to define a manifold structure on $X^\nu$, we need to give
  identifications between the tubular neighborhoods of boundaries and
  corners that are glued in \eqref{eq:xnudef}.  These identifications
  are the same as those between the corresponding pieces of the
  decomposition \eqref{eq:xnupre1} of $(X,\om_X)$.
\end{remark}

\begin{remark}
  {\rm(Projection to dual complex)}
  There is a natural projection map
  \begin{equation}
    \label{eq:pinb-ch3}
    \pi_{\nu B^\dual} : X^\nu \to \nu B^\dual  
  \end{equation}
  which is defined on the subset
  $\Phinv(P^\fillblack) \times \nu P^\dual$ as projection to
  $\nu P^\dual$ for any $P \in \PP$. The map $\pi_{\nu B^\dual}$ is
  continuous.
\end{remark}

\begin{remark} \label{rep:thus} By Definition \ref{def:neckstr}, the
  neck-stretched manifold $X^\nu$ is equipped with
\begin{enumerate}
\item for each $P$, a $P$-cylindrical structure on the subset
  $\Phinv(P^\fillblack) \times \nu P^\dual \subset X^\nu$; that is,
  there is a projection
  \begin{equation}
    \label{eq:xnupproj}
  \Phinv(P^\fillblack) \times \nu P^\dual \to \Phinv(P^\fillblack)/T_P  
  \end{equation}
  whose fibers are $T_P \times \nu P^\dual \subset T_{P,\C}$; and
\item a symplectic form on the base manifold
  $\Phinv(P^\fillblack)/T_P$ that is a $T_P$-reduction of the
  symplectic form $\om_X$ on $X$.
\end{enumerate}
\end{remark}

\begin{figure}[ht]
  \centering \scalebox{.8}{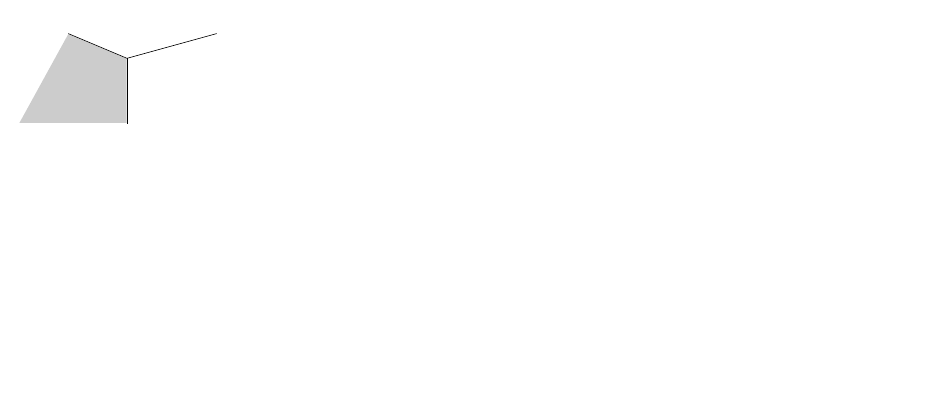}
  \caption{Stretching}
  \label{fig:stretch}
\end{figure}
\begin{remark}
  {\rm(A warning about Figure \ref{fig:stretch})} Figure \ref{fig:stretch} is a
  schematic representation of neck-stretched manifolds.  We recall
  symplectic broken manifolds are represented by the images of their
  tropical moment map, for example in Figure \ref{fig:broken-a3}. In contrast, we do not
  have a moment map on neck-stretched manifolds. A larger polytope is
  just used to indicate a larger cylinder.
\end{remark}

\begin{remark}
  The neck-stretched manifold $X^\nu$ is diffeomorphic to $X$, but the
  diffeomorphism is not canonical. As a result, there is no canonical
  symplectic form on $X^\nu$.
\end{remark}

\begin{example}\label{ex:1cutneck}
  {\rm(Neck-stretched manifolds for a single cut)} We describe
  neck-stretched manifolds in the case of a single cut.  Let $(X,\Phi,c)$
  be a tropical Hamiltonian action with a single cut along $\Phinv(c)$, and
  whose polytopes are as in Example \ref{ex:singlecutdual}.  The
  symplectic cylindrical structure consists of an $S^1$-equivariant
  symplectomorphism defined on a neighborhood
  \begin{equation}
    \label{eq:1cutsymcyl}
  \Phinv([c-\eps,c+\eps]) \xrightarrow{\phi_{P_0}} (\Phinv(c) \times [-\eps,\eps], \om_{\qu} + \d(t\alpha)),    
  \end{equation}
  where $\om_{\qu}$ is the reduced symplectic form on the quotient
  $\Phinv(c)/S^1$, $t \in (-\eps,\eps)$ is the coordinate function,
  and $\alpha \in \Om^1(\Phinv(c))$ is a connection one-form on the
  $S^1$-bundle $\Phinv(c) \to \Phinv(c)/S^1$. The other symplectic
  cylindrical structure maps $\phi_{P_-}$, $\phi_{P_+}$ are trivial
  since $P_+$, $P_-$ are top-dimensional. The neck-stretched manifold
  is then 
  \begin{equation}
    \label{eq:Xnu1eg}
  X^\nu:=\{\Phi \leq c- \eps\} \cup_{\Phinv(c)} (\Phinv(c)\times [-\eps \nu,\eps \nu]) \cup_{\Phinv(c)} \{\Phi \geq c+\eps\},   
  \end{equation}
  where the attachment maps are given by identifying
  $\{\Phi=c\pm \eps\}$ to $\{\Phi=c\}$ via the symplectic cylindrical
  structure $\phi_{P_\pm}$, and in the middle piece, $\{\Phi=c\}$ is
  identified to the ends $\{\Phi=\pm \eps \nu\}$. \label{rep:xpnu-eg}
  \footnote{The interval $[-\eps \nu, \eps \nu]$ in \eqref{eq:Xnu1eg}
    is chosen to be consistent with definitions, there is no canonical
    choice of end-points, as long as the length of the segment is
    $\nu$ times the neck length in $X$.}
\end{example}

\begin{definition}{\rm(Cylindrical almost complex structures on neck-stretched manifolds)}
  \label{def:cylneckst}
  Let $(X,\PP)$ be a tropical Hamiltonian action and let
  $\{X^\nu\}_\nu$ be a sequence of neck-stretched manifolds. Recall
  from \eqref{eq:xnudef} that a neck-stretched manifold has a
  decomposition
  \[X^\nu=   \left( \bigsqcup_{P \in \PP} \Phinv(P^\fillblack) \times \nu P^\dual \right) / \sim,\]
  where $\sim$ identifies boundaries of the different components. 
  \begin{enumerate}
  \item An almost complex structure $J^\nu$ on the neck-stretched
    manifold $X^\nu$ is \em{cylindrical} if $J^\nu$ is
    $P$-cylindrical in the sense of Definition \ref{def:cylp} in the
    subset
    \[\Phinv(P^\fillblack) \times \nu P^{\dual,\circ} \subset X^\nu.\]
    The space of cylindrical almost complex structures on $X^\nu$ is
    denoted by
    \[\J^\cyl(X^\nu).\]
    \index{J@$\J^\cyl$!$\J^\cyl(X^\nu)$}We say that $(X^\nu,J^\nu)$
    is a \em{family of neck-stretched almost complex structures}
    \index{Family of neck-stretched almost complex structures} if
    there are $\nu$-independent cylindrical almost complex structures
    $J_P$ on $\Phinv(P^\fillblack) \times \t_P^\dual$ for all $P$,
    such that on the $P$-cylindrical subset
    $\Phinv(P^\fillblack) \times \nu P^{\dual,\circ}$ of $X^\nu$,
    $J^\nu$ is the restriction of $J_P$.
  \item {\rm(Local tamedness and compatibility)} \label{part:tamefib}
    \index{Almost complex structure! Locally tamed} \index{Almost
      complex structure! Locally compatible} Let $(X^\nu,J^\nu)$ be a
    family of neck-stretched almost complex structures.  We say that
    each element $J^\nu$ is \em{locally tamed} resp. \em{locally
      compatible} if $J^1$ is $\om_X$-tamed
    resp. $\om_X$-compatible. (The definition is based on the
    observation that the neck-stretched manifold $X^1$ has a canonical
    diffeomorphism to $(X,\om_X)$, and this is not so for $X^\nu$,
    $\nu \neq 1$.)
  \item {\rm(Local strong tamedness)} \label{part:local-strong-t}
    \index{Almost complex structure!
      Locally strongly tamed} A cylindrical locally tamed almost
    complex structure $J$ on $X^\nu$ is \em{locally strongly tamed}
    if for all $P \in \PP$ the connection one-form $\alpha_{P,J}$
    underlying the $P$-cylindrical almost complex structure on
    $\Phinv(P^\fillblack) \times \nu P^{\dual,\circ}$ in Definition
    \ref{def:cylp} is equal to
    the connection one-form underlying the symplectic cylindrical
    structure in Definition
    \ref{def:sympcylstr}.
  \end{enumerate}
\end{definition}
\begin{remark}
  \label{rem:oncyl}
  The following are some remarks on cylindrical almost complex structures on neck-stretched manifolds.
  \begin{enumerate}
  \item \label{part:oncyl1}
    The cylindrical almost complex structures defined above in Definition \ref{def:cylneckst}
    is sometimes referred to as \em{$X$-cylindrical} 
    \index{Almost complex structure! X@ $X$-cylindrical}
    to emphasize its dependence on the $X$-inner product \eqref{eq:idtt} on $\{\t_P\}_P$. Recall that the dependence on the inner product arises because of the identification \eqref{eq:tp-id} between the symplectic cylinder $\Phinv(P) \times \t_P^\dual$ and the almost complex cylinder $\Phinv(P) \times \sqrt{-1}\t_P$
  \item The definition of local (strong) tamedness depends not just on the symplectic form $\om_X$,
  but also on the symplectic cylindrical structure in Definition \ref{def:sympcylstr}.
\item In the space of cylindrical almost complex structures, the
  cylindrical coordinate maps, taking values in
  $\Phinv(P^\fillblack) \times \nu P^\dual$, are held fixed, but the
  connection one-forms $\alpha_{P,J}$, $P \in \PP$ underlying a
  cylindrical almost complex structure $J$ are allowed to vary.
\item \label{part:tame-equiv} We give an alternate definition of local
  tamedness.  A cylindrical almost complex structure $J^\nu$ is \em{
    locally tamed} on $X^\nu$ if for each $P$, the base almost complex
  structure (see Definition \ref{def:cylp}) on $X_P$ induced by
  $J^\nu$ on the $P$-cylinder $\Phinv(\wP) \times \nu
  P^\dual$ \label{rep:nopush}
is tamed by
$\om_{X_P} + \bran{\d\alpha_P, c}$ for all 
  $c \in P^\dual - c_{P^\dual} \subset \t_P$.  Here
  $c_{P^\dual}:=\pi_{P^\dual}(P) \in \t_P$ is a constant and
  $\alpha_P$ is the $T_P$-connection one-form underlying the
  symplectic cylindrical structure in Definition \ref{def:sympcylstr}.
  Note that the definition of local tamedness does not involve $\nu$, since
  the ``same'' almost complex structure can be defined for any $\nu$.
    \end{enumerate}
\end{remark}
The following Lemma is a quantitative version of the statement that
local tamedness is a $C^0$-open property in the space of cylindrical
almost complex structures.  We use the alternate definition of local
tamedness from Remark \ref{rem:oncyl} \eqref{part:tame-equiv}.

\begin{lemma}\label{lem:tameC0}
  Let $J^0_P$ be an almost complex structure on $\Phinv(\wP)/T_P$ that
  is tamed by the form $\om_{X_P} + \bran{\d\alpha_P,\tau}$ for all
  $\tau \in P^\dual \subset \t_P$. Then there exist constants
  $\eps$, $c$ such that if $J_P \in \J(\Phinv(\wP)/T_P)$ is such that
  $\Mod{J_P - J_P^0}_{C^0} <\eps$ then
  \begin{multline}
    \label{eq:horiz-forms}  
    c^{-1} \om_{X_P}(v,J_Pv) \leq (\om_{X_P} + \bran{\d\alpha_P,
      c})(v,J_Pv) \leq c \om_{X_P}(v,J_Pv) \\ \forall v \in
    T(\Phinv(\wP)/T_P).
  \end{multline}
\end{lemma}

\begin{proof}
  By the compactness of the spaces $\wP/T_P$ and $P^\dual$, and the
  tamedness of $J_P^0$ for $\om_{X_P} + \bran{\d\alpha_P,\tau}$ for all
  $\tau \in P^\dual$, there are constants $c_0,c_1>0$ such that
  \[c_0|v|^2 \leq (\om_{X_P}+\bran{\d\alpha_P,\tau})(v,J_P^0v) \leq c_1|v|^2\]
  for all $v \in T(\Phinv(\wP)/T_P$, $\tau \in P^\dual$.
  For any $\eps>0$ and $J_P \in B_\eps(J^0_P)$, there are constants $c_0',c_1'>0$ such that
  \[(c_0-c_0'\eps)|v|^2 \leq (\om_{X_P}+\bran{\d\alpha_P,\tau})(v,J_P^0v) \leq (c_1+c_1'\eps)|v|^2.\]
  The Lemma follows by 
  choosing $\eps>0$ such that $c_0-c_0'\eps>0$ and $c:=\frac{c_1 + c_1'\eps}{c_0 - c_0'\eps}$.
\end{proof}
\section{Broken manifold as a degenerate limit}
\label{sec:cylbrokenmfd}

In this Section, we define cut spaces and broken manifolds as almost
complex manifolds with cylindrical ends.  For any $P$, the cut space
$X_P$ is defined as the direct limit of the $P$-cylindrical subsets of
the neck-stretched manifolds quotiented by the torus $T_{P,\C}$. The
broken manifold $\XX_P$ is a $T_{P,\C}$-fibration over $X_P$.  For any
$P$, the almost complex cut space $X_P$ resp. broken manifold $\XX_P$
is diffeomorphic to the corresponding symplectic object $X_P^\om$
resp. $\XX_P^\om$, but there is no canonical diffeomorphism.  See Lemma
\ref{lem:immdiffeo} and Remark \ref{rem:nosymp} for related
discussion.

\subsection{Defining almost complex broken manifolds}

The broken manifold is the degenerate limit of neck-stretched almost
complex manifolds as we now explain.  For a polytope $P$ and
$\nu \geq 1$, let
\begin{equation}
  \label{eq:xolpnu}
X^\nu_{\tP} :=  \left(\bigsqcup_{Q \in \PP : Q \subseteq P} (\Phinv(Q^\fillblack) \times \nu Q^\dual)/\sim\right) \subset X^\nu  
\end{equation}
be the subset of $X^\nu$ that has a $P$-cylindrical structure.  Thus
$X^\nu_{\tP}$ has a free partial $T_{P,\C}$-action (see
\eqref{eq:pact}).  (Here the equivalence relation $\sim$ is the same
as the one in the definition \ref{eq:xnudef} of $X^\nu$.)  Let
\begin{equation}
  \label{eq:xpnu}
  X^\nu_P:=X^\nu_{\tP}/T_{P,\C}  
\end{equation}
be the quotient of $X^\nu_{\tP}$ under the partial
$T_{P,\C}$-action. Consequently, there are projections
\begin{equation}
  \label{eq:zpnu}
  X^\nu_{\tP} \xrightarrow{\pi_P'} Z^\nu_P  \xrightarrow{\pi_P''} X^\nu_P,  
\end{equation}
where $Z^\nu_P  \xrightarrow{\pi_P''} X^\nu_P$ is a $T_P$-bundle, and the fibers of $\pi_P'$ are  $P^\dual \subset \t_P^\dual$. 
\begin{lemma}\label{lem:natinc}
  Let $(X^\nu,J^\nu)$ be a family of neck-stretched almost complex structures (as in Definition \ref{def:cylneckst}).  For any $P \in \PP$ and $\nu_0 < \nu_1$ there is a natural embedding 
\[i_{P, \nu_0,\nu_1}:(X^{\nu_0}_P,J^{\nu_0}) \to (X^{\nu_1}_P,J^{\nu_1}), \quad i_{Z_P,\nu_0,\nu_1}:Z^{\nu_0}_P \to Z^{\nu_1}_P,\]
and for any $\nu_0<\nu_1<\nu_2$,
\[i_{P, \nu_0,\nu_2}:=i_{P, \nu_1,\nu_2} \circ i_{P, \nu_0,\nu_1}, \quad i_{Z_P, \nu_0,\nu_2}:=i_{Z_P, \nu_1,\nu_2} \circ i_{Z_P, \nu_0,\nu_1}\]
\end{lemma}

\begin{proof}
  Let $P \in \PP$ be a top-dimensional polytope. Then,
  $X^\nu_P=X^\nu_{\tP}$ has a decomposition as in \eqref{eq:xolpnu}.
  Define embeddings
  \begin{equation}
    \label{eq:ipnuembed}
    i_{P, \nu_0,\nu_1}:=(\Id, \tau):\Phinv(Q^\fillblack) \times \nu_0 Q^\dual \to
    \Phinv(Q^\fillblack) \times \nu_1 Q^\dual  
  \end{equation}
  where $\tau$ is the restriction of a translation on $\t_Q^\dual$
  (recall that $\nu Q^\dual \subset \t_Q$ for all $\nu$) that
  maps the point $\nu_0 P^\dual \in \nu _0Q^\dual$ to the point
  $\nu_1 P^\dual \in \nu _1Q^\dual$. The maps \eqref{eq:ipnuembed}
  glue to yield an embedding
  $i_{P, \nu_0,\nu_1}:X^{\nu_0}_P \to X^{\nu_1}_P$.

  Next, consider a face $Q \in \PP$ of the top-dimensional polytope
  $P$. The embedding
  $i_{P, \nu_0,\nu_1} : X^{\nu_0}_{\tQ} \to X^{\nu_1}_{\tQ}$ descends
  to $i_{Z_Q,\nu_0,\nu_1} : Z^{\nu_0}_Q \subset Z^{\nu_1}_Q$ and
  $i_{Q, \nu_0,\nu_1} : X^{\nu_0}_Q \subset X^{\nu_1}_Q$.  We leave it
  to the reader to check that the embedding preserves the almost
  complex structure.
\end{proof}

Next, we define cut spaces and the broken manifold as manifolds with
cylindrical structures.  Cut spaces resp. broken manifolds were
already defined as symplectic manifolds in
Definition 
\ref{def:cutspace}
resp. \ref{def:brokenxxsymp}.  We now define these spaces as almost
complex manifolds.  See Lemma \ref{lem:immdiffeo} and Remark
\ref{rem:nosymp} for a reconciliation of the two viewpoints. The
symplectic spaces have a superscript $\om$ to distinguish them from
the corresponding almost complex spaces,
see Notation \ref{note:symp-vs-ac}.

\begin{definition}{\rm(Almost complex cut spaces and broken manifold)}
  \label{def:brokenJ}
  Let $(X,\PP,\Phi)$ be a tropical Hamiltonian action and let $\{X^\nu\}_\nu$ be
  the corresponding family of neck-stretched
  manifolds.
  \begin{enumerate}
  \item {\rm(Cut space)} \index{Cut space!Almost complex cut space $X_P$}
    \label{part:cutspace}
    For a polytope $P \in \PP$, the \em{cut space} $\XB_P$ is defined
    as the direct limit
    \label{rep:cutdef}
    \footnote{We do not just define the cut space as $\Phinv(P)/T_P$
      as in done in the case of a symplectic cut space. The definition
      by direct limits allows us to realize the cut space as the
      direct limit of almost complex manifolds $X^\nu_P$.}
    \begin{equation}
      \label{eq:xpdef}
      X_P:=(\cup_\nu X_P^\nu)/\sim,  
    \end{equation}
    where $X_P^\nu$ is the $T_P$-quotient of the $P$-cylindrical
    subset of $X^\nu$ (see \eqref{eq:xpnu}), and $\sim$ is given by
    the embeddings $i_{P, \nu_0,\nu_1}:X^{\nu_0}_P \to X^{\nu_1}_P$
    from Lemma \ref{lem:natinc}.  For any $\nu$, there is a natural
    embedding
    \[i_{P,\nu} : X_P^\nu \to X_P.\]
  \item {\rm($T_P$-bundle on the cut space $X_P$)} Define the
    $T_P$-bundle $\oZ_P \to X_P$ as the direct limit
    \begin{equation}
      \label{eq:zpdef}
      \oZ_P:=(\cup_\nu Z_P^\nu)/\sim  
    \end{equation}
    \index{Torus bundle! on almost complex cut space $Z_P$}
    of $T_P$-bundles $Z_P^\nu$ (defined in \eqref{eq:zpnu}), where
    $\sim$ is given by the embeddings
    $i_{Z_P,\nu_0,\nu_1}:Z^{\nu_0}_P \to Z^{\nu_1}_P$ from Lemma
    \ref{lem:natinc}.  There is a natural embedding
    \[i_{Z_P,\nu} : Z_P^\nu \to \oZ_P.\]
  \item{\rm(Broken manifold)} \index{Broken manifold! Almost complex
      broken manifold $\XC_P$} The broken manifold $\XX$ is the
    disjoint union
    \begin{equation}
      \label{eq:brokenmfd}
      \XX_\PP\quad \text{or} \quad \XX:=
      \bigsqcup_{P \in \PP} \XC_P ,
    \end{equation}
    where we use the notation $\XX$ when the polyhedral decomposition
    $\PP$ is clear from the context.
    Here,
    \[\XX_P:=\oZ_P \times \t_P \]
    and is a $T_{P,\C}$-bundle over $\XB_P$ with projection
    \begin{equation}
      \label{eq:pixp}
      T_{P,\C} \to \XX_P \xrightarrow{\pi_{X_P}} X_P.  
    \end{equation}
  \item {\rm(Cylindrical almost complex
      structures)}\index{Cylindrical! almost complex structure}
    \index{Almost complex structure! Cylindrical}
    \label{part:brokenJ-cyl}
    An almost complex
    structure $\JJ=(\JJ_P)_{P \in \PP}$ on $\XX$ consists of an
    almost complex structure $\JJ_P$ on each component
    $\XC_P \subset \XX$.  Such a $\JJ$ is \em{cylindrical}
    (or \em{$X$-cylindrical} to emphasize the dependence on the $X$-inner product, 
    \index{Almost complex structure! X@ $X$-cylindrical}
    see Remark \ref{rem:oncyl} \eqref{part:oncyl1})
    if it is
    the limit of a family $(J^\nu)_\nu$ of neck-stretched cylindrical
    almost complex structures on $X^\nu$. The space of cylindrical
    almost complex structures on $\XX$ is denoted by
    \[\J^\cyl(\XX) = \{ \JJ \ \text{cylindrical} \}.\]

    \index{J@$\J^\cyl$!$\J^\cyl(\XX)$} The almost complex structure
    $\JJ$ is \em{locally tame} or \em{locally compatible} or \em{
      locally strongly tamed} \index{Almost complex structure! Locally
      strongly tamed} \index{Almost complex structure! Locally tamed}
    if the corresponding property holds for $J^\nu$ (see Definition
    \ref{def:cylneckst}).
  \item {\rm(Cylindrical ends)} 
  \index{Cylindrical! P@ $P$-cylindrical end} For a pair of polytopes $Q \subset P$, the \em{
      $Q$-cylindrical end of $\XB_P$} is a subset
    $U_Q(\XB_P) \subset \XB_P$ defined as the exhaustion
    \begin{equation} \label{eq:uq} U_Q(\XB_P):=\cup_\nu
      U_Q(X_P^\nu) \end{equation}
    of $Q$-cylindrical subsets $U_Q(X_P^\nu) \subset X_P^\nu$ given by
    \[U_Q(X_P^\nu):= \cup_{R \subseteq Q} \Phinv(R^\fillblack) \times
      \nu R^\dual \subset X_P^\nu.\]
    The $Q$-cylindrical end in $\XX_P$ resp. $\oZ_{P}$ is the lift of
    $U_Q(X_P)$ by the projection map \eqref{eq:zpnu}
    $\pi''_P : \oZ_P \to \XB_P$ resp. $\pi_{X_P} : \XC_P \to \XB_P$,
    and is denoted by
    \[U_Q(\XC_P) \subset \XC_P \quad \text{resp.} \quad U_Q(\oZ_P)
      \subset \oZ_P.\]
  \end{enumerate}
\end{definition}

The following is immediate from the definitions:

\begin{lemma}\label{lem:immdiffeo}
  Given a broken almost complex manifold $\XX$, any piece $\XX_P$ of
  the broken manifold is diffeomorphic to the symplectic manifold $\XX^\om_P$.
\end{lemma}
The cut space $X_P$ has an orbifold compactification 
%
\index{Cut space!Almost complex cut space $\ol X_P$}
\begin{equation}
  \label{eq:xpcpt}
  \ol X_P
\end{equation}
that is diffeomorphic to the compactification $\ol X_P^\om$ of the
symplectic cut space.  The boundary $\ol X_P \bs X_P$ is a
collection of sub-orbifolds of codimension at least $2$, each
corresponding to either a relative divisor, or the intersection of a
collection of relative divisors in $\ol X_P^\om$ from Definition
\ref{def:brokenxxsymp} \eqref{part:bdrydiv}.  The compactification may
be constructed formally (without reference to $X_P^\om$) by adding to
$X_P$ limit points of torus orbits; for example, a divisor
$X_Q \subset \ol X_P$ corresponding to a facet $Q \subset P$,
$Q \in \PP$ with outward normal $\nu_Q \in \t_Q/\t_P$ is the set of
orbits
\[\{[x]=T_{\nu_Q,\C}x \enspace: \enspace x \in X_P\}, \]
where
$T_{\nu_Q,\C} \simeq \C^\times$ is the torus generated by $\nu_Q$, and 
each $[x]$ is the limit point of a $T_{\nu_Q,\C}$-orbit, that
is, $\lim_{s \to \infty} e^{(s+it)\nu_Q}x=[x]$.
 The cylindrical almost complex structure on $X_P$ does not extend to $\ol X_P$.
Similarly, a component $\XX_P$ of the broken manifold has a compactification $\ol \XX_P$ if
the $X$-inner product on $\{\t_P\}_P$ 
(from \eqref{eq:idtt}) is rational.

\begin{remark}\label{rem:recover}
  The neck-stretched manifolds $X^\nu$ can be recovered from the
  broken almost complex manifold equipped with cylindrical coordinates
  on its ends as
  \[X^\nu= \left(\bigcup_{P \in \PP, \codim(P)=0} X_P \right)/\sim_\nu\]
  where the equivalence relation $\sim_\nu$ is as follows.  For
  top-dimensional polytopes $P_0$, $P_1$, and $Q:=P_0 \cap P_1$,
  \begin{multline}
    \XB_{P_0} \ni x_0 \sim_\nu x_1 \in \XB_{P_1}
    \Longleftrightarrow \\ x_0 \in U_Q(X_{P_0}),\enspace x_1 \in U_Q(X_{P_1}),
    \enspace i_Q^{P_0}(x_0) = e^{\nu(P_0^\dual - P_1^\dual)}
    i_Q^{P_1}(x_1).
  \end{multline}
  Here, $P_0^\dual - P_1^\dual \in \t_Q$, and the maps $i_Q^{P_0}$, $i_Q^{P_1}$ are $Q$-cylindrical coordinate maps defined in Lemma \ref{lem:cyl-br}. 
   This ends the
  Remark.
\end{remark}
\begin{example}{\rm(Cylindrical ends in a single
    cut)}\label{ex:1cutcyl}
  We continue Example \ref{ex:1cutneck} where we described
  neck-stretched manifolds in the case of a single cut.  Let $(X,\Phi,c)$
  be a tropical Hamiltonian action with a single cut.  First, we describe the
  $P$-cylindrical part in the neck-stretched manifold $X^\nu$.  We
  have,
  \[X^\nu_{P_+}=\{\Phi \geq c+\eps\} \cup_{\Phinv(c)} (\Phinv(c)
    \times [-\eps \nu, \eps \nu]),\]
  with $X^\nu_{P_-}$ defined analogously, and
  \[X^\nu_{P_0}=(\Phinv(c) \times [-\eps \nu, \eps \nu]).\]
  Next, we describe the $P_0$-cylindrical end in $X_{P_+}$.  Since
  $\codim(P_+)=0$, the spaces $\XB_{P_+}$, $\oZ_{P_+}$ and $\XC_{P_+}$
  are all the same. Furthermore, 
  \begin{equation}
    \label{eq:1cutcylend}
    \XB_{P_+}=\{\Phi \geq c+\eps\} \cup_{\Phinv(c)} \left( (-\infty,0] \times \Phinv(c)\right),  
  \end{equation}
  where $\Phinv(c+\eps)$ is mapped to $\Phinv(c)$ via the symplectic
  cylindrical structure map, see \eqref{eq:1cutsymcyl}. The latter may be 
  identified with $\{0\} \times \Phinv(c)$.
  Note that $\oZ_{P_0} = \Phinv(c)$, the normal cone
  $\NCone_{P_+^\dual}P_0^\dual$ is $(-\infty,0]$.  The subset
  $(-\infty,0] \times Z$ in $\XB_{P_+}$ is the $P_0$-cylindrical end.
\end{example}
\begin{remark}\label{rem:nosymp}
  {\rm(Symplectic versus almost complex broken manifolds)} We have
  defined the broken manifold as a collection of almost complex
  manifolds $(\XX_P,J_P)$ in \eqref{eq:brokenmfd}, and as symplectic
  manifolds $(\XX_P^\om,\om_{\XX_P})$ in \eqref{eq:XolP}.  In this
  Remark, we see how to relate the two spaces, and also why they can
  not be viewed as the same space.  Since a component $\XX_P$ of a
  broken manifold is a $T_{P,\C}$-bundle over a cut space $X_P$, we
  focus on cut spaces in this discussion.
  \begin{enumerate}
  \item \label{part:nosymp1} {\rm(Maps respecting cylindrical
      structures)} A surjective map $\phi : X_P \to X_P^\om$
    \em{respects the cylindrical structure} if
    \begin{itemize}
    \item $\phi$ is the natural embedding on the complement of
      cylindrical ends, that is,
      \begin{equation}
        \label{eq:ipembed}
        \phi(X_P \bs (\cup_{Q \subset P} U_Q(X_P))) = \Phinv(\wP)/T_P  \subset (\ol
        X_P^\om,\om_{X_P}), 
      \end{equation}
      arising from the definition \eqref{eq:xolpnu}, \eqref{eq:xpnu},
      \eqref{eq:xpdef} of $X_P$, and the definition \eqref{eq:cutsp}
      of the symplectic cut space $X_P^\om$;
    \item $\phi$ is $T_Q$-equivariant on the $Q$-cylindrical end; and
    \item $\phi$ commutes with the projection to $X_Q$ on the
      $Q$-cylindrical end, that is,
      \[\pi_Q^P = \pi_Q^{P,\om} \circ \phi \quad \text{on }U_Q(X_P)
        \bs \cup_{R \subset Q}U_R(X_P),\]
      where the projection $\pi_Q^P$ to $X_Q$ on the $Q$-cylindrical
      end is defined in \eqref{eq:pipq} and the corresponding
      projection $\pi_Q^{P,\om}$ to $X_Q^\om$ for symplectic manifolds
      is defined in \eqref{eq:pip-om}.  In the above equation, both
      sides map to the complement of the cylindrical ends in $X_Q$:
      The left hand side maps to
      $X_Q \bs (\cup_{R \subset Q} U_R(X_Q))$, and the right hand side
      maps to $\Phinv(\wQ)/T_Q$, and the two spaces are canonically
      identified as in \eqref{eq:ipembed}.
    \end{itemize}
    For a broken manifold $\XX$, there is a large collection of maps
    that respect the cylindrical structure. In Chapter \ref{chap:hof},
    we construct squashing maps that are continuous and piecewise
    smooth. Those maps can be approximated by smooth maps that satisfy
    the above properties.  Furthermore, any two maps respecting the
    cylindrical structure are homotopic via a family of maps
    respecting the cylindrical structure.  These maps are used to
    define the symplectic area of pseudoholomorphic maps to the broken
    manifold (Definition \ref{def:br-area}), and to define Donaldson
    divisors on broken manifolds (Definition \ref{def:brokediv}).
  \item {\rm(No canonical diffeomorphisms)} The manifolds $X_P$ and
    $X_P^\om$ are diffeomorphic, but there is no canonical
    diffeomorphism $\phi : X_P \to X_P^\om$ that respects the
    cylindrical structure. Assuming $X_P$ has a cylindrical locally
    tamed almost complex structure $J_P$, we do not have a method of
    constructing diffeomorphisms $\phi: X_P \to X_P^\om$ that respect
    the cylindrical structure and for which $J_P$ is tamed by
    $\phi^*\om_{X_P}$.
  \item {\rm(Canonical diffeomorphisms without gluability)} The lack
    of taming embeddings of broken almost complex manifolds into
    compact symplectic broken manifolds can be remedied by introducing
    a slight weakening in the definition of cylindrical almost complex
    structures on broken manifolds.  In particular, if one allows the
    connection one-forms $(\alpha_P)_{P \in \PP}$ and the inner
    products $\{g_P\}_{P \in \PP}$ to differ across the set of cut
    spaces $\{X_P\}_{P \in \PP^{(0)}}$, one can construct taming
    diffeomorphisms of $X_P$ into $X_P^\om$, see Lemma
    \ref{lem:omxx-cptble}.  The cost of this choice is that the new
    kind of almost complex structures on pieces of $\XX$ do not glue
    to give a neck-stretched almost complex structure in $X$.  That
    is, the almost complex structures are not
    \em{gluable}. \index{Gluable} Gluability of almost complex
    structures on broken manifolds is necessary for the proof of
    homotopy equivalence of the Fukaya algebras defined on the
    unbroken manifold $(X,L)$ and the broken manifold $(\XX,L)$. Once
    this result is proved, one has greater flexibility in choosing
    almost complex structures on broken manifolds while keeping the
    compactness and regularity results, as we point out in Section
    \ref{sec:brokenind}.  In particular, one can work with
    $\om_\XX$-tamed almost complex structures on $\XX$ (Definition \ref{def:omxxcyl})
    -- these are  tamed
    almost complex structures defined on compact symplectic manifolds $\ol X_P$, 
    and the distinct notations $X_P$, $X_P^\om$ for almost complex and
    symplectic manifolds can be dropped.  This finishes the Remark.
  \end{enumerate}
\end{remark}

\subsection{Coordinates on cylindrical ends}
Our next task is to produce identifications between $Q$-cylindrical
ends of different components of a broken manifold. These
identifications are used in writing down the matching conditions at
nodes of a broken map: The matching condition compares the evaluation
of maps lying in different manifolds (say $\XX_P$, $\XX_{P'}$),
although neighborhoods of both lifts of the node map to the
$Q$-cylindrical end.  The identifications between the cylindrical ends
are via certain natural cylindrical ``coordinates'' that take values
in cones of torus Lie algebras such as $\t_P^\dual$.

\begin{remark}\label{rem:cone-exhaust}
  We give some relations between dual polytopes, cones and normal
  cones (Definition \ref{def:normcones}).  Later in Remark
  \ref{rem:cone-exhaust1}, these relations will be shown to be the
  polytope analogues of the fact that cut spaces are degenerate limits
  of neck-stretched manifolds.  For a pair $Q \subset P$ of polytopes
  with $\codim(P)=0$, there are inclusions
  \[\tilde i_{P,Q,\nu_0,\nu_1}:\nu_0 Q^\dual \to \nu_1 Q^\dual, \quad \forall \nu_0 <
    \nu_1 \]
  that are translations and which map the vertex
  $\nu_0P^\dual \in \nu_0 Q^\dual$ to
  $\nu_1 P^\dual \in \nu_1 Q^\dual$. Then $\Cone_{P^\dual}Q^\dual$ is
  the exhaustion
  \begin{equation}
    \label{eq:conelim}
    \Cone_{P^\dual}Q^\dual = \sqcup_\nu \nu Q^\dual/\sim,  
  \end{equation}
  where, for any $\nu_0<\nu_1$, $\sim$ identifies $\nu_0 Q^\dual$ to
  its image under $\tilde i_{P,Q,\nu_0,\nu_1}$.  In general for any
  pair $Q \subset P$ (with $\codim(P)>0$ possibly) there are
  inclusions
  \[i_{P,Q,\nu_0,\nu_1}:\nu_0 Q^\dual/\t_P \to \nu_1 Q^\dual/\t_P,
    \quad \forall \nu_0 < \nu_1\]
  which take the point $\nu_0 P^\dual/\t_P$ in the domain to the
  point $\nu_1 P^\dual/\t_P$ in the target space.  The resulting
  exhaustion is the normal cone
  \begin{equation}
    \label{eq:nconelim}
    \NCone_{P^\dual}Q^\dual = (\sqcup_\nu \nu Q^\dual/\t_P)/\sim,  
  \end{equation}
  where, for any $\nu_0<\nu_1$, $\sim$ identifies
  $\nu_0 Q^\dual/\t_P$ to its image under $i_{P,Q,\nu_0,\nu_1}$.
\end{remark}
\begin{figure}[ht]
  \centering \scalebox{.8}{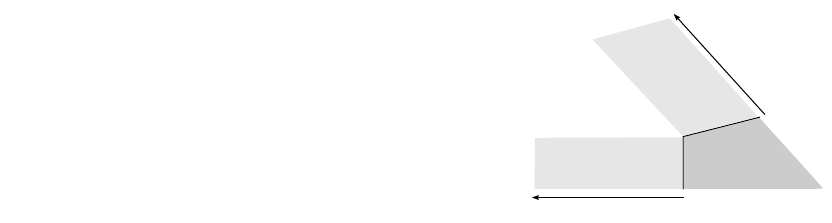}
  \caption{Some components of the broken manifold in the limit $\nu \to \infty$ of the neck-stretching in Figure \ref{fig:stretch}.}
  \label{fig:stretchlimit}
\end{figure}
The following lemma gives an identification between the
$Q$-cylindrical end in a component $\XC_P$ of the broken manifold
$\XX$ to the ``standard'' $Q$-cylinder $\oZ_Q \times \t_Q$
(which is equal to  $\XC_Q$ by definition), see Figure
\ref{fig:stretchlimit}.
\begin{lemma}{\rm(Cylindrical ends on broken manifolds)} 
\index{Cylindrical! P@ $P$-cylindrical end}
  \label{lem:cyl-br}
  For any pair of polytopes $Q \subset P$ in $\PP$, there are natural embeddings
  \begin{equation}
    \label{eq:Pcoord}
    \begin{split}
      i_Q^P: U_Q(\XB_P) &\to (\oZ_Q/T_P) \times \NCone_{P^\dual}Q^\dual,\\
      i_Q^{Z,P}: U_Q(\oZ_P) &\to \oZ_Q \times \NCone_{P^\dual}Q^\dual,\\
      i_Q^{\tP}: U_Q(\XC_P) &\to \oZ_Q \times (\NCone_{P^\dual}Q^\dual \times \t_P) \subset \oZ_Q \times \t_Q \simeq \XC_Q. 
    \end{split}
  \end{equation}
\end{lemma}

\begin{proof}
  We prove the Lemma for the case $\codim(P)=0$. In this case, the
  torus $T_P$ is trivial, $\XB_P=\oZ_P=\XC_P$, and
  $\NCone_{P^\dual}Q^\dual=\Cone_{P^\dual}Q^\dual$.  The domain of
  $i_Q^P$ has a decomposition (into sets whose interiors are disjoint)
  \begin{equation}
    \label{eq:iqpdomain}
  U_Q(X_P)=\cup_{R \subseteq Q} (\Phinv(R^\fillblack) \times
    \Cone_{P^\dual}R^\dual)  .
  \end{equation}
  Here, we use the fact that $\Cone_{P^\dual}R^\dual$ is the limit of
  the polytopes $\nu R^\dual$, see \eqref{eq:conelim}.
  The target space of $i_Q^P$ has a decomposition (into sets whose
  interiors are disjoint)
  \begin{equation}
    \label{eq:iqptarget}
  \oZ_Q=\Phinv(Q^\fillblack) \cup (\cup_{R \subset Q}
    \Phinv(R^\fillblack) \times \NCone_{Q^\dual}R^\dual)  
  \end{equation}
  which is a limit of the decompositions
  \[Z_Q^\nu=\Phinv(Q^\fillblack) \cup (\cup_{R \subset Q}
    \Phinv(R^\fillblack) \times \nu R^\dual/\t_Q), \]
  since $\NCone_{Q^\dual}R^\dual$ is the limit of
  $\nu R^\dual/\t_Q$ as $\nu \to \infty$ (see
  \eqref{eq:nconelim}).  The embedding $i_Q^P$ maps a subset in the
  domain decomposition \eqref{eq:iqpdomain} to the corresponding
  subset in the target decomposition \eqref{eq:iqptarget}.  Here
  $Z_Q^\nu$ is a $T_Q$-bundle defined in \eqref{eq:zpnu}.  For any
  $R \subseteq Q$, the map
  \[i_Q^P:\Phinv(R^\fillblack) \times \Cone_{P^\dual}R^\dual \to
    (\Phinv(R^\fillblack) \times \NCone_{Q^\dual}R^\dual) \times
    \Cone_{P^\dual}Q^\dual\]
  is defined in an following way:  It is the
  identity on $\Phinv(R^\fillblack)$ and
  \[ \Cone_{P^\dual}R^\dual \to \NCone_{Q^\dual}R^\dual \times
  \Cone_{P^\dual}Q^\dual \] 
  is an orthogonal splitting.  We leave to the
  reader the construction of the maps $i_Q^P$, $i_Q^{Z,P}$ when
  $\codim(P)>0$.  Finally, we point out that the domain resp. target
  space of the map $i_Q^{\tP}$ is the product of $\t_P$ and the
  domain resp. target space of $i_Q^{Z,P}$. Therefore the map
  $i_Q^{\tP}$ is defined as such a product once $i_Q^{Z,P}$ is known.
\end{proof}

\begin{remark}\label{rem:cone-exhaust1}
  We can now relate the direct limit definition of cut spaces to the
  corresponding relations between cones from Remark
  \ref{rem:cone-exhaust}.  For any pair $Q \subset P$, Lemma
  \ref{lem:cyl-br} gives a coordinate map
  $\pi_{\NCone_{P^\dual}Q^\dual}:U_Q(X_P) \to \NCone_{P^\dual}Q^\dual$
  on the $Q$-cylindrical end $U_Q(X_P)$ of $X_P$. The maps
  $\pi_{\NCone_{P^\dual}Q^\dual}$, for all $Q \subset P$, assemble to
  give a map
  \begin{equation}
    \label{eq:pincone}
    \pi_{\NCone_{P^\dual}B^\dual} : X_P \to \NCone_{P^\dual}B^\dual,  
  \end{equation}
  where the complement $X_P \bs (\cup_{Q \subset P} U_Q(X_P))$ is
  mapped to the point $\NCone_{P^\dual} P^\dual$.  The continuity of
  the map is a consequence of the construction in Lemma
  \ref{lem:cyl-br}.  There is a commutative diagram
  \begin{equation}
    \label{eq:cdet1}
    \begin{tikzcd}
      X_P^\nu \arrow{r}{i_{P,\nu}} \arrow[swap]{d}{\pi_{\nu P^\dual}} & X_P \arrow{d}{\pi_{\NCone _{P^\dual} B^\dual}} \\
      \nu B^\dual/\t_P \arrow{r}{} & \NCone_{P^\dual}B^\dual
    \end{tikzcd}
  \end{equation}
  where the left downward arrow is the map \eqref{eq:pinb-ch3}
  composed with the $\t_P$-quotient map, the top arrow is from
  Definition \ref{def:brokenJ} \eqref{part:cutspace}, the bottom arrow
  is from \eqref{eq:nconelim}, and the right downward arrow is from
  \eqref{eq:pincone}.
\end{remark}

\begin{remark}
  {\rm(Projections on cylindrical ends)} The coordinates on the
  cylindrical ends of broken manifolds naturally yield projection
  maps.  For any pair of polytopes $Q \subset P$, the projection on the
  $Q$-cylindrical end of $X_P$
  \begin{equation}
    \label{eq:pipq}
    \pi_Q^P : U_Q(X_P) \to X_Q
  \end{equation}
  is defined by the first component of $i_Q^P$ from \eqref{eq:Pcoord} composed with the $T_Q$-quotient.
\end{remark}

Neck-stretched manifolds and components of a broken manifold are
equipped with a cylindrical metric.  We recall that the $X$-inner
product \eqref{eq:idtt} defines a metric $|\cdot|_{\t_P}$
resp. $|\cdot|_{\t_P^\dual}$ on $\t_P^\dual$ resp.  $\t_P^\dual$ for
all $P \in \PP$.
\begin{remark}
  {\rm(Compactifications of cut spaces)}
  \label{rem:compactify}
  Since the cylindrical ends of
  a cut space have a standard form as in \eqref{eq:Pcoord}, any cut space $X_P$ with
  a cylindrical almost complex structure $J_P$ has a compactification
  \[\ol X_P\]
  that is an almost complex orbifold, with $\ol X_P \bs X_P$ being a
  union of divisors, each corresponding to a facet $Q \subset P$,
  $Q \in \PP$.  In a similar way, if the $\t$-inner product from
  \eqref{eq:idtt} is rational, any component $\XX_P$ of the broken
  manifold $\XX$ has a compactification
  \[\ol \XX_P\]
  that is an almost complex orbifold, and $\ol \XX_P \bs \XX_P$
  consists of divisors corresponding to the facets of $\tP$, see
  \eqref{eq:facetsolp}. 
  The same construction defines a $T_P$ orbifold bundle
  \begin{equation}
    \label{eq:olzp}
    \ol Z_P \to \ol X_P  
  \end{equation}
  that is an extension of the $T_P$-bundle $Z_P \to X_P$, and
  is such that $\ol Z_P \times \t_P^\dual$ is the complement of
  the vertical divisors (Definition \ref{def:brokenxxsymp} \eqref{part:bdrydiv}) of $\ol \XX_P$.
 Note that the
  spaces $Z_P$ and $\ol \XX_P$ depend on the $\t$-inner product.
\end{remark}
%
Broken manifolds are equipped with a cylindrical metric.
\begin{definition}
  {\rm(Cylindrical metric)}
  \label{def:cylmet}
  A metric $g_P$ on $\XC_P \simeq \oZ_P \times \t_P^\dual$ is \em{
    $P$-cylindrical} if $g_P$ is a product metric, that is, the
  product of the linear metric $|\cdot|_{\t_P^\dual}$ on $\t_P^\dual$
  and a $T_P$-invariant metric $g_{Z_P}$ on $Z_P$ that satisfies
    \[\abs{\xi_{Z_P}}_{g_{Z_P}}= |\xi|_{\t_P} \quad \xi \in \t_P \]
    where $\xi_{Z_P}$ is the generating vector field as in
    \eqref{genvec}. \label{xiZP} On the multiply-stretched manifolds
    $X^\nu$, a metric $g_\nu$ is \em{cylindrical} if for any
    $P \in \PP$, $g_\nu$ is $P$-cylindrical in the region
    $\Phinv(P^\fillblack) \times \nu P^\dual$.
\end{definition}

\section{Translations: Relating neck-stretched and broken manifolds}
\label{sec:trans-def}
\index{Translation|(} To examine the convergence behavior of maps in
neck-stretched manifolds to a limit map in the broken manifold, we
need to embed $P$-cylindrical regions of the neck-stretched manifold
into the $P$-cylindrical component of the broken manifold. The
embedding maps, called \em{translations} are parametrized by elements
in the scaled dual complex $\nu B^\dual$.  Given an element
$t \in \nu B^\dual$ lying in the dual polytope $\nu P^{\dual,\circ}$,
the translation $\e^{-t}$ is a map from the $P$-cylindrical subset
$X^\nu_{\tP} \subset X^\nu$ to the component of the broken manifold
$\XX_P$.
 
We start by recalling some facts about neck-stretched and broken
manifolds:
  \begin{enumerate}
  \item For any polytope $P \in \PP$ and any $\nu$, the
    $P$-cylindrical subset of $X^\nu$ is $X^\nu_{\tP}$ which has a
    projection
    \[X^\nu_{\tP} \xrightarrow{\pi_{P,\nu}'} Z^\nu_P    \]
    where $Z^\nu_P$ is a $T_P$-bundle over $X^\nu_P$ and the fibers of
    $\pi_{P,\nu}'$ are subsets of $\nu P^\dual$, see
    \eqref{eq:zpnu}. In particular, there is a map
  \begin{equation}
    \label{eq:pipdual}
    \pi_{\nu P^\dual} : X^\nu_{\tP} \to \nu P^\dual,
  \end{equation}
  which is $\pi_{\nu B^\dual} : X^\nu \to \nu B^\dual$ composed with
  the orthogonal projection $Q^\dual \to P^\dual$ (from Definition
  \ref{def:dualcomplex} \eqref{part:dualpair}).
\item The $P$-cylindrical component $\XC_P \subset \XX$ is a product
    \[\XC_P = \oZ_P \times \t_P\]
    where $\oZ_P \to \XB_P$ is a $T_P$-bundle over $\XB_P$.
  \item There is a natural embedding $Z_P^\nu \to \oZ_P$,
    $X_\tP^\nu \to \XB_P$ from Definition \ref{def:brokenJ}.

  \end{enumerate}
 
  \begin{definition}
    {\rm($P$-translation)}
 For any $P \in \PP$, 
  $t \in \nu P^\dual$, denote by
\begin{equation}
  \label{eq:transinc}
  \e^{-t}_P: X^\nu_{\tP} \to \XC_P.
\end{equation}
the lift of the inclusion $Z_P^\nu \to \oZ_P$ that maps a level set
$\{\pi_{\nu P^\dual}=c\} \subset X^\nu_{\tP}$ to
$\oZ_P \times \{c-t\} \subset \XC_P$, where $\pi_{\nu P^\dual}$ is
from \eqref{eq:pipdual}, $c \in \nu P^\dual$, and so
$c-t \in \t_P$.
  \end{definition}

  The notion of $P$-translation may be expressed via the following
  commutative diagram for any $t \in \nu P^\dual$:
  \begin{equation}
    \label{eq:cdet}
    \begin{tikzcd}
         X_\tP^\nu \arrow{r}{\e^{-t}_P} \arrow[swap]{d}{\pi_{\nu B^\dual}} & \XC_P \arrow{d}{\pi_{\Cone _{P^\dual} B^\dual}} \\
B^\dual \supseteq \cup_{Q \subseteq P}\nu Q^\dual \arrow{r}{\e^{-t}_P} & \Cone_{P^\dual}B^\dual
\end{tikzcd}  
  \end{equation}
  Here, the left downward arrow $\pi_{\nu B^\dual}$ is from
  \eqref{eq:pinb-ch3}, and for any $Q \subseteq P$, the arrow
  $\e^{-t}_P : \nu Q^\dual \to \Cone_{P^\dual} Q^\dual$ is a
  translation by $t \in \t_P$.
  
  For any $t \in \nu P^\dual$, the translation $\e^{-t}_Q$ is the
  `same' for all $Q \subseteq P$. Indeed, 
  \[\e^{-t}_P|X^\nu_{\tQ}=\e^{-t}_Q\]
  since the restriction of $\e^{-t}_P$ to
  $X^\nu_{\tQ} \subset X^\nu_{\tP}$ maps to the $Q$-cylindrical end
  $U_Q(\XC_P) \subset \XC_P$; the latter may be viewed as a subset of
  $\XC_Q$.  This leads us to view the parameter $t$ in the translation
  $\e^{-t}$ as an element in the dual complex $B^\dual$ as in the
  following definition:

  \begin{definition}\label{def:transdef}
    {\rm(Generalized translation)}
    For any $t \in \nu B^\dual$,
    \[\e^{-t}:=\e^{-t}_P \quad \text{if $t \in P^{\dual,\circ}$}. \]
  \end{definition}
  
  For an element $t \in \nu P^{\dual,\circ}$, the inverse of the
  translation $\e^{-t}$ is well-defined on a subset
  $\XC_{P,\nu} \subset \XC_P$:
\begin{equation}
  \label{eq:transinc2}
  \e^t:=(\e^{-t})^{-1} : \XC_P \supset \XC_{P,\nu} \to X^\nu_{\tP}.
\end{equation}
The sequence of subsets $\XC_{P,\nu}$ exhaust $\XC_P$ as $\nu \to \infty$. 

\begin{example} 
  We illustrate translations in a multiple cut using an example with
  two non-intersecting single cuts.  Consider the tropical Hamiltonian action
  $(X,\Phi,\PP)$ where the torus is $T=S^1$, and the polytopes in
  $\PP$ are $P_0$, $P_1$, $P_2$, $P_{01}$, $P_{12} \subset \R$ shown
  in Figure \ref{fig:polytope1}.
\begin{figure}[h]
  {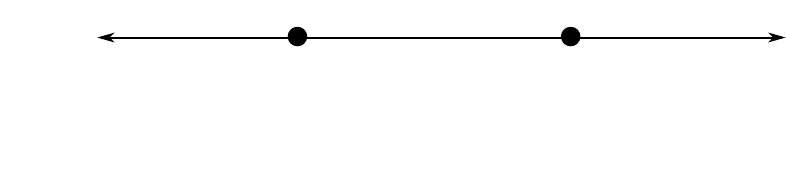}
\caption{A polyhedral decomposition $\PP$ and its dual complex $B^\dual$.}
\label{fig:polytope1}
\end{figure}
The dual complex $B^\dual$ is a subset of $\R$.  Let the point
$P^\dual_i$ be $g_i \in \R$ for $i=0,1,2$.  For $\nu>0$, the
neck-stretched manifold is
  \[X^\nu:=  (\wX_{P_0} \cup ([\nu g_0,\nu g_1] \times Z_0) \cup
     \wX_{P_1} \cup ([\nu g_1,\nu g_2] \times Z_1) \cup
     \wX_{P_2})/\sim,\]
  where $Z_i:=\Phinv(c_i)$, $\wX_{P_i}$ is $\ol X_{P_i}$ minus a
  tubular neighbourhood of relative divisors, and $\sim$ identifies
  the copies of $Z_0$ and $Z_1$ on the boundaries.  
  The translations are described as follows.
  \begin{itemize}
  \item Suppose $t \in \nu P_1^\dual$.  Then $t=\nu g_1$ and $\e^{-t}$
    is the embedding
    \begin{multline*}
      ([\nu g_0,\nu g_1] \times Z_0) \cup_{Z_0} \wX_{P_1} \cup_{Z_1}
      ([\nu g_1,\nu g_2] \times Z_1)\\ \to \XB_{P_1} \simeq
      ((-\infty,0] \times Z_0) \cup_{Z_0} \wX_{P_1} \cup_{Z_1}
      ([0,\infty) \times Z_1),
    \end{multline*}
    that is identity on $\wX_{P_1}$, and on the cylindrical ends it is
    a translation by $\nu g_1$, namely
    \[([\nu g_0,\nu g_1] \times Z_0) \xrightarrow{((\cdot - \nu g_1),\Id)} (-\infty,0] \times Z_0, \quad ([\nu g_1,\nu g_2] \times Z_1) \xrightarrow{((\cdot - \nu g_1),\Id)} ([0,\infty) \times Z_1).\]
    The translation $\e^{-t}$ is defined similarly when
    $t \in \nu P_0^\dual$ or $\nu P_2^\dual$.
  \item Suppose $t \in P^\dual_{i(i+1)}$. Then
    $t \in [\nu g_i, \nu g_{i+1}]$ and
\[\e^{-t}:[\nu g_i,\nu g_{i+1}] \times Z_i \to \R \times Z_i \]
     maps $\{c\}\times Z_i$ to $\{c-t\} \times Z_i$.
     \end{itemize}
  This ends the Example.
\end{example}

\index{Translation|)}

\section{Existence of symplectic cylindrical structures}\label{sec:sympcyl}
We prove that manifolds with tropical Hamiltonian actions possess
symplectic cylindrical structures, which were used in the definition
of neck-stretched almost complex manifolds.  We recall from Definition
\ref{def:sympcylstr} that in a tropical Hamiltonian manifold
$(X,\PP,\Phi)$, for a polytope $P \in \PP$, a \em{symplectic
  cylindrical structure} in a neighborhood of $\Phinv(P) \subset X$ is
a diffeomorphism
\[\phi_P : \Phinv(\tP) \to \Phinv(P) \times P^\dual, \]
where the polytope $\tP$ is a neighborhood of $P$, and
$(\phi_P^{-1})^*\om_X$ has a certain standard cylindrical form on
$\Phinv(P) \times P^\dual$.  The space $\Phinv(P) \times P^\dual$ is a
fibration over $\Phinv(P)/T_P$ whose fibers are subsets of $T_{P,\C}$;
via the identification by $\phi_P$, neck-stretched $T_P$-cylindrical
almost complex structures were defined on $\Phinv(\tP)$ in Section
\ref{sec:cylacs}.

However, for the proof of the main result of the book, namely the homotopy equivalence between unbroken and broken Fukaya algebras, we do not use the neck-stretched manifolds constructed via the symplectic cylindrical structures from Proposition \ref{prop:symcyl} below. 
For that purpose, neck-stretched manifolds are constructed by gluing a broken manifold equipped with a stabilizing divisor as in Section \ref{sec:stabpair}.

To construct the symplectic cylindrical structures, we may need to
modify the tropical moment map $\Phi$ in the following sense: Given a
tropical Hamiltonian manifold $(X,\PP,\Phi)$, a tropical moment map
$\Phi'$ is a \em{modification} of $\Phi$ if for every $P \in \PP$,
there is a sufficiently small $T_P$-invariant neighborhood $U_P \subset X$ of $\Phinv(P)$, 
and for which
\begin{equation}
  \label{eq:modify}
  \pi_{P^\dual} \circ \Phi=\pi_{P^\dual} \circ \Phi' : U_P \to \t_P^\dual.  
\end{equation}
In other words, the $\t_P^\dual$-component of $\Phi$ agrees with that
of $\Phi'$, wherever it is an honest $T_P$-moment map.
\begin{proposition}
  {\rm(Existence of symplectic cylindrical
    structures)}\label{prop:symcyl}
  There exists a symplectic cylindrical structure (see Definition
  \ref{def:sympcylstr})  for a tropical Hamiltonian manifold $(X,\PP,\Phi)$, after
replacing $\Phi$ by a modification of $\Phi$ (see previous paragraph).
\end{proposition}
\begin{proof}[Proof of Proposition \ref{prop:symcyl}]
  We construct the maps $\{\phi_P\}_P$ and connection one-forms
  $\{\alpha_P\}_P$ by induction on the dimension of $P$. Consider a
  polytope $P \in \PP$.  We assume that $\phi_Q$ and $\alpha_Q$ are
  determined for every proper face $Q \subset P$ that is an element of
  $\PP$.  We first construct a $T_P$-equivariant diffeomorphism
  \[\psi_P : \Phinv(\tP) \to \Phinv(P) \times P^\dual \]
  that satisfies the properties expected of $\phi_P$ in
  $\Phinv(\tP \cap \tQ)$ for all proper faces $Q \subset P$ as
  follows: The inner product on $\t$ in \eqref{eq:idtt} gives the
  following orthogonal projection
  \begin{equation}
    \label{eq:phip1}
    Q^\dual \to P^\dual \times \t_Q^\dual/\t_P^\dual, 
  \end{equation}
  and the image lies in $P^\dual \times \NCone_Q P$ if we restrict the
  map to an appropriate neighborhood in $Q^\dual$.  Viewing
  $\NCone_QP$ as a subset of $Q^\dual$ via the cylindrical structure
  map $\phi_Q$, we obtain an embedding
  \begin{equation}
    \label{eq:phip2}
    \Phinv(Q) \times \NCone_QP \to \Phinv(P)   
  \end{equation}
  defined on a neighborhood of the origin in $\NCone_QP$. The map
  $\psi_P$ is defined on $\Phinv(\tQ)$ by \eqref{eq:phip1} and
  \eqref{eq:phip2}.  Extend $\psi_P$ to a $T_P$-equivariant
  diffeomorphism on $\Phinv(\tP)$ that respects the $T_P$-moment map,
  that is, satisfies
  \[ \Pi_{P^\dual} \circ \psi_P=\pi_{P^\dual} \circ \Phi, \]
  where
  $\Pi_{P^\dual}$ is from Definition \ref{def:sympcylstr}
  \eqref{part:sympstr2} and $\pi_{P^\dual}$ is from Definition
  \ref{def:fibpoly}.  Next, we define a symplectic form on
  $\Phinv(P) \times P^\dual$ by choosing a connection one-form
  $\alpha_P$: For any proper face $Q \subset P$, on
  $\Phinv(\tQ \cap P)/T_P$, the connection one-form $\alpha_P$ is
  given by the consistency condition \eqref{eq:consis}.  We may choose
  any extension of $\alpha_P$ in the rest of $\Phinv(P)/T_P$.  By our
  choice of connection one-form $\alpha_P$ on
  $\Phinv(\tQ \cap P)/T_P$, $\psi_P$ is a symplectomorphism on
  $\Phinv(\tP \cap \tQ)$ for all $Q \subset P$.  Next, we apply Lemma
  \ref{lem:coiso} to isotope $\psi_P$ to a sympectomorphism $\phi_P$,
  so that $\phi_P \equiv \psi_P$ on $\Phinv(\tP \cap \tQ)$ for all
  $Q \subset P$.  As a last step, we replace $\Phi$ by a tropical
  moment map $\Phi'$, such that $\Phi \equiv \Phi'$ in $\Phinv(\tQ)$,
  $Q \subset P$, $\Phi'$ satisfies \eqref{eq:modify}, and
  $\Phi' \circ \phi_P=\pi_P \circ \Phi'$.
  \end{proof}

  The proof of Proposition \ref{prop:symcyl} uses the following
  result, which is a version of the coisotropic neighborhood theorem.
  
\begin{lemma}\label{lem:coiso}
  Let $Z \subset (X,\om)$ be a submanifold with boundary.  Suppose a
  neighborhood $U_Z \subset X$ of $Z$ has moment map 
  $\Phi_P : U_Z \to \t_P^\dual$ that is smooth and generates
  a free Hamiltonian $T_P$-action, for which $Z=\Phi_P^{-1}(0)$.  Then, for any connection one-form $\alpha_P \in \Om^1(Z,\t_P)$ on the bundle $Z \to Z/T_P$, there is a neighborhood $\Delta \subset \t_P^\dual$ of the origin, and a
  $T_P$-equivariant symplectomorphism
  \[\phi_P : (U_Z,\om) \to (Z \times \Delta, \om_{Z/T_P} +
    \d\bran{\alpha_P,i}) \]
  on $U_Z:=\Phi_P^{-1}(\Delta)$, where $i : \Delta \to \t_P^\dual$ is
  the inclusion map.  Furthermore, if
  $\psi_P : U_Z \to Z\times \Delta$ is a diffeomorphism that preserves
  the symplectic form on $N:=U(\partial Z) \times \Delta$, where
  $U(\partial Z) \subset Z$ is a neighborhood of $\partial Z$, then
  $\psi_P$ can be isotoped to a symplectomorphism $\phi_P$ relative to
  $N$.
\end{lemma}
  
The proof of the ordinary symplectic neighborhood (\cite[Lemma
3.14]{ms:intro}) can be used to prove the slightly stronger statement
of Lemma \ref{lem:coiso}.

\chapter{Broken disks}\label{chap:brokendisks}

\index{Holomorphic disks! Broken treed holomorphic disks}
\index{Map|seeonly{Holomorphic disks}} \index{Map! Broken maps} The
goal of this chapter is to define \em{broken treed holomorphic
  disks}. These are analogues of what Parker \cite{bp1} calls \em{exploded} holomorphic maps. These structures combine the features of
treed holomorphic disks and tropical (or broken) maps.
\begin{itemize}
\item Treed holomorphic disks consist of surface components that are
  nodal holomorphic disks or spheres in a symplectic manifold whose
  boundary lies in a Lagrangian submanifold, with the additional
  feature that disk nodes may be replaced by tree segments mapping to
  the Lagrangian submanifold. 
  Treed disks are generalizations of  pearly trajectories in
  Biran-Cornea \cite{bc:ql}.
\item On the other hand, a broken map is defined on the normalization
  of a nodal curve, and each of the domain components maps to a
  different component of the target broken manifold, and the maps
  satisfy a matching condition at the nodal lifts. The broken map is
  equipped with the additional data of a tropical graph in the dual
  complex associated with the degeneration of the symplectic manifold.
\end{itemize}
In a broken treed holomorphic disk, certain nodes in the domain are
\em{tropical nodes} which means that the curve components incident on the
node map to different target components; and the other nodes, called
\em{internal nodes}, are sphere nodes as seen in Gromov-Witten theory or
disk nodes with treed segments as in treed holomorphic disks.  Since we
assume that the Lagrangian submanifold is disjoint from the relative
divisors of the broken manifold, all disk nodes are internal nodes.

\begin{figure}[h]
  \centering \scalebox{.8}{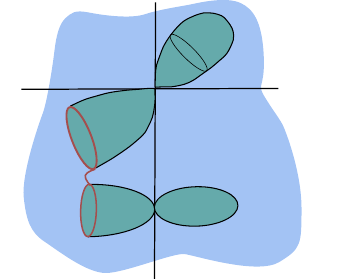}
  \caption{A broken treed holomorphic disk. The disk node $w_3$ has been replaced by a treed segment in the Lagrangian $L$, and $w_1$, $w_2$ are tropical nodes.}
  \label{fig:brokendisk}
\end{figure}

\section{Treed disks}

\em{Treed disks} are the domains of broken treed holomorphic disks, which are the main  objects defined in this Chapter. 
Treed disks 
are  analogues of pearly trajectories of Biran-Cornea \cite{bc:ql},
Cornea-Lalonde \cite{cl:clusters} and Seidel \cite{seidel:genustwo}, and 
they are combinations of trees, nodal disks and nodal spheres.

\begin{definition} \label{def:treeddisk} \index{Treed disk} \index{Curve| \seeonly {Disk}}
  {\rm(Treed disks)} 
  \begin{enumerate}   \index{Disk! Nodal disk}
  \item {\rm (Nodal disks)} A nodal disk $S$ is a union
    \begin{equation} \label{Sunion} S = \left(\bigcup_{\alpha =
          1,\ldots, d(\white)} S_{\alpha,\white} \right) \cup \left(
        \bigcup_{ \beta = 1,\ldots, d(\black)} S_{\beta,\black}
      \right)/ \sim \end{equation}
    where the index $\alpha$, ranging over $\{1,\dots,d(\white)\}$,
    indexes disk components, that is, for each $\alpha$, $S_\alpha$ is
    biholomorphic to a unit disk $\mathbb{D}^2 \subset \C$, and the
    index $\beta$, ranging over $\{1,\dots,d(\black)\}$, indexes
    sphere components, that is, for each $\beta$, $S_\beta$ is
    biholomorphic to the projective line $\P^1$, glued together by an
    equivalence relation $\sim$. The equivalence relation $\sim$ is
    generated by pairs of interior or boundary \em{nodal points}
    \begin{equation} w_e = (w^+_e, w^-_e) \in
      (\cup_\alpha \partial S_{\alpha})^2 \cup (\cup_{\alpha}
      (S_{\alpha}^\circ) \cup \cup_\beta S_{\beta})^2
    \end{equation} 
    of pairs of boundary points of disks, interior points of disks or
    spheres, with the property that the boundary $\partial S$ is
    connected and there are no cycles of components
    $S_{\alpha_1},\ldots, S_{\alpha_k} = S_{\alpha_1}$ connected by
    nodes (so that in particular, each node connects two different
    disk or sphere components of $S$).  \label{nocycles} A \em{
      marking} of a nodal disk is a collection of boundary and
    interior points
    \begin{equation}
      \label{eq:markorder}
\ul{z}_\white = (z_{\white,i} \in \partial S, i = 0,\ldots,
      d(\white)), \quad \ul{z}_\black = (z_{\black,i} \in S \bs
      \partial S, i = 1,\ldots, d(\black)) 
    \end{equation}
    distinct from the nodes.  A marked nodal disk is \em{stable} if
    it admits no automorphisms, or equivalently, if for each disk
    component $S_{\white,i}$ the sum of the number of special (nodal
    or marked) boundary points and twice the number of interior
    special points is at least three, and each sphere component
    $S_{\black,i}$ has at least three special points.
  \item {\rm (Combinatorial type of a nodal disk)} \label{part:combidisk}
    The \em{combinatorial type} of a
    nodal disk $S$ is the tree $\Gamma$ whose vertices $\Ver(\Gamma)$ 
    correspond to disk or sphere components of $S$, and whose edges
    $\Edge(\Gamma)$ correspond to markings or nodes, together with 
    \begin{enumerate}
    \item {\rm(Sphere and disk vertices)}
      \index{Vertex! Disk vertex} \index{Vertex! Sphere vertex}
      a partition of the vertex
      set
      \[ \Ver(\Gamma) = \Ver_{\black}(\Gamma) \cup
        \Ver_\white(\Gamma) \]
      into disk vertices $\Ver_{\white}(\Gamma)$ and sphere vertices
      $\Ver_{\black}(\Gamma)$;
    \item {\rm(Interior and boundary edges)} a partition of the edge
      set
      \index{Edge! Interior edge} \index{Edge! Boundary edge}
      \[ \Edge(\Gamma) = \Edge_{\black}(\Gamma) \cup
        \Edge_\white(\Gamma) \]
      into edges $\Edge_{\white}(\Gamma)$ of boundary type and edges
      $\Edge_{\black}(\Gamma)$ of interior type;
    \item {\rm(Leaf and non-leaf edges)} \index{Edge! Leaf edge}
      a partition of the edge set
      $\Edge(\Gamma) $
      \[ \Edge(\Gamma) = \Edge_-(\Gamma) \cup
        \Edge_{\rightarrow}(\Gamma)
      \]
      into leaf edges $\Edge_{\rightarrow}(\Gamma)$ which correspond
      to markings,
      and non-leaf edges $\Edge_{-}(\Gamma)$
      which correspond to nodes (and so, leaf edges are incident on a
      single vertex and non-leaf edges are each incident on two
      vertices);
    \item {\rm(Ordering of leaves)} an ordering of the boundary leaf
      edges $\Edge_{\white,\to}(\Gamma)$ and the interior leaf edges
      $\Edge_{\black,\to}(\Gamma)$ given by the corresponding ordering
      of markings (see \eqref{eq:markorder});
    \item {\rm(Ribbon structure)} a \em{ribbon structure} on
      $\Gamma$, which is a cyclic ordering $<_v$ on the set of
      boundary edges (both leaf and non-leaf)
      \[\Edge_\white^v(\Gamma):=\{e \in \Edge_\white(\Gamma): v \in e\} \]
      incident on each boundary vertex $v \in \Ver_\white(\Gamma)$
      such that the induced cyclic ordering on the set
      $\Edge_{\white,\to}(\Gamma)$ of boundary leaves corresponds to
      the cyclic ordering
      $z_{\white,0},\dots,z_{\white,d(\white)},z_{\white,0}$ of
      boundary markings.
      \label{cyclicstr}
    \item {\rm(Root edge and edge orientations)} The edge
      $e_0 \in \Edge_\white(\Gamma)$ corresponding to the first
      boundary marking $z_0$ is an outgoing edge \label{outgoing} and
      is called the \em{root edge}, all the other boundary markings
      are incoming edges, and all edges corresponding to nodes are
      oriented to point towards the root. \index{Edge! Root edge}
      \index{Root of a treed disk}
    \end{enumerate}
  \item{\rm(Treed segments)} A \em{treed segment} consists of a
    collection of closed intervals $I_1,\ldots, I_k$,
    \[ I_j \cong [0,\ell(I_j)], \quad  (-\infty,
        0], \quad \text{or} \quad I_j \cong [a_j,\infty)\quad \text{or} \enspace (-\infty,\infty) \]
      glued along infinite end-points, to produce a topological space
      isomorphic to some subset of $\R$.  For example,
    \begin{equation}
      \label{eq:brokenseg}
    [0,\infty) \cup_\infty (-\infty,\infty) \cup_\infty
      (-\infty,0]  
    \end{equation}
    is a treed segment with three components and finite end-points.
    Each treed segment $T$ has a length $\ell(T) \in [0,\infty]$ and a
    number of breakings $b(T) \in \Z_{\ge 0}$. A treed segment with
    $b(T)>0$ is called a \em{broken segment} and has
    $\ell(T) = \infty$.  We also consider treed segments with one
    infinite end such as
    \[[0,\infty) \cup_\infty (-\infty,\infty)\]
    or both infinite ends such as $(-\infty,\infty) \cup_\infty (-\infty,\infty)$.
  \item {\rm(Treed disk)} \label{part:treeddisk} \index{Disk! Treed
      disk} A \em{treed nodal disk} $C = S \cup T$ is
    \begin{enumerate}
    \item either obtained from a nodal disk $S_0$
      \begin{itemize}
      \item by assigning a length $\ell(e) \in [0,\infty]$
        \index{Edge! Edge length $\ell(e)$} to each boundary node
        $e \in \Edge_{\white,-}(\Gamma)$, and replacing any boundary
        node $w_e$, $e \in \Edge_{\white,-}(\Gamma)$ with $\ell(e)>0$
        by a treed segment $T_e$ with finite end-points, and the treed
        segment $T_e$ has length $\ell(e)$ if $\ell(e)<\infty$ or it
        is a broken segment as in \eqref{eq:brokenseg} if
        $\ell(e)=\infty$; and
      \item each boundary marking $w_e$,
        $e \in \Edge_{\white,\to}(\Gamma)$ is replaced by a (possibly
        broken) treed segment, one of whose end-points is infinite,
        and thus $\ell(e)=\infty$ for all boundary markings; or
      \end{itemize}
    \item $C$ has no surface component and consists only of a treed
      segment with two infinite ends.  The $-\infty$ resp. $\infty$
      end of the segment is regarded as the input resp. output, and
      therefore there is one output and one input, so
      $d(\white)=1$.
    \end{enumerate}
    A treed segment $T_e$, $e \in \Edge_\white(\Gamma)$, containing a
    breaking is called a \em{broken edge}. \index{Edge! Broken edge}
    A non-leaf edge $e \in \Edge_-(\Gamma)$ is broken if and only if
    $\ell(e)=\infty$.
  \item {\rm(Isomorphism of treed disks)}
    \label{part:isotreeddisk}
    An isomorphism of treed disks is a homeomorphism $\phi:C \to C'$
    that is a biholomorphism on each sphere or disk component,
    length-preserving on edges, and preserves the labelling of
    leaves.  \label{page:mordisks}
  \item{\rm(Combinatorial type of a treed
      disk)} \label{part:typetreeddisk} \index{Combinatorial type! of
      a treed disk} \index{Type|seeonly {Combinatorial type}} The
    combinatorial type of a treed disk consists of the combinatorial
    type $\Gamma$ of the underlying nodal disk which includes the
    vertex and edge partitions, ordering of markings, ribbon structure
    and edge orientations; and in addition a partition
      \[\Edge_{\white,-}(\Gamma)=\Edge_{\white,-}^0(\Gamma) \cup
        \Edge_{\white,-}^{(0,\infty)} \cup
        \Edge_{\white,-}^\infty(\Gamma)\]
      of boundary edges corresponding to boundary nodes with zero,
      finite, and infinite length edges. 
      Note that for a boundary edge $e$ with edge length
      $\ell(e) \in (0,\infty)$, the length of the treed segment $T_e$
      is not part of the combinatorial type.
    \item {\rm(Stable treed disk)} A treed nodal disk $C$ is \em{
        stable} if the underlying disk $S$ is stable, the treed
      segment at any node $T_e, e \in \Edge_{\white,-}(\Gamma)$ has at
      most one breaking, and treed segments at markings
      $T_e, e \in \Edge_{\rightarrow,\white}(\Gamma)$ are unbroken. A
      treed disk with no surface component (that is, consisting of a
      single possibly broken sequence of segments) is not stable.
    \item {\rm (Disconnected types)} \label{discon} Let $\Gamma$ be a
      disjoint union of treed disk types $\Gamma_1,\ldots,\Gamma_k$. A
      treed curve of type $\Gamma$ is a collection $u_1,\ldots, u_k$
      of treed disks of types $\Gamma_1,\ldots, \Gamma_k$.
  \end{enumerate}
\end{definition}
For integers $d(\black), d(\white) \geq 0$, denote by
$\M_{d(\black), d(\white)}$ the (possibly empty) moduli space of
isomorphism classes of stable treed disks with $d(\white)$ incoming
boundary markings, one outgoing boundary marking and $d(\black)$
interior markings.  For each combinatorial type $\Gamma$, denote by
$\M_\Gamma \subset \M_{d(\black), d(\white)}$ \index{Moduli space!of
  treed disks $\M_{\Gamma}$} the set of isomorphism classes of stable
treed disks of type $\Gamma$. The moduli space
$\M_{d(\black), d(\white)}$ then decomposes as
\begin{equation*}
  \M_{d(\black), d(\white)} 
=\bigcup_\Gamma \M_\Gamma.
\end{equation*}
For any type $\Gamma$, $\M_\Gamma$ has a compactification $\ol \M_\Gamma$ where the boundary is 
\[\ol \M_\Gamma \bs \M_\Gamma=\bigcup_{\Gamma'} \M_{\Gamma'},\]
and $\Gamma'$ ranges over all strata such that $\Gamma$ is obtained from $\Gamma'$ by either collapsing an interior edge or a zero length boundary edge, or making a zero length boundary edge non-zero, or making an infinite boundary edge finite, or performing a combination of these operations.  See \ref{fig:disk-moduli} for an example.

\begin{figure}[h]
  \centering \scalebox{.8}{
\begingroup%
  \makeatletter%
  \providecommand\color[2][]{%
    \errmessage{(Inkscape) Color is used for the text in Inkscape, but the package 'color.sty' is not loaded}%
    \renewcommand\color[2][]{}%
  }%
  \providecommand\transparent[1]{%
    \errmessage{(Inkscape) Transparency is used (non-zero) for the text in Inkscape, but the package 'transparent.sty' is not loaded}%
    \renewcommand\transparent[1]{}%
  }%
  \providecommand\rotatebox[2]{#2}%
  \newcommand*\fsize{\dimexpr\f@size pt\relax}%
  \newcommand*\lineheight[1]{\fontsize{\fsize}{#1\fsize}\selectfont}%
  \ifx\svgwidth\undefined%
    \setlength{\unitlength}{428.13776418bp}%
    \ifx\svgscale\undefined%
      \relax%
    \else%
      \setlength{\unitlength}{\unitlength * \real{\svgscale}}%
    \fi%
  \else%
    \setlength{\unitlength}{\svgwidth}%
  \fi%
  \global\let\svgwidth\undefined%
  \global\let\svgscale\undefined%
  \makeatother%
  \begin{picture}(1,0.23081474)%
    \lineheight{1}%
    \setlength\tabcolsep{0pt}%
    \put(0,0){\includegraphics[width=\unitlength,page=1]{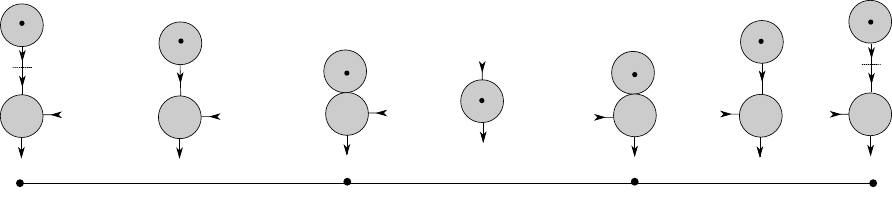}}%
    \put(0.49266273,0.00233964){\color[rgb]{0,0,0}\makebox(0,0)[lt]{\lineheight{1.25}\smash{\begin{tabular}[t]{l}$\Gamma_1$\end{tabular}}}}%
    \put(0.11828434,0.00209962){\color[rgb]{0,0,0}\makebox(0,0)[lt]{\lineheight{1.25}\smash{\begin{tabular}[t]{l}$\Gamma_2$\end{tabular}}}}%
    \put(0.82917392,0.00461347){\color[rgb]{0,0,0}\makebox(0,0)[lt]{\lineheight{1.25}\smash{\begin{tabular}[t]{l}$\Gamma_3$\end{tabular}}}}%
  \end{picture}%
\endgroup%
}
  \caption{The moduli space $\M_{1,1}$ has three top-dimensional strata corresponding to the types $\Gamma_1$, $\Gamma_2$, $\Gamma_3$.}
  \label{fig:disk-moduli}
\end{figure}
The top-dimensional cells in $\M_{d(\black), d(\white)}$ correspond to
strata where all boundary edges $e$ have finite non-zero
length, \label{page:topdim} that is, $\ell(e) \in (0,\infty)$. The
dimension of each of the top-dimensional cells is
$d(\white)+2d(\black)-2$.  For the stratum with no finite edges, so
containing a single disk, this follows immediately from the fact that
each boundary leaf edge resp. interior leaf edge contributes $1$
resp. $2$ to the dimension, and the automorphism group of the disk is
$PSL(2,\R)$ which has dimension $3$. The dimensions of the other
top-dimensional cells can be computed by first computing the
corresponding stratum of the moduli spaces of disks without trees, and
then adding one for each boundary edge with finite non-zero length.
The moduli space $\M_{d(\black),d(\white)}$ is a
manifold with corners where the codimension $k$ corner stratum
consists of curves containing $k$ broken treed edges.  We do not give a
proof, since for our purposes it is enough to view
$\M_{d(\black),d(\white)}$ as a cell-complex, with a manifold
structure on each of the strata $\M_\Gamma$.

\index{Orientation for moduli spaces!of treed disks} 
\begin{definition}\label{def:orientmg}
  {\rm(Orientation of moduli of treed disks)} We fix an orientation
  for moduli spaces of treed disks of type $\Gamma$, where $\Gamma$
  does not contain boundary edges $e \in \Edge_{\white,-}(\Gamma)$ of
  length $0$ or $\infty$.
  \begin{enumerate}
  \item{\rm(Single disk)} Let $\Gamma$ be a treed disk type with a
    single disk component with $\geq 2$ incoming boundary leaves and
    no other surface component.  The moduli space $\M_\Gamma$ is the
    quotient
    \[\on{Conf}_{d(\black),d(\white)}/PSL(2,\R),\]
    where
    \[\on{Conf}_{d(\black),d(\white)} \subset \{\ul z=(\ul z^\white, \ul z^\black) \in (\partial D)^{d(\white)
        +1} \times (D^\circ)^{d(\black)}\}\]
    is the subset of configurations where $\ul z$ consists of distinct
    points and the boundary markings are ordered counter-clockwise.
    Fixing $z_0^\white=-1$, $z_1^\white=1$, $z_2^\white=i$ gives a
    global slice $\Sigma$
    of the $PSL(2,\R)$-action.  Furthermore, 
    $\M_\Gamma$ is oriented via its identification to the slice
    $\Sigma$
    which is an open subset of
    $\R^{d(\white)-2} \times \C^{d(\black)}$.  In the case that $\Gamma$ has less than 
    $2$ incoming boundary leaves, we fix the global slice 
$\Sigma$ \label{rep:sigma}
by setting
    $z_0^\white=-1$, $z_0^\black=0$.
  \item {\rm(Multiple disks)} Let $\Gamma$ be a treed disk type with
    no sphere components and 
    whose boundary edges all 
    have finite
    non-zero length.  The orientation on $\M_\Gamma$ is chosen by
    inducting on the number of boundary edges
    $e \in \Edge_{\white,-}(\Gamma)$.
    Suppose the treed disk types $\Gamma$, $\Gamma_0$, $\Gamma'$ are related by the following morphisms:
    \[\Gamma \xleftarrow{\text{Make the edge length $\ell(e)$ non-zero}} \Gamma_0 \xrightarrow{\text{Collapse $e$}} \Gamma'. \]
    Here,
    we point out that the edge length $\ell(e)$ in $\Gamma_0$ is zero.
    The moduli space 
    $\M_{\Gamma_0}$ is a codimension one boundary stratum of both
    $\ol \M_{\Gamma}$ and $\ol \M_{\Gamma'}$ via inclusions
    \[i_{\Gamma,\Gamma_0}:\M_{\Gamma_0} \to \ol \M_\Gamma, \quad
      i_{\Gamma',\Gamma_0}:\M_{\Gamma_0} \to \ol \M_{\Gamma'}.\]
    Assuming an orientation on $\M_\Gamma'$, 
    the stratum $\M_\Gamma$ is oriented so that 
    the orientations on
    $\M_{\Gamma_0}$ induced by $i_{\Gamma,\Gamma_0}$,
    $i_{\Gamma',\Gamma_0}$ are the opposite of each other. 
  \item {\rm(Disks and spheres)} Let $\Gamma$ be a treed disk type,
    and let $\Gamma'$ be the type obtained by collapsing interior
    edges $e \in \Edge_{\black,-}(\Gamma)$. Then the orientation on
    $\M_{\Gamma'}$
    and the standard orientation on the moduli space of spheres
    induces an orientation on $\M_{\Gamma} \subset \ol \M_{\Gamma'}$
    which is equal to the product orientation
    $\M_{\Gamma_\white} \times \prod_{v \in \Ver_\black(\Gamma)}
    \M_{\Gamma_v}$. Here, $\Gamma_\white \subset \Gamma$ is the
    subgraph consisting of all the disk vertices and $\Gamma_v$ is a
    graph with a single vertex $\{v\}$ and markings corresponding to
    all edges $e \in \Edge_\black(\Gamma)$ incident on $v$.
\end{enumerate}
\end{definition}

\begin{remark}
  Let $\Gamma$ be a treed disk type containing a single edge
  $e \in \Edge_{\white,-}(\Gamma)$ of infinite length, and 
  whose other boundary edges
  $e' \in \Edge_{\white,-}(\Gamma)$ all 
  have $\ell(e') \in (0,\infty)$.
  Let $\Gamma'$ be a treed disk type obtained from $\Gamma$ by making the edge length of $e$ finite. Then,
  \[i_{\Gamma,\Gamma'}: \M_\Gamma \hra \ol \M_{\Gamma'}\]
  is a boundary stratum of codimension one.  Suppose $\Gamma_1$,
  $\Gamma_2$ are treed disk types obtained by cutting the edge $e$ in
  $\Gamma$ (see Definition \ref{def:pertops}), and suppose the root of
  $\Gamma$ is contained in $\Gamma_1$. Then the boundary orientation
  on $\M_\Gamma$ induced by $i_{\Gamma,\Gamma'}$ differs from the
  product orientation $\M_{\Gamma_1} \times \M_{\Gamma_2}$ by
  \begin{equation}
    \label{eq:orient-diff}
    \eps(\Gamma_1,\Gamma_2):=(-1)^{\diamondsuit} 
  \end{equation}
  where $\diamondsuit$ depends only on the combinatorial types $\Gamma_1$, $\Gamma_2$, see \cite[(12.22)]{se:bo}. 
\end{remark}

The moduli spaces admit universal curves, which admit partitions into
one and two-dimensional parts.  For any combinatorial type $\Gamma$,
let $\U_\Gamma$ denote the universal treed disk consisting of
isomorphism classes of pairs $(C, z)$ where $C$ is a treed disk of
type $\Gamma$ and $z$ is a point in $C$, possibly on a disk component,
sphere component, or one of the edges of the tree. \label{udisk} The
map
\begin{equation} \label{eq:univcurv} \U_\Gamma \to \M_\Gamma, \quad [C,z]
  \mapsto [C]
\end{equation}
is the universal projection, whose fiber over $[C]$ is a copy of $C$.
The union over types $\Gamma'$ with
$\M_{\Gamma'} \subset \ol{\M}_\Gamma$ is denoted $\ol{\U}_\Gamma$.
Denote by
\[\S_\Gamma,\quad \text{resp.}\quad \T_\Gamma\]
the locus of points $[C,z] \in \U_\Gamma$ where $z$ lies on a
disk or sphere component resp. an edge of $C$.
Hence $\U_\Gamma =
\S_\Gamma \cup \T_\Gamma$, and $\S_\Gamma \cap \T_\Gamma$
is the set of points on the boundary of the disks meeting the edges of
the tree.
Denote by 
\[\ol \S_\Gamma,\quad \text{resp.}\quad \ol \T_\Gamma\]
the compactification of $\S_\Gamma$ resp. $\T_\Gamma$ in $\ol \U_\Gamma$.

On each stratum, the universal curve admits local
trivializations. \label{localtriv} For a stratum $\Gamma$, the spaces
\[\S_\Gamma \to \M_\Gamma, \quad \text{resp.} \quad \T_\Gamma \to \M_\Gamma\]
are smooth fibrations with $2$ resp. $1$-dimensional fibers and 
markings are sections
\[\ul z=(z_{i,\white}, z_{j,\black}): \M_\Gamma \to \S_\Gamma, \quad 0 \leq i \leq d(\white), 1 \leq j \leq d(\black).\]
We view breaking points on broken treed segments as sections
$\M_\Gamma \to \T_\Gamma$.  On a type $\Gamma$, the treed segment
$T_e$ corresponding to any edge $e \in \Edge_\white(\Gamma)$ has a
fixed number of breakings, each of which gives rise to a section.  The
union $\U_\Gamma=\S_\Gamma \cup \T_\Gamma$ has local trivializations:
For a curve $[C] \in \M_\Gamma$, $C=S \cup T$, and a small enough
neighborhood $U_C \subset \M_\Gamma$ of $[C]$, there is a
homeomorphism
\begin{equation}
  \label{eq:univlocaltriv}
  \U_\Gamma|U_C \simeq U_C \times C.  
\end{equation}
When restricted to each stratum of
$\S_\Gamma|U_C$ resp. $\T_\Gamma|U_C$, the homeomorphism above is a
diffeomorphism onto its image.  The markings
\[z_{i,\white}, z_{j,\black}|U_C: U_C \to S\]
are constant functions, whose values we denote by
$\ul{z_{U_C}}=(z_{i,\white,U(C)}, z_{j,\black,U(C)})$, where
$z_{i,\white,U(C)}, z_{j,\black,U(C)} \in S$.  The fibers of
$\S_\Gamma \to U_C$ possess a conformal structure, and the
trivialization in \eqref{eq:univlocaltriv} induces a map
\begin{equation}
  \label{eq:jmoduli}
  j: U_C \to \J(S)   
\end{equation}
to the space $\J(S)$ of conformal structures on $S$ such that for
$[C_1] \in U_C$, the  marked curve $(C,\ul{z_{U_C}})$ with complex structure 
$j([C_1])$ is biholomorphic to the marked curve
$(\S_\Gamma, \ul z)|_{[C_1]}$.

\begin{remark}\label{rem:smooth-breaks}
  The structure of a ``smooth fibration with sections'' on
  $\S_\Gamma \to \M_\Gamma$ and $\T_\Gamma \to \M_\Gamma$ breaks down on
  the extension to the compactification $\ol \M_\Gamma$.  For example,
  a finite treed segment $T_e$ in the fibers of $\T_\Gamma$ may be
  transformed to a segment of ``zero length'' in the compactification
  $\ol \T_\Gamma$; or two disconnected components in the fibers of
  $\S_\Gamma$ may connect at a disk node in the compactification
  $\ol \S_\Gamma$.
\end{remark}

\section{Treed pseudoholomorphic disks}

\label{page:intropar}
Treed pseudoholomorphic disks are maps from treed disks to a
symplectic manifold equipped with a Lagrangian submanifold.  The
symplectic manifold has a tamed almost complex structure and the
Lagrangian submanifold has a Morse function on it.  On the
two-dimensional part of the treed disk the map is pseudoholomorphic, 
and the boundaries of the disks map to the Lagrangian submanifold. On
the one-dimensional part of the domain, the map is a gradient flow
line of the Morse function on the Lagrangian submanifold, whose length
is the same as the length of the tree edge.  Later in this chapter, 
we will
adapt the definition of treed holomorphic disks in a symplectic
manifold to define broken treed holomorphic disks in a broken
manifold.  Even later in the text,  the almost complex structure, the Morse
function, and the metric on the Lagrangian will be given
domain-dependent perturbations in order to regularize the moduli
spaces of treed holomorphic (broken) disks.

We introduce the necessary notation for defining treed holomorphic
disks.  Let $(X,\om_X)$ be a symplectic manifold and $L \subset X$ be a
Lagrangian submanifold.  Let $J$ be an $\om_X$-tame almost complex
structure.  Let $G_L$ be a Riemannian metric on $L$ and let
$F \in C^\infty(L,\R)$ be a Morse function such that the pair
$(F,G_L)$ is Morse-Smale. \index{Morse function on the Lagrangian}
The \em{gradient vector field} is defined
by the condition
\[ \grad_F \in \Vect(L), \quad  \d F(\cdot)=G_L(\grad_F,\cdot) .\]

\begin{definition}\label{def:treedholdisk} 
  A \em{treed holomorphic disk} with boundary in
  $\lag \subset X$ consists of a treed disk $C = S \cup T$ and a
   continuous map
  \[u: C \to X\]
  satisfying the following conditions: 
  \begin{enumerate}
  \item The tree components $T$ and the boundary $\partial S$ of the
    surface components are mapped to the Lagrangian submanifold $L$;
    that is, $ u(T \cup \partial S) \subset \lag $;
  \item the map $\left. u \right|_S$ is a pseudoholomorphic map on the
    surface part:
    $ J d (\left. u \right|_S) = d(\left. u \right|_S)\circ j $; and
  \item the map $\left. u \right|_T$ is a union of gradient
    trajectories of $F$:
    \[ \dds \left. u \right|_T = - \grad_F(\left. u \right|_T) \]
    where $s$ is a unit velocity coordinate on $T$.
  \end{enumerate}
\end{definition}

A treed holomorphic disk $u: C = S \cup T \to X$ is \em{stable}
\label{page:finiteaut} if it has finitely many automorphisms, 
$\# \Aut(u) < \infty$, or equivalently
\begin{enumerate}
\item each surface component $S_v \subset S$ on which the map $u$ is
  constant is stable as a component of a nodal disk $S$ (see
  Definition \ref{def:treeddisk}); and
\item each treed segment $T_e$ on which the map $u$ is
  constant has at most one infinite end, that is, one of the ends of
  $T_{e}$ is an attaching point to a sphere or disk $S_v \subset S$.
\end{enumerate}
Note that the case $C \cong \R$ equipped with a non-constant Morse
trajectory $u: C \to L$ is allowed under this stability condition and
corresponds to the case of a single incoming edge, that is,
$d(\white) = 1$. \label{Rcase}
The area of a sphere or disk $u: S \to X$ is the symplectic area
\[ \Area(u) = \int_S (u|_S)^* \omega_X .\]

\section{Multiply-broken disks}
\label{sec:mbrokdisks}

A broken map is a map from a nodal curve to a multiply cut manifold
that is discontinuous at tropical nodes.  Different components of the
nodal curve map to different pieces of the multiply cut manifold, and
the lifts of the tropical nodal points satisfy an edge-matching
condition. The nodal points carry an additional data of intersection
multiplicity with relative divisors. This data is packaged into a \em{tropical structure}, which is part of the combinatorial type of the
broken map.

\begin{definition}{\rm(Tropical graph)}
  \label{def:tropgraph}
  Let $B^\dual$ be the dual complex for a set of polytopes
  $\PP=\{P \subset \t^\dual\}$ as in Definition \ref{def:dualcomplex}.
  A \em{pre-tropical graph} is a triple \index{Tropical graph}
  $(\Gamma,P,\cT)$ consisting of
  \begin{enumerate}
  \item a graph $\Gamma$ with vertex set $\Ver(\Gamma)$ and edge set
    $\Edge(\Gamma) \subset \Ver(\Gamma) \times \Ver(\Gamma)$,
  \item polytope assignments
    \begin{equation} \label{eq:passign}
      P: \Ver(\Gamma) \cup \Edge(\Gamma) \to \PP
    \end{equation} 
    such that for any edge $e=(v_+,v_-)$, $P(e)=P(v_+) \cap P(v_-)$, and
  \item edge \em{{direction}s}
    \index{{Direction}! of an edge}
    \[\cT(e) \in \t_{P(e),\Z}\bs \{0\}, \quad \forall \enspace e\in
      \Edge(\Gamma). \]
  \end{enumerate}
  A \em{tropical graph} is a pre-tropical graph $(\Gamma,P,\cT)$ that
  is \em{realizable}
  \index{Realizability of a tropical graph}
  in the  sense that there is a collection of \em{tropical vertex positions} 
  \index{Vertex positions! of a tropical graph $\W(\Gamma)$}
  \begin{equation} \label{eq:tv} \cT(v) \in P(v)^{\dual,\circ} \subset
    B^\dual, \quad v \in \Ver(\Gamma)
  \end{equation}
  satisfying the following \em{direction condition} \index{Direction
    condition! for tropical graphs}
  \begin{equation} \label{eq:bslope} {\text{\rm (Direction condition)}}
    \quad \cT(v_+) - \cT(v_-) \in \R_{> 0} \cT(e) .
  \end{equation} 
  The set of
  tropical vertex position maps on a tropical graph $\Gamma$ is
  denoted
  \begin{equation}
    \label{eq:tweights}
    \W(\Gamma) = \{ ( \cT(v))_{v \in \Ver(\Gamma)} : \cT \text{ is a tropical vertex position map on $\Gamma$} \}. 
  \end{equation}
  A tropical graph $\Gamma$ is \em{rigid} if $\W(\Gamma)$ has exactly
  one element.  \index{Rigid! tropical graph}
\end{definition}

\begin{remark} 
  \label{rem:samegraph}
  A deformation of vertex positions for a tropical graph $\Gamma$ is shown
  in Figure \ref{fig:movetrop} in the dual complex $B^\dual$
  corresponding to two orthogonal cuts (see Figure \ref{fig:break} and
  Figure \ref{fig:sqdual}). By the definition of tropical graphs,
  moving the vertices $v_3$ and $v_4$ to the dotted positions produces
  a one-parameter space of tropical vertex positions.  Thus, 
  the space of vertex positions is
  \[\W(\Gamma) \simeq (0,1).\]
 This remark is continued in Remark \ref{rem:samegraph2}.
\begin{figure}[h]
  \centering \scalebox{.8}{
\begingroup%
  \makeatletter%
  \providecommand\color[2][]{%
    \errmessage{(Inkscape) Color is used for the text in Inkscape, but the package 'color.sty' is not loaded}%
    \renewcommand\color[2][]{}%
  }%
  \providecommand\transparent[1]{%
    \errmessage{(Inkscape) Transparency is used (non-zero) for the text in Inkscape, but the package 'transparent.sty' is not loaded}%
    \renewcommand\transparent[1]{}%
  }%
  \providecommand\rotatebox[2]{#2}%
  \newcommand*\fsize{\dimexpr\f@size pt\relax}%
  \newcommand*\lineheight[1]{\fontsize{\fsize}{#1\fsize}\selectfont}%
  \ifx\svgwidth\undefined%
    \setlength{\unitlength}{144.2134395bp}%
    \ifx\svgscale\undefined%
      \relax%
    \else%
      \setlength{\unitlength}{\unitlength * \real{\svgscale}}%
    \fi%
  \else%
    \setlength{\unitlength}{\svgwidth}%
  \fi%
  \global\let\svgwidth\undefined%
  \global\let\svgscale\undefined%
  \makeatother%
  \begin{picture}(1,0.52349845)%
    \lineheight{1}%
    \setlength\tabcolsep{0pt}%
    \put(0,0){\includegraphics[width=\unitlength,page=1]{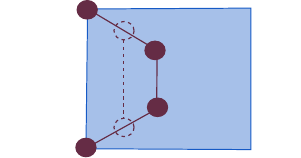}}%
    \put(0.1162315,0.47694421){\color[rgb]{0,0,0}\makebox(0,0)[lt]{\lineheight{0}\smash{\begin{tabular}[t]{l}$P_0^\dual$\end{tabular}}}}%
    \put(0.123555,0.01854872){\color[rgb]{0,0,0}\makebox(0,0)[lt]{\lineheight{0}\smash{\begin{tabular}[t]{l}$P_1^\dual$\end{tabular}}}}%
    \put(0.72068672,0.23764086){\color[rgb]{0,0,0}\makebox(0,0)[lt]{\lineheight{0}\smash{\begin{tabular}[t]{l}$P_\cap^\dual$\end{tabular}}}}%
    \put(0.53070827,0.39610028){\color[rgb]{0,0,0}\makebox(0,0)[lt]{\lineheight{1.25}\smash{\begin{tabular}[t]{l}$v_3$\end{tabular}}}}%
    \put(0.56612609,0.11301018){\color[rgb]{0,0,0}\makebox(0,0)[lt]{\lineheight{1.25}\smash{\begin{tabular}[t]{l}$v_4$\end{tabular}}}}%
    \put(-0.00441826,0.2295931){\color[rgb]{0,0,0}\makebox(0,0)[lt]{\lineheight{1.25}\smash{\begin{tabular}[t]{l}$\Gamma$\end{tabular}}}}%
  \end{picture}%
\endgroup%
}
  \caption{Moving the vertices $v_3$ and $v_4$ to the dotted positions
    gives a different vertex position map for the same tropical graph $\Gamma$.}
  \label{fig:movetrop}
\end{figure}
\end{remark}

\begin{remark}\label{rem:Wpoly}
  {\rm(The set of vertex positions is the interior of a polytope)}
  For a tropical graph $\Gamma$, the closure $\ol \W(\Gamma)$ of the space $\W(\Gamma)$ of vertex positions is given by 
  \begin{multline*}
    \ol \W(\Gamma):=\{(\cT(v) \in  P(v)^\dual)_{v \in \Ver(\Gamma)} : \\
    \cT(v_+) - \cT(v_-) \in \R_{\geq 0} \cT(e), \enspace \forall
    e=(v_+,v_-) \in \Edge(\Gamma)\}.
\end{multline*}
  Thus $\ol \W(\Gamma)$ is obtained by weakening the conditions \eqref{eq:tv} and \eqref{eq:bslope} in the definition of $\W(\Gamma)$. The closure $\ol \W(\Gamma)$ is a compact polytope in the affine space (of the same dimension) 
  \begin{multline*}
    A_\Gamma:=\{(\cT(v) \in \t_{P(v)})_{v \in \Ver(\Gamma)} :\\
    \cT(v_+) - \cT(v_-) \in \R \cT(e), \enspace \forall e=(v_+,v_-) \in \Edge(\Gamma) \}.
  \end{multline*}
  The facets of the polytope $\ol \W(\Gamma)$ are given by the conditions
  \[\cT(v) \in P(v)^\dual, \quad  \cT(v_+) - \cT(v_-) \in \R_{\geq 0} \cT(e),\]
  for each vertex $v$ and each edge $e=(v_+,v_-)$ of $\Gamma$. 
\end{remark}

\begin{definition} \label{def:tropstrtreed}
  {\rm(Tropical structure on a treed disk type)}
  \index{Tropical structure}
  Let $\Gamma$ be the
  combinatorial type of a treed disk (see Definition \ref{def:treeddisk}\eqref{part:typetreeddisk}).  Let $\XX$ be a
  broken manifold with an underlying polyhedral decomposition $\PP$,
  and a Lagrangian submanifold $L$ contained in the component
  $X_{P_0} \subset \XX$, $P_0 \in \PP$.

  A \em{tropical structure} on $\Gamma$ consists of a tropical graph
  $(\Gamma_\tr, P, \cT)$ and an edge collapse morphism
  \begin{equation} \label{eq:tr}
  \tr : \Gamma \to \Gamma_\tr \end{equation}
called the \em{tropicalization}.\index{Tropicalization} Here
$P:\Ver(\Gamma_\tr) \cup \Edge(\Gamma_\tr) \to \PP$ is the polytope
assignment on the tropical graph (see \eqref{eq:passign}), and
$\cT=(\cT(e) \in \t_{P(e),\Z})_{e \in \Edge(\Gamma_\tr)}$ is the
collection of edge {direction}s.  The edges collapsed by $\tr$ are called
\em{internal edges} and uncollapsed edges are called \em{tropical
  edges}. The subset of internal resp. tropical edges of $\Gamma$ is
denoted by \index{Edge! Tropical edge} \index{Edge! Internal edge}
  \[ \Edge_\internal(\Gamma) \quad{\text{resp.}} \quad \Edge_\trop(\Gamma) \subset \Edge_-(\Gamma). \]
  All boundary edges are collapsed 
  by assumption. 
  Therefore,
  \[\Edge_{-,\white}(\Gamma) \subset \Edge_{\internal}(\Gamma), \quad \Edge_\trop(\Gamma) \subset \Edge_{-,\black}(\Gamma).\]
  The map $\tr$ is often suppressed in the notation.  The polygon
  assignment $P \circ \tr$ and edge {direction} $\cT \circ \tr$ maps on
  $\Gamma$ are often denoted by $P$, $\cT$.
  \end{definition}

  \begin{remark} \label{rem:zeroslope} In a treed disk type $\Gamma$
    equipped with a tropical structure, since the tropicalization map
    collapses all the boundary edges $e \in \Edge_\white(\Gamma)$, all
    the boundary vertices $v \in \Ver_\white(\Gamma)$ are mapped to a
    single vertex in $\Gamma_\tr$.
\end{remark}

We recall some notation required to define broken maps.  The domain of
a broken map is a treed disk $C=T \cup S$ (Definition
\ref{def:treeddisk}) where $S$ is the surface part and $T$ is the tree
part.  For a treed disk $C$ of type $\Gamma$ equipped with a tropical
structure, for every vertex $v \in \Ver(\Gamma)$ we denote by
\begin{equation}
  \label{eq:opensv}
  S_v^\circ = S_v \bs \{w_e : e \in \Edge_\trop(\Gamma)\}  
\end{equation}
the surface part with cylindrical ends obtained by deleting nodal
points corresponding to tropical edges. The target space of a broken
map is a broken manifold $\XX$.  We recall from 
\eqref{eq:brokenmfd} that a broken manifold with an underlying
polyhedral decomposition $\PP$ is a disjoint union
\[\XX=\bigsqcup_{P \in \PP}\XC_P,\]
and there is a top-dimensional polytope $P_0 \in \PP$ such that
$\XB_{P_0}$ contains a Lagrangian submanifold $L$.  Note that
$\XB_{P_0}=\XC_{P_0}$ since the torus $T_{P_0}$ is trivial.

\begin{definition}
  \label{def:bmap1s}
  {\rm(A surface component of a broken map)} Let $\Gamma$ be the
  combinatorial type of a treed disk equipped with a \em{tropical
    structure} and let $v$ be a vertex of $\Gamma$. Define the domain
  $S_v$ to be a disk if $v \in \Ver_\white(\Gamma)$ and a sphere if
  $v \in \Ver_\black(\Gamma)$.  The \em{surface component of a broken
    map} corresponding to $v$ consists of a punctured domain curve
  $S_v^\circ:=S_v \bs \{w_e : v \in e, e \in \Edge_\trop(\Gamma)\}$ and a
  holomorphic map
  \[u_v : S_v^\circ \to \XB_{P(v)}\]
  with the following behavior at punctures: For any tropical edge
  $e \in \Edge_\trop(\Gamma)$, $e \ni v$, and any holomorphic
  coordinate in a punctured neighborhood $U_{w_e}$ \label{rep:uw} of
  $w_e$,
  \[z_e : (U_{w_e} \bs \{w_e\}) \to (\C,0),\]
  the map $z_e^{-\cT(e)}u_v$ has a removable singularity at $w_e$.
  The quantity
  \begin{equation}
    \label{eq:tropev}
\ev_{w_e}^{\cT(e)}u_{v}:=\lim_{z_e \to 0} z_e^{-\cT(e)}u_{v} \in \XC_{P(e)}.  
\end{equation}
\index{Tropical evaluation map} 
is called the \em{tropical evaluation} at the puncture $w_e$ for the
choice of holomorphic coordinate $z_e$.  Here,
$\cT(e) \in \t_{P(e),\Z}$ is the {direction} of the edge $e$,
\footnote{It is enough to just consider broken maps whose edges have integral {direction}s $\cT(e)$,
  and not rational {direction}s, 
  because the free torus actions in the neighborhood of cut loci implies that in the neighborhood of a puncture, a map is asymptotic to a torus orbit generated by an integral element. Also, see the related Remark \ref{rem:integ-edge} following the proof of the removal of singularities result.
}
and for any
$z \in \C^\times$, $z^{\cT(e)}$ is an element of the one-dimensional
torus $T_{\cT(e),\C} \subset T_{P(e),\C}$. Via the cylindrical
coordinate map \eqref{eq:Pcoord} which identifies the
$P(e)$-cylindrical end to $\XC_{P(e)}$, we may assume that
  \[u_{v}(U_{w_e} \bs \{w_e\}) \subset \XC_{P(e)}.\]
  Recall that $\XC_{P(e)}$ has a holomorphic $T_{P(e),\C}$-action.
\end{definition}

It is often useful to have a projected version of the tropical evaluation map that does not involve the choice of a holomorphic coordinate in the neighborhood of the puncture.

\begin{definition} \label{def:proj-eval} {\rm(Projected tropical
    evaluation map)} \index{Projected tropical evaluation map} In the
  notation of Definition \ref{def:bmap1s} of a component of a broken
  map, for a puncture $w_e$ and coordinates $z_e$ in the neighborhood
  of $w_e$, the \em{projected tropical evaluation} at a puncture
  $w_e$ corresponding to an edge $e \in \Edge_\trop(\Gamma)$ is
  \begin{equation}
    \label{eq:evproj}
    \pi_{\cT(e)}^\perp (\ev_{w_e}^{\cT(e)}u_{v})\in \XX_{P(e)}/T_{\cT(e),\C},
  \end{equation}
  which is equal to $\lim_{w_e} (\pi^\perp_{\cT(e)} \circ u_v)$, and
  therefore, is independent of the coordinate $z_e$.  Here,
  $\pi^\perp_{\cT(e)} : \XX_{P(e)} \to \XX_{P(e)}/T_{\cT(e),\C}$ is
  the quotient map by the complex subtorus
  $T_{\cT(e),\C} \subset T_{P(e),\C}$.
\end{definition}
An unframed broken map consists of surface components as in Definition \ref{def:bmap1s}, for which projected tropical evaluations on any pair of nodal lifts match; and treed segments that are mapped to gradient flow lines on the Lagrangian submanifold.
\begin{definition} \label{def:ubmap} {\rm(Unframed broken map)} Let
  $\XX_\PP$ be a broken manifold and $L \subset \XB_{P_0}$ be a
  Lagrangian submanifold as above.  An \em{unframed broken map} $u$
  to $\XX$ is a datum consisting of
  \begin{enumerate}
  \item {\rm(Domain type and tropical structure)} a \em{domain type}
    $\Gamma$, which is the combinatorial type of a treed disk equipped
    with a \em{tropical structure} (Definition
    \ref{def:tropstrtreed}) for which disk vertices map to $P_0$ :
    \[v \in \Ver_\white(\Gamma) \implies P(v) = P_0; \]
  \item {\rm (Domain curve)} a treed nodal disk $C=S \cup T$ of type
    $\Gamma$ consisting of a surface component $S_v$ (sphere or disk) for every vertex $v$ of $\Gamma$, and a treed segment $T_e$ for every boundary edge $e \in \Edge_\white(\Gamma)$;
  \item {\rm (Map)} a collection of holomorphic maps on punctured
    surface components
    \begin{equation*} u_v: S_v^\circ \to \XC_{P(v)}, \quad v \in
      \Ver(\Gamma)
    \end{equation*}
    as in Definition \ref{def:bmap1s}, and continuous maps on the treed
    segments
    \[ u_e: T_e \to L, \quad e \in \Edge_\white(\Gamma) \]
    collectively denoted $u: C \to \XX $, such that
    \begin{itemize}
    \item {\rm(Behavior on the boundary)} the restriction to the
      components with boundary, namely
      \[u|(\cup_{v \in \Ver_\white(\Gamma)}S_v^\circ \cup \cup_{e \in
          \Edge_\white(\Gamma)}T_e), \]
      of $C$ is a treed holomorphic map (in the sense of Definition
      \ref{def:treedholdisk}) to the target space $(\XB_{P_0},L)$;
    \item {\rm (Matching at internal nodes)} for an interior internal
      edge
      \[e=(v_+,v_-) \in \Edge_\black(\Gamma) \cap
        \Edge_\internal(\Gamma),\]
      the corresponding nodal points $w_e^\pm \in S_{v_\pm}$ map to
      $\XC_{P(v_\pm)}$ and the map $u$ is continuous at the node, that is, 
      \begin{equation} \label{zeromatch} u(w^+_e)=u(w^-_e) \in
        \XC_{P(v_\pm)},\end{equation}
    \item {\rm (Matching at tropical nodes)} for a tropical edge
      $e=(v_+,v_-) \in \Edge_\trop(\Gamma)$, 
      the projected tropical
      evaluations \eqref{eq:evproj} on the nodal lifts are equal, that is,
  \begin{equation}
        \label{eq:trop-matching}
      (\pi_{\cT(e)}^\perp \circ u_{v_+})(w^+_e)=(\pi_{\cT(e)}^\perp
        \circ u_{v_-})(w^-_e),  
      \end{equation}
    \end{itemize}

  \end{enumerate}
  that is stable (as in Definition \ref{def:iso-stab} \eqref{part:stable}).  This
  ends the Definition.
\end{definition}

\begin{definition}\label{def:bmap}
  {\rm(Broken map)} A broken map $u : C \to \XX$ is an
unframed broken map (Definition \ref{def:ubmap}) with the additional
data of a \em{framing} at tropical nodes $w_e$,
$e \in \Edge_\trop(\Gamma)$. A framing is a linear isomorphism
\index{Framing}
\begin{equation} \label{eq:framingeq} \fr_e : T_{w^+_e} S_{v_+}
  \tensor T_{w^-_e} S_{v_-} \to \C,
\end{equation}
satisfying the following: 
For any holomorphic
  coordinate
  \begin{equation}
    \label{eq:holcoord}
    z_\pm : (U_{w_e^\pm}, w_e^\pm) \to (\C,0)  
  \end{equation}
  in the neighborhood of nodal lifts $w_e^\pm$ that respect the framing, that is,
  \begin{equation}
    \label{eq:framingcoord}
    \d z_+(w^+_e) \tensor \d z_-(w^-_e)=\fr_e,
  \end{equation}
  the following matching condition is satisfied at tropical nodes:
  \begin{equation}
    \label{eq:nodematch}  
    \lim_{z_- \to 0} z_-^{-\cT(e)} 
    u_{v_-}(z_-) = 
    \lim_{z_+ \to 0} z_+^{\cT(e)} 
    u_{v_+}(z_+). \index{Matching condition at nodes}
  \end{equation}
  The coordinates $z_+$, $z_-$ are called \em{matching coordinates}
  at the node $w_e$. \index{Matching coordinates}
  \end{definition}
  %
  %
  %

  We consider isomorphism classes of stable broken maps.

  \begin{definition}\label{def:iso-stab}
    {\rm(Stability and isomorphism)}
    \begin{enumerate}
    \item  \label{part:horizconst}\index{Horizontally constant}
  In a broken map $u$,
  a surface component  $\left. u \right|_{S_v^\circ}: S_v^\circ 
  \to \XX_{P(v)}$ is
  \em{horizontally constant} if its projection to $\XB_{P(v)}$ is constant.
  \item {\rm(Stability)} \label{part:stable}
  A broken map $u: C \to \XX$ is required to be \em{stable} in the sense that 
 any surface component $S_v \subset C$ on which the map 
  $u_v : S_v^\circ \to \XB_{P(v)}$ is horizontally constant is stable 
  as a marked curve, and any tree component $T_e \subset C$ on
  which $\left. u \right|_{T_e}$ is constant does not contain an infinite segment $\R \subset T_e$.
  \item \label{part:mapiso}
  {\rm(Isomorphisms)} An \em{isomorphism} between two unframed broken
  maps $u: C \to \XX$, $u': C' \to \XX$ is an isomorphism
  $\phi:C \to C'$ of treed disks (see Definition 
  \ref{def:treeddisk}\eqref{part:isotreeddisk}) such that
  $u=u' \circ \phi$. Two framed broken maps $(u,\fr)$, $(u',\fr')$ are
  isomorphic if the preceding relation holds, and in addition for any
  node $w$ in $C$ corresponding to an edge $e$ in the underlying
  graph, and lifts $w^\pm$,
  $\fr_e \circ (\d\phi(w^+) \tensor \d\phi(w^-)) =\fr_e'$.
    \end{enumerate}
  \end{definition}

The stability condition on broken maps implies that the
automorphism group of any broken map is finite.

\begin{remark}
  {\rm(Trivial cylinder components)}
  The stability condition implies
  that a broken map 
  cannot 
  have components without markings on which
  the map is a trivial cylinder, that is, the map is of the form 

  \[u_v : S_v^\circ \simeq \C^\times \to \XX_{P(v)}, \quad z \mapsto
    z^\mu x\]
  for some $\mu \in \t_{P(v),\Z}$ and $x \in \XX_{P(v)}$. Indeed, such
  a component has two nodes and is horizontally constant. Trivial
  cylinders with markings do not occur in zero and one-dimensional
  moduli spaces of broken maps, as in Example \ref{ex:triv1}.
\end{remark}

\begin{remark} \label{rem:extiso} In this book, we only use the zero-dimensional 
components of the moduli space, and the above notion of isomorphism
is sufficient.  To construct higher-dimensional moduli spaces of isomorphism classes of broken maps, one also needs to take into account the symmetries of the target manifolds.  In this case, one would define an isomorphism between two broken maps    $u: C \to \XX$, $u': C' \to \XX$ 
to be an isomorphism  of domains $\phi: C \to C'$
     and a tropical symmetry 
    \[ (g_v \in T_{P(v),\C}, v \in \Ver(\Gamma), z_e \in \C^\times, e
    \in \Edge(\Gamma) ) \]
     intertwining the maps $u,u'$ in the sense that
   \begin{equation} \label{eq:riso}
 u_v' = g_v u_v \circ (\left. \phi \right|_{C_v}) , \quad \forall v \in
     \Ver(\Gamma) . \end{equation} 
   In fact, Theorem \ref{thm:cpt-breaking} shows that if a sequence of
    maps in neck-stretched manifolds  converges to a limit map whose tropical graph $\Gamma$ is not rigid (and so, the group $T_\trop(\Gamma)$ of tropical symmetries is positive dimensional), the limit map
   is determined only up to the action of the identity component
of the tropical symmetry group, denoted by 
   $T_{\trop,\W}(\Gamma) \subset T_\trop(\Gamma)$.  
\end{remark}
%


\begin{definition}
    {\rm(Area of a broken map)} 
    \label{def:br-area}
    The \em{area} of a surface component
  $u_v: S_v^\circ \to \XX_{P(v)}$ 
  of a broken map is the symplectic area of its horizontal projection $\pi_{P(v)} \circ u_v : S_v^\circ \to  X_{P(v)}$. That is,
  \begin{equation}
    \label{eq:area-uv}
    \Area(u_v)=\bran{(\pi_{P(v)} \circ u_v)_*[S_v], \om_{X_{P(v)}}}. 
  \end{equation}
  Note that the pairing is defined via a map
  $\phi : X_{P(v)} \to X_{P(v)}^\om$ that respects the cylindrical
  structure (as made precise in Remark \ref{rem:nosymp}
  \eqref{part:nosymp1}).  Since at any puncture in $S_v^\circ$,
  $\pi_{P(v)} \circ u_v$ is asymptotic to a trivial cylinder (see
  \eqref{eq:tropev}), $\phi(\pi_{P(v)} \circ u_v)$ extends to a
  continuous map $\ol u_v : S_v \to \ol X_{P(v)}^\om$. A different
  choice of $\phi$ changes $\ol u_v$ by a homotopy. 
\end{definition}

\begin{remark}\label{rem:samegraph2}
  {\rm(Tropical graph as a combinatorial invariant)} The tropical
  graph of a broken map is a combinatorial invariant consisting of the
  data of edge {direction}s $\cT(e)$, $e \in \Edge_\trop(\Gamma)$ and vertex
  polytopes $P(v)$, $v \in \Ver(\Gamma)$, since varying the tropical
  positions of vertices does not produce a new broken map, see Remark \ref{rem:samegraph}. 
\end{remark}

The following remarks are easy conclusions of the definition of broken maps, and are intended to help the reader process the definition.

\begin{remark}{\rm(Continuity away from tropical nodes)}
  Suppose the tropical structure on a broken map is given by
  $\Gamma \xrightarrow{\tr} \Gamma_\tr$ where $\tr$ is the
  tropicalization map and $\Gamma_\tr$ is a tropical graph.  The
  domain of a broken map breaks up into connected components
  \[C \bs \{w_e : e \in \Edge_\trop(\Gamma)\} = \cup_{v \in
      \Ver(\Gamma_\tr)} C_v^\circ.\]
  For a vertex $v \in \Ver(\Gamma_\tr)$ in the tropical graph, if
  $\tr^{-1}(v)$ does not have disk vertices, then $C_v$ is a nodal
  sphere with punctures, with nodes corresponding to internal edges in
  $\tr^{-1}(v)$, and punctures corresponding to tropical edges
  incident on $\tr^{-1}(v)$. If $\tr^{-1}(v)$ has disk components then
  $C_v^\circ$ is a treed disk with punctures.  For any vertex
  $v \in \Ver(\Gamma_\tr)$ in the tropical graph, the restriction
  \[u|C_v^\circ : C_v^\circ \to \XC_{P(v)}\]
  is continuous.
\end{remark}
\begin{remark}\label{rem:1cutmatch}
  {\rm(Node matching for a single cut)} In the case of a single cut, the
  familiar form of the node matching condition from symplectic field
  theory corresponds to the matching condition \eqref{eq:trop-matching}
  for unframed maps.
  Indeed, for any edge $e$ in a tropical graph, $\codim(P(e))=\dim T_{P(e)}=1$, and so, $T_{\cT(e),\C}=T_{P(e),\C}$; the unframed matching condition \eqref{eq:trop-matching} then translates to the matching of evaluations on $\XB_{P(e)}$, namely, 
  $(\pi_{P(e)} \circ u_{v_+})(w^+_e)=(\pi_{P(e)} \circ u_{v_-})(w^-_e)$.
\end{remark}

\begin{definition}\label{def:primdirection}
  {\rm(Primitive direction, multiplicity of a tropical edge)} 
  In a tropical graph $\Gamma$, the {direction} $\cT(e) \in \t_{P(e),\Z} \bs \{0\}$ of a tropical edge $e$ \cwl{representing an interior node} can be written as a product
  \begin{equation*}
    \cT(e)=\mu_e\cT(e)_{\prim}
  \end{equation*}
  of a primitive integer vector $\cT(e)_\prim \in \t_{P(e),\Z}$, 
  called the \em{primitive direction}, and a positive integer
  $\mu_e \in \Z_{\geq 1}$ which is called the \em{multiplicity}
  of the edge $e$. \index{Primitive direction} \index{Multiplicity of an edge}
\end{definition}
\begin{remark}{\rm(Number of framings)} \label{rem:nfr} Any unframed
  broken map possesses framings.  For an unframed broken map $u$ of
  type $\Gamma$, the number of framings is equal to
  $\Pi_{e \in \Edge_{\trop}(\Gamma)}\mu_e$, where
  $\mu_e \in \Z_{\geq 1}$ is the multiplicity of the edge $e$
  (Definition \ref{def:primdirection}).  Indeed, the trivial cylinder
  \[ \C^\times \to T_{P(e)}, \quad z \mapsto z^{\cT(e)} \] 
which is the one-parameter subgroup induced by the Lie algebra element $\cT(e)$,  is a $\mu_e$-cover. 
 As a consequence, if
  $\fr_e$ is a framing for the node $w_e$, then
  $e^{2\pi i k/\mu_e}\fr_e$ is also a framing for
  $k=0,1,\dots, \mu_e-1$.
\end{remark}

\begin{remark}\label{rem:extend}
  {\rm(Extending broken map components over punctures)}  In the special
  case that the compactifications of the components of the broken
  manifold are manifolds, a broken map extends over lifts of nodal
  points to yield
  a map
  \begin{equation}
    \label{eq:extcv}
    u_v : S_v \to \ol \XC_{P(v)}, \quad v \in \Ver(\Gamma).
  \end{equation}
  Indeed, the existence of the tropical evaluation \eqref{eq:tropev}
  implies that, in $\XX_{P(v)}$, $u_v$ is asymptotically close to
  trivial cylinders at punctures.  If $\ol \XC_{P(v)}$ has orbifold
  singularities, that is, if the polytope $P(v)$ is simple and not
  Delzant, the extension \eqref{eq:extcv} is defined on a domain curve
  with orbifold singularities.
  Since compactifications of cut spaces \eqref{eq:xpcpt} are orbifolds, the projections extend over
  punctures to yield $\pi_{P(v)} \circ u_v : S_v \to \ol X_{P(v)}$. 
  For any
  nodal lift $w \in S_v$ corresponding to an edge $e=(v,v')$, we have
  \begin{equation}
    \label{eq:extpip}
    (\pi_{P(v)} \circ u)(w) \in X_{P(e)} \subset \ol X_{P(v)}.
  \end{equation}
  This can be seen as follows: For each relative divisor
  $X_Q \subset \ol X_{P(v)}$, the normal vector $\mu_Q$ in the moment
  polytope lies in $\t_Q/\t_{P(v)}$, and a punctured neighborhood of
  $X_Q$ in $X_{P(v)}$ is a subset of a $T_{\mu_Q,\C}$-fibration.  If
  the edge {direction} $\cT(e) \in \t_{P(e)}$ is spanned by normal vectors
  $\mu_{Q_1},\dots,\mu_{Q_k}$, where each $Q_i \in \PP$ is a facet of
  $P(v)$, then $X_{P(e)} = \cap_i X_{Q_i}$, and
  $(\pi_{P(v)} \circ u)(w) \in X_{P(e)}$ since the map is
  asymptotically close to a $T_{\cT(e),\C}$-orbit near the puncture
  $w$.
\end{remark}

\begin{remark}\label{rem:intmult}
   {\rm(Relationship of edge {direction} with intersection multiplicities)}
  Consider the special case when the compactifications 
  of the broken manifold are smooth manifolds.
  Consider a component $u_v : S_v \to \ol \XX_{P(v)}$ of a broken map. 
  Suppose for a nodal lift $w \in S_v$ corresponding to an edge
  $e=(v,v')$, $u_v(w)$ lies on the intersection of relative divisors
  $Y_1,\dots,Y_k \subset \ol \XC_{P(v)}$.  Furthermore, suppose
  $\nu_1,\dots,\nu_k \in \t$ are the primitive outward pointing normal
  vectors to the facets of $\tP(v)$ corresponding to the divisors
  $Y_1,\dots, Y_k$, and that the map $u_v$ intersects the divisor
  $Y_i$ with multiplicity $\mu_i \in \Z_+$. Then we have the following
  relation to the {direction} of the edge:
  \begin{equation}
    \label{eq:intmult}
    \sum \mu_i \nu_i=\cT(e).  
  \end{equation}
  In the case when $\ol \XX_{P(v)}$ is an orbifold, the intersection
  multiplicities $\mu_i$ may be fractional, though the relation
  \eqref{eq:intmult} still holds.
  However, we do not use orbifold
  intersection numbers in any of the proofs.  This ends the Remark.
\end{remark}

\begin{remark}\label{rem:expconv}
  {\rm(Exponential convergence to a trivial cylinder at tropical
    nodes)} The removal of singularity of punctures from Remark
  \ref{rem:intmult} can alternately be stated as that a broken map is
  asymptotically close to a trivial cylinder at the punctured
  neighborhood of any nodal lift.  For a tropical node $w_e$
  corresponding to $e \in \Edge_\trop(\Gamma)$ and holomorphic
  coordinates
  \begin{equation}
    \label{eq:nodecoord}
    (s,t) : U(w_e^\pm) \bs \{w_e^\pm\} \to \R_{\geq 0} \times S^1    
  \end{equation}
  in punctured neighborhoods of  the nodal lifts, the limit
  \[x_e^\pm:=\lim_{s \to \infty}e^{\pm\cT(e)(s+it)}u_{v_\pm}(s,t)\]
  exists and is a point in the $P(e)$-cylindrical end
  $U_{P(e)}(\XX_{P(v_\pm)})$.  Furthermore, the map $u_{v_\pm}$
  exponentially converges to a trivial cylinder with {direction} 
  $\cT(e)$ in
  the punctured neighborhood $U(w_e^\pm) \bs \{w_e^\pm\}$, that is,
  there is a constant $c> 0$ such that
  \begin{equation*}
    d(u_{v_\pm}(s,t),e^{\mp\cT(e)(s+it)}x_e^\pm) \leq ce^{-s}, \quad s \geq s_0, 
  \end{equation*}
  using the cylindrical metric in $\XC_{P(e)}$. The proof is left to the reader.
\end{remark}


\begin{remark}
  {\rm(A different view of node matching)}
\label{rem:hv-match}
  In the special case 
  that 
  the multiple cut is a collection of orthogonally intersecting single
  cuts, and each cut space is a manifold, the tropical node matching condition
  \eqref{eq:nodematch}  
  splits into 
  the following conditions:
  \begin{itemize}
  \item a \em{horizontal matching} condition on the intersection of
    relative divisors, and
  \item a \em{vertical matching} condition involving leading
    order Taylor coefficients in the directions normal to each of the
    relative divisors.
  \end{itemize}
  We consider a multiple cut given by a tropical moment map
  $\Phi:X \to \R^k$ with cuts along the hypersurfaces $\Phinv_i(0)$,
  $i=1,\dots,k$.  Denote the relative divisors by
  $Y_i:=\Phinv_i(0)/S^1$ (though $Y_i$ is a broken manifold).
  Consider a broken map $u$ with components $u_+$, $u_-$ in the cut
  spaces for the polytopes 
  \[ P_+:=\cap_{i=1}^k\{\Phi_i \geq 0\}, \quad P_-:=\cap_{i=1}^k\{\Phi_i \leq 0\} \] 
  respectively, sharing a node $w_e$. As discussed in Remark \ref{rem:intmult}, the map $u_\pm$ extends over the puncture at the nodal point $w_e^\pm$. The
  intersection multiplicities with the divisors $Y_1,\dots,Y_k$ are
  the same for $u_+$, $u_-$, and are denoted by
  $\mu_1,\dots,\mu_k \in \Z_{>0}$. The matching condition consists of
  \begin{itemize}
  \item {\rm(Horizontal matching)} a horizontal condition which says
    that
    \[u_+(w^+_e)= u_-(w^-_e) \in X_{P(e)},\]
    where $X_{P(e)}=\Phinv(0)/(S^1)^k \simeq Y:= \cap_{i=1}^k Y_i$;
  \item {\rm(Vertical matching)} and a vertical condition which says
    that for holomorphic coordinates
    \[z_\pm : (U(w_e^\pm),w_e^\pm) \to (\C,0)\]
    in a neighborhood of the nodal lifts that respect the framing (as
    in \eqref{eq:framingcoord}), for all $i$, the $(\mu_i+1)$-th
    derivative normal to $Y_i$ is equal for $u_{v_+}$ and
    $u_{v_-}$. In the direction normal to $Y_i$, since derivatives up
    to order $\mu_i$ vanish, the $(\mu_i+1)$-th normal derivative, or
    the $(\mu_i+1)$-jet normal to $Y_i$, denoted by 
    \[ j^{\mu_i}_{Y_i}u_\pm(w_e^\pm) \in N_\pm Y_i \bs Y_i\]
    is well-defined (see \cite[Section 6]{cm:trans}), where
    $N_\pm Y_i$ is the normal bundle of $Y_i$ in $\ol X_{P_\pm}$. The
    vertical matching condition says that
    \[j^{\mu_i}_{Y_i}u_+(w^+_e)=j^{\mu_i}_{Y_i}u_-(w^-_e).\]
  \end{itemize}

  This view of the matching condition is relevant only in the case of a
  collection of orthogonal single cuts. In general cases, given a
  node, both nodal lifts do not lie on the same set of relative
  divisors.  For example, in the broken map in Figure \ref{fig:mikh}
  corresponding to a Mikhalkin graph, there are edges of the form
  \begin{itemize}
  \item $e=(v_+,v_-)$, $\cT(e)=(1,1)$
  \item with $P(v_+)=P(v_-)=P_0$ which is the zero-dimensional
    polytope; and
  \item one of the lifts lies on a divisor of $\ol \XC_{P_0}$
    corresponding to a facet with normal vector $(1,1)$ and another
    lift lies on the intersection of divisors corresponding to facets
    with normal vectors $(0,-1)$ and $(-1,0)$.
  \end{itemize}
  This ends the Remark.
\end{remark}

\begin{example} \label{ex:unpack} We unpack the horizontal and
  vertical matching conditions from Remark \ref{rem:hv-match} in the
  first broken map in Figure \ref{fig:broken_eg}. The multiple cut
  consists of two orthogonal cuts shown in Figure \ref{fig:break}, and
  the node $w$ lies between maps $u_0:C_0 \to \ol X_{P_2}$,
  $u_1:C_1 \to \ol X_{P_0}$.  The nodes $u_0(w)$, $u_1(w) $ lies in
  $ \ol X_{P_\cap}$.  Suppose the leading order Taylor term of $u_0$ in
  the direction normal to $X_{P_{12}}$ resp. $X_{P_{23}}$ is
  $a_+ z_+^{\mu_h^+}$ resp. $b_+ z_+^{\mu_v^+}$, and suppose the
  leading order Taylor term of $u_1$ in the direction normal to
  $X_{P_{30}}$ resp. $X_{P_{01}}$ is $a_- z_-^{\mu_h^-}$ resp.
  $b_- z_-^{\mu_v^-}$. Then the matching condition says that
\begin{itemize}
\item {\rm (Matching of intersection multiplicities)}  $\mu_h^+=\mu_h^-$, $\mu_v^+=\mu_v^-$,
\item {\rm (Horizontal matching)} the projections of the nodal evaluation maps to $X_{P_\cap}$ are equal:
  $\pi_{P_\cap} (u(w^+))=\pi_{P_\cap}(u(w^-)) \in X_{P_\cap}$,
\item {\rm(Vertical matching)} and the leading order Taylor
  coefficients are equal: $a_+=a_-$, $b_+=b_-$.
\end{itemize}

\end{example}

\begin{figure}[ht]
  \centering \scalebox{.8}{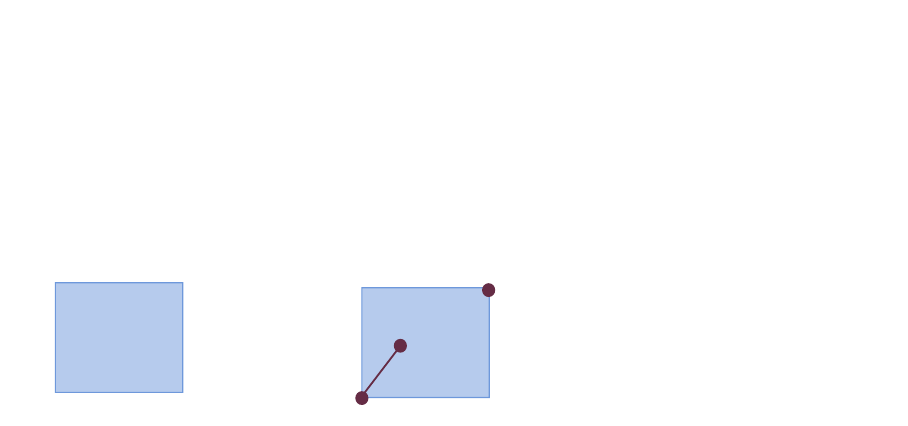}
  \caption{Broken maps and their dual graphs in the broken manifold of Figure \ref{fig:multpiece}.}
  \label{fig:broken_eg}
\end{figure}

\begin{remark}
  {\rm(Comparison to Ionel's refined matching)} In \cite[p14]{ion:nc},
  Ionel states the edge-matching condition using a ``refined
  evaluation map'', which is analogous to the projected tropical
  evaluation map defined in \eqref{eq:trop-matching}.  Suppose a node
  $w_e$ maps to the intersection of relative divisors $Y_1,\dots,Y_n$
  and the intersection multiplicity with each of the divisors is
  $\mu_1,\dots,\mu_n$.  (Here we work with the extension of the broken
  map $u$ over nodal lifts as in Remark \ref{rem:intmult}.)  The
  refined evaluation map $\ev_{\on{ref}}(u,0)$ for the map $u$ and the
  nodal lift $0 \in C$ is a point in the weighted projective space
  \[\ev_{\on{ref}}(u,0) = (x_1,\dots,x_n) \in \P_{(\mu_1,\dots,\mu_n)}(\oplus_i
  NY_i).\]
The point $\ev_{\on{ref}}(u,0)$ is well-defined since replacing the
domain coordinate $z$ by $az$ for some $a \in \C^\times$ has the
effect of changing $(x_1,\dots,x_n)$ to
$(a^{\mu_1}x_1,\dots,a^{\mu_n}x_n)$. The projected tropical evaluation
in \eqref{eq:trop-matching} and \eqref{eq:evproj}
  \[\pi^\perp_{\cT(e)}(\ev^\mu(0)) \in (Z \times \R^n)/T_{\cT(e),\C}\]
  is Ionel's evaluation map written using cylindrical coordinates on
  the target space, and the {direction} $\cT(e)$ of the edge $e$ is the
  vector $(\mu_1,\dots,\mu_n)$.  The cylindrical viewpoint appears
  more natural because in the image of the refined evaluation map none
  of the coordinates $x_i$ vanish.  On the other hand, all points in
  the cylinder $(Z \times \R^n)/T_{\cT(e),\C}$ can possibly be in the
  image of the evaluation map.
\end{remark}

\begin{remark} \label{rem:balance} {\rm(Balancing property)}
  We discuss the generalization of the balancing condition
  \eqref{eq:orig-balance} for Mikhalkin graphs to broken maps.
  First, consider the special case that the projection
  $(\pi_{P(v)} \circ u_v): S_v^\circ \to X_{P(v)}$ of a map component
  $u_v: S_v^\circ \to \XX_{P(v)}$ is constant. Then, the edges
  $e \in \Edge_\trop(\Gamma)$ emanating from the vertex $v$ satisfy
  \begin{equation}
    \label{eq:constt-bal}
    \sum_{e \ni v} \cT(e)_\ver = 0 ,
  \end{equation}
  where $\cT(e)_{\ver} \in \t_{P(v)}$ is the orthogonal projection of
  $\cT(e) \in \t_{P(e)}$ to $\t_{P(v)}$ with respect to the inner product
  \eqref{eq:idtt} on $\t_{P(e)}$.
  
  We can generalize this property to non-constant projections,
  assuming that the inner product \eqref{eq:idtt} on $\t$ is rational,
  in which case the compactification $\ol \XX_P$ of the broken
  manifold is an orbifold, and $\ol Z_P \to \ol X_P$ is an orbifold
  torus bundle (see Remark \ref{rem:compactify}) for all $P \in \PP$.
  Then we have the equality
  \begin{equation} \label{eq:balprop} \sum_{e \ni v} \cT(e)_{\ver} =
    c_1( (\pi_{P(v)} \circ u_v) ^* \ol Z_{P(v)} \to \ol X_{P(v)}),
  \end{equation} 
  where
  $c_1( (\pi_{P(v)} \circ u_v) ^* \ol Z_{P(v)} \to \ol X_{P(v)})$ is
  the first Chern class of the pull-back bundle, viewed as a vector in
  $\t_{P(v)} \simeq H^2(S_v,\t_{P(v)})$, and $S_v$ is the possibly
  orbifold compactification of $S_v^\circ$ over which the map $u_v$
  extends.
  Indeed, the map $u_v$ gives a section of the $T_{P(v)}$-principal
  bundle $(\pi_{P(v)} \circ u_v) ^* \ol Z_{P(v)} \to \ol X_{P(v)}$ on
  the complement of nodal points $w_e$, $e \ni v$, and the monodromy
  of the section around each such intersection is determined by
  $\cT(e)_\ver$. This ends the Remark.
\end{remark}

\begin{remark}
  {\rm(A comparison with symplectic field theory)} Holomorphic
  buildings in symplectic field theory \cite{bo:com} are defined
  slightly differently to broken maps on manifolds with a single
  cut. The target space of a holomorphic building consists of $X^+$,
  $X^-$ and $k-1$ copies of the neck piece $Z(\P^1)$ for some
  $k \geq 1$:
  \[\XX[k]:=\ol X_+ \cup_Y Z(\P^1) \cup_Y \dots \cup_Y Z(\P^1) \cup_Y \ol X_-, \]
  and any pair of consecutive pieces are identified along a divisor
  $Y$.  A holomorphic building $u:C \to \XX[k]$ is a continuous map,
  where nodes map to the divisor $Y$, and intersection multiplicities
  $m_{w_e^\pm}(u_{v_\pm}, Y)$ are equal on both sides.

  A holomorphic building differs from a broken map in two ways:  A
  holomorphic building is a continuous map, and the data for a
  holomorphic building includes an ordering for the neck piece
  components. In the broken map view, this ordering is not important:
  With suitable regularity assumptions a broken $u$ map with $m$
  components in neck pieces can be glued to give a $2m$-dimensional
  family of unbroken maps in $X^\nu$ for any $\nu$.  Any sequence of
  maps $u_\nu:C_\nu \to X^\nu$ lying in the glued family converges to
  a broken map $u': C' \to \XX$ that is related to $u: C \to \XX$ by a
  tropical symmetry (as in Definition \ref{def:tsym}).  In contrast,
  the holomorphic building limits of different sequences in the glued
  family may not all be the same, since the choice of the sequence
  $(u_\nu)_\nu$ determines the component of $\XX[k]$ to which a curve
  component $C_v \subset C$ maps.  For a broken map $u = (u_v)$, the
  ordering of the pieces $u_v$ is not part of the data of the the map.
  One effect of the differing definitions is the following: Unlike
  holomorphic buildings, broken maps do not have components $u_v$ that
  are trivial cylinders in the sense that they map into a fiber of
  $\ol \XC_{P(v)}$ with only two marked points and so are unstable. In
  holomorphic buildings, trivial cylinders have to be inserted
  whenever there is a node $w_e$ between components $C_{v_+}, C_{v_-}$
  that are not in adjacent levels in order to achieve continuity.
  This ends the Remark.
\end{remark}

\section{Symmetries of broken maps}
The moduli space of broken maps of a fixed type $\Gamma$ has the
action of a group, called the \em{tropical symmetry group}, arising
from torus actions on the ``neck pieces'' in the degeneration.  The
tropical symmetry group $T_\trop(\Gamma)$ can be read off from the
tropical graph $\cT$ of $\Gamma$. The identity component of
$T_\trop(\Gamma)$, is generated by the degrees of freedom of vertex
positions in the tropical graph $\cT$.  For example, the graph
$\Gamma_1$ in Figure \ref{fig:rigid} has one degree of freedom, and
therefore, the identity component of $T_\trop(\Gamma_1)$ is a
one-dimensional complex torus. A tropical graph is \em{rigid} if there
is no way of moving the vertices while keeping the edge {direction}s
fixed.  Such graphs have a finite tropical symmetry group.  For example,
in Figure \ref{fig:rigid}, $|T_\trop(\Gamma_2)|=2$.  The Mikhalkin
graphs from Section \ref{sec:mikh} give a class of examples of rigid
tropical graphs.  We prove later in this section that the size of the
symmetry group for Mikhalkin graphs is equal to the multiplicity of
the graph.

The tropical symmetry group also incorporates a group action on the
framings of the broken map.  Even one of the simplest kind of broken
maps occurring in a single cut, namely one whose tropical graph
consists of two vertices each lying in $P_+^\dual$ and $P_-^\dual$ (in
Figure \ref{fig:dualcomplex}) connected by an edge of multiplicity $k$
has a tropical group of size $k$ arising from the action on framings.

\begin{figure}[ht]
  \centering \scalebox{.8}{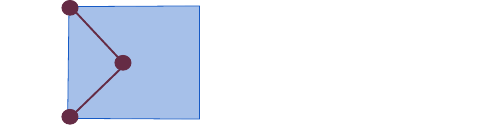}
  \caption{The dual complex $B^\dual$ for the multiple cut 
    in Figure \ref{fig:break} is a
    rectangle. The tropical graph $\Gam_1$ (left figure) is rigid, but
    $\Gam_2$ (right figure) is not rigid since the vertices inside the
    square can be moved to the dotted positions.}
  \label{fig:rigid}
\end{figure}

\index{Tropical symmetry}
\begin{definition} \label{def:tsym} {\rm(Tropical symmetry)} A \em{
    tropical symmetry} for a tropical graph $\Gamma$ 
  is a tuple
  \[(\ul g, \ul z)=((g_v)_{v \in \Ver(\Gamma)}, (z_e)_{e \in \Edge_\trop(\Gamma)}), \quad g_v \in T_{P(v),\C}, \quad z_e \in \C^\times\]
  consisting of a translation $g_v$ for each vertex and a change of
  local coordinate $z_e$ for each tropical edge that satisfies 
  \begin{equation} \label{connecting} g_{v_+} g_{v_-}^{-1} =
    z_e^{\cT(e)} \quad \forall e=(v_+,v_-) \in \Edge_\black(\Gamma),
  \end{equation} 
  where we assume $z_e=1$ for
  $e \in \Edge_\black(\Gamma) \bs \Edge_\trop(\Gamma)$.  A tropical
  symmetry $(g,z)$ acts on a broken map $(u,\fr)$ as
  \[u_v \mapsto g_vu_v, \quad \fr_e \mapsto z_e \fr_e. \]
  The group of tropical symmetries is denoted by
  \begin{equation} \label{eq:tsymg} T_{\on{trop}}(\Gamma)
:= \{((
    g_v)_{v \in \Ver(\Gamma)}, (z_e)_{e \in \Edge_\black(\Gamma)})|
    \text{\eqref{connecting}} \} .
  \end{equation}
\end{definition}

The condition \eqref{connecting} is a necessary and sufficient
condition for the translations $(g_v)_v$ to preserve the matching
condition at nodes.

\begin{remark}
  {\rm(Framing symmetry)} There are finitely many tropical symmetries
  $(\ul g, \ul z) \in T_{\on{trop}}(\Gamma)$, called \em{framing symmetries},
  for which the action on the unframed map $u$ of type $\Gamma$ is
  trivial, that is $g_v=\Id$ for all vertices $v \in
  \Ver(\Gamma)$. 
  The group of such framing symmetries $(\ul g, \ul z)$ is the
  product
  \begin{equation}
    \label{eq:zne}
  \prod_{e \in \Edge_\black(\Gamma)}\Z_{\mu_e}  \subset T_\trop(\Gamma),
\end{equation}
where $\mu_e$ is the multiplicity of the edge $e$ (Definition
\ref{def:primdirection}), and also the order of ramification of the broken
map at the nodal point $w_e$.
\end{remark}

\begin{lemma}\label{lem:same-unfr}
  Let $u_0$, $u_1$ be broken maps of type $\Gamma$ whose underlying
  unframed broken maps are the same. Then, $u_0$ and $u_1$ are related
  by an element of the framing symmetry group (from \eqref{eq:zne}).
\end{lemma}

\begin{proof}
  By Remark \ref{rem:nfr}, the number of possible framings of an
  unframed broken map is equal to
  $\prod_{e \in \Edge_\black(\Gamma)}\mu_e$.  This number
  is equal to the
  size of the framing symmetry group from \eqref{eq:zne} which has a
  free action on the set of broken maps.
\end{proof}

\begin{remark}\label{rem:symm-unfr}
  {\rm(Tropical symmetry for unframed broken maps)} A tropical
  symmetry on an unframed broken map is a tuple
  \[\ul g =(g_v)_{v \in \Ver(\Gamma)}\]
  that satisfies
  \begin{equation}
    \label{eq:unfr-symm}
  g_{v_+}g_{v_-}^{-1} \in \{z^{\cT(e)} : z \in \C^\times\} \quad \forall e=(v_+,v_-) \in \Edge_\black(\Gamma).  
  \end{equation}  
  The tropical symmetry group for unframed maps is equal to the
  quotient of the tropical symmetry group by the group of framing
  symmetries \eqref{eq:zne}.
\end{remark}

The next result shows that the identity component of the tropical
symmetry group is generated by tropical positions of vertices of the
tropical graph.  \index{Tropical symmetry}

\begin{lemma} {\rm(Tropical vertex positions generate tropical symmetries)}
  \label{lem:wtgen}
  \begin{enumerate}
  \item For a tropical graph $\Gamma$, the set of tropical vertex positions $\cW(\Gamma)$ 
 (defined in 
    \eqref{eq:tweights}) is convex.
  \item If a
    tropical graph $\Gamma$ has two distinct tropical vertex position maps 
    \[(\cT_0(v) , v \in \Ver(\Gamma) ) , \quad (\cT_1(v) , v \in \Ver(\Gamma)
    ) \]
  then the difference $\cT_1 - \cT_0$ generates a real-two-dimensional
  subgroup $\exp ( ( \cT_1 - \cT_0)( \cdot))$ of the tropical symmetry
  group $T_\trop(\Gamma)$ (defined in \eqref{eq:tsymg}).
  \item \label{part:idcpt} The subgroup 
    \begin{equation}
      \label{eq:Tidcpt}
    T_{\trop,\W}(\Gamma):=\bran{\exp((\cT_1-\cT_0)z) |
      \cT_0, \cT_1 \in \cW(\Gamma), z \in \C}  
    \end{equation}
    generated by tropical vertex position maps
    is the identity component of $T_\trop(\Gamma)$.
  \end{enumerate}
\end{lemma}
\begin{proof}
  For the first statement, if $\cT_0, \cT_1 \in \cW(\Gamma)$ are
  vertex position maps for a tropical graph $\Gamma$, then for any
  $t \in [0,1]$
  \[ (1-t)\cT_0 + t\cT_1 \in \cW(\Gamma) \]
  is also a vertex position map on $\Gamma$. Assume that $\cT_0$,
  $\cT_1$ are distinct, and that $l_{e,i} \in \R$ describe the
  difference between the tropical vertex positions in the sense that
\[\cT_i(v_+)-\cT_i(v_-)=l_{i,e}\cT(e), \quad i=0,1, e \in \Edge_\black(\Gamma). \]
Then the set of elements
\begin{equation}
  \label{eq:cTgen}
  g: \C \to T_\trop(\Gamma), \quad z \mapsto
    \begin{cases}
      e^{(\cT_1(v) -\cT_0(v))z}, \quad v \in \Ver(\Gamma)\\
      e^{(l_{1,e} - l_{0,e})z}, \quad e \in \Edge_{\black}(\Gamma)
    \end{cases}
\end{equation}
    is a non-trivial group of symmetries for broken maps modelled on
    $\Gamma$. In a similar way, one sees
    that the Lie algebra $\t_\trop(\Gamma)$ is equal to the vector
    space generated by differences of tropical vertex position maps
    $\cT_1-\cT_0$, $\cT_0, \cT_1 \in \W(\Gam)$, from which
    \eqref{part:idcpt} follows.
\end{proof}

The following is a corollary of Lemma \ref{lem:wtgen}.

\begin{corollary}
  {\rm(Rigid tropical graphs)} \index{Rigid! tropical graph} A
  tropical graph is rigid if and only if its tropical symmetry group
  $T_{\on{trop}}(\Gamma) $ is finite.
\end{corollary}

\begin{lemma}
  For any tropical graph $\Gamma$, the quotient
  $T_{\trop}(\Gamma)/T_{\trop,\W}(\Gamma)$ is finite and the
  group $T_\trop(\Gamma)$ has a finite number of connected components.
\end{lemma}

\begin{proof}
  The quotient $T_{\trop}(\Gamma)/T_{\trop,\W}(\Gamma)$ is discrete,
  because by Lemma \ref{lem:wtgen}, $T_{\trop,\W}(\Gamma)$ is the
  identity component of $T_{\trop}(\Gamma)$.  Furthermore, every
  connected component of $T_{\trop}(\Gamma)$ deformation retracts to a
  component of the maximal compact subtorus
  \[T_{\trop}(\Gamma)^{im}:=\{(\ul{g},\ul{z}) \in T_\trop(\Gamma) : g_v \in T_{P(v)}, |z_e|=1\}.\]
  Indeed, an element $(g,z) \in T_{\trop}(\Gamma)$ can be written as
  \[g_v=k_ve^{is_v}, \quad k_v \in T_{P(v)}, s_v \in \t_{P(v)}, \quad z_e=\theta_e e^{\alpha_e}, \quad \theta_e \in S^1, \alpha_e \in \R\]
  The tuple $(k,\theta):=((k_v)_v,(\theta_e)_e)$ is a tropical
  symmetry element in $T_{\trop}(\Gamma)^{im}$ which is connected to
  $(g,z)$ via the path
  \[ [0,1] \ni \tau \mapsto ((k_v e^{i\tau s_v})_v, (\theta_e e^{\tau \alpha_e})_e) \in T_{\trop}(\Gamma).\]
  The quotient $T_{\trop}(\Gamma)/T_{\trop,\W}(\Gamma)$ is finite
  because it is in bijection with the connected components of the
  compact subgroup $T_{\trop}(\Gamma)^{im}$.
\end{proof}

\begin{figure}[ht]
  \centering \scalebox{.8}{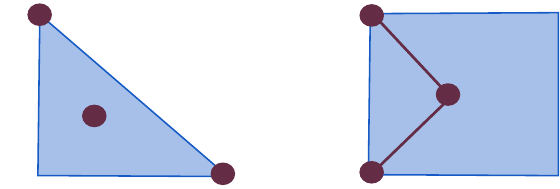}
  \caption{Rigid tropical graphs in Example \ref{eg:rigidfr}.}
  \label{fig:multfr}
\end{figure}

\begin{example} 
  \label{eg:rigidfr}
  The tropical graphs in Figure \ref{fig:multfr} are rigid, but have
  non-trivial tropical symmetry groups.  Suppose the tropical graph
  $\Gamma_1$ has edge {direction}s
  \[\cT(e_1)=(-1,-1), \quad \cT(e_2)=(-1,2), \quad
    \cT(e_3)=(2,-1).\]
  A tropical symmetry $(\ul{g},\ul{z})$ on $\Gamma_1$ satisfies the equations
  \[g_1=g_2=g_3=\Id, \quad g_1g_0^{-1}=z_{e_1}^{(-1,-1)}, \quad 
    g_2g_0^{-1}=z_{e_2}^{(-1,2)}, \quad 
    g_3g_0^{-1}=z_{e_3}^{(2,-1)}, \]
  and is therefore given by 
  \[g_0=(\om,\om), \quad z_{e_1}=z_{e_2}=z_{e_3}=\om,\]
  where $\om \in e^{2\pi i k/3}$ is a cube root of unity. Thus
\begin{equation} \label{three}
|T_\trop(\Gamma_1)|=3 .\end{equation}  
This tropical graph is similar to the example studied by
Abramovich-Chen-Gross-Siebert \cite[p51]{abram} and Tehrani
\cite[Section 6]{teh:deg}.  The tropical graph $\Gamma_2$ has edge
{direction}s
  \[\cT(e_1)=(-1,1), \quad \cT(e_2)=(-1,-1).\]
  By a similar calculation $|T_\trop(\Gamma_2)|=2$.
\end{example}

Recall from Section \ref{sec:mikh} that the multiplicity of a
Mikhalkin graph $\Gamma$ is the product of multiplicities of the
vertices of $\Gamma$, that is,
\[\mult(\Gamma):=\prod_{v \in \Ver(\Gamma)} \mult(v).\]
Assuming that all vertices are trivalent, the multiplicity of a vertex
$v$ of $\Gamma$ is the area of the parallelogram spanned by two of the
three edges incident on $v$. The next result, which was used in
Section \ref{sec:mikh}, shows that he multiplicity of the Mikhalkin
graph is equal to the size of the tropical symmetry group of the
augmented graph $\Gamma_\aug$.  We recall from the proof of Proposition
\ref{prop:tropbrok}, that $\Gamma_\aug$ is a tropical graph for the
polyhedral decomposition of $\P^2$ from Figure \ref{fig:stretchp2}.  We
also recall that
 \[\Ver(\Gamma_\aug)= \Ver(\Gamma) \cup \Ver_\to(\Gamma_\aug) \cup \Ver_1(\Gamma_\aug),\]
 where a vertex $v_e \in \Ver_1(\Gamma_\aug)$ corresponds to a leaf $e$ of the Mikhalkin graph,
 and each vertex in $\Ver_\to(\Gamma_\aug)$ has an interior marking, that is constrained to pass through a fixed point in the broken manifold.
 
\begin{lemma}{\rm(Mikhalkin multiplicity and tropical symmetry)}
  \label{lem:mikhsym}
  Let $\Gamma$ be a Mikhalkin graph all whose vertices are trivalent,
  and whose leaf edges have multiplicity $1$.  Let $\Gamma_\aug$ be
  the augmentation of $\Gamma$. Then,
  \[|T_\trop(\Gamma_\aug)|=\mult(\Gamma).\]
  Any two broken maps $u_1$, $u_2$ modelled on $\Gamma_\aug$ are
  related by a tropical symmetry $g \in T_\trop(\Gamma_\aug)$.
\end{lemma}
\begin{proof}
  A tropical symmetry of a broken map with tropical graph
  $\Gamma_\aug$ is a tuple
\[(\ul g, \ul z)=((g_v)_{v \in \Ver(\Gamma_\aug)}, (z_e)_{e \in \Edge(\Gamma_\aug)}), \quad g_v \in T_{P(v),\C}, \quad z_e \in \C^\times\]
satisfying
\begin{equation} \label{eq:mconn1} g_{v_+} g_{v_-}^{-1} =
    z_e^{\cT(e)} \quad \forall e=(v_+,v_-) \in \Edge(\Gamma_\aug),
  \end{equation}
  see Definition \ref{def:tsym}.  We count the number of solutions of
  \eqref{eq:mconn1}.  For vertices $v \in \Ver_\to(\Gamma_\aug)$
  containing markings, we must have $g_v=\Id$ in order to ensure that
  the evaluation of the marking satisfies the point constraint.  We
  solve for $(\ul g, \ul z)$, one vertex at a time, with vertex
  ordering respecting the marking orientation defined
  as part of the proof of Proposition \ref{prop:tropbrok}. 
  Consider a vertex $v \in \Ver(\Gam)$
  with incoming edges $e_1=(v_1,v)$ and $e_2=(v_2,v)$, and outgoing
  edge $e_3=(v,v_3)$.  Assuming the values of $g_{v_1}$, $g_{v_2}$ to
  be given, we solve the equations
\[g_v g_{v_1}^{-1}=z_{e_1}^{\cT(e_1)}, \quad g_v g_{v_2}^{-1}=z_{e_2}^{\cT(e_2)} \]
for $g_v \in T_{P_0,\C} \simeq (\C^\times)^2$, $z_{e_1}$,
$z_{e_2} \in \C^\times$ and show that the number of solutions is
$\mult(v)$. Use an integral basis of $\t_{P_0,\Z} \simeq \Z^2$ so that
\[ \cT(e_1)=(p_1,0), \quad \cT(e_2)=(q,p_2) \] 
for some integers
$p_1,p_2,q$. 
In this notation, the multiplicity $\mult(v)$ is $|p_1p_2|$.  We may
write $g_v=(g_v^0,g_v^1)$.  We solve in the order $g_v^1$, $z_{e_2}$,
$g_v^0$, $z_{e_1}$.  For a fixed value of $g_{v_1}$, $g_{v_2}$, there
are $p_1p_2$ solutions: There are $p_1$ solutions for $z_{e_2}$, and
for each of these solutions, there are $p_2$ solutions of $z_{e_1}$.
Thus we have shown that for a vertex
$v \in \Ver(\Gamma) \subset \Ver(\Gamma_\aug)$, the number of
solutions for $(g_v, z_{e_1},z_{e_2})$ is equal to the multiplicity of
$v$.  Finally, given a solution of \eqref{eq:mconn1} for vertices in
the set $\Ver(\Gamma) \cup \Ver_\to(\Gamma)$ and connecting edges, it
remains to solve \eqref{eq:mconn1} for $g_v$, $z_{e(v)}$ where
$v \in \Ver_1(\Gamma_\aug)$ is a univalent vertex corresponding to a
leaf of the Mikhalkin graph $\Gamma$ and $e(v) \in \Edge(\Gamma_\aug)$
is the edge incident on $v$.  There is a unique solution of
\eqref{eq:mconn1} for $g_v \in T_{P(v),\C}$ and
$z_{e(v)}\in \C^\times$ because the {direction} $\cT(e(v))$ of the edge
$e(v)$ is primitive, and $T_{P(v),\C}$ is orthogonal to
$T_{\cT(e(v)),\C}$ which contains the element $z^{\cT(e(v))}_{e(v)}$.
This shows that 
\[|T_\trop(\Gam_\aug)|=\prod_{v \in \Ver(\Gam)}\mult(v).\]

To prove the second statement of the Lemma, consider two unframed
broken maps $u$, $u'$ modelled on $\Gamma_\aug$.  For vertices
$v \in \Ver_\to(\Gamma_\aug)$ containing markings, $u_v=u_v'$. Indeed,
$u_v(z)=cz^\mu$ for some $\mu \in \Z^2$, and the point constraint
determines the constant $c \in (\C^\times)^2$ uniquely.  For all other
vertices $v$, there exist elements $g_v \in (\C^\times)^2$ satisfying
$g_v u_v=u_v'$ as follows: For any trivalent vertex $v$ we assume the
domain is parametrized so that $z_1$, $z_2$, $z_3 \in C_v$ are the
nodal points, and the edge incident at $z_i$ has {direction}
$\mu_i \in \Z^2$.  Then,
\[u_v=c\prod_{1 \leq i \leq 3}(z-z_i)^{\mu_i}, \quad u_v'=c'\prod_{1 \leq i \leq 3}(z-z_i)^{\mu_i},\]
for some $c,c' \in (\C^\times)^2$, and therefore $g_v u_v=u_v'$ for
some $g_v \in (\C^\times)^2$.  For a univalent vertex
$v \in \Ver_1(\Gamma_\aug)$ corresponding to an end of $\Gamma$, we
have $u_v(z)=c (z-z_0)^\mu$, $u_v'(z)=c'(z-z_0)^\mu$ for some
$c, c' \in (\C^\times)^2$, where $z_0$ is the nodal point and
$\mu \in \Z^2$ is the {direction} of the edge $e(v)$.
The matching condition at nodes for unframed broken maps implies that
for any edge $e=(v_+, v_-)$, the condition \eqref{eq:unfr-symm} for a
tropical symmetry element for unframed maps, namely
$g_{v_+}g_{v_-}^{-1} \in T_{\cT(e),\C}$, is satisfied.  Finally
consider (framed) broken maps $u$, $u'$ modelled on $\Gamma_\aug$. The
preceding discussion for unframed maps implies that, after applying a
tropical symmetry element, the unframed maps underlying $u$, $u'$ are
the same, and therefore, by Lemma \ref{lem:same-unfr}, $u$, $u'$ are
related by a framing symmetry, which is an element of the tropical
symmetry group $T_\trop(\Gamma)$.  This finishes the proof of the
Lemma.
\end{proof}

\chapter{Stabilizing divisors}\label{chap:stabdiv}

We use domain-dependent perturbations of the almost complex structure
to regularize the moduli space of pseudoholomorphic disks or spheres.
Moduli spaces of maps are defined modulo automorphisms of the domain
curve, and in order to define domain-dependent
perturbations on every component of the map, we would like the domain
components to be stable.  As in Cieliebak-Mohnke \cite{cm:trans}, we
consider domains with interior markings corresponding to points at
which the pseudoholomorphic map intersects a Donaldson divisor.
As a consequence, the almost complex structure can be perturbed to attain
regularity on all domain components. Of course, one could envision
using any of the current perturbation schemes to achieve virtual
fundamental chains; the stabilizing divisor approach is merely a
convenience.

\section{Stabilizing divisors in symplectic manifolds}
\label{sec:stab-manifold}
We recall results about stabilizing divisors for unbroken manifolds.
We consider divisors of Donaldson type, meaning Poincar\'e dual to
$k[\om]$ where $k \gg 0$ is a large integer.  To define and construct
Donaldson divisors we assume that all the symplectic forms are
rational.

\begin{definition} A symplectic orbifold $(X,\omega)$ is \em{
    rational} if $[\om] \in H^2(X,\Q)$, and \em{integral} if
  $[\om] \in H^2(X,\Z)$.  Given an integral form $\om$, a \em{
    prequantum line bundle} is a line-bundle-with-connection
  $\mL_X \to X$ whose curvature is
\[ \curv(\mL_X \to X) = (2\pi/i) \omega \in \Omega^2(X,i\R)  . \]   
A Lagrangian $L \subset X$ disjoint from the singular locus of $X$ is
\em{rational} if there exists a prequantum bundle $\mL_X$, an integer
$k$ and a flat section $s: L \to \mL_X^{\otimes k}$ of the restriction
$\mL_X^{\otimes k} | L$.
\end{definition}

A Donaldson divisor $D$ is a \em{stabilizing divisor} if there is an
almost complex structure $J$ for which $D$ is $J$-holomorphic, and $D$
intersects all $J$-holomorphic spheres and disks at sufficiently many
isolated points. In other words, the domains of $J$-curves can be
stabilized by treating intersection points with the divisor $D$ as
markings.
\begin{definition}\label{def:stabdiv}
{\rm(Stabilizing pair)}
  \begin{enumerate}
    \index{Stabilizing divisor}
    \index{Stabilizing pair}
  \item A \em{divisor} in $X$ is a symplectic suborbifold
    $D \subset X$ of real codimension $\codim(D) = 2$ such that
    locally if $X$ is the quotient of a symplectic manifold $X$ by a
    finite group $\Gamma$, then the divisor $D$ is the quotient of a
    $\Gamma$-invariant divisor in $X$.  A \em{Donaldson divisor} is a
    divisor $D$ that is Poincar\'e dual to $k[\om]$ for some large
    integer $k>0$, and $k$ is called the \em{degree} of the divisor
    $D$.
  \item An almost complex structure $J$ on $X$ is \em{adapted} to $D$
    if $D$ is a $J$-almost complex submanifold, that is,
    $J(TD)=TD$. The space of $D$-adapted almost complex structures is
    denoted by
    %
    \[ \J(X,D) = \{ J \in \J(X) \ : \ J(TD) = TD \} .\]
  \item Let $D \subset X -L$ be a divisor disjoint from the Lagrangian
    $L$.  For $E>0$, a $D$-adapted almost complex structure $J$ is
    \em{$E$-stabilizing} for the divisor $D$ if and only if
    \begin{enumerate}
    \item {\rm(No non-constant spheres)} $D$ does not contain any
      $J$-holomorphic \em{orbifold sphere} $u:C \to X$ with symplectic
      area $<E$ (by an orbifold sphere we mean that $C$ is an orbifold
      Riemann surface of genus zero), and
    \item {\rm(Sufficient intersections)} any $J$-holomorphic orbifold
      sphere in $X$ with symplectic area $<E$ has at least $3$
      distinct points of intersection with $D$, and any
      $J$-holomorphic disk with symplectic area less than $E$ has at
      least one intersection with $D$.
    \end{enumerate}
    A pair $(J,D)$ consisting of a divisor $D$ and an $\om$-tamed
    $J \in \J(X,D)$ is \em{stabilizing} if $J$ is $E$-stabilizing for
    all $E>0$.  We call the divisor $D$ in a stabilizing pair a \em{stabilizing divisor}.
  \end{enumerate}
\end{definition}

\begin{proposition} {\rm(Existence of a stabilizing pair,
    \cite[Section 4]{cw:traj}, \cite[Section 8]{cm:trans})}
  \label{prop:stab}
  Suppose $\om \in \Om^2(X)$ is a rational symplectic form on an
  orbifold $X$ and $L \subset X$ is a compact rational Lagrangian
  submanifold disjoint from the singular locus of $X$.  Then, there
  exists
  \begin{enumerate}
  \item a divisor $D \subset X - L$ such that $L$ is exact in $X-D$; and
  \item a tamed almost complex structure $J_0 \in \J(X,D)$ such
    that $(J_0,D)$ is a stabilizing pair and for every $E>0$ there is
    an open neighbourhood
    %
    \[\J(X,D;J_0,E) \subset \{J \in \J(X,D)\  : \ J|TD=J_0|TD\}\]
    of $J_0$ consisting of $E$-stabilizing almost complex structures
    adapted to $D$.
  \end{enumerate}
\end{proposition}

The proof of the Proposition is outlined below. The construction of a
stabilizing pair uses the following notion of a degree bound.

\begin{definition} \label{degbound} A constant $k_*>0$ is a \em{
    degree bound} for a tamed almost complex structure $J$ on
  $(X,\om)$ if for any $J$-holomorphic orbifold sphere $u:C \to X$,
  \[c_1(u):=\int_{C}u^*c_1(TX) \leq k_* \om(u).\]
\end{definition}

\begin{lemma}{\rm(Existence of uniform degree bounds)} (\cite[Lemma 8.11]{cm:trans})
  \label{lem:unifdegbd}
  Let $J_0$ be a compatible almost complex structure on the compact
  symplectic orbifold $(X,\om)$.  For any $0<\eps<1$, there exists a
  constant $k_*$ which is a degree bound for all tamed almost complex
  structures $J$ satisfying $\Mod{J-J_0}_{C^0(X)}<\eps$.
\end{lemma}
The proof in \cite{cm:trans}, carried out for the manifold case,
extends verbatim to the case of orbifolds.  We reproduce the proof
since an adaptation of the technique is used in the case of broken
manifolds.
\begin{proof}[Proof of Lemma \ref{lem:unifdegbd}]
  Let $\gamma \in \Om^2(X)$ be a closed two-form in the class of the
  first Chern class $c_1(TX) \in H^2(X)$. Using the norm
  $|v|^2:=\om(v,J_0v)$ on $TX$, and a $C^0$ norm on forms, for any
  $J \in B_\eps(J_0)$ and $v \in TX$, we have the bounds
  \begin{align*}
    \gamma(v,Jv) &\leq \Mod{\gamma}( 1 + \Mod{J-J_0})|v|^2 \leq \Mod{\gamma}(1+\eps)|v|^2  ,\\
    \om(v,Jv) &\geq (1 - \Mod{J-J_0})|v|^2 \geq  (1-\eps)|v|^2.
  \end{align*}
  Define $k_*:=\frac{1+\eps}{1-\eps}$.
  We have
  \[\gamma(v,Jv) \leq k_* \om(v,Jv). \]
  For a $J$-holomorphic orbifold sphere $u:C \to X$ this implies
  $c_1(u) \leq k_* \om(u)$.
\end{proof}

\begin{remark}\label{rem:whydeg} {\rm(Using the degree bound for
    stabilizing)}
  The degree bound on almost complex structures is related to the
  stabilizing property in Definition \ref{def:stabdiv} as follows.  In
  the manifold case, by \cite[Lemma 8.13]{cm:trans}, if a divisor $D$
  is Poincar\'e dual to $k[\om]$, and a $D$-adapted tamed almost
  complex structure $J$ has a degree bound $k_*$ satisfying
  \begin{equation}
    \label{eq:kkstar}
    k > 2\max\{k_*, k_* + \hh(\dim(X)) -2\},  
  \end{equation}
  then the expected dimension of moduli space of $J$-holomorphic
  spheres that are not stabilized by $D$ is negative.  The same holds
  in the orbifold case also.  Indeed, for an orbifold sphere
  $u:C \to X$, in the Hirzebruch-Riemann-Roch formula for the expected
  dimension of the moduli space of maps, the Chern number $c_1(u^*TX)$
  is replaced by the first Chern number of the ``de-singularization''
  of $u^*TX$, and the latter is lesser than the former, see
  \cite[p24]{cr:gw}, \cite[Proposition 4.2.1]{cr:coh}.  Therefore if
  $J$ were chosen generically, $(J,D)$ would be a stabilizing pair.
\end{remark}

\begin{proof}
  [Outline of proof of Proposition \ref{prop:stab}]
  In the case when $X$ is a manifold, the adaptation
  of Donaldson's construction in Auroux-Gayet-Mohsen
  \cite{auroux:complement} produces an approximately holomorphic
  divisor in $X-L$.  By \cite[Theorem 3.6]{cw:traj}, if the
  Lagrangian $L$ is rational, then for any divisor $D$ in $X-L$
  produced by Auroux-Gayet-Mohsen \cite{auroux:complement}, $L$ is
  exact in $X-D$. (We refer the reader to \cite[Example 3.2]{cw:traj}
  for the definition of ``exactness''.)  Exactness allows the following
  relation of area to intersection numbers: If $\lag$ is exact in
  $X-D$, then, the intersection number of a disk
  $u:(C,\partial C) \to (X,L)$ with the divisor $D$ is proportional to
  the area of the disk: If $[D]^\dual=k[\om]$ for some $k \in \Z$,
  then,
  \begin{equation}
    \label{eq:mult}
    \#u^{-1}(D)=k\int_Cu^*\om.   
  \end{equation}
  See \cite[Lemma 3.4]{cw:traj}.  As a consequence, for any $J$ that is
  adapted to $D$, any non-constant $J$-holomorphic disk is
  automatically stabilized since it has at least one marked point.

  Donaldson's divisor construction is extended to the orbifold case by
  Gironella-Mu\~{n}oz-Zhou \cite{munoz:orbidiv}.  Since the Lagrangian
  is smooth the modifications of Auroux-Gayet-Mohsen
  \cite{auroux:complement} do not interfere with the arguments in
  Gironella-Mu\~{n}oz-Zhou \cite{munoz:orbidiv}.  The relation between
  symplectic area and the number of divisor intersections extends to
  the orbifold case if the Lagrangian does not contain orbifold
  singularities, the proof from \cite{cw:traj} carries over verbatim.
  
  To construct a stabilizing pair one starts with a preliminary almost
  complex structure $J^\pre$ that is $\om$-compatible.  For any
  $\eps>0$ there exists a constant $k_*:=k_*(\eps,J^\pre)$ that is a
  degree bound for any $\om$-tamed almost complex structure $J$ on $X$
  satisfying $\Mod{J-J^\pre}_{C^0} \leq \eps$. This fact gives enough
  wiggle room to find a stabilizing pair as follows: For any $\eps'>0$
  there is a constant $k_d(\eps')$ such that if the degree of the
  Donaldson divisor $D$ is $\geq k_d(\eps')$, then there is a tamed
  almost complex structure $J_1 \in \B_{\eps'}(J^\pre)$ that is
  $D$-adapted.  We choose the degree $k$ of the Donaldson divisor to
  be high enough that it satisfies the bound \eqref{eq:kkstar}, that
  is,
\[ k > 2\max\{k_*, k_* + \hh(\dim(X)) -2\}  \]
and $k \geq k_d(\eps/2)$. Then, there is an open subset of $\J(X,D)$
contained in $B_\eps(J^\pre)$, and for any
$J \in \J(X,D) \cap B_\eps(J^\pre)$, $k_*$ is a degree bound in the
sense of Definition \ref{degbound}.  Therefore by Remark
\ref{rem:whydeg}, for a generic $J_0$ in
$\J(X,D) \cap B_\eps(J^\pre)$, the pair $(J_0,D)$ is stabilizing for
all orbifold spheres.  See \cite[Section 8]{cm:trans} for details.
\end{proof}

\section{Cylindrical almost complex structures, without gluability}
Donaldson's construction of approximately holomorphic divisors
requires a rational symplectic manifold with a compatible almost
complex structure.  Broken manifolds with cylindrical almost complex
structures, as defined in Chapter \ref{chap:bsymp}, do not meet this
requirement because they do not have natural embeddings into
symplectic broken manifolds.  We define a modified version of
cylindricity, called \em{$\XX$-cylindricity} for almost complex
structures, so that symplectic broken manifolds possess
$\XX$-cylindrical almost complex structures.  We point out that an
$\XX$-cylindrical almost complex structure is not \em{gluable},
\index{Gluable} that is, it 
cannot 
be glued on the ends to yield
neck-stretched almost complex structures on $X^\nu$. Indeed, the
$\t$-inner product underlying an $\XX$-cylindrical almost complex
structure is not the same for all the components $\ol X_P$ of the
broken manifold $\XX$, and has the following form.

\index{Inner product! $\XX$-inner product}
\begin{definition}{\rm($\XX$-inner product)}
  \label{def:xx-metric}
  Let $g_Q^P: \t_Q \times \t_Q \to \Q$ be a collection of inner products for all $Q \subset P$, where $P \in \PP^0$ ranges over top-dimensional polytopes, that satisfy
  \begin{enumerate}
  \item {\rm(Restriction)} \label{item:restrict}
    $R \subset Q \implies g_R^P|\t_Q=g_Q^P$; and 
  \item {\rm(Orthogonality)} for $Q \subset P$, $\dim(Q)=0$, which is the intersection of facets $Q=\cap_{i=1}^nQ_i$ of $P$, the inner product $g_Q^P$ satisfies the condition that the outward normals to $Q_i$ are orthogonal. 
  \end{enumerate}  
\end{definition}

We define symplectic cylindrical structure maps on symplectic cut spaces $\{\ol X_P^\om\}_{P \in \PP}$ using $\XX$-inner products for the consistency condition.

\index{Cylindrical! X@ $\XX$-symplectic cylindrical structure}

\begin{definition}{\rm($\XX$-symplectic cylindrical structure maps)}
  \label{def:xxsympstr}
  For any top-dimensional polytope $P$, let $(\phi^P_Q)_{Q \subset P}$
  be a collection of symplectic cylindrical structure maps
  \begin{equation}
    \label{eq:xxsympcyl}
    U_{\ol X^\om_P}(\ol X^\om_Q) \xrightarrow{\phi^P_Q} ((\Cone_QP \times \ol Z_Q^\om)/\sim, 
    \om_{\tQ}), \quad \om_{\tQ}:=\om_{X_Q} + d\bran{\alpha_Q^P,
    \pi_{\t_Q^\dual}}
  \end{equation}
  on neighborhoods $U_{\ol X^\om_Q}(\ol X^\om_P) \subset \ol X^\om_P$ of $\ol X^\om_Q$, where
  \begin{itemize}
  \item $\ol Z_Q^\om \to \ol X^\om_Q$ is the $T_Q$-bundle from \eqref{eq:zpom}, 
  \item $\alpha_Q^P \in \Om^1(\ol Z_Q^\om,\t_Q)$ is a connection one-form, where the collection
  $(\alpha_Q^P)_{Q \subset P}$ satisfies the consistency condition
  \eqref{eq:consis} with respect to the $\XX$-inner product; and
\item $\pi_{\t_Q^\dual} : \t^\dual \to \t_Q^\dual$ is the
  projection on $\Cone_QP \subset \t^\dual$, and so,
  $\lan \alpha_Q, \pi_{\t_Q^\dual} \ran$ is a one-form on
  $(\ol Z_Q^\om \times \Cone_QP)$; 
\item for any facet $R \subset P$ with $Q \subseteq R$, 
     the equivalence $\sim$ mods out the boundary $\Cone_QR \subset \Cone_QP$
     by the action of $T_R \simeq S^1$; and
     \footnote{By Lerman's construction $(\Cone_QP \times \ol Z^\om_Q,
     \om_{\tQ})/\sim$ is a manifold resp. orbifold if $\Cone_QP$ is Delzant resp. simple.}
 \item $\ol X^\om_Q$ is identically mapped by $\phi_Q^P$ to
     $Q \times \ol Z^\om_Q/T_Q$.
  \end{itemize}
  The map $\phi_Q^P$ induces a projection map
  \begin{equation}
    \label{eq:pip-om}
    \pi_Q^{P} : U_{\ol X^\om_Q}(\ol X^\om_P) \to \ol X^\om_Q.  
  \end{equation}
\end{definition}

Finally, we define $\om_\XX$-compatible almost complex structures. In
addition to being $\om$-compatible on each piece, in neighborhoods
$U_{\ol X^\om_Q}(\ol X^\om_P)$ of boundary submanifolds these almost
complex structures are integrable on the fibers of the map $\pi_Q^P$.

\begin{definition}{\rm($\XX$-cylindrical almost complex
    structures)} \label{def:omxxcyl} Let $\XX$ be a broken symplectic
  manifold equipped with $\XX$-symplectic cylindrical structure maps
  as in \eqref{eq:xxsympcyl} and resulting projection maps
    \[\pi_Q^P: U_{\ol X^\om_Q}(\ol X^\om_P) \to \ol X^\om_Q \]
    on neighborhoods
    $U_{\ol X^\om_Q}(\ol X^\om_P) \subset \ol X^\om_P$ of
    $\ol X^\om_Q$ for $Q \subset P$.
  \begin{enumerate}
  \item {\rm($\XX$-cylindrical)} 
  \index{Almost complex structure! X@ $\XX$-cylindrical}
  An \em{$\XX$-cylindrical}
  almost complex structure $\JJ=(\JJ_P)_{P \in \PP}$ on $\XX$ consists 
  of a cylindrical almost complex structure $\JJ_P$ on each compactified component $\ol \XX_P$ of $\XX$, $P \in \PP$, 
  whose projection to $\ol X_P$ (as in \eqref{eq:dpij}) is denoted by $J_P$, and the collection $(J_P)_{P \in \PP}$ is such that 
  \begin{enumerate}
  \item {\rm(Restriction)}
    for any pair $Q \subset P$ of polytopes,
    $\ol X^\om_Q$ is a $J_P$-holomorphic submanifold of $\ol X^\om_P$,
    and $J_P|\ol X^\om_Q=J_Q$, and 
  \item {\rm(Cylindrical structure)} for any pair $Q \subset P$ of
    polytopes, the fibers of $\pi_Q^P$ are $\JJ$-holomorphic and $\JJ$
    is integrable on each of the fibers.
  \end{enumerate}
\item {\rm($\om_\XX$-compatibility and tamedness)} \label{part:omxxcyl2}
  \index{Almost complex structure! W@ $\om_\XX$-compatibile}
  \index{Almost complex structure! W@ $\om_\XX$-tamed}
  An $\XX$-cylindrical almost complex
  structure $\JJ=(\JJ_P)_{P \in \PP}$ is \em{$\om_{\XX}$-compatible} resp.
  \em{$\om_{\XX}$-tamed} if
  for any $P \in \PP$, $J_P$ is $\om_{X_P}$-compatible resp. $\om_{X_P}$-tamed. The space of
  $\om_{\XX}$-compatible resp. $\om_{\XX}$-tamed almost complex structures is denoted by
  \[\J^\cyl_{\om_\XX}(\XX) \quad \text{resp.} \quad  \J^\cyl_{\om_\XX,\tau}(\XX).\]
  \end{enumerate}
\end{definition}

\begin{remark}
  A collection of $\XX$-cylindrical almost complex structures can
  not be glued along the necks because the underlying $\t$-inner
  products and connection one-forms vary across cut spaces.
\end{remark}

\begin{lemma}\label{lem:omxx-cptble}
  An $\om_\XX$-compatible almost complex structure exists.
\end{lemma}

\begin{proof}
  Starting from a broken symplectic manifold $\XX$, we construct an $\XX$-inner product, $\XX$-symplectic cylindrical structure maps, and an $\om_\XX$-compatible almost complex structure $\JJ$.  We first construct these structures in the neighborhoods of vertices $Q \in \PP_{(0)}$. Consider a top-dimensional polytope $P \in \PP^{(0)}$ and a vertex $Q \in P$ which is the intersection of facets $Q_1, \dots Q_q \subset P$.
  The symplectic cylindrical structure map at the corner $Q$ is a 
   $T_Q$-equivariant symplectic embedding
  \[U_{X^\om_Q}\ol X^\om_P \xrightarrow{\phi_Q^P} (\C^q/\Gamma, \om_{\on{std}}),  \]
  where $\Gamma$ is a finite group acting on $\C^q$ ($\Gamma$ is
  trivial if $Q$ is a smooth corner of $P$ as in \eqref{eq:zspan}),
  $T_Q$ is the torus acting on $U_{X^\om_Q} \ol X^\om_P$ with an
  identification $T_Q=(S^1)^q/\Gamma$, and $(S^1)^q$ has a standard
  action on $\C^q$ that descends to a $T_Q$-action on $\C^q/\Gamma$.
  On $U_{X^\om_Q} \ol X^\om_P$, we define the almost complex structure $\JJ$ by pulling back via $\phi_Q^P$; the inner product $g_Q^P$ is automatically defined.  We may similarly define
  $\JJ$ for all vertex neighborhoods $U_{X^\om_Q}\ol X^\om_P$, $P \in \PP^{(0)}$, $Q \in \PP_{(0)}$ so that
   the $\XX$-inner product satisfies the \hyperref[item:restrict]{(Restriction)} 
   condition.

   With the $\XX$-inner product fixed, we proceed to extend the
   $\om_\XX$-compatible almost complex structure $\JJ$ to all of the
   broken manifold by induction on the face structure of the polytopes
   in the polyhedral decomposition. Consider a polytope
   $Q \in \PP \bs \PP^{(0)}$ that is not top-dimensional, and assume
   that an $\om_\XX$-compatible collection of almost complex
   structures
  \begin{equation}
    \label{eq:jjrp}
    \JJ| \left(\cup_{R \subset Q, P \in \PP^{(0)}: Q \subset P}U_{\ol X^\om_R}(\ol X^\om_P)  \right)
  \end{equation}
  is already defined.  First, we extend
  $\JJ|\cup_{R \subset Q}U_{\ol X^\om_R}(\ol X^\om_Q)$ to an
  $\om_{X_Q}$-compatible almost complex structure over $\ol X^\om_Q$.
  Next, we extend the definition of $\JJ$ to neighborhoods
  $U_{\ol X^\om_Q}(\ol X^\om_P)$ for any $P \supset Q$, $P \in \PP^{(0)}$ as
  follows.  Suppose $Q_1,\dots, Q_q \in \PP$ are facets of $P$ such
  that $\cap_iQ_i=Q$. Identify
  $T_Q \simeq \prod_{i=1}^q T_{Q_i} \simeq (S^1)^q$.
    Choose a symplectic cylindrical structure map $\phi_Q^P$ (as in
  \eqref{eq:xxsympcyl}) 
  which is consistent with the
  existing maps, and let $\pi_Q^P$ be the resulting projection map
  as in \eqref{eq:pip-om}. 
  From $\phi_Q^P$ 
  we obtain a  family of $T_Q$-equivariant symplectic embeddings
for $x \in \ol X_Q^\om$ 
  \[(\pi_Q^P)^{-1}(x) \xrightarrow{\phi_x} (\C^q/\Gamma,\om_{\on{std}})  \]
  where $\Gamma$ is a finite group and there is an identification $T_Q \simeq (S^1)^q/\Gamma$ so that
  $T_Q$ has a standard action on $\C^q/\Gamma$.
  The embeddings $\phi_x$ vary smoothly with $x \in Q$, and
  $(\pi_Q^P)^{-1}(x) \cap \ol X^\om_{Q_i}$ is mapped to the subspace
  $\{z_i=0\} \subset \C^q$.  The fibers of the projection
  $\pi_Q^P: U_{\ol X^\om_Q}(\ol X^\om_P) \to \ol X^\om_Q$ are equipped with the
  complex structure pulled back by $\phi_x$.  For any
  $y \in U_{\ol X^\om_Q}(\ol X^\om_P)$, on the $\om$-complement of
  $\ker(d\pi_Q^P)_y$ in the tangent space $T_y\ol X^\om_P$,   define
  $\JJ$ to be equal to $\JJ_{\ol X^\om_Q}(\pi_Q^P(y))$. 
  The definition of $\JJ$ in neighborhoods of $\ol X^\om_Q$ agrees with
  the existing definition in neighborhoods of $\ol X^\om_R$ for any
  $R \subset Q$. Indeed, for any $R \subset Q$, and top-dimensional
  polytope $P \supset Q$, the definition of the $\XX$-inner product
  implies that $\JJ|U_{\ol X^\om_Q}(\ol X^\om_P)$ agrees with
  $\JJ|U_{\ol X^\om_R}(\ol X^\om_P)$ from \eqref{eq:jjrp} in the overlaps. In
  the concluding step of the induction, for each top-dimensional
  polytope $P \in \PP^{(0)}$, we extend
  $\JJ|\cup_{Q \subset P}U_{\ol X^\om_Q}(\ol X^\om_P)$ to an
  $\om_{X_P}$-compatible almost complex structure on $\ol X^\om_P$.
\end{proof}

\section{Stabilizing divisors in broken manifolds}
\label{sec:stabbroke}

In this Section, we construct stabilizing divisors in a broken
manifold by a modification of Donaldson's construction.  A broken
divisor $\DD$ in a broken manifold $\XX$ is a Donaldson divisor $D_P$
in each cut space $\ol X_P$ of $\XX$ that is cylindrical in the ends
of $\ol X_P$.  Such submanifolds are constructed as approximately
holomorphic submanifolds with respect to a compatible almost complex
structure.

\begin{definition}\label{def:brokediv}
  {\rm(Broken Donaldson divisor)} \index{Stabilizing divisor!Broken
    stabilizing divisor} Let $(X,\PP)$ be a tropical Hamiltonian
  manifold, where on any symplectic cut space $\ol X_P^\om$,
  $P \in \PP$, the symplectic form $\om_{X_P}$ is rational.  Let
  $\XX:=\XX_\PP$ be the broken manifold in the sense of Definition
  \ref{def:brokenJ}.  For $k \gg 0$, a \em{broken Donaldson divisor}
  of degree $k$ on $\XX$ is a collection of codimension two
  submanifolds
  \[ \bD = \{ \DD_P \subset \XX_P, \quad P \in \PP \} \]
  such that
  \begin{enumerate}
  \item {\rm(Cylindricity on neck pieces)} for any $P \in \PP$, there is a codimension two submanifold
    \[
      D_P \subset X_P \quad  \text{such that} \quad \DD_P=\pi_P^{-1}(D_P),\]
    where $\pi_P : \XX_P \to X_P$ is the projection map;
  \item {\rm(Cylindricity on ends)}\label{item:cylends} for any pair
    of polytopes $Q \subset P$, the divisor $D_P \subset X_P$ is
    $Q$-cylindrical in the $Q$-cylindrical end $U_Q(X_P)$, that is,
    $D_P \cap U_Q(X_P) = (\pi_Q^P)^{-1}(D_Q)$, where the projection
    $\pi_Q^P: U_Q(X_P) \to X_Q$ is from \eqref{eq:pipq}; and
  \item {\rm(Poincar\'e dual)} for any map $F : X_P \to \ol X_P^\om$
    that respects the cylindrical structure at the ends of $X_P$ (in
    the sense of Remark \ref{rem:nosymp} \eqref{part:nosymp1}), the
    closure $\ol {F(D_P)} \subset \ol X_P^\om$ is a suborbifold that
    is Poincar\'e dual to $k[\om_{X_P}]$.
  \end{enumerate}
\end{definition}

In the rest of this section, we construct a stabilizing divisor in the
cut spaces of a broken manifold. The construction of the divisors is
via a slight modification of Donaldson's technique \cite{don:symp}.  We
give an outline of Donaldson's construction following Auroux
\cite{Auroux:remark}.  Let $(X,\om)$ be a symplectic manifold with a
compatible almost complex structure $J$.  Let $\mL_X \to X$ be a
Hermitian line-bundle with connection $\alpha$ over $X$ whose
curvature two-form $\curv(\alpha)$ satisfies
$\curv(\alpha)= (2\pi/i) \omega$.  Since our symplectic manifolds are
rational we may always assume this to be the case after taking a
suitable integer multiple of the symplectic form.  We will construct
approximately holomorphic sections $s_k$ of the line bundles $\mL_X^k$
for large $k$ that are transverse in the sense that
$|\delbar s_k| \ll |\partial s_k|$ on the zero set $s_k^{-1}(0)$. Then
for large enough $k$ the zero set $s_k^{-1}(0)$ is transversally cut
out, and is a divisor of $X$ with degree $k$. To study sections on the
bundle $\mL_X^k$, we use the metric
\[ g_k:=k\om(\cdot,J\cdot) \]  
on $X$. Under this metric, the effect of the non-integrability of $J$
becomes negligible as $k$ increases.  We define the notions of
\em{approximate holomorphicity} and \em{transversality}:
\begin{definition} {\rm (Asymptotically holomorphic sequences of
    sections)} Let $(X,\om)$ be a symplectic manifold with
  $\om$-compatible almost complex structure $J$ and a prequantum line
  bundle $\mL_X \to X$.  Let $(s_k)_{k \ge 0}$ be a sequence of
  sections of $\mL_X^k \to X$.
\begin{enumerate} 
\item The sequence $(s_k)_{k \ge 0}$ is \em{asymptotically
    holomorphic} if there exists a constant $C$ and integer $k_0$ such
  that for $k \ge k_0$,
\begin{equation} \label{asymhol}
 | s_k | + | \nabla s_k| + | \nabla^2 s_k | \leq C, \quad |\olp s_k| +
 | \nabla \olp s_k| \leq C k^{-1/2} .\end{equation} 
\item The sequence $(s_k)_{k \ge 0}$ is \em{uniformly transverse} to
  $0$ if there exists a constant $\eta$ independent of $k$, and
  $k_0 \in \Z$ such that for any $x \in X$ and $k \geq k_0$ with
  $|s_k(x)| < \eta$, the derivative $\nabla s_k$ of $s_k$ is
  surjective at any point and satisfies $| \nabla s_k(x)| \ge \eta$.
\end{enumerate} 
In both definitions, the norms of the derivatives $\nabla s_k$ are
evaluated using the metric $g_k = k \omega( \cdot , J \cdot)$.
\end{definition} 

Donaldson's construction uses a family of asymptotically holomorphic
sections on $\mL_X^k$, that are concentrated at a given point in $X$ :
The Gaussian section centered at $x \in X$ is constructed by choosing
Darboux coordinates $(z_1,\dots,z_n) : (U_x,x) \to (\C^n,0)$ that are
approximately holomorphic in the following sense: 
Assuming that the Darboux coordinates map to a neighborhood $B \subset \C^n$ of the origin, and 
$\psi: \C^n \supset B \to X$ is the inverse map,
then the norm $|\psi^*J-i|$ of
the difference between the complex structures $\psi^*J$ and $i$ is at
most $c|z|$, where $c$ is a uniform constant independent of $x$, and
the derivative $|\nabla(\psi^* J -i)|$ is uniformly bounded.  We choose
a trivialization of the bundle $\mL_X$ so that the Hermitian
connection is $\sum_i(z_id\ol z_i - \ol z_i dz_i)$ plus terms of
higher order in $z, \ol z$. The Gaussian section is then defined as
\begin{equation}
  \label{eq:cutoff}
  \sig_{k,x}:=\beta_{k,x} \cdot e^{-k|z|^2}   
\end{equation}
on $U_x$, where $\beta_{k,x}$ is a cut-off function vanishing at a
$g_k$-distance of $k^{1/6}$ from $x$, and extended by zero outside
$U_x$.

The globalization process in Donaldson's construction uses the
following result:

\begin{lemma} \label{lem:qsard1} {\rm(Quantitative Sard's theorem,
    \cite[Theorem 20]{don:symp})} Suppose $0 < \delta < \frac 1 4$,
  and $f:B_+ \to \C$ is a function defined on a ball
  $B_+ \subset \C^{n}$ of radius $\frac {11}{10}$ that satisfies
  $\Mod{f}_{C^1} \leq \eta$, where
  $\eta:= \delta \log(\delta^{-1})^{-p}$. Then, there exists
  $w \in \C$, $|w| \leq \delta$ such that $f- w$ is $\eta$-transverse
  to $0$ over the interior ball $B$ of radius $1$.
\end{lemma}

\noindent Lemma \ref{lem:qsard1} is used to modify an approximately
holomorphic section in order to achieve uniform transversality in the
neighborhood of a given point: Given an asymptotically holomorphic
section $s_k: X \to \mL_X^k$ and a point $x \in X$, by applying Lemma
\ref{lem:qsard1} to the section
$f_k:=s_k/\sig_{k,x} : B_{g_k}(x,\frac {11}{10}) \to \C$, we obtain a
constant $w_k \in \C$ such that $f_k - w_k$ is uniformly transverse to
the zero section. As a consequence, the difference $s_k- w_k \sig_{k,x}$
is also uniformly transverse to the zero section in a neighborhood of
$x$.

Uniform transversality on the entire manifold is obtained by
iteratively adding successively smaller contributions, so that the
transversality over previous neighborhoods is not disturbed. In
particular, Donaldson's construction chooses a lattice of points
$\Lam_k \subset X$ such that $X$ is covered by the unit balls
$B_{g_k}(x,1)$, $x \in \Lam_k$; and partitions the lattice into $N$
sets (where $N$ is $k$-independent)
\[\Lam_k=I_1^k\cup \dots \cup I_N^k\]
such that any two points in a set $I_j^k$ are separated by a uniform
$g_k$-distance guaranteeing the following: For any $I_j^k$, the
constants $w_{k,x}$ for the balls centered at points $x \in I_j$ can
be chosen simultaneously and independently of each other.  If
$j_1<j_2$, the constants $w_{k,x} \in \C$ for $x \in I^k_{j_2}$ are
chosen to be small enough to not break the transversality in the balls
belonging to $I_k^{j_1}$.  We start with the zero section, and run the
iteration described above.  The iteration terminates in $N$ steps,
which is $k$-independent, and the resulting section is a sum
\[\sig_k:=\sum_{x \in \Lam_k} w_{k,x} \sig_{k,x}.\]
that is uniformly transverse on $X$ 
for large $k$.

\begin{remark}\label{rem:orbdiv3}
  In the case of a symplectic orbifold $(X,\om)$, we consider the
  pre-quantum orbibundle $\mL_X$, and a sequence of sections
  $s_k: X \to \mL_X^k$ for $k$ such that $\mL_X^k$ is a line bundle on
  $X$. Notions of asymptotic holomorphicity and transversality extend
  naturally to the orbifold case and are defined in
  Gironella-Mu\~{n}oz-Zhou \cite{munoz:orbidiv}.
\end{remark}

Next, we give a preliminary result constructing a gluable cylindrical
almost complex structure on the broken manifold $\XX$, that will be
used to construct Donaldson divisors.  We start with compatible almost
complex structures on the symplectic cut spaces (as in Definition
\ref{def:omxxcyl} and constructed in Lemma \ref{lem:omxx-cptble}), and
produce gluable almost complex structures by twisting the inner
product on $\t$ in a truncation of the cylindrical end.

\begin{lemma}\label{lem:tamingmap}
  {\rm(Taming map)} Let $\XX$ be a broken manifold with cylindrical
  end $U_Q(X_P) \subset X_P$ and projections
  $\pi_Q : U_Q(X_P) \to X_Q$ for each pair $Q \subset P$, and an
  $\om_\XX$-compatible almost complex structure $\JJ=(\JJ_P)_P$ with projections $J_P$, $P \in \PP$, as in
  Definition \ref{def:omxxcyl}.  Then, there exist
  \begin{enumerate}[label*=(\arabic*)]
  \item truncated cylindrical ends
    \[U_Q''(X_P) \subset U_Q'(X_P) \subset U_Q(X_P),\]
  \item an almost complex structure $\JJ'=(\JJ_P')_P$ on $\XX$ with projections $(J_P')_P$, that is
    strongly compatible (as in Definitions \ref{def:brokenJ} \eqref{part:brokenJ-cyl}, \ref{def:cylneckst}
    \eqref{part:tamefib}) and $X$-cylindrical (Definition \ref{def:brokenJ} \eqref{part:brokenJ-cyl}) 
    on the ends
    $U_Q''(X_P)$ for all $Q \subset P$, and
  \item  maps
    \(\phi_P : X_P \to \ol X_P^\om\)
  \end{enumerate}
  such that
  \begin{enumerate}
  \item on $X_P \bs U_Q'(X_P)$, $\phi_P$ is a diffeomorphism onto $X_P^\om$,  and
    $\phi_P^*J_P=J_P'$, and
  \item on $U_Q'(X_P)$, $\phi_P$ is
    projection to $X_Q$, that is, 
    $\phi_P|U_Q'(X_P)=\phi_Q \circ \pi_Q$.
  \end{enumerate}
\end{lemma}
\begin{proof}
  Starting from an $\XX$-inner product on the cylindrical end, we
  define domain-dependent inner products that are gluable near the
  infinite ends.  Choose truncated cylindrical ends $U_Q'(X_P)$,
  $U_Q''(X_P)$ for all pairs $Q \subset P$, and choose
  domain-dependent inner products
  \[g_Q^P : U_Q(X_P) \to (t_Q \tensor t_Q)^\dual, \]
  each of which is
  \begin{itemize}
  \item constant and equal to the $X$-inner product $g_Q$ from
    \eqref{eq:idtt} on $U_Q''(X_P)$; and 
  \item is constant and equal to an $\XX$-inner product (Definition
    \ref{def:xx-metric}) $g_Q^P$ on $U_Q(X_P) \bs U_Q'(X_P)$.
  \end{itemize}

  On cylindrical ends, the maps $\phi_P$ are determined by maps on
  cones.  For example, for a pair $Q \subset P$, on the domain of $\phi_P$, we have
  \begin{equation}
    \label{eq:uqxp-embed}
    U_Q(X_P)=(0,\infty)^k \times (Z_Q/T_P), \quad U_Q'(X_P)=\Pi_{i=1}^k(\delta_i,\infty) \times (Z_Q/T_P),  
  \end{equation}
  where $k=\codim_P(Q)$, and on the target space of $\phi_P$, 
  \begin{equation}
    \label{eq:uqxp-om-embed}
    U_{\ol X_Q^\om}\ol X_P^\om = (\Pi_{i=1}^k(-\eps_i,0] \times (Z_Q/T_P))/\sim,   
  \end{equation}
  where $\sim$ mods out boundaries by circle actions as in
  \eqref{eq:xxsympcyl}. The map $\phi_P$ is coordinate-wise given by
  maps $(0,\infty) \to (-\eps_i,0]$ that are equal to $0$ on
  $[\delta_i,\infty)$, and strictly increasing on $(0,\delta_i)$.  On
  the complement of the cylindrical ends, the map $\phi_P$ is standard
  as in \eqref{eq:ipembed}.

  The almost complex structure $J_P'$ is defined to be the pullback
  $\phi_P^*J_P$ on the complement of $\cup_Q U_Q'(X_P)$.   The pullback 
  can be extended to the ends $\cup_QU_Q(X_P)$ by keeping the base almost
  complex structures $d\pi_Q(J_P)$ fixed on $U_Q(X_P)$, ensuring
  $X$-cylindricity on $U_Q''(X_P)$, and interpolating the inner
  products and connection one-forms in the intervening subset
  $U_Q'(X_P) \bs U_Q''(X_P)$.
\end{proof}

The following is the main result of the section.
\begin{proposition}
  {\rm(Construction of a broken divisor)}
  \label{prop:brokdiv}
  Let $\JJ_0$ be a cylindrical almost complex structure as constructed
  in Lemma \ref{lem:tamingmap}.  For any $\theta>0$, and a large
  enough $k \in \N$, there is a cylindrical Donaldson-type divisor
  $\DD \subset \XX$ that is $\theta$-approximately holomorphic, and
  such that on any cut space $\ol X_P$, the induced divisor
  $D_P \subset \ol X_P$ is Poincar\'e dual to $k[\om_{X_P}]$.
\end{proposition}

We refer the reader to \cite{don:symp} for the definition of
$\theta$-approximate holomorphicity.

\begin{proof}[Proof of Proposition \ref{prop:brokdiv}]
  The sections are constructed by running Donaldson's procedure
  simultaneously for all the manifolds in the set
  $\{\ol X_P\}_{P \in \PP}$.  Our modification of Donaldson's
  algorithm is limited to choosing appropriate Gaussian sections in
  order to ensure \hyperref[item:cylends]{(Cylindricity on ends)} is
  satisfied.  The step of achieving global transversality by applying
  Lemma \ref{lem:qsard1} is the same as the original algorithm, and
  therefore not discussed.

  The prequantum bundle on a cut space $\ol X_P$ is defined as the
  pullback by a map to the the symplectic cut space $\ol X_P^\om$.
  Lemma \ref{lem:tamingmap} constructs a map
  $\phi_P: \ol X_P \to \ol X_P^\om$ and produces an almost complex
  structure $J_P$ such that
  \begin{itemize}
  \item a truncated cylindrical end $U_Q'(X_P) \subset U_Q(X_P)$ is
    mapped by $\phi_P$ to $X_Q$ via a projection map,
  \item on the complement of the truncated cylindrical end, $\phi_P$
    is a diffeomorphism onto its image and $J_P$ is
    $\phi_P^*\om_{X_P}$-compatible, and
  \item $\phi_P|X_Q=\phi_Q$.
  \end{itemize}
  Assuming that $\mL_P \to \ol X_P^\om$ is a prequantum line bundle
  with connection, we will construct asymptotically holomorphic
  sections on the pullback bundle $\phi_P^*\mL_P \to X_P$. On
  $U_Q'(X_P)$, the connection on the pullback bundle is flat on the
  fibers of $\pi_Q$, and we will define sections that are constant in
  the fiber direction.

  The notions of asymptotically holomorphicity and transversality are
  defined with respect to a dilated metric on the cut spaces.  On any
  $X_P$, the metric $g$ is equal to $\om_{X_P}(\cdot, J_P \cdot)$ in
  the complement of the cylindrical ends $U_Q'(X_P)$, $Q \subset P$,
  and on the cylindrical end $g_P$ is a product metric on the
  fibration $U_Q(X_P) \to X_Q$ that extends the metric $g_Q$ on $X_Q$.
  For any $k$, define a dilated metric
  \[g_{k}:=kg.\]

  We describe a set of lattice points at which the Gaussian section
  for each tensor power of the given line bundle are centered.  Given
  $k \gg 0$, a set of lattice points
  \[\Lam_{k,P} \subset X_P \bs \cup_{Q \subset P}U_Q'(X_P)\]
  is defined so that it contains $Ck^{\dim(\ol X_P)}$ number of
  points, and there is a covering
  \[ X_P\bs \left( \cup_{Q \subset P} B_{g_k}(U_Q'(X_P),1) \right) =
    \left( \cup_{x \in \Lam_{k,P}}B_{g_k}(x,1) \right).
  \]
  (The set $\Lam_{k,P}$ is called a lattice because it resembles a
  lattice locally in Darboux coordinates.)  For any
  $p \in \Lam_{k,Q}$, we will write down Gaussian sections
  \[ \sig_{k,p,P} : \ol X_P \to \phi_P^*\mL_P\]
  centered at $p$ for all $P \supseteq Q$, and determine coefficients
  $w_{p,k} \in \C$ so that
  \[\sum_{p \in \Lam_Q, Q \subseteq P} w_{p,k} \sig_{k,p,P} : \ol X_P
    \to \phi_P^*\mL_P\]
  is asymptotically holomorphic and uniformly transverse on $\ol X_P$.
  The set $\Lam_k$ is partitioned into subsets $I_1,\dots, I_N$ where
  $N$ is independent of $k$ while satisfying the following: For any
  pair
  \[ x \in \Lam_{k,P} \bs \cup_{Q \subset P} \Lam_{k,Q}, \quad y \in
    \cup_{Q \subset P} \Lam_{k,Q} \]
  we have
  \[x \in I_\alpha, y \in I_\beta \implies \beta < \alpha, \]
  and at the $\alpha$-th iteration the coefficients $w_{k,p}$ are
  fixed for $p \in I_\alpha$.  As a result, coefficients of the
  Gaussian sections centered at lattice points on $X_Q$ are fixed
  before those centered at lattice points on $X_P$ if $Q \subset P$.

  We define Gaussian sections concentrated at the lattice points.
  Consider $p \in \Lam_{k,Q}$.  We
  wish to construct
  Gaussian sections $\sig_{k,p,P} : X_P \to \phi_P^*\mL_P^k$ for all
  $P \supseteq Q$ such that
  \[ \sig_{k,p,P}|X_Q=\sig_{k,p,Q} . \]
  We will first define $\sig_{k,p,Q}$ and then define $\sig_{k,p,P}$
  so that it is constant on the fibers of the projection
  $U_Q(X_P) \to X_Q$, and and decays at an exponential rate along the
  fiber $U_Q(X_P) \bs U_Q'(X_P) \to X_Q$.  The detailed construction
  of the sections is as follows: For any $P \supseteq Q$, choose
  Darboux coordinates $(z_1,\dots,z_{n(P)})$ on a neighborhood
  $U_p \subset \ol X_P^\om$ centered at $p$, where
  $n(P):=\hh\dim(X_P)$, and the first $n(Q)$ coordinates are Darboux
  coordinates on $X_Q^\om$.  Choose a trivialization of $\mL_P^k|U_P$
  for which the connection is $k\sum_i (\ol z_i dz_i - z_i d\ol z_i)$.
  Let
  \[s_{k,p,P}(z):=\beta_{k,z} e^{-k|z|^2}:\ol X_P^\om \to \mL_P^k,\]
  where $\beta_{k,z}$ is a cut-off function vanishing at a
  $g_k$-distance of $k^{1/6}$ from $x$, and extended by zero outside
  $U_p$, and let
  \[\sig_{k,p,P}:=\phi_P^*s_{k,p,P} : X_P \to \phi_P^*\mL_P^k.\]
  As a consequence, $\sig_{k,p,Q}$ is an ordinary Gaussian section.
  For $P \supset Q$, the section $\sig_{k,p,P}$ is constant on the
  fibers decays outside $U_Q'(X_P)$.

  The globalization process consists of finding coefficients
  $\{w_p \in \C\}_{p \in \Lam_k}$ such that
  \[ \sig_{k,P}:=\sum_{p \in \Lam_{k,P}} w_p \sig_{k,p,P} \]
  is a uniformly transverse sequence of sections for each $P$. The
  proof of globalization carries over from \cite{don:symp}.  The only
  new feature is to determine each coefficient $w_p$ in a
  $P$-independent way. This can be done by Lemma \ref{lem:qsard} which
  is a modification of Lemma \ref{lem:qsard1}.

  Finally, for any pair $Q \subset P$, the zero set
  $\sig_{k,P}^{-1}(0)$ is $Q$-cylindrical, because a lattice point
  $p \in \Lam_{k,P}$ lies in the complement of the $Q$-cylindrical
  end, and the section $\sig_{k,p,P}$ is supported in
  $B_{g_k}(p,k^{1/6})$, which is equal to $B_g(p,k^{-1/3})$.
  Therefore, there is a truncated cylindrical end
  $U_Q''(X_P) \subset U_Q'(X_P)$ on which, for large enough $k$,
  $\sig_{k,P}^{-1}(0)$ is $Q$-cylindrical, since $U_Q''(X_P)$ is
  disjoint from $B_g(p,k^{-1/3})$ for all $p \in \Lam_{k,P}$.
\end{proof}

\begin{lemma} \label{lem:qsard} {\rm(Quantitative Sard's theorem)}
  Given a tuple of positive integers $(n_1,\dots,n_k) \in \Z^k$ there
  is an integer $p$ for which the following is satisfied.  Suppose
  $0 < \delta < \frac 1 4$, and $f_i:B_+^{n_i} \to \C$ is a set of $k$
  functions on balls $B_+^{n_i} \subset \C^{n_i}$ of radius
  $\frac {11}{10}$ that satisfy $\Mod{f_i}_{C^1} \leq \eta$, where
  $\eta:= \delta \log(\delta^{-1})^{-p}$. Then, there exists
  $w \in \C$, $|w| \leq \delta$ such that $f_i- w$ is
  $\eta$-transverse to $0$ over the interior ball $B^{n_i}$ of radius
  $1$ for each $i$.
\end{lemma}

The case $k=1$ is Theorem 20 of \cite{don:symp}, and is stated as
Lemma \ref{lem:qsard1}.  The proof in the case $k=1$ is by bounding
the size of the image $f(B_+)$ in the range.  For a finite $k$, the
volume is multiplied by a constant, and the proof in \cite{don:symp}
can be replicated by altering the constants.

\begin{remark}\label{rem:orbdiv4}
  The construction of a broken divisor in Proposition
  \ref{prop:brokdiv} extends to the orbifold case. Indeed, in the
  orbifold adaptation of the Donaldson construction in
  Gironella-Mu\~{n}oz-Zhou \cite{munoz:orbidiv}, the new features are
  the choice of an appropriate lattice compatible with the
  stratification of the orbifold; and adjusting the uniform
  transversality constants. Both of these features are compatible with
  the modifications we have introduced in the proof of Proposition
  \ref{prop:brokdiv}.
\end{remark}

\begin{figure}[ht]
  {
\begingroup%
  \makeatletter%
  \providecommand\color[2][]{%
    \errmessage{(Inkscape) Color is used for the text in Inkscape, but the package 'color.sty' is not loaded}%
    \renewcommand\color[2][]{}%
  }%
  \providecommand\transparent[1]{%
    \errmessage{(Inkscape) Transparency is used (non-zero) for the text in Inkscape, but the package 'transparent.sty' is not loaded}%
    \renewcommand\transparent[1]{}%
  }%
  \providecommand\rotatebox[2]{#2}%
  \newcommand*\fsize{\dimexpr\f@size pt\relax}%
  \newcommand*\lineheight[1]{\fontsize{\fsize}{#1\fsize}\selectfont}%
  \ifx\svgwidth\undefined%
    \setlength{\unitlength}{164.65382662bp}%
    \ifx\svgscale\undefined%
      \relax%
    \else%
      \setlength{\unitlength}{\unitlength * \real{\svgscale}}%
    \fi%
  \else%
    \setlength{\unitlength}{\svgwidth}%
  \fi%
  \global\let\svgwidth\undefined%
  \global\let\svgscale\undefined%
  \makeatother%
  \begin{picture}(1,0.1080029)%
    \lineheight{1}%
    \setlength\tabcolsep{0pt}%
    \put(0,0){\includegraphics[width=\unitlength,page=1]{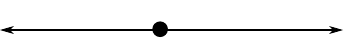}}%
    \put(0.21374648,0.05728465){\color[rgb]{0,0,0}\makebox(0,0)[lt]{\lineheight{1.25}\smash{\begin{tabular}[t]{l}$P_+$\end{tabular}}}}%
    \put(0.73598903,0.07108834){\color[rgb]{0,0,0}\makebox(0,0)[lt]{\lineheight{1.25}\smash{\begin{tabular}[t]{l}$P_-$\end{tabular}}}}%
    \put(0.44192833,0.0698634){\color[rgb]{0,0,0}\makebox(0,0)[lt]{\lineheight{1.25}\smash{\begin{tabular}[t]{l}$P_0$\end{tabular}}}}%
  \end{picture}%
\endgroup%
}
  \caption{The polyhedral decomposition of a single cut.} 
  \label{fig:polytope0}
\end{figure}

 \begin{figure}[ht]
    {
\begingroup%
  \makeatletter%
  \providecommand\color[2][]{%
    \errmessage{(Inkscape) Color is used for the text in Inkscape, but the package 'color.sty' is not loaded}%
    \renewcommand\color[2][]{}%
  }%
  \providecommand\transparent[1]{%
    \errmessage{(Inkscape) Transparency is used (non-zero) for the text in Inkscape, but the package 'transparent.sty' is not loaded}%
    \renewcommand\transparent[1]{}%
  }%
  \providecommand\rotatebox[2]{#2}%
  \newcommand*\fsize{\dimexpr\f@size pt\relax}%
  \newcommand*\lineheight[1]{\fontsize{\fsize}{#1\fsize}\selectfont}%
  \ifx\svgwidth\undefined%
    \setlength{\unitlength}{106.75564698bp}%
    \ifx\svgscale\undefined%
      \relax%
    \else%
      \setlength{\unitlength}{\unitlength * \real{\svgscale}}%
    \fi%
  \else%
    \setlength{\unitlength}{\svgwidth}%
  \fi%
  \global\let\svgwidth\undefined%
  \global\let\svgscale\undefined%
  \makeatother%
  \begin{picture}(1,0.5904632)%
    \lineheight{1}%
    \setlength\tabcolsep{0pt}%
    \put(0,0){\includegraphics[width=\unitlength,page=1]{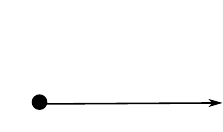}}%
    \put(0.04635769,0.32638066){\color[rgb]{0,0,0}\makebox(0,0)[lt]{\lineheight{1.25}\smash{\begin{tabular}[t]{l}$P_1$\end{tabular}}}}%
    \put(0.47925728,0.01767325){\color[rgb]{0,0,0}\makebox(0,0)[lt]{\lineheight{1.25}\smash{\begin{tabular}[t]{l}$P_2$\end{tabular}}}}%
    \put(-0.0062204,0.10449243){\color[rgb]{0,0,0}\makebox(0,0)[lt]{\lineheight{1.25}\smash{\begin{tabular}[t]{l}$P_0$\end{tabular}}}}%
    \put(0,0){\includegraphics[width=\unitlength,page=2]{polytope2.pdf}}%
    \put(0.58925656,0.36186388){\color[rgb]{0,0,0}\makebox(0,0)[lt]{\lineheight{1.25}\smash{\begin{tabular}[t]{l}$P_3$\end{tabular}}}}%
  \end{picture}%
\endgroup%
}
\caption{Multiple cut case.}
\label{fig:polytope2} 
\end{figure}

\begin{remark}
  {\rm(On the construction of broken stabilizing divisors)} This is a
  technical remark justifying the proof technique of Proposition
  \ref{prop:brokdiv} wherein we construct asymptotically holomorphic
  sequences of sections $\sig_{k,P}$ simultaneously for all polytopes
  $P \in \PP$ rather than construct one sequence $(\sig_{k,P})_k$ at a
  time.  We point out that the latter approach, which appears cleaner,
  is taken in \cite{cw:flips} for the case of single breaking.
  Consider a single cut with polyhedral decomposition shown in Figure
  \ref{fig:polytope0}.  Given a sequence of sections
  \[ (\sig_{k,{P_0}}: \ol X_{P_0} \to \mL_{\ol X_{P_0}}^k)_k, \]
  \cite[Lemma 7.15]{cw:flips} constructs an extension of the sequence
  of sections to 
  \[ (\sig_{k,{P_\pm}}: \ol X_{P_\pm} \to
    \mL_{\ol X_{P_\pm}}^{k})_k . \]   
  The first step in the construction is to consider the extension
  $\sig_{k,{P_0}} e^{-k|x|^2}$ and then turn on the contributions from
  Gaussians centered away from the divisor $\ol X_{P_0} \subset \ol X_{P_\pm}$
  so that the sections become transverse. If we apply this approach to
  spaces corresponding to the
  polytopes shown in Figure \ref{fig:polytope2},
  we would construct a sequence in the order $\sig_{k,{P_0}}$, then
  $\sig_{k,{P_1}}$, $\sig_{k,{P_2}}$ and then finally
  $\sig_{k,{P_3}}$.  We would like the sequence $\sig_{k,{P_3}}$ on
  $\ol X_{P_3}$ to be an extension of $\sig_{k,{P_0}}$ , $\sig_{k,{P_1}}$
  and $\sig_{k,{P_2}}$. Therefore, we would like
  to start Donaldson's globalization iteration with a sequence of Gaussian sections that are equal to
  $\sig_{k,{P_0}} e^{-k(|z_1|^2 + |z_2|^2)}$ in a neighborhood of
  $\ol X_{P_0}$, equal to $\sig_{k,{P_1}} e^{-k|z_1|^2}$ in a neighborhood
  of $\ol X_{P_1}$, and equal to $\sig_{k,{P_2}} e^{-k|z_2|^2}$ in a
  neighborhood of $\ol X_{P_2}$. The approach fails because the three
  sections do not agree on overlaps. Indeed, to define the sequence
  $\sig_{k,{P_1}}$, we would have used $\sig_{k,{P_0}} e^{-k|z_2|^2}$
  as a starting sequence, but then these sections would have been
  modified when contributions from Gaussians in nearby balls are
  turned on to achieve transversality.
\end{remark}

\begin{lemma}\label{lem:jjneck}
  Let $\JJ'$ be the $X$-cylindrical almost complex structure from
  Lemma \ref{lem:tamingmap}. There is a diffeomorphism
  $\phi : X \to X$ that preserves the tropical structure (in the sense
  of Remark \ref{rem:trop-preserve} below) and such that gluing $\JJ'$
  along the ends produces a sequence $(J')^\nu$ of almost complex
  structures on $X$, that is neck-stretching with respect to the
  symplectic form $\phi^*\om$, and locally  $\phi^*\om$-compatible.
\end{lemma}
We recall that ``$(J')^\nu$ being neck-stretching with respect to the
symplectic form $\phi^*\om$'' means that there is a symplectic
cylindrical structure (Definition \ref{def:sympcylstr}) on
$(X,\phi^*\om)$ which gives rise to neck-stretched manifolds
(Definition \ref{def:neckstr}) on which $(J')^\nu$ is cylindrical.
\begin{remark}\label{rem:trop-preserve}
  A diffeomorphism $\phi : X \to X$ preserves the tropical structure
  $\PP$ on $X$ if there are neighborhoods $U_P \subset X$ of
  $\Phinv(P)$ on which $\phi$ is $T_P$-equivariant and
  $\Phi_P \circ \phi=\Phi_P$ where $\Phi_P : U_P \to \t_P^\dual$ is the
  projection of the tropical moment map $\Phi$ to $\t_P^\dual$.
\end{remark}

\begin{proof}

  We recall that in Lemma \ref{lem:tamingmap}, we start out with a
  $\XX$-cylindrical almost complex structure $\JJ$, which is locally
   $\om$-compatible with cylindrical end $U_Q(X_P)$ for any pair $Q \subset P$, $P \in \PP^{(0)}$.  
  For any top-dimensional polytope $P \in \PP$,
  the corresponding symplectic $\XX$-cylindrical structure on $X_P^\om$ is
  given by maps
  \[ U_{\ol X^\om_P}(\ol X^\om_Q) \xrightarrow{\phi^P_Q} ((\Cone_QP
    \times \ol Z_Q^\om)/\sim, \om_{\tQ}), \quad \om_{\tQ}:=\om_{X_Q} +
    d\bran{\alpha_Q^P, \pi_{\t_Q^\dual}},\] for all $Q \subset P$
  where the various terms are as in \eqref{eq:xxsympcyl}.

  As a first step of the proof, we produce a symplectic
  $X$-cylindrical structure on symplectic cut spaces for which the almost complex
  structure $\JJ'$ is  locally compatible.  We recall from the
  proof of Lemma \ref{lem:tamingmap} that $\JJ'$ is obtained from
  $\JJ$ to changing the connection one-form to be $X$-cylindrical in
  the truncated cylindrical end $U_Q''(X_P)$ and interpolating between
  the two in the intervening subset $U_Q''(X_P) \bs U_Q''(X_P)$, where $Q \subset P$ and $P \in \PP^{(0)}$. 
  We may alter the connection one-form underlying the symplectic
  cylindrical structure in a similar way to obtain the following: For
  any top-dimensional $P \in \PP$, there is a symplectic form
  $\om_{P,1}$ on $\tX_P^\om$, which is equal to $\om_{X_P}$ in the
  complement of $U_{\ol X_Q^\om}(\ol X_P^\om)$, and on
  $U_{\ol X_Q^\om}(\ol X_P^\om)$,
  \begin{equation}
    \label{eq:om1end}
    ((\phi_Q^P)^{-1})^*\om_{P,1}=\om_{X_Q} + d\bran{\alpha^P_{Q, \t_Q^\dual}, \pi_{\t_Q^\dual}},   
  \end{equation}
  where $\alpha^P_{Q, \t_Q^\dual}$ is a $T_Q$-connection one-form on
  the $T_Q$-bundle $(\Cone_QP \times \ol Z_Q^\om)$, and for which
  $\JJ'$ is locally $\om_{1,P}$-compatible.  (The subscript $\t_Q^\dual$
  denotes dependence on the coordinate in
  $\Cone_QP \subset \t_Q^\dual$.)  Comparing \eqref{eq:om1end},
  \eqref{eq:xxsympcyl}, we see that the forms $\om_{Q,1}$, $\om_{X_Q}$
  are cohomologous. Furthermore, assuming that the cylindrical end is
  small enough the family
  $\{(1-t) \om_{Q,1}+t\om_{X_Q} : t \in [0,1]\}$ consists of
  symplectic forms on $\ol X_P$. Applying Moser's theorem, we obtain a
  symplectomorphism
  \[\psi_P: (\ol X_P^\om, \om) \to (\ol X_P^\om, \om_1).\]
  Note that $\phi_{Q,1}^P:=\phi_Q^P \circ \psi_P$ is a symplectic
  $X$-cylindrical structure.

  We glue the symplectic cut spaces $X_P^\om$ using the symplectic
  $X$-cylindrical structure, and call the resulting symplectic
  manifold $(X,\om_1)$. The almost complex structures
  $\{(J')^\nu\}_\nu$ obtained by gluing $\JJ'$ along the ends is
  locally $\om_1$-compatible.
  It remains to prove that the manifold $(X,\om_1)$ is
  symplectomorphic to the original manifold $(X,\om)$.  By Proposition
  \ref{prop:symcyl}, there is a symplectic $X$-cylindrical structure
  $\{\phi_P\}_P$ on $(X,\om)$. Applying the multiple cut $\PP$ to
  $(X,\om)$, we obtain cut spaces $\ol X_P^\om$ equipped with an
  symplectic $X$-cylindrical structure denoted by
  $\phi_{Q,0}^P$. There is a family of $X$-cylindrical structures
  $\{\phi_{Q,t}^P\}$ connecting $\phi_{Q,0}^P$ and
  $\phi_{Q,1}^P$. Gluing along each of these produces a family
  $(X,\om_t)$ with cohomologous symplectic forms. The Lemma follows
  from Moser's theorem.
\end{proof}
\section{Stabilizing pairs in neck-stretched manifolds}
\label{sec:stabpair}

In this Section, we prove the existence of a stabilizing pair
$(\JJ_0,\DD)$ on a broken manifold such that the family
$(J^\nu,D^\nu)$ obtained by gluing the pair $(\JJ_0,\DD)$ consists of
stabilizing pairs on neck-stretched manifolds.  We recall from Section
\ref{sec:stab-manifold} that in a smooth symplectic manifold
$(X,\om_X)$, a stabilizing pair $(J,D)$ consists of a Donaldson-type
divisor $D \subset X$ and a $\om_X$-tamed almost complex structure $J$
that is adapted to $D$ and for which any non-constant $J$-holomorphic
sphere intersects $D$ at least three points and is not contained in
$D$.  We start by defining the analogous notion for broken manifolds.

\begin{definition}
  {\rm(Adapted broken almost complex structure)}
\label{def:j-adapt}
  Given a broken
  divisor $\bD \subset \XX$ as in Definition \ref{def:brokediv}, we denote by 
  %
  \[ \J^\cyl(\XX, \bD) = \{ \JJ \in \J^\cyl(\XX)\ : \  \JJ(T\DD_P)=T\DD_P \text{ on $\ol \XC_P$ } \ \forall P
  \in \PP  \} \]
    \index{J@$\J^\cyl$!$\J^\cyl(\XX,\DD)$}
  the space of cylindrical almost complex structures that are adapted
  to
  $\bD$.
\end{definition}

A \em{stabilizing pair} in a broken manifold consists of a broken Donaldson divisor and a broken almost complex structure such that the restriction of the pair to any cut space is a stabilizing pair in the sense of Definition \ref{def:stabdiv}.
\index{Stabilizing pair!in a broken manifold}
\begin{definition}
  {\rm(Stabilizing pair in a broken manifold)}
  \label{def:stab-brok}
  Let $\bD \subset \XX$ be a broken divisor as in Definition
  \ref{def:brokediv} which is disjoint from the Lagrangian
  $L \subset \XX$.  For $E>0$, an adapted almost complex structure
  $\JJ \in \J^\cyl(\XX,\bD)$, is \em{$E$-stabilizing} if for every
  polytope $P \in \PP$, the almost complex structure $\JJ|\ol X_P$ is
  $E$-stabilizing in $(\ol X_P,\om_{X_P}, \cD_P)$. A pair $(\JJ, \bD)$
  is \em{stabilizing} if $\JJ$ is locally strongly tamed and is
  $E$-stabilizing for all $E>0$.  We call a divisor $\DD$ appearing in
  a stabilizing pair a \em{stabilizing divisor.}
\end{definition}

The following is the main result of the Section where we construct
stabilizing pairs on broken manifolds.  We adapt the construction of
stabilizing pairs in smooth manifolds outlined in Section
\ref{sec:stab-manifold}.

\begin{proposition} 
{\rm(Stabilizing pair in a broken manifold)}
  \label{prop:stab-broken}
  Suppose $\XX:=\XX_\PP$ is a broken manifold, such that on each cut space
  $\ol X_P$, $P \in \PP$ the symplectic form $\om_{X_P}$ is
  rational.   There exists a stabilizing pair
  $(\JJ_0, \bD)$ in $\XX$ such that 
  \begin{enumerate}
  \item the family $(J^\nu,D^\nu)$ obtained by
    gluing consists of stabilizing pairs for all $\nu \in [1,\infty)$, and
  \item for every $E>0$ there is a neighbourhood
    %
    \[\J^\cyl(\XX,\bD;\JJ_0,E)  \subset \{\J \in \J^\cyl(\XX,\bD) :\J|T\bD = \JJ_0|T\bD\}\]
     of $\JJ_0$ consisting of $E$-stabilizing locally tamed 
  cylindrical almost complex structures adapted to $\bD$.
\end{enumerate}
\end{proposition}
It follows from Lemma \ref{lem:jjneck} that the family of almost complex structures $\{J^\nu\}_\nu$ 
produced by Proposition \ref{prop:stab-broken} are neck-stretched in the sense of Definitions in Section \ref{sec:cylacs}.

 \begin{proof}[Proof of Proposition \ref{prop:stab-broken}]

\vskip .05in
\noindent    \textsc{ Step 1}: We start with a preliminary almost complex
   structure, and find a uniform degree bound on its neighborhood, in
   order to determine the degree of the Donaldson divisor that we
   require.  Suppose $\JJ^\pre_\om \in \J(\XX)$ is a
   $\om_\XX$-compatible almost complex structure (Definition
   \ref{def:omxxcyl}).  By Lemma \ref{lem:tamingmap}, there are maps
   $\phi_P : X_P \to \ol X_P^\om$ from cut spaces to symplectic cut
   spaces, and a locally compatible (hence gluable, see Definitions
   \ref{def:brokenJ} \eqref{part:brokenJ-cyl}, \ref{def:cylneckst}
   \eqref{part:tamefib}) almost complex structure on $(X_P)_P$, which
   we denote by $\JJ^\pre$.  The almost complex structure $\JJ^\pre$
   is cylindrical on a truncation of the cylindrical ends, henceforth
   in this proof, the cylindrical ends are replaced by these
   truncations, and are denoted by $U_Q(X_P) \subset X_P$ for any pair
   $Q \subset P$.  We recall that cylindrical ends are equipped with
   projections $\pi_Q : U_Q(X_P) \to X_Q$.  For any $P \in \PP^{(0)}$,
   Lemma \ref{lem:tamingmap} says that $\phi_P^*\om_{X_P}$ is a
   \em{basic} form on the cylindrical ends of $X_P$, in the sense that
   for any $Q \subset P$, on $U_Q(X_P)$, $\phi_P^*\om_{X_P}$ is a
   pullback of a form by $\pi_Q$.  Gluing the forms
   $\phi_P^*\om_{X_P}$ along the ends produces
   $\om_\nu \in \Om^2(X^\nu)$ which is cohomologous to
   $\om_X \in \Om^2(X)$, and which is basic in the cylindrical
   subsets.  To derive the bound we also choose representatives of the
   first Chern classes $[c_1(T\ol X_P)]$ which are basic on
   cylindrical ends. That is, we choose forms
   \[ \gamma=(\gamma_P)_{P \in \PP}, \quad \gamma_P \in
     \Om^2(\ol X_P), \quad [\gamma_P]= c_1(T\ol X_P).
   \]
  We choose forms on the neck-stretched manifolds
   \[
     \gamma_\nu \in \Om^2(X^\nu), \quad \gamma_\nu \in [c_1(TX)],\]
   which are basic in the the sense that
   on the $P$-cylindrical subset $\pi_P^\nu : X_{\tP}^\nu \to X_P$,
   each of the forms $\gamma_\nu$ is the pullback of the same form on $X_P$.
  
    For any $P \in \PP$, the supremum of the ratio of
   the first Chern class form and the basic symplectic form 
   given by 
   \[k_*(\eps,P):=\sup_{\substack{\JJ:\Mod{\JJ,\JJ^\pre}_{C^0}<\eps\\ 0 \neq
         v \in TX_P}} \tfrac {\gamma_P(v,\JJ v)}{\om_{X_P}(v,\JJ v)}\]
   is finite. The proof of finiteness is as in the proof 
   of Lemma \ref{lem:unifdegbd}, together with the fact that the forms $\gamma_P$ and $\om_{X_P}$
   are basic on $X_P$. 
   Similarly, the supremum
   \[k_*(\eps,X):= \sup_\nu\sup_{\substack{\JJ:\Mod{\JJ,\JJ^\pre}_{C^0}<\eps\\ 0 \neq
         v \in TX^\nu}} \tfrac {\gamma_\nu(v,J^\nu v)}{\om_\nu(v,J^\nu v)},\]
   where $J^\nu$ is given by gluing $\JJ$ on the cylindrical ends, 
   is finite. 
  We choose $k_*$
   satisfying
   \begin{equation}
     \label{eq:kstar1}
     k_* \geq k_*(\eps,X), k_*(\eps,P) \quad \forall P \in \PP^{(0)}.
   \end{equation}

\vskip .05in \noindent    \textsc{ Step 2} : We find a stabilizing divisor in the broken
   manifold.
   We will find a stabilizing divisor of degree
   \begin{equation}
     \label{eq:k1choose}
     k > 2\max\{k_*, k_* + \hh(\dim(X)) -2\}.  
   \end{equation}
   for reasons explained in Remark \ref{rem:whydeg}.  Let $\theta_0>0$
   be such that for any $\theta_0$-approximately
   $\JJ^\pre$-holomorphic divisor, there is an adapted tamed almost
   complex structure that is $\eps/2$-close to $\JJ^\pre$.  In
   addition to satisfying \eqref{eq:k1choose}, we may choose $k$ to be
   large enough that a Donaldson divisor of degree $k$ is
   $\theta_0$-approximately holomorphic.
   By Proposition \ref{prop:brokdiv}, 
   there is a divisor $\DD$
   cylindrical in $\{U_Q(X_P)\}_{Q \subset P}$.
There is an adapted strongly tamed almost complex structure 
   \begin{equation*}
     \JJ_0 \in B_{\eps/2}(\JJ^\pre),
   \end{equation*}
   and furthermore, there is a
   neighborhood of adapted almost complex structures
   \[U_{\JJ_0} :=\{\JJ \in \J^\cyl(\XX,\DD) : \text{$\JJ$ is locally strongly tamed},\Mod{\JJ-\JJ_0}_{C^0}<\eps/2\}\]
   that satisfy the degree bound of $k_*$ on the cut spaces
   $\{x_P\}_P$, and their gluings satisfy the same degree bound on the
   neck-stretched manifolds $\{X_\nu\}_\nu$.

\vskip .05in   \noindent    \textsc{ Step 3}: We finish the proof closely following the method
   in the unbroken case in Cieliebak-Mohnke \cite{cm:trans}, which is
   outlined in Proposition \ref{prop:stab}.  So far, we have shown
   that there is a neighborhood $U_{\JJ_0}$ in the space of locally
   strongly tamed $\DD$-adapted cylindrical almost complex structures
   on $\XX$ such that for any $\JJ \in U_{\JJ_0}$, $k_*$ is a degree
   bound for the glued family $J^\nu$.  By Remark \ref{rem:whydeg}
   generic almost complex structures in the set $U_{\JJ_0}$ are
   stabilizing.
  
   Next, we claim that for any $E>0$, an open and dense subset of
   $U_{\JJ_0}$ is $E$-stabilizing for all neck lengths.  The subset
   \[\J^{\reg,E,\infty}=\{\JJ \in U_{\JJ_0}
     : \JJ \text{ is $E$-stabilizing on $\XX$}\} \]
   is open and dense: Openness is a consequence of Gromov convergence
   applied to each of the cut spaces $\ol X_P$, $P \in \PP$ (see the
   proof of \cite[Corollary 8.16]{cm:trans}), and denseness follows
   because the $E$-stabilizing condition is generic.  The subset
   \[\J^{\reg,E}=\{\JJ \in \J^{\reg,E,\infty} : J^\nu \text{ is
       $E$-stabilizing on $X^\nu$ for $\nu \in [1,\infty)$}\}. \]
   is also open in $\J^{\reg,E,\infty} $ by Gromov
   compactness. Openness at the infinite neck length parameter is
   proved by Lemma \ref{lem:cpt-stab}. The subset
   $\J^{\reg,E} \subset \J^{\reg,E,\infty}$ is comeager by an
   application of Sard's theorem : The universal moduli space
   \[\M_\J:=\{u_\nu:\P^1 \to X^\nu: u_\nu \text { is
       $J^\nu$-holomorphic}, \JJ \in \J^{\reg,E,\infty}, \nu \in
     [1,\infty) \} \]
   is a Banach manifold and the projection
   $\pi_\J:\M_\J \to \J^{\reg,E,\infty}$ is a Fredholm map. Therefore,
   the subset $\J^{\reg,E}$ of regular values of $\pi_\J$ is comeager.

   Finally, let $E_k \to \infty$ be any sequence of real numbers with
   limit infinity.  The set of almost complex structures
   \[ \J^{\reg} = \bigcap_{k = 1}^\infty \J^{\reg,E_k} \]
   that is stabilizing for all $\nu \in [1,\infty]$ is the
   intersection of the set of $E_k$-stabilizing almost complex
   structures for all $k$.  The intersection $\J^{\reg}$ is non-empty
   because each of the sets in the intersection is open and dense.
 \end{proof}

 The following Lemma, used above in the proof of Proposition
 \ref{prop:stab-broken}, is an openness statement for stabilizing
 almost complex structures at $\nu=\infty$.

\begin{lemma}\label{lem:cpt-stab}
  Suppose $\bD \subset \XX$ is a cylindrical broken divisor, and
  $\JJ \in \J^\cyl(\XX,\bD)$ is a tamed adapted almost complex
  structure that is $E$-stabilizing.  Suppose the divisors
  $D^\nu \subset X^\nu$ are obtained by gluing, and the sequence
  $J^\nu \in \J(X^\nu, D^\nu)$ converges to $\JJ$. Then, there exists
  $\nu_0$ such that $J^\nu$ is $E$-stabilizing for $\nu \geq \nu_0$.
\end{lemma}

The proof of the Lemma requires terminology from Chapter
\ref{chap:hof} and Section \ref{sec:horiz-conv}. The proof is by a
hard rescaling argument similar to the one used in the proof of Gromov
convergence for breaking maps.

\begin{proof}[Proof of Lemma \ref{lem:cpt-stab}] Suppose
  $u_\nu : C \to X^\nu$ is a sequence of non-constant
  $J^\nu$-holomorphic spheres with area $\leq E$ that are not
  stabilized. That is, either the images are contained in the divisor
  $D^\nu$ or they have $\leq 2$ distinct points of intersection with
  the divisor.  
  
  First, consider the situation that the derivatives of $u_\nu$ are
  uniformly bounded.  By Lemma \ref{lem:cptconv}, 
  after passing to a subsequence,
  the sequence $u_\nu$ converges horizontally in
  some polytope $P$. 
  By Lemma \ref{lem:hcp}, there is a sequence of
  translations $t_\nu \in \nu P^\dual$ such that a subsequence of
  $\e^{-t_\nu}u_\nu$ uniformly converges to a non-constant
  $\JJ_P$-holomorphic map $u: C \to \XX_P$, that is unstabilized.  The
  projection $\pi_P \circ u : C \to X_P$ is non-constant, because it
  is not possible for the image of $u$ to be non-constant and lie in a
  single fiber of $\pi_P$, which is isomorphic to $T_{P,\C}$.  By the
  definition of $P$-Hofer energy, for closed curves in $X_P$,
  \[\bran{(\pi_P \circ u)_*[C], \om_{X_P}} \leq E_{P,\Hof}^*(\pi_P \circ u).\]
  We have 
  \[ E_{P,\Hof}^*(\pi_P \circ u) \leq \lim_\nu E_\Hof(u_\nu) = \bran{(u_\nu)_*[C],\om_X} \leq E,\]
  where the first inequality is by Proposition
  \ref{prop:hofer-breaking}, and from the fact that Hofer energy is
  monotonic for maps that are holomorphic with respect to a locally
  strongly tamed almost complex structure (Lemma
  \ref{lem:monot-strong}), and the equality is from Remark
  \ref{rem:hof-area}. As a consequence,
  \[ \Area(u)=\bran{(\pi_P \circ u)_*[C], \om_{X_P}} \leq E .\]

  If the derivatives on $u_\nu$ are not uniformly bounded, we produce
  a sphere by hard rescaling. By following the procedure in Step 2 of
  the proof of Theorem \ref{thm:cpt-breaking}, we obtain a rescaled
  sequence of $J^\nu$-holomorphic maps $v_\nu : B_{r_\nu} \to X^\nu$
  on balls $B_{r_\nu}$ that exhaust $\C$, a polytope $P$, and a
  sequence of translations $t_\nu \in \nu P^\dual$ such that
  $\e^{-t_\nu}u_\nu$ converges in $C^\infty_\loc$ to a non-constant
  map $v:\C \to \XC_P$. In this preceding step, the rescaled maps are
  found in the same way as in \cite[Lemma 4.6.5]{ms:jh}, the polytope
  $P$ and translation sequences are given by Lemma \ref{lem:hcp}.  As
  in Step 2 in the proof of Theorem \ref{thm:cpt-breaking}, the map
  $v$ has finite $P$-Hofer energy, and by Proposition
  \ref{prop:remsing}, the infinite end of $v$ is asymptotic to a
  trivial cylinder. Therefore, the projection $\pi_P \circ v$ extends
  to $\pi_P \circ v:C_0 \to \ol X_P$, for some weighted projective
  space $C_0=\P(1,n)$.  The projection
  $\pi_P \circ v: C_0 \to \ol X_P$ is non-constant: Otherwise, the
  image of $v(\C)$ is in a fiber $V_{P^\dual}$ which is a toric
  variety, and the map extends to $v: C_0 \to V_{P^\dual}$ with only
  one point $\infty \in C_0$ that maps to toric divisors
  $V_{Q^\dual}, Q \supset P$ of $V_{P^\dual}$, and therefore, $v$ is
  constant.  By the argument in the previous paragraph,
  $\Area(v)=\bran{(\pi_P \circ v)_*[C_0], \om_{X_P}} \leq E$.

  In both cases, the limit (orbifold) sphere is not stabilized and the
  Lemma follows.
\end{proof}

\chapter{Coherent perturbations and regularity}\label{chap:transv}

In order to obtain the necessary transversality for moduli spaces of
treed holomorphic maps, we use the Cieliebak-Mohnke perturbation
scheme \cite{cm:trans}.  This scheme has been adapted to define Fukaya
algebras of Lagrangians by Charest-Woodward \cite{cw:traj}.  The Morse
function on the Lagrangian submanifold and the almost complex
structure on the broken manifold are allowed to be
domain-dependent. Stabilizing divisors from the previous chapter are
essential in defining domain-dependent perturbations.

A domain-dependent perturbation is defined as a map from a universal
curve to the space of tamed almost complex structures.  We recall from
Chapter \ref{chap:brokendisks} that domain curves are treed disks.  For
any type $\Gamma$ of treed disks the moduli space $\M_\Gamma$ of treed
disks has a universal curve
\[\U_\Gamma \to \M_\Gamma\]
whose fiber over a point $[C] \in \M_\Gamma$ is isomorphic to $C$.
The perturbation datum is then a map from $\U_\Gamma$ to the space of
cylindrical almost complex structures on the broken manifold and Morse
functions on the Lagrangian submanifold.  Under this perturbation
scheme, marked points on the domain curve are the inverse images of
the intersection of the holomorphic map with the stabilizing
divisor. The stabilizing divisor is Poincar\'e dual to a large
multiple of the symplectic form. This ensures that generically there
are no holomorphic spheres in the divisor, and that any non-constant
holomorphic sphere has enough divisor intersections to ensure that its
domain is stable.

Domain-dependent perturbations are necessary for solving the multiple
cover problem. Indeed, for a fixed almost complex structure the
compactification of the moduli space of pseudoholomorphic maps may
contain nodal maps with components that are multiple covers of maps
with negative Chern class. It is not possible to achieve
transversality for the moduli space of such maps. Multiple covers can
be perturbed away by domain-dependent perturbations. As a toy example,
suppose for a domain-independent almost complex structure $J_0$, there
is a simple $J_0$-holomorphic sphere $S \subset X$ with negative Chern
number and one intersection point with the stabilizing divisor. Via a
domain-dependent perturbation of $J_0$, we can ensure that an
$n$-covered sphere ($n>1$) homologous to $n[S]$ does not occur in the
moduli space of perturbed holomorphic maps. Indeed, for such a cover
the domain would have $n$ marked points labelled $z_1,\dots,z_n$, and
the perturbation is not required to be invariant under automorphisms
of the domain curve that permute the marked points.  Breaking the
permutation symmetry adds a multiplicative factor of $\frac 1 {n!}$ to
the curve count (see Remark \ref{rem:pertsheet}).
\label{marklabels}

\section{Domain-dependent perturbations}\label{sec:domdep}
A domain-dependent perturbation consists of a domain-dependent almost complex structure on the surface parts and a domain-dependent Morse function on the treed segments of treed disks.
Domain-dependent perturbations $\Pe_\Gamma$ are defined for each
domain type $\Gamma$, and later in the section we describe a set of
\em{coherence} conditions on the set of perturbation data
$(\Pe_\Gamma)_\Gamma$ to ensure that the perturbations on various
strata fit 
together in an expected manner.

 To describe perturbations for a domain type,  
we first fix subsets of the universal treed disk outside of which
perturbations vanish.  Let $\Gamma$ be a combinatorial type of treed
disk and $\ol{\U}_\Gamma = \ol{\S}_\Gamma \cup \ol{\T}_\Gamma$ its
universal curve from \eqref{eq:univcurv}.  Fix a compact subset
\[  \ol{{\T}}^{\on{cp}}_{\Gamma} \subset \ol{{\T}}_{\Gamma} \]
in the complement of breaking points and disk nodes (that is, disk
nodes $w_e$ where the length $\ell(e)$ of the treed segment is zero),
and containing in its interior, for every edge
$e \in \Edge_\white(\Gamma)$ and every curve
$C \subset \ol \U_\Gamma$, at least one point $z \in T_e \subset C$ on
any infinite segment.  Also fix a compact subset
\[  \ol{\S}_\Gamma^{\on{cp}} \subset \ol{\S}_\Gamma - \{ w_e \in
\ol{\S}_\Gamma, \ e \in \Edge_-(\Gamma) \} \]
disjoint from the boundary and spherical nodes
$w_e \in C, e \in \Edge(\Gamma)$, containing in its interior at least
one point $z \in S_v$ on each sphere and disk component
$S_v \subset C$ for each fiber $C \subset \ol{\U}_\Gamma$.
Furthermore, the complement
$ \ol{{\S}}_{\Gamma} - \ol{{\S}}^{\on{cp}}_{\Gamma} \subset
\ol{{\S}}_{\Gamma} $
is a neighbourhood of the boundary and nodes; these neighborhoods must
be chosen compatibly with those already chosen on the boundary for the
inductive construction later.

Following Floer \cite{floer:grad}, we use a $C^\veps$-topology on the
space of almost complex structures.  For a section $\xi$ of a vector
bundle $E \to X$, the $C^\veps$-norm is
\[\Mod{\xi}_{C^\veps}:= \ssum_{k=0}^\infty \veps_k
\Mod{\xi}_{C^k(X,E)}.\]
Here $\veps=(\veps_i)_{i \in \N}$ is a fixed sequence of positive
numbers that converges fast enough to $0$ as $i \to \infty$. If the
convergence is sufficiently rapid, then the space of sections with a
bounded norm is a Banach space \cite[Lemma 5.1]{floer:grad} and
contains sections supported in arbitrarily small neighbourhoods of
$X$.

\begin{definition} {\rm (Perturbation data)} 
\label{def:domdep}
  \begin{enumerate}
  \item {\rm (Domain-dependent Morse functions)} \index{Morse function on the Lagrangian} Suppose that $\Gamma$ is a type of stable
    treed disk, and $\ol{\T}_\Gamma \subset \ol{\U}_\Gamma$ is the
    tree part of the universal treed disk.  Let
    \[ (F_0:L \to \R,G: T^{\otimes 2} L \to \R) \]
    be a Morse-Smale pair.  For an integer $l \ge 0$ a \em{
      domain-dependent perturbation} of $F_0$ of class $C^l$ is a $C^l$
    map
    \index{Perturbation! Domain-dependent perturbation}
    \begin{equation} \label{eq:FGam} F_{\Gamma}: \ol{{\T}}_{\Gamma}
      \times L \to \R \end{equation}
    equal to the given function $F$ away from the compact part:
    \[ F_\Gamma | (\ol{{\T}}_{\Gamma} - \ol{{\T}}^{\on{cp}}_{\Gamma})
    = \pi_2^* F_0 \]
  where $\pi_2$ is the projection on the second factor in
  \eqref{eq:FGam}.  Here $F_0$ is called the \em{background Morse
    function} for the domain-dependent perturbation $F_\Gamma$. The
  set of critical points of the background Morse function is denoted
  by
  \begin{equation}
    \label{eq:Igeom}
    \cI(L):=\crit(F).
  \end{equation}
\item {\rm (Domain-dependent almost complex structure)}
  \label{part:domdepJ}
  Let $\JJ_0 \in \J^{\on{cyl}}(\XX)$ be a locally strongly tamed
  cylindrical almost complex structure.  A \em{domain-dependent
    almost complex structure} of class $C^\veps$ for treed disks of
  type $\Gamma$ with \em{background almost complex structure} $\JJ_0$
  \index{Background almost complex structure} is a map from the
  two-dimensional part $\ol{{\S}}_{\Gamma} $ of the universal curve
  $\ol{{\U}}_{\Gamma}$ to $ \J^\cyl(\XX)$ given by
  \begin{equation}
    \label{eq:Jvary}
    \JJ_{\Gamma} : \ \ol{{\S}}_{\Gamma} \to \J^{\on{cyl}}(\XX)
    \footnote{Here we do not say that $\JJ_\Gamma$ maps to the locally tamed elements in $\J^\cyl(\XX)$.  But later in Definition \ref{def:adapt-pert} we constrain $\JJ_\Gamma$ to take values in a small enough neighborhood of $\JJ_0$ that automatically ensures local tamedness.}
  \end{equation}
  equal to the given $\JJ_0$ away from the compact part:
  \begin{equation}
    \label{eq:basej0}
    \JJ_\Gamma | 
    (\ol{{\S}}_\Gamma - \ol{{\S}}^{\on{cp}}_\Gamma ) = \JJ_0,     
  \end{equation}
  and on any fiber $S_\Gamma \subset \ol \S_\Gamma$, $\JJ_\Gamma - \JJ_0$
  has finite norm in $C^\veps(S_\Gamma \times \XX, \End(T\XX))$.
\item {\rm (Perturbation data)}  A perturbation datum for a type $\Gamma$ of stable treed disks is a
  pair $\Pe_\Gamma = (F_\Gamma,\JJ_\Gamma)$ consisting of a
  domain-dependent Morse function $F_\Gamma$ and a domain-dependent
  almost complex structure $\JJ_\Gamma$.
\end{enumerate}
\end{definition}
The following morphisms on the set of combinatorial types of stable
treed disks are used to define coherence of perturbations.

\begin{definition} \label{def:pertops}
  {\rm(Morphisms on treed disk types)}
  \begin{enumerate}
  \item \label{item:cuttingedgesmorphism} 
  {\rm (Cutting edges)}
      \index{Cutting an edge}
There is a (Cutting edges) morphism $\Gamma \to \Gamma'$ between 
    combinatorial types $\Gamma$, $\Gamma'$ of stable treed disks 
    if and only if $\Gamma$ is
    obtained by cutting a boundary edge $e \in \Edge_{\white,-}$ of $\Gamma'$
    that contains a breaking; see Figure \ref{fig:cutedge}.

  \item \label{item:collapsingedgesmorphism} {\rm (Collapsing edges)}
    \index{Collapsing an edge}
  A morphism $\Gamma \to \Gamma'$ is a (Collapsing edges) morphism if $\Gamma'$ is obtained from $\Gamma$ by collapsing an interior edge $e \in \Edge_{\black,-}(\Gamma)$ or a boundary edge $e \in \Edge_{\white,-}(\Gamma)$ with length zero, $\ell(e)=0$.
  \item \label{item:makingedgelengthmorph} {\rm(Making an edge length finite or
      non-zero)}
      \index{Making an edge length finite/non-zero}
A morphism $\Gamma \to \Gamma'$ is a  (Making an edge length finite or non-zero) morphism if  $\Gamma'$ is obtained from $\Gamma$ by
changing the edge length of a boundary edge $e \in \Edge_{\white,-}(\Gamma)$ from infinite or zero to finite non-zero.
\end{enumerate}
\end{definition}

\begin{figure}[h]
  \centering \scalebox{.8}{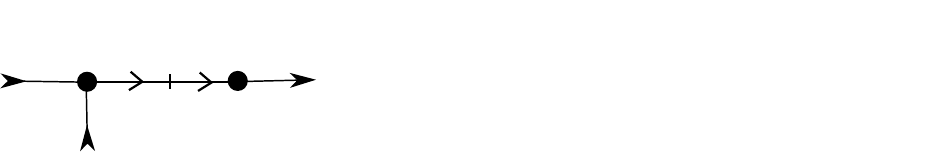}
  \caption{Cutting an edge $e$ relabels the boundary and interior markings while preserving their ordering on the pieces $\Gamma_0'$, $\Gamma_1'$.}
  \label{fig:cutedge}
\end{figure}

Perturbation data $\Pe_\Gamma$ for various treed disk types $\Gamma$
must be chosen coherently to obtain compactness.  For any two
treed types $\Gamma$, $\Gamma'$ related by a morphism from Definition \ref{def:pertops}, there is a
coherence condition relating $\Pe_\Gamma$, $\Pe_{\Gamma'}$.

  \begin{definition}\label{def:pertmorph}
    {\rm(Morphisms of perturbation data)}
\begin{enumerate} \label{defpi}
\item {\rm(Collapsing edges/Making an edge length finite/non-zero)}
  \index{Collapsing an edge} \index{Making an edge length finite/non-zero} Let $\Gamma' \to \Gamma$ be a (Collapsing
  edges/making an edge length finite/non-zero) morphism. 
  For perturbation data $\Pe_{\Gamma'}$,
  $\Pe_\Gamma$ there is a morphism $\Pe_{\Gamma'} \to \Pe_\Gamma$ if
  $\Pe_{\Gamma'}$ is induced by pullback of $\Pe_{\Gamma}$ under the
  natural inclusion of the universal curve
    \[\iota_\Gamma^{\Gamma'} : \ol{\U}_{\Gamma'} \to   \ol{\U}_{\Gamma}.\]
  \item {\rm(Cutting edges)}
    \index{Cutting an edge}
    Suppose $\Gamma \to \Gamma'$ is a (Cutting edges) morphism, where an edge $e \in \Edge_{\white,-}(\Gamma)$ is cut to yield leaf edges
    $e_+,e_- \in \Edge(\Gamma')$. There is a morphism of perturbation data $\Pe_{\Gamma} \to \Pe_{\Gamma'}$
    if $\Pe_{\Gamma'}$ is obtained by pushing forward $\Pe_{\Gamma}$ under the map
    \[
       \pi_\Gamma^{\Gamma'}: \ol{\U}_{\Gamma} \to \ol{\U}_{\Gamma'}\]
     defined by gluing at the leaf edges $e_\pm$ to form a single
     non-leaf edge $e$ with $\ell(e)=\infty$.  That is, define
\[ \JJ_{\Gamma'}(z',x) = \JJ_\Gamma(z,x), \quad \forall z \in
(\pi_\Gamma^{\Gamma'})^{-1}(z) .\]
The definition for the perturbed Morse datum $F_{\Gamma'}$ is similar.
\end{enumerate}

\end{definition}
We are now ready to define coherent collections of perturbation data.
These are data that behave well with each type of operation in
Definition \ref{def:pertops}.  \index{Perturbation! Coherent
  perturbation datum}
\begin{definition} \label{def:coherent} {\rm (Coherent families of
    perturbation data)} A collection of perturbation data
  $ \ul{\Pe} = (\Pe_\Gamma )_\Gamma $ is \em{coherent} if it is
  compatible with the morphisms of moduli spaces of different types in
  the sense that
  \begin{enumerate}
  \item \label{item:collapsing} {\rm (Collapsing edges)/(Making an
      edge length finite/non-zero)} \index{Collapsing an edge}
    \index{Making an edge length finite/non-zero} if $\Gamma$ is
    obtained from $\Gamma'$ by collapsing an edge or making an edge
    finite/non-zero, then $\Pe_{\Gamma'}$ is the pullback of
    $\Pe_{\Gamma}$;
  \item \label{item:cuttingedges} {\rm (Cutting edges) }
    \index{Cutting an edge} if $\Gamma$ is obtained from $\Gamma'$ by
    cutting a boundary edge $e \in \Edge_{\white,-}^\infty(\Gamma')$
    of infinite length, then $\Pe_{\Gamma'}$ is the push-forward of
    $\Pe_{\Gamma}$. Assuming $\Gamma$ is the union of types
    $\Gamma_1,\Gamma_2$, $\Pe_\Gamma$ is obtained from
    $\Pe_{\Gamma_1}$ and $\Pe_{\Gamma_2}$ as follows: For $k=1,2$, let
\[\pi_k: \ol{\M}_\Gamma \cong \ol{\M}_{\Gamma_1} \times
\ol{\M}_{\Gamma_2} \to \ol{\M}_{\Gamma_k}\]
denote the projection on the $k$-th factor, and therefore,
$\ol{\U}_\Gamma$ is the union of $\pi_1^* \ol{\U}_{\Gamma_1}$ and
$\pi_2^* \ol{\U}_{\Gamma_2}$.  We require that $\Pe_\Gamma$ is equal
to the pullback of $\Pe_{\Gamma_k}$ on $\pi_k^* \ol{\U}_{\Gamma_k}$:
\begin{equation} \label{eq:require} \Pe_\Gamma | \ol{\U}_{\Gamma_k} =
  \pi_k^* \Pe_{\Gamma_k} .\end{equation}
\end{enumerate}
We also require the perturbation data to satisfy the following
locality axiom which ensures that the perturbations on any component
only depend on special points on that component, and the length of the
treed segments on the boundary of the disk. We first set up some
notation: For a type $\Gamma$ underlying treed disks, and a vertex
$v \in \Ver(\Gamma)$, let $\Gamma(v)$ be a graph with a single vertex
$v$ and markings
\[\{e \in \Edge(\Gamma) : v \in e\}.\]
corresponding to each edge incident on $v$. \label{gammav} Let
$\U_{\Gamma,v} \subset \U_\Gamma$ be a fibration over $\M_\Gamma$
whose fiber over $m \in \M_\Gamma$ consists of the curve component
represented by $v$.  Define a map
\begin{equation}
  \label{eq:loc-proj}
  \pi_v : \U_{\Gamma,v} \to \U_{\Gamma(v)} \times ([0,\infty])^{|\Edge_{\white,-}(\Gamma)|},
\end{equation}
whose first component $\U_{\Gamma,v} \to \U_{\Gamma(v)}$ is the
natural projection map, and the second component is the length
function on boundary edges $e \in \Edge_{\white,-}(\Gamma)$.
\begin{itemize}
  \index{Locality axiom}
\item[] {\rm (Locality Axiom)}
\label{item:localityaxiom}
  The restriction of the perturbation
  datum $\Pe_\Gamma$ to $\U_{\Gamma,v}$ is the pullback via $\pi_v$ of
  some datum on
  $\U_{\Gamma(v)} \times ([0,\infty])^{|\Edge_{\white,-}(\Gamma)|}$.
\end{itemize}
This ends the Definition. \end{definition}

Let $C$ be a possibly unstable treed disk of type $\Gamma$.  The \em{
  stabilization} of $C$ is the stable treed disk $\st(C)$ of some type
$\st(\Gamma)$ obtained by collapsing unstable surface and tree
components.  Thus the stabilization $\st(C)$ of any treed disk $C$ is
the fiber of a universal treed disk $\U_{\st(\Gamma)}$.  Given a
perturbation datum for the type $\st(\Gamma)$, we obtain a
domain-dependent almost complex structure and Morse function for $C$,
still denoted $\JJ_\Gamma,F_\Gamma$, by pull-back under the map
$C \to \U_{\st(\Gamma)}$. If $\Gamma$ does not contain vertices,
i.e. if $C$ is a single infinite segment $T_e, e \in \Edge(\Gamma)$,
then the perturbation $\Pe_\Gamma$ vanishes on $C$.

\begin{remark} \label{rem:cohmot}
  \index{Collapsing an edge}
  \index{Making an edge length finite/non-zero}
  \index{Cutting an edge}
  Coherence conditions for perturbation data are motivated as follows.
  \begin{enumerate}
  \item If a sequence $C_\nu$ of (domain) curves of type $\Gamma$
    converges to a limit curve $C$ of type $\Gamma'$, we would like
    the perturbation data on $C_\nu$ to converge to the perturbation
    datum on $C$. If $\Gamma' \neq \Gamma$ then there is a
    \hyperref[item:collapsing]{(Collapsing edges)/(Making an edge
      length finite/non-zero)} morphism $\Gamma' \to \Gamma$, and the
    convergence of the perturbation datum is ensured by coherence
    under \hyperref[item:collapsing]{(Collapsing edges)/(Making an
      edge length finite/non-zero)}.
  \item Suppose $\Gamma \to \Gamma'$ is a cutting edges morphism, and
    suppose the disconnected type $\Gamma'$ has components
    $\Gamma'_0$, $\Gamma_1'$.  Once moduli spaces of perturbed
    holomorphic maps in a symplectic manifold $X$ are defined, we
    would like to be able to say that the moduli space $\M_\Gamma(X)$
    of maps with domain type $\Gamma$ is a product
    \[\M_\Gamma(X)=\M_{\Gamma_0'}(X) \times \M_{\Gamma_1'}(X),\]
    The product relation would hold only if the perturbation data is
    coherent under \hyperref[item:cuttingedges]{(Cutting edges)}.
  \end{enumerate}
  
\end{remark}

\begin{remark} 
  {\rm(On the locality axiom)} The locality axiom ensures that
  forgetting a marking $z_{e'}$ on a treed curve $C$ affects the
  perturbation datum only on the component containing $z_{e'}$. This
  feature is used in Proposition \ref{prop:noghosts}.  The dependence
  on the boundary edge lengths is useful for the following
  reason. Suppose $\Gamma$ is a combinatorial type of treed disk
  depicted in Figure \ref{fig:cutedge}.  By cutting an edge
  $e \in \Edge(\Gamma)$, we obtain two identical types. The cutting
  edge axiom requires that a coherent perturbation datum $\Pe_\Gamma$
  for $\Gamma$ is equal after restriction to the universal curve for
  the type $\Gamma'$ on both sides of the edge $e$. If in the locality
  axiom, the perturbation is $\Pe_\Gamma$ defined by pulling back by
  the map $\pi_v : \U_{\Gamma,v} \to \U_{\Gamma(v)}$, then the
  perturbation datum on both surface components will be required to be
  equal even when the edge length of $e$ is finite. This creates a
  problem, because in order to obtain transversality in the case
  $\ell(e)=0$, we need the perturbation datum on both surface
  components to be independent of each other.
\end{remark}

\begin{remark}\label{rem:gwpert}
  {\rm(Perturbations for moduli spaces of spheres)} The coherence
  conditions required to define moduli spaces of spheres are much
  simpler than those required for treed disks: The only coherence
  condition is \hyperref[item:collapsing]{(Collapsing edges)}, which
  translates to the statement that the perturbation datum is
  continuous on the compactified moduli space $\ol \M_n$ of $n$-marked
  spheres. An important difference from the disk case is that the
  perturbation datum on $\ol \M_n$ can be defined independently of
  $\ol \M_k$ for any $k<n$.
\end{remark}

We require perturbations to be adapted to the stabilizing pair
$(\JJ_0,\DD)$ in the sense that $\JJ_0$ is the background almost
complex structure; the domain-dependent almost complex structures
$\JJ_\Gamma$ are adapted to the stabilizing divisor (as in Definition
\ref{def:j-adapt}) in that the tangent space of the stabilizing
divisor is $\JJ_\Gamma(z)$-invariant for all points $z \in \ol \U_\Gamma$;
and finally, the domain-dependent almost complex structures take
values in small enough neighborhoods of $\JJ_0$, with the size of these
neighborhoods being smaller for types $\Gamma$ with more interior
markings. The reason for the last requirement is explained later in
Remark \ref{rem:stabrem} after perturbed holomorphic maps are defined.

\begin{definition} \label{def:adapt-pert} {\rm(Perturbations adapted
    to a stabilizing divisor)} Let $k \gg 0$, and $(\JJ_0,\bD)$ be a
  stabilizing pair on the broken manifold $\XX$ (as in Definition
  \ref{def:stab-brok}), such that $D_P \subset \ol X_P$ is dual to
  $k[\om_{X_P}]$ for all polytopes $P \in \PP$.  Furthermore, let
  $\bD$ be disjoint from the Lagrangian submanifold $L$.  Suppose
  $\Gamma$ is a type of treed disk. A perturbation datum
  $\Pe_\Gamma=(\JJ_\Gamma, F_\Gamma)$ on $(\XX,L)$ is adapted to the
  pair $(\JJ_0,\bD)$ if
  \begin{itemize}
  \item $\JJ_0$ is the background almost complex structure for
    $\JJ_\Gamma$, meaning that each $\JJ_\Gamma$ is a domain-dependent
    perturbation of $\JJ_0$, and 
  \item for any treed curve $C=S \cup T$, and a connected
    component $S' \subset S$ with $d_\black(S')$ interior markings,
  \begin{equation}
    \label{eq:enbhd}
    \JJ_\Gamma(S') \subset \J^{\cyl}(\XX,\DD;\JJ_0,\tfrac 1 k 
    d_\black(S')) \cap \U_{\JJ_0}.
  \end{equation}
  Here
  \[\J^{\cyl}(\XX,\DD;\JJ_0,\tfrac {1}{k}d_\black(S')) \subset
    \{\JJ \in \J^{\cyl}(\XX,\DD) : \JJ|T\bD = \JJ_0|T\bD\}\]
  is the neighbourhood of $\JJ_0$ consisting of $\frac
  {1}{k}d_\black(S')$-stabilizing almost complex structures (see
  Proposition \ref{prop:stab-broken}), and $U_{\JJ_0} \subset \J^\cyl(\XX)$ is a neighborhood of $\JJ_0$ on which the results on Hofer energy hold (see Lemma \ref{lem:monot}).
  \end{itemize}
  The set of perturbation data adapted to $(\JJ_0,\DD)$ is denoted by
  $\PPe_\Gamma(\XX,\JJ_0,\DD)$.
\end{definition}

\section{Perturbed maps}

Given domain-dependent perturbations adapted to a stabilizing divisor 
as in the previous section, we define 
perturbed pseudoholomorphic maps and their moduli spaces.

\begin{definition} {\rm (Perturbed holomorphic treed disks)}
  \label{def:pert-unbroken}
  Given a rational symplectic manifold $(X,\om_X)$, a rational Lagrangian $L \subset X$, 
  a coherent perturbation datum $\ul \Pe=\{\Pe_\Gamma=(J_\Gamma,F_\Gamma)\}_\Gamma$ on $X$ 
  adapted to a stabilizing divisor $D$ (as in Definition
  \ref{def:adapt-pert}), a
  \em{perturbed holomorphic treed disk} $u$ of stable domain type
  $\Gamma$ consists of 
  \begin{enumerate}
  \item a treed disk $C = S \cup T$ of type
  $\Gamma$, and
  \item a continuous map
\[ u: C \to X,  \] 
  \end{enumerate}
  such that the following hold:
  \begin{enumerate}
  \item {\rm (Boundary condition)} The tree components and the
    boundary of the surface components are mapped to the Lagrangian
    submanifold :
    \[u(\partial S \cup T) \subset L.\]
  \item {\rm (Surface equation)} On the surface part $S$ of $C$, the
    map $u$ is holomorphic with respect to the given domain-dependent
    almost complex structure: if $j$ denotes the complex structure on
    $S$ then
    \begin{equation}
      \label{eq:Jpert}
      J_\Gamma(z,u(z)) \ \d u_S = \d u_S \ j.   
    \end{equation}
  \item {\rm (Boundary tree equation)} On the boundary tree part $T
    \subset C$ the map $u$ is a collection of gradient trajectories:
    \begin{equation}
      \label{eq:Fpert}
    \dds u_T = -\grad_{ F_\Gamma(s,u(s)) } u_T   
    \end{equation}
    where $s$ is a coordinate on the segment so that the segment has
    the given length.  Thus, for each treed edge
    $e \in \Edge_{\white,-}(\Gamma)$ the length of the trajectory
    $u |_{T_e}$ is equal to $\ell(e)$.
  \item {\rm(Adapted to stabilizing divisor)}
\label{part:u-adapt-d}
    Each interior marking
  $z_e, e \in \Edge_{\black,\to}(\Gamma)$ maps to $D$ under $u$ and
  each connected component $C'$ of $u^{-1}(D) \subset C$ contains an
  interior marking.
  \end{enumerate}
  We say that the map $u$ is \em{$\Pe_\Gamma$-holomorphic}. 
  A perturbed holomorphic disk $u$ with an unstable domain type $\Gamma$
  is defined
  similarly except that $u_v, u_e$ are
  $\cP_{\st{\Gamma}}$-holomorphic, where $\st{\Gamma}$ is the
  stabilization of $\Gamma$. On surface and treed components that get collapsed in the stabilization, the map $u$ is pseudoholomorphic with respect to the background data $(J_0,F_0)$.
\end{definition}

\begin{remark} \label{rem:stabrem} Our domains below are stable on the
  surface parts because of the additional markings arising from
  intersections with the stabilizing divisor. 
  So the content of the last sentence is that
  on broken Morse trajectories, for segments that are infinite in both
  directions the gradient flow equation is unperturbed. \end{remark}

Perturbed holomorphic disks in broken manifolds are defined analogously to the smooth version in Definition
\ref{def:pert-unbroken}. 
\begin{definition} \label{def:pdisks} {\rm(Perturbed holomorphic
    broken treed disks)}
  Given
  a broken manifold $\XX:=\XX_\PP$ with rational symplectic cut spaces $X_P$, $P \in \PP$, a rational Lagrangian
  $L \subset X_{P_0}$, 
  a coherent perturbation datum
  $\Pe=\{\Pe_\Gamma=(\JJ_\Gamma,F_\Gamma)\}_\Gamma$ on $\XX$ that is adapted to a broken stabilizing
  divisor $\DD$ (as in Definition \ref{def:adapt-pert}), 
  a \em{perturbed holomorphic broken treed disk} 
  $u$ of stable domain type $\Gamma$ 
   consists of
  \begin{enumerate}
  \item a treed disk $C = S \cup T$ of type $\Gamma$;
  \item a tropical structure $\cT$ on $\Gamma$; 
  \item a collection of maps on surface components
    \[ u_v: S_v^\circ \to \XX_{P(v)}, \quad v \in
    \Ver(\Gamma) \]
  that are $\JJ_\Gamma$-holomorphic as in \eqref{eq:Jpert} where
  $S_v^\circ $ is the punctured domain curve with lifts of tropical
  nodal points deleted (see \eqref{eq:opensv}), satisfy the Lagrangian
  boundary condition $u_v( \partial S_v ) \subset L$ for all
  $v \in \Ver_\white(\Gamma)$ and are adapted to the stabilizing divisor $\DD$ (as in Definition \ref{def:pert-unbroken} \eqref{part:u-adapt-d}); and 
  \item a collection of $F_\Gamma$-gradient flow lines as
  in \eqref{eq:Fpert}
   \[ u_e: T_e \to L \subset X_{P_0},  \quad e \in 
    \Edge_{\white}(\Gamma) \]
  on the treed parts, 
  that satisfy all the matching conditions on
  broken treed holomorphic maps (Definition \ref{def:ubmap}).
  \end{enumerate}
  We say that the map $u$ is \em{$\Pe_\Gamma$-holomorphic}. 
\index{Holomorphic disks!$\ul \Pe$-holomorphic disks}
  Broken maps on unstable domain types are defined similarly as the unbroken case in
  Definition \ref{def:pert-unbroken}.
\end{definition}
%



\begin{remark}
  We explain why we use $\frac {1}{k}d_\black(S')$-stabilizing almost
  complex structures in \eqref{eq:enbhd}.  Let
  $\ul J=(J_\Gamma)_\Gamma$ be a collection of coherent
  domain-dependent almost complex structures satisfying
  \eqref{eq:enbhd}.  Let $S_\nu$ be a sequence of nodal curves of type
  $\Gamma$ with $d(\black)$ markings.  Then the limit $u$ of any
  sequence $u_\nu:S_\nu \to \XX$ of $J_\Gamma$-holomorphic broken maps
  cannot
  have an unstable domain component. Indeed, the area of the
  maps $u_\nu$ is $\frac {d_\black} k$, and so an unstable component
  $u_v$ of the $u$ has area $\leq \frac {d_\black} k$.  Therefore, $u$
  is holomorphic with respect to a domain-independent almost complex
  structure $\JJ_v$ which, by \eqref{eq:enbhd}, is
  $\frac {d_\black} k$-stabilizing.  This contradicts the existence of
  a non-constant map $u_v$.
\end{remark}

We now define moduli spaces of perturbed broken treed holomorphic disks. 
The moduli space of such maps is stratified by combinatorial type.

\begin{definition} \label{def:type-broken} \index{Combinatorial type!
    of a holomorphic broken treed disk} The \em{combinatorial type of
    an perturbed holomorphic broken treed disk} $u:C\to \XX$ adapted to a
  broken divisor $\DD \subset \XX - L$ consists of
    \begin{enumerate}
    \item the combinatorial type $\Gamma$ of its domain $C$,
    \item the tropical structure on $\Gamma$, which consists of an
      assignment of polytopes for vertices, and {direction}s for tropical
      edges :
      \[\Ver(\Gamma) \ni v \mapsto P(v) \in \PP, \quad
        \Edge_\trop(\Gamma) \ni e \mapsto \cT(e) \in \t_{P(e),\Z};\]
    \item a labelling
      \begin{equation*}
        d(v) :=(\pi_{P(v)} \circ u_v)_*[(S_v,\partial S_v)] \in
        \begin{cases}
          H_2(\ol X_{P_0}, L), \quad v \in \Ver_\white(\Gamma),\\
          H_2(\ol X_{P(v)}), \quad v \in \Ver_\black(\Gamma)
        \end{cases}
      \end{equation*}
      that assigns to each vertex $v$ of $\Gamma$ the homology class
      of the disk/sphere obtained by extending
      $\pi_{P(v)} \circ u_v: S^\circ_v \to X_{P(v)}$ over the
      punctures corresponding to nodal lifts; \footnote{In case the
        space $\XX_{P(v)}$ has an orbifold compactification
        $\XX_{P(v)}$ (see Remark \ref{rem:compactify}), the homology
        class of $(u_v)_*[S_v] \in H_2(\ol \XX_{P(v)})$ can be read
        off from $d(v)$ and the tropical graph.}
    \item a labelling
      \begin{equation*}
        \mu_{\bD}:\Edge_{\black,\rightarrow}(\Gamma) \to \Z_{> 0}
      \end{equation*}
      that records the order of tangency of the map $u$ to the divisor
      $\bD$ at markings that do not lie on horizontally constant
      components (with the convention that a transverse intersection
      has order $1$).
    \end{enumerate}
    The type is denoted simply as $\Gamma$, suppressing
    $\cT,d,\mu_{\bD}$ in the notation, or by $\Gamma_X$ if we wish to
    distinguish a type of map $\Gamma_X$ from a type of treed disk
    $\Gamma$.
  \end{definition}

  We introduce the following notations for moduli spaces.
  Given a collection of coherent perturbation data $\ul \Pe=(\Pe_\Gamma)_\Gamma$, 
  let 
  $\tM^\br(L,\ul \Pe)$ be the moduli space of isomorphism classes of
  stable treed broken holomorphic disks in $\XX$ with boundary in
  $L$, where isomorphism is modulo
  reparametrizations of domains and is defined in Definition
  \ref{def:iso-stab} \eqref{part:mapiso}.  \index{Moduli space!of broken maps $\M_\Gamma^\br(L)$}
  Let
  \[ \tM_{\Gamma}^\br(L,\Pe_\Gamma) \subset \tM^\br( L,\ul \Pe) \]
  be the \em{moduli space of broken maps} of combinatorial type
  $\Gamma$ modulo the action of domain reparametrizations.
  We drop the perturbation datum $\ul \Pe$ from the notation if it is obvious from
  the context.    We recall
  that domain reparametrizations are isomorphisms of broken maps as in
  Definition \ref{def:iso-stab} \eqref{part:mapiso}.  The group of tropical symmetries
  $T_\trop(\Gamma)$ acts naturally on $\tM_\Gamma^\br(L, \Pe_\Gamma)$.
  The quotient
\[ \M_{\Gamma,\red}^\br(L,\Pe_\Gamma) := \tM_\Gamma^\br(L,\Pe_\Gamma) /
T_{\on{trop}}(\Gamma) \]
\index{Moduli space!of broken maps, reduced $\M_{\Gamma,\red}^\br(L)$}
is the \em{reduced moduli space of broken maps} of combinatorial type
$\Gamma$.  For $\ul x\in \cI(L)^{d(\white)+1}$, let
\begin{equation}
  \label{eq:reduced-def}
  \tM_{\Gamma}^{\br}(L,\Pe_\Gamma,\ul{x}) \subset \tM_{\Gamma}^{\br}(L,\Pe_\Gamma)
\end{equation}
denote the subset of perturbed holomorphic treed disks of type
$\Gamma$ with limits
$\ul{x} = (x_0 ,\dots, x_{d(\white)}) \in \cI(L)$ along the root
and leaves. The union over all types with $d(\white)$ incoming leaves
is denoted
\[\tM_{d(\white)}^{\br}( L,\ul \Pe) := \bigcup_{\Gamma , \ul x} \tM_{\Gamma}^{\br}( L,\Pe_\Gamma,\ul{x}) .\]
The space ${\tM}_{d(\white)}^{\br}( L,\ul \Pe)$ has a natural topology
for which convergence is a version of Gromov convergence defined in
Chapter \ref{chap:cpt}.  For a type $\Gamma$ and labels $\ul x$ on leaves, the expected dimension of the moduli space $\tM_{\Gamma}^{\br}( L,\Pe_\Gamma,\ul{x})$ of broken maps is given by the index $i^\br(\Gamma, \ul x)$ defined in \eqref{eq:expdim} in Section \ref{sec:index}. For $d=0,1$, we denote by
\begin{equation}
  \label{eq:moduli-d}
\tM_{d(\white)}^{\br}( L,\ul \Pe)_d := \bigcup_{(\Gamma , \ul x): i^\br(\Gamma,\ul x)=d} \tM_{\Gamma}^{\br}( L,\Pe_\Gamma,\ul{x})   
\end{equation}
the union of $d$-dimensional strata with $d(\white)$ incoming boundary edges in the moduli space of broken maps. In this book, we show that the
compactifications of zero and one-dimensional components of the moduli
space $\tM_{d(\white)}^{\br}(L,\ul \Pe)^{< E}$ are manifolds.  For
one-dimensional components, the boundary consists of configurations
with disk bubbling.  For zero and one-dimensional components of
$\tM_{d(\white)}^{\br}(L,\ul \Pe)$ the symmetry tropical group is finite.

\begin{remark}\label{rem:pertsheet} 
  {\rm(Perturbed maps and curve counts)} A perturbed holomorphic map
  $u$ in a zero-dimensional moduli space makes a contribution of
  $\frac 1 {d(\black)!}$ to curve counts. Here $d(\black)$ is the
  number of interior leaves of $u$, and hence is the number of
  intersections of the map $u$ with the stabilizing divisor. To
  justify this multiplicative factor, let us consider the special case
  where for a certain type of map, the moduli space can be regularized
  using a domain-independent perturbation. Then there are $d(\black)!$
  ways of labelling the divisor intersections as
  $z_1,\dots,z_{d(\black)}$.  Each of these labellings is a different
  map in the moduli space, and therefore a corrective factor of
  $\frac 1 {d(\black)!}$ needs to be inserted.
\end{remark}

\section{Fredholm theory for broken maps}\label{sec:fredholm-broken}
%


We introduce a weighted Sobolev space needed for the transversality
 result.
The norm is defined by viewing the domain as having punctures
that map to cylindrical ends in the target. With this Sobolev
completion, we can enforce higher order tangencies with relative
divisors. We focus on surface components of a broken map, since the
transversality and gluing results use a standard norm on tree components.
Moduli spaces of broken holomorphic maps are cut out as zero sets of a Cauchy-Riemann operator on the weighted Sobolev space. 
Later in the section, we define another Cauchy-Riemann operator on the 
Sobolev completion of the space of maps on compact curves, and show that
the linearizations of both Cauchy-Riemann operators have kernels and
cokernels of the same dimension as each other. The latter operator is
often useful for index computations.

The Sobolev norm is defined component-wise for a broken map.  We start
by defining a \em{relative map}, which is a single component of a
broken map.  Tropical nodes in the broken map appear as markings on a
relative map, called \em{tropical markings}; they correspond to
 one-sided edges $\Edge_\trop(\Gamma)$ on the combinatorial type
$\Gamma$ of the relative map.

\begin{definition}\label{def:relmap}
  {\rm(Relative map)} A \em{relative type} is a graph $\Gamma$ with a
  single vertex $v$, a collection $\Edge_\trop(\Gamma)$ of edges each
  with a single end-point, a polytope $P(v) \in \PP$, {direction}s for the
  edges
  \[ \cT(e) \in \t_{P(e),\Z}, \quad e \in \Edge_\trop(\Gamma), \]
  and homology data $ d(v) \in H_2(\ol X_{P(v)})$.  A \em{relative
    map} is a map $u: C^\circ \to \XC_{P(v)}$ from a punctured
  domain $C^\circ:=C \bs \{z_e : e \in \Edge_\trop(\Gamma)\}$ where
  $C$ is a disk or sphere, to the manifold $\XC_{P(v)}$ with
  cylindrical ends, such that at a puncture $z_e$ $u$ is asymptotic to
  a $\cT(e)$-cylinder, and the extension of $\pi_{P(v)} \circ u$ over the
  punctures has homology class $d(v) \in H_2(\ol X_{P(v)})$.
\end{definition}

We introduce the space of smooth relative maps whose Sobolev
completion will be given later.  Let $(C,j)$ be a connected Riemann
surface with a set $\{z_e: e \in \Edge_\trop(\Gamma)\}$ of tropical
markings.  Let
\begin{equation}
  \label{eq:cmnodes}
  C^\circ:=C \bs \{z_e: e \in \Edge_\trop(\Gamma)\}  
\end{equation}
be the curve with the tropical marked points removed.  For the purpose
of defining weighted Sobolev norms, we introduce cylindrical
coordinates
\begin{equation}
  \label{eq:sete}
  (s_e,t_e): U_{z_e}\bs \{z_e\} \to [0,\infty) \times S^1  
\end{equation}
on the neighborhood $U_{z_e} \subset C$ of every puncture $z_e$.  Let
\[\Map_\Gamma(C^\circ,\XC_P)\]
be the set of relative maps, which at a puncture
$z_e \in C \bs C^\circ$ corresponding to an edge
$e \in \Edge_\trop(\Gamma)$ is asymptotically close to a
$\cT(e)$-cylinder
\[ u_{\on{vert},e} : [0,\infty) \times S^1 \to \XC_{P(e)}, \quad (s_e,t_e) \mapsto e^{\cT(e)(s_e+it_e)}x_e\]
for some point $x_e \in \XC_{P(e)}$.  Here, we recall that the
punctured neighborhood at $z_e$ in $C^\circ$ maps to a
$P(e)$-cylindrical end of $\XC_P$ and the $P(e)$-cylindrical end is
identified to $\XC_{P(e)}$, see \eqref{eq:Pcoord}.  An infinitesimal
deformation of a $\cT(e)$-cylinder is given by a vector
$\xi_e \in T_{x_e}\XC_{P(e)}$.  It defines a section
\begin{equation}
  \label{eq:xisec}
  \xi_{e,\cT(e)} := \tfrac \d {\d\tau} (e^{\cT(e)(s+it)}\exp_{x_e}(\tau\xi_e))|_{\tau=0} \in \Gamma(\R_+ \times S^1, u_{\on{vert},e}^*T\XC_{P(e)}) .
\end{equation}
The tangent space of $\Map_\Gamma(C^\circ,\XC_P)$ at $u$ consists of
sections
\[ \xi \in \Gamma(C^\circ,u^*T\XC_P) \]  
which, for any edge $e \in \Edge_\trop(\Gamma)$, have limits 
\begin{equation}
  \label{eq:xie-def}
  \xi_e :=\lim_{z \to z_e}\xi \in T_{x_e}\XC_{P(e)}   
\end{equation}
in the sense that a section $\xi$ is asymptotically close to the section
$\xi_{e,\cT(e)}$ \eqref{eq:xisec} as $z \to z_e$.

\begin{remark}
  The quantity $\xi_e$ defined in \eqref{eq:xie-def} is the
  infinitesimal change in the tropical evaluation map
  $\ev_{z_e}^{\cT(e)}(u) \in \XC_{P(e)}$ (see \eqref{eq:tropev}) when
  the broken map $u$ is infinitesimally deformed by $\xi$. Since the
  tropical evaluation map depends on the choice of cylindrical
  coordinates \eqref{eq:sete} in the punctured neighborhood of $z_e$
  in the domain curve, the quantity $\xi_e$ depends on infinitesimal
  changes in the domain cylindrical coordinates.  On the other hand,
  the projected evaluation map
  $\pi_{\cT(e)}^\perp(\ev_{z_e}(u)) \in \XC_{P(e)}/T_{\cT(e),\C}$ is
  independent of the domain cylindrical coordinates, and the same is
  true of its infinitesimal version
  $\pi_{\cT(e)}^\perp(\xi_e) \in T(\XC_{P(e)} /T_{\cT(e),\C})$.
\end{remark}

We now define the Sobolev norm.  Define a cutoff function
\begin{equation} \label{beta0} \beta \in C^\infty(\R, [0,1]), \quad
  \begin{cases} 
    \beta(s) = 0 & s \leq 0 \\ \beta(s) = 1 & s \ge
    1 \end{cases}. \end{equation}
 Let 
 \begin{equation}
   \label{eq:kappainf}
   \kappa : C^\circ \to \R   
 \end{equation}
 be equal to $\beta(s_e) s_e$ near the puncture corresponding to
 $z_e$, for any edge $e$, and $0$ outside all the edge neighbourhoods.
 For a section $\xi:C^\circ \to E$ of a vector bundle $E$ with
 connection $\nabla$, integers $k \geq 0$, $p>1$ and constant
 $\lam \in (0,1)$ define the $W^{k,p,\lam}$-norm of $\xi$ as
\[ \Mod{\xi}_{W^{k,p,\lam}}^p:=\sum_{0 \leq i \leq k}\int_{C^\circ}|\nabla^i \xi|^p
\exp(\lam \kappa p) \dvol_{C^\circ}. \]
The norm on a section $\xi$ is
\begin{equation}
  \label{eq:xinorm}
  \Mod{\xi}^\circ_\Gamma:=\ssum_e|\xi_e| + \Mod{\xi - \ssum_e\beta(s_e) \pT_u^e \xi_{e,\cT(e)}}_{W^{1,p,\lam}},
\end{equation}
where 
\begin{equation}\label{eq:xielim}
 \xi_e:=\lim_{z \to z_e}\xi \in T_{x_e}\XC_{P(e)}   
\end{equation}
near the puncture $z_e$ as in \eqref{eq:xielim}, $u$ is asymptotic to a vertical cylinder 
$u_{\on{vert},e}$,
\[ \pT_u^e: u_{\on{vert},e}^*T\XC_P
\to u^*T\XC_P \] 
is the parallel transport map along geodesics, and  $\xi_{e,\cT(e)}$  is a section of $u_{\on{vert},e}^*T\XC_P$ as in \eqref{eq:xisec}.

There are similar Sobolev spaces of maps and one-forms.  Let
$\Map^{1,p,\lam}_\Gamma(C^\circ,\XC_P)$ be the Banach completion of
the space of maps $\Map_\Gamma(C^\circ,\XC_P)$ under the norm
\eqref{eq:xinorm}.  That is, for a smooth map
$u \in \Map_\Gamma(C^\circ,\XC_P)$ and a section
$\xi \in \Gamma(C^\circ,u^*T\XX_P)$ satisfying
$\Mod{\xi}_{\Gamma}^\circ<\infty$, the map $\exp_u \xi$ belongs to the
completion $\Map^{1,p,\lam}_\Gamma(C^\circ,\XC_P)$.  Let
\[ \mE^{p,\lam} \to \Map_\Gamma^{1,p,\lam}(C^\circ,\XC_P) \]
be the vector bundle whose fiber
\[ \mE^{p,\lambda}_u =
\Om^{0,1}(C^\circ,u^*T\XC_P)_{L^{p,\lam}}\]
is the space of $(0,1)$-forms with respect to $(j,J)$, where $J$ is a
fixed cylindrical almost complex structure.  The vertical projection
of the linearization of the Cauchy-Riemann operator $\delbar_{j,J}$ at
$u$
\[D_u^\circ:T_u \Map_\Gamma^{1,p,\lam} \to \E_u^{p,\lam} \]
is a Fredholm operator by results of Lockhart-McOwen \cite{loc:ell}. \cwl{See Wehrheim-Woodward \cite{ww:quilt} for the Fredholm property in the Lagrangian case.}

Next, we describe a linearized $\delbar$-operator on the compact curve
$C$, denoted by $D_u$. \label{page:dudef} This operator will be useful
for index computations. It is not used for transversality and gluing
proofs where we stick to $D_u^\circ.$ Given a relative map
$u: C^\circ \to \XX_P$, we define an extension of the pullback bundle
$u^*T\XX_P \to C^\circ$ to a bundle over $C$, as follows.  Define a (partial) 
almost complex orbifold compactification of $\XX_P$, denoted by
$\XX_{P,\cT}$, as the union
\begin{equation}
  \label{eq:xxpt-def}
  \XX_{P,\cT} := \XX_P \cup \cup_{e \in \Edge_\trop(\Gamma)}Y_{\cT(e)}  
\end{equation}
obtained by adding a divisor at infinity $Y_{\cT(e)}$ corresponding to 
every tropical marking $z_e$.\footnote{The space $\XX_{P,\cT}$ is a different compactification
  of $\XX_P$ compared to $\ol \XX_P$ (when it exists).}
The divisor $Y_{\cT(e)}$ is 
defined as
\begin{equation}
  \label{eq:Ycte}
  Y_{\cT(e)}:=\{[x]=T_{\cT(e),\C} x \ : \ x \in \XX_P \},  
\end{equation}
where each $[x]$ is the limit point of a $T_{\cT(e),\C}$-orbit, that
is, $\lim_{s \to \infty} e^{(s+it)\cT(e)}x=[x]$, and
$Y_{\cT(e)}=Y_{\cT(e')}$ if $\cT(e)=\lam \cT(e')$ for some
$\lam \in \R_{>0}$.  The map $u$ extends over all tropical markings
$z_e$ to yield $u : C \to \XX_{P,\cT}$, with $u(z_e) \in Y_{\cT(e)}$.
The extension of $u^*T\XX_P \to C^\circ$ is the pullback bundle
$u^*T\XX_{P,\cT} \to C$.  Recall from Definition \ref{def:primdirection}
that for any tropical marking $z_e \in C \bs C^\circ$, the {direction}
$\cT(e) \in \t_{P(e),\Z}$ is the product
\begin{equation}
  \label{eq:prim-direction}
  \cT(e)=\mu_e \cT(e)_{\prim},  
\end{equation}
\index{Primitive direction} \index{Multiplicity of an edge} of the primitive direction $\cT(e)_{\prim} \in \t_{P(e),\Z}$ and a multiplicity
$\mu_e \in \Z_{\geq 1}$.  
Denote by 
\begin{equation}
  \label{eq:Dudef}
  D_u: W^{k,p}(\Om^0_{\ul \mu}(C,u^*T\XX_{P,\cT})) \to W^{k-1,p}(\Om^{0,1}_{\ul\mu
    -1}(C,u^*T\XX_{P,\cT}))
\end{equation}
the linearization of the $\delbar$-operator defined on the space of sections on the compactification $\XX_{P,\cT}$. 
Here, 
$\ul \mu=(\mu_e \in \N)_e$, and the subscript in the domain and
codomain indicate the order of vanishing of the sections at the
tropical markings $z_e$, $e \in \Edge_\trop(\Gamma)$ in the direction
normal to the divisor at infinity $Y_{\cT(e)}$; and the Sobolev
indices are such that
\[ k \in \Z_{>0}, \quad p>1, \quad k>\mu_e + 2/p \enspace \forall e \]  
 so that the sections with prescribed orders of vanishing at the divisors are well-defined.

\begin{proposition} \label{prop:sameind}
  Let $u \in \Map_\Gamma(C^\circ,\XX_P)$ be a relative holomorphic map.   
 There are isomorphisms
  \[\ker(D_u) \simeq \ker(D_u^\circ), \quad \coker(D_u) \simeq
  \coker(D_u^\circ).\]
\end{proposition}

\begin{proof}
  Since the proof is by a local computation near each of the special
  points, it is enough to prove the result for a map whose domain is a
  single curve component with a single tropical marking $z_e$.  By the
  description of $D_u$ (see \eqref{eq:Dudef}), the curve $u$ extends
  over $z_e$ to yield $u:C \to \XX_{P,\cT}:=\XX_P \cup Y$ (as in
  \eqref{eq:xxpt-def}), with $u(z_e)$ mapping to the divisor at
  infinity $Y$, and $u$ intersecting the divisor $Y$ at $z_e$ with
  multiplicity $\mu \in \N$. We abbreviate $\hat \XX_P:=\XX_{P,\cT}$.
  We assume that the marked point is $0 \in C$.
 
  The proof is based on the fact that in a neighborhood of the divisor
  the tangent space splits into a vertical and a horizontal subspace.
  Indeed, because of the cylindrical almost complex structure, there
  is a neighborhood $U_Y \subset \hat \XX_P$ of $Y$ for which there is a
  projection $\pi_Y: U_Y \to Y$ with holomorphic fibers. The tangent
  space splits into a $J$-holomorphic horizontal and vertical part: 
  \begin{equation}
    \label{eq:split}
     T\hat \XX_P|U_Y \simeq V
  \oplus H, \quad H:=\pi_Y^*TY, \quad V:=\ker(d\pi_Y) .
  \end{equation}
  Under
  the splitting \eqref{eq:split},
  \begin{equation}
    D_u =     \label{eq:dusplit}
    \bigl(
  \begin{smallmatrix}
    \delbar & A\\
    0 & D_u^Y
  \end{smallmatrix}
  \bigr) 
  \end{equation}
  in a neighbourhood $U_0 \subset C$ of the marked point.  Here
  $\delbar: \Gamma(U_0,u^*V) \to \Om^{0,1}(U_0,u^*V)$ is the standard
  Cauchy-Riemann operator,
  $A: \Gamma(U_0,u^*H) \to \Om^{0,1}(U_0,u^*V)$ is multiplication with
  a tensor and
  \[D_u^Y:\Gamma(U_0, u^*H) \to \Om^{0,1}(U_0,u^*H)\]
  is the lift of the linearized
  operator $D_{u_Y}$, where $u_Y:=\pi_Y \circ u$.

  The correspondence for the kernels follows by decay estimates.  For
  an element $\xi \in W^{k,p}(C,u^*T\hat \XX_P)$, we denote the horizontal
  part and vertical part by
\[ \xi_h \in  W^{k,p}(U_0,u^*H), \quad  \xi_v \in W^{k,p}(U_0,u^*V). \]
in the neighbourhood $U_0 \subset C$ of $0$.
Denote by $|\cdot|$ resp. $|\cdot|^\circ$ the ordinary metric on $\hat \XX_P$
resp. the cylindrical metric on $\XX_P$. Then, we have
  \begin{equation}\label{eq:cyl2r}
    |\xi_h(z)|^\circ \sim |\xi_h(z)|, \quad |\xi_v(z)|^\circ \sim |\xi_v(z)|/|u(z)|, \quad z \neq 0,  
  \end{equation}
  where the norm on $u(z)$ is a norm on the fiber of the projection
  $U_Y \to Y$.  Since 
\[ |\xi_h| \leq c, \quad |\xi_v(z)| \leq c|z|^\mu, \quad |u(z)| \sim
c|z|^\mu , \]  
we conclude that $|\xi_h|^\circ$ and $|\xi_v(z)|^\circ$ are uniformly
bounded and have a limit as $z \to 0$.  Since $\xi$ is smooth on $C$,
the convergence to $\xi(0)$ is exponential on $C^\circ$:
  \[ |\xi(z) - \xi(0)|^\circ \sim c|z| \sim ce^{-s}, \]
  where $(s,t)$ are cylindrical coordinates on $\P^1 \bs \{0\}$ near
  the puncture, and
  \[ |\nabla \xi_h(z)|^\circ, |\nabla \xi_v(z)|^\circ \sim ce^{-s}\]
  since the derivative is with respect to the cylindrical coordinates on $C^\circ$. 
    We conclude
  \[\xi \in W^{k,p}_\Gamma(C,\hat \XX_P) \implies \xi - \xi(0) \in
  W^{1,p,\lam}(C^\circ,u^*T\XX_P)\]
  for any $0<\lam<1$, which implies $\ker D_u \subset \ker D_u^\circ$.
  The converse $\ker D_u^\circ \subset \ker D_u$ is proved using
  removal of singularity, elliptic regularity, and the estimate \eqref{eq:cyl2r}.

  A similar argument holds for the cokernels.  We first prove the
  inclusion $\coker(D_u) \subset \coker (D_u^\circ)$.  Consider
  $\eta \in \coker (D_u)$. We view $\coker(D_u)$ as a subspace in the
  $(W^{k-1,p})^\dual$-completion of $\Om^{1,0}(C,u^*(T^*\hat \XX_P))$,
  which is a space of distributions.    (Here we use  $\Om^{1,0}(u^*(T^*\hat \XX_P))$
  instead of $\Om^{0,1}(u^*T\hat \XX_P)$ to avoid making choices of metrics on
  $C$ and $\hat \XX_P$.)  By elliptic regularity, the distribution
  $\eta$ is represented by a smooth section in the complement of
  marked points. So we focus attention in a neighbourhood
  $U_0 \subset C$ of $0$.  For any $W^{k,p}$-section
  $\xi:C \to u^*T\hat \XX_P$ that is supported in $U_0$, and is
  vertical, we have $z^\mu \xi \in W^{k,p}_\Gamma(C,u^*T\hat
  \XX_P)$. So for any such $\xi$,
  \[0=\int_C(\delbar(z^\mu \xi),\eta_v).\]
  Therefore, $\delbar(z^\mu \eta_v)=0$ weakly in $U_0$ and so,
  $z^\mu \eta_v$ can be represented by a smooth function. Next, using
  the split form \eqref{eq:dusplit}, we observe that any
  horizontal section $\xi$
  that is supported in $U_0$ satisfies 
  \[(A\xi, \eta_v) + (D_u^Y\xi,\eta_h)=0,\]
  and so, $A^*\eta_v + (D_u^Y)^*\xi_h=0$ weakly in $U_0$.  The tensor
  $A$ vanishes to order $\mu$ at $0 \in C$ in the $|\cdot|$-norm, as a
  consequence of the transformation relation \eqref{eq:cyl2r} and the
  fact that $A$ is bounded in the $|\cdot|^\circ$ norm.  So,
  $A^*\eta_v$ is smooth in $U_0$.  By elliptic regularity $\xi_h$ is
  smooth in $U_0$.  Finally, $\eta$ is in $\coker D_u^\circ$ because of
  the following transformations valid in $U_0$ :
  \begin{equation}\label{eq:coker-tr}
    |\eta_h(z)|^\circ \sim |\eta_h(z)|, \quad |\eta_v(z)|^\circ \sim |\eta_v(z)|\cdot|u(z)|, \quad z \neq 0.
  \end{equation}
  Indeed, $|\eta(z)|^\circ$ is bounded, and therefore is in
  $L^{p^*,-\lam}$ the dual space of $L^{p,\lam}$.  The reverse
  inclusion $\coker(D_u^\circ) \subset \coker(D_u)$ follows formally
  from the inclusion relation
  \[W^{k-1,p}_{\mu-1}(\Om^{0,1}(C,u^*T\hat \XX_P)) \to L^{p,\lam}(\Om^{0,1}(C^\circ,u^*T\XX_P))\]
  between the target spaces of $D_u$ and $D_u^\circ$. 
\end{proof}

\section{The index of a broken map}\label{sec:index}

The index of a broken map is the expected dimension of the moduli
space containing the broken map.  In this section, we give an
expression for the index in terms of the Fredholm index of the
linearized Cauchy-Riemann operator defined in the last section.  We
also prove that collapsing tropical edges in the tropical graph of the
broken map does not change the expected dimension of the moduli space.

We start by stating the expected dimension of the moduli space of
relative maps, which are single components of broken maps. We recall
that the \em{boundary Maslov index $I(E,F)$}
\index{Index! Maslov index, $I$}
\index{Index! boundary Maslov index, $I$}
is an integer
assigned to a pair $(E,F)$ consisting of a complex vector bundle $E$
on a Riemann surface $C$ with boundary, and a totally real subbundle
$F \subset E|_{\partial C}$ on the boundary (see \cite[Appendix
C]{ms:jh}).  In case $\partial C=\emptyset$, the Maslov index
$I(E,\emptyset)$ is twice the first Chern number $c_1(E)$.

\begin{lemma}\label{lem:adjmaslov}
  {\rm(Expected dimension for relative maps)}
  \index{Index!Adjusted Maslov index $I_\adj$}
  Let $u$ be a relative map of type
  $\Gamma_v$ with a collection of tropical markings $z_e$
  corresponding to $e \in \Edge_\trop(\Gamma_v)$, each with
  multiplicity $\mu_e \in \N$, and with simple intersections with the
  stabilizing divisor $\DD$.  Then, the expected dimension of the
  moduli space $\M_{\Gamma_v}(\XX)$ of relative maps of type
  $\Gamma_v$ is
\begin{equation}
    \label{eq:igammav}
    i(\Gamma_v)=
    \begin{cases}
      (d(\white)+1)  + i_\Morse(\ul x) + I_\adj(\Gamma_v)  - \dim \Aut(\mathbb D^2), & \text{$\Gamma_v$ is a disk },\\
      \dim(X) + I_\adj(\Gamma_v) - \dim \Aut(\P^1), & \text{$\Gamma_v$ is a sphere}
    \end{cases}
  \end{equation}
  where $I_\adj(\Gamma_v)$ is the {\em adjusted Maslov index} defined as 
  \begin{multline}
    \label{eq:adjmas2}
     {\text{\rm\textup{(Adjusted Maslov index)}}} \quad
     I_\adj(\Gamma_v):=I_\adj(u):=I((u^*T\XX_{P(v),\cT}), (\partial u)^*TL)\\
     - \sum_{e \in \Edge_\trop(\Gamma)}2(\mu^\trop_e-1);
  \end{multline}
  the Morse index of the tuple is the difference
    \[ i_\Morse(\ul{x}) := i_\Morse(x_0)- \ssum_{i =1}^{d(\white)}
      i_\Morse(x_i), \]
    and $i_\Morse(x_i)$ is the dimension of the stable manifold of
    $x_i$ in $L$.  Furthermore, 
    \begin{equation}
    \label{eq:adjmas}
    I_\adj(u) =     \ind(D_u) +
    2|\Edge_\trop(\Gamma)|.
  \end{equation}
\end{lemma}

\begin{proof}
  We recall from the paragraph preceding \eqref{eq:Dudef} that,
  assuming the tropical markings are fixed, the moduli space of
  relative maps of type $\Gamma_v$ is cut out as the zero set of $D_u$
  from the Banach space of sections
  $W^{k,p}_{\ul \mu}(C, u^*T\XX_{P,\cT})$ that vanish to order $\mu_e$
  in the direction normal to the divisor $Y_{\cT(e)}$ at the tropical
  marking $z_e$, and therefore the expected dimension is
  $\ind(D_u)$. Since the tropical markings are free to move on the
  domain, the expected dimension of $\M_{\Gamma_v}(\XX)$ is given by
  the first expression in \eqref{eq:adjmas}.  To obtain the second
  expression, we observe that the expected dimension of moduli space
  of holomorphic maps $u: C \to \XX_{P,\cT}$ is the Maslov index
  $I(u^*T\XC_{P,\cT}, (\partial u)^*TL)$, and for any
  $e \in \Edge_\trop(\Gamma_v)$, the choice of $z_e$ on $C$ adds $2$
  to the expected dimension, and the requirement that there is a
  tangency of order $\mu_e$ with the divisor $Y_{\cT(e)}$ at $z_e$
  subtracts $2\mu_e$.
\end{proof}

\begin{example}
  \index{Index!Adjusted Maslov index $I_\adj$} Suppose
  $\ol X_{P(v)}=\P^2$ with a Hamiltonian $(S^1)^2$-action,
  $\XB_{P(v)}$ is the complement of the 3 torus-invariant divisors,
  $L \subset \ol X_{P(v)}$ is a toric Lagrangian, and
  $u: (D,\partial D) \to (\ol X_{P(v)},L)$ is a disk through the point
  $[1:0:0]$ whose boundary Maslov index is $4$. In contrast, the
  adjusted Maslov index $I_\adj(u)$ is $2$, because the adjusted
  Maslov index accounts for the constraint that the disk passes
  through the point $[1:0:0]$ which cuts down the index by $2$.  The
  reason for the discrepancy is that the pullback bundle
  $u^*T\ol X_{P(v)}$ is different from $u^*T\XC_{P(v),\cT}$ since a
  tropical marking maps to the intersection of two divisors in
  $\ol X_{P(v)}$, which is replaced by a single divisor in
  $\XC_{P(v),\cT}$.
\end{example}

For a type of broken maps the expected dimension of the moduli space
is the same as that of the glued type of unbroken maps. We define
\em{glued type}:

  \begin{definition}
    {\rm(Glued type)} \index{Glued type $\Gamma_{\on{glue}}$} Let
    $\Gamma$ be the combinatorial type of a broken map. Then the \em{
      glued type} $\Gamma_\glue$ is the type of the unbroken map
    obtained by collapsing the tropical edges
    $e \in \Edge_\trop(\Gamma)$ in $\Gamma$.
\end{definition}

\index{Index! Morse index of a critical point $i_\Morse(x)$}
\begin{proposition}\label{prop:expdim}
  {\rm(Expected dimension)} Suppose $\Gamma$ is the combinatorial type
  of a broken map.

  \begin{enumerate}
  \item The expected dimension of the moduli space
    $\tM_\Gamma^\br(L,\ul x)$ of broken maps of type $\Gamma$ with
    limits $\ul x \in \cI(L)^{d(\white)+1}$ along the root and
    leaves is equal to
    \begin{equation}
      \label{eq:expdim}
      i^\br(\Gamma,\ul x):=  \sum_{v \in \Ver(\Gamma)}i(\Gamma_v)
      -(\dim(X)-2)\Edge_\trop(\Gamma) - c(\Gamma),
    \end{equation}
    where $i(\Gamma_v)$ is the expected dimension of the moduli space
    of relative maps of type $\Gamma_v$ (see \eqref{eq:igammav}), and
    \[ c(\Gamma) := 2|\Edge_{\internal,\black}(\Gamma)|+
      |\Edge_{\white,-}^0| + |\Edge_{\white,-}^\infty| + 2\ssum_{e \in
        \Edge_{\black,\rightarrow}(\Gamma)}(\mu_{\bD}(e)-1) \]
    is a factor accounting for internal disk and sphere nodes in
    $\Gamma$, and the last term accounts for higher order
    intersections with the stabilizing divisor.
  \item \label{part:glue-dim}
    Collapsing tropical edges does not affect the index, that is,
    denoting by $i(\Gamma_{\on{glue}},\ul x)$ the expected dimension
    of the moduli space of unbroken maps of type $\Gamma_{\on{glue}}$ with limits
    $\ul x \in \cI(L)^{d(\white)+1}$ along the root and leaves,
    \[ i^\br(\Gamma,\ul x)=i(\Gamma_{\on{glue}},\ul x) .\]
    \index{Index! of a broken map $i^\br$}
    \index{Index! of an unbroken map $i$}
  \end{enumerate}
\end{proposition}
\begin{notation}
For a broken map type $\Gamma$ and leaf labels $\ul x$,
the expected dimension of the reduced moduli space $\M_{\Gamma,\red}^\br(\XX,L,\ul x)$ is
\begin{equation}
  \label{eq:i-red-br}
  i^\br_\red(\Gamma,\ul x):=i^\br(\Gamma,\ul x) - \dim(T_\trop(\Gamma)).
\end{equation}
\end{notation}
\begin{proof}[Proof of Proposition \ref{prop:expdim}]
  The formula \eqref{eq:expdim} for the expected dimension of the
  moduli space of broken maps follows in a straightforward way: the
  expected dimension is the sum of expected dimensions $i(\Gamma_v)$
  for each of its components, viewed as relative maps minus the
  codimension of the matching conditions at nodes, and higher order
  tangency conditions with the stabilizing divisor. The matching
  condition at a tropical node $w_e$ is a condition of codimension
  $(\dim(X)-2)$, since it requires that the projected tropical
  evaluations for both nodal lifts agree (see \eqref{eq:trop-matching}), which accounts for
  the second term in \eqref{eq:expdim}. The last term incorporates
  matching conditions at non-tropical nodes and higher intersections
  with the stabilizing divisor.

  To prove \eqref{part:glue-dim}, we proceed by obtaining a relation
  between the Maslov indices of the components of the broken map and
  the Maslov index of the glued type.  For simplifying exposition, we
  first consider the case of a broken map $u$ of type $\Gamma$ with
  two components $u_{v_+}$, $u_{v_-}$ connected by a tropical edge
  $e=(v_+,v_-)$ with nodal lifts $w_e^\pm \in C_{v_\pm}$.  The map
  $u_{v_\pm} : C_{v_\pm} \bs \{w_e^\pm\} \to \XX_{P(v_\pm)}$ extends
  to a map on the compactified curve, denoted by
  \[\ol u_{v_\pm} : C_{v_\pm} \to
    \XX_{P(v_\pm),\cT},\]
  mapping to the compactification $\XX_{P(v_\pm),\cT}$ of
  $\XX_{P(v_\pm)}$ (from \eqref{eq:xxpt-def}).  The evaluations
  $\ol u_{v_\pm}(w_e^\pm)$ lie in $ Y_{\cT(e)}$, where
  $Y_{\cT(e)} \subseteq \XX_{P(v_\pm),\cT} \bs \XX_{P(v_\pm)}$ is a
  divisor at infinity as in \eqref{eq:Ycte}.  Recall from definition
  \ref{def:primdirection} that the edge {direction} $\cT(e) \in \t_{P(e),\Z}$ is the
  product
  \[\cT(e)=\mu^\trop_e\cT(e)_{\prim}\]
  of a primitive integer vector $\cT(e)_\prim \in \t_{P(e),\Z}$ and a
  positive integer $\mu^\trop_e$.  Because of the cylindrical complex
  structure on the ends, a neighbourhood $U^\pm_{Y_{\cT(e)}}$ of
  $Y_{\cT(e)}$ in $\XX_{P(v_\pm),\cT(e)}$ has a projection map
  \[\pi_\pm:U^\pm_{Y_{\cT(e)}} \to Y_{\cT(e)} \]
  with holomorphic fibers, and the tangent space splits as the sum of
  horizontal and vertical sub-bundles
  \[ TU^\pm_{Y_{\cT(e)}} = H \oplus V_\pm , \quad H
    =\pi_\pm^*TY_{\cT(e)}, \quad V_\pm=\ker(d\pi_\pm).\]
  Topologically, the maps $\ol u_{v_+}$, $\ol u_{v_-}$ are obtained by
  cutting $u_{\glue}$ to yield $u_{v_+}$, $u_{v_-}$ and adding a
  capping disk to either side. The horizontal bundle $(u_{v_\pm})^*H$
  extends trivially over the capping disk, but for $(u_\pm)^*V_\pm$,
  the bundle over the capping disk is glued in with a $\mu_e^\trop$ twist.
  Thus gluing at the node $w_e$ has the effect of adding the Maslov
  indices on both sides and subtracting $4\mu_e^\trop$. 
  
  Next, consider a general broken map type $\Gamma$, that is, it may have more than two
  components. Using the preceding observation, and summing over
  contributions from all tropical edges $e \in \Edge_\trop(\Gamma)$,
  we obtain
  \begin{multline}
    \label{eq:mascalc1}
    \sum_{v \in \Ver(\Gamma)}I(u_v^*T\XC_{P(v),\cT},(\partial u_v)^*TL) - \sum_{e \in \Edge_\trop(\Gamma)}4\mu_e^\trop\\ = \sum_{v \in \Ver(\Gamma_\glue)}I(u_v^*TX,(\partial u_v)^*TL).  
  \end{multline}
  Rewriting the left hand side of \eqref{eq:mascalc1} using \eqref{eq:adjmas2}, we get
  \begin{equation}
      \label{eq:maslovglue}
     \sum_{v \in \Ver(\Gamma)}I_\adj(\Gamma_v) - 4 \Edge_\trop(\Gamma)=\sum_{v \in \Ver(\Gamma_\glue)}I(u_v^*TX,(\partial u_v)^*TL). 
  \end{equation}

  The expected dimension of $\tM_{\Gamma}^\br(L, \ul \Pe,\ul{x})$ can
  be related to the index of the glued type by using the relation between the Maslov indices from the previous paragraph. 
  To lighten notation, let us assume that all nodes are tropical,
  since internal nodes can be accounted for in an obvious way, and
  that all intersections with the stabilizing divisor $\DD$ are
  simple, and consequently $c(\Gamma)=0$. The assumption also implies
  that there is a single disk component and $\Edge_\trop(\Gamma)$
  number of sphere components. Invoking \eqref{eq:expdim}, followed by
  \eqref{eq:igammav}, and then \eqref{eq:maslovglue}, we get
  \begin{multline*}
    i^\br(\Gamma,\ul x)=\sum_{v \in \Ver(\Gamma)} i(\Gamma_v) - (\dim(X)-2)|\Edge_\trop(\Gamma)|\\
    =\sum_{v \in \Ver(\Gamma)} I_\adj(\Gamma_v) + (d(\white) -2) + i_\Morse(\ul x) - 4|\Edge_\trop(\Gamma)|=i(\Gamma_\glue,\ul x).
  \end{multline*}
\end{proof}

\begin{figure}[h]
  \centering \scalebox{.8}{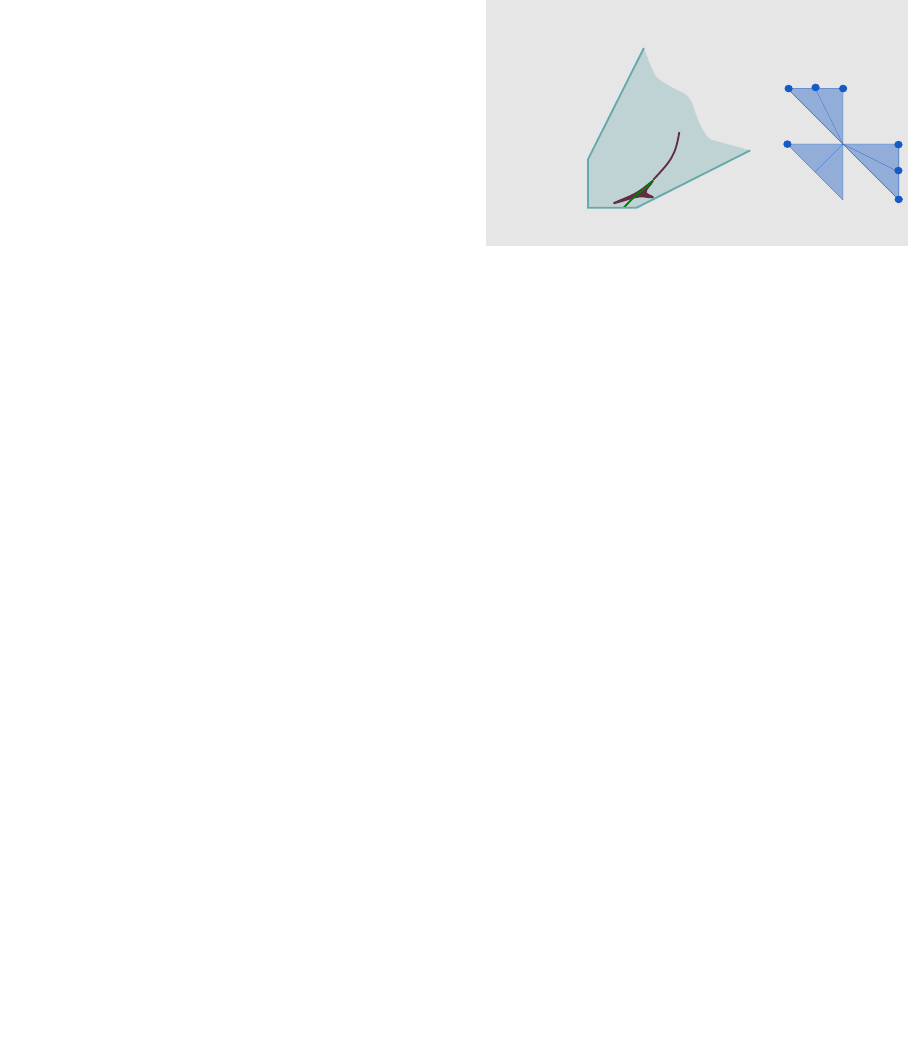}
  \caption{Broken maps corresponding to the Maslov index four disk class $[u_\glue]=2[\delta_{E_2}] +[E_2]=[\delta_{E_1}] + [\delta_D]$.}
  \label{fig:maslov4}
\end{figure}

\begin{remark}\label{rem:on-adj}
  {\rm(On the adjusted Maslov index)} Consider a relative map
  $u_v: S_v \to \ol \XX_P$ whose target space $\ol \XX_P$ is a
  manifold.  The adjusted Maslov index of $u$ is then the ordinary
  Maslov index $I(u)$ adjusted by the constraint that for any tropical
  marking $w_e$ corresponding to an edge $e \ni v$, the map has the
  appropriate tangency with relative divisors. In particular, if at
  $w_e$, the sum of orders of tangency with all relative divisors is
  given by $\mu_{e,v} \in \N$, then, the adjusted Maslov index is
  \begin{equation}
    \label{eq:adj-in-mfd}
    I_\adj(u)=I(u) - 2\sum_{e: v \in e}(\mu_{e,v}-1).   
  \end{equation}
\end{remark}

\begin{example}{\rm(Adjusted Maslov index
    computation)}\label{ex:adj-mas}
  In this example, we consider the multiple cut of the cubic surface
  from Section \ref{sec:cubic-intro}.  We consider the cubic surface
  $X$ with a deformed symplectic form so that it is a toric manifold
  (see \eqref{eq:Xeps-def}), $L \subset X$ is a toric Lagrangian, and
  the multiple cut is as in Figure \ref{fig:cubiccut}.  We list all
  broken maps whose gluing has homology class
  $2[\delta_{E_2}] + [E_2]$, and write down the adjusted Maslov
  indices of all the components. Here $\delta_{E_2} \in H_2(X,L)$ is
  the class of the disk of Maslov index two with a single intersection
  with the short divisor $E_2$, see the center of Figure
  \ref{fig:maslov4}. There are seven types $\Gamma$ of broken disks
  whose gluing has homology class $2[\delta_{E_2}] + [E_2]$ shown in
  Figure \ref{fig:maslov4}. We state the adjusted Maslov indices of
  the components of some of these types.
  \begin{enumerate}
  \item In $\Gamma^1$ the adjusted Maslov indices (defined in \eqref{eq:adjmas2}) are
    \[I_\adj(v_1)=4, \quad I_\adj(v_2)=4,\]
    which are equal to the ordinary Maslov indices of the components.
  \item \label{part:adj-mas3} The edge {direction}s in $\Gamma^3$ are
    \[\cT(v_1,v_2)=(0,-2), \quad \cT(v_2,v_3)=(-1,0).\]
    The adjusted Maslov indices of the vertices in $\Gamma^3$ are
    \[I_\adj(v_1)=2, \quad I_\adj(v_2)=6, \quad I_\adj(v_3)=4.\]
    Indeed $u_{v_2}$ is a curve with self-intersection number $2$ in
    $\ol \XC_{P(v_2)}$ which is the second Hirzebruch surface $F_2$,
    and therefore, $c_1(u_{v_2}^*T\ol \XC_{P(v_2)})=4$. The ordinary
    Maslov indices are $I(v_1)=4$, $I(v_2)=8$. The adjusted Maslov
    indices $I_\adj(v_1)$, $I_\adj(v_2)$ are each $2$ lower because
    the node between $v_1$, $v_2$ has an intersection multiplicity of
    $2$ with the relative divisor. To calculate $I_\adj(v_3)$, we need
    to choose an alternate almost complex compactification of
    $\XB_{P(v_3)}$ since $u_{v_3}$ has a nodal point $w_e$
    corresponding to the edge $e=(v_2,v_3)$ mapping to an orbifold
    singularity. We choose the compactification to be a resolution of
    the $A_1$-singularity as in Figure \ref{fig:a2resolve} so that the
    nodal point $w_e$ lies in a single divisor, from where, we
    conclude $I_\adj(v_3)=4$.
    \begin{figure}[h]
      \centering \scalebox{.8}{
\begingroup%
  \makeatletter%
  \providecommand\color[2][]{%
    \errmessage{(Inkscape) Color is used for the text in Inkscape, but the package 'color.sty' is not loaded}%
    \renewcommand\color[2][]{}%
  }%
  \providecommand\transparent[1]{%
    \errmessage{(Inkscape) Transparency is used (non-zero) for the text in Inkscape, but the package 'transparent.sty' is not loaded}%
    \renewcommand\transparent[1]{}%
  }%
  \providecommand\rotatebox[2]{#2}%
  \newcommand*\fsize{\dimexpr\f@size pt\relax}%
  \newcommand*\lineheight[1]{\fontsize{\fsize}{#1\fsize}\selectfont}%
  \ifx\svgwidth\undefined%
    \setlength{\unitlength}{176.54451866bp}%
    \ifx\svgscale\undefined%
      \relax%
    \else%
      \setlength{\unitlength}{\unitlength * \real{\svgscale}}%
    \fi%
  \else%
    \setlength{\unitlength}{\svgwidth}%
  \fi%
  \global\let\svgwidth\undefined%
  \global\let\svgscale\undefined%
  \makeatother%
  \begin{picture}(1,0.27829137)%
    \lineheight{1}%
    \setlength\tabcolsep{0pt}%
    \put(0,0){\includegraphics[width=\unitlength,page=1]{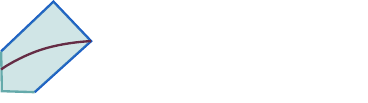}}%
    \put(0.04588631,0.06033051){\color[rgb]{0,0,0}\makebox(0,0)[lt]{\lineheight{1.25}\smash{\begin{tabular}[t]{l}$v_3$\end{tabular}}}}%
    \put(0,0){\includegraphics[width=\unitlength,page=2]{a2resolve.pdf}}%
    \put(0.32071675,0.21315675){\color[rgb]{0,0,0}\makebox(0,0)[lt]{\lineheight{1.25}\smash{\begin{tabular}[t]{l}Resolve $A_1$\\singularity \end{tabular}}}}%
    \put(0.82516978,0.06026){\color[rgb]{0,0,0}\makebox(0,0)[lt]{\lineheight{1.25}\smash{\begin{tabular}[t]{l}$v_3$\end{tabular}}}}%
  \end{picture}%
\endgroup%
}
      \caption{Orbifold singularities can be eliminated by changing
        the compactification.}
      \label{fig:a2resolve}
    \end{figure}
  \item The edge {direction}s in $\Gamma^2$ are
    \[\cT(v_1,v_2)=(0,-2), \quad \cT(v_2,v_3)=(-1,0), \quad
      \cT(v_3,v_4)=(-1,0),\]
    and the adjusted Maslov indices are
    \[I_\adj(v_1)=2, \quad I_\adj(v_2)=6, \quad I_\adj(v_3)=4, \quad
      I_\adj(v_4)=4.\]
    The adjusted Maslov indices of $v_3$ and $v_4$ are computed by
    passing to a toric resolution as in as in Figure
    \ref{fig:a2resolve}.
  \item In $\Gamma^5$ the edge $e = (v_1,v_2)$ has multiplicity
    $2$. The adjusted Maslov indices are
    \[I_\adj(v_1)=2, \quad I_\adj(v_2)=8, \quad I_\adj(v_3)=2.\]
    Note that the ordinary Maslov indices are different:  $I(v_1)=4$,
    $I(v_2)=10$.  The value of $I_\adj$ is lower for $v_1$, $v_2$
    because it accounts for the intersection multiplicity of $2$ at
    the node $w_e$ corresponding to the edge $e=(v_1,v_2)$. We point
    out that the map $u_{v_3}$ is homologous to the $(-1)$-divisor
    $E_2'$ which is obtained by cutting the $(-2)$-divisor $E_2$, see
    Figure \ref{fig:27polys} for notation. Since the almost complex
    structure is a perturbation of the toric almost complex structure,
    the image of $u_{v_3}$ does not lie on the toric divisor $E_2'$.

  \item In a broken map of type $\Gamma^7$, the component $u_{v_1}$
    has two distinct intersections with the relative divisor, each of
    multiplicity $1$.  For each of the vertices $v_i$ in $\Gamma^0$ the
    Maslov index $I(u_{v_i})$ is equal to the adjusted Maslov index
    $I_\adj(u_{v_i})$, and
    \[I_\adj(v_1)=4, \quad I_\adj(v_2)=6, \quad I_\adj(v_3)=2, \quad
      I_\adj(v_4)=4.\]
  \end{enumerate}
  For each type $\Gamma^i$, the Maslov index sum formula
  \eqref{eq:maslovglue} implies that the Maslov index
  $I(\Gamma^i_\glue)$ 
of the glued disk is $4$.  We discuss the moduli
  spaces of broken maps of types $\Gamma_i$ in the following Remark
  \ref{rem:maslov4-moduli}.
\end{example}

\begin{remark}\label{rem:maslov4-moduli} 
  We conjecturally describe the moduli space of  broken disks in the
  multiply cut cubic surface whose glued homology class is
  $[\delta_{E_1}] + [\delta_D]$. We refer to Figure \ref{fig:maslov4}
  for notations.  We 
  cannot 
  make this discussion rigorous because
  under the Cieliebak-Mohnke perturbation scheme, we can only
  compactify moduli spaces of dimension at most $1$.  For
  $i=1,\dots,7$ we denote by
  \[\M_i:=\M_{\Gamma^i}^\br(\XX)/T_\trop(\Gamma^i)\]
  the reduced moduli space of broken disks whose tropical graph is
  $\Gamma^i$ and which have a point constraint on the boundary. In
  other words, the disks do not have any inputs and the output marking
  is required to map to the maximum point of the Morse function on the
  Lagrangian $L$. Since in each of the types the glued disk has Maslov
  index $4$, we get $\dim(\M_{\Gamma^i}^\br(\XX))=2$ for all $i$.  For
  odd $i$, the tropical graph is rigid and $T_\trop(\Gamma^i)$ is
  finite.  For even $i$, $\dim_\C(T_\trop(\Gamma^i))=1$, and so  $\M_i$
  is $0$-dimensional.  For any even $i$,
  \[\M_i \subset \ol \M_{i-1} \bs \M_{i-1} \quad \text{and} \quad \M_i \subset \ol \M_{i+1} \bs \M_{i+1},\]
  because there are tropical edge collapse morphisms
  $\Gamma^{i+1} \to \Gamma^i$, $\Gamma^{i-1} \to \Gamma^i$ (see
  Definition \ref{def:tropcol1} and Example
  \ref{ex:maslov4-collapse}).  Besides having a non-compact end, the
  moduli spaces $\M_1$, $\M_7$ also have codimension one boundary
  components consisting of configurations with a broken edge
  $e \in \Edge_-(\Gamma^i)$, $\ell(e)=\infty$.  The moduli space of
  broken disks of class $[\delta_{E_1}] + [\delta_D]$ is
  \[\ol \M_1 \cup_{\M_2} \ol \M_3 \cup_{\M_4} \ol \M_5 \cup_{\M_6} \ol \M_7. \]
  Here we take Parker's viewpoint \cite{bp8} that the moduli spaces of broken
  maps are themselves broken manifolds (which in his papers, are
  called ``exploded manifolds''). The spaces $\ol \M_i$ for odd $i$ are
  the top-dimensional cut spaces, and $\M_i$ for even $i$ are the cut
  spaces of codimension $2$.
\end{remark}

\section{Transversality}

In this section, we show that for a set of comeager domain-dependent
perturbations, moduli spaces of broken maps with index at most one are
transversally cut out.  The perturbation scheme we use here only
achieves transversality for certain combinatorial types, as in
Cieliebak-Mohnke \cite{cm:trans}.  In particular, the Cieliebak-Mohnke
\cite{cm:trans} perturbation scheme 
cannot 
achieve transversality on
\em{crowded components}, which are components where the map is
constant and which contain more than one marking.  The moduli space of
such maps cannot be transversally cut out, since the constraint that
the first marking in such a component $S$ maps to the stabilizing
divisor $D$ together with the fact that $u|S$ is a constant map
guarantees that the second marking does as well.  In the context of
broken maps, the ``constant'' condition in the definition of
crowdedness needs to be replaced by ``horizontally constant''.  Recall
that a component $u_v: S_v^\circ \to \XX_{P(v)}$ is horizontally
constant if its projection to $X_{P(v)}$ is constant. We define
crowdedness for broken maps. \label{morexp}

\begin{definition} \label{def:crowded} The combinatorial type $\Gamma$
  of a broken map $u:C \to \XX$ is \em{crowded} if there exists a
  connected subgraph $\Gamma' \subset \Gamma$ such that $u|S_v$ is
  horizontally constant (see Definition \ref{def:iso-stab} \eqref{part:horizconst}) on all
  vertices $v \in \Ver(\Gamma')$ and $\Gamma'$ contains more than one
  interior leaf.
 \end{definition}

  For a type $\Gamma$ of treed disks, a perturbation datum
  $\Pe_\Gamma$ is \em{regular} if any $\Pe_\Gamma$-holomorphic map is
  regular, and so for any type $\Gamma_X$ of broken maps whose domain
  type is $\Gamma$, the moduli space $\M^\br_{\Gamma_X}(L,\Pe_\Gamma)$
  is a manifold of expected dimension.

The following theorem on the existence of regular perturbation data
for broken maps is the main result of this section.  Perturbation data
is defined strata-wise, and at each step we assume that the data on
the smaller strata is fixed, where the ordering on the strata is as
follows: For types $\Gamma'$, $\Gamma$ of broken maps,
\begin{equation}
  \label{eq:type-order}
  \Gamma' < \Gamma
\end{equation}
if $\Gamma$ is obtained from $\Gamma'$ by collapsing an edge or making
the length of a boundary edge finite or non-zero.  The ordering
relation helps in describing the boundary of the moduli spaces of
treed curves as
\[\ol \M_\Gamma \bs \M_\Gamma=\cup_{\Gamma' < \Gamma}\M_{\Gamma'}.\]
The background data $(\JJ_0, F_0)$ and the stabilizing divisor $\DD$
are the same for $\Pe_\Gamma$ for all domain types $\Gamma$.  For any
$\Gamma$, we will require that the perturbation datum $\Pe_\Gamma$ is
adapted to the stabilizing pair $(\JJ_0,\DD)$ in the sense of
Definition \ref{def:adapt-pert}.

 \begin{theorem}{\rm(Transversality)}
   \label{thm:transversality}
   Let $\XX$ be a broken manifold.  Let $\bD \subset \XX$ be a
   cylindrical broken divisor and $\JJ_0 \in \J^\cyl(\XX,\bD)$ be a
   locally strongly tamed cylindrical almost complex structure adapted
   to $\bD$ such that $(\bD,\JJ_0)$ is a stabilizing pair (Definition
   \ref{def:stab-brok}), and let $F_0: L \to \R$ be a Morse function.
   Suppose $\Gamma$ is a stable treed disk type, and the
   perturbation data
   \[\{\Pe_{\Gamma'}\}_{\Gamma' <\Gamma}\]
   is coherent, regular, has background $(\JJ_0,F_0)$, and is adapted
   to the pair $(\JJ_0,\DD)$ (as in Definition \ref{def:adapt-pert}).
   Then, for type $\Gamma$, there is
   \begin{itemize}
   \item a neighborhood $\mN_\Gamma \subset \ol \M_\Gamma$ of the boundary $\ol \M_\Gamma \bs \M_\Gamma$,  and a set of domain-dependent perturbations $\PPe_\Gamma(\DD)$  on curves of type $\Gamma$ whose value is fixed on
     $\mN_\Gamma$, and which have 
     background $(\JJ_0,F_0)$, are adapted to the pair $(\JJ_0,\DD)$, and which are 
     coherent with the data
   $\{\Pe_{\Gamma'}\}_{\Gamma' < \Gamma}$ (as in Definition
   \ref{def:coherent}); and
    \item a comeager subset
   \[\PPe_\Gamma^{\on{reg}}(\DD) \subset \PPe_\Gamma(\DD),\]
 \end{itemize}
 such that the following holds for any $\Pe_\Gamma \in \PPe_\Gamma^{\on{reg}}(\DD)$: 
 For an uncrowded type $\Gamma_X$ of broken
   maps with domain type $\Gamma$ (see Definition
   \ref{def:type-broken}),
   and leaf labels $\ul x$, for which the expected dimension of the reduced moduli space \eqref{eq:i-red-br} is
   \[i^\br_\red(\Gamma,\ul x) \leq 1,\]    
   the moduli space
   ${\M}_{\Gamma_X}(\XX, L,\Pe_\Gamma, \ul x)$ is a smooth oriented
   manifold of expected dimension.
   
   Consequently, there is a collection
   of coherent perturbation data
   \[\ul \Pe:=(\Pe_\Gamma)_\Gamma\]
   that 
   has background $(\JJ_0,F_0)$, and is 
   adapted to the pair $(\JJ_0,\DD)$, where $\Gamma$ ranges over all types of stable treed disks.
 \end{theorem}
\begin{proof}[Proof of Theorem \ref{thm:transversality}] 
  Transversality, as in the unbroken case, is an application of
  Sard-Smale as in Cieliebak-Mohnke \cite{cm:trans} and
  Charest-Woodward \cite{cw:flips} on the universal space of maps.
  The new feature is that in neck pieces of the broken manifold the
  almost complex structure is fixed in the fiber direction. This does
  not pose any issues for maps whose horizontal projection is
  non-constant.  Components of the map whose horizontal projection is
  constant will be shown to be automatically transversal.
 
  The moduli space is cut out as a zero set of a section of a Banach
  bundle which we now describe.  We restrict our attention to types of
  maps for which all intersections with the stabilizing divisor have
  multiplicity one. Other types are discussed later.  We construct the
  moduli space of broken maps without framing, since the framed
  version is a finite cover of the unframed one. (See Definition
  \ref{def:bmap} of framing and Remark \ref{rem:nfr}.)  We cover the
  moduli space of treed disks $\M_{\Gamma}$ by charts
  $\cup_i\M_{\Gamma}^i$, so that on a trivialization of the universal
  curve $\U_{\Gamma}^i$, each of the fibers is a fixed treed curve
  $C=S \cup T$ with fixed special points (see
  \eqref{eq:univlocaltriv}).  The complex structure on $S$ varies
  smoothly in the sense that it is given by a map
  \[\M_{\Gamma}^i \to \J(S_\Gamma), \quad m \mapsto j(m).\]

  In order to apply Sard-Smale, we pass to maps from the normalized
  curve of a fixed Sobolev class.  The domain of the map is the
  punctured curve $C^\circ \subset C$ with punctures at tropical nodal
  points $w_e$ (corresponding to $e \in \Edge_\trop(\Gamma)$).  Let
  \[\tC^\circ := \bigsqcup_{v \in \Ver(\Gamma)} S_v^\circ \sqcup \bigsqcup_{e \in
      \Edge_\white(\Gamma)} T_e \]
  denote the normalized curve in the sense that nodes in $C^\circ$
  (corresponding to internal edges $e \in \Edge_\internal(\Gamma)$)
  are lifted to double points in $\tC^\circ$ and the tree components
  in $C^\circ$ are detached from the surface components.  Choose $p>2$
  and $\lam \in (0,1)$.  Let
  \[\Map^{1,p,\lam}_{\Gamma}(\tC^\circ,\XX, L,\bD)\]
  denote the completion of maps under the weighted Sobolev norm
  $\Mod{\cdot}^\circ_\Gamma$ in \eqref{eq:xinorm} for surface
  components, and the ordinary $W^{1,p}$-norm for the tree components.
  That is, an element of this space consists of
  \begin{enumerate}
  \item a collection of $W^{1,p,\lam}$-maps
    \[ u_v:S_v^\circ \to \XC_{P(v)}, \quad u(\partial S_v) \subset L, \quad v \in \Ver(\Gamma) \]
    for each vertex $v$ (where $S_v^\circ$ is defined in \eqref{eq:opensv}), with the puncture at the node corresponding to any
      tropical edge $e \in \Edge_\trop(\Gamma)$ asymptotic to a $\cT(e)$-cylinder;
  \item and $W^{1,p}$ maps
    \[ u_e : T_e  \to L, \quad e \in \Edge_\white(\Gamma) \]
    from tree components to $L$ asymptotic to the critical points
    $\ul x$ on the leaves $T_e, e \in \Edge_{\white}(\Gamma)$.  If
    $\ell(e)$ is finite and non-zero, that is
    $e \in \Edge_{\white}^{(0,\infty)}(\Gamma)$, then
    $T_e \simeq [0,1]$ and if $\ell(T_e)=\infty$, each segment in
    $T_e$ is either $[0,\infty)$, $(-\infty,0]$ or $\R$.
  \end{enumerate}
  The metric on $\XC_{P(v)}$ is chosen to be cylindrical on the ends,
  and so that $L$ is totally geodesic.

  The perturbation data on the strata lower than $\Gamma$ glue to give
  regular perturbation datum in a small neighborhood in the boundary
  of the moduli space $\ol \M_\Gamma$: Any perturbation data for
  $\Gamma'$ with $\Gamma' < \Gamma$ induces perturbations for $\Gamma$
  in a neighborhood of $\U_{\Gamma'}$ in $\ol{\U_{\Gamma}}$ by the
  gluing construction in \eqref{eq:cdelta} and a similar gluing
  construction for the tree parts.  Let
  \[ \PPe_\Gamma(\DD):= \{ \Pe_\Gamma = (J_\Gamma,F_\Gamma) \} \]
  be the space of perturbation data defined on $\ol \U_\Gamma$ whose
  distance from the background data $(\JJ_0,F_0)$ is bounded in the
  $C^\veps$-norm, and that 
agrees with the glued perturbation datum on
  a fixed open neighbourhood $\mN_{\Gamma'}$ of $\U_{\Gamma'}$ for the
  strata $\Gamma' < \Gamma$.
  \label{suffsmall} For a sufficiently small neighborhood
  $\mN_{\Gamma'}$, any such perturbation $\Pe_\Gamma $ is already
  regular for maps $u$ with domain in $\mN_{\Gamma'}$ as follows: If
  not, one could take a sequence $u_\nu$ of irregular maps of fixed
  area with domain converging to $\U_{\Gamma'}$.  By Gromov
  compactness in Chapter \ref{chap:cpt}, after passing to a
  subsequence we may assume that $u_\nu$ converges to a limit
  $u_\infty$ of type $\Gamma'$.  The limit $u_\nu \to u_\infty$ does
  not involve sphere bubbling, otherwise,
  $i^\br_\red(u_\infty) \leq i^\br_\red(u_\nu) -2 <0$ by \eqref{eq:i-br-red-2}, and since
  $\Pe_{\Gamma'}$ is regular, $u_\infty$ does not actually exist.
  Therefore the tropical graphs of $u_\nu$ and $u_\infty$ are the
  same. The limit $u_\nu \to u_\infty$ either involves the formation
  of a disk bubble, or the length of a treed segment going to zero or
  infinity.  By the arguments in the proof of surjectivity of gluing
  in Case 1 in Theorem \ref{thm:bdry-glue} \eqref{tubs} (which is a
  simplification of the proof in Section \ref{surjofgluing}, since
  the gluing is only at non-tropical nodes), the linearized operator
  $D_{u_\nu}$ is surjective for $\nu$ sufficiently large, which is a
  contradiction. Since the gluing and compactness arguments do not use
  regularity, circularity of argument is avoided.
 
  We use the $C^\veps$-norm for $J_\Gamma$ and a $C^l$-norm for
  $F_\Gamma$ where $l>1$ is a fixed number.  For any broken curve $C$
  of type $\Gamma$ we obtain perturbation data on $C$ by identifying
  it conformally with a fiber of the universal tree disk
  ${\U}_{\Gamma}$.
  Let
  \[ \B^i_{p,\lam,l,\Gamma} := {\M}^i_{\Gamma} \times
    \Map^{1,p,\lam}_\Gamma(\tC,\XX,L,\bD) \times {\PPe}_\Gamma(\bD)
    .\]
  Let $\E^i = \E^i_{p,\lam,\Gamma}$ be the Banach bundle over
  $\B^i_{p,\lam,l,\Gamma}$ given by
  \[ (\E^i_{p,\lam,\Gamma})_{j, u, J } \subset
    L^{p,\lam}(\Omega^{0,1}_{j,J}(S, (u|S)^* T\XX)) \oplus
    L^p(\Omega^1(T, (u|T)^* TL)). \]
  Here the first summand is the space of $0,1$-forms with respect to
  $(j(m),J)$.  The Cauchy-Riemann and shifted gradient operators
  applied to the restrictions $u|S$ resp. $u|T$ of $u$ to the two
  resp. one dimensional parts of $C = S \cup T$ define a $C^{l-1}$
  section
  \begin{equation} \label{olp3} \olp_\Gamma: \bB_{p,\lam,l,\Gamma}^i
    \to \cE_{p,\lam,\Gamma}^i,\quad (C,u, (J_\Gamma,F_\Gamma)) \mapsto
    \left(\olp_{j(m),J_\Gamma} u|S , \left(\frac 1 {\lam_e} \dds +
        \grad_{F_\Gamma}\right) u|{T}\right)
  \end{equation}
  where $s$ is a local coordinate on the tree components with unit
  speed, and
  \[\lam_e=
    \begin{cases}
      \ell(e), \quad e \in \Edge^{(0,\infty)}_{\white,-}(\Gamma),\\
      1, \quad e \in \Edge^\infty_{\white}(\Gamma).
    \end{cases}
  \]
  is a factor accounting for the length of treed segments in case the
  edge length is finite and non-zero.  The evaluation maps at lifts of
  nodal points, markings, and lifts of $S \cap T$ give a smooth map
  \begin{equation} \label{egam} \ev_\Gamma : \bB_{p,\lam,l,\Gamma}^i
    \to \XX(\Gamma) \end{equation}
  where
  \begin{multline} \label{xgam} \XX (\Gamma) = \left( \prod_{e \in
        \Edge_{\trop}(\Gamma)} (\XX_{P(e)}/T_{\cT(e),\C})^2 \right)
    \times
    \left( \prod_{e=(v_+,v_-) \in \Edge_{\internal,\black}(\Gamma)}
      (\XC_{P(v_\pm)})^2 \right) \\
    \times \left(\prod_{x \in S \cap T} L^2 \right)
    \times \left(\prod_{e \in \Edge_{\white}^0(\Gamma)} L^2 \right)
    \times \left( \prod_{e \in \Edge_{\black,\to}(\Gamma)} \XX_{P(v(e))} \right).
  \end{multline}
  The first two factors in \eqref{xgam} correspond to lifts of
  interior nodes, the third term is a lift of boundary nodes $w_e$
  with no treed segments (that is, $\ell(e)=0$), the fourth term is a
  lift of $S \cap T$, and the last term corresponds to evaluation at
  an interior marking.  All the factors of $\ev_\Gamma$ are standard
  evaluation maps, except for the first factor which is the projected
  tropical evaluation map as in \eqref{eq:evproj}.  Let
  \[ \Delta(\Gamma) \subset \XX(\Gamma) \]
  be the submanifold that is the product of diagonals in the first
  four factors of $\XX(\Gamma)$ in \eqref{xgam}, and the stabilizing
  divisor $D_{P(v(e))} \subset \XC_{P(v(e))}$ in the last factor.  The
  \em{local universal moduli space} is
  \begin{equation} \label{localuniv}
{\M}^{\univ,i}_{\Gamma}(L,\bD) = (\olp,\ev_\Gamma)^{-1}
    (\bB^i_{p,\lam,l,\Gamma},\Delta(\Gamma)), \end{equation}
  where $\bB^i_{p,\lam,l,\Gamma}$ is embedded as the zero section in
  $\cE_{p,\lam,\Gamma}^i$.

  We will next show that this subspace is cut out transversally.  We
  first consider two-dimensional components of $\tC$ on which the map
  is not horizontally constant, and show that the linearization of
  $(\delbar, \ev_\Gamma)$ is surjective.  For components whose target
  space is not a neck piece, the surjectivity of the differential
  $D(\delbar, \ev_\Gamma)$ follows from \cite[Proposition
  3.4.2]{ms:jh}.  Recall from \eqref{eq:pixp} that a neck piece
  $\XC_P$ is a $T_{P,\C}$-bundle
  \[T_{P,\C} \to \XC_P \to \XB_P,\]
  and the almost complex structure on $\XC_P$ is $P$-cylindrical. (For
  non-neck pieces $T_P$ is trivial.)  Consider a component
  $S_v^\circ \simeq \P^1 \bs \{\text{special points}\}$ that maps to a
  neck piece $\XC_P$, and the horizontal projection
  $\pi_P \circ u:S_v^\circ \to \XB_P$ is non-constant. The linearized
  operator
  \[D_{u,J}^\circ(\xi,K) = D_u^\circ \xi + \hh K Du j. \]
  is surjective as follows.  Let
\[ \eta \in \coker(D_{u,J}^\circ) \subset  \Omega^{0,1}(u^*
T\XC_P) \] 
be a one-form in the cokernel of $D_{u,J}^\circ$.  Variations of tamed
almost complex structure of cylindrical type are $J$-antilinear maps
  \[K: T \XC_P \to T\XC_P\]
  that vanish on the vertical sub-bundle and are
  $T_{P,\C}$-invariant.\footnote{Perturbing $K$ is equivalent to
    perturbing the connection one-form $\alpha_P$ of the
    $P$-cylindrical almost complex structure, see
    \eqref{eq:connJ}.} \label{expK} Since the horizontal part of
  $D_z u$ is non-zero at some $z \in S_v^\circ$, we may find an
  infinitesimal variation $K$ of almost complex structure \em{of
    cylindrical type} by choosing $K(z)$ so that $K(z) D_zu j(z)$ is
  an arbitrary $(j(z),J(z))$-antilinear map from $T_z C$ to
  $T_{u(z)} \XC_P$.  Choose $K(z)$ so that $K(z) D_zu j(z)$ pairs
  non-trivially with $\eta(u(z))$ and extend $K(z)$ to an
  infinitesimal almost complex structure $K$ by a cutoff function on
  the domain curve.

  For two-dimensional components that are
  horizontally constant, by Proposition \ref{prop:toricfibreg}, the
  linearization $D\delbar$ is surjective, and additionally the
  evaluation map at a single marked point is surjective. (Note that,
  if we consider all the markings together, the linearization
  $D\ev_\Gamma$ may not be surjective on horizontally constant
  components.) Finally, for tree components, the linearization of the
  shifted gradient operator and the evaluation map at the finite end
  is surjective, since we can perturb the Morse function on the
  Lagrangian.
 
  From the discussion so far, we conclude that the local moduli space
  \eqref{localuniv} is cut out transversally except if there is a
  connected component in the domain on which the map is horizontally
  constant, and which contains more than one irreducible surface
  component. In this exceptional case, it remains to prove that the
  matching conditions at nodes between two horizontally constant
  components are cut out transversally. Such nodes are necessarily
  internal interior nodes, corresponding to edges
  $e \in \Edge_{\black,\internal}(\Gamma)$.  We consider a maximal
  connected subgraph $\Gamma' \subset \Gamma$ so that the map is
  horizontally constant on the vertices of $\Gamma'$, and $\Gamma'$
  does not have tree components.  By uncrowdedness, there is at most
  one marked point in $\Gamma'$.  So, it is possible to choose at most
  one special point (marked point or a lift of a nodal point) on each
  component of $\tilde C_{\Gamma'}$, so that for every nodal point one
  of its lifts is chosen.  By Proposition \ref{prop:toricfibreg}, for a
  horizontally constant component with a single tropical marking $z_e$
  corresponding to an edge $e \in \Edge_\trop(\Gamma)$, the linearized
  map $D(\delbar,\pi_{\cT(e)}(\ev_{z_e}^{\cT(e)}))$ is surjective.
  Since the evaluation map is surjective at each of the chosen lifts,
  an inverse of the linearized map $D(\delbar,\ev_\Gamma)$ can be
  constructed inductively, see \cite[p63]{cw:flips}.

  For types where the map has higher order intersections with the
  stabilizing divisor, the universal moduli space is cut out
  inductively as in \cite[Lemma 6.5]{cm:trans}. Each step of the
  induction cuts out a moduli space where the tangencies at one of the
  markings is increased by one. We start out with a moduli space cut
  out of $W^{k,p,\lam}$ where $k - \frac 2 p > \mu$ and $\mu$ is the
  largest order of tangency with the stabilizing divisor $\DD$ that
  occur in the type $\Gamma$.

  By the implicit function theorem,
  ${\M}^{\univ,i}_{\Gamma}(L,\bD)$ is a smooth Banach manifold, and
  the forgetful morphism
  \[\varphi_i: {\M}^{\univ,i}_{\Gamma}(L,\bD)_{k,p,l} \to
    \PPe_{\Gamma}(L,\bD)_l \]
  is a smooth Fredholm map.  By the Sard-Smale theorem, the set of
  regular values $\PPe^{i,\reg}_{\Gamma}(L,D)$ of $\varphi_i$ on
  ${\M}^{\univ,i}_{\Gamma}(L,\bD)_d$ in $\PPe_{\Gamma}(L,\bD)$ is
  comeager.  Let
  \[ \PPe^{\reg}_{\Gamma}(L,\bD) = \bigcap_i
    \PPe^{i,\reg}_{\Gamma}(L,\bD) .\]
  A standard argument shows that the set of smooth domain-dependent
  $\PPe^{\reg}_{\Gamma}(L,\bD)$ is also comeager.  Fix
  $(J_\Gamma,F_\Gamma) \in \PPe^{\reg}_{\Gamma}(L,\bD)$.  By
  elliptic regularity, every element of ${\M}^i_{\Gamma}(L,\bD) $
  is smooth. The transition maps for the local trivializations of the
  universal bundle define smooth maps
  \[ {\M}^i_{\Gamma}(L,\bD) |_{ {\M}^i_{\Gamma} \cap
      {\M}^j_{\Gamma}} \to {\M}^j_{\Gamma}(L,\bD) _{{\M}^i_{\Gamma}
      \cap {\M}^j_{\Gamma}} .\]
  This construction equips the space
  \[ {\M}_{\Gamma}(L,\bD) = \cup_i {\M}^i_{\Gamma}(L,\bD) \]
  with a smooth atlas.  Since $\M_\Gamma$ is Hausdorff and
  second-countable, so is $\M_\Gamma(L,\bD)$ and it follows that
  $\M_\Gamma(L,\bD)$ has the structure of a smooth manifold. Orientation of the moduli space is discussed in Remark \ref{rem:orientmap} below.
\end{proof}

For regular perturbations, the tangent space to the moduli space can
be identified with the kernel of an operator called the \em{
  linearized operator}.  Using notations from the proof of
transversality, we write down the linearized operator for future
reference.  For a regular perturbation $\Pe_\Gamma$, consider the map
\[ (\delbar,\ev_\Gamma): \bB_{k,p,\Gamma}^i \to \cE_{k,p,\Gamma}^i
  \times \XX(\Gamma) \]  
where $\ev_\Gamma$ is defined in \eqref{egam}. 
Its linearization is denoted
\begin{equation}
  \label{eq:linop}
  D_u : T_{[C],u}\bB_{k,p,\Gamma}^i \to
  (\cE_{k,p,\Gamma}^i)_{[C],u} \oplus \ev_\Gamma^* T\XX(\Gamma)/T\Delta(\Gamma).  
\end{equation}
For a regular broken map $u$, the operator $D_u$ is surjective.

\index{Orientation for moduli spaces!of maps} 
\begin{remark}{\rm(Orientation of moduli spaces of maps)}
  \label{rem:orientmap}
  Moduli spaces of broken maps are oriented using a relative spin
  structure on the Lagrangian as in Fukaya-Oh-Ohta-Ono \cite{fooo} as
  follows. An orientation for the moduli space at a given map is by
  definition a non-zero determinant line bundle of the Fredholm
  operator $D_u$, that is the linearization of the Fredholm map
  cutting out the moduli space. The determinant line bundle is defined as the
  tensor product of top exterior powers
  \[ \det(D_u):=\Lam^{\on{top}}(\ker(D_u)) \wedge
    \Lam^{\on{top}}(\coker(D_u)).\]
  Let $L$ be an oriented Lagrangian equipped with a relative spin
  structure (see \cite[Chapter 44]{fooo} for the definition).  Given a
  holomorphic disk $u:(\D^2,\partial \D^2) \to (X,L)$ and a complex
  linear Cauchy-Riemann operator
  \[\delbar : \Gamma(\D^2, \partial \D^2;u^*TX, (\partial u)^*TL) \to
    \Om^{0,1}(\D^2,u^*TX),\]
  the determinant line bundle can be identified with $\det(T_lL)$ for
  any point $l \in L$, that is, $\det(\delbar) \simeq \det(T_lL)$,  and the identification is canonical up to multiplication by positive
 scalars $\lam \in \R_+$, see \cite[Proposition 44.4]{fooo} or an
 explanation in \cite[Proposition 5.2]{Cho:pn}.  The argument in 
 \cite[Proposition 44.4]{fooo} also gives such an identification in
 case of a nodal disk with interior nodes, since that involves
 degenerating the bundle $u^*TX$ into a bundle on the disk and a
 complex vector bundle on the sphere.  The determinant bundle for 
 the linearized operator for broken maps can be
 identified to $\det(T_lL)$ in a similar way, since the matching conditions 
 at the cylindrical ends are complex-linear.   There are similar orientations for moduli spaces of treed disks with  contributions to the orientation from the
 infinite treed segments attached at boundary markings.  We choose
 orientations on Morse unstable and stable manifolds
  \[o(x) : W^\pm(x) \to TW^\pm(x), \quad \forall x \in \crit(F)\]
  for the Morse function $F$ on the Lagrangian,
  such that for any $x \in \crit(F)$ the map
  \[\det(T_xW^-(x) \oplus T_xW^+(x)) \to \det(T_xL)\]
  is orientation preserving. We view a broken treed holomorphic disk
  as defined by a condition $u(z_i^\white) \in W^-(x_i)$ on each of
  the boundary markings $e_i^\white \in \Edge_{\white,\to}$.  This
  gives an isomorphism of line bundles
  \[\det(D_u) \simeq \det(T\M_\Gamma) \wedge \det(TL) \wedge
    \det(W^+(x_0)) \wedge (\wedge_{i=1}^{d(\white)}
    \det(W^-(x_i))), \]
  up to multiplication by positive scalars. The choices made in the previous paragraphs  fix a positive section on the right-hand side.    The construction 
  induces an orientation on the moduli space
  $\M_\Gamma(\XX,\ul x)$ for any type $\Gamma$ of broken maps and a
  collection of inputs and output $\ul x$. The orientation of
  $\det(T \M_\Gamma)$ is described in Definition \ref{def:orientmg}.
For a broken map type $\Gamma$ with a broken edge, the
 difference in orientation induced by
  \hyperref[item:makingedgelengthmorph]{(Making an edge length
    finite)} morphism and the product orientation induced by
  \hyperref[item:cuttingedgesmorphism]{(Cutting an edge)} morphism
  differ by a quantity that depends on the type $\Gamma$ and the Morse
  indices of the labels on the end-points of the treed segments, see
  \cite[(12.25)]{se:bo}.  This ends the Remark.
\end{remark}

\section{The toric case}

In this section, we consider the question of regularity for
horizontally constant maps. These are maps in a toric fibration whose
image is contained in a single fiber.  Therefore, we wish 
to show that, in a toric variety with standard almost complex
structure, spheres that intersect torus-invariant divisors at isolated
points are regular.  The maps are allowed to have a single tropical
marking, in which case, the evaluation map at the marking is
submersive. This result is used in the proof of transversality in the
previous section to analyze horizontally constant components.

\begin{proposition}\label{prop:toricfibreg}
  Suppose $\XX$ has a cylindrical almost complex structure.  Let
  $\Gamma$ be a type of relative map with a single vertex $v$, the
  vertex $v$ being of sphere type and horizontally constant.  Any map
  $u: S_v^\circ \to \XX_{P(v)}$ of type $\Gamma$ is regular, and
  therefore, the moduli space of maps $\M_\Gamma(\XX)$ is a manifold
  of expected dimension.  For a tropical marking $z_e$ corresponding to
  any edge $e \in \Edge_\trop(\Gamma)$, the projected tropical
  evaluation map
  \[\pi^\perp_{\cT(e)} \circ \ev^{\cT(e)}_{z_e} : \M_\Gamma(\XX) \to \XX_{P(e)}/T_{\cT(e),\C}\]
  is submersive.
\end{proposition}

\begin{lemma}\label{lem:toric-sphere-reg}
  Let $X$ be a toric orbifold with the standard almost complex
  structure, and let $u : \P^1 \to X$ be a holomorphic sphere whose
  intersection points with the torus-invariant divisors
  $Y_1,\dots, Y_N$ are isolated. Then, $u$ is regular in $X$, that is,
  the linearized operator $D\delbar$ is surjective at $u$.
\end{lemma}

\begin{proof}
  Holomorphic spheres meeting the interior of $X$ are regular by the
  following argument. As in Delzant \cite{de:ha}, $X$ can be viewed as
  a geometric invariant theory quotient $\C^N \qu G$ where $N$ is the
  number of prime torus-invariant relative divisors of $X$, and
  $G \subset (\C^\times)^N$ is a complex torus whose quotient
  $(\C^\times)^N/G$ is $T_\C$.  Each of the torus-invariant divisors
  $Y_1,\dots, Y_N$ in $X$ lifts to a coordinate hyperplane
  $ \{ z_1 = 0 \}, \ldots, \{ z_N = 0 \}$ in $\C^N$.  Consider a
  holomorphic sphere $u:\P^1 \to X$, that is not contained in any
  toric divisor.  The vector bundle $u^*TX \oplus \ul \g$ on $\P^1$ is
  a sum of line bundles
\[ u^*TX \oplus \ul \g =  \bigoplus_{i=1}^Nu^*\mO(Y_i) \]  
where $\ul \g:=\g \times \P^1$ is the trivial bundle.  The degree
$\deg(u^* \mO(Y_i))$ of the line bundle $u^*\mO(Y_i)$ is given by the
intersection of $u$ with $Y_i$.  Hence, each of the degrees
$\deg(u^* \mO(Y_i))$ is non-negative.  As a result, the operator
  \begin{equation*}
    \olp:\Gamma(\P^1, \oplus_i u^*\mO(Y_i)) \to \Om^{0,1}(\P^1, \oplus_i u^*\mO(Y_i))
  \end{equation*}
  is onto.  Consequently, the cohomology group
  \begin{equation*}
    H^{0,1}(\P^1,\oplus_i u^*\mO(Y_i)) \simeq H^1(\P^1,\oplus_i u^*\mO(Y_i)) 
  \end{equation*}
  vanishes.  Consider the long exact sequence in {\v C}ech cohomology,
  corresponding to the short exact sequence of sheaves
  \begin{equation*}
    0 \to \ul \g \to \oplus_i u^*\mO(Y_i) \to u^*TX \to 0.
  \end{equation*}
  Vanishing of the zeroth resp. first cohomology of the first
  resp. second terms implies that the first cohomology of the third
  term
  \begin{equation*}
    H^{0,1}(\P^1,u^*TX) \simeq H^1(\P^1, u^*TX)
  \end{equation*}
  also vanishes.  Therefore the sphere $u$ is regular in $X$.
\end{proof}

\begin{proof}[Proof of Proposition \ref{prop:toricfibreg}]
  We recall that we may equivalently use the operators $D_u$ or
  $D_u^\circ$ since they have isomorphic kernels and cokernels.
  First, we use $D_u$ to show that $u$ is regular. Recall that
  \[  \XX_{P(v),\cT} := \XX_{P(v)} \cup \cup_{e \in \Edge_\trop(\Gamma)}Y_{\cT(e)}\]
  is a compactification of $\XX_{P(v)}$ obtained by adding a divisor
  $Y_{\cT(e)}$ (see \eqref{eq:Ycte}) to $\XX_{P(v)}$ corresponding to
  every edge $e$ incident on $v$.  A map
  $u : S_v^\circ \to \XX_{P(v)}$ of type $\Gamma$ extends
  holomorphically to $u: \P^1 \to \XX_{P(v),\cT}$ with
  $u(z_e) \in Y_{\cT(e)}$.  The fibers of the projection map
  $\pi_{P(v)} : \XX_{P(v),\cT} \to X_{P(v)}$ are toric orbifolds with
  the standard almost complex structure, and the image of $u$ is
  contained in a single fiber. Since the almost complex structure on
  $\XX_{P(v)}$ is cylindrical, the pullback bundle splits into
  holomorphic bundles
  \[u^*T\XX_{P(v)} = V \oplus H \to \P^1, \quad V=u^*\ker(d\pi_{P(v)}),\]
  with the horizontal bundle $H$ being trivial, since $u$ is
  horizontally constant. The linearization $D\delbar|V$ surjects onto
  $\Lam^{0,1}_{\P^1} \tensor V$ by Lemma \ref{lem:toric-sphere-reg},
  and since $H$ is trivial, the map
  $D\delbar: H \to \Lam^{0,1}_{\P^1} \tensor H$ is surjective.

  Next, we will show that the moduli spaces of spheres with prescribed
  tangencies at relative marked points are cut out transversally.  For
  a fixed point $p \in X_{P(v)}$, consider holomorphic maps
  $u: \P^1 \to \pi_{P(v)}^{-1}(p) \subset \XX_{P(v),\cT}$ of type
  $\Gamma$, and hence not contained in the divisors $Y_{\cT(e)}$.  The
  first Chern class of the sphere $u$, viewed as a map to the fiber
  $\pi_{P(v)}^{-1}(p)$, is equal to the sum of its intersection
  multiplicities $(u. Y_{\cT(e)})$ with all divisors at infinity
  $Y_{\cT(e)}$.  In fact, the moduli space containing $u$ is
  parametrized by the set of intersection points $z \in C$ with the
  toric divisors $Y_{\cT(e)}$.  For maps with a higher order tangency
  with a toric divisor $Y_{\cT(e)}$, some subsets of intersection
  points of $u$ with $Y_{\cT(e)}$ coincide.  Since the zeros of
  sections of positive line bundles may be chosen arbitrarily in genus
  zero, the locus of such maps is cut out transversally from the space
  of all maps.  Therefore, the moduli space of relative maps of any
  type $\Gamma$ is regular. We now switch to working with relative
  maps $u: S_v^\circ \to \XX_{P(v)}$ defined on punctured curves.

   The linearization of the tropical evaluation map \eqref{eq:tropev}
   at tropical markings is surjective because of the
   torus-equivariance of the evaluation map. In particular, for a
   tropical marking $z_e$ and a choice of holomorphic coordinate in a
   neighborhood of $z_e$, the tropical evaluation
  \[\ev^{\cT(e)}_{z_e}:  \M_\Gamma(\XX) \to \XX_{P(e)}\]
  satisfies
  \begin{equation}
    \label{eq:evequiv}
  \ev^{\cT(e)}_{z_e}(tu)=t\ev^{\cT(e)}_{z_e}(u), \quad \forall t
    \in T_{P(v),\C}.   
  \end{equation}
  We choose an embedding of $\XX_{P(e)}$ into the cylindrical end of
  $\XX_{P(v)}$, and view $\ev^{\cT(e)}_{z_e}$ as mapping to
  $\XX_{P(v)}$.  The fiber of the projection map $\pi_{P(v)}$ is a
  $T_{P(v),\C}$-orbit.  By \eqref{eq:evequiv}, the image of
  $\ev^{\cT(e)}_{z_e}$ consists of $T_{P(v),\C}$-orbits, and
  therefore, $D\ev^{\cT(e)}_{z_e}$ surjects onto the vertical part of
  the tangent space $V \subset T\XX_{P(v)}$. The differential
  $D\ev^{\cT(e)}_{z_e}$ also surjects onto the horizontal subspace
  $H$, since constant sections on $H$ are contained in
  $\ker(D\delbar)=T\M_\Gamma(\XX)$.
\end{proof}

\begin{example}
  {\rm(Trivial cylinders with a single marking)}
  \label{ex:triv1} 
  A trivial cylinder with a single marking is an uncrowded component. A broken map $u:C \to \XX$ of type $\Gamma$ containing such a component 
  has index at least two, as we now explain. Taking a slice of the action of the domain isomorphism group (Definition \ref{def:iso-stab} \eqref{part:mapiso}), we assume that the nodes are $z_0=0$, $z_1=\infty$, the marking is $z_2=1$ on $S_v \simeq \P^1$, and the map is
  \[u_v : S_v^\circ \simeq \C^\times \to \XX_{P(v)}, \quad z \mapsto z^\mu x,\]
    for some $\mu \in \t_{P(v),\Z}$. 
    Since $S_v$ has a marking, the image 
    of $u_v$ lies in a $T_{\mu,\C}$-orbit in the stabilizing divisor $\DD_{P(v)} \subset \XX_{P(v)}$.
    In the broken map $u$, the $T_{\mu,\C}$-orbit containing the image
    of $u_v$ is determined by the matching conditions at the nodes.
    The set of maps $u_v$ in a given orbit, for which matching
    conditions are satisfied, is $2$-dimensional, and therefore, the
    index of the broken map $u$ is at least $2$. In particular, in the
    moduli space $\M_\Gamma(\XX)$, the evaluation $\ev_{z_2}$ takes
    all possible values in the image orbit.  As an aside, observe that
    the tropical symmetry group $T_\trop(\Gamma)$ has at least one
    infinitesimal generator, since both edges incident on the vertex
    $v$ have slope $\mu$.
\end{example}

\chapter{Hofer energy and exponential decay}\label{chap:hof}

Compactness for sequences of broken pseudoholomorphic maps requires
bounds on area. Neck-stretched manifolds do not have a canonically-defined taming symplectic form.  However, neck-stretched
manifolds are non-canonically diffeomorphic to a tropical Hamiltonian manifold and inherit a cohomology class of degree two from the diffeomorphism.   The Hofer energy of a pseudoholomorphic curve is defined as a
supremum of symplectic areas where the symplectic form ranges over a
family of taming forms in this symplectic class. Each symplectic form in
the family is given by a map of complexes
\[\aleph:B_J \to B_\om\]
between the \em{$J$-complex} $B_J$ underlying the cylindrical almost
complex structure and the \em{$\om$-complex} $B_\om$ underlying the
symplectic form on the tropical manifold.  For example, in a
neck-stretched manifold $X^\nu$ the $\om$-complex is the dual complex
$B^\dual$ and the $J$-complex is $\nu B^\dual$.  Such a map $\aleph$
between complexes gives a map between manifolds
\[\psi_\aleph : (X^\nu,J^\nu) \to (X,\om_X) \]
which is taming (that is, $J^\nu$ is $\psi_\aleph^*\om_X$-tame) if
$\aleph$ satisfies certain conditions.
In case of a single cut, $\psi_\aleph$ is taming if
$\aleph : [\frac {-\nu} 2, \frac \nu 2] \to [0,1]$ is an increasing
diffeomorphism. This notion of energy as a supremum over a family of
symplectic areas was originally defined by Hofer in the context of
symplectic field theory \cite{Hofer:weinstein}.

The taming condition is more complicated for a multiple cut and leads
us to define a class of maps between complexes, called \em{squashing
  maps}.  A squashing map is a continuous piecewise smooth map between
complexes, which on any piece, is the composition of a translation,
dilation and an orthogonal projection.  For a pseudoholomorphic map
$u : C \to X^\nu$ in a neck-stretched manifold (or a broken manifold),
we define the \em{Hofer energy} as
\[E_\Hof(u):=\sup_{\aleph \text{ is a squashing map}}\int_C u^*(\psi_\aleph^*\om_X). \]
Of course, if $C$ is a closed curve, or if $C$ is a disk and $u$ maps
the boundary $\partial C$ to a Lagrangian submanifold, the value of
the integral is independent of $\aleph$ since the cohomology class of
$\psi_\aleph^*\om_X$ is independent of $\aleph$.

The main result of this Chapter is that a punctured pseudoholomorphic
curve $u$ with finite Hofer energy is asymptotic to a trivial
cylinder, which is a map
\[u_\triv : \C^\times \to X_P, \quad z \mapsto z^\mu x_0 \]
for some integral generator $\mu \in \t_{P,\Z}$ and $x_0 \in X_P$. More precisely,
we show that for such a map $u$, the tropical evaluation is
well-defined at the puncture point, that is, the twisting by $-\mu$,
given by
\[v(z):=z^{-\mu}u(z),\]
has a removable singularity at the puncture, where $z$ is a
holomorphic coordinate in the neighborhood of the puncture $z=0$.

For the removal-of-singularities result, we consider maps on a
punctured disk $B_1 \bs \{0\} \subset \C$, which is holomorphically
identified with
\[\Cyl := \R_{\geq 0} \times S^1. \]
For any $l \geq 0$, we refer to a truncated semi-infinite cylinder by
\[\Cyl(l) := [l,\infty) \times S^1. \]
We use a product metric on the domain $\Cyl$, and on the target broken
manifold $\XX$, we use a fixed cylindrical metric as in Definition
\ref{def:cylmet}.

The following result holds for maps that are holomorphic with respect
to a domain-dependent almost complex structure that takes values in a
$C^0$-neighborhood $U_{\JJ_0}$ of the space of cylindrical almost
complex structures. The neighborhood $U_{\JJ_0}$ is determined in
Lemma \ref{lem:monot}, \label{rep:uj} and is such that any
$J \in U_{\JJ_0}$ is weakly tamed (Definition \ref{def:weaklytaming})
by the form $\psi_\aleph^*\om_X$ for any squashing map $\aleph$.

\begin{proposition}\label{prop:remsing}
  {\rm(Removal of singularities)} Suppose $u: \Cyl \to \XX_P$ is a map
  that is holomorphic with respect to the domain-dependent almost
  complex structure $J : B_1 \to U_{\JJ_0}$ (holomorphically
  identifying $\Cyl \simeq B_1 \bs \{0\}$) and $U_{\JJ_0}$ is as in
  the preceding paragraph.  Suppose
  \begin{equation}
    \label{eq:remhyp}
    E_{P,\Hof}^*(u) <\infty, 
    \quad   \Mod{\d u}_{L^\infty(\Cyl)} <\infty. 
  \end{equation}
  Then, there exists a polytope $Q \subseteq P$, $\mu \in \t_{Q,\Z}$,
  and $L \geq 0$, such that the image $u(\Cyl(L))$ lies in the
  $Q$-cylindrical end $\XX_P$, and the twisted map
  \begin{equation}
    \label{eq:remconc}
  \ol u: \Cyl(L) \to \XX_Q, \footnote{The $Q$-cylindrical end of $\XX_P$ is embedded in $\XX_Q$ by \eqref{eq:Pcoord}.}
    \quad (s,t) \mapsto e^{-\mu(s+it)}u(s,t)  
  \end{equation}
  has a removable singularity at $\infty$.  The same result holds when
  the broken manifold $\XC_P$ is replaced by the cut space $\XB_P$.
\end{proposition}
\noindent Here, $E^*_{P,\Hof}$ is the unpartitioned $P$-Hofer energy of a map defined below in Definition \ref{def:PHof}.

We describe some other results from this Chapter, and indicate how
they are used in the proof of Gromov compactness in Chapter
\ref{chap:cpt}. Given a sequence of pseudoholomorphic maps in neck-stretched manifolds
which belong to a fixed homology class, Remark \ref{rem:hof-area} says
that the Hofer energy of these maps is uniformly bounded, since it is
a topological quantity for maps on closed domains and disks with
Lagrangian boundary conditions. Next, the monotonicity result (Lemma
\ref{lem:monot}) says that the Hofer energy of a map 
cannot
increase
if we restrict the map to a subdomain.  By Proposition
\ref{prop:hofer-breaking}, Hofer energy is well-behaved for a sequence
of maps on subdomains (not necessarily closed) under the
neck-stretching limit, which allows us to obtain Hofer energy bounds
on pieces of the Gromov limit. As a last step, this bound allows us to
apply the removal of singularities result (Proposition
\ref{prop:remsing}) on components of the Gromov limit.

\section{Symplectic forms on neck-stretched manifolds}
\label{sec:incmap}
In this section, we define a family of symplectic forms on
neck-stretched manifolds that tame cylindrical almost complex
structures that are locally strongly tame (Definition
\ref{def:cylneckst} \eqref{part:local-strong-t}).  For a
pseudoholomorphic disk $u:(C,\partial C) \to (X^\nu,L)$ in a neck
stretched manifold $X^\nu$ whose boundary maps to a Lagrangian
submanifold $L \subset X^\nu$, the symplectic area
\begin{equation}
  \label{eq:area-u}
\Area(u):=\lan u_*[C],[\om_X]\ran , 
\end{equation}
\index{Area!Symplectic area}is defined by identifying 
the neck-stretched
manifold $X^\nu$ (not canonically) with the symplectic
manifold $X$.  
Our task is to construct a family of symplectic
forms in the class $[\om_X] \in H^2(X^\nu)$.
As a first step, we describe 
how a map \[\aleph : \nu B^\dual \to B^\dual\] between complexes determines a map
\[\psi_\aleph : X^\nu \to X
\]
from the neck-stretched manifold $X^\nu$ to the corresponding symplectic manifold $X$. In the second half of the Section, we describe an \em{increasing} condition on the map $\aleph$ that ensures that $\psi_\aleph^*\om$ tames cylindrical almost complex structures that are locally strongly tamed (Definition
\ref{def:cylneckst} \eqref{part:local-strong-t}) and that lie in a small neighborhood. The condition on $\aleph$ may appear fairly restrictive, and the reason is that the symplectic form needs to tame a class of cylindrical almost complex structures whose underlying connection one-form is allowed to vary (see Remark \ref{rem:tame-compare} for more details). 
The extra freedom of varying the connection one-forms is needed for defining domain-dependent almost complex structures with which one can attain transversality of the linearized $\delbar$ operator.

In order to construct the map $\psi_\aleph$ from the map $\aleph$ of
complexes, we first define a map $\pi_{B^\dual}$ from the symplectic
manifold $(X,\om_X)$ to its $\om$-complex $B^\dual$ that is a
piecewise projection.  We recall from \eqref{eq:sympcylX} that on a
tropical manifold $(X,\om_X,\PP,\Phi)$ there is a symplectic
cylindrical structure on neck regions, that is, there is a
$T_P$-equivariant symplectomorphism%
\begin{multline}
  \label{eq:sympstr-hofer}
  (X,\om_X) \supset \Phinv(\tP) \xrightarrow{\phi_P} (\Phinv(P) \times
  P^\dual, \tilde \om_P), \\
  \tilde \om_P := (\om_X|\Phinv(P)) + d\bran{\alpha_P, \pi_{P^\dual}},
\end{multline}
for each polytope $P \in \PP$, where $\tP \subset \t^\dual$ is a
tubular neighborhood of the polytope $P \subset \t^\dual$ with
projection (see Figure \ref{fig:fibered-poly})
\[\pi_P : \tP \to P,\]
and $\alpha_P \in \Om^1(\Phinv(P),\t_P)$ is a $T_P$-connection
one-form. The projection
$\pi_{P^\dual} : \Phinv(\tP) \to P^\dual \subset \t_P^\dual$ of the
moment map $\Phi$ to $\t_P^\dual$ is a moment map for the
$T_P$-action, where $P^\dual$ is viewed as a subset of $\t_P^\dual$ via the isomorphism $\t_P \simeq \t_P^\dual$ from 
\eqref{eq:idtt}. 
Let $\wP \subset P$ be the complement of a neighborhood
of faces of $P$, namely
\begin{equation}
  P^\fillblack:=P \bs (\cup_{Q \subset P}\tQ). 
\end{equation}
Let
\begin{equation}
  \label{eq:olpblack2}
  \tP^\fillblack:=\pi_P^{-1}(P^\fillblack) \subset \tP
\end{equation}
be the thickening of $P^\fillblack$.  For a pair $Q \subset P$ with
$\codim_P(Q)=1$, the fibered polytopes
$\tP^\fillblack ,\tQ^\fillblack \subset \t^\dual$ share a facet, which
is isomorphic to $Q^\fillblack \times P^\dual$.  Then the image of
$\Phi$ has a cover
\begin{equation}
  \label{eq:phicover2}
  \im(\Phi)=\cup_{P \in \PP} i_{\tP}(\tP^\fillblack)/\sim,   
\end{equation}
where $\sim$ identifies shared facets of polytopes; see Figure
\ref{fig:wp}.
\begin{figure}[ht]
  \centering \scalebox{.8}{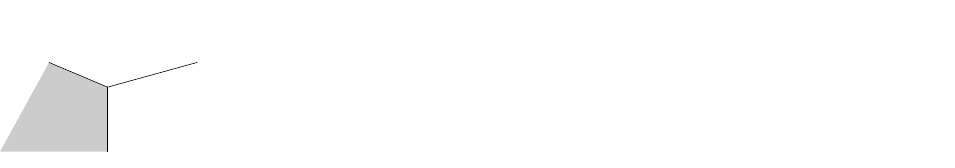}
  \caption{Decomposition of moment polytope induced by thickenings of
    polytopes in $\PP$.}
  \label{fig:wp}
\end{figure}
\noindent The partition of $\im(\Phi)$ pulls back to a partition of
the manifold $X$
\begin{equation}
  \label{eq:xnupre2}
  (X,\om_X) \simeq \left( \bigsqcup_{P \in \PP} \Phinv( \tP^\fillblack) \right) / \sim 
\end{equation}
into manifolds with corners, where the identifications in $\sim$ are
along the boundaries and are induced by the inclusions
$\Phinv( \tP^\fillblack) \to X$.  Furthermore,  the symplectic cylindrical
structure map $\ul \phi=(\phi_P)_P$ may be used to rewrite the
decomposition in \eqref{eq:xnupre2} as
\begin{equation}
  \label{eq:xnupre3}
  (X,\om_X) \simeq \left( \bigsqcup_{P \in \PP} \Phinv( P^\fillblack)  \times P^\dual \right) / \sim.
\end{equation}
(See the derivation of \eqref{eq:xnupre1} for more details.)  The
decomposition \eqref{eq:xnupre3} implies that there is a continuous
projection map
\begin{equation}
  \label{eq:pib-symp}
  \pi_{B^\dual}: (X,\om_X) \to B^\dual
\end{equation}
defined by projecting $\Phinv( P^\fillblack) \times P^\dual \subset X$
to $P^\dual$ for each $P \in \PP$.  For any polytope $P \in \PP$,
$\pi_{B^\dual}$ is a $T_P$-moment map for the $T_P$-action on
$\pi_{B^\dual}^{-1}(P^{\dual,\circ})$.  Thus, $P^\dual$ is the
$\om$-polytope on $\Phinv(\wP) \times P^\dual \subset X$, and the
union $B^\dual=\cup_{P \in \PP} P^\dual$ is the \em{$\om$-complex}
for $X$.

Next, we give map relating the neck-stretched manifold $X^\nu$ to
the $J$-complex $\nu B^\dual$. We recall that the neck-stretched manifold is
\begin{equation}
  \label{eq:xnuJ-dec}
  X^\nu=\left(\bigsqcup_{P \in \PP} \Phinv(\wP) \times \nu P^\dual \right)/\sim   
\end{equation}
where the equivalence $\sim$ from \eqref{eq:xnudef} identifies
boundary components.  The space $X^\nu$ is a smooth manifold, with the
identifications between tubular neighborhoods of boundaries and
corners that are glued in \eqref{eq:xnuJ-dec}, being the same as the
corresponding identifications in the decomposition \eqref{eq:xnupre3}
of $(X,\om_X)$. \label{page:xnuids} Analogously to \eqref{eq:pib-symp},
neck-stretched manifolds project to the scaled dual complex (see
\eqref{eq:pinb-ch3}): For any $\nu \geq 1$, there is a continuous
projection map
\begin{equation}
  \label{eq:pib-J}
  \pi_{\nu B^\dual} : X^\nu \to \nu B^\dual
\end{equation}
such that for any polytope $P \in \PP$,
$(\pi_{\nu B^\dual})^{-1}(P^{\dual,\circ})$ is a $P$-cylinder
$\Phinv(\wP) \times \nu P^{\dual,\circ}$.  Thus the $J$-polytope for
the $P$-cylindrical subset of $X^\nu$ is $\nu P^\dual$, and the \em{
  $J$-complex} of $X^\nu$ is $\nu B^\dual$.

Neck-stretched almost complex manifolds are mapped to compact
symplectic manifolds via maps of complexes between the $J$-complex of
the neck-stretched manifold and the $\om$-complex of the symplectic
manifold:
\begin{definition}
  For any $\nu\geq 1$, a map
  \[\aleph : \nu B^\dual \to B^\dual\] 
  is a \em{map of complexes} if it is a collection of maps of
  polytopes, that is, for any $P \in \PP$,
  $\aleph(\nu P^\dual) \subseteq P^\dual$.
\end{definition}
A map $\aleph$ of complexes induces a map of manifolds
\begin{equation}
  \label{eq:alephembed}
  \psi^\nu_\aleph: X^\nu \to (X,\om_X),
\end{equation}
where, for any $P \in \PP$, $\psi_\aleph$ maps
$\Phinv(\wP) \times \nu P^\dual$ to $\Phinv(\wP) \times P^\dual$ by
$(\Id, \aleph)$. We leave it to the reader to verify that the maps on
the subsets patch to yield a continuous map. The resulting map $\psi_\aleph$ fits into a commutative
diagram
\[
    \begin{tikzcd}
      X^\nu \arrow{r}{\psi_\aleph} \arrow[swap]{d}{\pi_{B_J}} & (X,\om_X) \arrow{d}{\pi_{B^\dual}}\\
      \nu B^\dual \arrow{r}{\aleph} & B^\dual.
    \end{tikzcd}
  \]

Next, we describe a
condition on the map $\aleph$ that ensures that the pullback of the
symplectic form by $\psi_\aleph$ tames a suitable class of cylindrical
almost complex structures. 
In case of a single cut, if the
$\aleph : [\frac {-\nu} 2, \frac \nu 2] \to [-\hh,\hh]$ is an
increasing diffeomorphism, then $\psi^\nu_\aleph$ is a diffeomorphism
for which $(\psi^\nu_\aleph)^*\om_X$ tames any cylindrical almost almost
complex structure that is locally strongly tamed (Definition
\ref{def:cylneckst} \eqref{part:local-strong-t}). This is precisely the class of maps $\aleph$
used to define Hofer energy in \cite{bo:com}.

To generalize the notion of Hofer energy to a multiple cut, we define
an analogue of increasing maps for higher dimensional polytopes in 
Definition \ref{def:increasing-maps}. The condition is that at any
point, the derivative of the map $\aleph$ is a symmetric matrix with
non-negative eigenvalues.  In other words, we require a non-decreasing
condition in a set of orthogonal directions.  

To simplify notation, in this Chapter, we assume that the polyhedral
decomposition $\PP$ has a single zero-dimensional polytope, or in
other words, the dual complex has a single top-dimensional polytope.
As a result, the inner product on $\t$ (in \eqref{eq:idtt}) is assumed
to be fixed throughout.  Extensions to the general case do not present
difficulties; in the proof of Proposition
\ref{prop:remsing}, which is the main result of the Chapter, we point out the necessary steps to
extend the proof to the general case.

\begin{definition}{\rm(Increasing maps)} \label{def:increasing-maps}
  A map of complexes $\aleph : \nu B^\dual \to B^\dual$ is \em{
    increasing} resp. \em{strictly increasing} if for any polytope $P \in \PP$ and a point
  $x \in \nu P^\dual$, the derivative
  $D(\aleph|\nu P^\dual)_x : \t_P^\dual \to \t_P^\dual$ is
  \begin{enumerate} 
\item symmetric (that is, 
  diagonalizable with a set of orthogonal eigenvectors with respect 
  to the $\t_P$-inner product \eqref{eq:idtt}), and 
  \item non-negative with norm at most $1$ resp. norm at most $1$ and invertible (that is, the eigenvalues
  $n_1,\dots, n_k$ lie in the interval $[0,1]$ resp. $(0,1]$.)
  \end{enumerate}
\end{definition}

The notion of tamedness is weakened as follows:
\begin{definition} \label{def:weaklytaming} An almost complex structure $J$ is \em{weakly tamed} by a two-form $\om$ 
if  $\om(v,Jv) \geq 0$ for all tangent vectors $v$. 
\end{definition}

Before stating the main result, we prove Lemma \ref{lem:strongtame}, which is a warm-up result dealing with the easier case of
locally strongly tamed cylindrical almost complex structures (in contrast with the general result which deals with locally tamed cylindrical almost complex structures). 

\begin{lemma}\label{lem:strongtame}
  For an increasing map $\aleph : \nu B^\dual \to B^\dual$ of
  complexes, and any locally strongly tamed cylindrical almost complex
  structure $J^\nu$ on $X^\nu$, $(\psi_{\aleph}^\nu)^*\om_X$ weakly tames
  $J^\nu$, where $\psi_\aleph^\nu :(X^\nu,J^\nu) \to (X,\om_X)$ is the
  map of manifolds induced by $\aleph$.
\end{lemma}
\begin{remark}
  We describe the squashed area form $\psi_\aleph^*\om_X$ induced by a
  squashing map $\aleph$, as it is used in the proofs of Lemmas
  \ref{lem:strongtame} and \ref{lem:dirinc}.  For any $P \in \PP$,
  recall that
  \begin{equation}
    \label{eq:pip-recall}
    \pi_P : \Phinv(\wP) \times \nu P^\dual \to \Phinv(\wP)/T_P  
  \end{equation}
  is a $(T_P \times \nu P^\dual)$-fibration, and on $\Phinv(\wP) \times \nu P^\dual \subset X^\nu$,
  \[ (\psi_{\aleph}^\nu)^*\om_X =
    \om_{X_P} + \d\bran{\aleph,\alpha_P},\]
  where $\aleph_P:=\aleph|(\nu P^\dual)$, and 
  $\alpha_P \in \Om^1(\Phinv(\wP), \t_P)$ is a $\t_P$-valued connection one-form from the symplectic cylindrical structure \eqref{eq:sympcylX}. 
  Indeed,
  $\aleph(\nu P^\dual) \subseteq P^\dual \subset \t_P^\dual$, and
  therefore $\bran{\aleph_P,\alpha_P}$ is well-defined as a one-form.  On
  the vertical sub-bundle $\ker(d\pi_P)$, the form
  $(\psi_{P,\aleph}^\nu)^*\om_X$ is equal to
  $\bran{d\aleph_P \wedge \alpha_P}$.  On the horizontal
  sub-bundle $\ker(\alpha_P) \subset T\Phinv(\wP)$, the form is
  $\om_{X_P} + \aleph_P d\alpha_P$.
\end{remark}
\begin{proof}[Proof of Lemma \ref{lem:strongtame}] 
  We recall from Definition \ref{def:cylneckst} that local strong
  tamedness implies that the fibers of the projection map $\pi_P$ in
  \eqref{eq:pip-recall} are $J^\nu$-holomorphic, and the horizontal
  tangent sub-bundle $\ker(\alpha_P) \subset T\Phinv(\wP)$ is
  $J^\nu$-invariant.  On the vertical sub-bundle $\ker(d\pi_P)$, the
  form $(\psi_{P,\aleph}^\nu)^*\om_X$ is equal to
  $\bran{d(\aleph|\nu P^\dual),\alpha_P}$, which is taming because the
  increasing map $\aleph$ satisfies
  \begin{equation}
    \label{eq:weakinc2}
  \lan D(\aleph|\nu P^\dual) \wedge \alpha_P \ran (v,J^\nu v) > 0, \quad \forall v \in \ker
    (d\pi_P).  
  \end{equation}
  (Note that \eqref{eq:weakinc2} is equivalent to  \eqref{eq:weakinc} below, since $J^\nu$ is strongly locally tamed.) 
  On the horizontal sub-bundle $\ker(\alpha_P) \subset T\Phinv(\wP)$,
  the form is $\om_{X_P} + \aleph_P d\alpha_P$ which tames $J^\nu$,
  since $J^\nu$ is locally strongly tamed.
\end{proof}

\begin{remark}\label{rem:tame-compare}
  The \em{increasing} condition of Definition \ref{def:increasing-maps} produces forms that are weakly taming for cylindrical almost complex structures that lie in a small $C^0$-neighborhood and which are locally tamed.
  If we restrict attention to cylindrical almost complex structures that are locally strongly tamed, that is, those whose underlying connection one-form is fixed, then
  the increasing condition may be replaced by the following condition on
  $\aleph$:
  \begin{equation}
  \label{eq:weakinc}
  \bran{D(\aleph|\nu P^\dual)_xv,v} \geq 0 \quad \forall P \in \PP, x \in \nu P^\dual, v \in T_x(\nu P^\dual),  
\end{equation}
which is equivalent to \eqref{eq:weakinc2} used in the proof of Lemma \ref{lem:strongtame}. 
The condition \eqref{eq:weakinc} is a consequence of the \em{increasing} condition of Definition \ref{def:increasing-maps}, and thus, is a weaker condition. 
\end{remark}

In the next result, which is the main result of this Section, 
we show that given a locally strongly tamed
$\JJ_0$, there is a $C^0$-neighborhood of locally tamed almost complex
structures that are weakly tamed by $\psi_\aleph^*\om_X$ for any 
$\psi_\aleph$ induced by an increasing map $\aleph$. 

\begin{lemma}\label{lem:dirinc}
    \index{Area!Symplectic area!from increasing maps}
  Suppose $\JJ_0 \in \J^\cyl(\XX)$ is a locally strongly tamed
  cylindrical almost complex structure. There
  exists $\eps>0$ and 
  a $C^0$-neighbourhood
  \[U_{\JJ_0}:=\{\JJ \in \J^{\cyl}(\XX) : \Mod{\JJ-\JJ_0}_{C^0}<\eps\}
  \]
  of $\JJ_0$ such
  that the following is satisfied for any $\JJ_1 \in U_{\JJ_0}$ and 
  any $\nu \in [1,\infty)$.  Let $\aleph:\nu B^\dual \to B^\dual$ be a map
  of complexes, and let $\psi_\aleph^\nu :X^\nu \to (X,\om_X)$
  be the resulting map of manifolds.
  
  \begin{enumerate}
  \item If $\aleph$ is increasing, then the form
    $(\psi_{\aleph}^\nu)^*\om_X$ weakly tames $J^\nu_1$, where
    $J_1^\nu \in \J^\cyl(X^\nu)$ is obtained by gluing $\JJ_1$ at
    cylindrical ends.
  \item If $\aleph$ is strictly increasing, then the
    form $(\psi_{\aleph}^\nu)^*\om_X$ tames $J^\nu_1$.
  \end{enumerate}
\end{lemma}
We point out that local tamedness is a $C^0$-open condition in
$\J^\cyl(\XX)$. Any $ \JJ_1 \in U_{\JJ_0}$ is locally tamed. This can
be seen, for example, by setting $\nu=1$, in which case,
$\aleph \equiv \Id$.
\begin{proof}[Proof of Lemma \ref{lem:dirinc}]
  We prove the first statement for increasing maps, since the second
  one is similar.  We consider a polytope $P \in \PP$, and prove the
  (Weakly taming) property of (Definition \ref{def:weaklytaming}) in
  the $P$-cylindrical region
  \[Z_\C:=\Phinv(\wP) \times \nu P^\dual \subset X^\nu. \]
  We denote $Z:=\Phinv(\wP)$.  We recall that for a cylindrical almost
  complex structure $\JJ$ with a family of neck-stretched almost
  complex structures $\{J_\nu\}_\nu$, the fibers of the projection
  \[T_P \times \nu P^\dual \to Z_\C \xrightarrow{\pi_P} Z/T_P \subset
    \XB_P\]
  are $J^\nu$-holomorphic.  Furthermore, on $Z_\C$, $J^\nu$ is
  determined by its projection to $\XB_P$, denoted $D\pi_P(J)$ (see
  \eqref{eq:dpij}), and the associated connection one-form
  $\alpha_{P,\JJ} \in \Om^1(Z,\t_P)$ defined by the condition that the
  horizontal complement of $\ker(D\pi_P)$ given by
  \[\ker(\alpha_{P,\JJ}) \subset TZ \subset TZ_\C\]
  is $J^\nu$-invariant (also see \eqref{eq:connJ}). We recall that
  $\alpha_{P,\JJ}$ is equal to $\alpha_P$ from the symplectic
  cylindrical structure \eqref{eq:sympcylX} if $\JJ$ is locally
  strongly tamed.

  To prove the weakly taming property, we first consider perturbations
  to the horizontal projection of the almost complex structure
  $\JJ_0$.  The horizontal part of the two-form $\psi_{\aleph}^*\om_X$
  is $\om_{X_P} + \aleph \d\alpha_{\JJ_0}$, which is a symplectic form
  since it occurs as a horizontal component of the symplectic form in
  \eqref{eq:sympstr-hofer}; and it tames $D\pi_P(\JJ_0)$ since $\JJ_0$
  is locally tamed (see Remark \ref{rem:oncyl}
  \eqref{part:tame-equiv}).  Since $\aleph$ takes values in a compact
  set $P^\dual$, and tamedness on a symplectic manifold is a
  $C^0$-open condition, we conclude the following:  There is a
  constant $\eps_H$ such that if
  \begin{equation}
    \label{eq:epsh}
    \Mod{D\pi_P(\JJ_1) - D\pi_P(\JJ_0)}_{C^0} < \eps_H 
  \end{equation}
  on $(Z/T_P)$ for some $\JJ_1 \in \J^\cyl(\XX)$, then $D\pi_P(\JJ_1)$
  is tamed by $\om_{X_P} + \aleph \d \alpha_{\JJ_0}$, see Lemma
  \ref{lem:tameC0} for a proof.

  Next, we study the effect of changing the connection associated to
  the almost complex structure. Suppose $\JJ_1 \in \J^{\cyl}$ is a
  cylindrical almost complex structure whose horizontal projection
  $J_P:=D\pi_P(\JJ_1)$ is $C^0$-close to $D\pi_P(\JJ_0)$ as in
  \eqref{eq:epsh}.  Let $\JJ_{10} \in \J^{\cyl}(Z_\C)$ be a locally
  strongly tamed almost complex structure whose horizontal projection
  is the same as that of $\JJ_1$, that is, $D\pi_P(\JJ_{10})=J_P$.
  Denote the $P$-connection forms by
  $\alpha_0:=\alpha_{\JJ_{10}}=\alpha_{\JJ_0},
  \alpha_1:=\alpha_{\JJ_1}$. The difference
  \[ A:=\alpha_{1} - \alpha_{0} \]
  descends to a $\t_P$-valued one-form on $X_P$.  For a vector
  $(v,t) \in T_{\zz,\tau}Z_{\C}$, where $\zz \in Z$,
  $\tau \in \nu P^\dual$, $v \in TZ$, $t \in \t_P^\dual$, the
  difference between $\JJ_1$ and $\JJ_{10}$ is
  \[(\JJ_1-\JJ_{10})(v,t) = -(A(J_P v_P))_Z + J_F (A v_P)_Z, \quad
    v_P:=D\pi_P(v),\]
  where $J_F$ is the complex structure on the fibers of $\pi_P$, and
  for any $\xi \in \t$, $\xi_Z \in \Vect(Z)$ is the vector field
  $\xi_Z(z):=\ddt(\exp(t\xi)z)|_{t=0}$ for any $z \in Z$.  We write
  \begin{equation}
    \label{eq:omJ1}
    (\psi_{\aleph}^*\om_X)((v,t),\JJ_1(v,t))=(\psi_{\aleph}^*\om_X)((v,t),\JJ_{10}(v,t)) + (\psi_{\aleph}^*\om_X)((v,t),(\JJ_1-\JJ_{10})(v,t)).
  \end{equation}
  By \eqref{eq:epsh}, there is a constant $c>0$ such that the first
  term in the right-hand side of \eqref{eq:omJ1} is bounded on the
  $P$-region as
  \begin{equation}
    \label{eq:omJ1r1}
    (\psi_{\aleph}^*\om_X)((v,t),\JJ_{10}(v,t)) \geq  c|v_P|^2 + \lan D\aleph_\tau(\alpha_0(v)),
    \alpha_0(v)\ran + \lan D\aleph_\tau(t), t \ran.
  \end{equation}
  In the second term in the right-hand side of \eqref{eq:omJ1}, the
  difference $(\JJ_1-\JJ_{10})(v,t)$ is in the fiber direction. The
  form $\psi_{\aleph}^*\om_X$ in the fiber is $D\aleph \wedge \alpha_0$
  and therefore,
  \begin{equation}
    \label{eq:omJ1r2}
    (\psi_{\aleph}^*\om_X)((v,t),(\JJ_1-\JJ_{10})(v,t))= \bran{\alpha_0(v), D\aleph_\tau(A( v_P))} +
    \bran{D\aleph_\tau(t),A(J_P v_P)}
  \end{equation}
  By the increasing property, $D\aleph_\tau$ is diagonalizable with
  orthogonal eigenvectors and eigenvalues $n_1,\dots,n_k \in
  [0,1]$.  For any element $\xi \in \t_P$ resp. $\t_P^\dual$, we denote
  by $\xi_i \in \t_P$ resp. $\t_P^\dual$ the projection of $\xi$ to
  the $i$-th eigenspace.  We write
  \[D\aleph_\tau(\xi)=\ssum_{i=1}^kn_i \xi_i.\]
  Using this eigen-decomposition, and the equations \eqref{eq:omJ1},
  \eqref{eq:omJ1r1} and \eqref{eq:omJ1r2}, we get
  \begin{multline*}
    (\psi_{\aleph}^*\om_X)((v,t),\JJ_1(v,t)) \geq c|v_P|^2 +
    \ssum_{i=1}^kn_i( |\alpha_0(v)_i|^2+ |t_i|^2 + \alpha_0(v)_i
    A(v_P)_i + t_i A(J_P
    v_P)_i)\\
    \geq \ssum_{i=1}^kn_i(\tfrac c k |v_P|^2 + |\alpha_0(v)_i|^2+
    |t_i|^2 + \alpha_0(v)_i A(v_P)_i + t_i A(J_P v_P)_i).
  \end{multline*}
  The last two terms are bounded as
  \[\alpha_0(v)_i A(v_P)_i \geq -\tfrac {1} 2(|A|^2|v_P^2| +
    |\alpha_0(v)_i|^2), \quad t_i A(J_P v_P)_i \geq -\tfrac {1}
    2(|A|^2|v_P^2| + |t_i|^2) \]
  where $|A|:=\Mod{A}_{C^0}$.  Therefore,
  $(\psi_{\aleph}^*\om_X)((v,t),\JJ_1(v,t))$ is non-negative if
  $ |A|^2 \leq \frac c {2k}$, leading to the proof of the weak
  tamedness of $J_1^\nu$.
\end{proof}

\section{Squashing maps} \label{sec:sqmaps}

\index{Squashing map|(}

In the last section we saw \em{increasing maps} between complexes of
polytopes induce two-forms that are weakly taming in the sense of
Definition \ref{def:weaklytaming} for cylindrical almost complex
structures in a neighborhood of a locally strongly tamed almost
complex structure.  In this section we construct a large class of
increasing maps, called \em{squashing maps}. Squashing maps are
continuous and piecewise smooth; and on each subset where the map is
smooth, it is a composition of an orthogonal projection, a translation
and a dilation.  In the next section, Hofer energy for a multiple cut
is defined as the supremum over the squashed areas of maps induced by
squashing maps between polytopes.

\begin{remark}\label{rem:poly-convention}
  In the rest of this Chapter, a polytope $P$ is assumed to be a
  subset of an affine space $V$. Two polytopes $P_1, P_2 \subset V$
  are isomorphic if one is a translate of the other in $V$. A
  Euclidean metric on $V$ is inherited by polytopes in $V$.
\end{remark}

\begin{definition}{\rm(Squashing maps)}
\label{def:sqmap}
  Let $V$ be an affine space equipped with a Euclidean metric
  and let $Q \subset V$ be a compact top-dimensional polytope.
  \begin{enumerate}
  \item {\rm(Undilated squashing map)} An \em{undilated squashing map} $\aleph : V \to Q$ has
    \begin{enumerate}
  \item an underlying polyhedral decomposition $\QQ$ of $Q$, that is,
    \begin{equation}
      \label{eq:qqdef}
      Q=\cup_{R \in \QQ}R,  
    \end{equation}
    into polytopes $R$ (of all dimensions in the range $[0,\dim(V)]$) such that
    \begin{itemize}
    \item for any $R_1$, $R_2 \in \QQ$, the interiors are disjoint, that is, $R_1^\circ \cap R_2^\circ=\emptyset$,
    \item and if $R \in \QQ$ and $R_1 \subset R$ is a face, then, $R_1 \in \QQ$;
    \end{itemize}
  \item and a decomposition
    \[V=\cup_{R \in \QQ} \tilde R\]
    into top-dimensional polytopes $\tilde R \subset V$, 
  \end{enumerate}
  such that $\aleph|\tilde R$ is an orthogonal projection onto $R$. That is, the fibers of 
  $\aleph|\tilde R$ are polytopes perpendicular to $R$. The map $\aleph$ is continuous on $V$.
\item{\rm(Squashing map)}
  A \em{squashing map} $\aleph : V \to Q$ is a composition
  \begin{equation}
    \label{eq:aleph-dilate}
    \aleph=\delta_t \circ \aleph_0  
  \end{equation}
  for some $t \geq 1$, 
  where $\aleph_0:V \to tQ$ is an undilated squashing map and $\delta_t : tQ \to Q$ scales by a factor of $\frac 1 t$.
\item {\rm(Unpartitioned squashing map)}
  \label{part:unpart} 
  We say that a squashing map $\aleph:V \to Q$ is \em{unpartitioned}
  if the underlying polyhedral decomposition $\QQ$ is equal to the set of faces of $Q$. 
  \index{Squashing map!Unpartitioned
    squashing map} \index{Squashing map|)}
\end{enumerate}
\end{definition}

\begin{remark}
A squashing map $\aleph: V \to Q$ is continuous, surjective and piecewise smooth.   
\end{remark}

\begin{remark}\label{rem:unpart-inv}{\rm(On unpartitioned squashing maps)}
  An unpartitioned squashing map $\aleph : V \to Q$ is a translation
  composed with scaling on $\aleph^{-1}(Q^\circ)$, and therefore,
  $\aleph$ has a right inverse $\aleph_{\on{inv}} : Q \to V$. That is,
  $\aleph \circ \aleph_{\on{inv}} = \Id_Q$. An unpartitioned squashing map
  is uniquely determined by its right inverse $\aleph_{\on{inv}}$. See
  Figure \ref{fig:sqeg1}.
\end{remark}
\begin{figure}[ht]
  \centering \scalebox{.8}{
\begingroup%
  \makeatletter%
  \providecommand\color[2][]{%
    \errmessage{(Inkscape) Color is used for the text in Inkscape, but the package 'color.sty' is not loaded}%
    \renewcommand\color[2][]{}%
  }%
  \providecommand\transparent[1]{%
    \errmessage{(Inkscape) Transparency is used (non-zero) for the text in Inkscape, but the package 'transparent.sty' is not loaded}%
    \renewcommand\transparent[1]{}%
  }%
  \providecommand\rotatebox[2]{#2}%
  \newcommand*\fsize{\dimexpr\f@size pt\relax}%
  \newcommand*\lineheight[1]{\fontsize{\fsize}{#1\fsize}\selectfont}%
  \ifx\svgwidth\undefined%
    \setlength{\unitlength}{184.74878413bp}%
    \ifx\svgscale\undefined%
      \relax%
    \else%
      \setlength{\unitlength}{\unitlength * \real{\svgscale}}%
    \fi%
  \else%
    \setlength{\unitlength}{\svgwidth}%
  \fi%
  \global\let\svgwidth\undefined%
  \global\let\svgscale\undefined%
  \makeatother%
  \begin{picture}(1,0.43730987)%
    \lineheight{1}%
    \setlength\tabcolsep{0pt}%
    \put(0,0){\includegraphics[width=\unitlength,page=1]{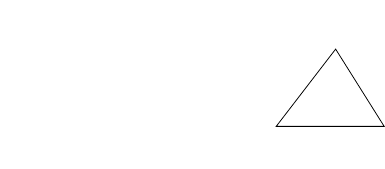}}%
    \put(0.92879793,0.26811308){\color[rgb]{0,0,0}\makebox(0,0)[lt]{\lineheight{0}\smash{\begin{tabular}[t]{l}$Q$\end{tabular}}}}%
    \put(0.1149727,0.33440934){\color[rgb]{0,0,0}\makebox(0,0)[lt]{\lineheight{0}\smash{\begin{tabular}[t]{l}$V$\end{tabular}}}}%
    \put(0,0){\includegraphics[width=\unitlength,page=2]{sqeg1.pdf}}%
    \put(0.29187927,0.18551487){\color[rgb]{0,0,0}\makebox(0,0)[lt]{\lineheight{0}\smash{\begin{tabular}[t]{l}$Q$\end{tabular}}}}%
    \put(0,0){\includegraphics[width=\unitlength,page=3]{sqeg1.pdf}}%
    \put(0.58128042,0.24458945){\color[rgb]{0,0,0}\makebox(0,0)[lt]{\lineheight{0}\smash{\begin{tabular}[t]{l}$\aleph$\end{tabular}}}}%
  \end{picture}%
\endgroup%
}
  \caption{An unpartitioned undilated squashing map $\aleph : V \to Q$. Each of the
    solidly shaded regions is mapped to a vertex of $Q$, ruled regions
    are mapped to the sides of $Q$ by contracting each ruling to a
    point, and the blank region is mapped isometrically to the
    interior of $Q$.}
  \label{fig:sqeg1}
\end{figure}
Unpartitioned squashing maps are sufficient for most purposes. More
general squashing maps are only used in the proof of Proposition
\ref{prop:hofer-breaking}, where we pass from maps in neck-stretched
manifolds $X^\nu$ to maps in broken manifolds in the limit
$\XX_P$. The rest of Section \ref{sec:sqmaps} is used only in the
proof of Proposition \ref{prop:hofer-breaking}. The following is a
preliminary remark about general squashing maps (beyond the
unpartitioned ones).

\begin{remark}
  {\rm(A partial inverse of a squashing map)} We continue Remark
  \ref{rem:unpart-inv} where we defined a right inverse for
  unpartitioned squashing maps.  For a general squashing map
  $\aleph : V \to Q$ with underlying polyhedral decomposition $\QQ$ of
  $Q$, with top-dimensional polytopes $Q_1, \dots,Q_k \in \QQ^{(0)}$,
  a right inverse $\aleph_{\on{inv}}$ exists for the restriction
  $\aleph|(\cup_i \aleph^{-1}(Q_i^\circ))$. The right inverse
  $\aleph_{\on{inv}}|Q_i^\circ : Q_i^\circ \to \aleph^{-1}(Q_i^\circ)$ is a
  translation composed with scaling.
\end{remark}

An undilated squashing map with an underlying polyhedral decomposition
corresponds to a dual complex of $\QQ$, as we show in Lemma
\ref{lem:dual2sq}.  Before stating the result, we point out that the
polyhedral decomposition in \eqref{eq:qqdef} in the definition of
squashing maps is a decomposition of a polytope $Q \subset V$, and not
of the entire vector space $V$. This is in contrast with the
polyhedral decomposition of Chapter \ref{chap:bsymp} which was defined
for the entire vector space.  A decomposition of the vector space $V$
induces a decomposition of $Q \subset V$ as we now describe:

Let $V$ be a vector space and let $Q \subset V$ be a polytope of the
same dimension as $V$.  A polyhedral decomposition $\PP$ of the vector
space $V$ for which all top-dimensional polytopes $P \in \PP^{(0)}$
intersect $Q^\circ$ induces a polyhedral decomposition $\PP(Q)$ of $Q$
whose top-dimensonal polytopes are
\[\PP(Q)^{(0)}:=\{P \cap Q : P \in \PP^{(0)}\},\]
and whose lower dimensional polytopes are faces of polytopes in
$\PP_Q^{(0)}$; see Figure \ref{fig:pqdual}.

\begin{figure}[ht]
  \centering \scalebox{.8}{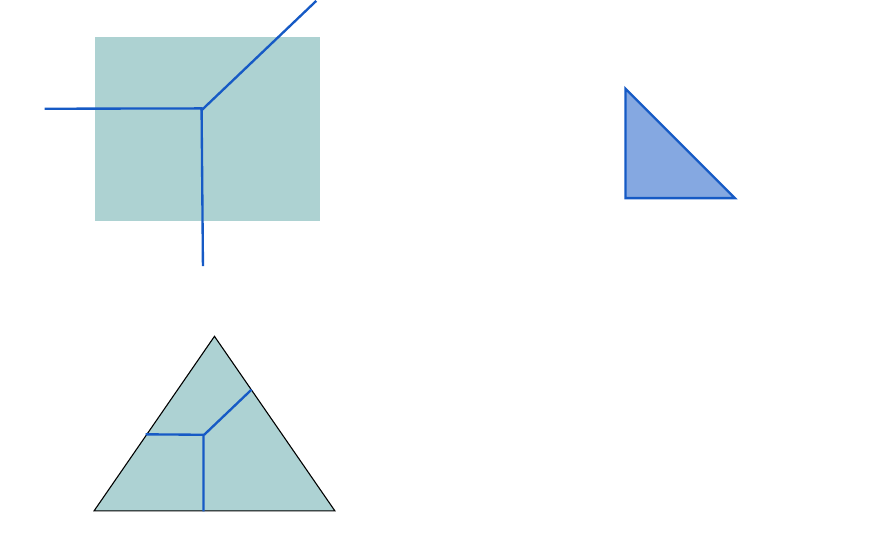}
  \caption{Left: A polyhedral decomposition $\PP$ of $V$ induces a polyhedral decomposition $\PP(Q)$ of $Q \subset V$. Right : The corresponding dual complexes.}
  \label{fig:pqdual}
\end{figure}

In Lemma \ref{lem:dual2sq}, we will construct a sequence of undilated
squashing maps $\aleph_\nu : V \to Q$ whose underlying polyhedral
decomposition of $Q$ is $\PP(Q)$ described in the previous paragraph.
The squashing maps will have the property that for any two
top-dimensional polytopes $Q_i, Q_j$ of $\QQ$, the distance between
$\aleph^{-1}(Q_i^\circ)$, $\aleph^{-1}(Q_j^\circ)$ goes to infinity as
$\nu \to \infty$.

\begin{lemma}\label{lem:dual2sq}
  {\rm(From a dual complex to squashing maps)}
  Let $Q$, $V$, $\PP$, $\PP(Q)$ be as above. Suppose $B^\dual_\PP$ is a dual complex of $\PP$ that respects the metric on $V$ (as in Remark \ref{rem:const-g}). 
  \begin{enumerate}
  \item \label{part:dual2sq1} There is a dual complex
    $B^\dual_{\PP(Q)}$ of $\PP(Q)$ that respects the metric on $V$ and
    for which $B^\dual$ is a subcomplex.
  \item \label{part:dual2sq2} For any $\nu \in \R_{>0}$, there is a
    squashing map $\aleph_\nu : V \to Q$ such that for any polytope
    $P \in \PP(Q)$,
    \begin{equation}
      \label{eq:nuinv}
      \aleph^{-1}_\nu(P) \xrightarrow{\text{isomorphic}}  P \times \nu P^\dual,
    \end{equation}
    where $P^\dual \subset V$ is the dual polytope of $P$ in
    $B^\dual_{\PP(Q)}$.
  \item \label{part:dual2sq3} Let $P \in \PP(Q)^{(0)}$ be a
    top-dimensional polytope, and for any $\nu$, let $i_\nu : P \to V$
    be an embedding (which is a translation).  For any $\nu$, there is
    a unique squashing map that satisfies \eqref{eq:nuinv} and the
    condition that $\aleph_\nu \circ i_\nu=\Id_P$.  For any two
    top-dimensional polytopes $Q_i, Q_j \in \PP(Q)^{(0)}$,
    \begin{equation}
      \label{eq:QiQjdV}
    d_V(\aleph^{-1}(Q_i^\circ),\aleph^{-1}(Q_j^\circ)) \to \infty
    \quad \text{as $\nu \to \infty$.}   
    \end{equation}
    \end{enumerate}
\end{lemma}

\begin{proof}
  To construct $B^\dual_{\PP(Q)}$, we start with
  $B^\dual_\PP \subset \t^\dual$, and add polytopes corresponding to
  $P_0 \cap Q_0$ where $P_0 \in \PP$ and $Q_0$ is a proper face of
  $Q$.  First consider the case that $Q_0$ is a facet and $P_0$ is
  top-dimensional. Then, the one-dimensional polytope
  $(P_0 \cap Q_0)^\dual$ is a ray pointing in the outward normal
  direction to $Q_0$,   whose starting point is
  $P_0^\dual \in B^\dual_{\PP}$. Next, consider the general case when
  the face $Q_0 \subset Q$ is the intersection of facets
  $Q_1,\dots,Q_k$, and the dual polytope of $P_0 \in \PP$ in
  $B^\dual_\PP$ has vertices $P_1, \dots, P_\ell$.  Then, we define

  \[(P_0\cap Q_0)^\dual:=\on{Convex-hull}(\{(P_j \cap Q_i)^\dual\}_{i,j}),   \]
  which has the right dimension for the following reason: For any $j$,
  the convex hull of $C_j:=\{P_j \cap Q_i\}_i$ is a cone spanned by
  the rays $(P_j \cap Q_i)^\dual$ with vertex $P_j^\dual$.  Denote
  $C:=C_j$, since for any $j'$, $C_{j'}$ is a translate of $C_j$.  The
  convex hull of $\{P_j \cap Q_i\}_{i,j}$ is the product
  $C \times P_0^\dual$, and therefore, has codimension
  $k + \codim(P_0)$.  Here, $P_0^\dual$ is the dual polytope of $P_0$
  in $B^\dual_\PP$.

  To prove \eqref{part:dual2sq2}, we observe that for any $\nu >0$
  there is an isomorphism
  \begin{equation}
    \label{eq:embedV}
    \left(\sqcup_{P \in \PP(Q)}(\nu P^\dual \times P)/\sim\right) \xrightarrow{i_\nu} V,  
  \end{equation}
  which is uniquely defined up to a translation in $V$. Here,
  $P^\dual$ is the dual polytope of $P$ in $B^\dual_{\PP(Q)}$, and for
  any pair $Q \subset P$ with $\codim_PQ=1$, $\sim$ identifies
  $(\nu P^\dual \times Q)$, which is a facet in both
  $(\nu P^\dual \times P)$ and $(\nu Q^\dual \times Q)$. In
  \eqref{eq:embedV}, each of the maps $i_\nu|(\nu P^\dual \times P)$
  is a translation.  The squashing map $\aleph_\nu$ is defined as
  projection to $P$ on the subset $i_\nu(\nu P^\dual \times P)$.

  To prove \eqref{part:dual2sq3}, we observe that \eqref{eq:embedV} is
  uniquely defined up to translation in $V$. The condition
  $\aleph_\nu \circ i_\nu=\Id_P$ cuts down this translation freedom
  and fixes a unique $\aleph_\nu$. The property \eqref{eq:QiQjdV}
  follows from the construction of $\aleph_\nu$, since the two sets
  being considered are separated by a dual polytope scaled up by
  $\nu$. This finishes the proof of the Lemma.
\end{proof}

\begin{figure}[ht]
  \centering \scalebox{.8}{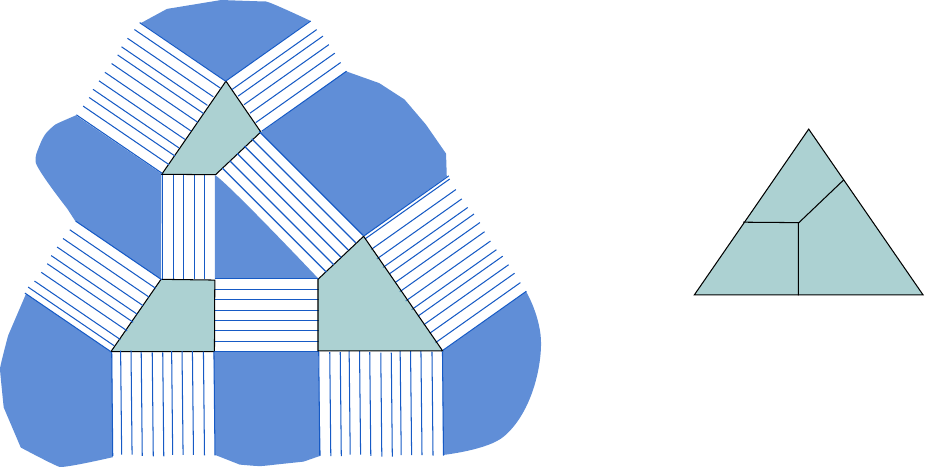}
  \caption{
    The undilated squashing map $\aleph:V \to Q$ corresponding to the dual complex $B^\dual_{\PP(Q)}$ in Figure \ref{fig:pqdual}. 
    In $V$, solidly shaded regions are mapped to points,
    ruled regions are mapped to lines by contracting each ruling to a
    point, and blank regions are mapped isometrically.}
  \label{fig:parted}
\end{figure}

The following result constructs a polyhedral decomposition (with dual
complex) of a polytope $Q$ that contains a prescribed top-dimensional
piece.  The polyhedral decomposition is used to construct squashing
maps in the proof of Proposition \ref{prop:hofer-breaking}.

\begin{lemma}\label{lem:dec-with-dual}
  Let $V$ be a vector space of dimension $n$ equipped with a Euclidean
  metric.  Let $C \subset V$ be a $k$-dimensional cone, that is, $C$ is
  the convex hull of $k$ rays originating at $0 \in V$ and pointing in
  linearly independent directions.  Let $C^\dual \subset V$ be an
  $(n-k)$-dimensional simplex that is orthogonal to $P$. Then, there
  is a polyhedral decomposition $\PP$ of $V$ that contains
  $P \times P^\dual$ as a top-dimensional polytope, and which has a
  dual complex $B^\dual_\PP$ respecting the metric on $V$.
\end{lemma}

\begin{proof}
  First, we construct a polyhedral decomposition $\PP_\pre$ of $V$
  that contains $C$ as a polytope, and whose dual complex is an
  $n$-simplex. Assuming that the cone $C$ is the convex hull of rays
  pointing in the directions $e_1,\dots, e_k$, choose vectors
  $e_{k+1},\dots, e_{n+1}$ such that $V$ is the convex hull of the
  rays $\R_{\geq 0} e_i$, $i=1,\dots, n+1$, and (a translate of)
  $C^\dual$ is a cross-section of the cone spanned by the rays
  $\{\R_{\geq 0}e_i\}_{k+1 \leq i \leq n+1}$.  Define a polyhedral
  decomposition
  \[\PP_\pre := \{P_I:=\R_+\bran{x_i}_{i \in I} \subset V : I \subset
    \{1,\dots,n+1\}\} \]
  of $V$, and observe that its dual complex (respecting the metric) is
  an $n$-simplex denoted by $B^\dual_{\PP_\pre}$.  Define the 
  polyhedral decomposition $\PP$ so that its top-dimensional polytopes
  are
  \[\tilde C:=C^\dual + \R_{\geq 0}\bran{e_i}_{1 \leq i \leq k} \simeq
    C \times C^\dual,\]
  and $P \bs \tilde C$ for all top-dimensional polytopes
  $P \in \PP_\pre^{(0)}$; see Figure \ref{fig:PCconst}.  The
  polyhedral decomposition $\PP$ has a dual complex $B^\dual_\PP$,
  which is obtained by adding a vertex $\tilde C^\dual$ to the face
  $P^\dual_{1,\dots,k}$ such that the convex hull of
  $\tilde C^\dual \cup P^\dual_{\{k+1,\dots,n+1\}}$ is perpendicular
  to $P^\dual_{\{1,\dots,k\}}$. The other polytopes in $B^\dual_\PP$
  are spanned by the given set of vertices, and can be determined
  combinatorially.
\end{proof}

\begin{figure}[ht]
  \centering \scalebox{.8}{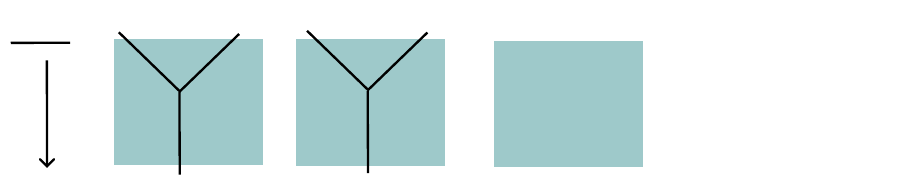}
  \caption{Constructing $\PP$ from $\PP_\pre$ in Lemma
    \ref{lem:dec-with-dual}. Here $n=2$, $k=1$.}
  \label{fig:PCconst}
\end{figure}

\section{Multi-directional Hofer energy}
Hofer energy for a multiple cut is defined as the supremum over
pullbacks of symplectic forms by squashing maps between manifolds that
are induced by squashing maps between polytopes.  The definition of
Hofer energy uses squashing maps from $\om$-complexes to
$J$-complexes.  In this section, we define these complexes for
neck-stretched and broken manifolds, and prove some useful properties
satisfied by Hofer energy.

In the last section, we defined squashing maps as maps from an affine
space to a compact polytope. The definition extends naturally to maps
between polytopes, and later, complexes.

\begin{definition}\label{def:polytope-squash}
  {\rm(Squashing map between polytopes)} Let $P, Q \subset V$ be
  polytopes, and let $Q$ be compact. A squashing map
  $\aleph : P \to Q$ is the restriction of a squashing map
  $\ol \aleph : V \to Q$ for which the subset $V_0 \subset V$ on which
  $\ol \aleph$ is a dilation is contained in $P$.
\end{definition}

\subsection{Hofer energy for neck-stretched manifolds}
We recall from Section \ref{sec:incmap} that for a neck-stretched
manifold $X^\nu$, the $\om$-complex is the dual complex $B^\dual$ and
the $J$-complex is $\nu B^\dual$.

\index{Hofer energy!for neck-stretched manifolds}
\begin{definition} \label{def:squashedarea} {\rm(Hofer energy for
    neck-stretched manifolds)} Let $\{X^\nu\}_\nu$ be a family of
  neck-stretched manifolds.  For any $\nu \geq 1$, the \em{Hofer
    energy} of a map $u:C \to X^\nu$ is
  \[E_{\Hof}(u) = \sup_{\aleph:\nu B^\dual \to B^\dual} \int_C
  (\psi_{\aleph} \circ u)^* \om_X, \]
where the supremum is over all maps of complexes
$\aleph:\nu B^\dual \to B^\dual$ for which
$\aleph|\nu P^\dual : \nu P^\dual \to P^\dual$ is a squashing map for
each polytope $P \in \PP$ (as in Definition
\ref{def:polytope-squash}); and
\[\psi_{\aleph}: X^\nu \to (X,\om_X)\]
is the map induced by $\aleph$ as in \eqref{eq:alephembed}.  For any
squashing map $\aleph$, the form $\psi_\aleph^*\om_X$ is called a \em{
  squashed area form}. The squashing map $\aleph$ and the squashed
area form are called \em{unpartitioned} if for each polytope $P$, the
squashing map $\aleph|\nu P^\dual$ is unpartioned in the sense of
Definition \ref{def:sqmap} \eqref{part:unpart}. The supremum over
unpartitioned squashed area forms is called \em{unpartitioned Hofer
  energy} and is denoted by
\[E^*_{\Hof}(u) = \sup_{\aleph:\nu B^\dual \to B^\dual \text { is unpartitioned}} \int_C
  (\psi_{\aleph} \circ u)^* \om_X.\]
See Figures \ref{fig:nondist} and \ref{fig:2cellnu} for examples of
squashing maps for neck-stretched manifolds. This ends the Definition.
\end{definition}
\begin{figure}[ht]
  \centering \scalebox{.8}{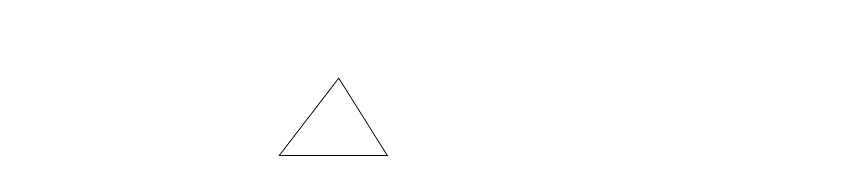}
  \caption{Examples of squashing maps $\aleph :\nu B^\dual \to B^\dual$. In both examples, solidly
    shaded regions are mapped to points, ruled regions are mapped to
    lines by contracting each ruling to a point, and blank regions are
    mapped isometrically. The map in the left is unpartitioned, the map in the right is not.}
  \label{fig:nondist}
\end{figure}
\begin{figure}[ht]
  \centering \scalebox{.8}{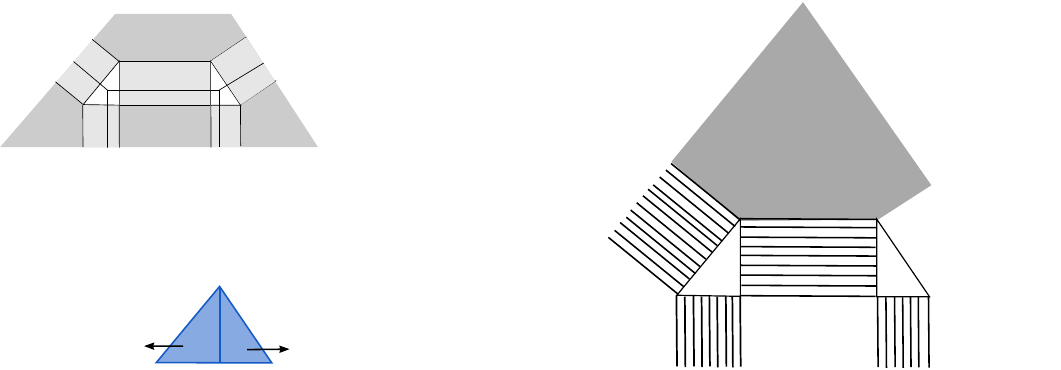}
  \caption{A squashing map $\aleph:\nu B^\dual \to B^\dual$ where there are two top-dimensional dual polytopes $Q_0^\dual$, $Q_1^\dual$ in $B^\dual$. The maps $\aleph|\nu Q_0^\dual$, $\aleph|\nu Q_1^\dual$ are equal on $\nu P^\dual$. The map $\aleph$ is unpartitioned.}
  \label{fig:2cellnu}
\end{figure}
For any squashing map $\aleph$, the map of manifolds
$\psi_\aleph : X^\nu \to X$
is continuous and piecewise-smooth
with a finite number of pieces. Therefore, the integral computing
squashed area is well-defined.
\begin{remark}\label{rem:sq-nondeg}
  {\rm(When is squashed area non-degenerate)} Given a squashing map
  $\aleph : \nu B^\dual \to B^\dual$, the squashed area form is
  non-degenerate on the subset $S \subset \nu B^\dual$ where $\aleph$
  is dilation by a positive factor.  For any polytope $P \in \PP$,
  suppose
  \[\aleph_P:= \aleph|\nu P^\dual : \nu P^\dual \to P^\dual\]
  has an underlying decomposition $\QQ_P$ of $P^\dual$ (as in
  \eqref{eq:qqdef}). Then, for any top-dimensional polytope
  $Q \in \QQ_P$, $\psi_\aleph^*\om_X$ is non-degenerate on
  $\pi_{\nu B^\dual}^{-1}(\aleph_P^{-1}(Q^\circ))$.
\end{remark}
\begin{remark}\label{rem:hof-area}
  For a compact curve $C$ with boundary, and a map
  $u:(C,\partial C) \to (X^\nu,L)$, the Hofer energy is equal to the
  area of $u$ :
  \[E_\Hof(u)=\bran{(\psi_\aleph \circ u)_*[C],[\om_X]}\]
  for any squashing map $\aleph : \nu B^\dual \to B^\dual$.  The
  quantity $E_\Hof$ is independent of $\aleph$, because any pair of
  squashing maps $\aleph_0$, $\aleph_1$ are isotopic via a family
  $\{\aleph_t\}_{t \in [0,1]}$ of continuous piecewise smooth maps,
  and therefore, the maps $(\psi_{\aleph_0} \circ u)$,
  $(\psi_{\aleph_1} \circ u) : C \to (X,\om_X)$ are isotopic.
\end{remark}

We point out that squashing maps satisfy the increasing property
(Definition \ref{def:increasing-maps}), which we initially set out to
achieve.
\begin{lemma}{\rm(Squashed area forms are weakly taming)}
  \label{lem:sq-are-inc}
  Suppose $\JJ_0 \in \J^\cyl(\XX)$ is a locally strongly tamed
  cylindrical almost complex structure, and suppose
  $U_{\JJ_0} \subset \J^{\cyl}(\XX)$ is a $C^0$-neighborhood of
  $\JJ_0$ from Lemma \ref{lem:dirinc}. Then, for any $\nu$ and any
  squashing map $\aleph: \nu B^\dual \to B^\dual$, the squashed area
  form $\psi_\aleph^*\om_X$ is weakly taming (Definition
  \ref{def:weaklytaming}) for any $J^\nu \in \JJ^\cyl(X^\nu)$ obtained
  by gluing $\JJ \in U_{\JJ_0}$ at cylindrical ends.
\end{lemma}
 \begin{proof}
   A squashing map $\aleph : \nu B^\dual \to B^\dual$ satisfies the
   increasing property on $\aleph|\nu P^\dual$.  Indeed, it is enough to
   check this property for squashing maps on vector spaces, which in
   turn, follows from the increasing property on undilated squashing
   maps, and finally, an undilated squashing map is piecewise smooth,
   and on each of these pieces, it is an orthogonal projection.
 \end{proof}
 The next result, which is a consequence of Lemma
 \ref{lem:sq-are-inc}, says that the squashed area
 $(\psi_{\aleph} \circ u)^* \om_X$ is pointwise non-negative if the map
 $u$ is pseudoholomorphic with respect to an almost complex structure
 that is close to a locally strongly tamed almost complex structure.
 \begin{lemma} \label{lem:hmon} {\rm(Monotonicity of Hofer
     energy)} \label{lem:monot} Suppose $\JJ_0 \in \J^\cyl(\XX)$ is a
   locally strongly tamed cylindrical almost complex structure, and
   suppose $U_{\JJ_0} \subset \J^{\cyl}(\XX)$ is a $C^0$-neighborhood
   of $\JJ_0$ from Lemma \ref{lem:dirinc}.  Let $u: C \to X^\nu$ be a
   map that is holomorphic with respect to the domain dependent almost
   complex structure $J^\nu : C \to U_{\JJ_0}$.  (Here, for any
   $z \in C$, $J^\nu(z)$ denotes both a broken almost complex
   structure $\JJ_z$ on $\XX$ and the almost complex structure
   $J_z^\nu$ on $X^\nu$ obtained by gluing $\JJ_z$ on the neck with
   neck length parameter $\nu$.)  For any open subset $\Om \subset C$,
   \[ E_{\Hof}(u,\Om) \leq E_{\Hof}(u,C) .\]
\end{lemma}

For some results in Chapter \ref{chap:stabdiv}, we consider maps that
are holomorphic with respect to a locally strongly tamed almost
complex structure, in which case, monotonicity of Hofer energy holds
without restricting to a $C^0$-neighborhood. The relevant statement is
as follows, and the proof is a consequence of Lemma
\ref{lem:strongtame}.

 \begin{lemma}\label{lem:monot-strong} 
   {\rm(Monotonicity in the locally strongly tamed case)} Suppose
   $\JJ_0 \in \J^\cyl(\XX)$ is a locally strongly tamed cylindrical
   almost complex structure.  Let $u: C \to X^\nu$ be a
   $J_0^\nu$-holomorphic map. Here, the almost complex structure
   $J_0^\nu$ on $X^\nu$ is obtained by gluing $\JJ_0$ on the neck with
   neck length parameter $\nu$.  For any open subset $\Om \subset C$,
   \[ E_{\Hof}(u,\Om) \leq E_{\Hof}(u,C) .\]
 \end{lemma}

\subsection{Hofer energy on a broken manifold} \label{sec:hofbr1}

Hofer energy on broken manifolds is defined in a similar way
to the Hofer energy 
on neck-stretched manifolds. The only new feature is that the
$\om$-complex and $J$-complex are different.

We describe the $\om$-complexes for pieces of the broken manifold.
For a polytope $P \in \PP$, the $\om$-complex of the cut space
$X_P^\om$ is the subset
\[\BB^\dual_P:=\pi_{B^\dual}(\XB_P^\om) \subset B^\dual,\]
where $\pi_{B^\dual}:(X,\om_X) \to B^\dual$ is the projection to the
dual complex \label{rep:xpom} from \eqref{eq:pib-symp} and we recall
that the symplectic cut space $\XB_P^\om=\Phinv(P)$ is a subset of
$(X,\om_X)$. (We recall that the symplectic cut space and symplectic
broken manifolds have a superscript $\om$, see
Notation \ref{note:symp-vs-ac}.)
The complex
$\BB^\dual_P$ may alternately be defined as
\begin{equation}
  \label{eq:pi2bbp}
   \BB^\dual_P:= \bigcup_{Q \in \PP, \dim(Q)=0, Q \in P}i_{\tQ}^{-1}(P) \subset B^\dual,
\end{equation}
where $i_{\tQ} : Q^\dual \to \on{im}(\Phi) \subset \t^\dual$ is the
embedding from \eqref{eq:ipdef}, noting that for a zero-dimensional
polytope $Q$, $\tQ=Q^\dual$.  See Figure \ref{fig:bbp-new}. The space
$\BB^\dual_P$ inherits the structure of a complex from $B^\dual$. It
is a union of polytopes
\begin{equation}
  \label{eq:bbpq}
  \BB^\dual_P=\cup_{Q \in \PP : Q \subseteq P}\BB^\dual_{P,Q}, \quad \BB^\dual_{P,Q}:=\BB^\dual_P \cap Q^\dual,   
\end{equation}
and for any pair $R \subset Q \subseteq P$, $\BB^\dual_{P,Q}$ is
identified to a face of $\BB^\dual_{P,R}$.  Note that
$\dim(\BB^\dual_P)=\dim(P)$, and
$\dim(\BB^\dual_{P,Q})=\dim(P)-\dim(Q)$.

\begin{figure}[ht]
  \centering \scalebox{.8}{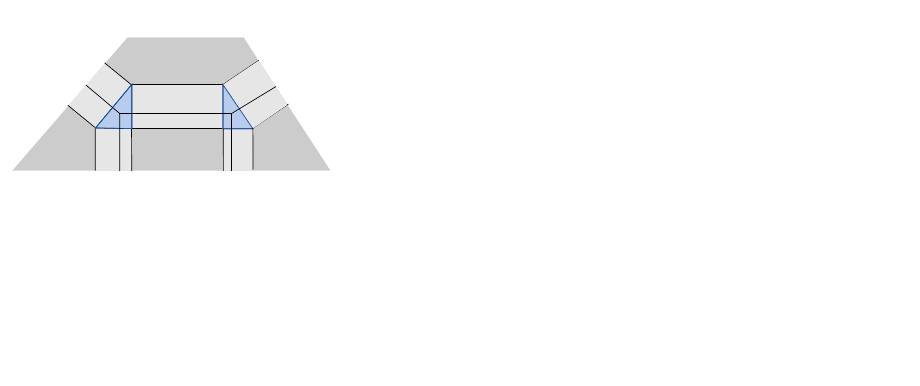}
  \caption{For a polyhedral decomposition $\PP$ of a tropical
    manifold, $\BB^\dual_P$ resp. $\BB^\dual_{\tP}$ is the
    $\om$-complex of the cut space $X_P$ resp.  Broken manifold
    $\XC_P$ for $P \in \PP$.}
  \label{fig:bbp-new}
\end{figure}

\begin{figure}[ht]
  \centering \scalebox{.8}{
\begingroup%
  \makeatletter%
  \providecommand\color[2][]{%
    \errmessage{(Inkscape) Color is used for the text in Inkscape, but the package 'color.sty' is not loaded}%
    \renewcommand\color[2][]{}%
  }%
  \providecommand\transparent[1]{%
    \errmessage{(Inkscape) Transparency is used (non-zero) for the text in Inkscape, but the package 'transparent.sty' is not loaded}%
    \renewcommand\transparent[1]{}%
  }%
  \providecommand\rotatebox[2]{#2}%
  \newcommand*\fsize{\dimexpr\f@size pt\relax}%
  \newcommand*\lineheight[1]{\fontsize{\fsize}{#1\fsize}\selectfont}%
  \ifx\svgwidth\undefined%
    \setlength{\unitlength}{91.56728189bp}%
    \ifx\svgscale\undefined%
      \relax%
    \else%
      \setlength{\unitlength}{\unitlength * \real{\svgscale}}%
    \fi%
  \else%
    \setlength{\unitlength}{\svgwidth}%
  \fi%
  \global\let\svgwidth\undefined%
  \global\let\svgscale\undefined%
  \makeatother%
  \begin{picture}(1,0.72909716)%
    \lineheight{1}%
    \setlength\tabcolsep{0pt}%
    \put(0,0){\includegraphics[width=\unitlength,page=1]{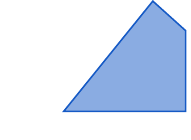}}%
    \put(0.61901045,0.30592187){\color[rgb]{0,0,0}\makebox(0,0)[lt]{\lineheight{1.25}\smash{\begin{tabular}[t]{l}$\BB_{P_0,R_0}^\dual$\end{tabular}}}}%
    \put(0.55139492,0.03818717){\color[rgb]{0,0,0}\makebox(0,0)[lt]{\lineheight{1.25}\smash{\begin{tabular}[t]{l}$\BB_{P_0,Q_1}^\dual$\end{tabular}}}}%
    \put(0.18988752,0.43809297){\color[rgb]{0,0,0}\makebox(0,0)[lt]{\lineheight{1.25}\smash{\begin{tabular}[t]{l}$\BB_{P_0,Q_3}^\dual$\end{tabular}}}}%
    \put(-0.00265905,0.1014019){\color[rgb]{0,0,0}\makebox(0,0)[lt]{\lineheight{1.25}\smash{\begin{tabular}[t]{l}$\BB_{P_0,P_0}^\dual$\end{tabular}}}}%
  \end{picture}%
\endgroup%
}
  \caption{Polytopes in the complex $\BB_{P_0}^\dual$ from Figure \ref{fig:bbp-new}.}
  \label{fig:bbp-polys}
\end{figure}

\begin{figure}[ht]
  \centering \scalebox{.8}{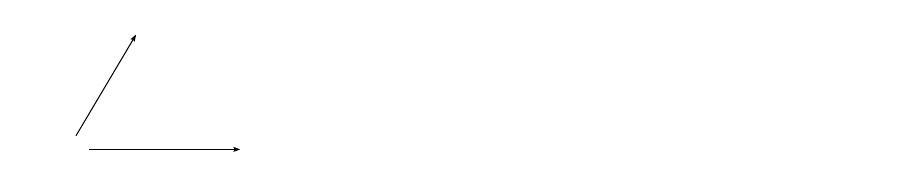}
  \caption{Some $J$-complexes for the polyhedral decomposition $\PP$ in Figure \ref{fig:bbp-new}.}
  \label{fig:Jcomplex}
\end{figure}

We also associate an $\om$-complex to thickened complexes $\tP$ in
order to define Hofer energy for maps in the broken manifold
$\XC_P$. We define $\BB^\dual_{\tP}$ as a thickening of $\BB^\dual_P$
:
\[\BB^\dual_{\tP} := \BB^\dual_P \times P^\dual, \]
where $P^\dual \subset \t_P$ is  small enough that 
there is an embedding
$\BB^\dual_{\tP} \to B^\dual$ whose image is a tubular neighborhood of
$\BB^\dual_P \subset B^\dual$ (if needed we assume that $P^\dual$ is a scaling of the dual polytope of $P$). The space $\BB^\dual_{\tP}$ inherits
the structure of a complex from $B^\dual$, and thus, consists of
polytopes
\[\BB^\dual_{\tP,Q}:=\BB^\dual_{P,Q} \times P^\dual \quad \forall Q \subseteq P.\]
The choice of the $\om$-complexes $\BB^\dual_P$, $\BB^\dual_{\tP}$ is
justified by the fact that the symplectic broken manifold or the symplectic cut space
can be re-constructed from the complexes as
\begin{equation}
  \label{eq:xp-symp-dec}
  \begin{split}
    \left( \cup_{Q \subseteq P} \Phinv(\wQ) \times \BB^\dual_{\tP,Q}\right)/\sim =\Phinv(P^\circ)=(\XX_{P}^\om,\om_{\XX_P}),\\
       \left( \cup_{Q \subseteq P} \Phinv(\wQ)/T_P \times \BB^\dual_{P,Q}\right)/\sim = \Phinv(P^\circ)/T_P= (X_P^\om,\om_{X_P}),
  \end{split}
\end{equation}
where $\sim$ is an identification on boundaries, which is a
restriction of the equivalence relation in \eqref{eq:xnupre1}.  The
decomposition in \eqref{eq:xp-symp-dec} is a consequence of the
decomposition of $(X,\om_X)$ in \eqref{eq:xnupre3}.

Next, we describe the $J$-complexes.  For any $P \in \PP$, the
$J$-complex for the broken manifold $\XB_P$ is
\[\Cone_{P^\dual} B^\dual := \left( \bigcup_{Q \subseteq P} \Cone_{P^\dual}Q^\dual \right)/\sim, \]
where, for any pair $Q_0 \subset Q_1$, the equivalence relation $\sim$
identifies $\Cone_{P^\dual}Q^\dual_1$ to a face of
$\Cone_{P^\dual}Q^\dual_0$. The $J$-complex for the cut space $\XB_P$
is the corresponding normal cone
\[\NCone_{P^\dual} B^\dual := \left(\bigcup_{Q \subseteq P} \NCone_{P^\dual}Q^\dual \right)/\sim, \]
and thus $\Cone_{P^\dual} B^\dual$ is a product of orthogonal spaces
\[\Cone_{P^\dual} B^\dual=\NCone_{P^\dual} B^\dual \times \t_P.\]
The almost complex broken manifold and cut spaces can be reconstructed from the $J$-complexes
\begin{equation}
  \label{eq:xolp-dec}
  \begin{split}
      \XC_P&=\left( \cup_{Q \subseteq P} \Phinv(\wQ) \times \Cone_{P^\dual}Q^\dual\right)/\sim,\\
  \XB_P&=\left( \cup_{Q \subseteq P} \Phinv(\wQ)/T_P \times \NCone_{P^\dual}Q^\dual\right)/\sim,
  \end{split}
\end{equation}
where, for any facet $Q \subset R$, $\sim$ identifies the boundary component
\[\Phinv(\wQ) \times \Cone_{P^\dual}(Q^\dual) \subset \Phinv(\wQ) \times \Cone_{P^\dual}(Q^\dual)\]
with the boundary component
\[\Phinv(\wQ) \times \Cone_{P^\dual}(Q^\dual) \subset \Phinv(\wR) \times \Cone_{P^\dual}(R^\dual)\]
by the identity map. Here we use the viewpoint that broken manifolds
are degenerate limits of neck-stretching, and therefore the
decomposition in \eqref{eq:xolp-dec} is the limit of the
$J$-decomposition in \eqref{eq:xnuJ-dec}. As a consequence of the
decomposition in \eqref{eq:xolp-dec}, there are projection maps
\begin{equation}
  \label{eq:proj2cone}
\pi_{\Cone_{P^\dual}B^\dual} : \XC_P \to \Cone_{P^\dual}B^\dual, \quad \pi_{\NCone_{P^\dual}B^\dual} : \XB_P \to \NCone_{P^\dual}B^\dual  
\end{equation}
for all polytopes $P \in \PP$.

\begin{definition}\label{def:sqbroken}  Let $\XX_\PP$ be a
  broken manifold and let $P \in \PP$ be a polytope.  \index{Squashing
    map!for cut spaces}
  \begin{enumerate}
  \item {\rm(Squashing map for cut spaces)}
    A \em{squashing map for the cut space} $\XB_P$ is a map of complexes
    \[\aleph : \NCone_{P^\dual} B^\dual \to \BB_P^\dual \]
    for which
    $\aleph(\NCone_{P^\dual}Q^\dual) \subset \BB_{P,Q}^\dual$ for any
    polytope $Q \subseteq P$, and
    $\aleph : \Cone_{P^\dual}Q^\dual \to \BB_{P,Q}^\dual$ is a
    squashing map of polytopes.
    \index{Squashing map! for a broken manifold}
  \item {\rm(Squashing map for a broken manifold)} A \em{squashing
    map} for the component $\XC_P$ of the broken manifold $\XX_\PP$ is a map 
  \[\aleph=(\aleph^P, \aleph^{P^\dual}): \Cone_{P^\dual} B^\dual \to \BB_{\tP}^\dual \]
  where $\aleph^P : \NCone_{P^\dual}B^\dual \to \BB_P^\dual$ is a squashing map for the cut space $\XB_P$,
  and $\aleph^{P^\dual}:\t_P \to P^\dual$ is a squashing map of polytopes.  Note that
  $\aleph$ is itself a squashing map of complexes.
\end{enumerate}
See Figure \ref{fig:sqbr-eg} for examples.
\end{definition}
\begin{figure}[ht]
  \centering \scalebox{.8}{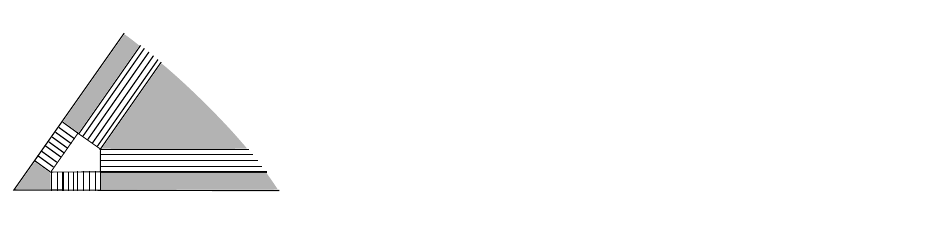}
  \caption{Squashing maps for a broken manifold. 
    $\aleph_1:\Cone_{P_0^\dual}R_0^\dual \to \BB_{P_0}^\dual$,
    $\aleph_2 : \Cone_{Q_1^\dual}R_0^\dual \to \BB_{Q_1}^\dual$,
 $\aleph_3:\Cone_{P_3^\dual}B^\dual \to \BB_{P_3}^\dual$ with
  $\om$-complexes and $J$-complex from Figures \ref{fig:bbp-new},
  \ref{fig:Jcomplex}.}
  \label{fig:sqbr-eg}
\end{figure}

Via the decompositions \eqref{eq:xp-symp-dec}, \eqref{eq:xolp-dec} of
a cut space into a union of fibrations over polytopes, a squashing map
$\aleph$ for the almost cut space $\XB_P$ induces a map of manifolds
\[\psi_\aleph : \XB_P \to X_P^\om, \]
that is, piecewise, a smooth submersion onto the symplectic cut space
$X_P^\om$.  Similarly for a component $\XC_P$ of an almost complex
broken manifold $\XX$, a squashing map $\aleph$ induces a map of
manifolds
\[\psi_\aleph : \XC_P \to \XX_P^\om \]
onto the symplectic broken manifold $\XX_P$.
\index{Hofer energy!for a broken manifold}
\begin{definition} {\rm(Hofer energy for a broken manifold)}
  \label{def:PHof}
  Let $P \in \PP$ be a polytope.  The \em{$P$-Hofer energy} of a map
  $u:C \to \XC_P$ is
  \[E_{P,\Hof}(u) = \sup_{\aleph} \int_C (\psi_\aleph \circ u)^*
    \om_{\XC_P}, \]
  where the supremum is over all squashing maps $\aleph$ for $\XC_P$
  (as in Definition \ref{def:sqbroken}).  The \em{unpartitioned
    $P$-Hofer energy} of $u$ is defined as
  \[E_{P,\Hof}^*(u) = \sup_{\aleph \text{ is unpartitioned}} \int_C (\psi_\aleph \circ u)^*
    \om_{\XC_P}, \]
  where the squashing map $\aleph$ is said to be unpartitioned if for
  each $Q \subseteq P$, $\aleph|\Cone_{P^\dual}Q^\dual$ is
  unpartitioned (as in Definition \ref{def:sqmap}
  \eqref{part:unpart}).  The $P$-Hofer energy resp. unpartitioned
  $P$-Hofer energy of a holomorphic map $u:C \to \XB_P$ to a cut space
  $\XB_P$, also denoted by $E_{P,\Hof}(u)$ resp. $E_{P,\Hof}^*$, is
  analogously defined.
\end{definition}
The proof of the following monotonicity result is the same as the
proof in the neck-stretched case (Lemma \ref{lem:monot}).
\begin{lemma}{\rm(Monotonicity of Hofer energy for broken manifolds)}
  \label{lem:monotXX}
  Suppose $\JJ_0 \in \J^\cyl(\XX)$ is locally strongly tamed. Then
  there is a $C^0$-neighborhood $U_{\JJ_0} \subset \J^\cyl$ of $\JJ_0$
  from Lemma \ref{lem:dirinc} such that for any $\JJ \in U_{\JJ_0}$,
  $P \in \PP$ and any map
  \[\psi_\aleph : (\XX_P,\JJ) \to (\XX_P^\om,\om_{\XX_P})\]
  induced by a squashing map
  $\aleph : \Cone_{P^\dual} B^\dual \to \BB_{\tP}^\dual$, $\JJ|\XC_P$
  satisfies
  \[\psi_\aleph^*\om_{\XC_P}(v,\JJ v)\geq 0\]
  for all
  $v \in T\XC_P$.
\end{lemma}

Often, it is useful to have squashing area forms that are
non-degenerate on a given compact subset of the cut space $X_P$.

\begin{lemma}\label{lem:Knondeg}
  For any $P \in \PP$ and a compact subset $K \subset X_P$, there is an unpartitioned squashing map
  $\aleph: \NCone_{P^\dual} B^\dual \to \BB_P^\dual$ such that the squashed area form
  $\psi_\aleph^*\om_{X_P}$ is non-degenerate on $K$. 
\end{lemma}

\begin{proof}
  The requirement in the Lemma is satisfied by a squashing map
  $\aleph$ that is a dilation on
  $\pi_{\NCone_{P^\dual}B^\dual}(K)$. Such a map exists by taking the
  dilation constant $t$ (from \eqref{eq:aleph-dilate}) to be large
  enough.
\end{proof}

\begin{proposition}\label{prop:hofer-breaking}{\rm(Limit of maps and Hofer energy)}
  Let $\Om \subset \C$ be a compact set, and let
  $u_\nu : \Om \to X^\nu$ be such that there is a polytope $P \in \PP$
  and a sequence of translations $t_\nu \in \nu P^\dual$ such that
  \begin{equation}
    \label{eq:dtnuP0}
    d(t_\nu, \nu P_0^\dual) \to \infty, \quad \forall P_0 \supset P  
  \end{equation}
  and the sequence of translated maps $\e^{-t_\nu} u : \Om \to \XX_P$
  \footnote{From \eqref{eq:transinc},
    $e^{-t_\nu} : X^\nu_{\tP} \to \XX_P$ is an embedding of the
    $P$-cylindrical subset $X^\nu_{\tP} \subset X^\nu$.}  converges
  uniformly on compact subsets to a limit $u:\Om \to \XX_P$. Then,
  \[E_{P,\Hof}^*(u) \leq \lim_\nu E_\Hof(u_\nu).\]
  Here $E^*_{P,\Hof}$ is unpartitioned $P$-Hofer energy from
  Definition \ref{def:PHof}.
\end{proposition}

\begin{proof}
  Consider an unpartitioned squashing map
  $\aleph : \Cone_{P^\dual} B^\dual \to \BB_{\tP}^\dual$. To prove the
  result, we need to construct a sequence of squashing maps
  $\aleph_\nu : \nu B^\dual \to B^\dual$ for which
  \begin{equation}
    \label{eq:sqlimgoal}
    \lim_\nu u_\nu^*(\phi_{\aleph_\nu}^*\om_X)=u^*(\phi_\aleph^*\om_{\XX_P}).
  \end{equation}
  In particular, we will construct $\aleph_\nu$ that fits into a
  diagram (whose commutativity we will describe later)
  
 \[
    \begin{tikzcd}
      \Cone_{P^\dual} B^\dual\arrow[r,shift left, "\aleph"] 
      \arrow[d, dashed, shift left, "i_\nu^{-1}"] & \BB_{\tP}^\dual \arrow[d, "i"] \arrow[l, shift left,"\aleph_{\on{inv}}"] \\
      \nu B^\dual \supseteq \cup_{Q \subseteq P}\nu Q^\dual \arrow[r, "\exists \aleph_\nu"] \arrow[u, shift left, "i_\nu"] & B^\dual,
    \end{tikzcd}
  \]
  where
  \begin{itemize}
  \item $i_\nu:=\e^{-t_\nu}_P$ from \eqref{eq:cdet}, the inverse $i_\nu^{-1}$ is defined on the image of $i_\nu$,
  \item $\aleph_{\on{inv}}$ is the right inverse of the unpartitioned squashing map $\aleph$ (see Remark \ref{rem:unpart-inv}),
    \item and $i : \BB_\tP^\dual \hra B^\dual$ is a tubular neighborhood of $\BB_P^\dual$ in $B^\dual$. 
  \end{itemize}
  By the hypothesis \eqref{eq:dtnuP0} on $t_\nu$, the images of
  $i_\nu$ exhaust $\Cone_{P^\dual}B^\dual$.  We will define
  $\aleph_\nu$ so that there is a sequence of increasing open subsets
  $S_\nu \subset \on{image}(i_\nu)$ that exhaust
  $\Cone_{B^\dual} P^\dual$ and on which the diagram commutes, that
  is,
  \begin{equation}
    \label{eq:aleph-cond}
    \aleph=\aleph_\nu \circ i_\nu \quad \text{on $S_\nu$}.   
  \end{equation}
  This condition ensures that for any point $x \in \XX_P$ for which
  $\pi_{\Cone_{P^\dual} B^\dual}(x) \in S_\nu$, there is an equality
  of squashed forms
  \[(\phi_{\aleph_\nu}^*\om_X)|_{(\e^{-t_\nu})^{-1} x} = (\phi_\aleph ^* \om_{\XX_P})_x,  \]
and consequently \eqref{eq:sqlimgoal} holds.

 We assume that $B^\dual$ has a single top-dimensional
  polytope, and therefore $B^\dual$ itself may be viewed as a
  polytope. The general case is a natural extension and is left to the
  reader. 

  The squashing map $\aleph_\nu$ is constructed using a polyhedral
  decomposition of the target space $\nu B^\dual$ that possesses a
  dual complex, as in Figures \ref{fig:pqdual} and \ref{fig:parted}.
  By Lemma \ref{lem:dec-with-dual} and Lemma \ref{lem:dual2sq}
  \eqref{part:dual2sq1}, there exists a polyhedral decomposition $\QQ$
  of $B^\dual$ for which $\BB_\tP^\dual$ is a top-dimensional
  piece, \label{page:bbp-embed} and which has a dual complex
  $B^\dual_\QQ$ \footnote{The dual complex $B^\dual_\QQ$ is that of
    the polyhedral decomposition of $B^\dual$. The space $B^\dual$
    itself happens to be a dual complex, but in this discussion it
    plays the role of the target space of a squashing map.}.  Note
  that Lemma \ref{lem:dec-with-dual} proves the result for a
  decomposition of a vector space, and Lemma \ref{lem:dual2sq}
  \eqref{part:dual2sq1} extends the construction of a dual complex for
  the decomposition of a polytope. The squashing maps $\aleph_\nu$ are
  constructed by applying Lemma \ref{lem:dual2sq}
  \eqref{part:dual2sq2} to the polyhedral decomposition induced on
  $B^\dual$ by $\QQ$.  By Lemma \ref{lem:dual2sq}
  \eqref{part:dual2sq2}, \eqref{part:dual2sq3} there is a unique
  family of squashing maps $\{\aleph_\nu\}_\nu$ for which
 \[\aleph=\aleph_\nu \circ i_\nu^{-1} \quad \text{on $\aleph^{-1}(\BB_{\tP}^{\dual,\circ})$},\]
 and such that for any  top-dimensional polytope $P_0 \in \QQ$ other than $\BB_\tP^\dual$,
 \begin{equation}
   \label{eq:P0leaves}
   d(\aleph_\nu^{-1}(P_0^\circ), \aleph_\nu^{-1}(\BB_{\tP}^{\dual,\circ})) \to \infty,
 \end{equation}
 see \eqref{eq:QiQjdV}.  By \eqref{eq:P0leaves}, the commutativity
 \eqref{eq:aleph-cond} holds on
 \[S_\nu:=\cup_{P_1 \in \QQ, P_1 \subseteq
     \BB_{\tP}^\dual}\aleph_\nu^{-1}(P_1). \]
 Informally, the reason for the commutativity on $S_\nu$ is that the
 squashing map sends $S_\nu$ to $\BB_{\tP}^\dual$, and so
 $\aleph_\nu$ does not see the effect of other top-dimensional
 polytopes of $\QQ$ on $S_\nu$. Thus, the behavior of $\aleph_\nu$ is
 exactly like $\aleph$ with a domain translation given by $i_\nu$. See
 Figures \ref{fig:movepoly} and \ref{fig:movepoly1} for examples. This
 finishes the proof of Proposition \ref{prop:hofer-breaking}.
 \end{proof}

  \begin{figure}[ht]
    \centering \scalebox{.8}{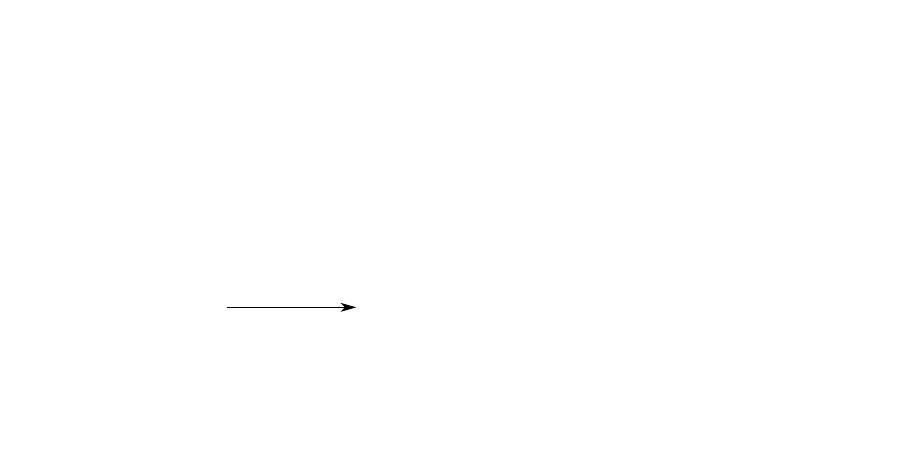}
    \caption{The squashing maps $\aleph_\nu : \nu B^\dual \to B^\dual$ (constructed in the proof of of Proposition \ref{prop:hofer-breaking}) converge to $\aleph:\Cone_{P_1^\dual}Q^\dual \to \BB_{P_1}^\dual$ as $\nu \to \infty$. Note that $\BB_{P_1}^\dual=\BB_{\tP_1}^\dual$ since $\codim(P_1)=0$.}
    \label{fig:movepoly}
  \end{figure}
   \begin{figure}[ht]
    \centering \scalebox{.8}{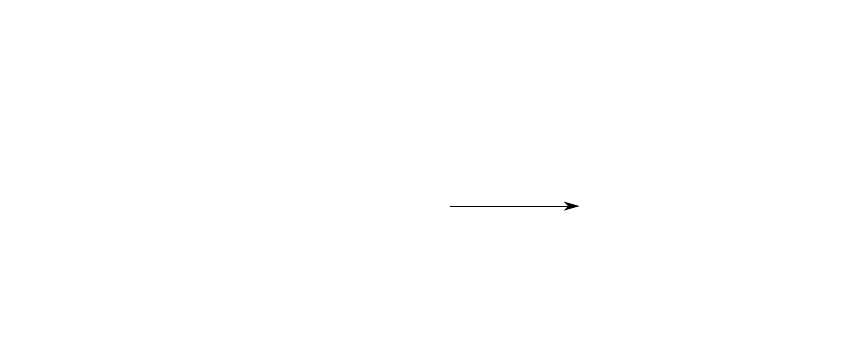}
    \caption{The squashing maps $\aleph_\nu : \nu B^\dual \to B^\dual$
      (constructed in the proof of of Proposition
      \ref{prop:hofer-breaking}) converge to
      $\aleph:\Cone_{P_4^\dual}Q^\dual \to \BB_{\tP_4}^\dual$ as
      $\nu \to \infty$, and $d_1, d_2, d_3 \to \infty$. Here $\PP$ and
      $B^\dual$ are from Figure \ref{fig:movepoly}.}
    \label{fig:movepoly1}
  \end{figure}

\begin{proposition}\label{prop:quothof}
  {\rm(Hofer energy and quotients)} Let $P \in \PP$ be a polytope with
  $\codim(P)>0$.  Let $U_{\JJ_0}$ be the $C^0$-neighborhood of
  cylindrical almost complex structures from Lemma \ref{lem:monotXX}
  for which Hofer energy is monotonic.  Let $u:C \to \XC_P$ be a
  holomorphic map with respect to a domain-dependent almost complex
  structure $J:C \to U_{\JJ_0}$.  Let $\pi_P : \XC_P \to \XB_P$ be the
  quotient under the action of $T_{P,\C}$. Then,
  \begin{enumerate}
  \item   $E_{P,\Hof}(\pi_P \circ u) \leq E_{P,\Hof}(u)$, and
  \item $E_{P,\Hof}^*(\pi_P \circ u) \leq E_{P,\Hof}^*(u)$
  \end{enumerate}
  where $E_{P,\Hof}$ denotes $P$-Hofer energy, and the superscript $*$
  refers to unpartitioned Hofer energy, see Definition \ref{def:PHof}.
\end{proposition}

\begin{proof}
  We recall that a squashing map for $\XC_P$ is a product map
  \[\aleph=(\aleph^P, \aleph^{P^\dual}): \Cone_{P^\dual} B^\dual \to
    \BB_{\tP}^\dual \]
  where $\aleph^P : \NCone_{P^\dual}B^\dual \to \BB_P^\dual$ is a
  squashing map for the cut space $\XB_P$, and
  $\aleph^{P^\dual}:\t_P \to P^\dual$ is a squashing map of
  polytopes. The $P$-Hofer energy is the supremum over squashed areas
  induced by $\aleph_P$, and Hofer energy is the supremum over
  squashed areas induced by $\aleph$.

  To prove the Proposition, consider a squashing map
  $\aleph^P : \NCone_{P^\dual}B^\dual \to \BB_P^\dual$.  Suppose 
  the projection
  $\pi_{\t_P} : \BB_{\tP}^\dual \to P^\dual \subset \t_P$ maps
  $\BB_P^\dual \subset \BB_{\tP}^\dual$ to the point $c_P \in \t_P$.
  For a map
  \begin{equation}
    \label{eq:alephP0}
    \aleph=(\aleph^P,c_P) : \Cone_{P^\dual} B^\dual \to
    \BB_{\tP}^\dual 
  \end{equation}
  and any map $u:C \to \XC_P$, we have
  \begin{equation}
    \label{eq:quotomP}
    (\pi_P \circ u)^*(\psi_{\aleph^P}^*\om_{X_P}) = u^*(\psi_{\aleph}^* \om_{\XX_P}).
  \end{equation}
  Adding a translation to the squashing map $\aleph$ alters the
  squashed area by a bounded multiplicative factor: Let $c_P'$ be a
  vertex of the polytope $P^\dual \subset \t_P$.  For the map
  \[\aleph'= (\aleph^P,c_P') : \Cone_{P^\dual}B^\dual \to
    \BB_\tP^\dual\]
  Lemma \ref{lem:tameC0} shows that there is a constant $c>0$ such
  that
  \[\psi_{\aleph'}^*\om_{\XX_P}(v,Jv) \geq c^{-1}
    \psi_\aleph^*\om_{\XX_P}(v,Jv)\]
  for any $J \in U_{\JJ_0}$ and $v \in T\XX_P$.

  Next, we construct a sequence of squashing maps that approximate
  $\aleph'$. (Note that $\aleph'$ itself is not a squashing map since
  it is not surjective.)  Define
  \[\tilde \aleph_\nu:=(\aleph^P, \aleph_\nu)\]
  where $\aleph_\nu :\t_P \to P^\dual$ is an unpartitioned squashing
  map whose right inverse $\aleph_{\nu,inv}: P^\dual \to \t_P$ is a
  translation such that the image $\aleph_{\nu,inv}(P^\dual)$ goes to
  $\infty$ in $\t_P$ in the direction $-c_P'$.  For example, we may
  take $\aleph_{\nu,inv}:=\aleph_{0,inv}- \nu c_P'$ for any $\nu$.
  Consequently, there exist increasing open sets
  $U_\nu \subset \Cone_{P^\dual}B^\dual$ that exhaust
  $\Cone_{P^\dual}B^\dual$ such that
  \[\aleph_\nu|U_\nu=(\aleph^P,c_P')=\aleph'.\]
  The Lemma now follows: For any squashing map $\aleph^P$ of $\XB_P$,
  define a sequence of squashing maps
  $\tilde \aleph_\nu:=(\aleph^P,\aleph_\nu)$ of $\XC_P$. We have
  \begin{multline*}
    E_\Hof(u) \geq \int_C(\psi_{\tilde \aleph_\nu} \circ
    u)^*\om_{\XX_P}
    \geq  \int_{C_\nu}(\psi_{\aleph'} \circ u)^*\om_{\XX_P} \\
    \geq c\int_{C_\nu}(\psi_{\aleph} \circ u)^*\om_{\XX_P} =
    \int_{C_\nu} (\psi_{\aleph_P} \circ \pi_P \circ u)^*\om_{X_P},
  \end{multline*}
  where $C_\nu:=(\pi_{\Cone_{P^\dual}B^\dual} \circ u)^{-1}(U_\nu)$.
  Since the sets $C_\nu$ exhaust the domain $C$, we conclude that for
  any squashing map $\aleph_P$ of $\XB_P$,
  \[\int_C(\psi_{\aleph_P} \circ \pi_P \circ u)^*\om_{X_P} \leq
    E_\Hof(u), \]
  and consequently $E_\Hof(\pi_P \circ u) \leq E_\Hof(u)$.
\end{proof}

\section{Removal of singularities}
\label{sec:remsing}
In this section, we prove the removal of singularities result,
Proposition \ref{prop:remsing}, for punctured holomorphic maps in a
piece of a broken manifold. \footnote{An approach to proving this
  result by embedding the broken almost complex manifold into a
  compact symplectic manifold fails, because the almost complex
  manifold $\XB_P$ 
  cannot 
  be embedded into the symplectic cut space
  $\ol X_P^\om$ via increasing maps.  For any vertex $Q \in \PP$ of
  $P$, an increasing map may not exist in the $Q$-cylindrical corner
  of $\XB_P$ because the fixed $\t$-inner product from \eqref{eq:idtt}
  is not equal to the natural $\t$-inner product at the
  $Q$-corner. The natural inner product is the one for which the edges
  $P_1 \subset P$ emanating from $Q$ form an orthogonal basis.
  Consequently, a Hofer energy bound in $\XB_P$ does not translate to
  a bound on symplectic area in $\ol X_P^\om$. } We prove that any
punctured holomorphic map with finite Hofer energy lies in some
$Q$-cylindrical region and the image of the projection to $\XB_Q$ lies
in a compact set in the complement of relative divisors.  The removal
of singularities result applies on the projected map, and consequently
on the original map.

Monotonicity for pseudoholomorphic maps is the main technical tool in
the proof of Proposition \ref{prop:remsing}. We state the monotonicity
result (see for example \cite[Proposition 3.12]{zinger:notes}).
\begin{proposition} \label{prop:monohol} {\rm(Monotonicity) } Let
  $(X,\om)$ be a compact symplectic manifold, $J_0$ an $\om$-tamed
  almost complex structure, and let
  \[U_{J_0}:=\{J \text{ is $\om$-tamed} : \Mod{J-J_0}_{C^0}<\eps\}\]
  for some $\eps>0$ be a $C^0$-neighborhood on the space of tamed
  almost complex structures.  There exist constants $c, r_0>0$ such
  that for any $x \in X$, \cwl{resp. $x \in L$} $0 < r \leq r_0$, a
  Riemann surface $C$ with boundary $\partial C$ \cwl{resp. with
    boundary $\partial C$, and possibly corners} and a
  pseudoholomorphic map $u:C \to X$ with respect to a domain-dependent
  almost complex structure $J:C \to U_\J$ whose image contains $x$ and
  $u(\partial C) \subset \partial B(x,r)$,
  \cwl{resp.
    $\partial B(x,r) \cup L$}
  \[\int_C u^*\om \geq cr^2. \]
\end{proposition}

We first prove the removal of singularities result (Proposition
\ref{prop:remsing}) in the case of a single cut to serve as a warm-up
for the more complicated proof in the case of multiple cuts.
 
\begin{proof}
  [Proof of Proposition \ref{prop:remsing} in the case of a single
  cut] We set up some notation first. We consider a single cut with
  polyhedral decomposition
  \[\PP = \{ P_+:=(-\infty,0],P_0:=\{ 0 \}, P_-:=[0, \infty) \},\]
  and dual complex $B^\dual \cong [\frac {-\delta} 2,\frac \delta 2]$.

  We first carry out the proof for maps whose target space is
  $\XB_{P_+}$, the case of $X_{P_-}$ being similar.  We recall from
  \eqref{eq:proj2cone}
  there is a projection to the $J$-complex
  \[\pi_{B_J} : \XB_{P_+} \to \Cone_{P_+^\dual} \simeq (-\infty,0].\]
  We choose any unpartitioned squashing map (as in Definition \ref{def:sqmap} \eqref{part:unpart})
  \[ \Cone_{P_+^\dual} \simeq (-\infty,0] \xrightarrow{\aleph} [\tfrac
    {-\delta} 2,\tfrac \delta 2] \simeq B^\dual. \]
  Such a map is a translation on an interval
  $[\tau,\tau+\delta] \subset \R_-$ and a locally constant map on the
  complement.

  \begin{definition}
    {\rm(Crossing)} A connected component $C \subset \Cyl$ of
    $(\pi_{B_J} \circ u)^{-1}([\tau,\tau+\delta])$ is called a \em{
      crossing} if $(\pi_{B_J} \circ u)(C)$ intersects both boundary
    components of $[\tau,\tau+\delta]$.
  \end{definition}

\vskip .05in \noindent   \textsc{Step 1}: We will establish a lower bound on the squashed
  area of any compact crossing $C$ by applying the monotonicity result
  Proposition \ref{prop:monohol}.  First, we recall that since any
  cylindrical almost complex structure $\JJ$ on $\XX$ is the limit of
  neck-stretched almost complex structures $(X^\nu,J^\nu)$, $\JJ$
  corresponds to an $\om_X$-tamed almost complex structure $J:=J^1$ on
  $(X,\om_X)$, and further, $\psi_{\aleph,*}\JJ=J$ on subsets of $\XX$
  where $\aleph$ is an isometry. We also note that $\psi_\aleph$ is an
  isometry on subsets of $\XX$ if $\aleph$ is an isometry on the
  correponding region in the $J$-complex. Since $\aleph$ is an
  isometry on the image of $\pi_{B_J} \circ u|C$, it is equivalent to
  work with
  \[\psi_\aleph \circ u : C \to (X,\om_X) \]
  instead of $u : C \to \XX$.  Choose
  $0<\delta_1<\min\{\frac \delta 2, r_0\}$ where $r_0$ is the constant
  from the monotonicity result applied to $(X,\om_X)$.  Let $z_0 \in C$
  be such that $\pi_{B_J}(u(z_0))=\tau+\frac \delta 2$. Applying the
  monotonicity result to the ball
  $B_{\delta_1}(\psi_\aleph \circ u(z_0)) \subset X$, we conclude
  there is a uniform constant such that for any compact crossing $C$,
  \begin{equation}
    \label{eq:crossing-single}
    \int_Cu^*(\psi_\aleph^*\om_X) \geq c.
  \end{equation}

\vskip .05in \noindent   \textsc{Step 2}: Next, we will show that the image $u(\Cyl)$ is
  either contained in a compact subset of $\XB_{P_+}$ or after
  truncating the domain cylinder by a finite amount, the image of $u$
  is contained in the $P_0$-cylindrical subset of $\XB_{P_+}$: After
  passing to a truncation of the domain cylinder, we may assume that
  the squashed area of the map is small enough to ensure that there
  are no compact crossings. If there is a non-compact crossing
  $C \subset \Cyl$, the image of $u$ is contained in a compact
  set. Indeed, there exists $\ell_0 \geq 0$ such that $C$ intersects
  the image $u(\{\ell\} \times S^1)$ for all $\ell \geq \ell_0$, and
  therefore the image $u(\Cyl(\ell_0))$ is contained within a radius
  $2\pi\Mod{du}_{L^\infty}$ of the compact subset
  $\{\tau \leq \pi_{B_J} \leq \tau +\delta\} \subset
  \XB_{P_+}$.  Finally, if $u$ does not have any crossings, the image
  of $u$ is either contained in the compact subset
  $\{\pi_{B_J} \geq \tau\} \subset \XB_{P_+}$; or it is contained in
  $\{\pi_{B_J} \leq \tau + \delta\}$ which is in the $P_0$-cylindrical
  subset of $\XB_{P_+}$.

\vskip .05in \noindent   \textsc{Step 3}: \label{page:1cutstep3} In case the image of $u$ is
  contained in a compact subset of $X_{P_+}$, the result follows from
  the removal of singularities result for compact symplectic
  manifolds. Indeed, for any compact subset $K$ of $X_{P_+}$, by Lemma
  \ref{lem:Knondeg}, there is a squashing area form $\om_\aleph$ that
  is a symplectic form on $K$, and
  $\int_\Cyl u^*\om_\aleph < E_\Hof(u)$.

\vskip .05in \noindent   \textsc{Step 4}: \label{page:1cutstep4} Next, we prove the result
  in the case when the image of $u$ is contained in the
  $P_0$-cylindrical end of $\XB_{P_+}$. The $P_0$-cylindrical end is a
  semi-infinite cylinder $Z_{P_0} \times (-\infty,0]$. The
  $(\om_{X_{P_0}} - \frac \eps 2 d\alpha_{P_0})$-area of the
  projection $u_{P_0}:=\pi_{P_0} \circ u$ is bounded by $E_\Hof(u)$
  because if we define the squashing map $\aleph_0$ so that
 \[ \aleph_0 \equiv \tfrac {-\eps} 2 \quad \text{on} \quad (-\infty,0],\]
 then
 \[\int_\Cyl u_{P_0}^*(\om_{X_{P_0}} - \tfrac \eps 2 d\alpha_{P_0})=\int_\Cyl u^*\om_{\aleph_0}  \leq E_\Hof(u).\]
 By the removal of singularities result for compact symplectic
 manifolds applied to the projected map
  \[u_{P_0} : \Cyl \to (X_{P_0}, \om_{X_{P_0}} - \tfrac \eps 2 d\alpha_{P_0}),\]
  \cwl{using the fact that the Lagrangian extends to a Lagrangian with
    clean self-intersection in the compactification} we conclude that
  $u_{P_0}$ extends holomorphically to
  \[ u_{P_0}:B_1 \to X_{P_0} . \]
  Consider a holomorphic trivialization of the pullback bundle
  $u_{P_0}^*(\oZ_{P_0} \times \R) \to B_1$.  The projection of $u$ to
  the fiber, denoted by
  \begin{equation}
    \label{eq:uv-punct}
    u_{\on{vert}}: B_1 \bs \{0\} \to S^1 \times \R ,   
  \end{equation}
  is holomorphic, and the $\R$-coordinate has an upper bound.
  Therefore, $u_{\on{vert}}$ extends over $0$ to a holomorphic map in
  $\P^1$.  Suppose $u_{\on{vert}}$ has a zero resp. pole of order
  $n \in \Z_{\geq 0}$ at $0 \in B_1$.  Then the twisted map
  $\ol u(z):=z^{-n}u(z)$ resp. $z^n u(z)$ has the same projection to
  $X_{P_0}$ as $u$, the vertical component $\ol u_{\on{vert}}$ has a
  removable singularity, and therefore, $\ol u$ has a removable
  singularity.  This proves \eqref{eq:remconc} in Proposition
  \ref{prop:remsing}.

  \textsc{Step 5}: So far we have proved the result in the cases when
  the target space is $\XB_{P_\pm}$. Next, consider the case of a map
  $u: \Cyl \to \XC_{P_0}$.  As in the previous paragraph, the
  singularity at $\infty$ can be removed for the projected map
  $\pi_{P_0} \circ u : \Cyl \to X_{P_0}$. To prove that the
  singularity can be removed for the vertical component $u_{\on{vert}}$ (as in
  \eqref{eq:uv-punct}) 
  it is enough to show that the $\R$-component
  has either an upper or a lower bound so that essential singularities
  are ruled out. This bound is a consequence of a lower bound on the
  squashed area $\psi^*\aleph$ for crossings. The details are exactly
  as in the case of $X_{P_\pm}$ and are therefore omitted.
\end{proof}

\begin{remark}\label{rem:integ-edge}
  {\rm(The integrality of edge {direction}s)}
  As part of the above proof, we have shown that if the punctured end has a {direction} $\mu \in \t_{P_0}$, then $\mu$ is integral, and not just fractional. This fact relies on the torus actions being free in the neighborhood of cut loci (see Definition \ref{def:tropmanifold} and Figure \ref{fig:orb-cut}). Indeed, in \eqref{eq:uv-punct} we use the fact that $Z_{P_0} \times \R \to X_{P_0}$ is a $S^1 \times \R$-bundle, and so the $S^1$ factor in \eqref{eq:uv-punct} is generated by an integral element in $\sqrt{-1}\t_{P_0}$. 
\end{remark}

\begin{proof}
  [Proof of Proposition \ref{prop:remsing} in the case of a multiple
  cut] The proof is by induction on $\dim(P)$.  By the induction
  hypothesis, we assume that the result holds for maps in the cut space $X_Q$ and the component $\XC_Q$ of the broken manifold for any $Q \in \PP$ with $\dim(Q) < \dim(P)$.
  
  We first consider the case when the target space is a cut space
  $\XB_P$, the case of a broken manifold $\XC_P$ is dealt with
  subsequently.  For notational simplicity, we assume that there is
  only one zero-dimensional polytope $R \in \PP$ contained in $P$, and
  so, $B^\dual=R^\dual$. We also denote $\t_R=\t$.  The generalization
  is easy and is described in Step 3
  of the proof. We 
  also assume that $\codim(P)=0$ so that $\t_P=\{0\}$. Otherwise, all
  instances of $\t$ are replaced by $\t/\t_P$.
To help the reader follow the proof, we give an example of an $\om$-complex of a cut space $X_P$ in Figure \ref{fig:bbp-rem}.
  
\begin{figure}[ht]
  \centering \scalebox{.8}{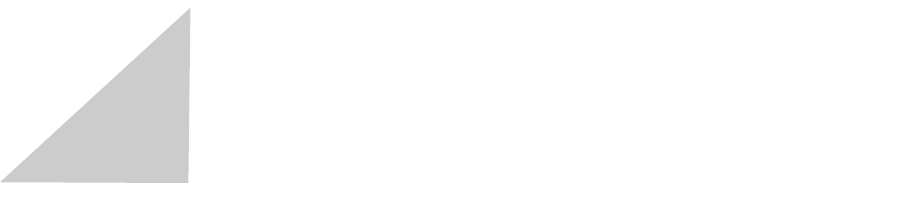}
  \caption{$\BB_P^\dual$ is the $\om$-complex of the cut space
    $\XB_P$. The faces of $\BB_P^\dual$ with dark grey labels are subcomplexes, and those with light grey labels are the $\om$-complexes of other cut spaces.}
  \label{fig:bbp-rem}
\end{figure}

The proof is similar to that of a single cut. The new feature is that
we now consider crossings defined with respect to various
one-dimensional subspaces of $\t$, which are defined as follows.
Let $P_1,\dots P_k \in \PP$ be facets of $P$. That is, for any
$\lam \in \{1,\dots,k\}$, $P_\lam \subset P$ and
$\dim(P_\lam)=\dim(P)-1$.  For any $\lam$, let $\pp_\lam \in \PP$ be
the the one-dimensional face of $P$ that is transverse to $P_\lam$,
that is, $\pp_\lam:=\cap_{1 \leq i \leq k, i \neq \lam}P_i$.  Let
\begin{equation}
  \label{eq:pilam}
  \pi_\lam : \t \to \R  
\end{equation}
be a linear projection that maps the codimension one dual polytope
$\pp_\lam^\dual \subset \t$ to a constant.  In particular, since
we assumed $P^\dual \in \t$ is the origin, we have
\begin{equation*}
  \pi_\lam(\pp_\lam^\dual)=0. 
\end{equation*}
Denote by
\[\pi_{B_J,\lam} :=\pi_{\lam} \circ \pi_{B_J}: \XB_P \to \R\]
the composition of $\pi_{B_J}: \XB_P \to \NCone_{P^\dual}R^\dual \subset \t$ with $\pi_\lam$ from \eqref{eq:pilam}.  The projection $\pi_{\lam}$ has the useful property that the $P_\lam$-cylindrical end in $\XB_P$ is
\begin{equation}
  \label{eq:lamend}
  U_{P_\lam}(\XB_P)=\{\pi_{B_J,\lam} >0\},
\end{equation}
and the complement $\XB_P \bs U_{P_\lam}(\XB_P)$ is equal to
$\{\pi_{B_J,\lam} =0\}$. 

We choose a point $\tau \gg 0$ in the interior of
$\pi_{\lam}(\NCone_{P^\dual}R^\dual) \subset \R$.

  \begin{definition}\label{def:ppcrossing}
    {\rm($\pp_\lam^\dual$-crossing)} Given $\delta>0$, a connected
    component $C \subset \Cyl$ of
    $(\pi_{B_J,\lam} \circ u)^{-1}([\tau - \delta, \tau + \delta])$ is
    a $(\pp_\lam^\dual,\delta)$-crossing if $\pi_{B_J,\lam}(u(C))$
    intersects both boundary components of
    $[\tau - \delta, \tau + \delta]$. If $\delta$ is clear from the
    context, we refer to $C$ as a $\pp_\lam^\dual$-crossing.
  \end{definition}

  We interrupt the proof of Proposition \ref{prop:remsing} (multiple
  cut case) to state and prove a technical Claim. The proof is
  continued
  after the proof of Claim \ref{claim:alephbd}.
  %
\end{proof}
  
The following Claim is the technical heart of the proof of the removal
of singularities result.

\begin{claim}\label{claim:alephbd}
  There are constants $\delta, c>0$ and a squashing map $\aleph$ such
  that for any compact $(\delta,\pp_\lam^\dual)$-crossing
  $C \subset \Cyl$, there is a lower bound
  \[\int_C(\psi_\aleph \circ u)^*\om_{X_P} \geq c.\]
\end{claim}
\begin{subproof}
  [Proof of Claim \ref{claim:alephbd}]
  We fix some constants in Step 1 of the proof. This is followed by an outline of the rest of the proof, with details given in Step 2 onwards. 
  \vskip .1in  \noindent \textsc{Step 1}:
  The first step is to fix the constant $\delta$ and the squashing map
  $\aleph$.  Cut up the $\om$-polytope $\BB_P^\dual \subset \t$
  along a level set of $\pi_\lam$ into two polytopes
  $\BB_{P,\pm}^\dual$ so that the facet
  $\BB^\dual_{P_\lam} \subset \BB_P^\dual$ is contained in
  $\BB_{P,-}^\dual$; see Figure \ref{fig:pp1crossing}.  Define
  \[\delta:=\hh \on{length}(\pi_{\lam}(\BB_{P,+}^\dual)).\]
  We will define $\aleph : \NCone_{P^\dual} R^\dual \to \BB_P^\dual$
  to be an unpartitioned undilated squashing map, and therefore
  $\aleph$ will possess a right inverse that is a translation
  \[\aleph_{\on{inv}} : \BB_P^\dual \to
    \ol{\aleph^{-1}(\BB_P^{\dual,\circ})}.  \]
  In fact, $\aleph$ is fully determined by the map $\aleph_{\on{inv}}$,
  which we define by the condition
  \begin{equation}
    \label{eq:alephinv}
    \aleph_{\on{inv}}(\BB^\dual_{P,P}):=\{\pi_{\lam} = \tau - \delta\} \cap \NCone_{P^\dual}P_\lam^\dual. 
  \end{equation}  
  We point out that both sides of the equation \eqref{eq:alephinv} are
  points.  The squashing map $\aleph$ defined in this manner has the
  property that for any facet $\NCone_{P^\dual}Q^\dual$ of
  $\NCone_{P^\dual} R^\dual$ that contains
  $\NCone_{P^\dual}P_\lam^\dual$ (that is, $Q \in \PP$,
  $Q \subseteq P_\lam$ and $\dim(Q)=1$),
  \begin{equation}
    \label{eq:bb-liesin}
    \aleph^{-1}(\BB_{P,Q}^\dual) \subset \NCone_{P^\dual}Q^\dual.   
  \end{equation}

  \begin{figure}[ht]
    \centering \scalebox{.8}{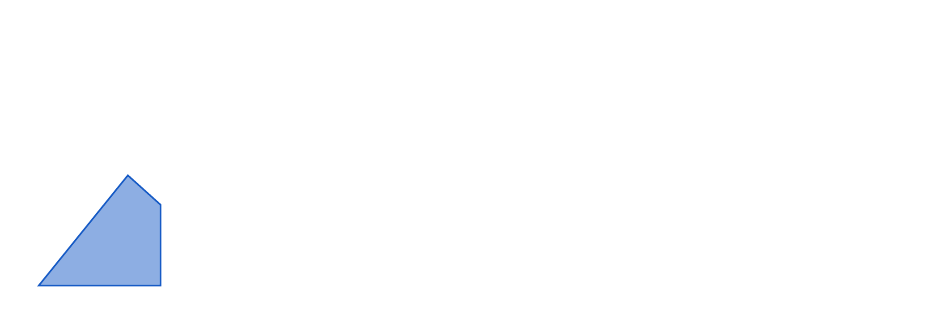}
    \caption{$\pp_\lam^\dual$-crossing with $\lam=1$.}
    \label{fig:pp1crossing}
  \end{figure}

  \vskip .1in \noindent \textsc{Outline of the proof of the Claim}: 
  Compared to
  the single cut, the most significant difference in the multiple cut
  case is that the image of a crossing may not be contained in a
  subset of $\XC_P$ where $\psi_\aleph$ is an isometry. This
  corresponds to the fact that $\aleph$ is not an isometry on
  $\pi_{\lam}^{-1}([\tau-\delta,\tau+\delta])$.
  The way around this issue is that if the image $u(C)$ is contained
  in a subset of $\XB_P$ where the map $\psi_\aleph$ is a projection
  whose fibers are $T_{Q_0,\C}$-orbits for some $Q_0 \in \PP$ for
  which $Q_0 \subseteq P$, $Q_0 \nsubseteq P_\lam$, Then we apply
  monotonicity result not on $u$ (as in the single cut proof) but on
  $u$ composed with a quotient by the $T_{Q_0,\C}$-action.  Observe
  that the fibers of $\psi_\aleph$ are $T_{Q_0,\C}$-orbits exactly if
  the map $\aleph$ is \em{$\t_{Q_0}$-squashing}, that is, it is a
  projection map whose fibers are $\dim(\t_{Q_0})$-dimensional and
  parallel to $\t_{Q_0}$.  We note that $\aleph$ is chosen so
  that in the subset $\pi_{\lam}^{-1}([\tau-\delta,\tau+\delta])$, the
  $\t_{P_\lam}$ direction itself is not squashed (that is,
  $Q_0 \nsubseteq P_\lam$); this feature is crucial in getting the
  lower bound on area.  The details of the rest of the proof
  of Claim \ref{claim:alephbd}
  are as follows.

  \vskip .1in  \noindent \textsc{Step 2}: \textit{ Determining the polytope
    $Q_0 \in \PP$ (from the proof outline) via a decomposition of the
    $J$-polytope.}
  The subset of the $J$-complex corresponding to the
  $\pp_\lam^\dual$-crossing has a decomposition
  \begin{equation}
    \label{eq:jdec}
    \pi_{\lam}^{-1}(\tau - \delta, \tau + \delta) \subset \bigcup_{Q \in \PP, Q \nsubseteq P_\lam, Q \subseteq P} S_Q,  
  \end{equation}
  where $S_Q \subset \NCone_{P^\dual} R^\dual$ is the subset on which
  $\aleph$ is \em{$T_Q$-squashing}, that is, $\aleph$ projects to the
  direction perpendicular to $\t_Q$. See Figure \ref{fig:sqdec}. The
  decomposition \eqref{eq:jdec} exists because \eqref{eq:bb-liesin}
  rules out any other kind of squashing in $\pi_{\lam}^{-1}(\tau - \delta, \tau + \delta)$.

  \begin{figure}[h]
    \centering \scalebox{.8}{
\begingroup%
  \makeatletter%
  \providecommand\color[2][]{%
    \errmessage{(Inkscape) Color is used for the text in Inkscape, but the package 'color.sty' is not loaded}%
    \renewcommand\color[2][]{}%
  }%
  \providecommand\transparent[1]{%
    \errmessage{(Inkscape) Transparency is used (non-zero) for the text in Inkscape, but the package 'transparent.sty' is not loaded}%
    \renewcommand\transparent[1]{}%
  }%
  \providecommand\rotatebox[2]{#2}%
  \newcommand*\fsize{\dimexpr\f@size pt\relax}%
  \newcommand*\lineheight[1]{\fontsize{\fsize}{#1\fsize}\selectfont}%
  \ifx\svgwidth\undefined%
    \setlength{\unitlength}{212.48620461bp}%
    \ifx\svgscale\undefined%
      \relax%
    \else%
      \setlength{\unitlength}{\unitlength * \real{\svgscale}}%
    \fi%
  \else%
    \setlength{\unitlength}{\svgwidth}%
  \fi%
  \global\let\svgwidth\undefined%
  \global\let\svgscale\undefined%
  \makeatother%
  \begin{picture}(1,0.7203025)%
    \lineheight{1}%
    \setlength\tabcolsep{0pt}%
    \put(0,0){\includegraphics[width=\unitlength,page=1]{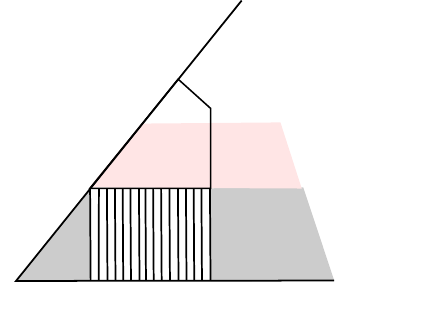}}%
    \put(0.62874542,0.54341115){\color[rgb]{0,0,0}\makebox(0,0)[lt]{\lineheight{1.25}\smash{\begin{tabular}[t]{l}$\Cone_{P^\dual}R^\dual$\end{tabular}}}}%
    \put(0.10024834,0.31464919){\color[rgb]{0,0,0}\makebox(0,0)[lt]{\lineheight{1.25}\smash{\begin{tabular}[t]{l}$S_P$\end{tabular}}}}%
    \put(0.39080295,0.67072288){\color[rgb]{0,0,0}\makebox(0,0)[lt]{\lineheight{1.25}\smash{\begin{tabular}[t]{l}$\t_{P_1}$\end{tabular}}}}%
    \put(0.6639375,0.01153752){\color[rgb]{0,0,0}\makebox(0,0)[lt]{\lineheight{1.25}\smash{\begin{tabular}[t]{l}$\t_{P_2}$\end{tabular}}}}%
    \put(0.73054639,0.43283258){\color[rgb]{0,0,0}\makebox(0,0)[lt]{\lineheight{1.25}\smash{\begin{tabular}[t]{l}$S_{P_2}$\end{tabular}}}}%
    \put(0.69996287,0.35191453){\color[rgb]{0,0,0}\makebox(0,0)[lt]{\lineheight{1.25}\smash{\begin{tabular}[t]{l}$\pi_1^{-1}([\tau-\delta,\tau+\delta])$\end{tabular}}}}%
    \put(0,0){\includegraphics[width=\unitlength,page=2]{sqdec.pdf}}%
    \put(-0.00128047,0.02923503){\color[rgb]{0,0,0}\makebox(0,0)[lt]{\lineheight{1.25}\smash{\begin{tabular}[t]{l}$\t_P$\end{tabular}}}}%
  \end{picture}%
\endgroup%
}
    \caption{The $\pp_1^\dual$-crossing decomposes into $S_P$ and $S_{P_2}$.}
    \label{fig:sqdec}
  \end{figure}
  
  The monotonicity result will be applied on small balls in the target
  space whose radius $\delta_1$ we determine next.  Let $N:=\dim(P)$.
  Choose a constant $0<\delta_1<\delta/(N+1)$ which is small enough
  that for any $x \in \pi_{\lam}^{-1}([\tau-N\delta_1,\tau+N\delta_1]) \subset
  \NCone_{P^\dual}R^\dual$ and any polytope $Q \subseteq P$,
  \begin{equation}\label{eq:bsq}
    B_{\delta_1}(x) \cap \NCone_{P^\dual}Q^\dual \neq \emptyset \Longrightarrow B_{\delta_1}(x) \subset \bigcup_{Q_1 \in \PP, Q_1 \supseteq Q}S_{Q_1}.
  \end{equation}

  Consider a compact $\pp_\lam^\dual$-crossing $C \subset \Cyl$.  We
  call $(x\in X_P, Q \in \PP)$ an \em {admissible pair} if
  \begin{itemize}
  \item $x \in \pi_{B_J}(u(C))$ and $\pi_\lam(x) \in [\tau- n_Q \delta, \tau+ n_Q \delta]$, where $n_Q:=\dim(Q)$; and
  \item $ Q_x^\dual=\cap_{Q \in \QQ}Q^\dual$, where 
\[\QQ:=\{Q \in \PP : Q \subseteq P, Q \nsubseteq P_\lam, B_{\delta_1}(x) \cap \NCone_{Q^\dual}R^\dual \neq \emptyset\}. \] 
  \end{itemize}
  Let
  \begin{equation}
    \label{eq:x0Q0}
    (x_0,Q_0)  
  \end{equation}
  be an admissible pair such that there is no admissible pair
  $(x,Q_x)$ such that $Q_0 \subset Q_x$.  In other words, we choose an
  admissible pair $(x_0,Q_0)$ for which the torus $T_{Q_0,\C}$ is
  minimal.  See Figure \ref{fig:sqcurve} for an example.
  
  \vskip .1in \noindent 
  \textsc{Step 3}: \textit{ Applying the monotonicity result and obtaining the lower bound on squashed area.}\\
  We will apply the monotonicity result (Proposition
  \ref{prop:monohol}) to the map $u$ quotiented by $T_{Q_0,\C}$.  Let
  \[u_{Q_0}: C \dashrightarrow X_{Q_0}, \quad u_{Q_0}:=u/T_{Q_0,\C}\]
  be a map defined on the subset of $C$ that is mapped by $u$ to the
  $Q_0$-cylindrical end of $X_P$.  Fix any $z_0 \in (\pi_{B_J,\lam} \circ u)^{-1}(x_0)$, and let
  \[C_1:=u_{Q_0}^{-1}(B_{\delta_1}(u_{Q_0}(z_0))).\]
  We will now prove that
  \begin{equation}
    \label{eq:uq0-bdry}
    u_{Q_0}(\partial C_1) \subset \partial B_{\delta_1}(u_{Q_0}(z_0)), 
  \end{equation}
  which is a hypothesis of the monotonicity result. If
  \eqref{eq:uq0-bdry} did not hold, there would be a point
  $z_1 \in \partial C_1$ such that $u(z_1)$ does not map to the
  $Q_0$-cylindrical end of $X_P$ (assuming $U_Q(X_P)$ is an open
  subset of $X_P$), and $\pi_{B,J}(u(z_1)) \in \Cone_{P^\dual} Q_1^\dual$ for some
  $Q_1 \in \PP$, $Q_1 \supset Q$. Then, setting
  $x_1:=\pi_{B_J}(u(z_1))$, $(x_1,Q_{x_1})$ is an admissible pair for
  some $Q_{x_1} \in \PP$ with $Q_{x_1} \subseteq Q_0$, where we use
  the observation that since $|\pi_\lam(x_0) - \tau|<n_{Q_0}\delta_1$,
  we have $|\pi_\lam(x_1) - \tau|<(n_{Q_0}+1)\delta_1 \leq n_{Q_{x_1}}
  \delta_1$. We have thus contradicted the minimality of $T_{Q_0,\C}$
  from \eqref{eq:x0Q0}, and shown that \eqref{eq:uq0-bdry} indeed
  holds.

  \begin{figure}[h]
    \centering \scalebox{.8}{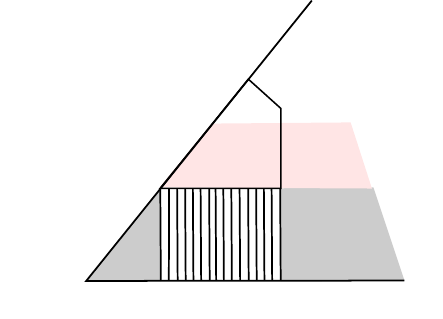}
    \caption{Assume the setting of Figure \ref{fig:sqdec}.  For the map $u$, $(x_1,P)$, $(x_2,P_2)$ are admissible pairs, with $(x_1, P)$ being the pair chosen in \eqref{eq:x0Q0}.}
    \label{fig:sqcurve}
  \end{figure}
  
  To finish the proof of the result, we relate the squashed areas of
  $u$ and its projection $u_{Q_0}$, via the following commutative
  diagram
  \[
    \begin{tikzcd}
      U_{Q_0}(\XB_P) \arrow{r}{\psi_\aleph} \arrow[swap]{d}{/T_{Q_0,\C}} & X^\om_P \arrow[dashed]{d}{/T_{Q_0,\C}}\\
      \XB_{Q_0} \arrow[dashed]{r}{\psi_{\aleph,\qu}} & X^\om_{Q_0}.
    \end{tikzcd}
  \]
  The maps denoted by dashed arrows are defined on subsets of the
  domain.  However these subsets contain the images of
  $\psi_\aleph \circ u$ and $u/T_{Q_0,\C}$. The left downward arrow is
  a map between almost complex manifolds. The right downward arrow is
  a map between symplectic manifolds, and is defined by an orthogonal
  projection map $\BB_{P,+}^\dual \to \BB_{P,Q_0}$. The map
  $\psi_{\aleph,\qu}$ is induced by a squashing
  $\aleph_{\qu} : \NCone_{Q_0^\dual}R^\dual \to \BB_{Q_0}^\dual$ which
  is defined on a subset of the domain as follows: The restriction
  $\aleph|\aleph^{-1}(\BB^\dual_{P,+} \cap \BB_{Q_0}^\dual)$
  orthogonally projects to $\BB_{Q_0}^\dual$.  Therefore, it descends
  to the $\t_Q$-quotient of the domain, and yields the map
  $\aleph_{\qu}$.  The map $\psi_{\aleph,\qu}$ is an isometry on the
  image of $(u/T_{Q_0,\C})|C_1$.  Applying the monotonicity theorem on
  the map $\psi_{\aleph,\qu}(u/T_{Q_0,\C})|C_1$, we conclude
  \[\int_{C_1}(\psi_{\aleph,\qu}(u/T_{Q_0,\C}))^*\om_{X_{Q_0}} > c_{Q_0}
    \delta_1^2 .\]
  Finally, there is a constant $c(\XX)>0$ such that
  \[\int_{C_1}(\psi_\aleph \circ u)^*\om_{X_P} \geq
    c\int_{C_1}(\psi_{\aleph,\qu}(u/T_{Q_0,\C}))^*\om_{X_{Q_0}} > c c_{Q_0}
    \delta_1^2. \]
  This leads to the lower bound on $\aleph$-squashed area of
  crossings, and finishes the proof of Claim \ref{claim:alephbd}.
\end{subproof}

\begin{proof}[Proof of Proposition \ref{prop:remsing} for multiple cuts, continued]
  \label{page:remsingcont}

  We recall that $P_1,\dots, P_k\in \PP$ are facets of $P$. 

\vskip .1in \noindent   {\sc Step 1}: \em{There is a constant $\ell_0$ such that for any
    facet $P_\lam \in \PP$ of $P$, one of the two possibilities
    occurs:
\begin{itemize}
\item Either the image $u(\Cyl(\ell_0))$ is contained in the
  $P_\lam$-cylindrical end $U_{P_\lam}(\XB_P)$, or
\item the projection $\pi_{B_J,\lam}(u(\Cyl(\ell_0)))$ has an
  upper bound.
\end{itemize}}
For a given $\lam$, we first consider the case that $u$ has a
non-compact crossing $C$. We may choose $\ell_\lam$ so that
$\{s\} \times S^1$ intersects $C$ for all $s>\ell_\lam$. Therefore,
$\pi_{B_J,P_\lam} \circ u|\Cyl(\ell_\lam)$ lies in a
$2\pi \Mod{du}_{L^\infty}$-radius of
$\pi_{B_J,P_\lam}^{-1}([\tau-\delta_\lam,\tau+\delta_\lam])$, and
therefore by \eqref{eq:lamend} the image of $u|\Cyl(\ell_\lam)$ lies
in the $P_\lam$-cylindrical end of $\XB_P$. Next, consider the case
that $u$ does not have any non-compact crossings. Claim
\ref{claim:alephbd} implies that there is a constant $\ell_\lam\geq 0$
such that $u|\Cyl(\ell_\lam)$ does not have any compact crossings.
Then, $\pi_{B_J,\lam}(u(\Cyl(\ell_\lam)))$ either has a lower bound of
$\tau_\lam - \delta_\lam$ or an upper bound of
$\tau_\lam + \delta_\lam$. In the former case, $u(\Cyl(\ell_\lam))$
lies in the $P_\lam$-cylindrical end.  Finally, by taking
$\ell_0:=\min_\lam \ell_\lam$, the condition stated in the beginning
of the paragraph is satisfied.

\vskip .1in \noindent 
{\sc Step 2:} We use induction to finish the proof of the result for a
map in a cut space $X_P$, and for which there is only one
zero-dimensional polytope $R \in \PP$ in $P$.  First consider the case
that for all facets $P_\lam$, the projection
$\pi_{B_J,\lam}(u(\Cyl(\ell_0)))$ has an upper bound. Then, the image
of $u$ is contained in a compact subset of $X_P$, and the proof of
removal of singularity follows exactly as in Step 3 of the proof in
the case of a single cut.
The
other possibility is that for some $\lam$, the image of
$u|\Cyl(\ell_0)$ is contained in the $P_\lam$-cylindrical end. Then,
we may view $u$ as a map in $\XC_{P_\lam}$ and apply the induction
hypothesis to the map $\pi_{P_\lam} \circ u : \Cyl \to X_{P_\lam}$.
For the proof of the induction hypothesis, it is enough to have a
bound on squashed area corresponding to squashing maps on
$\NCone_{P_\lam^\dual}R^\dual$ which are restrictions of squashing
maps $\NCone_{P^\dual}R^\dual$ with $\NCone_{P_\lam^\dual}R^\dual$
being embedded in $\NCone_{P^\dual}R^\dual$ far enough from the faces
corresponding $Q^\dual \subset P_\lam^\dual$; such squashed areas are
bounded by $E_{P,\Hof}(u)$.

\vskip .1in \noindent {\sc Step 3:}
\label{page:manyR}
If the target of $u$ is a cut space $X_P$ and there are multiple
zero-dimensional polytopes $R \in \PP$ that are vertices of $P$, the
only difference is in the definition of the projection map $\pi_\lam$
from \eqref{eq:pilam}.  First, observe that in the case of a single
zero dimensional polytope $R$, the map $\pi_\lam$ is actually defined
on $R^\dual$ and $\Cone_{P^\dual}R^\dual$, both of which were viewed
as subsets of $\t$. In the general case, we define
$\pi_{\lam,R}: \t_R \to \R$ for each $R$ so that
$\pi_{\lam,R}(\pp_{\lam,R})=0$ (where
$\pp_{\lam,R}^\dual \subset R^\dual$ is the face corresponding to
$\pp_\lam^\dual$) and $\pi_{\lam,R}|P^\dual_{\lam,R}$ is the same for
all $R \in P_\lam$. The notion of a $\pp_\lam^\dual$-crossing can now be
defined, and the rest of the proof carries over.

\vskip .1in \noindent {\sc Step 4:} Next, we consider the case that the target space of $u$
is a broken manifold $\XC_P$. We may assume that the projection
$u/T_{P,\C}$ lies in a compact subset of $X_P$; otherwise by applying
the preceding discussion to $u/T_{P,\C}$, we may conclude that its
image lies on a cylindrical end of $X_P$ and hence the induction
hypothesis is applicable. As in the previous paragraph, we conclude
that $u/T_{P,\C}$ has a removable singularity, and we assume that
$(u/T_{P,\C})(\infty)=x_P \in X_P$. Analogously to Step 4 of the proof
in the case of a single cut,
we
prove the result in $\XC_P$ by considering a holomorphic
trivialization of the $T_{P,\C}$-bundle
\((u/T_{P,\C})^*\XC_P \to B_1 \)
and projecting $u$ to the fiber direction in a neighborhood of the
puncture at $0 \in B_1$. To finish the proof, it only remains to show
that for any one-dimensional face $Q^\dual$ of $P^\dual$ (where
$Q \in \PP$), the projection of
\[\pi_{P,J}(u(\Cyl)) \subset \Cone_{Q^\dual}P^\dual \simeq
  \NCone_{Q^\dual}P^\dual \times \t_Q\]
to $\t_Q^\dual \simeq \R$ has a one-sided bound. This fact follows
from a lower bound on squashed area of $\t_Q$-crossings, which is
proved in a similar manner to Claim \ref{claim:alephbd}, and whose details 
are left to the reader. This finishes the
proof of Proposition \ref{prop:remsing}.
\end{proof}

The following is a version of the removal of singularities result for
punctures on the boundary. This result does not involve any technical
difficulties arising from neck-stretching because the Lagrangian
submanifold does not intersect the neck regions.
\begin{proposition}{\rm(Removal of singularities on the boundary)}
  \label{prop:remsing-disk}
  Suppose
  \[u:(\R_+ \times [0,1],\R_+ \times \{0,1\}) \to (X_{P_0},L)\]
  is a map that is holomorphic with respect to a domain-dependent
  cylindrical almost complex structure taking values in $U_{\JJ_0}$
  (from Lemma \ref{lem:monot}), and
  \begin{equation*}
    E_{\Hof}(u) <\infty, \quad   \Mod{\d u}_{L^\infty(\Cyl)} <\infty. 
  \end{equation*}
  Then, $u$ extends to a holomorphic map
  \[u: (B_1 \cap \H, B_1 \cap \R) \to (\ol X_{P_0},L).\]
  \cwl{resp. $(\ol \XC_P,L_{\tP})$}
\end{proposition}
\begin{proof}
  By the uniform bound on the derivative, the image of $u$ is
  contained in a compact neighborhood $K \subset X_P$ of the
  Lagrangian $L$.  By Lemma \ref{lem:Knondeg}, there is a squashed
  area form $\om_\aleph \in \Om^2(\XB_P)$ that is symplectic on the
  closure of $\on{im}(u)$ and that tames the almost complex
  structure. Therefore, the statement of the Proposition follows from the removal of
  singularity theorem for compact symplectic manifolds.
  \cwl{resp. with boundary on cleanly-intersecting Lagrangians.}
\end{proof}

\section{Hofer energy for Gromov compactness of broken maps}

The results stated so far in this Chapter will be used in the proof
that a sequence of breaking maps with bounded area converges to a
broken map. In this section, we discuss analogs of the results which
will be used to show that a sequence of broken maps with bounded area
converges to a broken map.  In components $\XX_P$ of the broken
manifold, we use $P$-Hofer energy of maps.

The first result, which is an analog of Remark \ref{rem:hof-area} for
broken maps, is that for a component of a broken map in $\XX_P$, the
$P$-Hofer energy is bounded by a constant that depends only on the
topology of the map.  In particular, we obtain a bound on Hofer energy
for components of broken maps $u : C^\circ \to \XX_P$ whose projection
to $X_P$ is in class $\beta \in \pi_2(\ol X_P, L)$, and whose tropical
graph is $\cT$. Note that both $\beta$ and $\cT$ are part of the
combinatorial type (Definition \ref{def:type-broken}) of the map $u$,
and that $\cT$ determines the number of punctures in the domain of
$u$, and the edge {direction} $\cT(z_i) \in \t_\Z$ of the map at each of the
punctures $z_i$.

\begin{proposition}\label{prop:hofertop} For any polytope $P \in \PP$,  a homotopy class $\beta \in \pi_2(\ol X_P, L)$ and a tropical graph $\cT$ underlying
  a relative map, there is a constant $c$ such that for a holomorphic map $u:C^\circ \to \XX_P$ 
  with tropical graph $\Gamma$, and whose projection $\pi_P \circ u$ is of class $\beta$,
  \[E_{P,\on{Hof}}(u) \leq c. \]
\end{proposition}
\begin{proof}
  We first prove the result assuming that the $X$-inner product $g$ (see \eqref{eq:idtt})  
  is rational, so that the compactification $\ol \XX_P$  of $\XX_P$
  is an orbifold with symplectic form $\om_{\XX_P}$.  Furthermore, 
  \[E_{P,\Hof}(u) = \int_C u^*\psi_\aleph^*\om_{\XX_P}
    =:\om_{\XX_P}(u),\]
  where $\aleph$ is any squashing form.  The above quantity is   independent of which $\aleph$ is used, because for any two choices
  $\aleph_1$, $\aleph_2$, the maps $\psi_{\aleph_1}$ and
  $\psi_{\aleph_2}$ are isotopic.
  
  Under the splitting of $\t = \t_P \oplus \t_P^\perp$, let
  $\cT_g^\ver(z_i) \in \t_P$ be the $\t_P$-component of $\cT(z_i) \in \t$.
  Recall that the fiber of $\pi_P : \ol \XX_P \to \ol X_P$ is
  \[P^\dual=\{x \in \t_P: \bran{\mu_i,x} \leq c_i, i=1,\dots,N\},\]
  where for any $i=1,\dots, N$, $c_i \in \R$ and $\mu_i \in \t_{P,\Z}^\dual$ is a 
  primitive outward pointing normal to a facet of $P^\dual$. 
  Suppose
  the base symplectic form $\om_{X_P}$ on $\ol X_P$ is given by reduction at the
  origin $0 \in \t_P$. Then,
  \begin{multline}
    \label{eq:fibarea}
    \bran{u_*[C], \om_{\XC_P}}=\bran{(\pi_P \circ u)_*[C], \om_{X_P}} + 2\pi \ssum_{i=1}^{d(\black)}  \ssum_j c_j \cdot m(z_i, D_j)\\
    \leq \bran{(\pi_P \circ u)_*[C], \om_{X_P}} + \sum_i |\cT_g^{\ver}(z_i)|\\
    \leq \bran{(\pi_P \circ u)_*[C], \om_{X_P}} + c_g\sum_i |\cT(z_i)|
  \end{multline}
  where $m(z_i,D_j)$ is the intersection multiplicity at $z_i$ with
  the toric divisor $D_j=\{\bran{\mu_j,x} = c_j\}$,
  $|\cT^{\ver}(z_i)|$ is the sum of the components of $\cT^\ver(z_i)$
  for a fixed $\Z$-basis. Indeed, \eqref{eq:fibarea} is a consequence of the
  cohomological relation
  \[[\om_{\XX_P}]=[\pi_P^*\om_{X_P}] + \sum_j c_j [D_j]^\dual, \]
where $ [D_j]^\dual \in H^2(\XX_P, \Z)$ is the Poincar\'e dual of $[D_j]$.  
  Since the origin $0 \in \t_P$ is in
  the interior of $P^\dual$, the constant $c_j$ is positive for all
  $j$.  The statement of the Proposition follows in the case when the
  $X$-inner product is rational.

  For a general $X$-inner product $g$, we consider a sequence
  $\{g_\nu\}_\nu$ of rational $X$-inner products uniformly converging
  to $g$.  A squashing map with respect to the $X$-inner product $g$
  is the limit of squashing maps $\aleph_\nu$ with respect to the
  $X$-inner product $g_\nu$, and therefore,
  $E_{P,\Hof,g} \leq \limsup_\nu E_{P,\Hof,g_\nu}$.  For any $\nu$,
  the constant $c_{g_\nu}$ is uniformly bounded, since it is chosen so
  that it satisfies
  $|\cT_{g_\nu}^{\ver}(z_i)| \leq c_{g_\nu}|\cT(z_i)|$, and 
  the statement of the Proposition follows.
\end{proof}

The following is an analog of Proposition \ref{prop:hofer-breaking}
for the convergence of a sequence of broken map components in $\XX_P$.

\begin{proposition} \label{prop:hofbv}
  {\rm(Hofer energy and limits of maps, broken version)}\label{prop:Hbr-limit}
  Let $\Om \subset \C$ be an open domain, and let $u_\nu : \Om \to \XX_P$ be such that there is a
  polytope $Q \subset P$ and a 
  sequence of translations $t_\nu \in \Cone_{P^\dual} Q^\dual$ such that
  \begin{equation}
    \label{eq:dtnuP0b}
    d(t_\nu, \nu P_0^\dual) \to \infty, \quad \forall P \supseteq P_0 \supset Q,
  \end{equation}
  and the sequence of translated maps $e^{-t_\nu} u : \Om \to \XX_Q$
  converges uniformly on compact subsets to a limit $u:\Om \to
  \XX_Q$. Then,
  \[E_{Q,\Hof}^*(u) \leq \lim_\nu E_{P,\Hof}(u_\nu).\]
\end{proposition}
The proof is the same as that of Proposition
\ref{prop:hofer-breaking}.

\chapter{Gromov compactness}\label{chap:cpt}

We prove two convergence results for maps in broken manifolds in this
Chapter.  The first result, Theorem \ref{thm:cpt-breaking}, concerns
the limit of holomorphic maps to the neck-stretched manifolds $X^\nu$
in the limit $\nu \to \infty$.  The second result, Theorem
\ref{thm:cpt-broken}, concerns the convergence behavior of broken
maps.  Both theorems require only an area bound on the sequence of
holomorphic maps.  In the breaking case, area of a map in $X^\nu$
refers to $\om_X$-area \eqref{eq:area-u}.  In the case of broken maps
in Theorem \ref{thm:cpt-broken}, area refers to the sum of the
$\om_{X_P}$-areas \eqref{eq:area-uv} of the projections of the map
components to cut spaces $X_P$. \label{rep:area1}

We sketch the definition of Gromov convergence, with a more precise
and detailed definition postponed to Section \ref{sec:convdef}.  As in
symplectic field theory, convergence in neck-stretched manifolds is
modulo target rescalings. Rescalings are analogous to
\em{translations} defined in Section \ref{sec:cylbrokenmfd}.  Given
$\nu \in \R_+$, the set of translations is parametrized by the dual
complex $\nu B^\dual$.  A translation
$t \in \nu P^\dual \subset B^\dual$ is an embedding
\[\e^{-t} : X^\nu_{\tP} \to \XC_P \subset \XX\]
of the $P$-cylindrical subset $ X^\nu_{\tP}$ of the neck-stretched
manifold $X^\nu$.  Gromov convergence of a sequence of maps
$u_\nu : C_\nu \to X^\nu$ on neck-stretched manifolds to a broken map
$u: C \to \XX$ encapsulates the following
conditions: \label{conv-sketch}
\begin{enumerate}
\item {\rm(Convergence of domains)}\label{item:dconv1} The sequence of treed disks $C_\nu$
  converges to a treed disk $C$;
\item \label{item:mconv1}
{\rm (Convergence of maps)} for each component $C_v$ of $C$,
  there is a sequence of translations $t_\nu(v) \in \nu P(v)^\dual$
  such that the translated maps $\e^{-t_\nu}u_\nu$ converges to a map
  $u_v$;
\item \label{item:thincyl1} 
{\rm(Thin cylinder convergence)} for a tropical node $w$ in the
  domain $C$ of $u$ corresponding to an edge $e=(v_+,v_-)$, on the midpoints
  \[m_\nu:=(0,0) \in [-l_\nu/2,l_\nu/2] \times \R/2 \pi \Z \subset C_\nu\]
  \cwl{Resp. 
\[  m_\nu:=(0,0) \in [-l_\nu/2,l_\nu/2] \times [0,1] \subset C_\nu  \]
in the case of strips}
of the sequence of annuli in $C_\nu$ converging to the node $w$, the
translated evaluations $\e^{\hh(t_\nu(v_+)+ t_\nu(v_-))}u_\nu(m_\nu)$
converge to the tropical evaluation (see \eqref{eq:tropev}) of $u$ at
the nodal point $w$.
\end{enumerate}
The translation sequences yield tropical positions of the vertices 
in the
tropical graph of the limit map, justifying the representation in
Figure \ref{fig:trop}.  For any $\nu$
\[\cT(v):= \tfrac {t_\nu(v)}{\nu} \in P(v)^\dual, \quad v \in \Ver(\Gamma) \]
is a vertex position map for the combinatorial type $\Gamma$ of the
limit map.  For a Gromov-converging sequence of maps, area is
automatically preserved in the limit, since the number of domain
marked points is preserved in the limit, and this number is a constant
multiple of area (see \eqref{eq:mult}). \label{rep:area2} We remark
that \hyperref[item:thincyl1]{(Thin cylinder convergence)} is not part
of the definition of convergence of stable maps in smooth manifolds,
since it can be deduced from convergence of maps. However, in case of
neck-stretching, \hyperref[item:thincyl1]{(Thin cylinder convergence)}
does not follow from \hyperref[item:mconv1]{(Convergence of
  maps)}. Moreover \hyperref[item:thincyl1]{(Thin cylinder
  convergence)} is used in the proof of surjectivity of gluing in
Chapter \ref{chap:glue}.

In the statement of Gromov convergence for maps on neck-stretched
manifolds, the perturbations on neck-stretched manifolds are obtained
by gluing a perturbation datum on the broken manifold.  A cylindrical
divisor $\bD$ in a broken manifold $\XX$ can be glued to give a family
of divisors $D^\nu$ in neck-stretched manifolds $X^\nu$.  There is a
natural correspondence of tamed cylindrical divisor-adapted almost
complex structures
\begin{equation} \label{rhonu} \rho_\nu : \J^{\cyl}(\XX,\bD) \to
  \J^{\cyl}(X^\nu,D^\nu) .\end{equation}
\begin{definition} \label{def:breakpert} {\rm(Breaking perturbation
    datum)} Let $\ul{\Pe}$ be a perturbation datum for the broken manifold $\XX$ adapted to a cylindrical divisor $\bD$. A \em{
    breaking perturbation datum} is a family of perturbation data
  $\{\ul{\Pe}_\nu\}_{\nu \in [\nu_0,\infty]}$ for the neck-stretched
  manifolds $(X^\nu,D^\nu)$ satisfying the following property: There
  exists a constant $\nu(\Gamma)$ so that if $\nu > \nu(\Gamma)$ then
  \[\Pe_{\nu,\Gamma}:=(\rho_\nu\circ \JJ_\Gamma, F_\Gamma), \quad \text{where } \Pe_\Gamma=(\JJ_\Gamma,F_\Gamma).\]
 \end{definition}
 The statement of the convergence theorem does not require
 perturbations be regular, but it does require perturbations to be
 adapted to a stabilizing pair $(D^\nu, J^\nu)$ on $X^\nu$.  Such a
 pair is obtained by gluing a stabilizing pair $(\DD,\JJ)$ on the
 broken manifold (constructed in Section \ref{sec:stabpair}).  We
 recall that by Definition \ref{def:stab-brok} of a stabilizing pair
 $(D^\nu,J^\nu)$, any non-constant $J^\nu$-holomorphic sphere in
 $X^\nu$ intersects $D^\nu$ transversally 
 at at least 3 points.  The
 theorem is regarding a sequence of treed holomorphic disks in
 neck-stretched manifolds $X^\nu$ converging to a stable broken disk
 in a broken manifold $\XX$. We recall from Definition \ref{def:bmap}
 that a \em{broken disk} is a broken map, one of whose components in
 the piece $X_{P_0} \subset \XX$ is a treed disk with the disk
 boundary and treed segments mapping to the Lagrangian
 $L \subset X_{P_0}$.
 
 \begin{theorem} \label{thm:cpt-breaking} {\rm(Gromov convergence for
     breaking maps)} Suppose $(\JJ_0,\bD)$ is a stabilizing pair for
   the broken manifold $\XX$ in the sense of Definition
   \ref{def:stab-brok}, and $\ul{\Pe}^\infty$ is a perturbation datum
   on $\XX$ adapted to $(\JJ_0,\bD)$. Suppose $\{\ul{\Pe}^\nu\}_\nu$
   is a sequence of breaking perturbation data (as in Definition
   \ref{def:breakpert}) on neck-stretched manifolds
   $\{X^\nu\}_\nu$.  Let
  \[ u_\nu:C_\nu \to X^\nu \]  
  be a sequence of treed $\ul{\Pe}^\nu$-holomorphic disks with uniformly
  bounded area $\Area(u_\nu)$.  There is a
  subsequence of $\{u_\nu\}_\nu$ that Gromov converges to a
  $\ul{\Pe}^\infty$-holomorphic stable broken disk $u:C \to \XX_\PP$ modelled
  on a tropical graph $\Gamma$. 
  The limit $u$ is unique up to domain
  reparametrizations and the action of the identity component of the
  tropical symmetry group $T_{\on{trop},\W}(\Gamma)$ (see
  \eqref{eq:Tidcpt}).
\end{theorem}

The second result of the Chapter concerns convergence of a sequence of
broken maps with uniformly bounded area. After passing to a
subsequence the maps $u_\nu$ have the same type, say $\Gamma$ with
tropical structure $\cT$. The limit $u$ typically has additional
domain components. The tropical structure $\cT'$ of the limit is
related to $\cT$ by a \em{tropical edge collapse morphism}, which is
a morphism of graphs $\cT' \xrightarrow{\kappa} \cT$ that collapses a
subset of edges of $\cT'$, and the {direction}s of uncollapsed edges are the
same in $\cT$ and $\cT'$. See Figure \ref{fig:tcollapse} for an
example.
  \index{Collapsing an edge!Tropical edge collapse}

\begin{figure}[h]
  \centering\scalebox{.8}{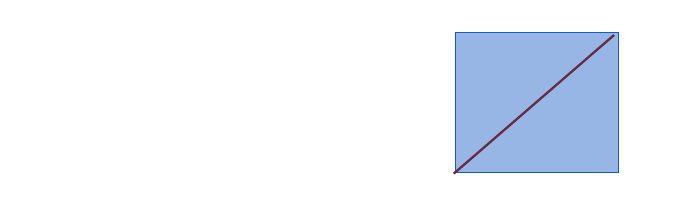}
  \caption{Tropical edge collapse morphism between tropical graphs.}
  \label{fig:tcollapse}
\end{figure}
 
\begin{theorem}
  \label{thm:cpt-broken}
  {\rm(Gromov convergence for broken maps)} Suppose $(\JJ_0,\bD)$ is a
  stabilizing pair for the broken manifold $\XX_\PP$, and $\ul{\Pe}$
  is a perturbation datum on $\XX_\PP$ adapted to
  $(\JJ_0,\bD)$. Suppose $u_\nu:C_\nu \to \XX_{\cP}$ is a sequence of
   broken $\ul{\Pe}$-holomorphic disks with uniformly bounded
  area.

  After passing to a subsequence, the type of the broken disks $u_\nu$
  is $\nu$-independent, and is equal to, say, $\Gamma$. The sequence
  $u_\nu$ Gromov converges to a broken $\ul{\Pe}$-holomorphic disk
  $u:C \to \XX_{\cP}$ of type $\Gamma'$ for which there is a tropical
  edge-collapse morphism $\Gamma' \to \Gamma$. The limit $u$ is unique
  up to domain reparametrizations and the action of the identity
  component of the tropical symmetry group $T_{\on{trop},\W}(\Gamma)$
  (see \eqref{eq:Tidcpt}).

  If the edge-collapse $\Gamma' \to \Gamma$ is non-trivial (in the
  sense of Definition \ref{def:tropcol1}) then
  $\dim_\C(T_\trop(\Gamma')) > \dim_\C(T_\trop(\Gamma))$.
\end{theorem}

\begin{remark}
  \label{rem:convunique}
  In theorems \ref{thm:cpt-breaking} and \ref{thm:cpt-broken}, the
  limit broken map is uniquely determined if it is rigid.  Indeed, for
  a rigid tropical graph $\Gamma$ the tropical symmetry group is
  finite.
\end{remark}

\section{Gromov convergence}\label{sec:convdef}

In this Section, we define the notion of convergence in the
compactness results, Theorem \ref{thm:cpt-breaking} and Theorem
\ref{thm:cpt-broken}.  Much of the Section is devoted to fixing
identifications between domain curves that are close to each other in
the compactified moduli space of stable curves. In addition to fixing such
identifications in the complements of nodes, we also define a notion
of when a sequence of annuli in nearby curves converges to a
neighborhood of a node in a nodal curve; this notion is used in
\hyperref[item:thincyl1]{(Thin cylinder convergence)}, which is part
of the statement of Gromov convergence.  Finally, we define
\em{translation sequences} which are sequences of rescalings of the
target; the convergence of maps is modulo these rescalings.

To fix identifications between domain
curves that are close to each other in the compactified moduli space
of stable curves, we fix a choice of an exponential map in
the neighborhood of nodes.  Different choices lead to the same notion
of convergence.
To prove the surjectivity of the gluing operation in Chapter
\ref{chap:glue}, it is helpful to fix these identifications and use
the same identifications in the statement of the convergence result
and in the gluing construction.
We start by recalling a construction of deformations of a stable nodal curve.  We describe the construction for rational
stable nodal curves (without boundary). The extensions to treed disks
are explained later in Remark \ref{rem:id-treed-disks}.  Let $\Gamma$
be the combinatorial type of a rational stable nodal curve with
$d(\black)$ markings. The moduli space of curves $\M_\Gamma$ is a
submanifold of the compactified moduli space $\M_{d(\black)}$ whose
tubular neighborhood can be described as follows.  For a stable curve
$S$ in $\M_\Gamma$, let $\tS$ be the normalization of $S$ at the nodal
points.  For any edge $e \in \Edge_{\black,-}(\Gamma)$, we fix a map
\begin{multline}
  \label{eq:cexp}
  \exp_{w_e^\pm}^S : (U(T_{w_e^\pm}\tS),0) \to (U_{w_e^\pm}(\tS),w_e^\pm) , \\U(T_{w_e^\pm}\tS) \subset T_{w_e^\pm}\tS, \quad  U_{w_e^\pm}(\tS) \subset \tS
\end{multline}
that biholomorphically maps a neighborhood $U(T_{w_e^\pm}\tS)$ of the
origin in the tangent space onto a neighborhood $U_{w_e^\pm}(\tS)$ of
the lift of the node $w_e$, satisfies $d\exp_{w_e^\pm}^S(0)=\Id$, and
varies smoothly with $S$.  The family of maps
$\{\exp_{w_e^\pm} : e \in \Edge_{\black,-}(\Gamma)\}$ is fixed for the
rest of this book. Whenever we choose complex coordinates
\[z_\pm:(U_{w_e^\pm}(\tS),w_e^\pm) \to (\C,0) \]
in neighborhoods of the node we assume that they are compatible with
$\exp_{w_e^\pm}^S$. That is, $z_\pm$ is the composition of
$\exp_{w^\pm}$ with a linear map on the tangent space.  On an open
neighborhood
\[U_{\M_\Gamma} \subset \M_{d(\black)} \]
of $\M_\Gamma$, there is a projection map
\[\pi_\Gamma: U_{\M_\Gamma} \to \M_\Gamma. \]
such that curves in a fiber $\pi_\Gamma^{-1}(S)$ are obtained by
gluing the interior nodes of $S$ as follows. (See, for example,
Siebert \cite[Proposition 2.4]{siebert} in the closed case; the open
case is similar.)

\begin{definition} Given a nodal sphere $S$ of type $\Gamma$, a \em{
    collection of gluing parameters} \index{Gluing parameter} is a
  tuple
  \[ \delta=(\delta_e)_{e \in \Edge_-(\Gamma)}
    ,\quad
    \delta_e \in T_{w^+_e}\tS \tensor T_{w^-_e}\tS.\]
 The \em{curve corresponding to a gluing parameter} $\delta$ is
 defined by
\begin{multline} \label{eq:cdelta}
  S^\delta:=(S \bs \cup_e U'_{w_e^\pm}(\tS))/\sim,\\
 z_+ \sim z_-  \Leftrightarrow (\exp^S_{w^+_e})^{-1}(z_+) \tensor (\exp^S_{w^-_e})^{-1}(z_-)=\delta_e. 
\end{multline}
\end{definition}
\noindent Here $U'_{w_e^\pm}(\tS) \subset \tS$ is an open neighborhood of
$w_e^\pm$ such that the boundary $\partial U'_{w_e^\pm}(\tS)$ is
identified to $\partial U_{w_e^\mp}(\tS)$ (with reversed orientation)
by the equivalence relation $\sim$. Thus, for any node $w_e$, the
relation $\sim$ identifies the pair of annuli
\begin{equation}
  \label{eq:necke}
  U_{w^+_e}(\tS)\bs U'_{w^+_e}(\tS) \xrightarrow{\sim} U_{w^-_e}(\tS)\bs U'_{w^-_e}(\tS).
\end{equation}
The resulting annulus \cwl{(resp. rectangle in the case of Lagrangian
  boundary conditions)} in $S^\nu$ is called $\Neck_e(S^\nu)$, and
\begin{equation}
  \label{eq:neckC}
  \Neck(S^\nu):=\cup_e \Neck_e(S^\nu) \subset S^\nu.
\end{equation}
The gluing construction maps a neighborhood of zero in the space of
gluing parameters to a neighbourhood of $S$ in the fiber
$\pi_\Gamma^{-1}(S)$.  Curves in the neighborhood $U_{\M_\Gamma}$ of
$\M_\Gamma$ possess a \em{gluing parameter} \index{Gluing parameter}
function for every smoothed node
\[\nl=(\nl_e)_e:U_{\M_\Gamma} \to \prod_{e \in
    \Edge_{\black,-}(\Gamma)} T_{w^+_e}\tS \tensor T_{w^-_e}\tS, \quad
  S^\delta \mapsto \delta.\]
\begin{remark} 
  Suppose that at a node $w_e$ we choose a framing
  $\fr : T_{w^+_e}\tS \tensor T_{w^-_e}\tS \to \C$ as in
  \eqref{eq:framingeq} or complex coordinates
  \[z_+:(\tS,w^+_e) \to (\C,0), \quad z_-:(\tS,w^-_e) \to (\C,0).\]
  Then the gluing parameter $\nl_e$ may be identified with a complex
  number.
\end{remark}
\begin{remark} \label{idents} {\rm(Identifications between nearby
    curves)} Consider
  a rational nodal curve $[S] \in \M_\Gamma$ and
  a smooth curve 
  $[S'] \in \M_{d(\black)}$ in the
  neighborhood of $S$. The complement of nodes $S^\circ$ can be
  identified to subsets of the curve $S'$ as follows.  First, suppose
  that $S'$ is obtained by gluing $S$. That is,
  $S' \in \pi_\Gamma^{-1}(S)$.  For a vertex $v \in \Ver(\Gamma)$, let
  $S'(v) \subset S'$ be the subset corresponding to the component
  $S_v \subset S$ including the necks $\Neck_e(S')$, $v \in e$.  The
  subsets $S'(v)$ cover $S'$.  By the gluing construction, there are
  natural inclusions
  \begin{equation}
    \label{eq:iccv}
    i_{S_v,S'}: S'(v) \to S_v \subset S.  
  \end{equation}
  Next, suppose $S'$ is not in $\pi_\Gamma^{-1}(S)$.  Consider
  a neighborhood $U_{\Gamma,S} \subset \M_\Gamma$ of $S$ on which the
  universal curve $\U_\Gamma$ can be trivialized, so that the
  variation of the complex structure is given by a map
  \begin{equation}
    \label{eq:univtriv}
    U_{\Gamma,S} \to \J(\U_\Gamma), \quad m \mapsto j(m)  
  \end{equation}
  for which $j(m)$ is $m$-independent in the neighborhood of special
  points.  The trivialization of the universal curve $\U_\Gamma$ gives
  a diffeomorphism $\phi : \pi_\Gamma(S') \to S$. The maps
  $i_{S_v,S'}$ are defined as compositions
  \[i_{S_v,S'}:=\phi \circ i_{\pi_\Gamma(S')_v,S'}.\]
    If a sequence $S_\nu$ of curves converges to $S$, then the images
  $i_{S_v,S_\nu}(S_\nu(v))$ exhaust the complement of nodes $S_v^\circ$
  as $\nu \to \infty$.
\end{remark}

\begin{remark} \label{rem:uniqueid} {\rm(Uniqueness of
    identifications)} The identifications between regions of nearby
  curves are unique in the following sense. Suppose a sequence
  $[S_\nu] \in \ol \M_{d(\black),d(\white)}$ converges to a curve
  $[S]$.  For any node $e=(v_+,v_-)$ in $S$, the identification in the
  neck region
  \begin{equation}
    \label{eq:neckid}
    i_{S_{v_\pm},S^\nu} : \Neck_e(S^\nu) \to  S_{v_\pm}  
  \end{equation}
  is uniquely determined by the choice of the exponential map
  \eqref{eq:cexp}. On the complement of the neck region, let
  \[\phi_\nu, \phi_\nu': S^\nu \bs \Neck(S^\nu) \to S \]
  be two possible identifications given by trivializations of the
  universal curve.  The maps
  $\phi_\nu$, $\phi_\nu'$ have the same image, and the maps
  $\phi_\nu' \circ \phi_\nu^{-1}$ converge to the identity uniformly
  in all derivatives.
\end{remark}

\begin{remark}\label{rem:id-treed-disks}
  {\rm(Identifications of nearby treed disks)}
  The identifications between nearby nodal spheres extend to treed disks.  Let $C$ be a stable treed disk of type $\Gamma$ with surface part $S$ and treed part $T$.
  \begin{enumerate}
  \item {\rm(Glued treed disk)}
    For a disk node $e \in \Edge^0_\white(\Gamma)$ of zero length, $\ell(e)=0$, a \em{gluing parameter} at the node $w_e \in C$ is an element
    \[\delta_e \in (T_{w_e^+}\partial S \tensor T_{w_e^-} \partial S)_{\leq 0} \subset  T_{w_e^+}S \tensor T_{w_e^-} S\]
    where the orientation on
    $T_{w_e^+}\partial S \tensor T_{w_e^-} \partial S$ is derived from
    the orientation on the boundaries of components of $S$. Given
    gluing parameters
    \[\delta:=(\delta_e)_{e \in \Edge_{\black,-}(\Gamma) \cup \Edge_\white^0(\Gamma)}\]
    for all nodes in the surface component $S$, a glued surface
    $ S^\delta $ is defined exactly as in \eqref{eq:cdelta}. The glued
    treed curve is $C^\delta:=S^\delta \cup T$, that is, the treed
    part in $C^\delta$ is the same as that of $C$.
  \item{\rm(Domain identifications)} Let $C$ be a treed curve of type
    $\Gamma$, and let $\Gamma'$ be a treed disk type obtained by
    collapsing a subset of interior edges
    $e \in \Edge_{\black,-}(\Gamma)$ and zero length boundary edges
    $e \in \Edge_{\white,-}^0(\Gamma)$.  Let $C'=S' \cup T'$ be a
    treed disk of type $\Gamma'$ that is close to $C$.  For any surface
    component $S_v \subset C$ corresponding to a vertex
    $v \in \Ver(\Gamma)$, there is a subset $S'(v) \subset S'$ and an
    identification
    \[i_{S_v,S'} : S'(v) \to S_v\]
    exactly as in \eqref{eq:iccv}. 
  \end{enumerate}
\end{remark}

Next, we identify punctured neighborhoods of interior nodes to
annuli in nearby curves. These identifications are needed for stating
the \hyperref[item:thincyl2]{(Thin cylinder convergence)} for tropical nodes in the definition of Gromov convergence. We make these identifications only for interior nodes since all tropical nodes are interior nodes.

Informally, we may restate the definition below as follows: A sequence
of annuli converging to a node $w$ is obtained by gluing neighborhoods
$U_{w^+}$, $U_{w^-}$ of the nodal lifts $w^+$, $w^-$.  Given
holomorphic coordinates in a neighborhood of the nodal $w^+$, $w^-$, a
sequence of centered annuli converging to a node $w$ is given by
gluing neighborhoods of $w^+$, $w^-$ that have the same radius.

\begin{definition} \label{def:ann2node} {\rm(Annuli converging to a
    node)} Let $C_\nu$ be a sequence of treed disks converging to a
  limit curve $C$ for which the arguments
  $\frac {\nl_e(C_\nu)}{|\nl_e(C_\nu)|}$ of the gluing parameters
  converge for all interior edges $e \in
  \Edge_{\black,-}(\Gamma)$.  Let $w \in C$ be an interior node
  corresponding to the edge $e=(v_+,v_-)\in \Edge_{\black,-}(\Gamma)$.
  \begin{enumerate}
  \item A sequence of \em{annuli $A_\nu \subset C_\nu$ converges to a
      node $w$} in $C$ if there are open neighborhoods
      $U_{w^\pm} \subset S_{{v}_\pm}$ of
    the lifts $w^+$, $w^-$ of $w$ such that
    \[A_\nu=i_{S_{v_+},C_\nu}^{-1}(U_{w^+}) \cap
      i_{S_{v_-},C_\nu}^{-1}(U_{w^-}).\]
    We say that the sequence of annuli $A_\nu$ is obtained by \em{
      gluing the node $w$}.
  \item {\rm(Centered annuli converging to a node)} Suppose the node
    $w \in C$ is equipped with complex coordinates
    \[ z_\pm : (U_{w^\pm}, w^\pm) \to (\C, 0) \]
    on neighborhoods $U_{w^\pm} \subset C_{{v}_\pm}$ of its lifts
    $w^+$, $w^-$.
    A \em{sequence of centered annuli converging to the node $w$} is
    a sequence of parametrized annuli
    \[A_\nu:=[-l_\nu'/2, l_\nu'/2] \times \R/2\pi\Z \hra C_\nu\]
  \cwl{resp. $ [-l_\nu'/2, l_\nu'/2] \times [0,1] \hra C_\nu $ }
    for which there exists $\eps>0$ such that
    \[A_\nu=i_{S_{v_+},C_\nu}^{-1}(\{|z_+|\leq \eps\}) \cap
      i_{S_{v_-},C_\nu}^{-1}(\{|z_-|\leq \eps\}),\]
    and the map
    \[A_\nu \xhookrightarrow{i_{S_{v_\pm},C_\nu}} C_{v_\pm}
      \xrightarrow{z_\pm} \C \quad \text{is equal to} \quad (s,t)
      \mapsto \exp(\mp (s+it)-(l_\nu+i\theta_\nu)/2).\]
    Here $e^{-(l_\nu+i\theta_\nu)} \in \C^\times$ is the neck length
    parameter for the curve $C_\nu$ resulting from the choice of
    coordinates $z_\pm$.
  \end{enumerate}
\end{definition}

The following definition describes rescalings of the target spaces
required in the statement of convergence of components of the broken
map.  We recall from Definition \ref{def:transdef} that a translation
is a way of identifying a subset of a neck-stretched manifold $X^\nu$
to a subset of the broken manifold $\XX$.  The set of translations is
parametrized by $\nu B^\dual$ where $B^\dual$ is the dual complex of
the neck-stretching. A translation sequence for a tropical graph
consists of a translation sequence corresponding to each of the
vertices of the tropical graph that are compatible with the edge
{direction}s of the tropical graph.

\begin{definition}{\rm(Translation sequence for a tropical graph)}
  \label{def:trseq}
  \index{Translation!Translation sequence} Suppose $\Gamma$ is a
  pre-tropical graph (as in Definition \ref{def:tropgraph}). A \em{
    $\Gamma$-translation sequence} consists of a collection of
  sequences
\[\{ t_\nu(v) \in \nu B^\dual
  \}_\nu, \quad  v \in \Ver(\Gamma)\]
 such that the following conditions hold : 
  \begin{enumerate}
  \item {\rm(Polytope)} For each vertex $v$
    \begin{equation} \label{pcond} t_\nu(v) \in \nu P(v)^\dual
      \subset \nu B^\dual,
    \end{equation}
    and for any polytope $P_0 \in \PP$, $P_0 \supset P(v)$,
    \[d_{B^\dual}(t_\nu(v),\nu P_0^\dual) \to \infty \quad \text{as} \quad \nu \to \infty.  \] 
  \item {\rm(Direction)} For any node $e$ between vertices $v_+$, $v_-$,
    there is a sequence $l_\nu(e) \to \infty$ such that
    \begin{equation}
      \label{eq:direction}
    t_\nu(v_+) - t_\nu(v_+) = \cT(e) l_\nu.  
    \end{equation}
  \end{enumerate}
\end{definition}

\begin{remark}
  \label{rem:trans2pos}
  The existence of a translation sequence $\{t_\nu\}_\nu$ on a
  pre-tropical graph $\Gamma$ implies that $\Gamma$ is a tropical
  graph, since, for any $\nu$,
  \[\Ver(\Gamma) \ni v \mapsto \tfrac{t_\nu(v)}{\nu} \in \nu P(v)^\circ\]
  is a tropical vertex position for $\Gamma$.
\end{remark}

Finally, we define the notion of convergence of a sequence of maps in neck-stretched manifolds to a broken map in a broken manifold.
\begin{definition} \label{def:grom-breaking}
  {\rm(Gromov convergence,
    multiple cuts)} Let $C_\nu=S_\nu \cup T_\nu$ be a sequence of stable treed disks of type $\Gamma_0$. A sequence of holomorphic maps
  $u_\nu:(C_\nu,\ul z_\nu) \to X^\nu$ \em{(Gromov) converges} to a broken map
  $u:C=(S \cup T) \to \XX_\cP$ of type $\Gamma$ and framing $\fr$ if the following are satisfied. 
  \begin{enumerate}
  \item \label{item:dconv2} {\rm(Convergence of domains)} The sequence
    of treed disks $C_\nu$ converges to $C$ and for any tropical node
    $w_e$ of the limit map $u$ the arguments
    $\frac {\nl_e(C_\nu)}{|\nl_e(C_\nu)|}$ of the gluing parameters
    converge to a limit.  Using the fixed holomorphic exponential map
    \eqref{eq:cexp}, let $S_\nu(v) \subset S_\nu$ be the subset
    corresponding to a vertex $v \in \Ver(\Gamma)$, and let
    \[i_{v,\nu}:=i_{S_v,C_\nu}: S_\nu(v) \to S_v, \quad S_\nu(v)
      \subset S_\nu,\]
    be embeddings from \eqref{eq:neckid} whose images
    $i_{v,\nu}(S_\nu(v))$ exhaust the complement of nodes $S_v^\circ$
    as $\nu \to \infty$.
  \item \label{item:mconv2} {\rm(Convergence of maps)} There is a
    $\Gamma$-translation sequence $\{t_\nu(v)\}_{v,\nu}$ such that for
    any irreducible surface component $S_v \subset S$, the sequence of
    maps
    \[  S_v^\circ \supset i_{v,\nu}(S_\nu(v)) \xrightarrow{\e^{-t_\nu(v)}(u_\nu \circ i_{v,\nu}^{-1})}
    \XC_{P(v)} \]
  converges in $C^\infty_{\loc}(S_v^\circ)$ to
  $u_v:S_v^\circ \to \XC_{P(v)}$. The map
  $\e^{-t_{\nu}(v)}: X^\nu_{\tP(v)} \to \XX_{P(v)}$ is defined in
  \eqref{eq:transinc}.  For each boundary edge $e$ in $\Gamma_0$, the
  map $u_\nu|T_{e,\nu}$ on the treed segment converges to a (possibly
  broken) treed segment in $u$.
\item \label{item:thincyl2} {\rm(Thin cylinder convergence)} For a
  node $w$ in $C$ corresponding to a tropical edge
  $e=(v_+,v_-) \in \Edge_\trop(\Gamma)$, let
    \[ z_\pm : (U_{w^\pm}, w^\pm) \to (\C, 0) \]
    be matching coordinates (see Definition \ref{def:bmap} following
    \eqref{eq:nodematch}) on neighborhoods
    $U_{w^\pm} \subset S_{{v}_\pm}$ of the nodal point which respect
    the framing $\fr_e$.  Let
    \[ A(l_\nu):=[-l_\nu/2,l_\nu/2] \times S^1 \subset S_\nu \]
    be a sequence of centered annuli converging to the node $w$, see
    Definition \ref{def:ann2node}.  Then the sequence
    \[ x_\nu:=\e^{-\hh(t_\nu(v_+) + t_{\nu}(v_-))} u_\nu(0,0) \in \XC_{P(e)}\]
    converges to a limit $x_0$, and the components $u_{v_\pm}$ of the
    broken map are asymptotically close to
    \[z_\pm \mapsto z_\pm^{\cT(e)} x_0\]
    in the sense of Remark \ref{rem:expconv}.
  \end{enumerate}
\end{definition}

\begin{remark}
  {\rm(On the uniqueness of limits)}
\label{rem:addedref}
The uniqueness of the limit, when it exists, is explained in Theorem
\ref{thm:cpt-breaking}.  The Gromov limit does not depend on the various
choices made in this section. In particular, the proof of uniqueness
only depends on the \hyperref[item:dconv2]{(Convergence of domains)}
and \hyperref[item:mconv2]{(Convergence of maps)}. The latter item
depends on the maps $i_{v,\nu}$ which are identifications between
complements of neighborhoods of nodes in nodal curves, and complements
of corresponding neck regions in nearby smooth curves. Different
choices of these maps converge in the limit as pointed out in Remark
\ref{rem:uniqueid}.

We remind the reader that we should not expect to prove that the
Gromov limit is independent of perturbation data which includes the
choice of Donaldson divisors, since the definition of
pseudoholomorphic maps is dependent on these choices. Later, in
Chapter \ref{chap:bfa}, we will prove that different choices of
perturbation data produce homotopy-equivalent Fukaya algebras.
\end{remark}

\section{Horizontal convergence}\label{sec:horiz-conv}
\index{Horizontal convergence}

In this section, we discuss a notion of convergence for a sequence of
points in neck-stretched manifolds, called \em{horizontal
  convergence}, that is useful in the proof of convergence for
breaking maps. It has the feature that any sequence of points
$\{x_\nu \in X^\nu\}_\nu$ in the family of neck-stretched manifolds
$X^\nu$ has a subsequence that horizontally converges.

\begin{definition}\label{def:cptcnv}
  {\rm(Horizontal convergence)} A sequence of points $x_\nu \in X^\nu$
  \em{horizontally converges} to a point $x \in \XB_P$ for a polytope
  $P \in \PP$ if
  \begin{itemize}
  \item $x_\nu \in X^\nu_{\tP}$ for all $\nu$, and the sequence
    $\pi_P^\nu(x_\nu)$ converges to $x$, where
    $\pi_P^\nu:X^\nu_{\tP} \to \XB_P$ is the projection map
    $X^\nu_{\tP} \to X^\nu_P$ from \eqref{eq:xpnu} composed with the
    inclusion $X^\nu_P \to \XB_P$,
  \item and for any subsequence of $\{x_\nu\}_\nu$, the above
    condition is not satisfied for any polytope $P_0 \supset P$.
  \end{itemize}
\end{definition}

We present some results on horizontal convergence of sequences: 

\begin{lemma}\label{lem:horizcpt}
  For any sequence $x_\nu \in X^\nu$, there exists a subsequence of
  $\{x_{\nu_k}\}_k$ that converges horizontally.  There is a unique
  polytope $P \in \PP$ for which the subsequence $\{x_{\nu_k}\}_k$
  converges horizontally in $X_P$.
\end{lemma}

\begin{proof}
  Recall from \eqref{eq:pib-J} that there is a projection to the dual
  polytope $\pi_{\nu B^\dual} : X^\nu \to \nu B^\dual$ for all $\nu$.
  There is a polytope $P \in \PP$ such that, after passing to a
  subsequence, $(x_\nu)_\nu$ satisfies the equation
  \begin{equation}
    \label{eq:dpdual}
    \sup_\nu d(\pi_{\nu B^\dual}(x_\nu), \nu P^\dual)<\infty, \quad d(P_0^\dual,\pi_{\nu B^\dual}(x_\nu)) \to \infty \enspace \forall P_0 \supset P.
  \end{equation}
  It is then clear that a subsequence of $(x_\nu)_\nu$ and the
  polytope $P$ satisfy the conclusions of the Lemma.
\end{proof}

The following Lemma relates horizontal convergence to a property of translation sequences. 
\begin{lemma}\label{lem:hcp}
  Suppose $P \in \PP$ is a polytope with $\codim(P)>0$, and suppose
  $x_\nu \in X^\nu$ is a sequence of points, and
  $t_\nu \in \nu P^\dual$ is a sequence of translations such that
  $\e^{-t_\nu}x_\nu$ converges in $\XC_P$. Then, the following are
  equivalent:
  \begin{enumerate}
  \item \label{part:hc} The sequence $x_\nu \in X^\nu$ converges horizontally in $X_P$.
    \item \label{part:pc} For any $P_0 \supset P$ (that is, $P$ is a proper face of $P_0$), $d(t_\nu, \nu P_0^\dual) \to \infty$. 
  \end{enumerate}
\end{lemma}
\begin{proof}
  Firstly, we observe that the convergence of the translated sequence
  $\e^{-t_\nu}x_\nu$ implies that
  $d(\pi_{\nu B^\dual}(x_\nu), \nu P^\dual)$ is uniformly bounded, and
  that $(\pi_P(x_\nu))_\nu$ converges to a limit in $\XB_P$.  The
  condition in \eqref{part:pc} then implies that $(x_\nu)_\nu$
  horizontally converges in $X_P$. The converse is similar and is left
  to the reader.
\end{proof}

The next result may be seen as the horizontal convergence version of
the Arzela-Ascoli theorem.
\begin{lemma}\label{lem:cptconv}
  Suppose $C$ is a connected curve and $u_\nu:C \to X^\nu$ is a
  sequence of differentiable maps satisfying
  $\sup_\nu\Mod{\d u_\nu}_{L^\infty}<\infty$. There exists a
  subsequence of maps $\{u_{\nu_k}\}_k$ and a polytope $P \in \PP$
  such that
  \begin{enumerate}
  \item \label{part:cptconva}  for all $z \in C$ the sequence $u_{\nu_k}(z)$ converges
    horizontally in $X_P$,
  \item \label{part:cptconvb} and there is a sequence of translations $t_\nu \in \nu P^\dual$  such that
    $\e^{-t_\nu}u_\nu$ converges uniformly in compact sets to a map
    $u : C  \to \XC_P$.
  \end{enumerate}
  For the subsequence $\{u_{\nu_k}\}_k$, the polytope $P$ is unique.
\end{lemma}

\begin{proof}
  First, assume that $C$ is compact. The proof of Lemma
  \ref{lem:horizcpt} carries over. Indeed, by the uniform bound on
  $\sup_\nu\Mod{\d u_\nu}_{L^\infty}$, the sequence
  $d(\pi_{\nu B^\dual}(u_\nu(z_0)), P^\dual)$ is uniformly bounded for
  a fixed point $z_0 \in C$ iff the sequence
  $d(\pi_{\nu B^\dual}(u_\nu(z)), P^\dual)$ is uniformly bounded for
  any $z \in C$. The result also holds for non-compact curves since
  they are exhausted by a sequence of compact curves.

  For \eqref{part:cptconvb}, we view $u_\nu$ as mapping to the
  $P$-cylinder, since for any translation
  $t_\nu \in \nu P^\dual \subset \t_P$ there is an embedding
  \[\e^{-t_\nu}: X^\nu_{\tP} \to \XC_P,\]
  see \eqref{eq:transinc}.  We choose a sequence of translations
  $t_\nu \in t_P^\dual$ so that
  $\pi_{\t_P} \circ (\e^{-t_\nu}u_\nu)(z_0)$ is
  $\nu$-independent.  The horizontal convergence of $u_\nu$ and the
  uniform bound on the derivatives of $u_\nu$ imply that for any
  compact set $K \subset C$, the images $\e^{-t_\nu}u_\nu(K)$ are
  contained in a compact set of $\XC_P$.  The Arzela-Ascoli theorem
  then implies that the sequence $\e^{-t_\nu}\tiu_\nu$ converges.
\end{proof}

\section{Breaking annuli}

In the next proposition, called the \em{breaking annulus lemma}, we
consider a sequence of maps on annuli with small areas and bounded
Hofer energy that converge to a pair of maps on punctured disks, the
maps on punctures being asymptotic to trivial cylinders.
See Figure \ref{fig:breakingst} for a depiction of the maps. 
The main
conclusion of the proposition is that there is a trivial cylinder
$u_\triv$ to which the sequence of annuli is asymptotically
close. That is, the distance from $u_\triv$ decays exponentially
towards the middle of the cylinders.  Furthermore, if the limit
punctured disks lie in the components $\XX_{P_+}$ and $\XX_{P_-}$ of
the broken manifold, where $P_+$, $P_-$ are polygons in the polyhedral
decomposition $\PP$, then, the trivial cylinder $u_\triv$ lies in
$\XX_{P_\cap}$, where $P_\cap:=P_+ \cap P_-$, or in other words, the
{direction} $\mu$ of $u_\triv$ is in $\t_{P_\cap}$.

In the proof of convergence of maps in neck-stretched manifolds to a
broken map in Section \ref{sec:breakingconv}, the breaking annulus
lemma is used to show that a sequence of maps on annuli (as in the
previous paragraph) converges to a neighborhood of a node in a broken
map; in particular, to show that the \em{matching condition} is
satisfied at nodes. The matching condition says that on the punctured
neighborhoods of both lifts of the node, the map is asymptotic to the
same trivial cylinder.  In this setting, both punctured neighborhoods
are asymptotic to $u_\triv$. The asymptotic closeness of the annuli to
$u_\triv$ also implies that the convergence of the annuli satisfies
the \hyperref[item:thincyl2]{(Thin cylinder convergence)}
property. (Horizontal matching), which is one of the conclusions of
the breaking annuli result, is a sub-condition of the node-matching
condition.

The breaking annulus lemma also gives an approximation of the
difference in the translation sequences in the convergence of maps on
the two ends of the annuli (see \eqref{part:breaking-anna}), which is
used in constructing the tropical structure of the limit map in the
proof of Gromov convergence in Section \ref{sec:breakingconv}. This
approximation follows from the fact that the maps on the annuli are
well-approximated by trivial cylinders of the same length.  The
breaking annulus lemma is analogous to results in Parker \cite[Lemmas
5.11, 5.13]{bp9}.

In the statement of the breaking annulus lemma, the hypothesis on
small area is replaced by a bubble-free condition (Definition
\ref{def:bubble-free}).  For any $L > 0$, we denote by
%
  \[ A(L) := [-\tfrac L 2 , \tfrac L 2] \times S^1 \cong \{ z \in \C \ : \ |z| \in [e^{-L/2},
  e^{L/2}] \}  \]
the cylinder with length $L$ equipped with the product metric. 

\begin{figure}[ht]
  \centering \scalebox{.8}{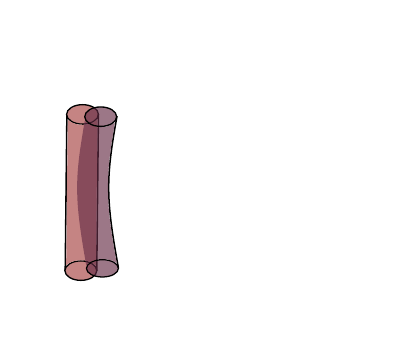}
  \caption{A sequence of annuli converging to a neighborhood of a node in a broken map.}
  \label{fig:breakingst}
\end{figure}

\begin{definition}\label{def:bubble-free}
  {\rm(Bubble-free)} A sequence of maps
  \[ u_\nu: A(l_\nu) \to X^\nu , \quad l_\nu \to \infty\]
  satisfying $\sup_\nu \Mod{du_\nu}_{L^\infty}<\infty$ is \em{
    bubble-free} if, for any $P \in \PP$, there is no non-constant map
  $v: \R \times S^1 \to X_P$ that is a limit of a subsequence of
  reparametrized maps
   \[\pi_P \circ u_\nu(\cdot - r_\nu) \quad \text{with } r_\nu \in \R, \left|r_\nu \pm \tfrac {l_\nu} 2 \right| \to \infty.\]
 \end{definition}

 \begin{proposition}
  \label{prop:breaking-ann}
  {\rm(Breaking annulus lemma)} Let $\JJ_0$ be a cylindrical almost
  complex structure on $\XX$ for which Hofer energy is monotonic (see
  Lemma \ref{lem:monot}), and suppose that for any $\nu$ the almost
  complex structures $J_0^\nu$ is obtained by gluing $\JJ_0$ on the
  necks.  There are constants $0<\rho<1$, $c>0$ such that the
  following hold.  For a sequence $l_\nu \to \infty$, let
  \[ u_\nu: A(l_\nu) \to X^\nu \]
  be a sequence of $J_0^\nu$-holomorphic maps satisfying 
  \begin{enumerate}[label*=(\arabic*)]
  \item\label{ba-hyp1} $\sup_\nu E_{\Hof}(u_\nu) <\infty$,
    $\sup_{z \in \R \times S^1, \nu} |\d u_\nu(z)|<\infty,$ where the
    domain annuli
    $A(l_\nu):=[-\tfrac {l_\nu} 2 , \tfrac {l_\nu} 2] \times S^1 $ are
    equipped with the product metric.
\item \label{ba-hyp2} The sequence $\{u_\nu\}_\nu$ is bubble-free.
\item There exist polytopes $P_+, P_-$ such that the sequence of maps
  $u_\nu(\cdot \pm \frac {l_\nu} 2)$ converges horizontally in
  $X_{P_\pm}$.
\item \label{part:banhypd} There exists a sequence of translations
  $t_\nu^\pm \in \nu P^\dual_\pm$ such that the sequence
  $\e^{-t_\nu^\pm}u_\nu(\cdot \pm \frac {l_\nu} 2)$ converges in
  $C^\infty_\loc$ to
    \[u_\pm : \R_\mp \times S^1 \to \XC_{P_\pm},\] and the map
    extends holomorphically over $\mp \infty$, possibly after passing to a finite cover in the orbifold case. 
  \end{enumerate}
  Then, there exists $\mu \in \t_{P_\cap,\Z}$, \cwl{resp. $\t_P$ in
    the case with Lagrangian boundary} $P_\cap:=P_+ \cap P_-$ for
  which the following hold after passing to a subsequence of
  $\{u_\nu\}_\nu$.
  \begin{enumerate}
  \item \label{part:breaking-anna} The sequence
    $t_\nu^+ - t_\nu^- - \mu l_\nu \in \t_{P_\cap}$ is uniformly
    bounded. 
  \item {\rm(Horizontal matching)} The points
    $(\pi_{P_+} \circ u_+)(-\infty)$,
    $(\pi_{P_-} \circ u_-)(+\infty)$ lie in
    $\XB_{P_\cap} \subset \ol X_{P_\pm}$ and
    $(\pi_{P_+} \circ u_+)(-\infty)=(\pi_{P_-} \circ u_-)(+\infty)$.
    \footnote{The point
      $(\pi_{P_\pm} \circ u_\pm)(\mp\infty) \in \ol X_{P_\pm}$ is
      defined, see \eqref{eq:extpip}.}
  \item {\rm(Asymptotic decay)} Let $\xi_\nu$ be a section defined by
    the relation
    \begin{equation*}
      u_\nu=\exp_{u_{\nu,\on{triv}}}\xi_\nu, \quad u_{\nu,\on{triv}}(s,t):=e^{\mu(s+it)}u_\nu(0,0).
    \end{equation*}
    There exists $l \geq 0$ and a subsequence such that
    \begin{equation}\label{eq:ann-est}
      \abs{\xi_\nu(s,t)} \leq c(e^{\rho(s-\frac {l_\nu} 2)} +
      e^{\rho(-s-\frac {l_\nu} 2)}), \quad 
      \forall s \in \left[-\tfrac {l_\nu} 2 +l,\tfrac {l_\nu} 2-l \right].
    \end{equation}
  \end{enumerate}
\end{proposition}
\begin{lemma}
  \label{lem:ba1}
  Let $l_\nu \to \infty$ be a sequence. Suppose
  $u_\nu : [0,l_\nu] \times S^1 \to X^\nu$ is a sequence of maps that
  converges horizontally in $P_0 \in \PP$, and for which there is a
  sequence of translations $t_\nu \in \nu P_0^\dual$ such that
  $\e^{-t_\nu}u_\nu$ converges uniformly on compact subsets to a map
  $u:\R_{\geq 0} \times S^1 \to \XC_{P_0}$, whose projection
  $\pi_{P_0} \circ u$ has a removable singularity at the infinite end
  with $u(\infty) \in \XB_{P_1} \subset \ol X_{P_0}$ for some
  $P_1 \subseteq P_0$. Then, there is a subsequence $r_\nu \to \infty$
  such that
  \[u_\nu(r_\nu,\cdot) : S^1 \to X^\nu\]
  converges horizontally in $P_1$.
\end{lemma}

\begin{proof}
  The removal of singularity for $u$ implies that there exists a point
  $x_0$ in the $P_1$-cylindrical end of $\XC_{P_0}$ and an element
  $\mu \in \t_{P_1,\Z} \bs (\cup_{P \supset P_1}\t_P)$
  \cwl{resp. $\t_{P_1}$} such that
  \[d(u(s,t), e^{\mu(-s - it)} x_0) \leq ce ^{-|s|} \quad (s,t) \in
    \R_{\geq 0} \times S^1,\] 
  see Remarks \ref{rem:extend} and \ref{rem:expconv}. 
    \cwl{resp. $[0,1]$} 
    Then, there is a sequence $r_\nu \to \infty$ and translations
    $t_\nu \in \nu P_0^\dual$ such that
    \[\d(u_\nu(r_\nu, \cdot), \e^{t_\nu}e^{\mu(-r_\nu - it)} x_0) \to
      0 \quad \text{ as }\nu \to \infty,\]
    where $\e^{t_\nu}$ maps a subset of $\XC_{P_0}$ to $X^\nu_{\tP_0}$
    (see \eqref{eq:transinc2}). It follows that $u_\nu(r_\nu,\cdot)$
    horizontally converges to the constant map $\pi_{P_1}(x_0)$ in
    $\XB_{P_1}$. (Note that the values of $t_\nu$ are not relevant to
    the conclusion of horizontal convergence since
    $t_\nu \in \nu P_0^\dual \subset i\t_{P_0}$, and for the purpose
    of horizontal convergence in $X_{P_1}$, we quotient the map by
    $T_{P_1,\C}$ which contains $T_{P_0,\C}$ as a subtorus.)
\end{proof}

\begin{lemma}
  \label{lem:ba2}
  Suppose $l_\nu \to \infty$ and
  $u_\nu : [0,l_\nu] \times S^1 \to X^\nu_{\tP} \subset X^\nu$ is a
  sequence of maps satisfying hypotheses \ref{ba-hyp1}, \ref{ba-hyp2}
  of Proposition \ref{prop:breaking-ann}.  Furthermore, assume that
  $\pi_P \circ u_\nu$ converges on compact subsets to a constant map
  $u:\R_{\geq 0} \times S^1 \to \XB_P$ with value $x_0 \in
  \XB_P$. Then, there is a compact subset $K \subset \XB_P$ that
  contains the image of $(\pi_P \circ u_\nu)$ for all
  $\nu$.\footnote{For any $\nu$, there is a natural inclusion
    $X^\nu_P \to \XB_P$, see Lemma \ref{lem:natinc} and Definition
    \ref{def:brokenJ}.}
\end{lemma}

\begin{proof}
  The Lemma is proved by applying the annulus lemma (Proposition \ref{prop:ann-reg})
  on a compact subset of $\XB_P$.  \cwl{(resp.
    $[0,l_\nu] \times [0,1]$ in the case of Lagrangian boundary
    conditions)}    The intervals
  \begin{equation*}
    I_\nu:=\d_{\nu B^\dual}(\pi_{\nu B^\dual}(u_\nu([0,l_\nu] \times S^1)), \nu P^\dual) \subseteq [0,\infty]
  \end{equation*}
  are connected, and the the conclusion in the Lemma is equivalent to
  \begin{equation}
    \label{eq:goaldnu}
    \sup_\nu\on{length}(I_\nu)<\infty. 
  \end{equation}
  Let
  \begin{equation}
    \label{eq:duminusp}
    \tau:= \d(\pi_{\NCone_{P^\dual}B^\dual} \circ u, P^\dual) = \d(\pi_{\NCone_{P^\dual}B^\dual}(x_0), P^\dual),
  \end{equation}
  so that $\tau \in I_\nu$ for all $\nu$.  Let us assume that
  \eqref{eq:goaldnu} does not hold.  Consider any $\kappa>0$.  If
  \eqref{eq:goaldnu} does not hold, after passing to a subsequence,
  \begin{equation*}
    \tau + \kappa \in I_\nu \quad \forall \nu.
  \end{equation*}
  For each $\nu$, let $s_\nu \in [0,l_\nu]$ be the least value for
  which there is a point $z_\nu=(s_\nu,\theta_\nu)$ in the domain of $u_\nu$ such
  that
  \begin{equation}
    \label{eq:snuchoose}
    d_{\nu B^\dual}(\pi_{\nu B^\dual}(u_\nu(z_\nu)), \nu P^\dual)=
    \tau+\kappa.
  \end{equation}
  The minimality of $s_\nu$ implies that
  \begin{equation}
    \label{eq:snumin}
    d_{\nu B^\dual}((\pi_{\nu B^\dual} \circ u_\nu)([0,s_\nu] \times S^1), \nu P^\dual)
    \leq \tau+\kappa. 
  \end{equation}
  Because the limit $u$ is at a distance of $\tau$ from $P^\dual$, we
  have $s_\nu \to \infty$.  We also have $l_\nu-s_\nu \to \infty$,
  otherwise, \eqref{eq:snumin} together with the uniform bound on
  $\Mod{du_\nu}_{L^\infty}$ implies that the intervals $I_\nu$ are
  uniformly bounded.  We rescale the domain centered at $s_\nu$.  By
  the uniform bound on the derivatives $|\d u_\nu|$, after passing to a
  subsequence, the sequence $\pi_P \circ u_\nu(\cdot - s_\nu)$
  converges uniformly on compact subsets to a limit
  $v : \R \times S^1 \to \XB_P$.  By the bubble-free condition, $v$ is
  a constant map whose value is, say, $x_1 \in \XB_P$.  We now arrive
  at a contradiction by applying the annulus Lemma \ref{prop:ann-reg}
  on compact symplectic manifolds to the sequence
  $\pi_P \circ u_\nu|[0, s_\nu]$.
  By \eqref{eq:snumin}, there is a compact set $K_0 \subset \XB_P$
  that contains the images of
  $(\pi_P \circ u_\nu)([0,s_\nu] \times S^1)$.  By Lemma
  \ref{lem:Knondeg}, there exists a squashed area form $\om_\aleph$ on
  $\XB_P$ that is a symplectic form on $K_0$.  Proposition
  \ref{prop:thinconvcpt} on $(K_0,\om_\aleph)$ implies that the
  sequence $(\pi_P \circ u_\nu)|([0,s_\nu] \times S^1)$ \footnote{The
    image of $(\pi_P \circ u_\nu)$ lies in $X_P^\nu$, which is
    naturally embedded in $X_P$ by Definition \ref{def:brokenJ}
    \eqref{part:cutspace}.}  \cwl{resp. $[0,1]$} converges to a pair
  of disks $(u, u')$ connected at a node. The second disk $u'$ is the
  same as $v$ which is a constant $x_1$.  Furthermore, $x_1 \neq x_0$
  because $d(\pi_{\NCone_{P^\dual}B^\dual}(x_0), P^\dual)=\tau$, and
  \eqref{eq:snuchoose} implies that
  $d(\pi_{\NCone_{P^\dual}B^\dual}(x_1), P^\dual)=\tau + \kappa$,
  resulting in a contradiction.
\end{proof}

\begin{proof}
  [Proof of Proposition \ref{prop:breaking-ann}]
  In Step 1 we will find a polytope $P \subset P_+ \cap P_-$, and in subsequent steps we 
  will prove the assertions of the Proposition with $P$ playing
  the role of $P_\cap$. In one of the final steps, we will show that
  $P$ is equal to $P_\cap$.
  \vskip .1in \noindent 
  \textsc{Step 1} : \em{There
    is a polytope $P \subset P_+ \cap P_-$ such that
    $(\pi_{P_\pm} \circ u_+)(\mp \infty) \in \XB_P$.  Furthermore,
    there exists a constant $L\geq 0$ such that
    \begin{equation}
      \label{eq:bapP}
    u_\nu(A(l_\nu -2L)) \subset X_{\tP}^\nu ,
  \end{equation}
  and a compact set $K \subset \XB_P$ that contains the image
  $\pi_Pu_\nu(A(l_\nu -L))$ for all $\nu$.} \footnote{For any $\nu$,
  there is a natural inclusion $X^\nu_P \to \XB_P$, see Lemma
  \ref{lem:natinc} and Definition \ref{def:brokenJ}.}

Let $P \subset P_-$ be the polytope for which 
\begin{equation}
  \label{eq:bap1}
  (\pi_{P_-} \circ u_-)(\infty) \in \XB_P
  \subset \ol X_{P_-}.
\end{equation}
The convergence of $\pi_{P_-} \circ u_-(\cdot + \frac {l_\nu} 2)$ to
$\pi_{P_-}(u_-)$, together with \eqref{eq:bap1} implies that there
exists $L_- \geq 0$ and a sequence $s_\nu$ such that
 \begin{equation}
   \label{eq:bap2}
   u_\nu([-\tfrac {l_\nu} 2 + L_-, s_\nu] \times S^1) \subset
   X^\nu_{\tP},
 \end{equation}
 %
 and satisfying $s_\nu + \frac {l_\nu} 2 \to \infty$.
 We assume that for any $\nu$, $s_\nu$ is the largest value which
 satisfies \eqref{eq:bap2}. Also, Lemma \ref{lem:ba1} implies that
 there is a sequence $r_\nu$ such that
 \begin{equation}
   \label{eq:bap3}
   u_\nu(r_\nu,\cdot) \text{ horizontally converges in $P$}, \quad r_\nu + \tfrac {l_\nu} 2 \to \infty.
 \end{equation}
 Automatically $r_\nu<s_\nu$.
 %
 \begin{claim}
   The sequence $\frac {l_\nu} 2 - s_\nu$ is bounded.
 \end{claim}
 \begin{subproof}
   Suppose $\frac {l_\nu} 2 - s_\nu \to \infty$. Then, the rescaled
   sequence
   $u_\nu(\cdot - s_\nu)$,
   after passing to a subsequence,
   converges horizontally in $P_1$ to a limit
   $v: \R \times S^1 \to \XB_{P_1}$.
   The maximality of $s_\nu$ in \eqref{eq:bap2} implies that $P_1 \supset P$.
   The bubble-free condition implies
   that $v$ is a constant map. Lemma \ref{lem:ba2} then implies that
   there is a compact set $K \subset \XB_{P_1}$ containing the images
   $(\pi_{P_1} \circ u_\nu)([-\frac {l_\nu} 2 + L_-, s_\nu] \times
   S^1)$.  But since, $P \subsetneq P_1$, this contradicts
   \eqref{eq:bap3} leading to the proof of the Claim.
 \end{subproof}
Since $u_\nu(s_\nu,\cdot)$ converges horizontally in $P_1 \in \PP$, the Claim implies that $P_1=P_+$. Taking $L:=\max\{L_-,\frac {l_\nu} 2-s_\nu\}$,
we obtain
\[u_\nu(A(l_\nu-2L)) \subset X^\nu_{\tP}. \]
The sequence of maps $\pi_P \circ u_\nu$ converges to a limit
$v_0 : \R \times S^1 \to \XB_P$, which by the bubble-free condition is
a constant map.  Lemma \ref{lem:ba2} applied to both
$u_\nu|[0,\frac {l_\nu} 2-L]$ and $u_\nu|[0,-(\frac {l_\nu}2-L)]$ shows that there is a
compact subset $K \subset \XB_P$ that contains the images of
$\pi_P \circ u_\nu(A(l_\nu -2L))$.

Finally, since for any sequence
$r_\nu' \in [-\frac {l_\nu} 2 +L, \frac {l_\nu} 2 -L]$, the maps
$u_\nu(r_\nu',\cdot)$ horizontally converge in $\XB_P$, Lemma
\ref{lem:ba1} implies that $(\pi_{P_+} \circ u_+)(-\infty) \in \XB_P$.

\vskip .1in \noindent \textsc{Step 2} : \em{Determining the edge {direction} $\mu$.}\\
In this step we read off the {direction} $\mu$ of the edge from the topology
of the cylinders \cwl{or strips} $u_\nu$ and show that it is equal to
the {direction}s of the limit maps at the nodal point.  By Step 1, the
images $\pi_P(u_\nu(A(l_\nu-2L)))$ are contained in a compact subset
$K$ of $X_P$.  Choose a squashing area form $\om_\aleph$ on $X_P$ that
is a symplectic form on $K$; such a form exists by Lemma
\ref{lem:Knondeg}.  The $\om_\aleph$-area of $\pi_P \circ u_\nu$ is
uniformly bounded, since
  \[\int_{A(l_\nu -2L)}(\pi_P \circ u_\nu)^*\om_\aleph \leq E_\Hof(\pi_P \circ u_\nu) \leq E_\Hof(u_\nu),\]
  where the last inequality is by Proposition \ref{prop:quothof}.  By
  the annulus lemma for compact symplectic manifolds (Proposition
  \ref{prop:thinconvcpt}), the annuli $\pi_P \circ u_\nu$ converge to
  a pair of disks connected by an interior \cwl{or boundary} node
  $x_0 \in \XB_P$.  Therefore, for some $L_1>0$, the images
  $\pi_P(u_\nu(A(l_\nu-L_1)))$ lie in a neighborhood
  $B_{\eps}(x_0) \subset \XB_P$. Consider a trivialization
  \[B_{\eps}(x_0) \times T_{P,\C} \simeq \pi_P^{-1}(B_{\eps}(x_0))
    \subset \XC_P.\]
  of $\pi_P : \XC_P \to \XB_P$.  Viewing the target space of $u_\nu$
  as the product $B_\eps \times T_{P,\C}$, the homotopy class
  $(u_\nu)_*[A(l_\nu)] \in \pi_1(T_P)$ corresponds to an element
  $\mu_\nu \in \t_{P,\Z}$. \cwl{or $\t_P$} Let $\mu \in \t_{P,\Z}$ be
  defined so that \cwl{resp. $\t_P$} $u_-$ is asymptotically close to
  a trivial cylinder of {direction} $\mu$ at $\infty$ in the sense of Remark
  \ref{rem:expconv}.  Therefore $\mu_\nu=\mu$ for large
  $\nu$. Similarly $u_+$ is also asymptotically close to a trivial
  cylinder of {direction} $\mu$ at $-\infty$.

  \vskip .1in \noindent \textsc{Step 3}: \em{The sequence of twisted
    maps
    \[\ol u_\nu:A(l_\nu) \to \XC_P, \quad (s,t) \mapsto
      e^{-\mu(s+\frac {l_\nu} 2 + it)} (\e^{-t_\nu^-}u_\nu)
\footnote{\text{By \eqref{eq:transinc}, there is an inclusion $\e^{-t_\nu^-}: X^\nu_{\tP} \to \XX_P$.} }
    \]
    converges to a pair of disks connected at an interior point.}\\

  Because the derivatives of the maps $(\ol u_\nu)_\nu$ are uniformly
  bounded, we expect the Gromov-Floer limit of the sequence to consist
  of two disks connected by a path of spheres.  \cwl{resp. disks in
    the case of Lagrangian boundary condition} Each of the components
  is obtained by a sequence of rescalings of the cylinder
  \cwl{resp. strip} of the form $(s,t) \mapsto (s+ s_\nu, t)$ for some
  sequence $s_\nu \in \R$. The bubble-free condition rules out spheres
  \cwl{or disks} in the limit, leading to the conclusion of Step 3.
  The proof is standard, except for the fact that we need to choose an
  appropriate notion of symplectic area.

  The first component in this Gromov-Floer limit of $\ol u_\nu$ is a
  disk $\ol u_-$ which is a twisted version of $u_-$.  This is seen as
  follows.  The convergence of $\e^{-t _\nu^-}u_\nu$ to $u_-$ implies
  that the sequence $\ol u_\nu(\cdot - \frac {l_\nu} 2)$ converges to
  \begin{equation}
    \label{eq:oluminus}
    \ol u_-:=e^{-\mu(s+it)}u_-.  
  \end{equation}
  The image of $\ol u_-$ is compact in $\XC_P$ since $u_-$ is
  asymptotically close to the $\mu$-cylinder
  $(s,t) \mapsto e^{\mu(s+it)}$. Since $\pi_P\circ u_-$ extends over
  $\infty$, the same is also true for $\ol u_-$. We denote
  $\ol x_0:=\ol u_-(\infty) \in \XX_P$.

  Next, we finish the description of the Gromov-Floer limit.  Choose
  any taming symplectic form $\om_{\tP}$ defined in a neighborhood
  $U'_{\ol x_0} \subset \XC_P$ of $\ol x_0$.  Let
  $U_{\ol x_0} \Subset U'_{\ol x_0}$ be a smaller open neighborhood
  whose closure is contained in $U'_{\ol x_0}$.  Let $\kappa>0$ be any
  constant.  Since $\ol u_\nu(\cdot - \frac {l_\nu} 2)$ converges to
  $\ol u_-$, there exists a constant $r_0$, and a sequence $r_\nu$
  such that
  \begin{multline}
    \label{eq:rnudef}
    \ol u_\nu([\tfrac {-l_\nu} 2 + r_0, r_\nu]
    \times S^1) \subset \ol U_{\ol x_0} \\
    \text{and }\om_{\tP}(u_\nu,[\tfrac {-l_\nu} 2 + r_0,r_\nu]
    \times S^1) \leq \om_{\tP}(\ol u_-,[r_0,\infty) \times S^1) +
    \kappa.
  \end{multline}
  We assume that for each $\nu$, $r_\nu$ is the maximum value
  satisfying the above condition, and therefore
  $r_\nu + \frac {l_\nu} 2 \to \infty$.  Applying the annulus Lemma
  for compact manifolds (Proposition \ref{prop:thinconvcpt}) to the
  sequence of maps
  $u_\nu|([\tfrac {-l_\nu} 2 + r_0, r_\nu] \times S^1) $ with target
  space $(\ol U_{\ol x_0},\om_{\tP})$, we conclude that the sequence
  of annuli converges to a pair of disks $(\ol u_-, \ol u_+)$. Here, we
  note that the first map is the same as the map in
  \eqref{eq:oluminus}. The second map
  \[\ol u_+ : (-\infty,L_1) \times S^1 \to \XC_P\]
  is the limit of rescaled maps
  \[\ol u_\nu^+(s,t):= \ol u_\nu(\cdot + r_\nu) : [\tfrac {-l_\nu} 2
    -r_\nu, \tfrac {l_\nu} 2 - r_\nu] \times S^1 \to \XC_P, \]
  where 
  \begin{equation}
    \label{eq:L1def}
    L_1:=\lim_{\nu}(\tfrac {l_\nu} 2 - r_\nu)  \in (0,\infty].
  \end{equation}
  Proposition \ref{prop:thinconvcpt} also implies that the images of
  the components $\ol u_-$ and $\ol u_+$ connect at the nodal point,
  that is, $\ol u_-(\infty)=\ol u_+(-\infty)$.  Finally, $L_1$ is
  finite, because otherwise, the sequence $\ol u_\nu(\cdot + r_\nu)$
  converges to a sphere \cwl{or disk} in $\XC_P$, contradicting the
  bubble-free assumption.

  We have shown that the limit of the twisted maps is a pair of disks
  $(\ol u_-, \ol u_+)$.  Since
  $L_1=\lim_\nu (\frac {l_\nu} 2 - r_\nu)$ is finite, after truncating
  the domain cylinders by a $\nu$-independent amount, we may replace
  $r_\nu$ by $\frac {l_\nu} 2$.  The limit $\ol u_+$ will be altered
  by a domain reparametrization and $\ol u_+(-\infty)=\ol u_-(\infty)$
  continues to hold.

  Part \eqref{part:breaking-anna} of the Proposition is now proved as
  follows.  Both sequences of maps
  \[\ol u_\nu^+(s,t):=e^{-\mu(s+it)}(\e^{-t_\nu^- - \mu
      l_\nu}u_\nu(s+\tfrac {l_\nu} 2,t)) \]
  and $\e^{-t_\nu^+}u_\nu(\cdot + \frac {l_\nu} 2)$ converge on
  $\R_{\geq 0} \times S^1$, the former by our proof and the latter by
  the hypothesis.  At the point $(s,t)=(0,0)$ the sequences
  $\ol u_\nu^+(s,t)$ and $\e^{-t_\nu^+}u_\nu(\cdot + \frac {l_\nu} 2)$
  differ by a translation by $e^{t_\nu^+ - t_\nu^- - \mu
    l_\nu}$. Since both sequences of points converge, we conclude that
  the limit
  \[\delta:=\lim_\nu(-\mu l_\nu - t_\nu^- + t_\nu^+)\]
  exists (which proves \eqref{part:breaking-anna}) and that
  \[\ol u_+(s,t)=e^\delta e^{-\mu(s+it)}u_+(s,t).\]

  \vskip .1in \noindent \textsc{Step 4} : \em{Proof of the decay estimate.}\\
  We have shown that after truncating the domain cylinders by a
  $\nu$-independent amount, say $L$, the sequence of the twisted maps
  $\ol u_\nu$ converges to a pair of disks $(\ol u_-,\ol u_+)$, and the
  images of the maps lie in a compact set $\ol U_{\ol x_0}$ with a
  taming symplectic form $\om_{\tP}$.  The sequence of domains can
  truncated again by a finite amount to ensure that
  $\om_{\tP}(\ol u_\nu)< \hbar$, since the $\om_{\tP}$-area on the
  sequence of cylinders converges to the $\om_{\tP}$-area of the pair
  of disks $(\ol u_-,\ol u_+)$.  We apply the annulus lemma for
  compact manifolds (Proposition \ref{prop:ann-reg}) to the maps
  %
  \[\ol u_\nu : A(l_\nu -2L) \to (\ol U_{\ol x_0},\om_{\tP}) \subset
    \XC_P.\]
  The decay estimate for the twisted maps $\ol u_\nu$ implies the
  asymptotic decay estimate \eqref{eq:ann-est} for $u_\nu$ required by
  the Proposition.

  \vskip .1in \noindent \textsc{Step 5} : \em{Proof of Horizontal Matching and $P=P_+ \cap P_-$.}\\
  By Step 1, the nodal lifts are mapped to on $\XB_P$, that is,
  \[(\pi_{P_\pm} u_\pm)(\mp \infty) \in \XB_P \subset
    \ol X_{P_\pm}.\]
  Therefore at the nodal point, $u_\pm$ intersects all the divisors
  $\ol X_{Q_\pm}$ of $\ol X_{P_\pm}$ which contain $\ol X_P$, that is,
  $P \subseteq Q_\pm \subseteq P_\pm$. Consequently,
  \[\mu \in \t_P \bs \cup_{Q \supset P}\t_Q.\]
  On the other hand $t_\nu^+ - t_\nu^- \in \t_{P_\cap}$, where
  $P_\cap:=P_+ \cap P_-$.  Since $t_\nu^+ - t_\nu^- - \mu l_\nu$ is
  uniformly bounded, we conclude $\mu \in \t_{P_\cap}$.  Therefore
  $P_\cap=P$. Horizontal matching now follows from Step 2, where we
  showed that the sequence $(\pi_P \circ u_\nu)$ converges in $X_P$ to
  pair of annuli connected at a node. This finishes the proof of the
  breaking annulus lemma (Proposition \ref{prop:breaking-ann}).
\end{proof}

The proof of the breaking annulus lemma was based on the following results (Propositions \ref{prop:ann-reg} and \ref{prop:thinconvcpt}) on compact symplectic manifolds. 
\begin{proposition}\label{prop:ann-reg} {\rm(Annulus lemma on compact
    manifolds \cite[Lemma 4.7.3]{ms:jh})}
  Suppose $(X,\om_X)$ is a compact symplectic manifold with a tamed
  almost complex structure $J$.  
  There 
  exist
  constants $0<\rho<1$,
  $\hbar >0$, $c>0$ such that the following holds for any
  $J$-holomorphic map $u: A(\ell) \to X$ with $E(u) \leq \hbar$.  For
  $x=u(0,0)$, there is a map
\[ \xi:A(\ell-1) \to T_xX \ \text{ such that }  \ u=\exp_x\xi  \]  
on $A(\ell-1)$ and
  \begin{equation}
    \label{eq:anndist}
    \abs{\xi(s,t)} \leq c(e^{\rho(s-\ell)} + e^{\rho(-s-\ell)}), 
\quad 
  \forall s \in [-\ell +1,\ell-1]. 
  \end{equation}
  \cwl{The same statement holds in the case that $A(\ell)$ is replaced by $[-\ell,\ell] 
  \times [0,1]$ and $u$
  takes Lagrangian boundary conditions on $[-L,L] \times \{ 0,1 \}$   .}
 The constants $\rho$, $\hbar$, $c$ depend continuously on $J$ with respect to the $C^2$-topology. 
\end{proposition}
\begin{proposition} \label{prop:thinconvcpt} {\rm(Convergence of long
    cylinders)} Suppose $(X,\om_X)$ is a compact symplectic manifold
  with a tamed almost complex structure $J$.  Let
  $u_\nu : A(l_\nu) \to X$ be a sequence of holomorphic cylinders
  \cwl{or strips} with a uniform bound on $\Mod{du_\nu}_{L^\infty}$
  and the $\om$-areas of $u_\nu$.  Furthermore, the sequence
  $\{u_\nu\}_\nu$ satisfies the bubble-free condition: For any
  sequence $r_\nu$ with $|r_\nu \pm l_\nu/2| \to \infty$, if a
  subsequence
  of $\{u_\nu(\cdot - r_\nu)\}_\nu$ converges to a limit
  $v : \R \times S^1 \to X$, then $v$ is a constant map.  After
  passing to a subsequence, $u_\nu(\cdot \pm \frac {l_\nu} 2)$
  converges in $C^\infty_\loc$ to
  \[u_\pm : \R_\mp \times S^1 \to X, \]
  the map $u_\pm$ extends holomorphically over $\mp \infty$, and
  $u_-(\infty)=u_+(-\infty)$.
\end{proposition}
\begin{proof}
  This Proposition is proved as part of the ``bubbles connect'' result
  in \cite[Proposition 4.7.1]{ms:jh}.
\end{proof}

We give an annulus lemma for sequences of holomorphic strips in
$X^\nu$ whose boundaries lie on the Lagrangian submanifold $L$. The
result is simpler than the breaking annulus lemma since the strips
converge to a nodal disk in the complement of the relative divisors,
and not a tropical node. The result is stated for maps whose domains
are strips, defined as
\[A_\white(\ell):=[-\tfrac \ell 2 , \tfrac \ell 2] \times [0,1]\]
for any $\ell>0$. The result is a boundary version of the ``bubbles
connect'' result in McDuff-Salamon \cite[Proposition 4.7.1]{ms:jh}.
However, since the manifold $X_{P_0}$ has cylindrical ends, we need to
use a Hofer energy bound.
\begin{proposition}{\rm(Annulus lemma with boundary)}
  \label{prop:annbdry}
  Suppose that the almost complex structures $\JJ_0 \in \J^\cyl(\XX)$
  and $J_0^\nu \in \J^\cyl(X^\nu)$ satisfy the conditions in the
  statement of Proposition \ref{prop:breaking-ann}. There are
  constants $0<\rho<1$, $c>0$ such that the following hold.  For a
  sequence $l_\nu \to \infty$, let
  \[ u_\nu: A_\white(l_\nu) \to X^\nu, \quad u_\nu(\partial A_\nu) \subset L \]
  be a sequence of $J_0^\nu$-holomorphic strips satisfying the
  following.
  \begin{enumerate}[label*=(\arabic*)]
  \item $\sup_\nu E_{\Hof}(u_\nu) <\infty$,
    $\sup_{z \in \R \times [0,1], \nu} |\d u_\nu(z)|<\infty.$
  \item The sequence $\{u_\nu\}_\nu$ satisfies a disk bubble-free
    condition: For any sequence $r_\nu$ with
    $|r_\nu \pm l_\nu/2| \to \infty$, if a subsequence
  of $\{u_\nu(\cdot - r_\nu)\}_\nu$ converges to a limit $v : \R \times [0,1] \to X$, then $v$ is a constant map. 
\item The sequence of maps $u_\nu(\cdot \pm \frac {l_\nu} 2)$
  converges in $C^\infty_\loc$ to a limit
    \[u_\pm : \R_\mp \times [0,1] \to \XB_{P_0},\]
    and the map
    extends holomorphically over $\mp \infty$.
  \end{enumerate}
  Then, $u_+(-\infty)=u_-(+\infty) \in L$.
\end{proposition}
\begin{proof}
  There is a compact subset $K \subset X_{P_0}$ that contains the
  image all the strips $\{u_\nu\}_\nu$. Indeed, $L \subset X_{P_0}$ is
  compact and the images of the strips are contained in a
  $\sup_\nu\Mod{du_\nu}_{L^\infty}$-radius of $L$.  By Lemma
  \ref{lem:Knondeg}, there is a squashing area form $\om_\aleph$ that
  is a symplectic form on $K$. The $\om_\aleph$-area of the strips
  $u_\nu$ are uniformly bounded. Therefore, the target space may be
  seen as a compact symplectic manifold, and the proof is analogous to
  the ``bubbles connect'' result for long cylinders in
  \cite[Proposition 4.7.1]{ms:jh}.
\end{proof}

\section{Proof of convergence for breaking maps}
\label{sec:breakingconv} 

The arguments used in the proof of convergence of breaking maps are
similar to those used in the proof of Gromov compactness for stable
maps. The new features include a more refined analysis of sequences of
long cylinders that converge to nodes of the limit broken maps. The
decay estimate on such cylinders coming from the breaking annulus
lemma (Proposition \ref{prop:breaking-ann}) allows us to prove
\hyperref[item:thincyl2]{(Thin cylinder convergence)}, and also prove
the existence of a tropical graph, that is, the graph underlying the
limit map is realizable (see Definition
\ref{def:tropgraph}) in the dual complex.  The stabilizing divisor
plays a role in the proof, since, the choice of perturbation data
implies that any non-constant limit component has a stable
domain. \label{page:cpt-intro}

\begin{proof}[Proof of Theorem \ref{thm:cpt-breaking}]
  \vskip .1in \noindent \textsc{Step 1} : \em{Domain components for
    the limit map.} \\ The sequence of domains $(C_\nu,z_\nu)$
  converges to a stable treed nodal curve $(C,z)$ modelled on a tree
  $\Gamma$.  The perturbation maps $\ul{\Pe}_\nu$ converge to a
  perturbation datum $\ul{\Pe}_\infty=(J_\infty,F_\infty)$ defined on
  $C$. In the next few steps, we will show that there are no
  additional domain components in the limit map.

  In the rest of the proof, we focus on the convergence of maps on
  irreducible surface components. The proof of convergence of gradient
  trajectories on treed components does not involve any technical
  difficulties arising from multiple cutting, and are left to the
  reader.
  
  \vskip .1in \noindent \textsc{Step 2} : \em{Boundedness of
    derivatives.}\\
  In this step, we show a bound on the derivatives of the sequence of
  maps by ruling out bubble trees with unstable domains in the limit.
  In particular, we will show that for any $v \in \Ver(\Gamma)$, the
  derivative of
  \[u_{v,\nu}:=u_\nu \circ i_{v,\nu}^{-1} :C_v^\circ \to X^\nu\]
  is uniformly bounded for all $\nu$. The norm on the derivative is
  with respect to the cylindrical metric on $X^\nu$. On the domain we
  use a metric on $C^\circ$ that is cylindrical (or strip-like) in the
  punctured neighborhood of nodal points and on the neck regions in
  $C_\nu$.

  Assume by way of contradiction that the derivatives are not
  uniformly bounded.  We will construct a sequence of rescaled maps
  which converges to a constant map with positive derivative, which
  cannot exist.  After passing to a subsequence, there exists a
  sequence of points $z_\nu \in C_v^\circ$ and a point
  $z_\infty \in C_v$ for some $v \in \Ver(\Gamma)$ such that
  \[ z_\nu \to z_\infty, \quad |\d u_{v,\nu}(z_\nu)| \to \infty .\]
  We first carry out the proof assuming that $z_\infty$ is not a nodal
  point and does not lie on the boundary $\partial C_v$, and thus
  $z_\infty \in C_v^\circ$. We apply Hofer's lemma \ref{lem:hof} to
  the function $|\d u_{v,\nu}|$, $x=z_\nu$ and the constant
  $\delta=|\d u_\nu(z_\nu)|^{-1/2}$.
  \begin{lemma}
    \label{lem:hof} {\rm(Hofer's Lemma, \cite[Lemma 4.6.4]{ms:jh})}
    Suppose $(X,d)$ is a metric space, $f:X \to \R_{\geq 0}$ is a
    continuous function, and $x \in X$, $\delta>0$ are such that the
    ball $B_{2\delta}(x)$ is complete. Then there exists a positive
    constant $\eps \leq \delta$ and a point $\zeta \in B_{2\eps}(x)$
    such that
    \begin{equation*}
      \sup_{z \in B_\eps}f(z) \leq 2 f(\zeta), \quad \eps f(\zeta) \leq \delta f(x). 
    \end{equation*}
  \end{lemma}
  \noindent We obtain another sequence $\zeta_\nu \in C_v^\circ$
  converging to $z_\infty$, and a sequence of constants
  $\eps_\nu \to 0$ such that
  \[c_\nu:=|\d u_{v,\nu}(\zeta_\nu)| \to \infty, \quad \sup_{z \in
      B_{\eps_\nu}}|\d u_\nu(z)| \leq 2c_\nu, \quad c_\nu \eps_\nu \to
    \infty. \]
  The rescaled maps are
  \[\tiu_\nu:=u_{v,\nu}((\cdot -\zeta_\nu)/c_\nu):B_{\eps_\nu c_\nu}
    \to X^\nu,\]
  and $|\d\tiu_\nu| \leq 2$, $|\d\tiu_\nu(0)|=1$.

  Each of the rescaled maps converges to a limiting map with domain
  the affine line as follows.  By Lemma \ref{lem:cptconv} there is a
  polytope $P \in \PP$ and a sequence of translations
  $t_\nu \in t_P^\dual$ so that after passing to a subsequence, the
  maps $\e^{-t_\nu}\tiu_\nu$ converge in $C^\infty_\loc$ to a
  non-constant $J_{z_\infty}$-holomorphic limit $\tiu:\C \to \XC_P$.
  
  The limit map $\tiu:\C \to \XC_P$ in the previous paragraph is
  asymptotic to a trivial cylinder at infinity: Hofer energy
  $E_{\Hof}(u_\nu, C_\nu)$ is equal to area (Remark
  \ref{rem:hof-area}) and therefore is uniformly bounded for all
  $\nu$.  By monotonicity of Hofer energy (Lemma \ref{lem:monot}), the
  Hofer energy on a ball $E_{\Hof}(\tiu_\nu,B_{\eps_\nu c_\nu})$ is
  uniformly bounded.  By Proposition \ref{prop:hofer-breaking},
  $E_\Hof$ is preserved in the neck-stretching limit (in a certain
  sense), which implies that the quantity $E^*_{P,\Hof}(\tiu,\C)$ is
  finite.  By Proposition \ref{prop:remsing}, there exists
  $Q \subseteq P$ and $\mu \in \t_Q$ such that at $\infty$, $\tiu$ is
  asymptotically close to $z \mapsto z^\mu x_0$ for some
  $x_0 \in \XX_Q$.

  Finally, we arrive at a contradiction by showing that the limit map
  is constant. The domain reparametrizations $\phi_{v,\nu}$ were
  derived from the stable map compactification. %
  It follows from the definition of the compactification there is at
  most a single marked point $z_{i,\nu}$ that is contained in each of
  the regions $B_{\eps_\nu}(\zeta_\nu)$. Therefore, the projection
  $\pi_P \circ \tiu : \C \to \XB_P$ either lies in the stabilizing
  divisor $D_P$, or it has at most one intersection with the divisor
  $D_P$. Neither
  possibility can happen for a non-constant map $\pi_P \circ \tiu$, 
  since $(\JJ_0,\bD)$ is a stabilizing pair for $\XX$, and the
  perturbation $\ul{\Pe}$ is adapted to $(\JJ_0,\bD)$.
  Therefore, $\pi_P \circ \tiu$ is constant, and so the
  image of $\tiu$ lies in a single toric fiber $V_{P^\dual}$. The
  image $\tiu(\C)$ does not intersect torus-invariant divisors of
  $V_{P^\dual}$, and therefore $\tiu$ is a constant map.

  We now consider the case that the sequence of points with increasing
  derivatives converges to an interior nodal point $w_e$. The sequence
  $z_\nu$ lies on the neck region
  \[ A_\nu:= \left[\tfrac {- l_\nu} 2, \tfrac {l_\nu} 2 \right] \times
    S^1 \subset C_\nu\]
  \cwl{(resp. $\tfrac {- l_\nu} 2, \tfrac {l_\nu} 2] \times [0,1]$ in
    the case of Lagrangian boundary conditions}
  obtained by gluing the
  node $w_e$, and $l_\nu \to \infty$ as $\nu \to \infty$.  Since
  $z_\nu$ converges to the node $w_e$, there is a constant $\eps>0$
  such that the ball $B_\eps(z_\nu)$ is contained in $A_\nu$ for all
  $\nu$.  The proof in the preceding few paragraphs using Hofer's
  Lemma and a rescaling can now be applied to the maps
  $u_\nu|B_\eps(z_\nu)$, because there are no marked points in
  $A_\nu$.  We conclude that $|\d u_{v,\nu}|$ is uniformly bounded for
  all $\nu$.

  Finally, if the sequence of points with increasing derivatives
  converges to a point on the boundary of the domain, then the same
  steps can be repeated to show the existence of a non-constant disk
  bubble with unstable domain, leading to a contradiction. The removal
  of singularities in this case is proved by Proposition
  \ref{prop:remsing-disk}.
    
  \vskip .1in \noindent \textsc{Step 3} : \em{Determining stable components of the limit map.}\\
  In Step 2 we showed that the derivatives of the maps $u_{v,\nu}$ are
  uniformly bounded, where $v$ is any vertex of $\Gamma$.  By Lemma
  \ref{lem:cptconv}, there is a polytope $P(v)$ and a sequence of
  translations $t_\nu(v) \in \nu P(v)^\dual \subset i\t_{P(v)}^\dual$
  such that after passing to a subsequence, the maps
  $\e^{-t_{\nu}(v)}u_{v,\nu}$ converge in $C^\infty_{\loc}$ to a limit
  $u_v: C_v^\circ \to \XC_{P(v)}$.  By arguments as in Step 2,
   the Hofer energy 
   $E_{\Hof}(u_v, C_v^\circ)$ is bounded.  By removal of singularities
   Proposition \ref{prop:remsing},
  \begin{itemize}
  \item $u_v$ has a removable singularity at the node $w$, or
  \item in a punctured neighborhood of the node $w$, $u_v$ is
    asymptotic to a vertical cylinder, and this latter case can only
    happen if $w$ is an interior point of the domain.
  \end{itemize}
  We obtain a $J_\infty$-holomorphic map
  $u_v : C_v^\circ \to \XC_{P(v)}$ that is adapted to the
  stabilizing divisor $\DD_{P(v)} \subset \DD$.

  For later use, we remark that the convergence continues to hold if
  the sequence $t_\nu(v)$ is replaced by a sequence $t_\nu(v)'$ for
  which $\sup_\nu |t_\nu(v)' - t_\nu(v)|<\infty$.  In that case, after
  passing to a subsequence, the limit $u_v$ would be replaced by
  $e^{-t}u_v$, where $t:=\lim_\nu(t_\nu(v)' - t_\nu(v))$.

  \vskip .1in \noindent \textsc{Step 4 :} \em{Determining the
    tropical structure
    and constructing translation sequences.}\\
  We have so far determined the domain treed disk of the limit, and
  the limit map $u_v$ corresponding to each vertex.  
  The construction above determines the polytope $P(v) \in \PP$ for
  each component $v \in \Ver(\Gamma)$ of the limit.  Furthermore, for
  every $v \in \Ver(\Gamma)$ the map $u_v$ is the limit of
  $\e^{-t_\nu(v)}u_\nu$ where $t_\nu(v) \in \nu P^\dual$ is a sequence
  of translations.

  First, we observe that all boundary edges are internal edges, and
  are collapsed by the tropicalization morphism of \eqref{eq:tr}.
  Vertices $v \in \Ver_\white(\Gamma)$ corresponding to disk
  components map to a single vertex $v_\white$ in the tropical graph
  and the translation sequence $t_\nu(v_\white)$ is the point
  $P_0^\dual$.  \cwl{This sentence should be removed in the case of
    Lagrangian boundary condition.}

   Edge {direction}s of the tropical graph are determined using the breaking
   annulus lemma.  For any node $w_e$ of $C$ corresponding to an
   interior node $e=(v_+,v_-) \in \Edge_{\black,-}(\Gamma)$, choose a
   complex coordinate
   \begin{equation}
     \label{eq:prelimzpm}
     z_e^\pm: (U_\pm,w_e) \to (\C,0)  
   \end{equation}
   in a neighborhood $U_\pm \subset C_{v_\pm}$ of the lift $w_e^\pm$
   of the node, and let $A_{e,\nu}:=A(l_\nu(e)) \subset C_\nu$ be a
   sequence of centered annuli converging to the neighborhoods $U_+$,
   $U_-$ of the nodal lifts $w_e^+$, $w_e^-$ (as in Definition
   \ref{def:ann2node}). Since there are no markings $z_{i,\nu}$ on the
   annuli $A_{e,\nu}$, the maps $u_\nu|A_{e,\nu}$ are bubble-free in
   the sense of Definition
   \ref{def:bubble-free}. \label{page:ann-bubb-free} Indeed, if for
   some $P \in \PP$ and some reparametrization of the domain, the
   sequence of projections $\pi_P \circ u_\nu|A_{e,\nu}$ converges to
   a map $v: \R \times S^1 \to X_P$, then $v$ extends to a map to
   $\ol X_P$ defined on an orbifold completion of $\R \times S^1$;
   such a limit map $v$ is constant since it either lies on the
   stabilizing divisor $D_P$ or does not intersect $D_P$ at all.
   Since the sequence of annuli $u_\nu|A_{e,\nu}$ is bubble-free, the
   breaking annulus lemma (Proposition \ref{prop:breaking-ann}) is
   applicable, 
   and we obtain a {direction} $\cT(e) \in \t_{P(e),\Z}$
   corresponding to the edge $e$.  The breaking annulus lemma implies
   that
   \begin{equation}
     \label{eq:approxdirection}
     \sup_\nu\{|t_\nu(v_+)-t_\nu(v_-) - \cT(e) l_\nu(e)|\}<\infty,
   \end{equation}
   which is the \hyperref[item:approxdirection]{(Approximate Direction)} condition in Definition
   \ref{def:approxt}.

   The tropicalization of $\Gamma$ is determined as follows. An edge
   $e \in \Edge_{\black,-}(\Gamma)$ is an internal (non-tropical) edge
   exactly if $\cT(e)$.  In that case, the sequence is
   $|t_\nu(v_+)-t_\nu(v_-)|$ is uniformly bounded for all
   $\nu$. Therefore, we can replace $t_\nu(v_-)$ by $t_\nu(v_+)$ (or
   the other way around), and the convergence of Step 2 still
   holds.  By performing such replacements for all internal edges, we
   may assume $t_\nu(v_+)=t_\nu(v_-)$ for all
   $e=(v_+,v_-) \in \Edge_\internal(\Gamma)$.  Therefore, $t_\nu$
   descends to the tropical graph $\Gamma_\tr$ and is an approximate
   $\Gamma_\tr$-translation sequence.
  
   Next, we obtain translation sequences and determine tropical vertex
   positions. Since $\{t_\nu(v) : v \in \Ver(\Gamma_\tr)\}$ is an
   approximate $\Gamma_\tr$-translation sequence, by Lemma
   \ref{lem:approxpoly}, there is a $\Gamma_\tr$-translation sequence
   $(t_\nu(v)')_{v,\nu}$ such that
   \begin{equation}
     \label{eq:newtclose}
     \sup_{\nu,v}|t_\nu(v)'-t_\nu(v)|  <\infty.
   \end{equation}
   The convergence of maps in Step 2 continues to hold if $t_\nu$ is
   replaced by $t_\nu(v)'$.  Any element of the translation sequence
   gives a tropical vertex position map. Indeed, since
   \[t_\nu'=(t_\nu'(v))_{v \in \Ver(\Gamma_\tr)}, \quad t_\nu'(v)\in
     \nu P(v)^\dual\] satisfies the (Direction) condition
   \eqref{eq:direction},
   \[\Ver(\Gamma_\tr) \ni v \mapsto \tfrac {t_\nu'(v)} \nu\]
   is a tropical vertex position map.
   \vskip .1in \noindent \textsc{Step 5:} \em{Finishing the proof of convergence.}\\
   To finish the proof of convergence it remains to show that the
   collection of limit maps satisfy matching conditions at nodes, and
   that \hyperref[item:thincyl2]{(Thin cylinder convergence)} is
   satisfied at tropical nodes. We first consider disk nodes.  For a
   disk node $w_e \in C$ corresponding to an edge
   $e=(v_+,v_-) \in \Edge_\white(\Gamma)$ with length $\ell(e)=0$, let
   $A_{e,\nu} \subset C_\nu$ be a sequence of strips converging to the
   node $w_e$.  The maps $u_\nu|A_{e,\nu}$ are disk bubble-free since
   there are no markings $z_{i,\nu}$ on $A_{e,\nu}$;
   the reasoning is
   identical to the one used following \eqref{eq:prelimzpm}
   to
   show that on the annuli converging to an interior node, maps are
   bubble-free.  Therefore, the boundary version of the annulus lemma
   (Proposition \ref{prop:annbdry}) is applicable on $u_\nu|C_\nu$ and
   we conclude that node matching holds, that is,
   $u_{v_+}(w_e^+)=u_{v_-}(w_e^-) \in L$.

   Next, we determine matching coordinates for tropical nodes and
   prove node matching using the breaking annulus Lemma.  Matching
   coordinates at a tropical node $w_e$, $e \in \Edge_{\trop}(\Gamma)$
   are obtained by applying a correction to the holomorphic coordinate
   chosen in \eqref{eq:prelimzpm}.  Multiplying a constant to the
   coordinates $z_e^+$, $z_e^-$ has the effect of adding a constant to
   the sequence of neck length parameters
   $l_\nu(e) + i \theta_\nu(e)$.  We recall that as a consequence of
   \eqref{eq:approxdirection} and \eqref{eq:newtclose} the
   $\Gamma$-translation sequence $t_\nu'$ satisfies
   \[\sup_\nu\{|t_\nu'(v_+)-t_\nu'(v_-) - \cT(e)
     l_\nu(e)|\}<\infty. \]
   Since $t_\nu'$ satisfies the (Direction) condition \eqref{eq:direction},
   there is a sequence $l_\nu'(e) \to \infty$ for every
   $e \in \Edge_{\black,-}(\Gamma)$ such that
   \[t_\nu'(v_+)-t_\nu'(v_-) = \cT(e) l_\nu'(e), \quad \text{and}
     \quad \sup_\nu|l_\nu(e)-l_\nu'(e)|<\infty.\]
   Therefore, we can adjust the coordinates $z_e^+$, $z_e^-$ by scalar
   multiplication so that a subsequence of neck length parameters
   $l_\nu(e)+i\theta_\nu(e)$ satisfies
   \begin{equation}
     \label{eq:limneck}
     \lim_\nu(t_\nu'(v_+)-t_\nu'(v_-)-\cT(e) l_\nu(e))=0, \quad
     \lim_\nu\theta_\nu(e)=0.
   \end{equation}
   We will show that $(z_e^+,z_e^-)$ are matching coordinates at the
   node $w_e$ for the broken map $u$.  For interior non-tropical
   nodes, we leave $(z_e^+,z_e^-)$ from \eqref{eq:prelimzpm}
   unchanged, and the following discussion is valid.  To simplify
   calculations, we use logarithmic coordinates in the neighborhood of
   the node
   \[C_\pm \supset U_\pm \bs \{w_e^\pm\} \xrightarrow{\pm \ln
       z_e^\pm} \R_\mp \times S^1. \]
   The annulus $A(l_\nu(e))$ is then identified to the limit curve by
   translations
   \[A(l_\nu(e)) \to U_\pm, \quad (s,t) \mapsto s+it \mp \hh(l_\nu +
     i\theta_\nu). \]
   By Step 2,
   \begin{equation}
     \label{eq:locconv}
     \e^{-t'_\nu(v_\pm)}u_\nu(s+it \pm \hh(l_\nu + i\theta_\nu)) \to u_{v_\pm} \quad \text{in } C^\infty_\loc(U_\pm \bs \{w_e^\pm\}).  
   \end{equation}
   We apply the breaking annulus lemma on the maps $u_\nu$ on the
   annuli $A(l_\nu(e))$.  The resulting decay estimate together with
   \eqref{eq:locconv}, \eqref{eq:limneck} implies the convergence of
   the sequence
   \[\e^{-\hh(t'_\nu(v_+) + t'_{\nu}(v_-))} u_\nu(0,0) \to x_0 \quad
     \text{in } \XC_{P(e)}.\]
   It follows that $u_{v_\pm}$ are asymptotically close to the
   cylinder $(s,t) \mapsto e^{\cT(e)(s+it)}x_0$ in the sense of
   \ref{rem:expconv}. We conclude that $z_e^+$, $z_e^-$ are matching
   coordinates at $w_e$ and \hyperref[item:thincyl2]{(Thin cylinder
     convergence)} is satisfied. Note that we have shown node matching
   for all interior edges, both tropical and internal, in a unified
   way. In case of an internal node $w_e$, $\cT(e)=0$, and we have
   shown that $u_{v_\pm}(\mp \infty)=x_0 \in \XC_{P(v_\pm)}.$

   \vskip .1in \noindent \textsc{Step 6} : \em{Uniqueness of the limit.}\\
   The limit of the domain curves is unique up to reparametrization,
   because the limit is a stable curve. The identifications between
   subsets of $C^\nu$ to the limit curve are unique in the following
   sense (see Remark \ref{rem:uniqueid}) : The neck regions in $C^\nu$
   are parametrized in a unique way, and the difference between any
   two choices of identifications of the complement of the neck in
   $C^\nu$ to $C$ converge uniformly to identity as $\nu \to
   \infty$.  Let $\Gamma$ be the combinatorial type of the limit curve
   $C$.  For every vertex $v \in \Ver(\Gamma)$, the polytope $P(v)$ in
   the limit map is uniquely determined as follows: Suppose there is a
   translation sequence
   $\{t_\nu'(v) \in P'(v)^\dual : v \in \Ver(\Gamma)\}_\nu$ for which
   Gromov convergence holds. Then the property
   \[d_{B^\dual}(t_\nu(v),P_0^\dual) \to \infty \quad \forall P_0 \in
     \PP, P_0 \supset P'(v)\]
   of a translation sequence implies, by Lemma \ref{lem:hcp}, that 
   the maps $u_{\nu,v}$ horizontally converge in $P'(v)$. However, $P'(v)=P(v)$ because the
   polytope of horizontal convergence is unique by Lemma \ref{lem:cptconv}. 
   
   Translation sequences are well-determined up to uniformly bounded
   perturbations as follows : Suppose $t_\nu$, $t_\nu'$ are two
   distinct translation sequences, such that the sequence
   $\e^{-t_\nu}u_\nu$ resp. $\e^{-t_\nu'}u_\nu$ converges to a broken
   map $u$ resp. $u'$ . Then for all vertices $v \in \Ver(\Gamma)$,
   there is a uniform bound \label{unifbound}
   \[ \sup_\nu |t_\nu(v)-t'_\nu(v)| <\infty,\]
   because both the sequences $\e^{-t_\nu(v)}u_\nu$,
   $\e^{-t_\nu'(v)}u_\nu$ converge pointwise in $C_v^\circ$.  After
   passing to a subsequence, we may assume that there exists a limit
   \[ t(v):=\lim_\nu t_\nu(v)-t'_\nu(v) .\]
   Then, for each vertex $v$, $u_v=e^{t(v)} u_v'$. Since $u_v$ and
   $u_v'$ satisfy matching conditions at nodes we conclude that $t$ is
   an element of $T_{\trop,\W}(\Gamma)$, which is the identity
   component of $T_\trop(\Gamma)$. We have thus shown that the limit
   is unique up to the action of $T_{\trop,\W}(\Gamma)$.
 \end{proof}
 %
 The proof of convergence of breaking maps used an \em{approximate translation sequence}, which was eventually refined to obtain a tranlation sequence for the convergence.
 \begin{definition}
  \label{def:approxt}
  {\rm(Approximate translation sequence)}
    \index{Direction condition! Approximate direction condition! on a translation sequence}
  Suppose $\Gamma$ is
  a pre-tropical graph (as in Definition \ref{def:tropgraph}). An \em{approximate $\Gamma$-translation sequence}
  consists of sequences $\{ t_\nu(v) \in \nu P(v)^\dual \}_\nu$ for
  each $v \in \Ver(\Gamma)$ such that
  \begin{itemize}[]
  \item \label{item:approxdirection}
  {\rm (Approximate Direction)} For any edge
    $e=(v_+,v_-) \in \Edge_-(\Gamma)$, there exists a sequence
    $l_\nu(e) \to \infty$ such that
    \[\sup_\nu(t_\nu(v_+) - t_\nu(v_+) - \cT(e) l_\nu) < \infty.\]
  \end{itemize}
\end{definition}

The differences appearing in the
\hyperref[item:approxdirection]{(Approximate Direction)} condition will be
referred to later using the following notation:

\begin{definition}
  On a tropical graph $\Gamma$ define the \em{discrepancy} function
  on any edge $e=(v_+,v_-) \in \Edge_{\black,-}(\Gamma)$ as
  \[\Diff_e : \oplus_{\Ver(\Gamma)}\t_{P(v)} \to
    \t_{P(e)}/\bran{\cT(e)}, \quad (t_v)_{v \in \Ver(\Gamma)} \mapsto
    (t_{v_+}-t_{v_-}) \mod \cT(e).\]
    \end{definition}

\begin{lemma}{\rm(From approximate to exact translation sequences)}
  \label{lem:approxpoly}
  Suppose $\Gamma$ is a pre-tropical graph and $t_\nu$ is an
  approximate $\Gamma$-translation sequence. Then, after passing to a
  subsequence, there is a $\Gamma$-translation sequence $\ol t_\nu$
  such that $|\ol t_\nu(v) - t_\nu(v)|$ is uniformly bounded for all
  $\nu$, $v \in \Ver(\Gamma)$.  Consequently, $\Gamma$ is a tropical
  graph.
\end{lemma}

\begin{proof}    
  The \hyperref[item:approxdirection]{(Approximate Direction)} condition says
  that the sequences of discrepancies $(\Diff_e(t_\nu))_\nu$ are
  uniformly bounded. Via uniformly bounded adjustments to $t_\nu$, we
  aim to make this quantity vanish for all edges.

  We give an algorithm that transforms $t_\nu$ into a bounded sequence
  $t_\nu^k \in \oplus_{\Ver(\Gamma)}\t_{P(v)}$, and will prove
  later that $t_\nu - t_\nu^k$ is an exact translation sequence. The
  algorithm is as follows:

  \vskip .1in \noindent \textsc{Step 1}: \em{Relativisation}.
  In this step, we replace $t_\nu$ by
  \[t^0_\nu(v):=t_\nu(v) - \nu \lim_\nu (t_\nu(v)/\nu) \in
    \t_{P(v)}. \]
  The limit in the right-hand side exists after passing to a
  subsequence because the original translation sequences $t_\nu$ lie
  in $\nu B^\dual$ and $B^\dual$ is compact.  For any
  $v \in \Ver(\Gamma)$, the discrepancies across edges are preserved:
  \begin{equation}
    \label{eq:diffe1}
    \Diff_e(t_\nu)=\Diff_e(t^0_\nu).   
  \end{equation}

  \vskip .1in \noindent \textsc{Step 2}: \em{Subtracting fastest
    growing sequences}. 
  By a sequence of further transformations, we will change $t^0_\nu$
  to a bounded sequence $t^k_\nu \in \t$.  At each step, the
  sequence $t^i_\nu$ is replaced by $t^{i+1}_\nu$ defined as
  follows.
  For each vertex $v \in \Ver(\Gamma)$, choose a basis for $\t_{P(v)}$, and
  write $t_\nu(v)=(t_{\nu,j}(v))_j$ and we run an iterative algorithm 
on the set of sequences
  \[\{t_{\nu,j}(v)\}_{j,v}.\]
  The resulting sequences at the $i$-th step are denoted by
  \[\{t_{\nu,j}^i(v)\}_{j,v}.\]
  From this set of sequences, 
  we choose a fastest growing sequence $t_{\nu,j_0}^i(v_0)$, in the sense that for all pairs $(v, j)$,
  $\lim_\nu |t^i_{\nu,j}(v)|/|t^i_{\nu, j_0}(v_0)|$ is finite.  Such a pair $(v_0,j_0)$ can
  indeed be chosen, because after passing to a subsequence, the limit
  $\lim_\nu |t^i_{\nu, j_1}(v_1)|/|t^i_{\nu, j_2}(v_2)|$ exists in $[0,\infty]$ for
  pairs $(v_1,j_1)$, $(v_2,j_2)$.  Now, define
  \[r^i_\nu:=|t^i_{\nu,j_0}(v_0)|,\]
  and 
  \begin{equation}
    \label{eq:indti}
    t^{i+1}_{\nu,j}(v):=t^i_{\nu,j}(v) - r^i_\nu\lim_\nu\frac
    {t^i_{\nu,j}(v)}{r^i_\nu} \in \t_{P(v)}.  
  \end{equation}
  We stop the iteration when the sequence $t^i_\nu(v)$ corresponding
  to every vertex is bounded, and suppose the final sequence is  $t^k_\nu$.

The process terminates in a finite number of steps.  Indeed, notice
that for the pair $(v_0,j_0)$ chosen as above, 
$t^{i+1}_{\nu, j_0}(v_0)=0$ for all $\nu$, and the sequence stays zero in every subsequent iterative step. 
The iterations of the algorithm preserve the discrepancies across
edges: For all tropical edges $e \in \Edge_{\trop}(\Gamma)$
  \begin{equation}
    \label{eq:inddiff}
    \Diff_e(t^{i+1}_\nu)=\Diff_e(t^i_\nu).       
  \end{equation}
  Indeed, \eqref{eq:indti} implies
  \[\Diff_e(t^{i+1}_\nu)=\Diff_e(t^i_\nu)-
    r^i_\nu\lim_\nu\frac {\Diff_e(t^i_\nu)}{r^i_\nu},\]
  and the second term in the right-hand-side vanishes because
  $\Diff_e(t^i_\nu)$ is uniformly bounded and $r^i_\nu \to \infty$ as
  $\nu \to \infty$.

  We claim that $t_\nu - t_\nu^k$ is a translation sequence in the sense of
  Definition \ref{def:trseq}.
  For all vertices $v$ the (Polytope) condition \eqref{pcond}
  $(t_\nu - t_\nu^k)(v) \in P(v)^\dual$ is satisfied because
  $t_\nu(v) \in P(v)^\dual$ and
  $t_\nu^k(v) \in \t_{P(v)} \simeq TP(v)^\dual$. The (Direction)
  condition is satisfied because $\Diff_e(t_\nu) - \Diff_e(t^k_\nu)=0$
  by \eqref{eq:diffe1} and \eqref{eq:inddiff}. The last statement that
  $\Gamma$ is a tropical graph follows from Remark
  \ref{rem:trans2pos}.
\end{proof}

\section{Convergence for broken maps}\label{sec:broken-conv}

In this section, we prove Theorem \ref{thm:cpt-broken} on Gromov
compactness for broken maps.  The limit map may have additional
components because of bubbling and consequently the tropical graph of
the limit map may have additional vertices. The tropical graph of the
limit map is related to the tropical graph of the maps in the sequence
by a \em{tropical edge collapse} relation defined below.  We show that
such bubbling happens only in families whose dimension is at least
two, and so does not occur in the zero-dimensional moduli spaces we
use to define the Fukaya algebra.

\begin{definition} \label{def:tropcol1} {\rm(Collapsing edges
    tropically)} \index{Collapsing an edge!Tropical edge collapse} A
  \em{tropical edge collapse} is a morphism of tropical graphs
  $\Gamma' \xrightarrow{\kappa} \Gamma$ that collapses a subset of
  edges $\Edge(\Gamma') \backslash \Edge(\Gamma)$ in $\Gamma'$
  inducing a surjective map on the vertex sets
  \[ \kappa:\Ver(\Gamma') \to \Ver(\Gamma), \]
  and satisfies the following conditions:
  \begin{enumerate}
  \item for any vertex $v \in \Ver(\Gamma')$,
    $P(v) \subseteq P(\kappa(v))$; and
  \item the edge {direction} is unchanged for uncollapsed edges, i.e. if
    $\cT$, $\cT'$ are the edge {direction} functions for $\Gamma$,
    $\Gamma'$, then $\cT(\kappa(e)) = \cT'(e)$ for any uncollapsed
    edge $e \in \Edge(\Gamma')$.
  \end{enumerate}
  \noindent Since the edge {direction} function $\cT'$ extends $\cT$, we
  often use the same notation for both.  A tropical edge collapse
  $\Gamma' \xrightarrow{\kappa} \Gamma$ is \em{trivial} if no edge is
  collapsed, and $P_{\Gamma'}(v)=P_{\Gamma}(\kappa(v))$ for all $v$.
\end{definition}

The definition of the tropical edge collapse morphism
$\kappa: \Gamma' \to \Gamma$ is meaningful even if $\Gamma'$ is a
pre-tropical graph (Definition \ref{def:tropgraph}). This extension of
terminology is often used in proofs at points when the realizability
of $\Gamma'$ has not yet been ascertained.

\begin{definition}
  {\rm(Tropical edge collapse for a disk type)}
  Let $\Gamma_1$, $\Gamma_2$ be combinatorial types of  treed disks that are equipped with a tropical structure given by tropicalization maps
  \[\tr_1 : \Gamma_1 \to \Gamma_{1,\tr}, \quad \tr_2 : \Gamma_2 \to \Gamma_{2,\tr}. \]
  An  edge collapse morphism $\Gamma_1 \to \Gamma_2$ is a \em{tropical edge collapse} if it is a lift of a tropical edge collapse map $\Gamma_{1,\tr} \to \Gamma_{2,\tr}$ between the tropical graphs.
\end{definition}

\begin{example} 
  In Figure \ref{fig:rigid}, collapsing the middle edge in $\Gam_2$ gives a tropical edge collapse morphism $\kappa:\Gam_2 \to \Gam_1$. See Figure \ref{fig:tcollapse} for another example of a tropical edge collapse morphism.
\end{example}

\begin{example}
  \label{ex:maslov4-collapse}
  Figure \ref{fig:maslov4} lists the types of broken maps whose gluing
  has a certain fixed homology class. The tropical graphs in Figure
  \ref{fig:maslov4} are labelled $\Gamma_1,\dots,\Gamma_7$.  For every
  even $i$ there are edge collapse morphisms
  $\Gamma_i \to \Gamma_{i-1}$ and $\Gamma_i \to \Gamma_{i+1}$.
\end{example}

\subsection{Relative translations}
In the definition of convergence of broken maps, the role of translation sequences is played by \em{relative translation sequences}, introduced in Definition \ref{def:relweight} below. 
Components of a relative translation correspond to maps of broken
manifolds that rescale the coordinates on cylindrical ends.  For a
pair $Q \subseteq P$ of polytopes, we recall that \eqref{eq:Pcoord}
gives an embedding
\begin{equation}
  \label{eq:Pcoord2}
  i_Q^{\tP}: U_Q(\XC_P) \to  \XC_Q
\end{equation}
where $U_Q(\XC_P) \subset \XC_P$ is the $Q$-cylindrical end of $\XC_P$
from \eqref{eq:uq}. If $Q=P$, then $U_Q(\XC_P) =U_P(\XC_P) = \XC_P$
and the map in \eqref{eq:Pcoord2} is the identity map.  A relative
translation $t \in \Cone_{P^\dual}Q^\dual$ gives an embedding
  \begin{equation}
    \label{eq:reltrans}
    e^{-t} : U_Q(\XC_P) \to \XC_Q  
  \end{equation}
  defined as $i_Q^{\tP}$ in \eqref{eq:Pcoord2} composed with a
  translation by $-t$ in the $\t_Q$-coordinate in
  $\XC_Q \simeq \oZ_Q \times \t_Q$. Here, we assume
  $\Cone_{P^\dual}Q^\dual \subset \t_Q$ by fixing a point in
  $P^\dual$ to be the origin.

  For a sequence of converging broken maps, each with tropical graph
  $\Gamma$, certain components may escape into cylindrical ends in the
  limit, or in the compactified setting, sink into (intersections of)
  relative divisors.  We examine the convergence of such map
  components after applying a translation sequence, which in the
  compactified setting, is a sequence of rescalings of the target
  space.  In order for the translated maps to converge, the
  translation sequences must go to infinity (in the sense of
  Definition \ref{def:tttrans} below). These translations must be
  thought of as happening at a much smaller scale, compared to the
  tropical graph $\Gamma$. In other words, for the map component
  corresponding to a vertex $v$ in $\Gamma$, the tropical position
  $v_0$ of the limit of a sequence of translated maps should be seen
  as being much closer to $v$, compared to other vertices $v'$ of
  $\Gamma$. In fact, one may think of vertex positions in the tropical
  graph $\Gamma'$ of the limit map as
  $\cT_\Gamma + \eps \cT_{\on{rel}}$, where $\cT_\Gamma$ is a vertex
  position map on $\Gamma$, $\cT_{\on{rel}}$ is an element in the
  translation sequence for the convergence of the maps, and $\eps$ is
  an infinitesimal.  \footnote{The last few sentences tell the reader
    how to think about translation sequences. These statements 
    cannot
    be stated formally because, although a tropical graph is
    realizable in the dual complex as in Definition
    \ref{def:tropgraph}, the vertex position map is not part of the
    data of a tropical map.}
\label{rep:trans-scale}
However, we avoid the use of infinitesimals in our statements, since
for any $\cT_\Gamma$ and $\cT_{\on{rel}}$, there exists $t_0>0$ such
that for any $t \in (0,t_0)$, $\cT_\Gamma + t \cT_{\on{rel}}$ is a
vertex position map for $\Gamma'$.

\begin{definition} \label{def:relweight} {\rm(Relative translations)}
  \index{Translation!Relative translation} Suppose
  $\kappa: \Gamma' \to \Gamma$ is a tropical edge-collapse morphism,
  where $\Gamma$ is a tropical graph and $\Gamma'$ is a pre-tropical
  graph.  For any vertex $v$ of $\Gamma'$, let
  \begin{multline}\label{eq:conekv}
    \Cone(\kappa,v) := \Cone_{ P(\kappa(v))^\dual}(P(v)^\dual)\\
    :=\{\alpha(t - t_0) \in \t: t \in P(v)^\dual, t_0 \in
    P(\kappa v)^\dual,\alpha \in \R_{\geq 0}\}
  \end{multline}
  be the cone in the polytope $P(v)^\dual$ based at points in
  $P(\kappa(v))^\dual$ (as in Definition \ref{def:normcones}), and
  $\Cone(\kappa,v) \subset \t_{P(v)}$ by fixing a point in
  $P(\kappa(v))^\dual$ to be the origin.  A \em{relative translation}
  or a \em{$(\Gamma',\Gamma)$-translation} is an element
  \[\cT_{\Gamma',\Gamma} =(\cT_{\Gamma',\Gamma}(v) \in \Cone(\kappa,v))_{v \in \Ver(\Gamma')} \]
  satisfying
  \begin{equation}\label{eq:direction-rel}
    \cT_{\Gamma',\Gamma}(v_+) - \cT_{\Gamma',\Gamma}(v_-) \in
    \begin{cases}
      \R_{\geq 0} \cT(e), & e \in \Edge(\Gamma') \bs \Edge(\Gamma),\\
      \R\cT(e), & e \in \Edge(\Gamma)
    \end{cases}
  \end{equation}
  for any edge $e=(v_+,v_-)$ in $\Gamma'$. The space of $(\Gamma',\Gamma)$-translations is denoted by $w(\Gamma',\Gamma)$. 
 \end{definition}

 Relative translations appearing in the convergence of broken maps
 \em{go to infinity} in the following sense:
\begin{definition}\label{def:tttrans}
  {\rm(Relative translation sequence going to infinity)}
\index{Translation!Relative translation going to infinity}
  Suppose
  $\kappa:\Gamma' \to \Gamma$ is a tropical edge collapse,
   where $\Gamma$ is a tropical graph and $\Gamma'$ is a pre-tropical graph. 
  A
  $(\Gamma',\Gamma)$-translation sequence 
  \[t_\nu(v) \in \Cone(\kappa,v)=\Cone_{P(\kappa
      v)^\dual}(P(v)^\dual), \quad v \in \Ver(\Gamma') \]
(where $\Cone(\kappa,v)$ is as
  defined in \eqref{eq:conekv}) \em{goes to infinity} if 
  \begin{enumerate}
  \item {\rm(Polytope)} For any vertex $v \in \Ver(\Gamma')$ and a
    polytope $Q \in \PP$ such that
    $P(v) \subset Q \subseteq P(\kappa(v))$,
    \[d(t_\nu(v), \Cone_{P(\kappa v)^\dual}(Q^\dual)) \to \infty.\]

  \item {\rm(Direction for collapsed edge)} \index{Direction condition! on
      relative translations} For any edge
    $e \in \Edge(\Gamma_{\kappa^{-1}(v)})$ connecting vertices $v_+$,
    $v_-$, the sequence $l_\nu \in \R$ defined as
    \[t_\nu(v_+) - t_\nu(v_-)=\cT(e) l_\nu\]
    goes to infinity, that is, $l_\nu \to \infty$.
    (Note that the existence of $l_\nu$ follows from the Direction condition \eqref{eq:direction-rel} on relative translations.)
  \end{enumerate}
\end{definition}

The following result says that the existence of a
$(\Gamma',\Gamma)$-translation sequence implies that $\Gamma'$ is a
tropical graph, in that it has a vertex position map. The proof is
straightforward and is left to the reader.

\begin{lemma}\label{lem:rel2trop}
  Suppose $\kappa: \Gamma' \to \Gamma$ is a tropical edge-collapse
  morphism, where $\Gamma$ is a tropical graph and $\Gamma'$ is a
  pre-tropical graph.  Suppose $\{\cT_\nu\}_\nu$ is a sequence of
  $(\Gamma',\Gamma)$-translations going to infinity.  Then, $\Gamma'$
  is a tropical graph.  For any vertex position $\cT_\Gamma$ of
  $\Gamma$ and a large enough $\nu$, there is a constant $t_0>0$ such
  that
  \[\Ver(\Gamma') \ni v \mapsto \cT_\Gamma(\kappa(v)) + t \cT_\nu(v)\]
  is a vertex position for $\Gamma'$ for all $t \in (0,t_0)$.
\end{lemma}

\begin{remark}\label{rem:rel-trans-cone}
  {\rm(The space of relative translations is a cone)} Given a tropical
  edge collapse morphism $\kappa : \Gamma' \to \Gamma$ between
  tropical graphs $\Gamma'$, $\Gamma$, the space of relative
  translations $\fw(\Gamma',\Gamma)$ is a cone
  \begin{equation}
    \label{eq:rel-is-cone}
      \fw(\Gamma',\Gamma)=\Cone_{\ol \W(\Gamma)}\ol \W(\Gamma')
  \end{equation}
  based at the polytope $\ol \W(\Gamma)$ (as in Definition
  \ref{def:normcones}).  Here, we recall from Remark \ref{rem:Wpoly}
  that the closures of the spaces of vertex positions
  $\ol \W(\Gamma')$ and $\ol \W(\Gamma)$ are polytopes, and
  \[\ol \W(\Gamma) \hra \ol \W(\Gamma'), \quad \cT \mapsto (v \mapsto
    \cT(\kappa(v)))\]
  is a face. We obtain \eqref{eq:rel-is-cone} by observing that for
  any vertex position maps $\cT_{\Gamma'} \in \W(\Gamma')$ and
  $\cT_\Gamma \in \W(\Gamma)$, and $\alpha \geq 0$, the scaled
  difference
  \[\Ver(\Gamma') \ni v \mapsto \alpha(\cT_{\Gamma'}(v) - \cT_\Gamma(\kappa v)) \in \t_{P(v)} \]
  is a $(\Gamma',\Gamma)$-translation. If $\Gamma$ is rigid,
  $\W(\Gamma)$ is a single point, and the space $w(\Gamma',\Gamma)$ of
  relative translations is a cone based at a single point.
\end{remark}

\begin{example} \label{ex:tropedgecollapse} For the tropical edge
  collapse $\kappa : \Gamma' \to \Gamma$ in both Figure
  \ref{fig:rigid} and Figure \ref{fig:tcollapse}, the tropical graph
  $\Gamma$ is rigid, and the space of vertex positions of $\Gamma'$ is
  (linearly) isomorphic to the interval $(0,1)$, where $0$ corresponds
  to the vertex position where the edges collapsed by $\kappa$ have
  length $0$, and therefore, by \eqref{eq:bslope}, is not a vertex
  position for $\Gamma'$; and $1$ corresponds to some vertex $v$ of
  $\Gamma'$ lying in a lower dimensional polytope
  $Q^\dual \subset P^\dual(v)$, and so, is not a legitimate vertex
  position of $\Gamma'$.  In both cases the space of relative
  translations is
  \[\fw(\Gamma',\Gamma)= \Cone_{0}[0,1]= \R_{\geq 0}.\]
\end{example}

\subsection{Defining Gromov convergence of broken maps}
The definition of Gromov convergence for broken maps is largely similar to the definition of convergence of breaking maps (Definition \ref{def:grom-breaking}) with the following two differences:
\begin{itemize}
\item The definition of convergence for breaking maps is given for a
  sequence of irreducible maps in neck-stretched manifolds.  For
  convergence of broken maps, we need to consider maps with multiple
  components. Thus, in the definition below, we give the conditions
  under which a sequence of broken maps with tropical graph $\Gamma$
  converges to a broken map with tropical graph $\Gamma'$, where the
  tropical graphs are related by an edge collapse
  $\kappa : \Gamma' \to \Gamma$.  For each vertex $v \in \Ver(\Gamma)$
  of the tropical graph, the sequence of smooth maps $u_{v,\nu}$, with
  markings corresponding to nodal lifts $w_e \in C_v$, converges to a
  limit map $u_v$.  The limit map $u_v$ may have multiple components,
  corresponding to vertices that are collapsed to $v$ by the morphism
  $\kappa$.  That is, $u_v=u|(\cup_{v' \in \kappa^{-1}(v)}C_{v'})$.
\item Translations occurring in the convergence of breaking maps are
  replaced by relative translations for the case of convergence of
  broken maps. The corresponding maps of manifolds $\e^{-t}$ in the
  breaking case (see Definition \ref{def:transdef}) are replaced by
  $e^{-t}$ (see \eqref{eq:reltrans}) in the case of broken maps.
\end{itemize}

  \index{Collapsing an edge!Tropical edge collapse}
  \begin{definition}
    {\rm(Gromov convergence for broken maps)} Suppose $\Gamma'$,
    $\Gamma$ are combinatorial types of stable treed disks that are
    equipped with a tropical structure. Suppose
    $\Gamma' \xrightarrow{\kappa} \Gamma$ is a tropical edge collapse
    morphism which induces a vertex map
    $\kappa:\Ver(\Gamma') \to \Ver(\Gamma)$.  A sequence of broken
    maps $u_\nu:C_\nu \to \XX$ of type $\Gamma$ \em{(Gromov)
      converges} to a limit broken map $u:C \to \XX$ of type $\Gamma'$
    if the following conditions are satisfied.
    \begin{enumerate}
    \item {\rm(Convergence of domains)}
      \label{item:dconv3}
      The sequence of treed disks $C_\nu$ converge to the treed disk
      $C$ and for any tropical node $w_e$,
      $e \in \Edge_\trop(\Gamma')$ that is collapsed by $\kappa$, the
      arguments $\frac {\nl_e(C_\nu)}{|\nl_e(C_\nu)|}$ of the gluing
      parameters converge to a limit.  Let
      $S_\nu(v) \subset (S_\nu)_{\kappa(v)} \subset C_\nu$ be the
      subset corresponding to a vertex $v \in \Ver(\Gamma')$, and let
      \[i_{v,\nu}:=i_{S_v,S_{\nu,\kappa(v)}}: S_\nu(v) \to S_v, \quad
        S_\nu(v) \subset S_\nu,\]
      be embeddings from \eqref{eq:neckid} whose images
      $i_{v,\nu}(S_\nu(v))$ exhaust $S_v^\circ$ as $\nu \to
      \infty$. Here
      \[S_v^\circ:=S_v \bs \{w_e : v \in e, \enspace e \in
        \Edge_-(\Gamma') \text{ is collapsed by $\kappa$} \}.\]
    \item{\rm(Convergence of maps)}
      \label{item:mconv3}
      There is a $(\Gamma',\Gamma)$-translation sequence 
      \[(t_\nu(v))_{v \in \Ver(\Gamma')}\]
      going to infinity in the
      sense of Definition \ref{def:tttrans}, such that for any vertex
      $v \in \Ver(\Gamma')$, the sequence of maps
      \[ S_v^\circ \supset i_{v,\nu}(S_\nu(v))
        \xrightarrow{e^{-t_\nu(v)}(u_\nu \circ i_{v,\nu}^{-1})}
        \XC_{P(v)} \]
      converges in $C^\infty_{\loc}(S_v^\circ)$ to
      $u_v:S_v^\circ \to \XC_{P(v)}$. The map
      $e^{-t_{\nu}(v)}: U_{P(v)}(\XC_{P(\kappa v)}) \to \XC_{P(v)}$ is
      defined on the $P(v)$-cylindrical end
      $U_{P(v)}(\XC_{P(\kappa v)}) \subset \XC_{P(\kappa v)}$, see
      \eqref{eq:reltrans}.  For each boundary edge
      $e \in \Edge_\white(\Gamma)$, the maps $u_\nu|T_{e,\nu}$ on the
      treed segment converge to a (possibly broken) treed segment in
      $u$.
    \item \label{item:thincyl3} For a node $w$ in $C$ corresponding to
      a tropical edge $e$ collapsed by $\kappa$, that is,
      $e=(v_+,v_-) \in \Edge_\trop(\Gamma') \bs \Edge_\trop(\Gamma)$,
      let
    \[ z_\pm : (U_{w^\pm}, w^\pm) \to (\C, 0) \]
    be matching coordinates (see Definition \ref{def:bmap} following
    \eqref{eq:nodematch}) on neighborhoods
    $U_{w^\pm} \subset S_{{v}_\pm}$ of the nodal point which respect
    the framing $\fr_e$.  Let
    \[ A(l_\nu):=[-l_\nu/2,l_\nu/2] \times S^1 \subset S_\nu \]
    be a sequence of centered annuli converging to the node $w$, see
    Definition \ref{def:ann2node}.  Then the sequence
    \[ x_\nu:=e^{-\hh(t_\nu(v_+) + t_{\nu}(v_-))} u_\nu(0,0) \in \XC_{P(e)}\]
    converges to a limit $x_0$, and the components $u_{v_\pm}$ of the
    broken map are asymptotically close to
    \[z_\pm \mapsto z_\pm^{\cT(e)} x_0\]
    in the sense of Remark \ref{rem:expconv}.
    \end{enumerate}
  \end{definition}

 \subsection{From an area bound to a bound on types} 
 The following result is the first ingredient in the proof of Theorem
 \ref{thm:cpt-broken}, which shows that there is a finite number of
 tropical graphs corresponding to broken maps that satisfy an area
 bound.
 \begin{proposition} \label{prop:finno} {\rm(Finite number of tropical
     graphs underlying broken maps)} For any $E > 0 $,
   $d(\white) \geq 1$ there are a finite number of tropical graphs
   $\Gamma$ that occur as tropical graphs of broken maps
   $u: C \to \XX$ (see Definition \ref{def:type-broken}) of area at
   most $E$ and $d(\white)$ boundary leaves.
\end{proposition}

\begin{proof}
  Consider a broken map $u: C \to \XX$ of type $\Gamma$ and area at
  most $E_0$, whose edge {direction}s are given by a map
  $\cT: \Edge(\Gamma) \to \t_\Z$. We first consider the case when the
  $X$-inner product (see \eqref{eq:idtt}) is rational, so that the
  compactification $\ol \XX_P$ of $\XX_P$ is an orbifold for any
  $P \in \PP$. We discuss the generalization at the end of the proof.

  \vskip .1in \noindent \textsc{Step 1}: \em{Uniform bound on the number of vertices}.\\
  The number of interior markings in $\Gamma$ is bounded by $kE_0$,
  where $k$ is the degree of the stabilizing divisor from
  \eqref{eq:kkstar}.  Since $u$ is adapted to the stabilizing divisor (as in Definition \ref{def:pert-unbroken} \eqref{part:u-adapt-d}), all the surface
  components $S_v, v \in \Ver(\Gamma)$ in its domain $C$ are stable.
  Given a uniform bound on the number of interior and boundary
  markings on the type of domain curve, we obtain a uniform bound on
  the number irreducible surface components.

  \vskip .1in \noindent \textsc{Step 2}: \em{Uniform bound on the
    sum of vertical components of
    edge {direction}s}.\\
  Let $u_v: S_v \to \ol \XC_P$ be a component of the broken map
  corresponding to a vertex $v \in \Ver(\Gamma)$, and $P:=P(v)$.  By
  the balancing property \eqref{eq:balprop}, the sum of the edge
  {direction}s projected to $\t_P$ is
  \[ \sum_{e \ni v} \pi_{\t_P}(\cT(e)) = c_1( (\pi_P \circ u_v)
    ^*\ol Z_P \to \ol X_P). \]
    The right-hand side, which is the pairing of
    $(\pi_P \circ u_v)_*[C]$ and the Chern class $c_1(\ol Z_P \to \ol X_P)$, has an $E_0$-dependent bound.
 Indeed, for any
    $\eps \in (0,1)$, there is a constant 
    $k_0$ such that for any domain-dependent almost
    complex structure $J \in B_\eps(J_{\ol X_P})_{C^0}$ and a
    $J$-holomorphic sphere $u:\P^1 \to \ol X_P$,
  \[ \int_{\P^1}u^*c_1(\ol Z_P \to \ol X_P)\leq  k_0 \int_{\P^1}
  u^*\om_{X_P} \leq  k_0 E_0
    .\]
  This estimate is similar to the one in Lemma \ref{lem:unifdegbd} and the
  proof is the same -- by choosing a two-form on $\ol X_P$ representing
  the Chern class and bounding it pointwise by $\om_{X_P}$.

  \vskip .1in \noindent \textsc{Step 3}: \em{Uniform bound on the horizontal components of edge {direction}s}.\\
  For a vertex $v$ and an incident edge $e$, the horizontal component
  of the {direction} $\cT(e)$ is the sum of intersection multiplicities at 
  the node $w_e$ with horizontal relative divisors of $\ol X_P$:
  \[\pi_{\t_P}^\perp(\cT(e))=\ssum_{Q \subset P}
    m_{w_e}(u_{v,P},\ol X_Q)\nu_Q, \quad u_{v,P}:=\pi_P \circ u_v : 
    S_v \to \ol X_P,\]
  where the sum ranges over facets $Q \in \PP$ of $P$, and $\nu_Q$ is the normal vector of the facet $Q \subset P$.  For
  any relative divisor $\ol X_Q \subset \ol X_P$, the sum
  $\sum_{e \ni v} m_{w_e}(u_v, \ol X_Q)$ is bounded by
  $c \om_{X_P}(u_{v,P})$ for a uniform constant $c(\ol X_P,\ol X_Q)$.  The
  proof is similar to the vertical case, by expressing the
  intersection number with any divisor as an integral of a two-form. %

  \vskip .1in \noindent \textsc{Step 4}: \em{Finishing the proof}.\\
  We will show that the tropical edge {direction}s $\cT(e)$ of $\Gamma$ are
  uniformly bounded in $\t$ for the set of all broken maps with
  area at most $ E_0$. So far we have shown that (Step 3) the
  horizontal components of the edge {direction}s are uniformly bounded, and
  the vector sum of the vertical {direction} components of the edges
  incident on any vertex are uniformly bounded.  From here, we
  conclude that for a vertex $v$ and an edge $e_0$ incident on $v$:
  \begin{equation}
    \label{eq:hor-vert}
    \exists c(E) : |\cT(e_0)| \leq \sum_{e \ni v,e \neq e_0 } |\cT(e)| + c(E).
  \end{equation}
  Recall that $\Gamma$ is a tree, any edge
  $e \in \Edge_\black(\Gamma)$ is oriented so that it points away from
  the root vertex. The {direction} of any incoming edge can be bounded by
  the {direction} of outgoing edges by \eqref{eq:hor-vert}.  Applying
  \eqref{eq:hor-vert} iteratively, we  conclude that for any
  edge $e$ in $\Gamma$
  \begin{equation*}
    |\cT(e)| \leq c(E) |\Ver(\Gamma)|,
  \end{equation*}
  where the constant $c(E)$ is the same as the one in
  \eqref{eq:hor-vert}. The Proposition now follows from the bound on
  the number of vertices in Step 1 in cases when the $X$-inner product is rational.

  For a general $X$-inner product $g$, we consider a sequence
  $\{g_\nu\}_\nu$ of rational $X$-inner products uniformly converging
  to $g$.  A squashing map with respect to the $X$-inner product $g$
  is the limit of a squashing maps $\aleph_\nu$ with respect to the
  $X$-inner product $g_\nu$, and therefore,
  \[ E_{P,\Hof,g} \leq \limsup_\nu E_{P,\Hof,g_\nu} . \]
  The Proposition then follows from the uniform bound 
  on $E_{P,\Hof,g_\nu}$ for all rational
  $g_\nu$ proved in the preceding paragraphs.
\end{proof}

\subsection{A translation sequence from an approximate translation sequence}

Another ingredient in the proof of Gromov compactness for broken maps
is a result that produces a relative translation sequence, starting
from an \em{approximate relative translation sequence} via a sequence
of adjustments that is uniformly bounded. We first define the approximate sequences that we require. 
\begin{definition}
  \label{def:approxt-rel}
  \index{Direction condition! Approximate direction condition! on a relative
    translation sequence} {\rm(Approximate
    $(\Gamma',\Gamma)$-translation sequence)} Suppose the tropical
  graph $\Gamma$ is obtained by collapsing edges in $\Gamma'$ and the
  induced map on the vertex set is
  $\kappa:\Ver(\Gamma') \to \Ver(\Gamma)$. Then, an \em{approximate
    $(\Gamma',\Gamma)$-translation sequence} consists of a sequence
  \[t_\nu(v) \in \Cone(\kappa,v), \]
  for every $v \in \Ver(\Gamma')$ satisfying the (Polytope) and (Direction for collapsed edges) 
  conditions 
  in Definition \ref{def:tttrans} and the following weakened version
  of the (Direction for uncollapsed edges) condition.
  \begin{itemize}
  \item[]{\rm(Approximate direction for uncollapsed edges)}
    \label{item:approxdirection-rel}
    For an edge $e$
    of $\Gamma'$ that is not collapsed in $\Gamma$,
    \[ t_\nu(v_+) - t_\nu(v_-) \mod \cT(e)\]
    is a bounded sequence in $\t/\R\cT(e)$. 
  \end{itemize}
\end{definition}
An approximate $(\Gamma',\Gamma)$-translation sequence can be adjusted
by a uniformly bounded amount to produce an actual
$(\Gamma',\Gamma)$-translation sequence.

\begin{lemma}\label{lem:relapprox}
  {\rm(From an approximate to an exact $(\Gamma',\Gamma)$-translation
    sequence)} Let $\kappa:\Gamma' \to \Gamma$ be
 a tropical edge
  collapse, and let $\{t_\nu\}_\nu$ be an approximate
  $(\Gamma',\Gamma)$-translation sequence. There exists a
  $(\Gamma',\Gamma)$-translation sequence $\{\ol t_\nu\}_\nu$ such
  that
  \[\sup_\nu|\ol t_\nu(v) - t_\nu(v)|<\infty\]
  for all $v \in \Ver(\Gamma')$.
\end{lemma}
\begin{proof}
  The proof is by replicating the iteration in Step 2 of the proof of
  Lemma \ref{lem:approxpoly}.
  For each vertex $v$ of $\Gamma'$, we choose a basis of $\t_{P_\Gamma'(v)}$ and write $t_\nu(v)=(t_{\nu,j}(v))_j$. 
  At the start of the iteration, we set
  $t^0_{\nu,j}(v):=t_{\nu,j}(v)$ for each vertex $v$. At the $(i+1)$-th step, we construct $t^{i+1}_\nu$
  as follows. As in the proof of Lemma \ref{lem:approxpoly}, there
  exists a vertex $v_0 \in \Ver(\Gamma')$ and a basis element $j_0$ such that the sequence
  $|t^i_{\nu, j_0}(v_0)|$ has the fastest growth rate. That is, for all pairs 
  $(v,j)$, $\lim_\nu |t^i_{\nu, j}(v)|/|t^i_{\nu, j_0}(v_0)|$ is
  finite. Define
  \[t^{i+1}_\nu(v):=t^i_\nu(v) - |t^i_{\nu, j_0}(v_0)|\lim_\nu\frac
    {t^i_\nu(v)}{|t^i_{\nu, j_0}(v_0)|}.\]
  For the sequences $\{t^{i+1}_\nu(v)\}_\nu$, $v \in \Ver(\Gamma')$,
  the quantity
 \begin{equation*}
      \pi_{\cT(e)}^\perp(t^{i+1}_\nu(v_+)-t^{i+1}_\nu(v_-)), \quad
      e=(v_+,v_-) \in \Edge(\Gamma')
  \end{equation*}
  vanishes for collapsed edges, and is uniformly bounded for
  uncollapsed edges.  Furthermore, for any vertex $v$, 
  $t^{i+1}_\nu(v) \in \t_{P(v)}$. After, say, $k$ steps, the
  sequence $|t^k_\nu(v)|$ is uniformly bounded for all vertices
  $v$. Then, $\ol t:=t-t^k$ is an exact $(\Gamma',\Gamma)$-translation
  sequence.
\end{proof}

\subsection{Proof of Gromov compactness for broken maps}
The notion of horizontal convergence extends in a natural way to
broken manifolds, though it is easier to state in this case.
\index{Horizontal convergence}
\begin{definition}\label{def:horiz-br}
  {\rm(Horizontal convergence in a broken manifold)} Let $P \in \PP$
  be a polytope. A sequence of points $x_\nu \in \XX_P$
  horizontally converges in $X_Q$ for a polytope $Q \subseteq P$ if the sequence
  $\pi_P(x_\nu) \in X_P$ converges to a point
  $x \in X_Q \subseteq \ol X_P$. 
\end{definition}

\begin{remark}\label{rem:br-analogs}
  Analogs of Lemma \ref{lem:hcp} and \ref{lem:cptconv} hold for
  horizontal convergence in broken manifolds. An analog of the
  breaking annulus lemma (Proposition \ref{prop:breaking-ann}) also
  holds for broken maps.  For all these results, we make the following
  substitutions in the original result to obtain the broken version
  for maps in $\XX_P$: a translation sequence is replaced by a
  relative translation sequence going to infinity, any instance of
  the translation map $\e^{-t}$ is replaced by $e^{-t}$ where
  $t \in \Cone_{P^\dual}B^\dual$ (defined in \eqref{eq:reltrans}), and
  Hofer energy is replaced by $P$-Hofer energy.  As an example, we
  state the analog of Lemma \ref{lem:hcp}, leaving the rest to the
  reader.
\end{remark}

\begin{lemma}\label{lem:hcp-broken}
  {\rm(Broken analog of Lemma \ref{lem:hcp})} Let $P \in \PP$ be a
  polytope.  Suppose $x_\nu \in \XX_P$ is a sequence of points, and
  $t_\nu \in \Cone_{P^\dual}Q^\dual$ is a sequence of translations for
  a proper face $Q \subset P$, such that $e^{-t_\nu}x_\nu$ converges
  in $\XX_Q$. Then, the following are equivalent:
  \begin{enumerate}
  \item \label{part:hc-broken} The sequence $x_\nu \in X_P$ converges
    horizontally in $X_Q$ (in the sense of Definition
    \ref{def:horiz-br}).
  \item \label{part:pc-broken} For any $P \supseteq P_0 \supset Q$ (that is, $Q$ is a proper face of $P_0$), $d(t_\nu, \nu P_0^\dual) \to \infty$. 
  \end{enumerate}
\end{lemma}

\begin{proof}
  [Proof of Theorem \ref{thm:cpt-broken}] After passing to a
  subsequence, the tropical graph $\Gamma$ underlying the maps $u_\nu$
  is $\nu$-independent.  Indeed by Proposition \ref{prop:finno}, there
  is a finite number of tropical graphs that underlie broken maps with
  area $<E$.  By Proposition \ref{prop:hofertop}, the Hofer energy for
  a map $u:(C,\partial C) \to (\XX,L)$ can be read off from its area
  and its intersection data with relative divisors, which is provided
  by the {direction}s $\cT(e)$ of the edges of the tropical graph
  $\Gamma$. Consequently, the Hofer energy $E_{\Hof}(u_\nu)$ of the
  maps $u_\nu$ is uniformly bounded.

  We first find the limit map at each of the vertices of the tropical
  graph $\Gamma$.  For each vertex $v \in \Ver(\Gamma)$, we assume
  that the domain curve $C_{\nu,v}$ for $u_{\nu,v}$ has marked points
  corresponding to lifts of nodes $e \in \Edge_-(\Gamma)$, $v \in e$,
  in addition to the marked points corresponding to leaves
  $e \in \Edge_\to(\Gamma)$.  The proof of breaking maps carries over
  verbatim to the sequence of maps $u_{\nu,v}$.  Indeed, all the
  results used in the proof of convergence of breaking maps have
  analogs for the case of broken maps; see Propositions
  \ref{prop:hofertop}, \ref{prop:Hbr-limit}, Remark
  \ref{rem:br-analogs} and Lemma \ref{lem:hcp-broken}.  The conclusion
  is that there is a limit map $u_v$ modelled on a tropical graph
  $\Gamma_v$, and the convergence is via a translation sequence
  $t_\nu(v')$ going to infinity, $v' \in \Gamma_v$. The convergence
  includes the following statements:
  \[ P(v') \subseteq P(v); \quad
  t_\nu(v') \in \Cone_{P(v)^\dual}P(v')^\dual \subset \t_{P(v)}
  \simeq \sqrt{-1} \t_{P(v)} ,\] 
  where the inclusion is obtained by
  fixing a point in $P(v)^\dual$ as the origin of the cone; the
  sequence of maps $e^{-t_\nu(v')} u_{\nu,v}$ lie in the
  $P(v')$-cylindrical end of $\XX_{P(v)}$, and converge uniformly on
  compact subsets to a limit $u_{v'}$ in $\XX_{P(v')}$.  For any
  tropical edge $e=(v_+,v_-)$ of $\Gamma$, the limit of the nodal
  lifts $w_{\nu,\pm}^e \in S_{\nu,v_\pm}$ lies on some domain
  component of $u_v$, corresponding to a vertex $v_{\pm,e}$ of the
  graph $\Gamma_{v_\pm}$.  By connecting the graphs
  $\{\Gamma_v\}_{v \in \Ver(\Gamma)}$ by edges $(v_{+,e}, v_{-,e})$
  for all edges $e$ of $\Gamma$, we obtain a pre-tropical graph
  $\Gamma'$ for which there is an edge collapse morphism
  $\kappa : \Gamma' \to \Gamma$.  We will subsequently show that
  $\Gamma'$ is a tropical graph and $\kappa$ is a tropical edge
  collapse morphism.

  Next, we simultaneously prove that the limit map $u$ satisfies
  matching conditions on the edges of $\Gamma$, and that $\Gamma'$ is
  a tropical graph.  As a first step, we show that the translation
  sequence $\{t_\nu(v): v \in \Ver(\Gamma')\}$ in the previous
  paragraph is an approximate $(\Gamma',\Gamma)$-translation sequence
  in the sense of Definition \ref{def:approxt-rel}. The convergence to
  the maps $u_v$ implies that for any $\nu$, $t_\nu$ satisfies the
  {direction} condition \eqref{eq:direction-rel} for edges collapsed by
  $\kappa$, that is, for edges
  $e \in \Edge(\Gamma') \bs \Edge(\Gamma)$. Therefore, it remains to
  prove that the \hyperref[item:approxdirection-rel] {(Approximate direction
    for uncollapsed edges)} condition is satisfied. Consider an edge
  $e$ of $\Gamma$, that is incident on $v_+$, $v_- \in \Ver(\Gamma')$,
  and the nodal point corresponding to $e$ is the pair $w^\pm_e$ on
  $C_{v_\pm}$. The edge matching condition for broken maps
  \eqref{eq:trop-matching} implies
  \begin{equation}
    \label{eq:seqvert}
    (\pi^\perp_{\cT(e)} \circ u_\nu)(w_{\nu,+}^e)=(\pi^\perp_{\cT(e)}\circ
    u_\nu)(w_{\nu,-}^e) \quad \text{in }\XC_{P(e)}/T_{\cT(e),\C}.  
  \end{equation}
  On a sequence of converging relative maps, the projected tropical evaluations of the
  relative marked points converge, that is, 
  \begin{equation*}
    (\pi^\perp_{\cT(e)} \circ (e^{-t_\nu(v)}u_\nu))(z_{i,\nu}) \to
    (\pi^\perp_{\cT(e)} \circ u)(z) \quad \text{in }\XC_{P(e)}/T_{\cT(e),\C}.
  \end{equation*}
  This convergence is a consequence of the convergence of the marked
  points $z_{i,\nu}$ on the domain curve and the convergence of maps
  $e^{-t_\nu(v)}u_\nu$.  Therefore, for the edge $e=(v_+,v_-)$,
\begin{multline*}
  d((\pi^\perp_{\cT(e)} \circ (e^{-t_\nu(v_+)}u_\nu))(w^e_{\nu,+}), (\pi^\perp_{\cT(e)}\circ (e^{-t_\nu(v_-)}u_\nu)) (w^e_{\nu,-})) \to \\
  d((\pi^\perp_{\cT(e)} \circ u)(w^e_+), (\pi^\perp_{\cT(e)} \circ u)(w^e_-)),
\end{multline*}
where $d$ is the cylindrical metric distance in
$\XC_{P(e)}/T_{\cT(e),\C}$.  Using \eqref{eq:seqvert}, the sequence
  \begin{equation} \label{eq:tdiff}
    \pi^\perp_{\cT(e)}(t_\nu(v_+)-t_\nu(v_-)) \in
    \t_\C/\t_{\cT(e),\C}  
  \end{equation}
  is bounded, implying that $\{t_\nu(v):v \in \Ver(\Gamma')\}$ is an
  approximate $(\Gamma',\Gamma)$-translation sequence.
  An approximate $(\Gamma',\Gamma)$-translation sequence can be
  adjusted by a uniformly bounded amount to produce an actual
  $(\Gamma',\Gamma)$-translation sequence by Lemma
  \ref{lem:relapprox}. After making such an adjustment the quantity in
  \eqref{eq:tdiff} vanishes. Then by \eqref{eq:seqvert} the limit map
  $u$ satisfies the edge matching condition \eqref{eq:trop-matching}
  for edges in $\Gamma$, and $u$ is therefore a broken
  map.  Furthermore, we now have that $\{t_\nu\}_\nu$ is a relative
  translation sequence going to infinity (as in Definition
  \ref{def:tttrans}), and then, Lemma \ref{lem:rel2trop} implies that
  $\Gamma'$ is a tropical graph, and therefore, the limit map
  $u=(u_v)_{v \in \Ver(\Gamma')}$ is a broken map.
   
  The proof of uniqueness of the limit is similar to the case of
  breaking maps: The domain curve is uniquely determined, and any two
  $(\Gamma',\Gamma)$-translation sequences $t_\nu$, $t'_\nu$ differ by
  a uniformly bounded amount :
  $\sup_\nu|t_\nu(v)-t_\nu'(v)|<\infty$. After passing to a
  subsequence, the limit \label{page:tinfty}
  \[ t_\infty(v):=\lim_\nu(t_\nu(v)-t_\nu'(v))  \]
  exists. The tuple $(t_\infty(v))_v$ is an unsigned version of a
  $(\Gamma',\Gamma)$-translation since
  $t_\infty(v) \in \Cone(\kappa,v)$ and for any edge
  $e=(v_+,v_-) \in \Edge_\trop(\Gamma')$,
  \[t_\infty(v_+) - t_\infty(v_-) \in \R\cT(e).\]
  Therefore, we obtain a tropical symmetry transformation
  $e^{-t_\infty} \in T_\trop(\Gamma')$ that relates the limit maps
  $\lim_\nu e^{-t_\nu}u_\nu$, $\lim_\nu
  e^{-t_\nu'}u_\nu$. Consequently, the limit is unique up to the
  action of the identity component $T_{\trop,\W}(\Gamma')$ of the
  tropical symmetry group.

  It remains to prove the last statement of the Theorem.  Suppose the
  tropical edge collapse $\kappa:\Gamma' \to \Gamma$ is non-trivial.
  Then, either a tropical edge $e_0$ of $\Gamma'$ is collapsed by
  $\kappa$, or there is a vertex $v_0$ of $\Gamma'$ for which
  $P(v_0) \subsetneq P(\kappa(v_0))$.  In both cases, for large enough
  $\nu$, the translation $(t_\nu(v))_{v \in \Ver(\Gamma')}$ generates
  a subgroup
  \[\{(\exp z t_\nu(v))_{v \in \Ver(\Gamma')} : z \in \C^\times\} \subset T_\trop(\Gamma')\]
  (see \eqref{eq:cTgen}) 
  that is not contained 
  in $T_\trop(\Gamma)$.  Indeed, in the first case of a collapsed
  tropical edge $e_0=(v_+,v_-)$
  \[t_\nu(v_+)-t_\nu(v_-) \in \R_+ \cT(e_0)\]
  and in the second case where $P(v_0) \subsetneq P(\kappa(v_0))$ we
  have
  \[t_\nu(v_0) \notin P(\kappa(v_0))^\dual.\]
  Therefore $\dim_\C(T_\trop(\Gamma')) > \dim_\C(T_\trop(\Gamma))$. 
\end{proof}

\begin{remark}
  If a sequence $\{u_\nu\}_\nu$ of broken maps with
  tropical graph $\Gamma$ converges to a limit $u_\infty$ whose 
  tropical graph $\Gamma'$ is different from $\Gamma$, then,
  \begin{equation}
    \label{eq:i-br-red-2}
    i^\br_\red(\Gamma',\ul x) \leq i^\br_\red(\Gamma,\ul x)-2.  
  \end{equation}
  Indeed, \eqref{eq:i-br-red-2} follows from the definition
   of the reduced index $i^\br_\red$ in \eqref{eq:i-red-br}, Theorem
  \ref{thm:cpt-broken} which implies
  $\dim(T_\trop(\Gamma')) \geq \dim(T_\trop(\Gamma)) +2$, and the
  index being preserved in the limit:
  $i^\br(\Gamma',\ul x) = i^\br(\Gamma,\ul x)$ from Proposition
  \ref{prop:expdim}.
\end{remark}

\section{Boundaries of rigid strata}\label{sec:bdry}

In the moduli space of maps, a stratum is \em{rigid} if it does not
deform to another type; that is, it does not occur in the boundary of
a compactification of a different stratum.  To count isolated maps
under generic perturbations one must restrict attention to rigid
strata. In this section, we formally define rigid strata of broken and
unbroken maps, and describe the codimension one boundary components of
rigid strata.

For unbroken maps, rigid types are those whose domain type is a
top-dimensional stratum in the moduli space of treed disks. We recall
that top-dimensional strata in the moduli space of treed disks
correspond to disk types which do not have any interior nodes, and at
boundary nodes, the length of the treed segment is finite,
 see Figure \ref{fig:disk-moduli}. 
Rigidity also includes the condition that the
map has simple intersections with the stabilizing divisor, because
maps with non-simple intersections can be deformed to maps with
multiple simple intersections.

\begin{definition}{\rm(Rigid types of unbroken
    maps)} \label{def:unbrokenrigidtype} \index{Rigid! unbroken map}
  The combinatorial type $\Gamma$ of a (unbroken) treed holomorphic
  map is \em{rigid} if the only edges $e \in \Edge_-(\Gamma)$
  corresponding to nodes are boundary edges
  $e \in \Edge_{\white}(\Gamma)$ of finite non-zero length
  $\ell(e) \in (0,\infty)$, and intersection multiplicity
  $m_{w_e}(u,D_P)$ of $u$ of type $\Gamma$ at $T_e \cap S$ with the
  stabilizing divisor $\bD= (\cD_P, P \in \cP)$ is $1$.
\end{definition}

Broken maps necessarily have tropical nodes, therefore in the
definition of rigidity of broken maps, the \em{no interior nodes}
condition of the unbroken case is replaced by the condition that the
tropical graph is rigid.
\begin{definition}{\rm(Rigid types of broken maps)}
  \label{def:rigidtype}
  \index{Rigid! broken map} The combinatorial type $\Gamma$ of a
  broken treed holomorphic map is \em{rigid} if the tropical graph
  $\cT(\Gamma)$ is rigid, and the only internal (non-tropical) edges
  $e \in \Edge_\internal(\Gamma)$ are boundary edges
  $e \in \Edge_{\white}(\Gamma)$ of finite non-zero length
  $\ell(e) \in (0,\infty)$, and intersection multiplicity
  $m_{w_e}(u,D_P)$ of $u$ of type $\Gamma$ at $T_e \cap S$ with the
  stabilizing divisor $\bD= (\cD_P, P \in \cP)$ is $1$.
\end{definition}
\begin{remark}
  The rigidity for types of broken maps includes the additional
  condition of the rigidity of the tropical graph because any
  non-rigid tropical graph $\Gamma$ is related to a rigid tropical
  graph $\Gamma_0$ by a tropical edge collapse morphism, and therefore
  tropical symmetry orbits of broken maps modelled on $\Gamma$ occur
  in the compactification of the moduli space of maps modelled on
  $\Gamma_0$.
\end{remark}

Next, we address the question of which strata of broken maps occur in
the compactification of a one-dimensional moduli space
$\M_\Gamma^\br(\ul x)$ where $\Gamma$ is a rigid type. The following
result proves that the compactification does not contain broken maps
of \em{crowded type} (Definition \ref{def:crowded}). Roughly speaking,
being of crowded type means that there are no ghost components with
more than one interior marking. We recall that interior markings
correspond to intersections with the stabilizing divisors, and in the
context of broken maps, a ghost component is a domain component on
which the map is \em{horizontally constant} (Definition
\ref{def:iso-stab} \eqref{part:horizconst}).  Once this is established, we prove in
Proposition \ref{prop:truebdry} that the only strata that occur in the
codimension one boundary are those that contain either a broken edge,
or an edge with length zero.

\begin{proposition}\label{prop:noghosts}{\rm(No crowded strata in the compactification)} 
  Let $\ul \Pe$ be a regular perturbation datum.  Let $\Gamma$ be a
  type of an uncrowded rigid broken map and let
  $\ul x \in (\cI(L))^{d(\white)+1}$ be a tuple of limit points of
  boundary leaves such that the expected dimension of the moduli space
  $\M_\Gamma(L, \ul \Pe, \ul x)$ is at most $1$.  Then, for any curve
  $u$ in the compactification $\ol{\M}_{\Gamma}(L,\ul \Pe,\ul x)$,
  \begin{itemize}
  \item there are no horizontally constant components that contain interior
    markings, and
  \item  $u$ has simple intersections with the stabilizing
  divisor.
  \end{itemize}
\end{proposition}

\begin{proof}
  First, we point out that \cwl{in the case that the Lagrangian is
    contained in the interior of $\ol X_{P_0}$} a horizontally
  constant component containing an interior marking 
  cannot be a disk
  component. Indeed, the Lagrangian $L$ is contained in a
  top-dimensional cut space $\XB_{P_0}$ and the torus $T_{P_0,\C}$ is
  trivial. Therefore, a horizontally constant disk component
  $u_v: S_v \to \XB_{P_0}$ is in fact constant and maps to a point in
  $L$. Since $L$ is disjoint from the stabilizing divisor there are no
  interior markings on $S_v$.

  Next, we rule out horizontally constant spherical components that
  contain markings. Suppose the map
  $u \in \M_\Gamma(L, \ul \Pe, \ul x)$ is horizontally constant on a
  component $S_v$, for a vertex $v$ of $\Gamma$, and suppose $S_v$ has
  an interior marking $z_e$.
  Since $\Gamma$ is an uncrowded stratum, $S_v$ has no other
  interior marking besides $z_e$. The component has at most two nodes
  -- otherwise moving
  the marking $z_e$ on the component $S_v$ gives a two-dimensional
  family of regular maps adapted to the stabilizing divisor.  By
  rigidity of the type $\Gamma$, all edges $e$ corresponding to
  interior nodes $w_e$ are tropical edges. Since $u|_{S_v}$ is
  horizontally constant, the balancing property \eqref{eq:constt-bal}
  implies that there are exactly two edges $e_1,e_2 \in \Edge(\Gamma)$
  incident on $v$, and they have the same {direction} $\cT(e_1) = \cT(e_2)$.
  This means the type $\Gamma$ has a non-trivial tropical symmetry
  group $T_\trop(\Gam)$, contradicting the rigidity of $\Gamma$.
  (See related Example \ref{ex:triv1}.) 

  Next, we consider a map in a boundary of the moduli space.  Suppose
  $\{u_\nu\}_\nu$ is a sequence in $\M_{\Gamma}(L,\ul \Pe,\ul x)$ that
  converges to a limit $u:C \to \XX$ of type $\Gamma'$.  First we prove
  the proposition assuming that the limit $u$ is uncrowded. Then $u$
  is a regular map adapted to the stabilizing divisor. Since the index of $u_\nu$ is at most
  $1$, standard arguments (see the proof of Proposition
  \ref{prop:truebdry}, for example) imply that the type $\Gamma$ of
  $u_\nu$ is obtained from the limit type $\Gamma'$ by the morphism
  (Making an edge length finite/non-zero). In particular, the tropical
  graph associated to $\Gamma'$ is the same as that of $\Gamma$.  By
  the argument in the previous paragraph, there are no markings in
  horizontally constant components of the limit map $u$.

  Next, consider the case that the limit $u$ has a crowded
  component. Starting from $u$, we obtain a map $u'$ of
  uncrowded type as follows: Forgetting all but one interior marking
  on each of the crowded components yields a map $u'$ adapted to the stabilizing divisor.  If in
  this process, a domain component $S_v \subset S$ (belonging to a
  crowded component of $u$) becomes unstable, it is collapsed. Indeed,
  the vertex $v$ has valence $1$ or $2$, and in either case, there is
  an obvious way to drop $v$ from the graph. Denote by $\Gamma_s$ the
  resulting type of the map.  In $\Gamma_s$, the remaining interior
  marking, corresponding to the leaf $e$, is assigned a multiplicity
  of $\mu_\DD(e)$ plus the sum of multiplicities $\mu_\DD(e')$ of all the forgotten leaves $e' \neq e$.  The
  limit $u'$ is $\Pe_{\Gamma_s}$-holomorphic because of the (Locality
  axiom) in Definition \ref{def:pertmorph}.  Indeed, forgetting
  markings changes the type of the limit curve, but it does not affect
  the perturbation datum $\Pe_\Gamma$ on the other curve components on
  which the map is horizontally non-constant. Therefore, $u'$ is
  regular.  If no component is collapsed in $\Gamma' \to \Gamma_s$,
  then, the expected dimension of the type $\Gamma_s$ is the same as
  that of $\Gamma'$. In this case, $u$ is an uncrowded map with a
  marking in a horizontally constant component. This possibility has
  been ruled out in the last paragraph. If a component is collapsed,
  the expected dimension of $\Gamma_s$ is at least two lower than that
  of $\Gamma$, and therefore, the map $u$ does not exist.

  Finally, we show that for a map $u$ in the compactification
  $\ol \M_{\Gamma}(L,\ul \Pe,\ul x)$, the intersections with the
  stabilizing divisor are simple.  We assume that $u$ is the limit of a
  sequence of maps $\{u_\nu\}_\nu$ in $\M_{\Gamma}(L,\ul \Pe,\ul x)$. 
  Since $\Gamma$ is rigid, the maps $u_\nu$ have simple intersections
  with $\DD$ and each intersection is an interior marking.  The map
  $u$ can have a non-simple intersection with $\DD$ only if two
  sequences interior markings $z_{e,\nu}$, $z_{e',\nu}$ on the maps
  $\{u_\nu\}_\nu$ coincide in the limit, which means that $u$ has a
  ghost component containing the two markings.  Since we have shown
  that $u$ is uncrowded, we have arrived at a contradiction.
\end{proof}

\begin{proposition}\label{prop:truebdry}
  {\rm(Boundary strata)} Let $\ul \Pe$ be a regular perturbation
  datum.  Suppose $\Gamma$ is a rigid type for a broken map and
  $\ul x \in (\cI(L))^{d(\white)}$ is a tuple of limit points of
  boundary leaves such that the expected dimension $i(\Gamma,\ul x)$
  is $\leq 1$.  The moduli space $\M_{\Gamma}(L, \ul \Pe, \ul x)$
  admits a compactification $\ol \M_{\Gamma}(L, \ul \Pe, \ul x)$.  The
  boundary strata
\[\ol \M_{\Gamma}(L, \ul \Pe, \ul x)\bs \M_{\Gamma}(L, \ul \Pe, \ul x)\]
correspond to types $\Gamma'$ with a single broken boundary trajectory
$e \in \Edge_{\white}({\Gamma'}), \ell(e) = \infty$ or a single
boundary edge with length zero $\ell(e)=0$.
\end{proposition}
\begin{proof}
  For a sequence of maps $u_\nu: C \to \XX$ in the moduli space
  $\M_{\Gamma}(L, \ul \Pe, \ul x)$, a subsequence converges to a limit
  $u:C \to \XX$ by Theorem \ref{thm:cpt-broken}.  By Proposition
  \ref{prop:noghosts}, $u$ is uncrowded and therefore, $u$ is regular
  and adapted to the stabilizing divisor.  Interior nodes $w_e$ corresponding to non-tropical
  edges $e \notin \Edge_\trop(\Gamma)$ are ruled out in $u$ for
  dimension reasons using the index relation \eqref{eq:expdim}. We
  next claim that the tropical type of $u$ is $\Gamma$. If not, it is
  of type $\Gamma'$, and there is a non-trivial tropical edge collapse
  morphism $\Gamma' \to \Gamma$. The last statement in Theorem
  \ref{thm:cpt-broken} implies that $\dim(T_\trop(\Gam')) \geq 2$.  By
  Proposition \ref{prop:expdim}, the expected dimension of the moduli
  space of broken maps is unaffected by collapsing tropical
  edges. Therefore, the index of $u$ is at most $1$.  However, the
  positive dimensionality of the tropical symmetry group
  $T_\trop(\Gam')$, and the fact that the action of the tropical
  symmetry group does not have infinitesimal stabilizers implies that
  $i^\br(u) \geq 2$, leading to a contradiction. We conclude that the
  tropical type of $u$ is the same as that of the sequence $u_\nu$.
  The only other phenomenon which occurs in the limit of
  $\{u_\nu\}_\nu$ is the formation of a boundary node $w \in C$
  corresponding to an edge $e$ of length $\ell(e)$ zero, or the length
  of a boundary edge $\ell(e)$ going to zero or infinity.
\end{proof}

\chapter{Gluing} \label{chap:glue}

In this Chapter, we show that a rigid broken map can be glued at nodes
to produce a family of unbroken maps, one in each neck-stretched
manifold.  A fixed perturbation datum $\ul{\Pe}$ on a broken manifold
$\XX$ can be glued in a natural way to produce a perturbation datum
$\ul{\Pe}^\nu$ for $X^\nu$ which is equal to $\ul{\Pe}$ away from the
neck regions.  With respect to these perturbation data, we construct a
bijection between rigid broken map and rigid maps in neck-stretched
manifolds $X^\nu$ with sufficiently large neck lengths.
We recall from \eqref{eq:moduli-d} that $\tM^{\br,< E}(\XX,L)_0$ is the union of zero-dimensional strata in the moduli space of broken maps whose energy is at most $E$.

\begin{theorem} {\rm (Gluing)} \label{thm:gluing} Suppose that
  $\cu: C \to \XX$ is a regular broken disk of index zero.  There
  exists $\nu_0 > 0$ such that if $\nu \geq \nu_0$ then there exists a
  family of unbroken disks $u_\nu: C^\nu \to X^{\nu}$ of index zero,
  with the property that $\lim_{\nu \to \infty} [u_\nu] = [\cu]$.  For
  any area bound $E >0$ there exists $\nu_0$ such that for
  $ \nu \geq \nu_0$ the correspondence $[u] \mapsto [u_\nu]$ defines
  an orientation-preserving bijection between the rigid moduli spaces
  $\M^{< E}(X^\nu,L)_0$ and $\tM^{\br,< E}(\XX,L)_0$ 
  for $\nu \geq \nu_0$.
\end{theorem}
\begin{remark}
  The gluing operation is defined on broken maps, and not on tropical
  symmetry orbits of broken maps. In fact if the tropical symmetry
  group is non-trivial (it is necessarily finite for rigid broken
  maps), gluing different elements in the tropical symmetry orbit
  produces distinct sequences of unbroken maps.  The convergence
  result also distinguishes between rigid maps in the same tropical
  symmetry orbit.  Indeed, in Theorem \ref{thm:cpt-breaking}, if the
  limit map is rigid then the limit is uniquely determined up to
  domain reparametrization.
\end{remark}
Similar to other gluing theorems in pseudoholomorphic curves, the
proof of Theorem \ref{thm:gluing} is an application of a quantitative
version of the implicit function theorem for Banach manifolds.  The
steps are: construction of an approximation solution called the
pre-glued map; construction of an approximate inverse to the
linearized operator; quadratic estimates; application of the
contraction mapping principle, and surjectivity of the gluing
construction.  Through the proof of the gluing theorem, the notation
$c$ denotes a $\nu$-independent constant whose value may be different
in every occurrence. The fact that gluing preserves the orientation of
rigid moduli spaces is discussed in Remark \ref{rem:orientglue}.

\section{The approximate solution} \label{sec:approx} 

Given a broken map, we start by constructing a family of approximate solutions, called
\em{pre-glued maps}, with one map in each neck-stretched manifold $X^\nu$. These maps are obtained by
gluing the broken map in the neck region using a cut-off function, and therefore, they are not
holomorphic in the neck regions. 

The pre-glued family for a broken map is constructed using tropical
vertex positions of the broken map. We recall that a rigid broken map is modelled on a
rigid tropical graph $\Gamma$, whose vertex positions 
$\{\cT(v): v \in \Ver(\Gamma)\}$ are uniquely determined. 
These positions determine the neck lengths for the
approximate solution as follows.  For any edge $e=(v_+,v_-)$ of
$\Gamma$, there exists $l_e>0$ such that
\begin{equation}
  \label{eq:ledef}
  \cT(v_+) -\cT(v_-)=l_e\cT(e), 
\end{equation}
where $\cT(e) \in \t_\Z$ is the {direction} of the edge.

The domain of any map in the glued family is a treed disk obtained by
replacing tropical nodes in $C$ with necks and leaving the tree part
in $C$ unchanged: For any $\nu \in \R_+$ the curve $C_\nu$ has a neck
of length $\nu l_e$ in place of a tropical node $w_e$ corresponding to
$e \in \Edge_\trop(\Gamma)$ in $C$.  Denote by
\begin{equation}
  \label{eq:gamma-glue}
  \Gamma_{\glue}  
\end{equation}
the type of the glued curve.  Recall that the lift of a tropical node
$w_e$ in $C$ has matching coordinates (as in Definition
\ref{def:bmap}) in the neighbourhoods $U_e^+$, $U_e^-$ of the lifts
$w^+_e$, $w^-_e$, which are denoted by
\begin{equation}
  \label{eq:match4glue}
  S_{v_\pm} \supset (U_e^\pm,w_e^\pm) \xrightarrow{z_e^\pm} (\C,0).   
\end{equation}
The coordinates can be chosen to be compositions of the complex
exponential map \eqref{eq:cexp} and linear functions from tangent
spaces to $\C$.  The glued curve $C_\nu$ is obtained from $C$ by
deleting a small disk in $U_e^\pm $ for every edge $e$ in $\Gamma$,
and gluing the remainder of the neighbourhoods $U_e^\pm$ using the
identification $z_e^+ \sim e^{-\nu l_e}z_e^-$, and leaving the tree
part $T$ unchanged. So the surface and tree parts of $C_\nu$ are
\[C_\nu=S_\nu \cup T.\]
For future use in the proof we point out that the punctured
neighbourhoods $U_e^\pm \bs \{w_e^\pm\}$ have matching logarithmic
coordinates
\[(s_e,t_e) : U_e^- \bs \{w^-_e\} \to [0,\infty) \times S^1, \quad
  (s_e,t_e) : U_e^+ \bs \{w^+_e\} \to (-\infty,0] \times S^1\]
given by $(s_e,t_e):=\pm \ln(z_e^\pm)$.

Translated maps can be glued at the nodal points to yield a sequence
of approximate solutions for the holomorphic curve equation in
neck-stretched manifolds.  We recall that for a rigid tropical graph,
the translation sequences are unique and are given by
\begin{equation}
  \label{eq:tnu}
  t_\nu: \Ver(\Gamma) \to \nu B^\dual, \quad v \mapsto \nu \cT(v).  
\end{equation}
Translation sequences give identifications between broken manifolds
and neck-stretched manifolds.  For any $\nu$, $v \in \Ver(\Gamma)$ the
map
\[\e^{t_\nu(v)} : \XC_{P(v)} \to X^\nu_{\tP} \subset X^\nu\]
is defined in \eqref{eq:transinc2}.  The complement of the neck region
in $S^\nu$ is a disjoint union
\[S^\nu \bs \Neck(S^\nu)=\cup_{v \in \Ver(\Gamma)}S^\nu(v).\]
On $S^\nu(v)$, the approximate solution is defined as the translation
of the map $\cu_v$, that is,
\begin{equation}
  \label{eq:preoffneck}
  u^\pre_\nu:=\e^{t_\nu(v)}\cu_v \quad \text{on $S^\nu(v)$}.   
\end{equation}
On the treed part $T \subset C^\nu$, we have
\[u^\pre_\nu:=\cu|_T \quad \text{on $T \subset C^\nu$}. \]
Next, we define $u^\pre_\nu$ on the neck region $\Neck_e(S^\nu)$
corresponding to a tropical edge
$e=(v_+,v_-) \in \Edge_\trop(\Gamma)$.  The matching condition on
$\cu$ at the tropical node $w_e$ implies that there is a point
$x_e \in \XC_{P(e)}$ such that the map $u_{v_\pm}$ is asymptotically
close to the trivial cylinder
\begin{equation}
  \label{eq:xe}
  u_e^\ver : \R_\pm \times S^1 \to \XC_{P(e)}, \quad (s,t) \mapsto
  e^{\cT(e)(s+it)}x_e  
\end{equation}
near the nodal lift $w_e^\pm \in C_v^\pm$. That is, if we define
$\zeta_e^\pm \in T_{u_e^\ver}\XC_{P(e)}$ by the condition
\begin{equation}
  \label{eq:xipmdef}
  \cu_{v_\pm}=\exp_{u_e^\ver}\zeta_e^\pm, 
\end{equation}
then
\begin{equation}
  \label{eq:enddecay}
  \Mod{D^k\zeta_e^\pm(s,t)} \leq ce^{-|s|}  
\end{equation}
for any $k \geq 0$.  Indeed, in an almost complex compactification of
the target space $\XX_{P(v_\pm)}$ of $u^0_{v_\pm}$ as in
\eqref{eq:xxpt-def}, the map extends smoothly over $\mp \infty$, from
where we obtain the decay \eqref{eq:enddecay} in cylindrical domain
coordinates.  In \eqref{eq:xipmdef} we note that the target space of
the map $\cu_{v_\pm}$ resp. $\exp_{u_e^\ver}\zeta_e^\pm$ is
$\XC_{P(v_\pm)}$ resp. $\XC_{P(e)}$, but the relation
\eqref{eq:xipmdef} is well-defined because near the node the image of
$u_{v_\pm}$ lies in the $P(e)$-cylindrical end of $\XC_{P(v_\pm)}$
which is canonically identified to $\XC_{P(e)}$.  Define a cylinder
in $\Neck_e(C_\nu)$ corresponding to $u_e^\ver$ 
by 
\[u_{e,\nu}^{\ver}:=\e^{\hh(t_\nu(v_+) + t_\nu(v_-))}u_e^\ver : [\tfrac {-\nu l_e} 2, \tfrac {\nu l_e} 2] \times S^1
  \to X^\nu. \]
The approximate solution on $\Neck_e(C_\nu)$ is the trivial cylinder
$u_{e,\nu}^{\ver} $ corrected by a section obtained by patching
$\zeta_e^+$, $\zeta_e^-$ defined as
\begin{multline} \label{eq:preglued} u^{\pre}_\nu(s,t) :=
  \exp_{u_{e,\nu}^{\ver}} ( \zeta^\nu_e(s,t)) : [\tfrac {-\nu l_e} 2, \tfrac {\nu l_e} 2] \times S^1
  \to X^\nu, \\ \quad
  \zeta^\nu_e(s,t) := \beta(-s) \zeta_e^-(s + \tfrac{\nu l_e}{2}, t) +
  \beta(s ) \zeta_e^+(s - \tfrac {\nu l_e}{2},t) .
\end{multline}
where 
\begin{equation} \label{beta} \beta \in C^\infty(\R, [0,1]), \quad
  \begin{cases} 
    \beta(s) = 0 & s \leq 0 \\ \beta(s) = 1 & s \ge
    1 \end{cases} \end{equation}
is the cutoff function from \eqref{beta0}.  In other words, in the
cylinder in $C_\nu$ corresponding to an edge $e$, one translates the
domain on both ends by an amount $\frac {\nu l_e}{2}$, and then
patches the map together using the cutoff function and geodesic
exponentiation.  The definition \eqref{eq:preoffneck} of $u_\nu^\pre$
in the complement of the neck, and its definition \eqref{eq:preglued}
on the neck together give a smooth map $u_\nu^\pre:C_\nu \to X^\nu$.
Indeed, using the relation $t_\nu(v_+) - t_\nu(v_-)=\nu l_e$, we
conclude that the expression in \eqref{eq:preglued} is equal to
$\e^{t_\nu(v_\pm)}\cu_{v_\pm}$ in the neighborhood of the boundary
component $ \tfrac{\pm \nu l_e}{2} \times S^1$.

\begin{remark}\label{rem:glueparam}\index{Map gluing parameter}
  The number $\nu \in \R_+$ is called the \em{map gluing parameter}
  to differentiate it from the parameter used for gluing nodes in a
  curve.  In standard gluing results, the map gluing parameter $\nu$
  is taken to be the neck length in the domain $C_\nu$ of the glued
  map $u^\nu$. In our setting, the map gluing parameter is the neck
  length $\nu$ in the target space $X^\nu$. The domain neck lengths,
  approximately equal to $\nu l_e$, are allowed to vary in the Picard
  iteration argument.  Note that resolving nodes $w_e$ corresponding
  to tropical edges $e$ does not increase the index of the map. As a
  result gluing an index zero broken map $u$ produces an isolated map
  $u^\nu$ in each $X^\nu$.
\end{remark}
\section{Fredholm theory for glued maps}
\label{sec:fredglued}

We define a map between suitable Banach spaces whose zeroes describe
pseudoholomorphic curves close to the approximate solution.
Pseudoholomorphic maps are zeroes of the section
\begin{multline} \label{eq:Fmap}
\F_\nu:=(\delbar,\ev): \M_{\Gamma_{\on{glue}}} \times \Map(C_\nu,X^\nu) \to \\
\Om^{0,1}(S_\nu,u^*TX) \oplus \Om^1(T,u^*TL) \oplus X(\Gamma_\glue) \end{multline} 
defined in \eqref{olp3}.  The components of the map consist of
$\delbar$ operators on surfaces, shifted gradient operators on treed
segments, and various evaluation maps. In this section we define
metrics/norms in each of the spaces. We recall the operator $\F_\nu$
later in the section.

We describe a metric on the moduli space of treed disks.  Recall that
the domain of $u$ is $C$, which is a curve of type $\Gamma$ with
surface part $(S,j_S)$ and tree part $T$, and recall from
\eqref{eq:gamma-glue} that the type $\Gamma_\glue$ of the glued curve
is given by collapsing all the interior edges in $\Gamma$.  In a
neighborhood $U_{\M_\Gamma} \subset \ol \M_{\Gamma_\glue}$ of
$\M_\Gamma$, there is a projection map
\[\pi_\Gamma : U_{\M_\Gamma} \to \M_\Gamma \]
such that any curve $C' \in U_{\M_\Gamma}$ is obtained by gluing at
tropical nodes of the curve $\pi_\Gamma(C')$ using the coordinates
\eqref{eq:match4glue}. We remark that the domain curves $C_\nu$ of the
approximate solutions which were constructed by gluing the tropical
nodes in $C$ lie in the neighborhood $U_{\M_\Gamma}$, and
$\pi_\Gamma(C_\nu)=C$. The subset
$U_{\M_\Gamma} \cap \M_{\Gamma_\glue}$ close to the boundary stratum
$\M_{\Gamma}$ is equipped with a metric
\begin{equation}
  \label{eq:gm}
  g_{\Gamma_\glue}:    T(\M_{\Gamma_\glue} \cap
  U_{\M_\Gamma})^{\otimes 2} \to \R 
\end{equation}
that is cylindrical in the fibers of $\pi_\Gamma$. That is, each fiber
of $\pi_\Gamma$ is isometric to a product of cylinders
$\Pi_{e \in \Edge_\trop(\Gamma)}(\R \times S^1)$ parametrized by
gluing parameters $(s_e,t_e)$.

In the neighborhood of each of the glued curves, we give a description
of the complex structures on the surface components. As a preliminary
step, we describe the complex structures on curves close to the domain
$C$ of $u$.  Let the complex curve $(S,j_S)$ be the surface part of
$C$.  A trivialization
\begin{equation}
  \label{eq:sgam-triv}
  \S_\Gamma|U_{\Gamma,C} \simeq S \times U_{\Gamma,C}  
\end{equation}
of the universal curve $\S_\Gamma \to \M_\Gamma$ in a neighborhood
$U_{\Gamma,C} \subset \M_\Gamma$ of $C$ yields a family of complex
structures (as in \eqref{eq:jmoduli})
\[U_{\Gamma,C} \to \J(S), \quad m \mapsto j_\Gamma(m).\]
We write $j_\Gamma$ as a sum
\[j_\Gamma(m)=j_S + \fj_\Gamma(m)\]
and assume that the trivialization in \eqref{eq:sgam-triv} is chosen
so that $ \fj_\Gamma(m)=0$ in neighborhoods of special points.
Consider a glued treed disk $C_\nu \in \M_{\Gamma_\glue}$ whose
surface part $(S_\nu,j_{S_\nu}) \subset C_\nu$ is obtained by gluing
the tropical nodes in $S \subset C$.  In a neighborhood
$U_{C_\nu} \subset \M_{\Gamma_\glue}$, choose a trivialization of the
universal curve $\S_{\Gamma_\glue} \to \M_{\Gamma_\glue}$ so that the
induced family of complex structures
\begin{equation}
  \label{eq:jnu}
  U_{C_\nu} \to \J(S_\nu), \quad m \mapsto  j_\nu(m),  
\end{equation}
satisfies the following : The function $j_\nu$ is a sum
\[j_\nu = j_{S_\nu} + \fj_\black(m) +
  \fj_{\on{neck}}^\nu(m)\]
where
\begin{itemize}
\item $j_{S_\nu}$ is the complex structure on the glued curve $S_\nu$
  and is constant over $U_{C_\nu}$,
\item the function $m \mapsto \fj_\black(m)$ is supported in the
  complement of the neck regions of $S_\nu$ and is equal to
  $\fj_\Gamma(\pi_\Gamma(m))$,
\item and $\fj_{\neck}^\nu(m)$ is supported in the neck regions of
  $S_\nu$, and the support is contained in a uniformly bounded
  neighborhood the boundary of the neck, that is, there is a
  $\nu$-independent constant $L$ such that
  \begin{equation}
    \label{eq:suppfj}
  \supp(\fj^\nu_{\on{neck}}) \subset \bigcup_{e \in
      \Edge_\trop(\Gamma)}\{s:\tfrac {\nu l_e} 2 -L \leq |s| \leq
    \tfrac {\nu l_e} 2\} \times S^1 \subset A(l_\nu) \subset S_\nu.  
  \end{equation}
  Furthermore, there is a $\nu$-independent constant $c$ such that on
  the neck region corresponding to any edge
  $e \in \Edge_\trop(\Gamma)$ and for any $\nu$
  \begin{equation}
    \label{eq:neckC1}
    \Mod{\fj^\nu_{\on{neck}}(m)}_{C^1} \approx c|\nl_e(m) - \nu l_e|,   
  \end{equation}
  where $\nl_e(m) \in \R_+ \times S^1$ is the gluing parameter used at
  the node $e$ to smoothen the node $w_e$ in the curve
  $\pi_\Gamma(m)$.
\end{itemize}
Indeed, \eqref{eq:neckC1} can be satisfied by making suitable choices,
as it amounts to choosing a diffeomorphism from a cylinder of length
$\delta_e(m)$ to a cylinder of length $\nu l_e$ that is a translation
map outside a length $L$ sub-cylinder, and whose derivative is bounded
as in \eqref{eq:neckC1}.  Such a choice of the trivialization of the
universal curve ensures that there is a $\nu$-independent constant $c$
such that for any two curves represented by $m_1$, $m_2$
\begin{equation}
  \label{eq:fjcont}
  |\fj^\nu(m_2)-\fj^\nu(m_1)| \leq c d_{g_{\Gamma_\glue}}(m_1,m_2).   
\end{equation}
These uniform estimates are used in the proof of the quadratic
estimate in Section \ref{sec:quadratic}.

The second space in the domain of \eqref{eq:Fmap} is a space of
$W^{1,p}_\loc$ maps
\[ \Map(C_\nu,X^\nu)_{1,p} \subset \Map(S_\nu,X^\nu)_{1,p} \times
\Map(T,X^\nu)_{1,p}  \]
defined by requiring that the maps from $S_\nu$ and $T$ agree on the
intersection $S_\nu \cap T$.  We recall that $T$ is the disjoint union
of treed segments $\cup_e T_e$ corresponding to boundary edges
$e \in \Edge_\white(\Gamma)$ with positive lengths; and $T_e=[0,1]$ if
the segment is finite, and $\R_{\geq 0}$, $\R_{\leq 0}$ or $\R$ if
$\ell(e)=\infty$.  The tangent space of $\Map(C_\nu,X^\nu)_{1,p}$ at a
map $u:C_\nu \to X^\nu$ is the space of sections
\begin{equation}
  \label{eq:om0cnu}
  \Omega^0(C_\nu, u^* TX) = \Omega^0(S_\nu, (u|_S)^*TX^\nu) \oplus \Omega^0(T, (u|_T)^* TL).   
\end{equation}
As in Abouzaid \cite[Equation (5.38)]{ab:ex} the first summand in
\eqref{eq:om0cnu} is equipped with a weighted Sobolev norm based on
the decomposition of the section into a part constant on the neck and
the difference on the neck corresponding to each edge
$e \in \Edge_\trop(\Gamma)$ described as follows.  Denote by
\[ (s_e,t_e) \in [-\nu l_e/2, \nu l_e/2] \times S^1 \]
the coordinates on the neck region created by the gluing at the node
corresponding to the edge $e \in \Edge_\trop(\Gamma)$.  Let
\[ \lambda \in (0,1) \]
be a \em{Sobolev weight}.  Define a \em{Sobolev weight function}
\begin{equation} \label{kappa} \kappa_\nu: C_\nu \to [0,\infty), \quad
  \kappa_\nu := \ssum_{e \in \Edge_\trop(\Gamma)}\beta(\nu l_e/2 -
  |s_e|) ( \nu l_e/2 - |s_e| ).\end{equation}
Here, the function $\beta(\nu l_e/2 - |s_e|)$ is extended by zero
outside the neck region corresponding to $e$.  As $\nu \to \infty$,
$\kappa_\nu$ converges to the Sobolev weight function $\kappa$ defined
on the punctured curve $C - \{w_e : e \in \Edge_\trop(\Gamma)\}$ in
\eqref{eq:kappainf}.  Given a section
\[ \xi = (\xi_S,\xi_T) \in
\Omega^0(C_\nu, u^* TX) \] 
define
\begin{multline} \label{1pl2} \Vert \xi \Vert_{1,p,\lambda} := \Vert
  \xi_S \Vert_{1,p,\lambda}
  + \Vert \xi_T \Vert^p_{1,p} \\
  \Vert \xi_S \Vert_{1,p,\lambda} := \left( \ssum_e\Vert
    \xi_{S,e}(0,0) \Vert^p + \int_{C_\nu} ( \Vert \nabla \xi_S \Vert^p
  \right. \\ \left. + \Vert \xi_S - \ssum_e\beta(\nu l_e/2 - |s_e|)
    \pT^u_s \xi_{S,e}(0,0) \Vert^p ) \exp( \kappa_\nu \lam p) \d
    \Vol_{C_\nu} \right)^{1/p}
\end{multline}
where $\xi_{S,e}$ is the restriction of $\xi_S$ to the neck region
$[\frac {-\nu l_e} 2, \frac {\nu l_e} 2] \times S^1$ corresponding to
the edge $e$, and $\pT^u_s$ is parallel transport from $u(0,t)$ to
$u(s,t)$ along $u(s',t), s' \in [0,s]$.  Let
$ \Omega^0(C_\nu, u^* TX)_{1,p,\lam}$ be the Sobolev completion of
$W^{1,p}_{\loc}$ sections with finite norm \eqref{1pl2}; these are
sections whose difference from a covariant-constant section on the
neck has an exponential decay behavior governed by the Sobolev
constant $\lambda$.

The target space of \eqref{eq:Fmap} is a space of $(0,1)$-forms, which
we equip with a weighted $L^p$ norm.  For a $(0,1)$-form
$\eta \in \Omega^{0,1}( S_\nu, u^* TX^{\nu})$ define
\[ \Vert \eta \Vert_{0,p,\lambda} = \left( \int_{S_\nu} \Vert \eta
  \Vert^p \exp( \kappa_\nu \lambda p) \d \Vol_{S_\nu} \right)^{1/p}
.\]

Next, we describe a $\delbar$-operator pulled back by an exponential
map on which the implicit function theorem will be applied.  Pointwise
geodesic exponentiation defines a map (using Sobolev multiplication
estimates)
\begin{equation}
  \exp_{u_\nu^{\pre}}: \Omega^0(C_\nu, (u_\nu^{\pre})^*
  TX^\nu)_{1,p,\lambda} \to \Map_{1,p}(C_\nu,X^\nu) \end{equation}
where $\Map_{1,p}(C_\nu,X^\nu)$ denotes maps of class $W_{1,p}^{\loc}$
from $C_\nu$ to $X^\nu$.
We define
\begin{multline}
  \label{eq:cnu-delbar}
  \delbar: U_{C_\nu} \times \Omega^0(C_\nu, (u_\nu^{\pre})^*
  TX^\nu)_{1,p,\lam} \to \Omega^{0,1}(S_\nu,
  (u_\nu^{\pre})^* TX^{\nu})_{0,p} \oplus \Om^1(T,(u_\nu^{\pre})^*TL), \\
  (m,\xi) \mapsto \pT_{u_\nu^{pre},\xi}^{-1}
  \left(\delbar_{j(m),J}(\exp_{u_\nu^{\pre}}\xi)|_{S_\nu},\left(\frac 1
  {\lam_e}\dds  + \grad_F\right)(
\exp_{u_\nu^{\pre}}\xi)|_T \right),
\end{multline}
where
\begin{multline*}
  \pT_{u_\nu^{\pre},\xi} : \Omega^{0,1}(S_\nu, (u_\nu^{\pre})^*TX^{\nu})_{0,p,\lambda} \oplus \Om^1(T, (u_\nu^{\pre})^*TL)_{0,p}\\
  \to
  \Omega^{0,1}(S_\nu, (\exp_{u_\nu^{\pre}}(\xi))^*TX^{\nu})_{0,p,\lambda}
  \oplus \Om^1(T, (\exp_{u_\nu^{\pre}}(\xi))^*TL)_{0,p}
\end{multline*}
is the parallel transport defined using an almost-complex connection,
and we recall that $U_{C_\nu} \subset \M_{\Gamma_\glue}$is a
neighborhood of $C_\nu$; and
\[\lam_e:=
  \begin{cases}
    \ell_e(m), \quad e \in \Edge_{\white,-}^{(0,\infty)}(\Gamma),\\
    1, \quad e \in \Edge_{\white}^\infty(\Gamma).
  \end{cases}
\]
In \eqref{eq:cnu-delbar}, it is necessary to scale the derivative
$\dds$ by $\lambda_e$ because we took $T_e=[0,1]$ for any finite treed
segment.

In order to construct local models for moduli of adapted tree disks,
we require that the treed disks $C_\nu$ have a collection of interior
leaves $e_1,\lldots, e_{d(\black)}$ and
\[ (\exp_{u_\nu^{\pre}} (\xi) )(e_i) \in D, \quad i = 1,\lldots, n .\]
Additionally, we require matching conditions at boundary nodes and
lifts of $S_\nu \cap T_\nu$. Using notation from the proof of
transversality (Theorem \ref{thm:transversality}), these constraints
may be incorporated into $\cF_\nu$ from \eqref{eq:Fmap} to produce a
map
\begin{multline*}
  \cF_\nu: U_{C_\nu} \oplus \Omega^0(S_\nu, (u_\nu^{\pre})^*
  TX^{\nu})_{1,p,\lam} \oplus \Omega^0(T, (u_\nu^{\pre})^*
  TL)_{1,p}  \\
  \to \Omega^{0,1}(C_\nu, (u_\nu^{\pre})^* TX^{\nu})_{0,p, \lam}\oplus
  \Om^1(T, (u_\nu^{\pre})^*TL)_{0,p} \oplus T
  X(\Gamma_{\on{glue}})/\Delta(\Gamma_{\on{glue}}).
\end{multline*}
whose zeroes correspond to \em{adapted} pseudoholomorphic maps near
the pre-glued map $u_\nu^{\pre}$.

\section{Error estimate}\label{sec:errorest}

We estimate the failure of the approximate solution to be an exact
solution in the Banach norms of the previous section.
To derive the
estimate, we split the treed disk $C_\nu$ into neck regions
corresponding to tropical nodes in $C$, namely
\[\Neck_e(S_\nu):=\{(s_e,t_e) \in [-\tfrac {\nu l_e} 2, \tfrac {\nu l_e} 2] \times S^1\} \subset S_\nu, \quad \forall e \in \Edge_\trop(\Gamma); \]
and its complement 
\[C_\nu^\black:=C_\nu \bs \cup_e \Neck_e(S_\nu), \quad\text{and} \quad C_\nu^\black=S_\nu^\black \cup T.\]
The one-form $\cF_\nu(0) $ has contributions from the difference
between the perturbation data $P(C)$ and $P(C_\nu)$ on the complement
of the neck regions, and from the cutoff function and the difference
between $J_{u}$ and $J_{u_\nu^{\pre}}$ on the neck regions :
\begin{multline} \label{eq:err1}
  \Vert \cF_\nu(0) \Vert_{L^{p,\lam}(S_\nu)} = \Mod{\delbar_{J_{u_\nu^{\pre}}} u^\nu_\pre}_{L^{p,\lam}(S_\nu^\black)} 
  + \sum_{e \in \Edge_\white(\Gamma)}\Mod{\tfrac 1 {\lam_e} + (\grad_F)_{u_\nu^\pre}}_{L^p(T_e)}\\
+\ssum_e\Vert \olp \exp_{e^{(s+it)\cT(e)}x_e} ( \beta(-s_e)
  \zeta_e^-(s_e + \nu l_e/2, t_e) \\+ \beta(s_e) \zeta_e^+(s_e - \nu
  l_e/2,t_e)) \Vert_{L^{p,\lam}(\Neck_e(S_\nu))}.
\end{multline}
The first two terms may not vanish because the perturbation is domain
dependent: in the complement of the neck regions, the map $u_\nu^\pre$
is $P(C)$-holomorphic but not $P(C_\nu)$-holomorphic.  For any metric
$d_{\ol \M}$ on the compactified moduli space $\ol \M_{\Gamma_\glue}$
of treed disks, the distance between the domain curves is bounded as
\[d_{\ol \M}(C_\nu, C) \leq \co \max_{e \in \Edge_\trop(\Gamma)} \exp(-\nu l_e).\]
Therefore, the distance between the domain-dependent perturbations has
a similar bound. On the complement of the necks $u_\nu^\pre$ is
$J(C)$-holomorphic, so
\begin{equation}
  \label{eq:nonneckbd}
\Mod{\delbar_{J_{u_\nu^{\pre}}} u^\nu_\pre}_{L^{p,\lam}(S_\nu^\black)} \leq \co\Mod{J(C)-J(C_\nu)}_{L^\infty} \leq \co \max_{e \in \Edge_\trop(\Gamma)} \exp(-\nu l_e).  
\end{equation}
The last term in the right-hand side of \eqref{eq:err1} is equal to  
\begin{multline*}
   \ssum_e \Vert ( D \exp_{e^{(s+it)\cT(e)}x_e} ( \d \beta(-s_e)
  \zeta_e^-(s_e + \nu l_e/2, t_e) + \d \beta(s_e ) \zeta_e^+(s_e - \nu
  l_e/2,t_e)) + \\ 
 ( \beta(-s_e) \d \zeta_e^-(s_e + \nu l_e/2, t_e)
  + \beta(s_e) \d \zeta_e^+(s_e - \nu l_e/2,t_e)) )^{0,1}
  \Vert_{L^{p,\lambda}(\Neck_e(S_\nu))}.
\end{multline*}
On the neck regions, the almost complex structure is
domain-independent since it is equal to the background almost complex
structure (see \eqref{eq:basej0}).  Holomorphicity of $u$ implies that
the terms are non-zero only in the support of $d\beta_e$ which is
contained in the interval
\[[-1,1] \times S^1 \subset [-\tfrac {\nu l_e} 2, \tfrac {\nu l_e} 2] \times S^1 = \Neck_e(S_\nu).\]
Both $\zeta_e^\pm$ and its derivative decay at the rate of $e^{-s_e}$
on the cylindrical end $\R_+ \times S^1$ in $S^\circ$, see
\eqref{eq:enddecay}. As a result both the difference between $J_{u}$
and $J_{u_\nu^{\pre}}$, and the terms containing $d\beta$ are bounded
by $\co e^{-l_e \nu/2}$ where $c$ is a constant independent of
$\nu$. The Sobolev weight function \eqref{kappa} has a multiplicative
factor of $e^{\lam l_e \nu/2}$, and therefore,
\begin{equation}
  \label{zeroth}
  \Mod{\F_\nu(0)} \leq \co\ssum_{e \in \Edge_\trop(\Gamma)}e^{-(1-\lam)l_e\nu/2} ,  
\end{equation}
with $\co$ a constant independent of $\nu$. (See Abouzaid
\cite[Equation (5.10)]{ab:ex}).

\section{Uniform right inverse}

In this Section, we construct a uniformly bounded right inverse for the linearized
operator of the approximate solution. Using the right inverses of the
pieces of the broken map, we construct an approximate inverse.  For
large neck lengths, the error of the approximate inverse is
shown to be small enough that an actual inverse can be constructed.

First, for any element in the target space of the linearized operator of the approximate solution,
we give a nearby element in the  target space of the linearized operator of the broken map.
We denote the linearized operator of the approximate solution by
\[D_{u^\pre_\nu}:=D\F_\nu([C_\nu],0).\]
Given an element
\[\eta=(\eta_S,\eta_T) \in \Omega^{0,1}(S_\nu,
  (u^{\pre}_\nu|S_\nu)^* TX^\nu)_{L^{p,\lam}} \oplus \Om^1(T,
  (u^{\pre}_\nu|T)^*TL)_{L^p} \]
in the target space of $D_{u^\pre_\nu}$ one obtains an element in the
target space of the linearized operator $D_{\cu}$ of the broken map
\[ \tilde{\eta} = (\eta_v)_{v \in \Ver(\Gamma)} \oplus \eta_T, \quad \eta_v \in
\bigoplus_{v \in \Ver(\Gamma)} \Omega^{0,1}(S_v^\circ, (\cu_v)^* T\ol \XC_{
  P(v)}^\circ) \]
as follows.  The element $\tilde \eta$ is equal to $\eta$ in the tree
components $T \subset C$ and in the complement of the neck region on
the surface components $S_v \subset C$. On the neck region, for an
edge $e=(v_+,v_-) \in \Edge_\trop(\Gamma)$, $\tilde \eta$ is defined
by restricting $\eta|\Neck_e(S_\nu)$ to half the neck and extending by
zero :
\[ \eta_{v,+} : (-\infty,0] \times S^1 \to \XC_{P(e)}, \quad (s,t)\mapsto 
  \begin{cases}
    \pT^{+,\nu}_e  \eta \left(s+\frac {\nu l_e} 2,t \right) , \quad s \geq \frac {-\nu l_e} 2 \\
    0, \quad s< \frac {-\nu l_e} 2,
  \end{cases} \]
and 
\[
  \eta_{v,-} : [0,\infty) \times S^1 \to \XC_{P(e)}, \quad (s,t) \mapsto 
  \begin{cases}
    \pT^{-,\nu}_e  \eta \left(s-\frac {\nu l_e} 2,t \right) , \quad s \leq \frac {\nu l_e} 2 \\
    0, \quad s> \frac {\nu l_e} 2,
  \end{cases},\]
where
\begin{equation}
  \label{eq:ptpm}
\pT^{\pm,\nu}_e : \Gamma((u_\nu^\pre)^*T\XC_{P(e)}) \to \Gamma((u_v^\pm)^*T\XC_{P(e)}) 
\end{equation}
is parallel transport along the path
\[\exp_{e^{(s+it)\cT(e)}x_e} ( \rho \zeta^\nu_e(s \pm \tfrac {\nu l_e} 2 ,t) + (1- \rho)
  \zeta_e^\pm(s,t)), \quad \rho \in [0,1] , \]
$\zeta_\nu^e$ is defined in \eqref{eq:preglued} and $\zeta_e^\pm$ is
defined in \eqref{eq:xipmdef}.

We apply the inverse of the linearization of the broken map to the element constructed in the
previous paragraph. 
Since the broken map $\cu$ is regular
and isolated, its linearized operator is bijective.  We recall that
the linearized operator is a map of Banach spaces (see
\eqref{eq:linop})
\begin{multline*}
  D_{\cu} : T_m \M_\Gamma \times \Map(C,\XX)_{1,p,\lam} \to
  \Om^{0,1}(S,(\cu|S)^*T\XX) \\
  \oplus \Om^1(T,(\cu|T)^*TL) \oplus
\ev_\Gamma^* 
T\XX/T\Delta.  
\end{multline*}
Bijectivity of $D_{\cu}$ implies there is an inverse $(m,\xi)$ for the
element $(\tilde \eta, 0 \in \ev_\Gamma^* T\XX/T\Delta)$.  We write
$\xi=((\xi_v)_{v \in \Ver(\Gamma)}, \xi_T)$.  The vanishing of the
last term in $D_{\cu}(m,\xi)$ means that $\xi$ satisfies matching
conditions at tropical and boundary nodes, and the interior markings
$z_i$ satisfy the divisor constraint :
\[ \xi(z_i) \in T_{u_\pm(z_i)} \DD . \]
The matching at tropical nodes implies that for any tropical edge
$e=(v_+,v_-)$, the limit of $\xi_{v_+}$, $\xi_{v_-}$ at the
cylindrical end $e$ is equal :
\[\xi_{v_+,e}=\xi_{v_-,e} =: \xi_e \in
T_{x_e}\XC_{P(e)}  .\]

We now define the approximate inverse by patching the inverse of the
linearization of the broken map.  We denote the approximate inverse of
$D_{u_\nu^\pre}$ by $Q^\nu$. On the complement of the neck regions in
$C_\nu$ we define
\[Q^\nu(\eta):=\xi \quad \text{on }C_\nu \bs \cup_e \Neck_e(S_\nu);\]
and on the neck region corresponding to a tropical edge
$e \in \Edge_\trop(\Gamma)$ we patch the solutions $\xi_{v_\pm}$
together using a cutoff function:
\begin{align} Q^\nu \eta &:= \beta \left(-s + \sqq \nu l_e \right)
  ( ( \pT^{-,\nu}_e)^{-1} \xi_{v_-}(s+\nu l_e/2) - \pT^\nu_e \xi_e)\\
  \nonumber & + \beta \left( s + \sqq \nu l_e \right) ((
  \pT^{+,\nu}_e)^{-1} (\xi_{v_+}(s-\nu l_e/2)) - \pT^\nu_e \xi_e) + \pT^\nu_e
  \xi_e\\
  \nonumber & \in \Omega^0(C_\nu, (u^{\pre}_{\nu})^*
  TX)_{1,p,\lambda},
\end{align}
where $\pT^{\pm,\nu}_e$ is defined in \eqref{eq:ptpm} and $\pT_e^\nu$
is the parallel transport from $\Gamma((u^\ver_{\nu,e})^*T\XC_{P(e)})$
to $\Gamma((u_\nu^\pre)^*T\XC_{P(e)})$.

The approximate inverse $Q^\nu$ is uniformly bounded for all $\nu$: It
follows easily from the construction of $Q^\nu$ that the operations
$\eta \mapsto \tilde \eta$, $\tilde \eta \mapsto \xi$ are uniformly
bounded operators for all $\nu$, and that on the complement of the
neck regions there is a uniform bound on the norm of the operator
$\xi \mapsto (Q_\nu\eta)|C_\nu^\black$. On the neck $\Neck_e(S_\nu)$
corresponding to a tropical edge $e \in \Edge_\trop(\Gamma)$ there is
a uniform bound
\[\Mod{((Q^\nu\eta) - \pT_e\xi_e)e^{\kappa_\nu \lam} }_{L^p(\Neck_e(S_\nu))} \leq c\Mod{\xi}_{W^{1,p,\lam}}.\]
However, we need a $\kappa_\nu$-weighted Sobolev norm bound on $((Q^\nu\eta) - \pT_e(Q^\nu \eta)(0,0)_e)$, which follows from the observation
\[|Q_\nu(0,0)_e - \xi_e| \leq \abs{\xi_{v_+}(\tfrac {-\nu l_e} 2,0)_e - \xi_e} + \abs{\xi_{v_-}(\tfrac {\nu l_e} 2,0)_e - \xi_e} \leq ce^{-\lam \nu l_e/2}\Mod{\eta}_{L^{p,\lam}},\]
which implies that
\[\Mod{(\pT_e Q^\nu(0,0)_e - \pT_e \xi_e)e^{\kappa_\nu \lam}} \leq c\Mod{\eta}_{L^{p,\lam}}. \] 
We have thus shown that there is a $\nu$-independent constant $c$ such that
\[\Mod{Q^\nu} \leq c.\]

Next, we give an error estimate for the approximate inverse.  We need
to bound the quantity $D_{u_{\pre}^\nu} Q^\nu \eta - \eta$.  On the
complement of the neck regions of $C_\nu$ (including treed segments),
this quantity is bounded by the difference in the domain-dependent
perturbation on $(C,j)$ and $(C_\nu,j^\nu)$. Similar to
\eqref{eq:nonneckbd} in the estimate of $\F^\nu(0)$, we have the bound
\begin{equation}
  \label{eq:fest}
  \Mod{D_{u_{\pre}^\nu} Q^\nu \eta - \eta}_{L^{p,\lam}(C_\nu^\black)} \leq \co \max_{e \in \Edge_\trop(\Gamma)} \exp(-\nu l_e).
\end{equation}
We bound $\Mod{D_{u_{\pre}^\nu} Q^\nu \eta - \eta}_{L^{p,\lam}}$ on
$\Neck_e(S_\nu)$ corresponding to an edge
$e=(v_+,v_-) \in \Edge_\trop(\Gamma)$.  We drop the parallel transport
notation in the analysis, noting that both the transport map and its
derivative contribute smooth multiplicative factors that decay as
$ce^{\nu l_e /2 -|s|}$ on the neck.  We recall that
$\Neck_e(S_\nu) \simeq \{(s,t) \in [\frac {-\nu l_e} 2,\frac {\nu l_e}
2] \times S^1\}$, and analyze the half $\{s \geq 0\}$, the analysis on the other half being similar. On $\{s \geq 0\}$, we have
\[\eta=\eta_{v_+}=D_{u^0}\xi_{v_+}, \quad Q^\nu \eta = \xi_{v_+} + \beta(-s + \tfrac {\nu l_e}4)(\xi_{v_-}-\pT_e\xi_e),
\]
and therefore,
\begin{equation}
  \label{eq:estr}
  D_{u^\nu_{\pre}}Q^\nu \eta - \eta = D_{u^\nu_{\pre}}(\beta(-s + \tfrac {\nu l_e}4)(\xi_{v_-}-\pT_e\xi_e)) + (D_{u^0_{v_+}}-D_{u^\nu_{\pre}})\xi_{v_+}.
\end{equation}
We have
\begin{equation}
  \label{eq:dudiff}
\Mod{(D_{u^0_{v_+}}-D_{u^\nu_{\pre}})\xi_{v_+}}_{L^{p,\lam}(S_\nu)} \leq c e^{-\nu l_e/2}\Mod{\xi_{v_+}}_{W^{1,p,\lam}(S_{v_+}^\circ)}  \leq c e^{-\nu l_e/2}\Mod{\eta}_{L^{p,\lam}(C_\nu)}
\end{equation}
since
\[ d_{X^\nu}(u^0_{v_+}(s-\frac {\nu l_e} 2,t),u^\nu_{\pre}(s,t)) \leq c
e^{-\nu l_e/2} . \]
Since $D_{u^0_{v_-}}\xi_{v_-}=0$ on $\{s \geq 0\}$,
by a similar bound to \eqref{eq:dudiff} on
$(D_{u^0_{v_-}}-D_{u^\nu_{\pre}})\xi_{v_-}$, we have
\begin{equation}
  \label{eq:xinegbd}
  \Mod{D_{u^\nu_\pre} \xi_{v_-}}_{L^{p,\lam}(\{s \geq 0\})} \leq  c e^{-\nu l_e/2}\Mod{\xi_{v_-}}_{W^{1,p,\lam}(S_{v_-}^\circ)} \leq c e^{-\nu l_e/2}\Mod{\eta}_{L^{p,\lam}(C_\nu)}.
\end{equation}
From \eqref{eq:estr}, \eqref{eq:dudiff}, \eqref{eq:xinegbd} we conclude 
\begin{equation}
  \label{eq:errorbd}
  \Mod{ D_{u^\nu_{\pre}}Q^\nu \eta - \eta }_{L^{p,\lam}} \leq \Mod{\beta (D_{u^\nu_{\pre}}(\pT_e\xi_e))} +\Mod{d\beta(\xi_{v_-}-\pT_e\xi_e)} + c e^{-\nu l_e/2}\Mod{\eta}.   
\end{equation}
To bound the first term in the right-hand side of \eqref{eq:errorbd} we observe that $\pT_e\xi_e$ is the parallel transport of a covariant constant section on the trivial cylinder $u_e^{\on{ver}}$, and therefore,
\[\Mod{D_{u^\nu_{\pre}}(\pT_e\xi_e)} = \Mod{(D_{u^\nu_{\pre}}-D_{u_e^{\on{ver}}})(\pT_e\xi_e)} \leq c e^{-\nu l_e/2}|\xi_e| \leq c e^{-\nu l_e/2}\Mod{\eta}_{L^{p,\lam}}. \]
It remains to bound the second term which is supported in the unit
interval $[\frac {\nu l_e} 4, \frac {\nu l_e} 4 +1] \times S^1$ in the
neck $[\frac {-\nu l_e} 2, \frac {\nu l_e} 2] \times S^1$. In this
interval, the Sobolev weight in the curve $S_{v_-}$ and the glued
curve $S_\nu$ differ by $\frac {\nu l_e} 2$ :
\[ \kappa_\nu(s,t) = \kappa_{S_{v_-}^\circ}(s+\nu l_e/2) - \nu l_e/2,
\quad \mp s \ge \nu l_e/2. \]
Therefore there is a $\nu$-independent constant $\co$ such that 
\begin{equation} \label{best} \Mod{\d \beta( s - \nu
    l_e/4)\xi_{v_-}}_{L^{p,\lam}(S_\nu)} \leq \co e^{- \lambda \nu
    l_e/2}\Mod{\xi_{v_-}}_{W^{1,p,\lam}(S_v^\circ)} \leq \co e^{- \lambda \nu
    l_e/2}\Mod{\eta}_{L^{p,\lam}(S_\nu)}.\end{equation}

From \eqref{eq:fest}, \eqref{best}, one obtains an estimate as in Fukaya-Oh-Ohta-Ono
\cite[7.1.32]{fooo}, Abouzaid \cite[Lemma 5.13]{ab:ex}:  For some
constant $\co > 0$, for any $\nu$
\begin{equation} \label{eq:first-pre} \Vert D_{ u^{\pre}_\nu } Q^\nu -
  \on{Id} \Vert < \co \min_{e \in \Edge_\trop(\Gamma)}( \exp(-\lambda
  \nu l_e/2) , \exp(-(1 - \lambda)\nu l_e/2)) .\end{equation}
It follows that for $\nu$ sufficiently large an actual inverse may be
obtained from the Taylor series formula
\[ D_{u^{\pre}_\nu}^{-1} = Q^\nu (D_{u^{\pre}_\nu}Q^\nu)^{-1} = Q^\nu
\sum_{k \ge 0} (I - Q^\nu D_{u^{\pre}_\nu})^k .\]
The approximate inverse $Q^\nu$ is uniformly bounded for all
$\nu$.  For large enough $\nu$, \eqref{eq:first-pre} implies that
$\Mod{D_{ u^{\pre}_\nu } Q^\nu - \on{Id}} \leq \hh$, and so,
\begin{equation}
  \label{first}
 \Mod{D_{u^{\pre}_\nu}^{-1}} \leq 2\Mod{Q^\nu} \leq c.
\end{equation}

\section{Uniform quadratic estimate}\label{sec:quadratic}
We obtain a uniform quadratic estimate for the non-linear terms in the
map cutting out the moduli space locally. We will prove that there
exists a constant $\co$ such that for all $\nu$
\begin{equation} \label{second} \Mod{ D_{(m_1,\xi_1)} \cF_\nu
    (m_2,\xi_2) - D_{u^{\pre}_\nu} (m_2,\xi_2)} \leq \co 
  \Mod{(m_1,\xi_1)}_{1,p,\lambda}
  \Mod{(m_2,\xi_2)}_{1,p,\lambda}.
\end{equation}
We prove the quadratic estimate for the $\delbar$ term on surface components. The other terms in the operator $\F_\nu$ are left to the reader as the proof is similar.

As a preliminary step we prove a quadratic estimate on a simpler operator. 
Define 
\begin{multline*}
  \G : W^{1,p,\lam}(\Om^0(S_\nu, (u_\nu^\pre)^*TX^\nu)) \to L^{p,\lam}(\Om^1(S_\nu,(u_\nu^\pre)^*TX^\nu)), \\
  \xi \mapsto \pT_\xi^{-1}d(\exp_{u_\nu^\pre} \xi). 
\end{multline*}

\begin{claim}\label{cl:gbd}
  There is a $\nu$-independent constant $c$ such that
  \[\Mod{D_{\xi_1}\G(\xi_2) - D_0\G(\xi_2)}_{0,p,\lam} \leq c \Mod{\xi_1}_{1,p,\lambda}
  \Mod{\xi_2}_{1,p,\lambda}  \]
if $\Mod{\xi_1}_{1,p,\lambda}$ is small enough.
\end{claim}

\begin{subproof} [Proof of Claim]
  For $x \in X^\nu$, $\xi_1,\xi_2 \in T_x X^\nu$, define
  \[\ol \pT_{-\xi_1}^x : T_x X \to T_x X, \quad \xi_2 \mapsto
    \pT_{\xi_1}^{-1} \tfrac \d {\d\tau}
    \exp_x(\xi_1+\tau\xi_2)|_{\tau=0}.\]
  It extends to a map on sections.  For a map $u: S_\nu \to X^\nu$, and
  $\xi_1 \in \Gamma(S_\nu,u^*T X^\nu)$ define
  \[\ol \pT_{-\xi_1}^u : \Gamma(S_\nu,u^*T X^\nu) \to
    \Gamma(S_\nu,u^*T X^\nu), \quad \xi_2 \mapsto (z \mapsto \ol
    \pT_{-\xi_1(z)}^{u(z)}(\xi_2(z))).\]
  We have
  \begin{align} \label{eq:reverse}
 \nonumber   D_{\xi_1}\G(\xi_2) - D_{u_\nu^\pre}\G(\xi_2)
    =\tfrac \d {\d\tau} \pT_{\xi_1+\tau \xi_2}^{-1} \d_z(\exp_{u_\nu^\pre}(\xi_1+\tau \xi_2)) 
    - \tfrac \d {\d\tau} \pT_{\tau \xi_2}^{-1} \d_z(\exp_{u_\nu^\pre}(\tau \xi_2))\\
    =\nabla_z(\ol \pT^{u_\nu^\pre}_{-\xi_1} \xi_2 - \xi_2),
  \end{align}
  where in the second line $\nabla_z$ is differentiation along the
  domain curve $S_\nu$. In the second equality we switch the order of
  differentiation between $\tfrac \d {\d\tau}$ and $\d_z$.  For any
  $z \in S_\nu$ we have a uniform pointwise estimate
  \begin{align*}
    \abs{D_{\xi_1}\G(\xi_2) - D_{u_\nu^\pre}\G(\xi_2)} 
    &\leq \abs{\nabla_z(\ol \pT^{u(z)}_{-\xi_1(z)} - \Id)} \cdot \abs{\xi_2(z)} + \abs{\ol \pT^{u_\nu^\pre(z)}_{-\xi_1(z)} - \Id} \cdot \abs{(\nabla_z\xi_2)(z)}\\
    &\leq c((\abs{du_\nu^\pre(z)}\cdot \abs{\xi_1(z)} + \abs{\nabla_z\xi_1(z)} )\cdot \abs{\xi_2(z)} + \\ & \quad \abs{\xi_1(z)}\cdot \abs{\nabla_z \xi_2(z)})
  \end{align*}
  where the constant $c$ is $\nu$-independent. Indeed, such a uniform
  constant exists because the complement of the neck regions is
  compact and identical for all $\nu$; and the neck regions have a
  cylindrical metric, and only the lengths of the cylinders vary with
  $\nu$. Then,
  \begin{multline*}
    \Mod{D_{\xi_1}\G(\xi_2) - D_{u_\nu^\pre}\G(\xi_2)} \leq 
    c(\Mod{\d u_\nu^\pre}_{L^\infty}\cdot \Mod{\xi_1}_{L^\infty}\cdot \Mod{\xi_2}_{L^{p,\lam}} + \Mod{\nabla_z\xi_1(z)}_{L^{p,\lam}}\cdot \Mod{\xi_2}_{L^\infty}\\
    + \Mod{\xi_1}_{L^\infty}\cdot \Mod{\nabla_z \xi_2(z)}_{L^{p,\lam}}) 
     \leq c \Mod{\xi_1}_{1,p,\lambda}
  \Mod{\xi_2}_{1,p,\lambda},
\end{multline*}
where for the last inequality, we use the fact that
$\Mod{\d u_\nu^\pre}_{L^\infty}$ is uniformly bounded for all $\nu$ by
construction, and the following bound from Sobolev embedding : For any
section $\xi \in W^{1,p,\lam}(S_\nu)$ there is a $\nu$-independent
constant $c$ such that
$\Mod{\xi}_{L^\infty} \leq c\Mod{\xi}_{W^{1,p,\lam}}$. This proves the
Claim.
\end{subproof}

We obtain the quadratic estimate \eqref{second} for $\F_\nu$ by
adapting the proof of the above claim. We note that, compared to $\G$,
the operator $\F_\nu$ additionally consists of a projection to
$(0,1)$-forms :
 \[\F_\nu : (m,\xi) \mapsto \pT_\xi^{-1}\pi^{0,1}_{j(m),J(\xi)} \d(\exp_{u_\nu^\pre}\xi),\]
 where we abbreviate $J_{\exp_{u_\nu^\pre} \xi}$ as $J(\xi)$.  The
 analog of \eqref{eq:reverse} is
 \begin{multline}
   D_{(m_1,\xi_1)}\F_\nu(\xi_2) - D_{u_\nu^\pre}\F_\nu(\xi_2)  \\
   =\tfrac \d {\d\tau} \pT_{\xi_1+\tau \xi_2}^{-1} \pi^{0,1}_{j_\nu(m_1+\tau m_2), J(\xi_1+\tau\xi_2)}\d_z(\exp_{u_\nu^\pre}(\xi_1+\tau \xi_2)) |_{\tau=0}
   \\ - \tfrac \d {\d\tau} \pT_{\tau \xi_2}^{-1} \pi^{0,1}_{j_\nu(\tau m_2), J(\tau\xi_2)}\d_z(\exp_{u_\nu^\pre}(\tau \xi_2))|_{\tau=0}\\
   =\pi^{0,1}_{j_\nu(m_1),J(\xi_1)}\nabla_z(\ol \pT^u_{-\xi_1} \xi_2) + (\nabla_z(\pi^{0,1}_{j_\nu(m_1),J(\xi_1)}))(\ol \pT^u_{-\xi_1} \xi_2)\\ - \pi^{0,1}_{j(C_\nu),J(0)}(\nabla_z \xi_2) - (\nabla_z \pi^{0,1}_{j(C_\nu),J(0)})\xi_2,
 \end{multline}
 where $\tau m_2$ resp. $m_1 + \tau m_2$, $\tau \in [0,1]$ is a path
 in a neighborhood $U_{C_\nu}$ of $C_\nu$ in the moduli space
 $\M_{\on {glue}}$.  The estimate \eqref{second} can be obtained from
 the above expression in a similar way to the proof of Claim
 \ref{cl:gbd}. We point out that the metric $g_{C_\nu}$ on $U_{C_\nu}$
 is cylindrical on the non-compact ends of $\M_{\on {glue}}$, and with
 respect to this metric
 $\Mod{j_\nu(m_1) - j_\nu(m_2)}_{C^1} \leq
 cd_{g_{\Gamma_\glue}}(m_1,m_2)$, see \eqref{eq:fjcont}.

\section{Picard iteration}
\label{picardit} 

We apply the implicit function theorem to obtain an exact solution for
a glued holomorphic map.  We recall the following version
of the Picard Lemma \cite[Proposition A.3.4]{ms:jh}.

\begin{lemma}{\rm(Picard Lemma)} \label{lem:pic}
  Let $X$ and $Y$ be Banach spaces, $U \subset X$ be an open set
  containing $0$, and $f:U \to Y$ be a smooth map. Suppose $df(0)$ is
  invertible with inverse $Q:Y \to X$. Suppose $\co$ and $\eps>0$ are
  constants such that $\Mod{Q} \leq \co$, $B_\eps \subset U$, and
  \[\Mod{df(x) - df(0)} \leq \tfrac 1 {2\co} \quad \forall x \in
  B_\eps(0).\]
  Suppose $f(0) \leq \frac {\eps} {4\co}$. Then, there is a unique point
  $x_0 \in B_\eps$ satisfying $f(x_0)=0$.
\end{lemma}

Picard's Lemma and the estimates \eqref{zeroth}, \eqref{first},
\eqref{second} imply the existence of a solution $(m(\nu),\xi(\nu))$
to the equation
\[ \cF^\nu(m(\nu),\xi(\nu)) = 0 \]
for each $\nu$. The map
\[ u_\nu := \exp_{u_\nu^{\pre}}(\xi(\nu)) \]
is a $(j(m(\nu)), J^\nu)$-holomorphic map to $X^\nu$.  Additionally,
there is a $\nu$-independent constant $\eps>0$ such that
$(m(\nu),\xi(\nu))$ is the unique zero of $\cF^\nu$ in an
$\eps$-neighbourhood of $((C_\nu,j^\nu),u_\nu^\pre)$ with respect to
the $g_{\Gamma_\glue}$-norm on $m(\nu)$ and the weighted Sobolev norm
$W^{1,p,\lam}$ on $\xi(\nu)$.

\section{Surjectivity of gluing} \label{surjofgluing}

We show that the gluing construction gives a bijection between rigid
maps in neck-stretched manifolds and rigid broken maps.  Note that any
family $\{u'_\nu: C_\nu' \to X^\nu\}_\nu$ converges to a broken map
$u: C \to \XX$ by Theorem \ref{thm:cpt-breaking}.  To prove the
bijection we must show that any such family of maps is in the image of
the gluing construction.  Since the implicit function theorem used to
construct the gluing gives a unique solution in a neighbourhood, it
suffices to show that the maps $u'_\nu$ are close, in the Sobolev norm
used for the gluing construction, to the approximate solution
$u^\pre_\nu$ defined by \eqref{eq:preglued}.

We first show that the domain curves of the converging sequence of
maps are close enough to the domains of the approximate solution with
respect to the cylindrical metric $g_{\Gamma_\glue}$ from
\eqref{eq:gm}.  In the definition of Gromov convergence, the
convergence of domains implies that $C_\nu' \to C$ in the compactified
moduli space $\ol \M_{\Gamma_\glue}$, which implies
\begin{equation}
  \label{eq:curvconv}
\pi_\Gamma(C_\nu') \to C \quad \text{in $\M_\Gamma$.}
\end{equation}
We additionally need to prove that the distance
$d_{g_{\Gamma_\glue}}(C_\nu',C_\nu) \to 0$ where the metric
$g_{\Gamma_\glue}$ is cylindrical in the non-compact ends of
$\ol \M_{\Gamma_\glue}$.
By assumption the limit map $u$ does not have any tropical
symmetry. Therefore, the translation sequence $t_\nu$ is uniquely
determined by the tropical graph of $u$ and coincides with the
translations used for gluing.  By \hyperref[item:thincyl2]{(Thin
  cylinder convergence)}, for any tropical edge
$e=(v_+,v_-) \in \Edge_\trop(\Gamma)$, the gluing parameters
$l_\nu'(e) + i\theta_\nu'(e)$ of the curves $C_\nu'$ satisfy
\[ \lim_{\nu \to \infty}\theta_\nu'(e)=0, \quad \lim_{\nu \to \infty}(t_\nu(v_+)-t_\nu(v_-) - \cT(e) l_\nu'(e)) = 0\]
The gluing parameter of $C_\nu$ at the edge $e$ is $\nu l_e$, which satisfies
the relation
\[t_\nu(v_+)-t_\nu(v_-) = \cT(e) \nu l_e.\]
Therefore,
$l_\nu'(e) - \nu l_e \to 0$, and $d_{g_{\Gamma_\glue}}(C_\nu,C_\nu') \to 0$ as
$\nu \to \infty$.  In addition there is a biholomorphism
\begin{equation}
  \label{eq:bihol}
  \phi_\nu : (C_\nu,j_\nu([C_\nu'])) \to C_\nu',
\end{equation}
where $j_\nu$ is defined in \eqref{eq:jnu}.
 
Next, we show that the maps in the converging sequence are close
enough to the approximate solutions.  Via the identification
\eqref{eq:bihol}, we view $u_\nu'$ as a map on $C_\nu$.
We need to
bound the section $\xi_\nu' \in \Om^0(C_\nu,(u_\nu^{pre})^*TX^\nu)$
defined by the equation $u_\nu'=\exp_{u_\nu^{pre}}\xi_\nu'$ in the
weighted Sobolev norm \eqref{1pl2}.  Consider the neck region in $C_\nu$
corresponding to an edge $e$ with coordinates
\[(s_e,t_e) \in [-\tfrac {\nu l_e} 2, \tfrac {\nu l_e} 2] \times \R/2 \pi \Z.\]
Denote the midpoint of the neck as 
\[ 0_e:=\{(s_e,t_e)=(0,0)\} \in
C_\nu . \]  
In the neck region, the maps $u_\nu^{pre}$ and $u_\nu'$ are equal to a
vertical cylinder perturbed by a quantity that decays exponentially in
the middle of the cylinder. The vertical cylinder is determined by
$u_\nu^{pre}(0_e)$ resp. $u_\nu'(0_e)$.  The sequence
$u^{pre}_\nu(0_e)$ converges to $x_e$ because of the asymptotic decay
of the sections $\zeta^\pm$. The sequence $u_\nu'(0_e)$ converges to
$x_e$ by \hyperref[item:thincyl2]{(Thin cylinder convergence)}.  Indeed,  since the complex
structure $\fj^\nu(C_\nu')$ is standard on a truncation
$[-\tfrac {\nu l_e} 2 - L, \tfrac {\nu l_e} 2 - L] \times S^1$ of the
neck (see \eqref{eq:suppfj}), and the mid point
of the cylinder is preserved by the biholomorphism $\phi$ in
\eqref{eq:bihol}, \hyperref[item:thincyl2]{(Thin cylinder convergence)} is applicable with the coordinates
$(s_e,t_e)$.  On the neck region, the section $\xi_\nu'$ and its
derivatives decay exponentially : 
\[|D^k\xi_\nu'(s_e,t_e)| \leq \co e^{-(\nu l_e/2 -|s_e|)}, \quad k \in \{0,1\}. \]
This inequality follows from  the decay of the terms $\zeta_\nu^\pm$ in the
definition of $u_\nu^{pre}$, and the breaking annulus lemma applied to
$u_\nu'$. Consequently $\Mod{\xi_\nu'}_{W^{1,p,\lam}}$ can be made
small enough by taking a large $\nu$ and shrinking the neck by a fixed
amount: that is, we decrease the cylinder length to $\nu l_e - C$ where $C$ is a constant independent of $\nu$.  Next we consider the complement of the neck regions. Here,
the sequences $u_\nu'$ and $u_\nu^{pre}$ uniformly converge to
$u$. So, by taking $\nu$ large enough the maps
$u_\nu^{pre}$ and $u_\nu'$ are $W^{1,p}$-close enough in the
complement of the neck regions.
This finishes the proof of Theorem \ref{thm:gluing}.  

\begin{remark}{\rm(Gluing preserves orientations)}
  \label{rem:orientglue}
  From the definition of orientation of moduli spaces in Remark
  \ref{rem:orientmap} it is easily seen that gluing preserves
  orientation in the special case that the Cauchy-Riemann operator on
  each surface component of
  the broken map $u^0$ is complex
  linear. It remains to consider the effect of the interpolation from the
  linearized operator $D_{u^0}$ resp. $D_{u_\nu}$ to a complex linear
  operator $\delbar_{u^0}$ resp. $\delbar_{u_\nu}$.  This interpolation 
  can be accomplished by an interpolation on the Cauchy-Riemann operator one ach component $u |_{S_v}$ to a complex-linear operator; the interpolations can be then be glued to obtain a interpolation of $D_{u_\nu}$ to a complex linear operator.  It follows that the gluing sign is the same as in the complex-linear case, which is to say, positive. 
  \end{remark}

\section{Tubular neighbourhoods and true boundary}
\label{sec:tubular}
In Section \ref{sec:bdry} we proved a convergence result that the
boundaries of one-dimensional moduli spaces of rigid broken maps
consist of maps that either have a length zero boundary edge, or a
broken boundary edge.  Theorem \ref{thm:bdry-glue} is the converse
gluing result, and implies 
that given a broken map $u$ containing a
boundary edge that is broken or has zero length, the map $u$ indeed
occurs as the codimension one boundary of the expected moduli spaces.
The convergence and gluing results, together, allow us to identify the
true boundary of one-dimensional moduli spaces.  A stratum containing
a zero length edge is actually a \em{fake boundary stratum} as it is
the boundary of two different rigid strata with opposite induced
boundary orientations. Thus, the \em{true boundary strata} are those
that have a broken boundary edge, see Remark \ref{rem:true-fake}.

\begin{theorem}\label{thm:bdry-glue}
  Let $\ul \Pe=(\Pe_\Gamma)_\Gamma$ be a coherent regular perturbation
  datum for all types $\Gamma$.  Suppose $\Gamma$ is a type of broken
  treed disks and $\ul x \in \cI(L)^{n+1}$ is a set of limits for
  boundary leaves such that $i(\Gamma,\ul x) = 1$.
  \begin{enumerate}
  \item \label{tubs} {\rm (Tubular neighbourhoods)} If a type
    $\Gamma_{\glue}$ of broken maps is obtained from a type $\Gamma$
    by collapsing an edge $e \in \Edge_{\white,-}(\Gamma')$ with
    $\ell(e)=0$, or by making an edge length finite/non-zero, then the
    stratum ${\M}^\br_{\Gamma}(L, \ul x)$ is a codimension one
    boundary in $\ol{\M}^\br_{\Gamma_\glue}(L, \ul x)$ and has a
    tubular neighborhood in it; and
  \item \label{oriens} {\rm (Orientations)} the orientations are
    compatible with the morphisms \hyperref[item:cuttingedges]{(Cutting edges)} and 
      (\hyperref[item:collapsingedgesmorphism]{Collapsing an edge}/\hyperref[item:makingedgelengthmorph]{Making an edge finite/non-zero}) in the
    following sense:
    \begin{enumerate}
    \item \label{part:orien1} Suppose $\Gamma$, $\Gamma_0$, $\Gamma_c$
      are types of broken maps related by the morphisms
      \[ \Gamma \xrightarrow{\text{Make $\ell(e)$ zero}} \Gamma_0
        \xrightarrow{\text{Collapse $e$}} \Gamma_c,\]
      and $i^\br(\Gamma,\ul x)=i^\br(\Gamma_c, \ul x)=1$,
      $i^\br(\Gamma_0, \ul x)=0$.  Then, the boundary orientation
      induced by $\M_\Gamma^\br(L , \ul x)$ on
      $\M_{\Gamma_0}^\br(L , \ul x)$ is the opposite of the boundary
      orientation induced by $\M_{\Gamma_c}^\br(L , \ul x)$ on
      $\M_{\Gamma_0}^\br(L , \ul x)$.
    \item \label{part:orien2} Suppose $\Gamma_f$, $\Gamma$,
      $\Gamma_1$, $\Gamma_2$ are types of broken maps and
      $\ul x_1 \in \cI(L)^{d_\white(\Gamma_1)+1}$,
      $\ul x_2 \in \cI(L)^{d_\white(\Gamma_2)+1}$ are labels such that
      there are morphisms
      \[ (\Gamma_f,\ul x) \xleftarrow{\text{Make $\ell(e)$ finite}}
        (\Gamma,\ul x) \xrightarrow{\text{Cut $e$}} (\Gamma_1,\ul x_1)
        \times (\Gamma_2,\ul x_2),\]
      and $i^\br(\Gamma_f,\ul x)=1$,
      $i^\br(\Gamma, \ul x)=i^\br(\Gamma_1,\ul x_1)=i^\br (\Gamma_2,
      \ul x_2)=0$.  Then, there is an isomorphism
      \[\M^\br_\Gamma(L,\ul x) \simeq \M^\br_{\Gamma_1}(L,\ul x_1)
        \times \M^\br_{\Gamma_1}(L,\ul x_2),\]
      and the boundary orientation on $\M_{\Gamma}^\br(L , \ul x)$
      induced by $\M_{\Gamma_f}^\br(L , \ul x)$ is related to the
      product orientation by a sign $(-1)^{\diamondsuit}$ 
      that depends only on
      the domain type $\Gamma$ and the labels $\ul x$.
    \end{enumerate}

  \end{enumerate}

\end{theorem}
\begin{proof}
  [Outline of proof] The tubular neighborhood is constructed by gluing
  and there are different cases depending on the morphism:

\vskip .1in \noindent 
  \textsc{Case 1} : \em{Collapsing a boundary edge $e$ of zero length}.\\
  The proof is by gluing a boundary node, which is non-tropical since
  the Lagrangian is disjoint from relative divisors. The gluing proof
  is on the same lines as the gluing of tropical nodes in Theorem
  \ref{thm:gluing}, so we point out the differences, referring the
  reader to \cite[Chapter 7.1]{fooo} for the full proof.
  
  \begin{itemize}
  \item Firstly, we choose strip-like coordinates on a neighborhood of
    the boundary node $w_e$ and we use weighted Sobolev spaces on
    these strips, see Definition 7.1.3 in \cite{fooo}. A difference
    from our Fredholm set-up for broken maps in Section
    \ref{sec:fredholm-broken} is that we do not need cylindrical
    coordinates in the target space since there are no relative
    divisors.
  \item The next difference is that the map gluing parameter $\nu$
    refers to the neck length parameter in the glued curve (see Remark
    \ref{rem:glueparam}), and therefore in the domain of the operator
    $\F_\nu$ in \eqref{eq:Fmap} the component $\M_{\Gamma_\glue}$ is
    replaced by a subspace where the neck length corresponding to the
    node $w_e$ is fixed to be $\nu$. The proof of `surjectivity of
    gluing' also simplifies, since it is enough to show convergence of
    the domain curves away from the neck as in \eqref{eq:curvconv} and
    the arguments following \eqref{eq:curvconv} in that paragraph are
    not necessary.
  \end{itemize}

  Cases 2 and 3 do not involve gluing of surface components.\\

  \vskip .1in \noindent 
  \textsc{Case 2} : \em{Making a boundary edge length $\ell(e)$ non-zero.}\\
  The proof structure of Theorem \ref{thm:gluing} can be used in this
  case also.  However, the map gluing parameter $\nu$ corresponds to
  the edge length $\ell(e)$ in the glued curve and goes to zero in the
  limit. The domain curve for the glued map $u_\nu : C_\nu \to \XX$ is
  allowed to vary among the set of curves of type $\Gamma_\glue$ with
  $\ell(e)=\nu$.  The approximate solution $u_\nu^\pre:C_\nu \to \XX$
  is defined to be equal to $u$, and the map $u_\nu^\pre$ is defined
  to be constant on the treed segment $T_e$, so the error
  $\Mod{D\F_\nu(u_\nu^\pre)}$ is $\leq c\nu$.  The rest of the steps
  are analogous to the proof of Theorem \ref{thm:gluing}.

  \vskip .1in \noindent   \textsc{Case 3} : \em{Making a boundary edge length $\ell(e)$ finite.}\\
  Similar to Case 2, the map gluing parameter $\nu$ is equal to the
  length $\ell(e)$ of the edge $e$, but unlike Case 2, it goes to
  infinity in the limit.  The domain curve for the glued map
  $u_\nu : C_\nu \to \XX$ is allowed to vary among the set
  $\Gamma_\glue$ curves with $\ell(e)=\nu$.  The approximate solution
  $u_\nu^\pre:C_\nu \to \XX$ is defined to be equal to $u$ on all
  components except the treed segment $T_e$, where it is defined by
  pre-gluing the Morse trajectory (see (2.66) in Schwarz
  \cite{schw:morse}) at the edge $e$ in the map $u$. The norm of the error is bounded by $ce^{-\nu}$ and is derived in a similar way to Section \ref{sec:errorest}.  The rest of the steps are analogous to
  the proof of Theorem \ref{thm:gluing}.

  For the orientations result, part \eqref{part:orien1} follows from
  the definition of orientations in Remark \ref{rem:orientmap} and the
  sign computation for part \eqref{part:orien2} is carried out in
  Seidel's book \cite[(12.25)]{se:bo}.
\end{proof}

\begin{figure}[ht]
  \centering \scalebox{.8}{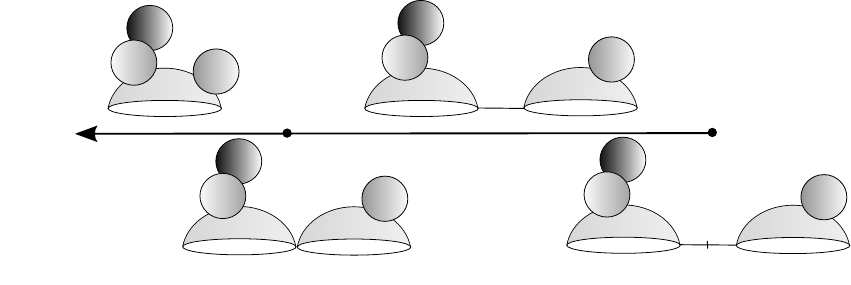}
  \caption{True and fake boundary strata of a one-dimensional
    component of the moduli space of treed holomorphic disks. The
    sphere components lie in different pieces of the tropical
    manifold. }
  \label{fig:true-fake}
\end{figure}

\begin{remark}{\rm(True and fake boundary strata)}
  \label{rem:true-fake}
  There are two types of strata that occur as the codimension one
  boundary of a one-dimensional moduli space -- one with a boundary
  edge of length zero, and the second with a boundary edge containing
  a breaking. The first is a \em{fake} boundary, and the second one
  is a \em{true} boundary as we explain: If a type $\Gamma_0$ of
  broken maps has a boundary edge of length zero, we may make the edge
  length $\ell(e)$ non-zero or we may collapse the edge (by disk
  gluing) to produce the following types
  \[ \Gamma \xrightarrow{\text{Make $\ell(e)$ zero}} \Gamma_0
    \xrightarrow{\text{Collapse $e$}} \Gamma_c.\]
  Let $\ul x$ be a tuple of input and output such that
  $i^\br(\Gamma,\ul x)=0$.  By Theorem \ref{thm:bdry-glue}, the moduli
  space $\M^\br_{\Gamma_0}(\ul x)$ is the boundary of the
  one-dimensional moduli spaces $\M^\br_{\Gamma_c}(\ul x)$ and
  $\M^\br_{\Gamma}(\ul x)$ with opposite orientations.  So,
  $\M_{\Gamma_0}^\br(\ul x)$ does not represent a component in the
  topological boundary of the compactified moduli space
  \[\bigcup_{\Gamma : \Gamma \text{ is rigid, } i^\br(\Gamma, \ul x)=1}\ol \M^\br_\Gamma(L,\ul x).\]
  This is the fake boundary in Figure \ref{fig:true-fake}. The only
  (true) boundary components of one-dimensional strata thus consist of
  maps with a single broken Morse trajectory; see Figure
  \ref{fig:true-fake}.
\end{remark}

\chapter{Broken Fukaya algebras}\label{chap:bfa}

In this Chapter we describe \ainfty algebra structures defined by
counting treed holomorphic disks on broken and unbroken manifolds, and
show that they are equivalent up to \ainfty homotopy.

\section{\ainfty algebras}

The set of treed holomorphic disks has the structure of an \ainfty
algebra. \ainfty algebras were introduced by Stasheff \cite{st:ho} in
order to capture algebraic structures on the space of cochains on loop
spaces.  We follow the sign convention in Seidel \cite{seidel:sub}.  A
\em{$\Z_2$-graded \ainfty algebra} consists of a $\Z_2$-graded vector
space $A$ together with for each $d \ge 0$ a multilinear degree zero
\em{composition map}
\[ m^d: \ A^{\otimes d} \to A[2-d] \]
\index{\ainfty! algebra} \index{\ainfty! composition map}
\index{\ainfty! associativity relations} satisfying the \em{\ainfty
  associativity equations} \cite[(2.1)]{seidel:sub}
\begin{multline} \label{ainftyassoc} 0 = \sum_{j,k\ge 0, j+k
      \leq d} (-1)^{j + \sum_{i=1}^{j} |a_i|}
  m^{d-k+1}(a_1,\ldots,a_{j}, \\ m^k(a_{j+1},\ldots,a_{j+k}),
  a_{j+k+1},\ldots,a_d)
\end{multline}
for any $d \geq 0$ and any tuple of homogeneous elements
$a_1,\ldots,a_d$ with degrees $|a_1|, \ldots, |a_d| \in \Z_2$.  The
notation $[2-d]$ indicates a degree shift of $2-d$, so that the degree
of $m^d(a_1,\dots,a_d)$ \label{noplusses} is
$\sum_i|a_i| + 2-d \in \Z_2$.  The signs appearing in
\eqref{ainftyassoc} are the \em{shifted Koszul signs}, that is, the
Koszul signs for the shifted grading in which the structure maps have
degree one as in Kontsevich-Soibelman \cite{ks:ainfty}.  The first of
these associativity relations is $m^1(m^0(1))=0$, and the second one
is
\begin{equation}
  \label{eq:ainf1}
  m^2(m^0(1),a) - (-1)^{|a|}m^2(a,m^0(1)) + m^1(m^1(a))=0, \quad \forall a \in A
\end{equation}
which may be interpreted as saying that $m^1$ is a differential if
$m^0(1)$ vanishes.  A \em{strict unit} for $A$ is an element
$1_A \in A$ such that
\begin{equation} \label{strictunit} m^2(1_A,a) = a = (-1)^{|a|}
  m^2(a,1_A), \quad m^{d}(\ldots, 1_A, \ldots) = 0, \forall d \neq 2
  .\end{equation}
\index{Unital \ainfty algebra} A \em{strictly unital} \ainfty algebra
is an \ainfty algebra that admits a strict unit.
\label{unitaldef}

The element $m^0(1) \in A$ (where $1 \in \R$ is the unit) is called
the \em{curvature} of the algebra.  The \ainfty algebra $A$ is \em{
  flat} if the curvature vanishes.  The \em{cohomology} of a flat
\ainfty algebra $A$ is defined by
\[ H(m^1) = \frac{ \on{ker}(m^1)}{\on{im}(m^1)} .\]
The algebra structure on $H(m^1)$ is given by
\begin{equation} \label{hcomp} [a_1 a_2] = (-1)^{|a_1|} [m^2(
  a_1,a_2)] .\end{equation}
The \ainfty algebras in our work are \em{curved}, that is $m^0(1)$
typically does not vanish.

The coefficient ring of our 
\ainfty algebras are Novikov rings defined as follows. 
Let $q$ be a formal variable and $\Lambda$ the \em{universal Novikov
  field} of formal sums \label{rep:rational} 
\begin{equation*}  \Lambda = \Set{ \sum_{i \in \Z_{\geq 0}} c_i
    q^{\alpha_i} | \quad c_i \in \C, \ \alpha_i \in \R, \quad
    \alpha_i \to \infty  }. \end{equation*}
The subring
\[\Lam_{\geq 0} :=\Set{ \sum_{i \in \Z_{\geq 0}} c_i  q^{\alpha_i} \in \Lam | \alpha_i \geq 0}\]
  with only the non-negative exponents is called the \em{Novikov ring}.
  \index{Novikov ring}
Denote by 
$\Lambda_{ >0}$ the subring with only  positive exponents, 
and by
\[ \Lambda^\times = ( \C - \{ 0 \}) + \Lambda_{> 0} \subset
\Lambda_{\ge 0}  \] 
the subgroup of formal power series with invertible leading
coefficient and non-negative exponents.
The Novikov ring has natural
adic topology, in which a sequence of elements converges iff it
converges in the quotient by any subring $q^E \Lambda_{\geq 0}$ for any
$E \in \R$.

Cohomology can be defined for a strictly
unital \ainfty algebra if the curvature is a multiple of the unit :
$m^0(1) \in \Lam_{\geq 0} 1_A$, because in this case, $(m^1)^2=0$ by
\eqref{eq:ainf1}.
More generally, the cohomology exists for any solution to the
projective Maurer-Cartan equation \cite{fooo}. The
projective Maurer-Cartan equation for $b \in A$ is
\index{Maurer-Cartan equation}
\begin{equation} \label{eq:mc} 
  m^0(1) + m^1(b) + m^2(b,b) + \ldots \in 
  \Lam_{\geq 0} 1_{A}.
\end{equation}
A solution $b \in A$ of odd degree to the equation \eqref{eq:mc} is
called a \em{weakly bounding cochain} and the set of all the odd
solutions is denoted $MC(A)$. Given a weakly bounding cochain
$b \in MC(A)$, the 
\em{deformed composition map} is defined by 
\begin{multline}
  \label{eq:mnb}
   m^n_b(a_1,\ldots,a_n) = \sum_{i_1,\ldots,i_{n+1}} m^{n + i_1 +
  \ldots + i_{n+1}}(\underbrace{b,\ldots, b}_{i_1}, a_1,\\
\underbrace{b,\ldots, b}_{i_2}, a_2,b, \ldots, b, a_n,
\underbrace{b,\ldots, b}_{i_{n+1}}) 
\end{multline}
over all possible combinations of insertions of the element
$b \in A^+$ between (and before and after) the elements
$a_1,\ldots, a_n$. The maps $m^n_b$ define an \ainfty structure on $A$
if $b$ has odd degree.  By the weakly bounding cochain condition
$m^0_b(1) \in \Lam_{\geq 0} 1_A$, and the \ainfty relations imply
\[ (m^1_b)^2(a) = m^2_b(m^0_b(1), a) -
m^2_b(a,m^0_b(1)) =0.\]
Consequently, the cohomology
\[ H(m^1_b) = \on{ker}(m^1_b)/\on{im}(m^1_b) \]
is well-defined.  The function
\begin{equation} \label{eq:potdef} W: MC(A) \to \Lambda_{\geq 0}, \quad b \mapsto W(b), \end{equation}
where $m^0_b(1) =W(b)1_A$ is called the \em{potential} of
the \ainfty algebra $A$.
\index{Potential}
\index{Potential of an \ainfty algebra|seeonly {Potential}}

  \index{Convergent!\ainfty algebra} An \ainfty algebra $A$
with coefficients in $\Lam_{\geq 0}$ is \em{convergent} if
$m^0(1) \in \Lam_{>0}$. This property ensures that the expression
$m^0_b$ in the Maurer-Cartan equation \eqref{eq:mc} is well-defined
for any $b \in A$ with positive $q$-valuation.

One also has homotopy notions of algebra morphisms.  \index{\ainfty!
  morphism} Let $A_0,A_1$ be \ainfty algebras.  An \em{\ainfty
  morphism} $\F$ from $A_0$ to $A_1$ consists of a sequence of linear
maps
\[\F^d: \ A_0^{\otimes d} \to A_1[1-d], \quad d \ge 0\]
such that the following holds:
\begin{multline} \label{faxiom} \sum_{i + j \leq d} (-1)^{i +
    \sum_{j=1}^i |a_j|} \F^{d - j + 1}(a_1,\ldots,a_i,
  m_{A_0}^j(a_{i+1 },\ldots,a_{i+j}),a_{i+j+1},\ldots,a_d) = \\
  \sum_{i_1 + \ldots + i_m = d} m_{A_1}^m(\F^{i_1}(a_1,\ldots,
  a_{i_1}), \ldots, \F^{i_m}(a_{i_1 + \ldots + i_{m-1} +
    1},\ldots,a_d))
\end{multline}
where the first sum is over integers $i,j$ with $i+j \leq d$, the
second is over partitions $d = i_1 + \ldots + i_m$.  The first
relation in the family of relations in \eqref{faxiom} is
\[\F^1(m_{A_0}^0)=m_{A_1}^0 + m_{A_1}^1(\F^0) + m_{A_1}^2(\F^0,\F^0) + \dots\]
\index{Unital \ainfty morphism} An \ainfty morphism $\F$ is \em{
  unital} if and only if
\begin{equation}
  \label{eq:unital-morph}
  \F^1(1_{A_0}) = 1_{A_1}, \quad \F^k(a_1, \ldots, a_i, 1_{A_0},
  a_{i+2}, \ldots, a_k) = 0  
\end{equation}
for every $k\geq 2$ and every $0\leq i \leq k-1$, where $1_{A_0}$
resp. $1_{A_1}$ is the strict unit in $A_0$ resp. $A_1$.  The
\em{composition} of \ainfty morphisms $\F_0,\F_1$ is defined by
\begin{multline} \label{composefunc}
 (\F_0 \circ \F_1)^d(a_1,\ldots,a_d)
 = \sum_{i_1 + \ldots + i_m =d}
  \F_0^{m}( \F_1^{i_1}(a_1,\ldots,a_{i_1}),  \\ \ldots, \F_1^{i_m}(a_{d -
    i_m + 1},\ldots,a_d)), \quad d \geq 0. \end{multline}

There is also a natural notion \ainfty homotopy equivalence of
morphisms.  Homotopies are defined as pre-natural transformations of
\ainfty morphisms, which we define following Seidel \cite{se:bo} and
Charest-Woodward \cite{cw:flips}: Given \ainfty algebras $A_0$, $A_1$,
the set of morphisms is an \ainfty category denoted by $\cQ$ whose
objects are \ainfty morphisms $\F : A_0 \to A_1$, and whose morphisms
are \em{pre-natural transformations}.  For \ainfty morphisms
$\F_0, \F_1: A_0 \to A_1$, a pre-natural transformation
$\cT=(\cT^d)_{d \geq 0} \in \Hom_{\cQ}(\F_0,\F_1)$ is a family of maps
\[\cT^d : A_0^{\tensor d} \to A_1.\]
Pre-natural transformations are equipped with composition maps
$m_{\cQ}^e$ for all $e \geq 1$ defined as follows: For any $d \geq 0$,
\begin{multline} \label{mu1}
 (m^1_{\cQ} \TT)^d (a_1,\ldots,a_d) = \sum_{k,l}
  \sum_{i_1,\ldots,i_l} (-1)^\dagger m^l_{A_2}(
  \F_0^{i_1}(a_1,\ldots,a_{i_1}), \F_0^{i_2}(a_{i_1 + 1},\ldots, a_{i_1+i_2}), 
  \ldots, \\ \TT^{i_k}(a_{i_1 + \ldots + i_{k-1} + 1},\ldots, a_{i_1 +
    \ldots + i_k}), \F_1^{i_{k+1}}( a_{i_1 + \ldots + i_k + 1},\ldots,
  ) ,\ldots, \F_1^{i_l}(a_{d - i_l},\ldots, a_d)) \\ - \sum_{i,e}
  (-1)^{i + \sum_{j=1}^i |a_j| + |\TT| - 1} \TT^{d - e +
    1}(a_1,\ldots,a_i, m^e_{A_1}(a_{i+1},\ldots,
  a_{i+e}),a_{i+e+1},\ldots,a_d), \end{multline}
where the sign $\dagger$ is from \cite[p82]{cw:flips}. For \ainfty
morphisms $\F_0, \F_1, \F_2 : A_0 \to A_1$, and pre-natural
transformations $\TT_0 \in \Hom(\F_0,\F_1)$,
$\TT_1 \in \Hom(\F_1,\F_2)$, the binary composition
$m^2_{\cQ}(\TT_0,\TT_1) \in \Hom(\F_0,\F_2)$ is defined as
\begin{multline}  \label{T2T1}
 (m^2_{\cQ}(\TT_1,\TT_2))^d(a_1,\ldots,a_d) =
\sum_{m,k,l} \sum_{i_1,\ldots,i_m} 
(-1)^\ddag m^l_{A_2}( \F_0^{i_1}(a_1,\ldots,a_{i_1}),
\ldots, \F_0^{i_{k-1}}(\ldots),
 \\
\TT_1^{i_k}(a_{i_1 + \ldots + i_{k-1} + 1},\ldots, a_{i_1 + \ldots + i_k}),
\F_1^{i_{k+1}}(\ldots),\ldots, \F_1^{i_{l-1}}(\ldots), 
\\
\TT_2^{i_l} (a_{i_1 + \ldots + i_{l-1} + 1},\ldots, a_{i_1 + \ldots + i_{l}}),
\F_2^{i_{l+1}}(\ldots),\ldots, \F_2^{i_l}(a_{d - i_l},\ldots,a_d)),
\end{multline}
where the sign $\ddag$ is from \cite[p82]{cw:flips}. The higher order
composition maps $m^e_{\cQ}$ are defined analogously.

Finally, we define a homotopy between \ainfty morphisms. A pre-natural
transformation $\cT$ between \ainfty morphisms
$\F_0, \F_1 : A_0 \to A_1$ is a \em{homotopy} if
\begin{equation} \label{eq:diffmor}
\F_1 - \F_0=m^1_{\cQ}(\cT).\end{equation}
Furthermore, if $\F_1 = (\F_{1,d})_{d \ge 0}$ is only a collection of
maps from $A_0^{\otimes d}$ to $A_1[1-d]$ and $\cF_0$ satisfies the
\ainfty morphism axiom then so does $\cF_1$.  Homotopy is an
equivalence relation in $\cQ$, with transitivity seen as follows: If
$\cT_0 \in \Hom(\F_0, \F_1)$ and $\cT_1 \in \Hom(\F_1, \F_2)$ are
homotopies, then,
\begin{equation}
  \label{eq:homotopy-compose}
\TT_1 \circ \TT_0 :=   \cT_0 + \cT_1 + m^2_{\cQ}(\cT_0, \cT_1) \in \Hom(\F_0,\F_2)
\end{equation}
is a homotopy, see \cite[p16]{se:bo}. 
The \ainfty algebras $A_0$, $A_1$ are \em{homotopy equivalent} if there are \ainfty morphisms $\F_0 : A_0 \to A_1$, $\F_1 : A_0 \to A_1$ such that $\F_0 \circ \F_1 : A_1 \to A_1$ and $\F_1 \circ F_0 : A_0 \to A_0$ are both homotopic to the identity \ainfty morphism.    

Homotopies are preserved by compositions with \ainfty morphisms. Given \ainfty morphisms
\[A_0 \xrightarrow{\F_0} A_1 \xrightarrow{\F_1, \F_2} A_2 \xrightarrow{\F_3} A_3,\]
and a pre-natural transformation $\cT \in \Hom(\F_1, \F_2)$, left and right compositions
as in \cite[p11]{se:bo} 
give pre-natural transformations. That is, $L_{\F_3} \cT \in \Hom(\F_3 \circ \F_1, \F_3 \circ F_2)$
is defined as 
\begin{multline} \label{eq:lcomp}
 (L_{\F_3} \TT)^d (a_1,\ldots,a_d) = \sum_{k,l}
  \sum_{i_1,\ldots,i_l} (-1)^\dagger \F_3^l(
  \F_0^{i_1}(a_1,\ldots,a_{i_1}), \F_0^{i_2}(a_{i_1 + 1},\ldots, a_{i_1+i_2}), 
  \ldots, \\ \TT^{i_k}(a_{i_1 + \ldots + i_{k-1} + 1},\ldots, a_{i_1 +
    \ldots + i_k}), \F_1^{i_{k+1}}( a_{i_1 + \ldots + i_k + 1},\ldots,
  ) ,\ldots, \F_1^{i_l}(a_{d - i_l},\ldots, a_d)), \end{multline}

where $\dagger$ is as in \eqref{mu1}, 
and $R_{\F_0} \cT \in \Hom(\F_1 \circ \F_0, \F_2 \circ \F_0)$
is defined as
\begin{multline}
  \label{eq:rcomp}
   (R_{\F_0} \cT)^d(a_1,\ldots,a_d)
 = \sum_{i_1 + \ldots + i_m =d}
  \cT^{m}( \F_0^{i_1}(a_1,\ldots,a_{i_1}),  \\ \ldots, \F_0^{i_m}(a_{d -
    i_m + 1},\ldots,a_d)).
\end{multline}
Furthermore, a straightforward verification shows that if $\cT$ is a
homotopy from $\F_1$ to $\F_2$, then, $L_{\F_3} \cT$ is a homotopy
from $\F_3 \circ \F_1$ to $\F_3 \circ \F_2$, and $R_{\F_0} \cT$ is a
homotopy from $\F_1 \circ \F_0$ to $\F_2 \circ \F_0$.

\section{Composition maps}

In this section, we describe \ainfty algebras, called \em{Fukaya algebras}
 whose composition maps
are given by counts of treed holomorphic maps, both in an ordinary
symplectic manifold and in a broken manifold. The boundary of the
disks map to a Lagrangian, which in the broken case, is contained in
the complement of relative divisors.

Lagrangians will be equipped with additional data called \em{brane
  structures}.  \index{Brane structure on the Lagrangian} We assume
Lagrangians are compact, connected and oriented.  A brane structure
consists of \label{page:brane}:
\begin{itemize}
\item a relative spin structure (see \cite[Chapter 44]{fooo} for the definition),
\item and a local system, which is an element 
\[\Hol_L \in \RR(L) = \Hom(\pi_1(L), \Lam^\times).\]
\end{itemize}
Typically, a grading in the sense of Seidel \cite{se:gr} is included
in the definition of brane; however, since gradings are needed only
while working with multiple Lagrangians, we do not use them in this
book.

The Fukaya algebra of a Lagrangian brane involves counts of disks that
are holomorphic with respect to coherent perturbation data
$\ul \Pe:=(\Pe_\Gamma)_\Gamma$ for all types $\Gamma$ of stable treed
disks.  The data $\ul{\Pe}$ consists of perturbations of a background
almost complex structure $J_0$ on the manifold $X$ and a Morse
function $F : L \to \R$ on the Lagrangian.  Theorem
\ref{thm:transversality} constructs coherent perturbations for a
broken manifold $\XX$; for the corresponding result on a smooth
manifold, see \cite[Theorem 4.20]{cw:flips}.  We first define
composition maps giving a (unbroken) Fukaya algebra without strict
units called the \em{geometric Fukaya algebra}; we later explain how
to upgrade the construction to include strict units.

\begin{definition} {\rm (Geometric Fukaya algebra)}
  \label{def:munbroken}
  For a Lagrangian brane $L$ and coherent perturbation datum
  $\ul \Pe:=(\Pe_\Gamma)_\Gamma$ on the manifold $X$, the
  \em{geometric Fukaya algebra} is an \ainfty algebra consisting of
  the space of \em{Floer cochains} over the Novikov ring $\Lam_{\geq 0}$
  \index{Fukaya algebra!Unbroken $CF(L,\ul \Pe)$}
\[ CF^{\on{geom}}(L, \ul \Pe) := \bigoplus_{d \in \Z_2} CF^d(L, \ul \Pe), \quad
CF^d(L, \ul \Pe) := \bigoplus_{x \in \cI_d(L) } \Lambda_{\geq 0} \bran {x} \]
where $ \cI_d(L)$ is the set of index $d$ critical points of the
Morse function $F:L \to \R$, see Definition
\ref{def:treedholdisk}; \label{page:morseref} and equipped with
\em{composition maps}
  \[ m^{d(\white)}: (CF^{\on{geom}}(L))^{\otimes d(\white)} \to
    CF^{\on{geom}}(L), \quad d(\white) \geq 0 \]
  defined 
  on generators $x_i \in \crit(F)$
  by
  \begin{equation} \label{eq:munbroken}
    m^{d(\white)}(x_1,\ldots,x_{d(\white)}) =
    \sum_{\Gamma,x_0,u \in \tM_{\Gamma}(X, L, 
\Pe_\Gamma, \ul{x})_0} w(u)  x_0 \end{equation}
  where the sum is over all rigid types $\Gamma$ of maps from Definition 
  \ref{def:unbrokenrigidtype} 
   with $d(\white)$
  incoming boundary edges, and
  \begin{equation}
    \label{eq:wtwu}
    w(u):=(-1)^{\heartsuit}(d_\black(\Gam)!)^{-1} \Hol_L([\partial u]) \eps(u)
    q^{A(u)}.
  \end{equation}
  The factors in \eqref{eq:wtwu} are as
  below:
  \begin{enumerate}
  \item $\heartsuit = {\ssum_{i=1}^{d(\white)} i|x_i|}$;
  \item $\Hol_L([\partial u]) \in \Lam^\times$ is the evaluation of the local system
    $\Hol_L \in \RR(L)$ on the homotopy class of loops
    $[\partial u] \in \pi_1(L)$ defined by going around the
    boundary of each disk component in the treed disk once;
  \item $d_\black(\Gamma)$ the number of interior markings on the map $u$; and
  \item $\eps(u) \in \{\pm 1\}$ is the orientation sign, as in Remark \ref{rem:orientmap}.
  \end{enumerate}
  This ends the Definition.
\end{definition}

\noindent In the count of disks \eqref{eq:munbroken}, we only consider rigid types, because types of maps that are not rigid occur in the boundary of a compactification of a larger dimensional stratum.

The \ainfty relation follows from the description of the boundary of
the moduli space. We recall from Section \ref{sec:tubular}, that the
\em{true boundary} of one-dimensional moduli spaces of treed
holomorphic disks consists of maps with a broken edge.
\label{truerem} 
Indeed, as a stratified space, the closures of the one-dimensional
moduli spaces have zero-dimensional strata which include both strata
where edge lengths have become zero and strata where edge lengths have
become infinity; however, only the latter types form topological
boundary strata, since the length-zero strata lie in the closures of
two one-dimension strata either by deforming to positive length or
performing the gluing construction for holomorphic disks.
\begin{theorem} \label{thm:yields-unbr} {\rm (\ainfty algebra for a Lagrangian in a symplectic manifold)} 
For a coherent 
  perturbation system $\ul{\Pe} = (\Pe_\Gamma)_\Gamma$ on the manifold
  $(X,\om_X)$, the maps $(m^{d(\white)})_{d(\white) \ge 0}$ on
  $CF^{\on{geom}}(L)$ satisfy the axioms of a convergent (possibly
  curved) \ainfty algebra $CF^{\on{geom}}(L)$.
\end{theorem}

\begin{proof}
  We prove the \ainfty associativity relations by counting the ends of
  one-dimensional moduli spaces of treed holomorphic disks. We perform
  the count for a fixed number of interior and boundary markings, and
  fixed limits on the input and output treed segments, and then sum
  over all choices.  Consider integers $d(\white), d(\black) \geq 0$
  and a tuple $\ul x \in (\cI(L))^{d(\white)+1}$ of inputs and
  outputs. We denote the one-dimensional component of the moduli space
  of treed disks of rigid type with $d(\white)$ boundary inputs and
  $d(\black)$ interior markings by
  \[\M_{d(\black),d(\white)}(L,\ul x)_1:=\cup_{\Gamma : i(\Gamma,\ul
      x)=1}\M_\Gamma(L,\ul x),\]
  where $\Gamma$ ranges over rigid types containing $d(\black)$
  interior markings and $d(\white)$ boundary inputs.  By the unbroken
  version of Proposition \ref{prop:truebdry} and Remark
  \ref{rem:true-fake} \label{rep:unbrokenversion} (obtained by
  replacing broken maps with unbroken maps in the statement; this was
  proved in \cite[Remark 4.22]{cw:flips}), the true boundary of
  $\M_{d(\black),d(\white)}(L,\ul x)_1$ consists of disks with a single
  broken edge, that is, an $e \in \Edge_{\white,-}$ with
  $\ell(e)=\infty$. The boundary points of one-dimensional moduli
  spaces occur in pairs, and therefore,
  \begin{equation}
    \label{eq:1d-ends}
    \sum_{u \in \M_\Gamma(L,\ul x)}
    \eps_\partial(u)(d(\black)!)^{-1}=0,   
  \end{equation}
  where, in the summation, $\Gamma$ ranges over all types in the true
  boundary of the moduli space $\M_{d(\black),d(\white)}(L,\ul x)_1$, and the sign
  $\eps_\partial(u) \in \{\pm 1\}$ is given by the induced orientation
  on the boundary of $\M_{d(\black),d(\white)}(L,\ul x)_1$.  Consider
  one such type $\Gamma$, and suppose $\Gamma_+$, $\Gamma_-$ are the
  treed disk types obtained by cutting the broken edge, each
  containing $d_\pm(\black):=d_\black(\Gamma_\pm)$ interior markings.
  For any critical point $x \in \cI(L)$, denote by $(\ul x, x)_+$,
  $(\ul x, x)_-$ the input-output labelling on $\Gamma_+$, $\Gamma_-$
  where both ends of the broken edge are labelled $x$. There is a
  bijection
  \begin{equation}
    \label{eq:bij-break}
    \M_\Gamma(L,\ul x) \simeq \bigcup_{x \in \cI(L)} \M_{\Gamma_+}(L,(\ul x,x)_+))_0 \times \M_{\Gamma_-}(L,(\ul x,x)_-))_0.  
  \end{equation}
  However, any map in the right-hand side corresponds to
  $\binom {d(\black)}{d_+(\black)}$ maps in the true boundary of
  $\M_{d(\black),d(\white)}(L,\ul x)_1$ as follows: There are
  $\binom {d(\black)}{d_+(\black)}$ bijections
  \[f:\{1,\dots,d_+(\black)\} \sqcup \{1,\dots,d_-(\black)\} \to
    \{1,\dots,d(\black)\}\]
  whose restrictions to $\{1,\dots,d_+(\black)\}$,
  $\{1,\dots,d_-(\black)\}$ are order-preserving. Given a pair in
  $(u_+,u_-)$ in the right-hand side of \eqref{eq:bij-break}, each
  bijection $f$ corresponds to a map $u$ in the true boundary of
  $\M_{d(\black),d(\white)}(L,\ul x)_1$, where the $i$-th interior
  marking in $u_\pm$ is labelled $f(i)$ in $u$. Therefore, the
  expression in \eqref{eq:1d-ends} is equal to
  \begin{multline}\label{eq:1d-ends-prod}
    0=\sum_{(\Gamma_+,\Gamma_-),x} \left\{\left(\sum_{u_+ \in \M_{\Gamma_+}(L,(\ul x,x)_+)}\eps(u_+) (d_\black(\Gamma_+)!)^{-1}\right) \cdot \right.\\
    \left.  \left(\sum_{u_- \in \M_{\Gamma_-}(L,(\ul x,x)_-)}\eps(u_-)
        (d_\black(\Gamma_-)!)^{-1}\right)
      \frac{\eps_\partial(u)}{\eps(u_+)\eps(u_-)}\right\},
  \end{multline}
  where $x$ ranges over critical points in $\cI(L)$ and
  $(\Gamma_+,\Gamma_-)$ range over all pairs of types obtained by
  cutting an edge in a treed disk type occurring in the boundary of
  $\M_{d(\black),d(\white)}(L,\ul x)_1$.  The sign contribution
  $\tfrac{\eps_\partial(u)}{\eps(u_+)\eps(u_-)}$ is equal to the
  difference in the orientation of the moduli space of treed disks
  with a broken edge when viewed as a boundary of a larger stratum,
  and when viewed as a product, and is equal to the shifted Koszul
  sign in the \ainfty associativity relation \eqref{ainftyassoc} by
  \cite[(12.25)]{se:bo}.  Adding the equations \eqref{eq:1d-ends-prod}
  corresponding to all $d(\black) \geq 0$, we obtain the associativity
  relation on the maps $(m^d)_{d \geq 0}$ on $CF^{\on{geom}}(L)$
  corresponding to $d=d(\white)$.  \label{convprop} The convergence
  property follows from the fact that any treed disk with no incoming
  edges must, by stability, have a disk component with interior
  markings.  Indeed, the disk components that are furthest from the
  outgoing edge have a single adjoining boundary edge, and so must
  intersect the stabilizing divisor.  By \eqref{eq:mult}, the area of
  the holomorphic disk is proportional to the number of its
  intersections with the stabilizing divisor, and as a result, the
  $q$-valuation of $m^0(1)$ is positive.
  \end{proof}
 
  The broken Fukaya algebra is defined analogously by counts of broken
  disks.

\begin{definition} {\rm (Geometric broken Fukaya algebra)}
  Let $\XX$ be a broken manifold, let $L \subset \XX$ be a Lagrangian
  brane that is contained in a single piece of $\XX$ and does not
  intersect relative divisors, and let $\ul \Pe:=(\Pe_\Gamma)_\Gamma$
  be coherent perturbation data on $\XX$ (constructed by Theorem
  \ref{thm:transversality}).  The \em{geometric broken Fukaya algebra}
  is an \ainfty algebra consisting of the space of \em{Floer
    cochains} over the Novikov ring $\Lam_{\geq 0}$
  \index{Fukaya algebra!Broken $CF_\br(L,\ul \Pe)$}
  \[ CF^{\on{geom}}_\br(L, \ul \Pe) := \bigoplus_{d \in \Z_2}
    CF^d_\br(L, \ul \Pe), \quad CF^d_\br(L, \ul \Pe) := \bigoplus_{x \in
      \cI_d(L, \ul \Pe) }
    \Lambda_{\geq 0} \bran {x}, \]
  where $ \cI_d(L)$ is as in Definition \ref{def:munbroken}; and
  equipped with \em{composition maps}
  \[ m^{d(\white)}_\br: (CF^{\on{geom}}_\br(L, \ul \Pe))^{\otimes d(\white)}
    \to CF^{\on{geom}}_\br(L, \ul \Pe), \quad d(\white) \geq 0 \]
  defined on generators $x_i \in \crit(F)$ by
  \begin{equation} m^{d(\white)}_\br(x_1,\ldots,x_{d(\white)}) =
    \sum_{x_0,\Gamma,u \in \tM^\br_{\Gamma}(\XX, L,D,\ul{x})_0}
    w(u) x_0.
  \end{equation}
%
  Here, $x_0$ ranges over critical points of the Morse function $F$,
  the combinatorial type $\Gamma$ of the broken map $u$ ranges over
  all rigid types (see Definition \ref{def:rigidtype}) with
  $d(\white)$ boundary inputs, and $w(u)$ is as in
  \eqref{eq:wtwu}. The orientation sign $\eps(u) \in \{\pm 1\}$, which
  is a factor in $w(u)$, is determined as in Remark
  \ref{rem:orientmap}.
\end{definition}

\begin{theorem} \label{thm:yields-br} {\rm (\ainfty algebra for a
    Lagrangian in a broken manifold)} For any coherent perturbation
  system $\ul{\Pe} = (\Pe_\Gamma)_\Gamma$ on the broken manifold
  $\XX$, the maps $(m_{d(\white)}^\br)_{d(\white) \ge 0}$ on
  $CF^{\on{geom}}_\br(L)$ satisfy the axioms of a convergent
  (possibly curved) \ainfty algebra $CF^{\on{geom}}_\br(L)$.
\end{theorem}

We sketch two proofs.  The first proof is an almost verbatim repeat of
the proof of the unbroken \ainfty relation in Theorem
\ref{thm:yields-unbr} from the description of the true boundary of the
moduli space in Proposition \ref{prop:truebdry}. The combinatorial
factors arising from the distribution of interior markings are
accounted exactly as in the unbroken case.  Another proof uses the
fact that the structure maps for the broken Fukaya algebra are the
limits of those in the unbroken case, by Proposition 
\ref{prop:breakingpert}, and is completed after Corollary
\ref{mdconv}.

\section{Homotopy units}\label{sec:units}
\index{Unital \ainfty algebra!Homotopy units|(}
\index{Homotopy units}

In the Fukaya algebra constructed in the previous section, a homotopy
unit construction can be applied to produce a strictly unital \ainfty
algebra. Recall that the Morse function $F:L \to \R$ used in the
construction of $CF^{\on{geom}}(L)$ is assumed to have a unique
maximum point denoted $x^{\blackt} \in \crit(F)$. In an idealized
situation where domain-dependent perturbations are not required,
$\bran{x^{\blackt}}$ is a strict unit for $CF^{\on{geom}}(L)$. This
is because a boundary marked point mapping to the unstable locus of
$x^{\blackt}$ is an empty constraint, and such a marking can be
forgotten without affecting the disk. In our setting, marked points
cannot 
be forgotten because domain-dependent perturbations depend on
them. The homotopy unit construction is a way of enhancing the Fukaya
algebra so that the perturbation system admits forgetful maps, so that
the algebra admits a strict unit and the potential of
\eqref{eq:potdef} is well-defined.

We outline the idea of the homotopy units construction: For the Fukaya
algebra to have a unit, we would like the perturbations on the Morse
function to be domain-independent on the treed segments asymptotic to
the maximum point $x^{\blackt} \in \crit(F)$. However, we 
cannot
impose such a condition on the perturbation, because the perturbation
depends only on the ``domain'' which does not include the information
about which critical point a Morse trajectory asymptotes to. One can
not include the critical point label into the domain data, because
Morse trajectories may break in the limit, and the label at the
breaking point is not known beforehand.  Therefore, we add a new kind
of boundary marking to the domain, which is ``forgettable'', and will
serve as a unit. The treed segment at a forgettable leaf is required
to asymptote to the maximum of the Morse function $F$, but is labelled
$x^\whitet$ (to distinguish it from $x^\blackt$).  We also add
`weighted boundary markings'', labelled $x^{\greyt}$, to the domain
which give us a way of homotoping between treed segments with
domain-dependent perturbations asymptoting to $x^\blackt$ and those
with domain-independent perturbations asymptoting to $x^\whitet$. The
data of the domain now includes the information about whether a leaf
is forgettable, weighted or unforgettable. This new kind of domain is
called a \em{weighted treed disk} and is defined below after the
statement of the main Theorem.

\begin{theorem} \label{thm:wts}{\rm(Homotopy unit construction)}
  Suppose $\ul{\Pe}$ is a coherent perturbation datum for treed
  holomorphic disks, and suppose
  $CF^{\on{geom}}(L):=CF^{\on{geom}}(L,\ul{\Pe})$ is the \ainfty
  algebra whose composition maps count $\ul{\Pe}$-holomorphic disks.  Then
  there exists a convergent strictly unital \ainfty structure on the
  vector space
  \begin{equation}
    \label{eq:genunit}
    CF(L) := CF^{\on{geom}}(L) \oplus \Lam_{\geq 0} x^{\greyt}[1] \oplus \Lam_{\geq 0} x^{\whitet},   
  \end{equation}
  with gradings 
  \[|x^{\whitet}|  = 0, \quad |x^{\greyt}| = -1,  \] 
  whose composition maps count weighted $\ul{\Pe}^{\on{wt}}$-holomorphic disks (see Definition \ref{def:wt-treed-hol}),
  where $\ul{\Pe}^{\on{wt}}$ is an extension of the perturbation datum
  $\Pe$ to weighted disks; and in the resulting \ainfty algebra,
  \begin{enumerate}
  \item \label{part:wts2}
    $x^{\whitet}$ is a strict unit,
  \item \label{part:wts1} $CF^{\on{geom}}(L) \subset CF(L)$ is a
    \ainfty sub-algebra, and
  \item \label{part:wts3}
    \begin{equation}
      \label{eq:m1grey}
      m^1(x^{\greyt})=x^{\blackt} - x^{\whitet} \mod \Lam_{>0}.     
    \end{equation}
  \end{enumerate}
\end{theorem}
\noindent The theorem is proved later in the section after defining
weighted treed holomorphic disks.  We remark that the expression
\eqref{eq:m1grey} for $m^1(x^{\greyt})$ is similar to that in
Fukaya-Oh-Ohta-Ono \cite[(3.3.5.2)]{fooo}.
\index{Unital \ainfty algebra!Homotopy units|)}

The condition that $x^{\whitet}$ is a strict unit determines all
\ainfty structure maps involving occurrences of $x^{\whitet}$. In the
following geometric construction of a homotopy unit, the axioms are
designed keeping this fact in mind.

\begin{definition} \label{wdef}
  \begin{enumerate}
  \item {\rm (Weightings)} A \em{weighting} of a treed disk
    $C= S \cup T$ of type $\Gamma$, with $S \neq \emptyset$, consists
    of a partition of the boundary semi-infinite edges
    \[\Edge^{\blackt}(\Gamma) \sqcup \Edge^{\greyt}(\Gamma) \sqcup
      \Edge^{\whitet}(\Gamma) = \Edge_{\white,\rightarrow}(\Gamma) \]
    into \em{unforgettable} resp. \em{weighted} resp. \em{
      forgettable}, and a \em{weight} on semi-infinite edges
    $\rho: \Edge_{\white,\rightarrow}(\Gamma) \to [0,\infty]$
    satisfying
    \[
    \rho(e) \in
    \begin{cases} \{ 0 \} & e \in \Edge^{\blackt}(\Gamma) \\
      [0,\infty] & e \in
      \Edge^{\greyt}(\Gamma) \\
      \{ \infty \} & e \in \Edge^{\whitet}(\Gamma).
    \end{cases}
    \]
    The weighting $\rho$ satisfies the following axiom:
    \begin{itemize}
    \item[] \label{item:outgoing} {\rm(Outgoing edges axiom)}
      \index{Outgoing edges axiom for weighted disks} A disk output
      $e_0 \in \Edge_{\white}(\Gamma)$ can be weighted only if the
      disk has exactly one weighted input
      $e_1 \in \Edge^{\greyt}(\Gamma)$, all the other inputs
      $e_i \in \Edge(\Gamma), i \neq 1$ are forgettable, and there are
      no interior leaves, $\Edge_{\black}(\Gamma) = \emptyset$.  In
      this case, the output $e_1$ has the same weight
      $\rho(e_1) = \rho(e_0)$ as the weighted input $e_0$.  A disk
      output $e_0$ can be forgettable only if all the inputs are
      forgettable, and there are no interior leaves.  In all the other
      cases, the output of a disk is unforgettable.
    \end{itemize}
    In the exceptional case that the treed disk $C$ is an infinite
    tree segment and does not have surface components, the only
    possible labels are
    \begin{equation}
      \label{eq:only-labels}
      \greyt \to \whitet, \quad \greyt \to \blackt, \quad \text{or}
      \quad \blackt \to \blackt .
    \end{equation}
    In the first two cases, the input has weight $\rho(e)$ equal to
    $\infty$ resp. $0$.
  \item {\rm (Stability)} A weighted treed disk $C=S \cup T$ with
    $S \neq \emptyset$ is \em{stable} if $C$ is stable as a treed
    disk.  In case $S=\emptyset$ and $C$ is an infinite segment, then
    $C$ is stable iff the labels are $\greyt \to \whitet$ or
    $\greyt \to \blackt$.
  \item {\rm (Isomorphism)} Two weighted treed disks $C$ and $C'$ are
    isomorphic if there is an isomorphism of treed disks
    $\phi:C \to C'$, the edge labels are identical, and the following
    is true.
    \begin{enumerate}
    \item If the output edge is not weighted, then the weights on the
      inputs of $C_1$ and $\phi(C_1)$ are equal;
    \item if the output edge $e_0$ is weighted, then the weights on
      the inputs are equal up to scalar multiplication, i.e.
      \begin{equation} \label{scalar} \exists \lambda\in (0,\infty): \
        \forall e \in \Edge_{\circ,\rightarrow}(C) \bs \{e_0\} \quad \rho(e)
        = \lambda \rho'(\phi(e)).
      \end{equation}
    \end{enumerate}
    Consequently, if the output edge is weighted, since there is
    exactly one incoming edge by the
    \hyperref[item:outgoing]{(Outgoing edges axiom)}, graphs with
    different weights may be isomorphic.
  \item{\rm(Combinatorial type)}
\index{Combinatorial type! of a weighted treed disk}
    The \em{type} of a weighted treed disk is
    given by the type of the treed disk, and the labels
    $\{\whitet,\greyt,\blackt\}$ at the inputs and outputs, and
    whether the weight at any vertex is zero, infinite or neither. Thus the combinatorial type of a weighted treed disk includes a partition of weighted edges
    \[\Edge^{\greyt}(\Gamma)=\Edge^{\greyt}_0(\Gamma) \cup \Edge^{\greyt}_{(0,\infty)}(\Gamma) \cup \Edge^{\greyt}_\infty(\Gamma)\]
    into edges of weight $0$, non-zero finite, and infinity.
  \end{enumerate}
\end{definition}

The moduli space $\M_\Gamma$ of weighted treed disks can be identified
with
\begin{equation*}
  \begin{cases}
    \M_{\Gamma'} \times [0,\infty]^{|\Edge^{\greyt}(\Gamma)|}, \quad \text{if the output label is not $\greyt$}\\
    \M_{\Gamma'}, \quad \text{if the output edge is $\greyt$},
  \end{cases}
\end{equation*}
where $\Gamma'$ is the type of treed disk obtained by forgetting the
weighting. If the type $\Gamma$ is $\greyt \to \whitet$ resp.
$\greyt \to \blackt$, then $\M_\Gamma$ is a point.

The (Cutting edges) morphism has some additional features for weighted
treed disks.
\begin{definition}\label{item:cutedgeinweighted} 
  {\rm(Cutting edges in weighted treed disks)} Given a type $\Gamma$
  of a weighted treed disk, we say that the weighted disk types
  $\Gamma_+$, $\Gamma_-$ (here $\Gamma_+$ contains the root of
  $\Gamma$) are produced by cutting $\Gamma$ at an infinite edge
  $e \in \Edge_{\white,-}(\Gamma)$ if $\Gamma_+$, $\Gamma_-$ (with
  weights forgotten) are the treed disk types produced by cutting the
  edge $e$ in $\Gamma$ as in Definition \ref{def:pertops}, and labels
  and weights assigned as follows: Assuming that
$e_\pm \in \Edge_{\white,\rightarrow}(\Gamma_\pm)$ are the pair of leaves 
created by the cutting,  the label ($\blackt$,
$\greyt$ or $\whitet$) and the weight at $e_+$ are the same as that of
$e_-$, and the label and weight at $e_-$ is determined by the
\hyperref[item:outgoing]{(Outgoing edges axiom)} applied to
$\Gamma_-$.
\end{definition}

There is a new type of morphism for weighted disk types,
wherein we cut a weighted incoming edge of weight $0$ or $\infty$.
\begin{itemize}
\item[] \label{item:cuttingweightedinput} {\rm(Cutting a weighted input
    edge)} Suppose $e \in \Edge^{\greyt}(\Gamma)$ is an input, and
  $\rho(e)=0$ resp. $\infty$. Cutting $e$ produces two types :
  $\Gamma_-$ is an infinite segment $\greyt \to \blackt$ resp.
  $\greyt \to \whitet$, and $\Gamma_+$ is $\Gamma$ with $e$ as an
  unforgettable resp. forgettable edge; see Figure \ref{fig:eqwt}.
  Additionally, if $\Gamma$ has a weighted output $e_0$, we make the
  output edge $e_0$ in $\Gamma_+$ unforgettable resp. forgettable.
\end{itemize}  

\begin{figure}[ht]
  \centering \scalebox{.8}{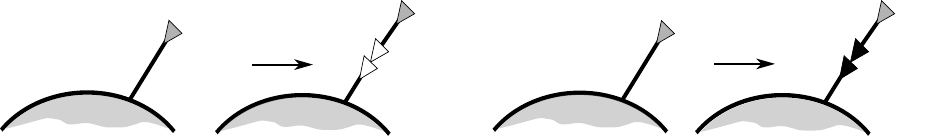}
  \caption{Cutting a weighted input edge.   
  The arrows on each input and output edge indicate whether
  the weighting is infinite (white) 
  finite and non-zero (grey) or zero (black).}
  \label{fig:eqwt}
\end{figure}

Perturbation data defined on the moduli space of weighted treed disks
are required to be coherent with respect to
\hyperref[item:collapsingedgesmorphism]{(Collapsing edges)},
\hyperref[item:makingedgelengthmorph]{(Making an edge length finite or non-zero)},
\hyperref[item:makingedgewt]{(Making an edge weight finite or non-zero)},
\hyperref[item:cutedgeinweighted]{(Cutting edges)},
\hyperref[item:cuttingweightedinput]{(Cutting a weighted input edge)},
and the
\hyperref[item:localityaxiom]{(Locality axiom)}.
The perturbation vanishes on infinite segments
$\greyt \to \blackt$ and $\greyt \to \whitet$.  Additionally, the
following coherence conditions are satisfied:
\begin{enumerate}
  \index{Making an edge weight finite/non-zero} 
\item {\rm(Making an edge weight finite/non-zero)}
  \label{item:makingedgewt}
  If the weighted
  treed disk type $\Gamma$ is obtained from $\Gamma'$ by making the
  weight of an edge $e \in \Edge^{\greyt}(\Gamma')$ finite or
  non-zero, then $\Pe_{\Gamma'}$ is the pullback of $\Pe_{\Gamma}$ by
  the inclusion of the universal moduli space
  $\ol \U_{\Gamma'} \to \ol \U_\Gamma$.
  \index{Forgetting edges morphism}
\item \label{item:forgettingedges} {\rm (Forgetting edges)} Suppose
  $e$ is an input edge in a weighted treed type $\Gamma$ that is
  either forgettable or weighted with infinite weight, and $\Gamma'$
  is the type obtained by forgetting $e$. Then $\Pe_\Gamma$ is the
  pullback of $\Pe_{\Gamma'}$.
  %
\end{enumerate}

\begin{definition}\label{def:wt-treed-hol}
  {\rm(Weighted treed holomorphic disks)}
  \index{Holomorphic disks!Weighted treed holomorphic disks}
  Let $\ul \Pe=(\Pe_\Gamma)_\Gamma$ be a coherent perturbation datum for
  weighted treed disks. A \em{weighted $\Pe_\Gamma$-holomorphic disk}
  is a map $u:C \to X$ that is
  a $\Pe_\Gamma$-holomorphic disk, and additionally satisfies the following:
\begin{itemize}
\item[] {\rm(Label axiom)} A treed input or output segment labelled
  $\greyt$ resp. $\whitet$ asymptotes to the maximum point
  $x^{\blackt} \in \crit(F)$.
\end{itemize}
\noindent An weighted holomorphic disk $u: C \to X$ is \em{
  stable} if for any component of the domain $C_v$, either the map
$u_v$ is non-constant, or the domain $C_v$ is stable in the sense of
weighted treed disks. The new feature of stability for weighted
holomorphic disks is that a stable map $u$ may be constant on an
infinite tree segment labelled $\greyt \to \blackt$ or
$\greyt \to \whitet$.  The \em{combinatorial type}
\index{Combinatorial type! of a weighted holomorphic disk} of a
weighted treed holomorphic disk consists of the type of the domain
weighted treed disk, together with the combinatorial type data of the
treed holomorphic map, namely homology classes of the maps on surface
components, and intersection multiplicities with the stabilizing
divisor at the markings. The type $\Gamma$ of a weighted treed
holomorphic disk is \em{rigid} \index{Rigid! weighted map} if all the
edges $e \in \Edge_-(\Gamma)$ are boundary edges with finite non-zero
length, and in case of weighted inputs or output, the weight $\rho(e)$
is finite and non-zero.

There is an additional notion of equivalence for holomorphic weighted
treed disks. Two weighted treed holomorphic disks $u$, $u'$ are equivalent,
if $u$ has a weighted leaf $e$ with weighting $\rho(e_i) = \infty$ resp. $0$,
 and $u'$ is obtained from $u$ by attaching to $e$ a constant trajectory
$u'': \R \to L$ with weighted incoming $e^-$ and forgettable
resp. unforgettable outgoing edge $e^+$. See Figure \ref{fig:eqwt}. 
This ends the Definition.
\end{definition}

The expanded set of labels on the ends of treed segments is denoted by
\begin{equation}
  \label{eq:Ihat}
  \hat \cI(L):=\cI(L) \cup \{x^{\greyt},
  x^{\whitet}\},
\end{equation}
where we recall from \eqref{eq:Igeom} that $\cI(L)$ is the set of critical points of the Morse function on the Lagrangian $L$.

\begin{proposition}\label{prop:wt-transv}
  {\rm(Transversality for weighted treed holomorphic disks)} Given a
  regular coherent perturbation datum $\ul{\Pe}$ for treed holomorphic
  disks, $\ul{\Pe}$ extends to a regular perturbation datum
  $\ul{\Pe}^{\on{wt}}$ on weighted treed holomorphic disks such that
  the following holds: For an uncrowded type $\Gamma$ of weighted
  holomorphic maps, and a prescribed tuple of inputs
  $\ul x:=(x_1,\dots,x_{d(\white)}) \in \hat \cI(L)^{d(\white)}$ and
  an output $x_0 \in \hat \cI(L)$ respecting the (Label axiom) and for
  which $i(\Gamma,\ul x) \leq 1$, the moduli space
\[\M_{\Gamma}(L,\ul \Pe, \ul{x})\]
of weighted $\ul \Pe$-holomorphic treed disks with limits $\ul{x}$ is a
smooth manifold of expected dimension. The moduli space
$\M_{\Gamma}(L,\ul \Pe, \ul{x})$ is compact if $i(\Gamma,\ul x)=0$,
and if $i(\Gamma,\ul x)=1$, $\M_{\Gamma}(L,\ul \Pe, \ul{x})$ has a
compactification whose boundary consists of configurations with a
boundary node with length $0$ or $\infty$, or a weighted edge with
weight $0$ or $\infty$.
\end{proposition}
\begin{proof}[Proof of Proposition \ref{prop:wt-transv}]
  We extend the perturbation datum $\ul \Pe=\{\Pe_\Gamma\}_\Gamma$ on
  domain curves that only contain unforgettable leaves to those that
  have other kinds of leaves as well.  First, consider domain types
  whose output is unforgettable.  Suppose that the domain type
  $\Gamma$ has a single weighted/forgettable input leaf $e$, since the
  other cases follow inductively.
  Let $\Gamma_0$ resp. $\Gamma_\infty$ be the domain type obtained by
  making the weight of $e$ in $\Gamma$ zero resp. infinity.  By the
  \hyperref[item:cuttingweightedinput]{(Cutting a weighted edge)} morphism,
  on the subset $\{\rho(e)=0\}$ resp. $\{\rho(e)=\infty\}$ of
  $\M_\Gamma$, the perturbation $\tilde \Pe_\Gamma$ is given by
  $\tilde \Pe_{\Gamma_0}$ resp.  $\tilde \Pe_{\Gamma_\infty}$. The
  perturbation $\tilde \Pe_{\Gamma_0}$ is equal to
  $\Pe_{\Gamma_0}$. The perturbation $\tilde \Pe_{\Gamma_\infty}$ is
  defined via the \hyperref[item:forgettingedges]{(Forgetting edges)}
  axiom, which means that the perturbation is independent of
  forgettable treed segment on the domain curve.  By standard
  arguments, the perturbation $\tilde \Pe_\Gamma$ can be extended over
  $\{\rho(e)\in (0,\infty)\}$ while satisfying regularity.

  For domain curves whose output is weighted or forgettable, we define
  the perturbation to be domain-independent. Indeed, on such domains
  maps are constant and lie in the maximum point in
  $\crit(F) \subset L$; and strata with only unforgettable leaves and
  root do not occur in the compactification of strata where the output
  is forgettable/weighted.

  Next we prove the statement on compactification.  We analyze the
  cases where a sequence of weighted maps of type $\Gamma$ that are
  non-constant on either the surface or tree part converges to a limit
  that has a breaking on a weighted/forgettable leaf $e$, and the map
  is non-constant on both segments incident at the breaking. The other
  cases are covered either by the compactification result for treed
  holomorphic maps without weights; or, if the maps in the sequence
  are constant, by straightforward formal arguments.

  If $e$ is weighted, for dimension reasons, the infinite segment
  $(-\infty,\infty)$ in the broken segment maps to the maximum point
  in $L$, and therefore, the map is constant on this segment. If $e$
  is forgettable, by the \hyperref[item:forgettingedges]{(Forgetting
    edges)} axiom, the only case to be considered is when
  $i(\Gamma,\ul x)=1$ and $\Gamma$ is of the form shown in Figure
  \ref{fig:forget-eg1}.
\begin{figure}[ht]
    \scalebox{.8}{
\begingroup%
  \makeatletter%
  \providecommand\color[2][]{%
    \errmessage{(Inkscape) Color is used for the text in Inkscape, but the package 'color.sty' is not loaded}%
    \renewcommand\color[2][]{}%
  }%
  \providecommand\transparent[1]{%
    \errmessage{(Inkscape) Transparency is used (non-zero) for the text in Inkscape, but the package 'transparent.sty' is not loaded}%
    \renewcommand\transparent[1]{}%
  }%
  \providecommand\rotatebox[2]{#2}%
  \newcommand*\fsize{\dimexpr\f@size pt\relax}%
  \newcommand*\lineheight[1]{\fontsize{\fsize}{#1\fsize}\selectfont}%
  \ifx\svgwidth\undefined%
    \setlength{\unitlength}{117.74408031bp}%
    \ifx\svgscale\undefined%
      \relax%
    \else%
      \setlength{\unitlength}{\unitlength * \real{\svgscale}}%
    \fi%
  \else%
    \setlength{\unitlength}{\svgwidth}%
  \fi%
  \global\let\svgwidth\undefined%
  \global\let\svgscale\undefined%
  \makeatother%
  \begin{picture}(1,0.31834121)%
    \lineheight{1}%
    \setlength\tabcolsep{0pt}%
    \put(0,0){\includegraphics[width=\unitlength,page=1]{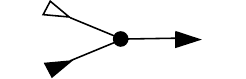}}%
    \put(-0.00468968,0.030698){\color[rgb]{0,0,0}\makebox(0,0)[lt]{\lineheight{1.25}\smash{\begin{tabular}[t]{l}$x_1$\end{tabular}}}}%
    \put(0.83511636,0.13145561){\color[rgb]{0,0,0}\makebox(0,0)[lt]{\lineheight{1.25}\smash{\begin{tabular}[t]{l}$x_0$\end{tabular}}}}%
  \end{picture}%
\endgroup%
}
\caption{A type of map. The black arrowheads mean that the input/output is unforgettable. The white attowhead means that the input/output is forgettable.}
\label{fig:forget-eg1}
\end{figure}
The map is constant on the surface, and the Morse index of $x_0$,
$x_1 \in \crit(F)$ differ by $1$. We may assume the domain-dependent
Morse perturbation on the unforgettable edges is small enough that the
ends of the moduli space are given by the breaking of the
unforgettable input or output edge as in Figure \ref{fig:forget-eg2}
below.
\begin{figure}[ht]
  \begin{equation*}
    \scalebox{.8}{
\begingroup%
  \makeatletter%
  \providecommand\color[2][]{%
    \errmessage{(Inkscape) Color is used for the text in Inkscape, but the package 'color.sty' is not loaded}%
    \renewcommand\color[2][]{}%
  }%
  \providecommand\transparent[1]{%
    \errmessage{(Inkscape) Transparency is used (non-zero) for the text in Inkscape, but the package 'transparent.sty' is not loaded}%
    \renewcommand\transparent[1]{}%
  }%
  \providecommand\rotatebox[2]{#2}%
  \newcommand*\fsize{\dimexpr\f@size pt\relax}%
  \newcommand*\lineheight[1]{\fontsize{\fsize}{#1\fsize}\selectfont}%
  \ifx\svgwidth\undefined%
    \setlength{\unitlength}{304.31340119bp}%
    \ifx\svgscale\undefined%
      \relax%
    \else%
      \setlength{\unitlength}{\unitlength * \real{\svgscale}}%
    \fi%
  \else%
    \setlength{\unitlength}{\svgwidth}%
  \fi%
  \global\let\svgwidth\undefined%
  \global\let\svgscale\undefined%
  \makeatother%
  \begin{picture}(1,0.14579494)%
    \lineheight{1}%
    \setlength\tabcolsep{0pt}%
    \put(0,0){\includegraphics[width=\unitlength,page=1]{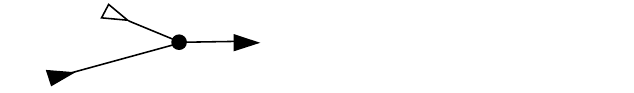}}%
    \put(-0.00181452,0.01574557){\color[rgb]{0,0,0}\makebox(0,0)[lt]{\lineheight{1.25}\smash{\begin{tabular}[t]{l}$x_1$\end{tabular}}}}%
    \put(0.41523554,0.06842307){\color[rgb]{0,0,0}\makebox(0,0)[lt]{\lineheight{1.25}\smash{\begin{tabular}[t]{l}$x_0$\end{tabular}}}}%
    \put(0,0){\includegraphics[width=\unitlength,page=2]{forget-eg2.pdf}}%
    \put(0.19490727,0.01863118){\color[rgb]{0,0,0}\makebox(0,0)[lt]{\lineheight{1.25}\smash{\begin{tabular}[t]{l}$x_0$\end{tabular}}}}%
    \put(0,0){\includegraphics[width=\unitlength,page=3]{forget-eg2.pdf}}%
    \put(0.65483619,0.01084986){\color[rgb]{0,0,0}\makebox(0,0)[lt]{\lineheight{1.25}\smash{\begin{tabular}[t]{l}$x_1$\end{tabular}}}}%
    \put(0.93620369,0.07348569){\color[rgb]{0,0,0}\makebox(0,0)[lt]{\lineheight{1.25}\smash{\begin{tabular}[t]{l}$x_0$\end{tabular}}}}%
    \put(0,0){\includegraphics[width=\unitlength,page=4]{forget-eg2.pdf}}%
    \put(0.83099653,0.03854806){\color[rgb]{0,0,0}\makebox(0,0)[lt]{\lineheight{1.25}\smash{\begin{tabular}[t]{l}$x_1$\end{tabular}}}}%
  \end{picture}%
\endgroup%
},
  \end{equation*}
\caption{Breaking an edge}
\label{fig:forget-eg2}
\end{figure}
This rules out the possibility of breaking on the forgettable edge.
\end{proof}

For a coherent perturbation system $\ul \Pe=(\Pe_\Gamma)_\Gamma$ for
weighted treed types define composition maps on the generators of the
Fukaya algebra
\[CF(L,\ul \Pe):=CF^{\on{geom}}(L,\ul \Pe) \oplus
  \Lam_{\geq 0}\bran{x^{\greyt}} \oplus \Lam_{\geq 0}\bran{x^{\whitet}}\]
as
\begin{equation} \label{ongenerators-wt}
  m^{d(\white)}(x_1,\ldots,x_{d(\white)}) = \sum_{x_0,u
    \in\ol{\M}_{\Gamma}(L,\ul \Pe,\ul{x})_0}  w(u) x_0
\end{equation} 
where 
\[w(u)=(-1)^{\heartsuit} (d(\black)!)^{-1} \Hol_L([ \partial u]) q^{A(u)} \eps(u),
  \quad \heartsuit = {\sum_{i=1}^{d(\white)} i|x_i|}.\]
The sum ranges over all rigid types $\Gamma$ of weighted treed
holomorphic disks as in Definition \ref{def:wt-treed-hol}. That is,
$\Gamma$ is a rigid type for broken maps (Definition
\ref{def:rigidtype}) after forgetting the weighting, and additionally,
the weight for any weighted input is finite non-zero in the type
$\Gamma$. Rigidity includes the latter condition because a weighted
edge with weight $0$ or $\infty$ may be deformed to a weighted edge
with finite non-zero weight.

We now prove the main Theorem of this section.

\begin{proof}
  [Proof of Theorem \ref{thm:wts}] By Proposition
  \ref{prop:wt-transv}, the coherent perturbation data $\ul \Pe$ for
  treed holomorphic disks extend to coherent data $\ul{\Pe}^{\on{wt}}$
  for weighted treed holomorphic disks.  We first prove properties
  \eqref{part:wts1}-\eqref{part:wts3}, and then proceed to show the
  \ainfty associativity relations.  The geometric part
  $CF^{\on{geom}}(L)$ is a sub-algebra because if the inputs to a
  treed disk are unforgettable, the output is also unforgettable,
  proving \eqref{part:wts1}.

  The element $x^{\whitet}$ is a strict unit for the following
  reasons. For $d(\white)>2$, we have
  \[ m^{d(\white)}(\dots,x^{\whitet},\dots)=0 \]
  because the input $x^{\whitet}$ is an empty constraint, and can be
  forgotten because the perturbation satisfies the
  \hyperref[item:forgettingedges]{(Forgetting edges)} axiom. The term
  $m^1(x^{\whitet})$ is also zero : any disk that is counted has
  interior markings, and therefore, placing the marked point
  $x^{\whitet}$ adds one to the dimension, and therefore the moduli
  space is not zero-dimensional. Finally, by the same argument,
  $m^2(x^{\whitet},y)$ and $m^2(y,x^{\whitet})$ do not count any disk
  with interior markings. The only contributions are from constant
  disks.  We conclude that both terms are equal to $\pm y$, for any
  generator $y$, proving \eqref{part:wts2}.

  For \eqref{part:wts3}, we recall from \eqref{eq:only-labels} that
  the infinite tree segments with labels $\greyt \to \blackt$,
  $\greyt \to \whitet$ are rigid configurations with orientations $+1$
  and $-1$ (see \cite[Remark 4.24]{cw:flips}). These contribute
  $x^{\blackt}-x^{\whitet}$ to $m^1(x^{\greyt})$. The other terms in
  $m^1(x^{\greyt})$ necessarily arise from configurations with at
  least one non-constant disk component, by the definition of
  stability, and so have positive $q$-valuation.

  To show that the \ainfty associativity relationships are satisfied,
  we consider one-dimensional moduli space of weighted maps of rigid
  type.  By Proposition \ref{prop:wt-transv}, the true boundary strata
  contain one of the following configurations : a configuration with a
  weight $0$ or $\infty$ at a weighted input which is equivalent to
  the broken configuration in Figure \ref{fig:eqwt}, a boundary node
  with a broken segment, or a broken Morse trajectory. These
  configurations exactly correspond to the terms in the \ainfty
  associativity relations, and so $CF(L)$ is an \ainfty algebra.
\label{rep:addedassoc}
In particular, the former boundary configuration contributes
$m^i(\dots, x^{\blackt}, \dots)$, $m^i(\dots, x^{\whitet}, \dots)$ to
terms of the form $m^i(\dots,m^1(x^{\greyt},\dots))$ in the
associativity relation.  Any other kind of contribution to a term of
the form $m^i(\dots,m^j(\dots),\dots)$ in the associativity relation,
comes from a non-constant disk breaking off, that is, the map is
constant on either the surface or the tree part.
  \end{proof}

  We recall from \cite{cw:flips} a useful criterion for a Lagrangian
  brane $L$ to be weakly unobstructed.

  \begin{lemma}\label{lem:unobs-cond}
    {\rm(A criterion for unobstructedness \cite[Lemma
      4.43]{cw:flips})} Suppose for a Fukaya algebra $CF(L)$ of a Lagrangian submanifold
    $L \subset X$, 
    $m^0(1)=W x^\blackt$ for some $W \in \Lam_{\geq 0}$ and every non-constant
    disk has positive Maslov index. Then, $b:=Wx^{\greyt} \in MC(L)$ and
    $m^0_b(1)=W x^\whitet$.
\end{lemma}

\section{Quilted disks}
\label{qdisks} 

Morphisms between Fukaya algebras are defined by counts of quilted
holomorphic disks.  The domains of such disks are \em{quilted disks},
which are ordinary disks with a \em{quilting circle} as defined in
Ma'u-Woodward \cite{mau:mult}. In this
Section, we describe moduli spaces of treed nodal quilted disks.

\index{Disk!Quilted disk}

\begin{definition} {\rm (Quilted disks)}\label{def:qdisk}
\label{def:tqdisk}
\begin{enumerate}
\item A \em{quilted disk} $(S, Q, \ul z)$ is a complex disk
  $S \simeq \D^2 \subset \C$ with a collection $\ul z$ of interior and
  boundary markings, and a \em{quilting circle}, which is a circle
  $Q \subset S$ tangent to the $0$-th boundary marking $z_0$. Two
  quilted disks $(S_0,Q_0, \ul z^0)$, $(S_1,Q_1,\ul z^1)$ are \em{
    isomorphic} if there is a biholomorphism which maps each marking
  in $S_0$ to the corresponding marking in $S_1$, and which maps the
  quilting circle $Q_0 \subset S_0$ to the quilting circle
  $Q_1 \subset S_1$.
\item{\rm(Treed quilted disk)}\label{part:tqdisk} A \em{treed quilted
    disk} is a treed nodal disk $C$, a subset of whose components are
  quilted components.  The subset of quilted components satisfies the
  following:
  \begin{itemize}
  \item A path in $C$ from any (non-root) boundary or interior marking
    $z_e$, $e \in \Edge_{\to} \bs \{0\}$ to the root marking $z_0$
    intersects exactly one quilted disk.
  \item \label{item:equal-lengths} {\rm(Equal lengths condition)} For the disk
    component $S_{v_0}$ containing the outgoing leaf, and quilted
    components $S_{v_1}, S_{v_2}$, the sum of lengths $\ell(e)$ of
    treed segments $T_e$ connecting $S_{v_0}$ to
    $S_{v_k}, k \in \{1,2\}$ is independent of $k$.
  \end{itemize}
    
\item {\rm (Combinatorial types)} \index{Combinatorial type! of a
    quilted treed disk} The combinatorial type $\Gamma$ of a quilted
  treed disk is the combinatorial type of the treed disk (without the
  quilting datum) together with a subset of the boundary vertex set
  \[\Ver^\quilt(\Gamma) \subset \Ver_\white(\Gamma),\]
  corresponding to quilted disks, and are called \em{quilted
    vertices}.
  \index{Vertex! Quilted vertex}
\item{\rm(Stability)} A treed quilted disk $C=S \cup T$ is stable if
  the automorphism group of any surface component
  $S_v, v \in \Ver(\Gamma)$ is trivial, and there are no treed
  segments both of whose ends are infinite. This means a quilted disk is
  stable if it has at least two special points.  
\end{enumerate}
\end{definition}

\begin{remark}
  {\rm(Interior markings on quilted disks)}
   An interior point in a quilted disk component
  may lie inside, outside or on the quilting circle; the relative position of
  interior points with respect to the quilting circle does not affect the
  combinatorial type of the treed quilted
  disk. \label{rep:footnote-qs} \footnote{In \cite{cw:flips}, the
    compactification of the moduli space of quilted disks contains
    configurations with ``quilted spheres''. Quilted spheres arise in
    the compactification only if we impose a requirement that markings
    cannot
    lie on dark components, that is, on unquilted components
    lying below the quilting circle.  }
\end{remark}

\begin{remark}\label{rem:qcircle}
  {\rm(Quilted disks via affine structures)}
 A quilted disk may alternately be defined as 
  a disk $S \simeq \D^2 \subset \C$ with markings $\ul z$, and a
    biholomorphism
    \[\phi:(S \bs \{z_0\},\partial S) \xrightarrow{\phi} (\H, \partial \H),\]
    which we call an \em{affine structure}.  Two affine structures
    $\phi_0$, $\phi_1$ are equivalent if there exists $\xi \in \R$
    such that $\phi_1(z)=\phi_0(z)+\xi$.  Two quilted disks
    $(S_0,\ul z^0,\phi_0)$, $(S_1,\ul z^1,\phi_1)$ are \em{
      isomorphic} if there exists $\xi \in \R$ such that, defining
    $\tau(z):=z+\xi$, the biholomorphism
    $\phi_1^{-1} \circ \tau \circ \phi_0 : S_0 \to S_1$ maps each
    marking in $S_0$ to the corresponding marking in $S_1$.  A choice
    of a quilting circle is equivalent to an affine structure, by
    taking the quilting circle to be $Q = \{ \on{Im}(z) = 1 \}$.
\end{remark}

The moduli space of stable quilted disks with interior and boundary
markings is homeomorphic to a compact cell complex.  More precisely,
the moduli space is a compact topological space admitting a
Thom-Mather stratification into smooth submanifolds with toric
singularities.  As the interior and boundary markings go to infinity,
the markings bubble off onto either quilted disks or unquilted disks
or spheres.  The case of combined boundary and interior markings is a
straight-forward generalization of Ma'u-Woodward \cite{mau:mult}; see
also Bottman-Oblomkov \cite{bottman:marked} where a generalization to
markings on arbitrary number of quilting circles is constructed.  In the case with no
interior markings, the moduli spaces are cell complexes called the
``multiplihedra'' introduced by Stasheff \cite{st:ho}.  \index{Moduli
  space!of quilted disks $\M^\quilt_{\Gamma}$} We denote by
\[{\M}^\quilt_{d(\black), d(\white)}\]
the moduli space of stable marked quilted treed disks with $d(\white)$
boundary leaves and $d(\black)$ interior leaves.

\begin{remark}\label{rem:shading}
  {\rm(Shadings on a treed quilted disk)} In Figures, we represent
  quilted disks using shadings.  The subset of a quilted disk lying
  inside resp. outside the quilting circle is shaded dark
  resp. light. An unquilted surface component $S_v$ of a treed quilted
  disk is shaded light resp. dark if the path from $v$ to the root
  vertex contains resp. does not contain a quilted vertex. See Figures
  \ref{fig:MWc}, \ref{fig:quilt-02} for examples.
\end{remark}

\begin{example}
  Figure \ref{fig:MWc} depicts the moduli space ${\M}^\quilt_{0,2}$
  of treed quilted disks with two incoming leaves, one outgoing leaf,
  and no interior marking; and Figure \ref{fig:quilt-02} depicts the
  moduli space ${\M}_{1,0}^q$ of treed quilted disks with one
  outgoing leaf and one interior marking. In both cases, there are 3
  open strata, interspersed with four strata of codimension one. 
  Figure \ref{fig:MWc} shows 
  the type of each of these strata. 
  The hashes on the line segments $T_e$ indicate breakings.  We
  justify the compactification of an open stratum in the moduli space
  of quilted disks with one interior marking.  Let $\Gamma$ be a type
  of quilted disk with a single interior marking, a boundary output
  and no boundary inputs, and no boundary nodes (fourth figure from
  the left in Figure \ref{fig:quilt-02}).  The stratum
  $\M^q_\Gamma \subset {\M}_{1,0}^q$ is an open interval
  $(0,\infty)$ parametrizing disks with markings
  \begin{equation}
    \label{eq:diskfamily}
  z^\white_{0,\nu}=\infty, \quad z^\black_{0,\nu}=i\nu, \quad \nu \in (0,\infty),    
  \end{equation}
  and quilting circle $\{ \on{Im}(z) = 1 \}$, where the disk is
  identified to $\H \cup \{\infty\}$.  In the limit $\nu \to \infty$
  resp. $\nu \to 0$, we obtain the strata to the left resp. right of
  $\Gamma$ in Figure \ref{fig:quilt-02}. In particular, in the limit
  $\nu \to \infty$, the limit of the markings in \eqref{eq:diskfamily}
  gives the quilted disk in the limit, and the unquilted dark disk is
  given by the reparametrization $z \mapsto \frac z \nu$, so that the
  markings in the sequence are $z^\white_{0,\nu}=\infty$,
  $z^\black_{0,\nu}=i$, and the quilting circle is
  $\{ \on{Im}(z) = \frac 1 \nu \}$.  This gives
  rise to a limit disk which is fully in the dark region.  This finishes the example. 
\end{example}
\begin{figure}[ht]
  \includegraphics[width=4.5in]{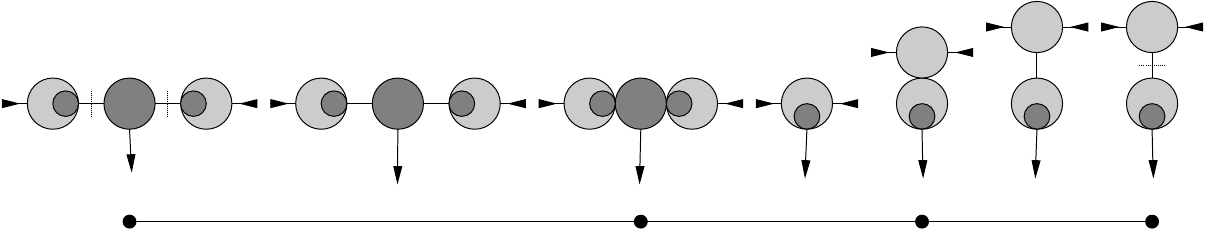}
  \caption{Moduli space ${\M}^\quilt_{0,2}$ of stable quilted treed disks.}
  \label{fig:MWc}
\end{figure}
\begin{figure}[ht]
  \centering\scalebox{.8}{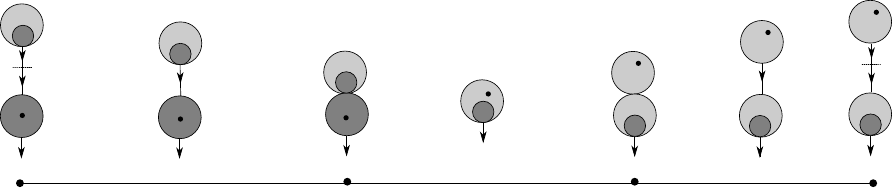}
  \caption{Moduli space ${\M}^\quilt_{1,0}$ of stable quilted treed disks.}
  \label{fig:quilt-02}
\end{figure}

As in the unquilted case, the top-dimensional cells in
${\M}^\quilt_{d(\black),d(\white)}$ consist of strata $\M_\Gamma$
which do not have any interior edges, that is,
$\Edge_{\black,-}(\Gamma)=\emptyset$, and all boundary edges
$e \in \Edge_{\white,-}(\Gamma)$ have finite non-zero length, and the
dimension of these strata is equal to
\[d(\white) + 2d(\black) - 1. \]
The true boundary of ${\M}^\quilt_{d(\black),d(\white)}$ consists
of strata $\M_\Gamma$ where either 
\begin{itemize}
\item $\Gamma$has a single broken edge
  $e \in \Edge_\white(\Gamma)$, or
\item $\Gamma$ has a collection of broken edges $e_1,\ldots, e_k$,
  such that the types $\Gamma_0,\dots,\Gamma_k$ obtained by
  disconnecting $\Gamma$ at the breakings at $e_1,\ldots, e_k$ consist
  of an unquilted disk type $\Gamma_0$, all whose vertices have
  negative distance from the quilting circle in $\Gamma$, and a collection
  $\Gamma_1,\dots,\Gamma_k$ of quilted disk types.
\end{itemize}
For example, the combinatorial types in Figure \ref{fig:qbroken} occur
in the codimension one stratum of $\M^\quilt_{0,4}$,
$\M^\quilt_{0,5}$.

The following is a useful quantity defined on treed quilted disks.
\begin{definition}\label{def:dist-seam}
  {\rm(Distance to the quilting circle)} Let $C=S \cup T$ be a treed quilted
  disk of type $\Gamma$. 
  Given a point $x \in C$
  and a surface component $S_v \subset C$, 
  define the \em{distance} between $x$ and $S_v$ as
  \begin{equation}
    \label{eq:dsv}
    \ol d(x,S_v):= \ell(\gamma \cap T)= \sum_{e \in \Edge_{\white,-}(\Gamma)} \ell(\gamma \cap T_e) \in [0,\infty], 
  \end{equation}
  where $\gamma$ is any non self-intersecting path from $x$ to any point in $S_v$.
  For a point $z \in C$,
  the \em{distance
    to the quilting circle}
is the signed distance to the closest quilted surface component, that is, 
  \begin{equation}
    \label{eq:seamdist}
    d(z) := \pm \ol d(z,S_v) \in    [-\infty,\infty],
  \end{equation}
  where $S_v$, $v \in \Ver^\quilt(\Gamma)$ is the closest quilted
  surface component to $z$ with respect to $\ol d$ in \eqref{eq:dsv},
  and the sign is $+$ resp. $-$ if $z$ is above resp.  Below the
  quilted disk components (that is, further from resp. closer to the
  root than the quilted disk components).  Note that $d$ is constant
  on any surface component of $C$. Therefore, the distance
  $d(v)$ is defined for any vertex $v$ in the graph underlying $C$.
  The distance function $d$ is zero on a quilted disk component and
  the spheres attached to it. To define perturbations, $d$ is composed
  with an increasing diffeomorphism
  \begin{equation}
    \label{eq:deltadef}
    \delta: [-\infty,\infty] \to [0,1]  
  \end{equation}
  which is assumed to be fixed.
\end{definition}

\begin{remark}{\rm(Morphisms for quilted disk types)}
  \label{rem:quiltmorph}
Morphisms of graphs (Cutting an edge, collapsing edges, making edge
lengths finite or non-zero)
from Definition \ref{def:pertops} 
induce morphisms of moduli spaces of
stable quilted treed disks as in the unquilted case. The new feature
is that (Cutting edges) is done
such that one of the pieces is quilted and the other unquilted. This
implies that output edges of quilted disks are cut simultaneously, and
therefore the output has a disconnected type.

\index{Cutting an edge!for a quilted disk type}

For example, in Figure \ref{fig:qbroken}, in the left picture, one can
cut the $e$ at the breaking to obtain an unquilted disk with positive
distance from the quilting circle, and a quilted disk. In the picture to the
right, the edges $e_1$ and $e_2$ get cut simultaneously to yield an
unquilted disk with negative distance from the quilting circle and a disconnected
type consisting of two quilted disks.
\begin{figure}[ht]
  \centering \scalebox{.8}{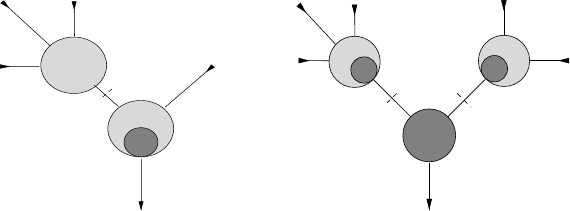}
  \caption{A quilted treed disk with edges of infinite length.}
  \label{fig:qbroken}
\end{figure}

\index{Collapsing an edge!for a quilted disk type} \index{Making an edge length finite/non-zero!for a quilted disk type} In a similar
vein, the morphisms
\hyperref[item:collapsingedgesmorphism]{(Collapsing edges)} and
\hyperref[item:makingedgelengthmorph]{(Making an edge length finite or non-zero)} 
may involve several edges
instead of a single one. For example, in the quilted disk with three
boundary edges of length zero shown in Figure \ref{fig:collapse-q},
\begin{figure}[ht]
  \scalebox{.8}{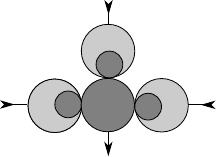}
\caption{A quilted disk with length zero edges. } 
\label{fig:collapse-q}
\end{figure}
there is no \hyperref[item:collapsingedgesmorphism]{(Collapsing
  edges)} or \hyperref[item:makingedgelengthmorph]{(Making an edge length finite or non-zero)}
morphism for
one of the edges alone; these operations can only be performed for all
three edges simultaneously in order to respect the
\hyperref[item:equal-lengths]{(Equal lengths condition)}.

For any combinatorial type $\Gamma$ of quilted disk there is a \em{
  universal quilted treed disk} $\ol{\U}_\Gamma \to \ol{\M}_\Gamma$
which is a cell complex whose fiber over $[C]$ is isomorphic to $C$.
The universal disk splits into surface and tree parts
$ \ol{\U}_\Gamma = \ol{\S}_\Gamma \cup \ol{\T}_{\white,\Gamma} \cup
\ol{\T}_{\black,\Gamma}$, where the last two sets are the boundary and
interior parts of the tree respectively. This ends the Remark.
\end{remark}

Weights can be added to the inputs and output of quilted treed disks
as in the case of treed disks.  We suppose there is a partition of the
boundary markings
\[\Edge^{\greyt}(T) \sqcup \Edge^{\whitet}(T) \sqcup \Edge^{\blackt}(T) =
\Edge_{\white,\rightarrow}(T) \]
into \em{weighted} resp. \em{forgettable} resp.  \em{unforgettable}
edges as in the unquilted case.  The outgoing edge axiom is the same as in
the unquilted case. In the quilted case, the trees in Figure
\ref{trivq} are stable.  Isomorphism of weighted quilted disks is the same
as the unquilted case, and therefore, the moduli space with a single
weighted leaf and no markings is a point.

\begin{figure}[ht]
  \begin{picture}(0,0)%
    \includegraphics{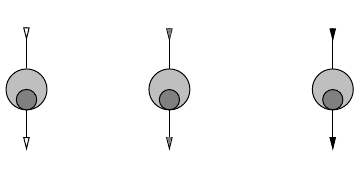}%
  \end{picture}%
  \setlength{\unitlength}{4144sp}%
  \begingroup\makeatletter\ifx\SetFigFont\undefined%
  \gdef\SetFigFont#1#2#3#4#5{%
    \reset@font\fontsize{#1}{#2pt}%
    \fontfamily{#3}\fontseries{#4}\fontshape{#5}%
    \selectfont}%
  \fi\endgroup%
  \begin{picture}(2700,1371)(4216,-2041)
    \put(4231,-1986){\makebox(0,0)[lb]{\smash{{\SetFigFont{8}{9.6}{\rmdefault}{\mddefault}{\updefault}{\color[rgb]{0,0,0}$x^{\whitet}$}%
          }}}}
    \put(5352,-1986){\makebox(0,0)[lb]{\smash{{\SetFigFont{8}{9.6}{\rmdefault}{\mddefault}{\updefault}{\color[rgb]{0,0,0}$x^{\greyt}$}%
          }}}}
    \put(6597,-1986){\makebox(0,0)[lb]{\smash{{\SetFigFont{8}{9.6}{\rmdefault}{\mddefault}{\updefault}{\color[rgb]{0,0,0}$x^{\blackt}$}%
          }}}}
    \put(4259,-785){\makebox(0,0)[lb]{\smash{{\SetFigFont{8}{9.6}{\rmdefault}{\mddefault}{\updefault}{\color[rgb]{0,0,0}$x^{\whitet}$}%
          }}}}
    \put(5417,-799){\makebox(0,0)[lb]{\smash{{\SetFigFont{8}{9.6}{\rmdefault}{\mddefault}{\updefault}{\color[rgb]{0,0,0}$x^{\greyt}$}%
          }}}}
    \put(6662,-799){\makebox(0,0)[lb]{\smash{{\SetFigFont{8}{9.6}{\rmdefault}{\mddefault}{\updefault}{\color[rgb]{0,0,0}$x^{\blackt}$}%
          }}}}
    \put(5515,-1082){\makebox(0,0)[lb]{\smash{{\SetFigFont{8}{9.6}{\rmdefault}{\mddefault}{\updefault}{\color[rgb]{0,0,0}$\rho$}%
          }}}}
    \put(5523,-1653){\makebox(0,0)[lb]{\smash{{\SetFigFont{8}{9.6}{\rmdefault}{\mddefault}{\updefault}{\color[rgb]{0,0,0}$\rho$}%
          }}}}
  \end{picture}%

  \caption{Unmarked stable treed quilted disks.}
  \label{trivq}
\end{figure}

Orientations of the moduli space of quilted treed disks are defined as
follows.  Each main stratum of ${\M}^\quilt_{d(\black), d(\white)}$
can be oriented using the isomorphism of the stratum made of quilted
treed disks having a single disk with $\R$ times
${\M}_{d(\black), d(\white)}$, the extra factor corresponding to
the quilting parameter.  The boundary of the moduli space is naturally
isomorphic to a union of moduli spaces:
\begin{multline} \label{unions} \partial \ol{\M}^\quilt_{d(\black),
    d(\white)} \cong \bigcup_{\substack{i,j \\
      I_1 \cup I_2=[d(\black)]}} \left(
    \ol{\M}^\quilt_{|I_1|,d(\white)-i+1} \times
    \ol{\M}_{|I_2|,i}\right)\\ \cup \bigcup_{\substack{r\geq 1, I_1 \cup \dots I_r= [d(\black)],\\ m_1 + \dots + m_r=d(\white)}} \left( \ol{\M}_{|I_0|,r} \times
    \prod_{j=1}^r \ol{\M}^\quilt_{|I_j|,m_j} \right) .\end{multline}
Here $[d(\black)]:=\{1,\dots,d(\black)\}$.  In the first union, $j$ is
the index of the attaching leaf in the quilted tree and so ranges from
$1$ to $n -i + 1$.  The first union ranges over all partitions of the
set $[d(\black)]$ into $I_1$, $I_2$.  The second union ranges over
partitions of $[d(\black)]$ into $I_0,\dots,I_r$.  By construction,
for the facet of the first type, the sign of the inclusions of
boundary strata are the same as that for the corresponding inclusion
of boundary facets of the moduli space
$\ol{\M}_{d(\black), d(\white)}$ of unquilted disks, that is,
$(-1)^{i(n-i-j) + j} $.  For facets of the second type, the gluing map
is
\[ (0,\infty) \times \M_{r,m_0} \times \prod_{j=1}^r
\M^\quilt_{|I_j|,m_j} \to \M^\quilt_{d(\black), d(\white)}\]
given for boundary markings by
\begin{multline} \label{gluemap2} (\delta,x_1,\ldots, x_r,
  (x_{1,j} = 0, x_{2,j},\ldots, x_{m_j,j} )_{j=1}^r) \mapsto \\
  (x_1, x_1 + \delta^{-1}x_{2,1}, \ldots, x_1 +
  \delta^{-1}x_{m_1,1}, \ldots, x_r, x_r + \delta^{-1}x_{2,r}
  ,\ldots, x_r + \delta^{-1}x_{m_r,r}).
\end{multline}
This map views the markings as lying in the affine half plane
$\H \subset \C$; an interior $x \in I_j$ is mapped to
$x_j+\delta^{-1}x$; and the map is well-defined for $\delta$ that is
large enough to ensure that the ordering of the boundary markings is
preserved.  This map changes orientations by
$\ssum_{j=1}^r (r-j) (m_j - 1);$ in case of non-trivial weightings,
$m_j$ should be replaced by the number of incoming markings
plus non-trivial but finite weightings on the $j$-th component.

\section{Quilted pseudoholomorphic disks}

In this Section and the next, we prove that the Fukaya algebra of a
Lagrangian submanifold is independent of the choice of perturbation
data up to homotopy equivalence.  Given \ainfty algebras defined using
perturbation data $\ul \Pe_0$, $\ul \Pe_1$ whose underlying
stabilizing divisor has the same degree, in this Section, we construct
an \ainfty morphism between them.  The \ainfty morphism is defined by
counts of quilted treed holomorphic disks.  The perturbation system
corresponding to such disks is
 called a \em{perturbation morphism}, and is 
defined by extending the perturbation
systems $\ul \Pe_0$, $\ul \Pe_1$; that is, the perturbation system for
quilted holomorphic disks restricts to $\ul \Pe_0$ and $\ul \Pe_1$ on
light and dark unquilted components respectively.

The following is the main result
which shows \ainfty homotopy equivalence for  \ainfty algebras defined using two different perturbations whose stabilizing divisors have the same degree. We refer the reader to \cite[Section 5.6]{cw:flips} for the case when the stabilizing divisors have different degrees. 
In the statement of the result,
a unital \ainfty morphism is defined in \eqref{eq:unital-morph},
a convergent \ainfty morphism is defined in Remark \ref{rem:conv-morph}, and perturbation morphisms
 are from Definition \ref{def:pert-morph}.

\begin{proposition}
  \label{prop:samedegree}
  Suppose $\ul \Pe^0$, $\ul \Pe^1$ are regular perturbation data that are
  defined using stabilizing pairs $(J^0,D^0)$ and $(J^1,D^1)$, which
  are connected by a path of stabilizing pairs $\{(J^t,D^t)\}_{t \in
    [0,1]}$. There exists a coherent perturbation
  morphism 
  %
  $\ul \Pe^{01}$
  which induces a convergent unital \ainfty morphism 
  \[\phi : CF(L,\ul \Pe^0) \to CF(L,\ul{\Pe}^1), \]
  which is a convergent unital \ainfty homotopy equivalence.
\end{proposition}

 Proposition \ref{prop:samedegree} is used in Section
 \ref{sec:homo-u2b} to show that the Fukaya algebra of a Lagrangian in
 a neck-stretched manifold is independent of the neck length parameter
 up to homotopy equivalence.  The \ainfty morphism $\phi$
 required by Proposition \ref{prop:samedegree} is
 constructed later in this Section after defining quilted holomorphic
 disks. The fact that $\phi$ is a homotopy equivalence is proved
 in Section \ref{sec:homotopies} as part of Corollary \ref{cor:hequiv}.

\begin{remark}
  \label{rem:conv-morph}
  {\rm(Convergent \ainfty morphisms)} An \ainfty morphism
  $\F: A_0 \to A_1$ between \ainfty algebras with Novikov coefficients
  is said to be \em{convergent} \index{Convergent!\ainfty morphism} if
  $\F^0$ has positive $q$-valuation.  A convergent \ainfty morphism
  $\F: A_0 \to A_1$, induces a well-defined map
  $MC(\F) : MC(A_0) \to MC(A_1)$ on the space of Maurer-Cartan
  solutions and a map on cohomology
  $H(\F) : H(A_0,b) \to H(A_1,\F(b))$ for any $b \in MC(A_0)$,
  see \cite[Lemma 5.2, 5.3]{cw:flips}.  In Proposition
  \ref{prop:samedegree}, the \ainfty morphism $\phi$ being convergent
  means that $\phi^0(1) \in \Lam_{>0}\bran{\hat \cI(L)}$.
\end{remark}

\begin{definition}\label{def:pert-morph}
  Given perturbation data
  \[\ul{\Pe}^0=(J^0_\Gamma, F^0_\Gamma)_\Gamma, \quad  \ul{\Pe}^1=(J^1_\Gamma, F^1_\Gamma)_\Gamma\]
  on unquilted
  treed disks with respect to stabilizing divisors $D^0$ and $D^1$
  that have the same degree and a path $\{D^t\}_{t \in [0,1]}$ of Donaldson divisors,
  a \em{perturbation morphism}
  $\ul{\Pe}^{01}$ from $\ul{\Pe}^0$ to $\ul{\Pe}^1$ for the  type $\Gamma$ of a quilted treed disk consists of
  \begin{enumerate}
  \item \label{part:pert-morph2} a domain-dependent Morse function
    \[ F_\Gamma^{01}: \ol{\T}_{\white,\Gamma} \to \R, \]
    which for $k=0,1$, is equal to the domain-independent Morse
    function $F^k$ on the neighbourhood
    $\ol{{\T}}_{\white,\Gamma} - \ol{{\T}}^{\on{cp}}_{\white,\Gamma}$
    of the endpoints for which the distance to the quilting circle
    $\delta \circ d$ is $k$, where $F^k:L \to \R$ is the background
    Morse function for $F^k_\Gamma$ for $k=0,1$;
  \item \label{part:pert-morph3} and a domain-dependent almost complex structure
    \[ J_\Gamma^{01}: \ol{\S}_\Gamma \to \bigsqcup_{t \in [0,1]}\{J \in \J(X,D^t): J \text{ is $\om$-tamed}\}, \]
    with the property that
    \begin{itemize}
    \item on the curve $\S_{\Gamma,m} \subset \ol \S_\Gamma$
      associated to any point $m \in \M_\Gamma$, $J_\Gamma^{01}$ is
      locally constant on $\S_{\Gamma,m} - \S^{\on{cp}}_{\Gamma,m}$,
      where the compact set $\S^{\on{cp}}$ is as defined in Section
      \ref{sec:domdep}; and on a component
      $\S_{v,m} \subset \S_{\Gamma,m}$ corresponding to
      $v \in \Ver(\Gamma)$, $J_\Gamma^{01}$ is adapted to the divisor
      $D^{\delta \circ d(v)}$ where $d$ is the distance from the quilting circle
      function \eqref{eq:seamdist}, that is,
      $J_\Gamma^{01}(\S_{v,m}) \subset \J(X,D^{\delta \circ d(v)})$.
    \item Furthermore,  for $k=0,1$, denoting by $\Gamma_k \subset \Gamma$
      the sub-tree where the distance from the quilting circle $\delta \circ d$
      is $k$, $J_\Gamma^{01}$ is equal to the domain-dependent complex
      structure $J_{\Gamma_k}^k$ on the (unquilted) treed disk
      component of type $\Gamma_k$.
    \end{itemize}
  \end{enumerate}
  A collection of perturbation morphisms $\ul \Pe=(\Pe_\Gamma)_\Gamma$
  defined on moduli spaces of quilted disks is \em{coherent} if it is
  compatible with the (Cutting edges), (Making an edge length/weight
  finite or non-zero) and {(Forgetting
    edges)} morphisms on weighted quilted disk types as in Remark \ref{rem:quiltmorph},
  and satisfies the (Locality axiom) from Definition \ref{def:coherent}.
\end{definition}
 %

Next, we define the quilted version of perturbed holomorphic treed disks (Definition
\ref{def:pdisks}).
\begin{definition}\label{def:hol-quilt}
  {\rm(Holomorphic quilted treed disk)}
\index{Holomorphic disks! Holomorphic quilted treed disks}
  Let
  $\ul \Pe=(\Pe_\Gamma)_\Gamma$ be a coherent perturbation morphism.
  A holomorphic quilted treed disk is a map $u : C \to X$ where $C$ is
  a quilted disk of type $\Gamma$, and $u$ is $\Pe_\Gamma$-holomorphic
  in the sense of ordinary perturbed holomorphic disks (see Definition
  \ref{def:pdisks}), and \em{adapted} to a family of divisors
  $\{D^t\}_{t \in [0,1]}$ in the sense that
  \begin{itemize}
  \item each 
  interior marking $z_i$ maps to
    $D^{\delta \circ d(z_i)}$, where $d$ is the distance from the quilting circle 
    \eqref{eq:seamdist} and $\delta$ is from
    \eqref{eq:deltadef},
  \item and for $t \in [0,1]$, each component of
    $u^{-1}(D^t) \cap (\delta \circ d)^{-1}(t)$ contains a marking.
  \end{itemize}
  If the domain quilted treed disk $C$ is unstable, we obtain a stable
  quilted treed disk $C'$ by collapsing unstable surface and tree
  components, and we denote the collapsing map by $f:C \to C'$.
  Pulling back by $f$ we obtain a datum on $C$, still denoted
  $(J^{01}_\Gamma, F^{01}_\Gamma)$.
  A quilted holomorphic treed disk $u: C \to X$ is \em{stable} if
  every (surface or tree) component of $C$ on which $u$ is
  non-constant is stable in the sense of weighted quilted disks. The
  \em{type} of a quilted holomorphic disk consists of the
  combinatorial type of the domain quilted treed disk, and the
  tangency and homology data of the map components as in the type of a
  treed holomorphic disk, see Definition \ref{def:type-broken}.
  \index{Combinatorial type! of a quilted holomorphic disk}
\end{definition}
\begin{remark}{\rm(Symplectic area and the number of markings)}
  For a quilted treed disk $C=S \cup T$, on any surface component
  $S_v \subset S$, the function $\delta \circ d|S_v$ is constant and
  therefore a quilted holomorphic disk $u:C \to X$ is adapted to the
  divisor $D^{\delta \circ d(S_v)}$ on the component $S_v$.
  Furthermore, the divisor $D^{\delta \circ d(S_v)}$ is holomorphic
  with respect to the background almost complex structure
  $J^{\delta \circ d(S_v)}$, and therefore all intersections between
  the map $u|S_v$ and the divisor $D^{\delta \circ d(S_v)}$ are
  positive. The intersection number of $u$ with $D^{\delta}$ is
  $k\om[u]$, where the divisor $D_0$ (and $D_1$) is Poincar\'{e} dual
  to $k\om$.
\end{remark}

\begin{remark}\label{rem:jtpert}
  {\rm(Quilted holomorphic disks for a path of perturbations)}
  \begin{enumerate}
  \item Often a perturbation morphism is constructed using a generic
    path of perturbations $\Pe^t_\Gamma=(J^t_\Gamma,F_\Gamma)$,
    $t \in [0,1]$, whose Morse datum $F_\Gamma$ is $t$-independent.
    Let $\tGam$ be a quilted disk type, for which forgetting the
    quilting yields the disk type $\Gamma$.  One may define a
    perturbation morphism $\Pe_\tGam^{01}$ connecting $\Pe^0_\Gamma$
    and $\Pe^1_\Gamma$ by setting the domain-dependent data to be
  \begin{equation}
    \label{eq:path-pert}
  J_\tGam^{01}(z) := J^{\delta \circ
    d(z)}_{\Gamma}(z), \quad F_\tGam^{01}(z) := F_{\Gamma}(z)    
  \end{equation}
  for any $z \in \S_\tGam$ where $d$ is the ``distance from the quilting circle'
  function from \eqref{eq:seamdist}, and
  $\delta: [-\infty,\infty] \to [0,1]$ is a fixed increasing
  diffeomorphism from \eqref{eq:deltadef}.  Thus on any surface
  component of a $J_\tGam^{01}$-holomorphic quilted disk, the map is
  $J^t_\Gamma$-holomorphic for some $t \in [0,1]$.
\item {\rm(A path of regular perturbations)}\label{part:jtpert2} In
  some special cases, a one-dimensional component of quilted disks is
  made up of a family of unquilted $\Pe_t$-holomorphic disks for
  $t \in [0,1]$. Such a special case arises when the perturbation
  $\Pe^t_\Gamma=(J^t_\Gamma,F_\Gamma)$ is regular for a disk homology
  class $\beta \in H_2(X,L)$ and input/output tuple
  $\ul x=(x_0,\dots,x_d)$ for all $t$. Then the zero dimensional
  component of the moduli space of perturbed holomorphic maps
  $\M_\beta(\ul{\Pe}^t,\ul x)_0$ is in bijection with
  $\M_\beta(\ul{\Pe}^0,\ul x)_0$ for any $t \in [0,1]$.  Furthermore,
  for a perturbation morphism as in \eqref{eq:path-pert}, a path of
  moduli spaces
  $\cup_{t \in [0,1]}\M_\Gamma(\ul{\Pe}^t,\ul x)^{\leq E_0}_0$ of
  $\ul{\Pe}^t$-holomorphic disks gives a one-dimensional component
  $(u_t)_{t \in [0,1]}$ of quilted $\ul{\Pe}^{01}$-holomorphic disks
  where
  \begin{itemize}
  \item if $\delta^{-1}(t)>0$ (where
    $\delta:[-\infty,\infty] \to [0,1]$ is a fixed increasing
    diffeomorphism), $u_t$ is of the form in Figure \ref{fig:hol-q}
    (a), $u_t$ is constant on quilted components, $J_t$-holomorphic on
    the dark shaded disk, and boundary edges have length
    $\tau:=\delta^{-1}(t)$;
  \item if $\delta^{-1}(t)<0$, $u_t$ is of the form in Figure
    \ref{fig:hol-q} (c), $u_t$ is constant on quilted components and
    $J_t$-holomorphic on the light shaded disk, and boundary edges
    have length $\tau:=-\delta^{-1}(t)$;
  \item if $\delta^{-1}(t)=0$, $u_t$ is of the form in Figure
    \ref{fig:hol-q} (b), on the quilted components $u_t$ is
    $J_t$-holomorphic.
  \end{itemize}

  \begin{figure}[ht]
  \centering \scalebox{.8}{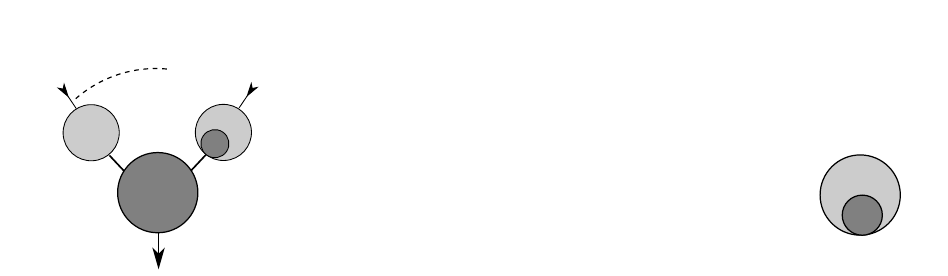}
  \caption{Types of quilted holomorphic disks occurring in the one-dimensional moduli space described in Remark \ref{rem:jtpert} \eqref{part:jtpert2}. 
    In (a) and (c) the map is constant on the quilted component.}
  \label{fig:hol-q}
\end{figure}

The analysis of the special case shows that in an idealized setting
moduli spaces of quilted holomorphic disks interpolate between moduli
spaces of unquilted holomorphic disks.  But in general cases, one has
to account for disk bubbling in the moduli spaces, leading to a
variety of configurations occurring in the codimension one boundary of
moduli spaces of quilted holomorphic disks, see Remark
\ref{rem:qtypes}.
\end{enumerate}
\end{remark}

For any combinatorial type $\Gamma$ of quilted holomorphic disks we
denote by $\ol{\M}^\quilt_\Gamma(X, L)$ the compactified moduli
space of equivalence classes of  treed quilted holomorphic 
disks.  The moduli space of quilted disks breaks into components
depending on the limits along the root and leaf edges.  Denote by
\index{Moduli space!of quilted holomorphic disks
  $\M_\Gamma^\quilt(X,L)$}
$\M_\Gamma^\quilt(X,L,\ul{x}) \subset {\M}_{\Gamma}^\quilt(X,L)$
the moduli space of isomorphism classes of stable 
holomorphic
quilted treed disks of type $\Gamma$ with boundary in $L$ and limits
$\ul{x}$ along the root and leaf edges, where
$\ul{x} = (x_0,\ldots, x_{d(\white)}) \in \widehat{\cI}(L)$.
A quilted holomorphic disk type $\Gamma$ is \em{rigid} if
it satisfies all the conditions for a holomorphic map type to be rigid, as in Definition
\ref{def:unbrokenrigidtype}.  \index{Rigid! quilted map}
For a rigid quilted holomorphic disk type $\Gamma$, the expected dimension of the moduli space
$\M_\Gamma^\quilt(X,L,\ul{x})$ is \index{Index! of a quilted map $i^\quilt$}
\begin{equation}
  \label{eq:expdim-q}
  i^\quilt(\Gamma, \ul x)=i(\Gamma', \ul x) + 1,  
\end{equation}
where $i$ is the index for types of treed holomorphic disks (see
\eqref{eq:expdim}), and $\Gamma'$ is the unquilted type obtained by
forgetting the quilting in $\Gamma$. The extra dimension for the
moduli space of quilted disks arises from the quilting datum. In case
$n$ quilted disk components, there are $n-1$ equations in the
\hyperref[item:equal-lengths]{(Equal lengths condition)}, leading to an extra
dimension of $1$ compared to the unquilted case.

For a comeager subset of perturbation morphisms extending those chosen
for unquilted disks, the uncrowded moduli spaces of expected dimension
at most one are smooth and of expected dimension.  For sequential
compactness, it suffices to consider a sequence $u_\nu: C_\nu \to X$
of quilted treed disks of fixed combinatorial type $\Gamma=\Gamma_\nu$
for all $\nu$.  Coherence of the perturbation morphism implies the
existence of a stable limit $u: C \to X$ which we claim is adapted.
If a sequence of markings $z_{i,\nu} \in C_\nu$ converges to
$z_i \in C$, then, $u(z_i) \in D^{\delta \circ d(z_i)}$. Indeed, since
the distance from the quilting circle $d(z_{i,\nu})$ converges to $d(z_i)$, the
divisor $D^{\delta \circ d(z_i)}$ is the limit of the divisors
$D^{\delta \circ d(z_{i,\nu})}$.  For types of index at most one, each
component of $u^{-1}(D^{\delta \circ d(z_i)})$ is a limit of a unique
component of $u^{-1}_\nu( D^{\delta \circ d(z_{i,\nu})})$, otherwise
the intersection degree would be more than one which is a codimension
two condition. Therefore, every marking in $C$ is a transverse divisor
intersection. There are no other divisor intersections because the
intersection number with $D^{\delta \circ d}$ is preserved in the
limit for topological reasons.  The moduli space of quilted broken disks then has the same
transversality and compactness property as in the unquilted case, by
similar arguments.

Counts of quilted disks define \ainfty morphisms. 
Given a regular, stabilized and coherent perturbation morphism
$\ul{\Pe}^{01}$ from $\ul{\Pe}^0$ to $\ul{\Pe}^1$, define an \ainfty
morphism
$\phi=(\phi^d)_{d \geq 0} : CF(L,\ul{\Pe}^0) \to CF(L,\ul{\Pe}^1)$ as
\begin{multline}\label{eq:phid-def}
  \phi^{d}: CF(L;\ul{\Pe}^0)^{\otimes d} \to CF(L;
  \ul{\Pe}^1) \\ (x_1,\ldots,x_{d}) \mapsto \sum_{x_0,u \in
    \M_\Gamma(L,D,x_0,\ldots,x_{d})_0} (-1)^{\heartsuit} w(u)
 x_0
\end{multline}
where the weight $w(u)$ is given by
\begin{equation}\label{eq:wu}
 w(u) = 
\eps([u]) 
  ( d_\black(\Gamma)!)^{-1} q^{ E([u])} \Hol_L([\partial u]) 
 x_0 \end{equation}
the sum is over uncrowded strata $\Gamma$ of weighted treed quilted
holomorphic disks whose boundary edges
$e \in \Edge_{\white,-}(\Gamma)$ have finite non-zero length
$\ell(e) \in (0,\infty)$ and whose input and output labels are
compatible with $(x_0,\dots,x_{d})$ in terms of the (Label axiom) in
Definition \ref{def:wt-treed-hol}, and $\eps([u]) =\pm 1$ is the
orientation sign.

\begin{remark} {\rm (Lowest area terms)} For any
  $x \in \crit(F^0) \cup \{x^{\greyt}, x^{\whitet}\}$, the element
  $\phi_1(x)$ contains zero area terms coming from the count of a
  quilted treed disk with no interior marking, that is, a treed disk
  with only one disk that is quilted and mapped to a point. The domain
  is one of those in Figure \ref{trivq}.  If $x$ is $x^{\greyt}$
  resp. $x^{\whitet}$ there is one such configuration whose output is
  weighted resp. forgettable.  In the latter case, it will be the only
  term with a forgettable output.
\end{remark}

\begin{remark} \label{rem:qtypes} {\rm(Codimension one boundary
    strata)} The codimension one strata are of several possible types:
  either there is one (or a collection of) edge(s) $e$ of length
  $\ell(e)$ infinity, there is one (or a collection of) edge(s) $e$ of
  length $\ell(e)$ zero, or equivalently, boundary nodes, or there is
  an edge $e$ with zero or infinite weight $\rho(e)$.  The case of an
  edge of zero or infinite weighting is equivalent to breaking off a
  constant trajectory, and so may be ignored.  In the case of edges of
  infinite length(s), then either $\Gamma$ is
  \begin{enumerate}
  \item {\rm (Breaking off an unquilted tree)} a pair
    $\Gamma_1 \sqcup \Gamma_2$ consisting of a quilted tree $\Gamma_1$
    and an unquilted tree $\Gamma_2$ as in the left side of Figure
    \ref{fig:qbroken}; necessarily the breaking must be a leaf of
    $\Gamma_1$; or
  \item {\rm (Breaking off quilted trees)} a collection consisting of
    an unquilted tree $\Gamma_0$ containing the root and a collection
    $\Gamma_1,\ldots,\Gamma_r$ of quilted trees attached to each of
    its $r$ leaves as in the right side of Figure \ref{fig:qbroken}.
    Such a stratum $\M_\Gamma$ is codimension one because of the
     \hyperref[item:equal-lengths]{(Equal lengths condition)} 
     which implies
    that if the length of any edge between $e_0$ to $e_i$ is infinite
    for some $i$ then the path from $e_0$ to $e_i$ for any $i$ has the
    same property.
  \end{enumerate}

  In the case of a zero length(s), one obtains a fake boundary
  component with normal bundle $\R$, corresponding to either deforming
  the edge(s) to have non-zero length or deforming the node(s). This
  ends the Remark.
\end{remark}

We can now prove the first part of Proposition \ref{prop:samedegree}.
\begin{proof}
  [Proof of Proposition \ref{prop:samedegree}]
  We have so far chosen a regular coherent perturbation morphism $\ul \Pe^{01}$, and defined
  a collection $\phi=(\phi^d)_{d \geq 0}$
  of maps 
  in \eqref{eq:phid-def}, which is a convergent unital \ainfty morphism for the following reason: 
The true boundary strata of one-dimensional moduli spaces of quilted
holomorphic disks are those described in Remark \ref{rem:qtypes} and
correspond to the terms in the axiom for \ainfty morphisms
\eqref{faxiom}.
The signs are similar to those in \cite{cw:flips} and
omitted.  The assertion on the strict units is a consequence of the
existence of forgetful maps for infinite values of the weights.  By
assumption the terms involving $\phi_{d(\white)}$ \label{rep:phid} and
$x^{\whitet}$ as inputs involve counts of quilted treed disks using
perturbation that are independent of the disk boundary incidence
points of the first leaf marked $x^{\whitet}$ asymptotic to
$x_M \in X$.  Since forgetting that semi-infinite edge gives a
configuration of negative expected dimension, if non-constant, the
only configurations contributing to these terms must be the constant
maps.  Hence
 \[ \phi^1(x^{\whitet}) = x^{\whitet},\quad \phi^{d(\white)}( \ldots,
 x^{\whitet}, \dots) = 0, n \ge 2 .\]
In other words, the only regular quilted trajectories from the
maximum, considered as $x^{\whitet}$, being regular are the ones
reaching the other maximum that do not have interior markings
(i.e. non-constant disks).  The convergence property follows from the
fact that any treed quilted disk with no incoming edges must, by
stability, have a disk component with interior markings. The \ainfty
morphism $\phi$ is an \ainfty homotopy equivalence by Corollary
\ref{cor:hequiv}.
\end{proof}

\section{Homotopies}
\label{sec:homotopies}


In this Section, we define homotopies between \ainfty morphisms via
counts of twice-quilted disks.  The main result is that the \ainfty
morphism between Fukaya algebras constructed in the previous Section
is a homotopy equivalence.

 A \em{twice-quilted disk} $(C,Q_1,Q_2)$ is defined
in the same way as once-quilted disks, but with two quilting circles
$Q_1,Q_2 \subset C$ that are either equal $Q_1 = Q_2$ or with the
second $Q_2 \subset \on{int}(Q_1)$ contained inside the first, say
with radii $\rho_1 < \rho_2$ satisfying certain balancing conditions. 
\index{Disk! Twice-quilted disk}
The combinatorial type $\Gamma$ of a nodal 
quilted disk comes with the data of subsets of the boundary vertex set
\[\Ver^\quilt_1(\Gamma), \Ver^\quilt_2(\Gamma) \subset \Ver_\white(\Gamma), \quad \Ver^\qq_{12}(\Gamma)=\Ver^\quilt_1(\Gamma) \cap \Ver^\quilt_2(\Gamma),\]
where $\Ver^\quilt_k(\Gamma)$ corresponds to disk components which contain the quilting circle $Q_i$. 
 \index{Combinatorial type! of a twice-quilted treed disk} 
The vertices in $\Ver^\qq_{12}(\Gamma)$ are \em{twice-quilted vertices}. 
\index{Vertex! Twice-quilted vertex}
See Figure \ref{fig:vert-q2}. 
For a twice-quilted vertex $v \in \Ver^\qq_{12}(\Gamma)$, we denote by
\begin{equation}
  \label{eq:lamSdef}
  \lam_S(v):=\rho_2(v)/\rho_1(v)  
\end{equation}
the ratio 
of the radii of the quilting circles.

\begin{figure}[ht]
  \centering \scalebox{.8}{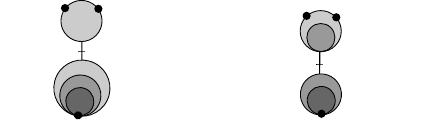}
  \caption{For $k=1,2$, components corresponding to vertices in $\Ver_k^\quilt \subset \Ver_\white$ contain the quilting circle $Q_i$. Components corresponding to vertices
  in $\Ver_{12}^\qq \subset \Ver_\white$ are twice-quilted.}
  \label{fig:vert-q2}
\end{figure}

Treed twice-quilted disks are defined by replacing disk nodes
by treed segments, whose lengths satisfy two equal lengths conditions:
The lengths of edges occurring in a path $P$ from the root vertex to
a quilted disk vertex $v \in \Ver_2^\quilt(\Gamma)$ containing the inner quilting circle $Q_2$ is independent of $v$:
 For any two vertices $v_1,v_2 \in \Ver^\quilt_2(\Gamma)$ connected to the root vertex $v_0$ by non self-crossing paths $P(v_0,v_1)$, $P(v_0,v_2)$ in $\Gamma$,  
 \begin{equation} \label{eq:balanced2}
   {\text{\rm\textup{(Equal lengths to inner circle)}}} \quad 
\sum_{e \in P(v_0,v_1)} \ell(e) = 
\sum_{e \in  P(v_0,v_2)} \ell(e).
 \end{equation}
 Secondly, for a fixed disk vertex $v_2 \in \Ver_2^\quilt(\Gamma)$ containing an inner quilting circle,
 and any two vertices $v_1, v_1' \in \Ver_1^\quilt(\Gamma)$ containing an outer quilting circle and for which the paths $P(v_1,v_2)$, $P(v_1',v_2)$ do not containing the root vertex (unless $v_2$ is the root vertex),
 \begin{equation} \label{eq:balanced2b}
   {\text{\rm\textup{(Equal lengths from inner to outer circle)}}} \quad 
\sum_{e \in P(v_2,v_1)} \ell(e) = 
\sum_{e \in  P(v_2,v_1')} \ell(e).
 \end{equation}
Note that if $v_2$ is a twice-quilted vertex, then \eqref{eq:balanced2b} is automatically satisfied since all path lengths occurring in \eqref{eq:balanced2b} are zero. 
 A twice-quilted disk is \em{stable} if it has at least two special points.

 The treed disks
 with two quilting circles defined so far are called \em{unbalanced twice-quilted disks}. The moduli space of unbalanced treed twice quilted disks of type $\Gamma$ is denoted by $M^{\qq,\univ}_\Gamma$.
 An unbalanced twice-quilted disk $C \in \M^{\qq,\univ}_\Gamma$ 
 is a \em{twice-quilted disk} 
 \index{Disk! Twice-quilted disk}
 if it satisfies the \em{balanced ratio of radii} condition, which is as follows:
 \begin{itemize}
 \item Either all quilted disks are twice-quilted, that is $\Ver_1^\quilt(\Gamma)=\Ver_2^\quilt(\Gamma)=\Ver^\qq_{12}(\Gamma)$, and the 
    ratio \eqref{eq:lamSdef} of radii 
is required to be the same for all twice-quilted vertices
$v \in \Ver^\qq_{12}(\Gamma)$, that is,
\begin{equation}
  \label{eq:lamS}
   {\text{\rm\textup{(Balanced ratios, surfaces)}}} \quad
 \lam_S(v)=\lam_S(v') \quad \forall v, v' \in \Ver^\qq_{12}(\Gamma); 
\end{equation}
or
\item there are no twice-quilted disks and the following treed version of the balanced ratio condition holds:  For a disk vertex 
  $v_2 \in \Ver^2(\Gamma)$ containing the inner circle, and a vertex
  $v_1 \in \Ver^1(\Gamma)$ containing the outer quilting circle, and
  whose path to the root vertex contains $v_2$,  denoting
 \begin{equation}
   \label{eq:lamTdef}
   \lam_T(v_2):=\sum_{e \in P(v_1,v_2)}\ell(e),  
 \end{equation}
 where $P(v_1,v_2)$ is the path in $\Gamma$ from $v_1$ to $v_2$, 
 we require that
 \begin{equation}
   \label{eq:lamT}
   {\text{\rm\textup{(Balanced ratios, trees)}}}  \quad \textup{$\lam_T$ is independent of $v_2$.}
 \end{equation}
 \end{itemize}
 Note that by \eqref{eq:balanced2b}, the quantity $\lam_T(v_2)$ does
 not depend on the choice of $v_1$.  The balanced ratio conditions
 \eqref{eq:lamS}, \eqref{eq:lamT} on surfaces and trees can be stated
 in a unified way as follows.  First, we fix an identification
\[\phi: ((\{1 \leq \lam_S \leq \infty\} \sqcup \{0 \leq \lam_T \leq \infty\})/(\lam_S=\infty) \sim (\lam_T=0)) \xrightarrow{\simeq} [0,\infty]. \]
Define a \em{generalized ratio of radii} as
\[\lam_\Gamma : \M^{\qq, \univ}_\Gamma \to [0,\infty]^{\Ver^\quilt_2(\Gamma)}, \quad \lam_\Gamma(C)=(\lam_\Gamma(C)_v)_{v \in \Ver^\quilt_2(\Gamma)}, \]
where, for any vertex $v \in \Ver^\quilt_2(\Gamma)$ containing an
inner quilting circle,  
\[\lam_\Gamma(C)_v:=\phi \left(
    \begin{rcases}
      \begin{dcases}
      \lam_S(v), \quad \text{$v$ is twice-quilted},\\
      \lam_T(v). \quad \text{$v$ is not twice-quilted}
      \end{dcases}
    \end{rcases}
\right),\]
 An unbalanced twice-quilted curve $C$ of type $\Gamma$ then satisfies the balanced ratio condition if  $\lam_\Gamma(C)$ maps to the diagonal $\Delta$, that is, $\lam_\Gamma(C)_v$ is the same for all $v \in \Ver^\quilt_1(\Gamma)$. The moduli space of twice quilted disks of type $\Gamma$ is denoted by
\[\M_\Gamma^\qq:=\lam_\Gamma^{-1}(\Delta). \]

The moduli space of twice-quilted treed
disks is a cell complex constructed in a similar way to the space of
once-quilted treed disks. 
Denote the moduli space
of twice-quilted disks 
with  $d(\black)$ interior leaves and $d(\white)$ non-root boundary leaves by $\M^\qq_{d(\black),d(\white)}$.  
The moduli space admits a
stratification by combinatorial type, that is,
\[\M^\qq_{d(\black),d(\white)}=\cup_\Gamma \M^\qq_\Gamma,\]
where $\Gamma$ ranges over all twice-quilted disk types with
 $d(\black)$ interior leaves and $d(\white)$ non-root boundary leaves. The open strata in
$\M^\qq_{d(\black),d(\white)}$ have dimension 
\(d(\white)+2d(\black). \)
A twice-quilted holomorphic disk type $\Gamma$ is \em{rigid}
 if it satisfies all the conditions for a holomorphic map type to be rigid, as in Definition
\ref{def:unbrokenrigidtype}. \index{Rigid! twice-quilted map}

For example, the moduli space $\M^\qq_{0,1}$ of twice-quilted disks
with a single non-root boundary marking is one-dimensional. It has two
open strata interspersed by three strata of codimension one, whose
types are shown in Figure \ref{fig:oneleaf}. The generalized ratio $\lam_\Gamma:\M^\qq_{0,1} \to [0,\infty]$ is a diffeomorphism.  The moduli space $\M^\qq_{0,2}$ of twice-quilted disks
with two non-root boundary markings is two-dimensional, and Figure \ref{fig:pent-2q} shows some level sets of $\lam_\Gamma$ in $\M^\qq_{0,2}$. 
\begin{figure}[ht]
\begin{picture}(0,0)%
\includegraphics{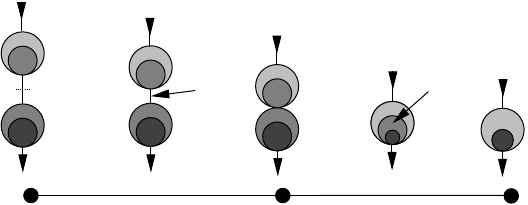}%
\end{picture}%
\setlength{\unitlength}{4144sp}%
\begingroup\makeatletter\ifx\SetFigFont\undefined%
\gdef\SetFigFont#1#2#3#4#5{%
  \reset@font\fontsize{#1}{#2pt}%
  \fontfamily{#3}\fontseries{#4}\fontshape{#5}%
  \selectfont}%
\fi\endgroup%
\begin{picture}(4001,1552)(-16950,-2025)
\put(-15423,-1220){\rotatebox{360.0}{\makebox(0,0)[lb]{\smash{{{$\lambda_T$}%
}}}}}
\put(-13644,-1173){\rotatebox{360.0}{\makebox(0,0)[lb]{\smash{{{$\lambda_S$}%
}}}}}
\end{picture}%
\caption{Moduli space $\M^\qq_{0,1}$ of treed twice-quilted disks with one leaf. 
}
\label{fig:oneleaf}
\end{figure} 

\begin{figure}[h]
  \centering \scalebox{.8}{
\begingroup%
  \makeatletter%
  \providecommand\color[2][]{%
    \errmessage{(Inkscape) Color is used for the text in Inkscape, but the package 'color.sty' is not loaded}%
    \renewcommand\color[2][]{}%
  }%
  \providecommand\transparent[1]{%
    \errmessage{(Inkscape) Transparency is used (non-zero) for the text in Inkscape, but the package 'transparent.sty' is not loaded}%
    \renewcommand\transparent[1]{}%
  }%
  \providecommand\rotatebox[2]{#2}%
  \newcommand*\fsize{\dimexpr\f@size pt\relax}%
  \newcommand*\lineheight[1]{\fontsize{\fsize}{#1\fsize}\selectfont}%
  \ifx\svgwidth\undefined%
    \setlength{\unitlength}{267.13313065bp}%
    \ifx\svgscale\undefined%
      \relax%
    \else%
      \setlength{\unitlength}{\unitlength * \real{\svgscale}}%
    \fi%
  \else%
    \setlength{\unitlength}{\svgwidth}%
  \fi%
  \global\let\svgwidth\undefined%
  \global\let\svgscale\undefined%
  \makeatother%
  \begin{picture}(1,0.72495488)%
    \lineheight{1}%
    \setlength\tabcolsep{0pt}%
    \put(0,0){\includegraphics[width=\unitlength,page=1]{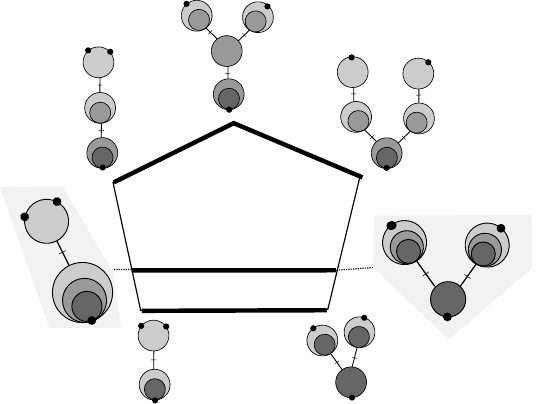}}%
    \put(0.37092193,0.17444363){\makebox(0,0)[lt]{\lineheight{1.25}\smash{\begin{tabular}[t]{l}$\lam=0$\end{tabular}}}}%
    \put(0.05901417,0.16862402){\makebox(0,0)[lt]{\lineheight{1.25}\smash{\begin{tabular}[t]{l}$\lam_S$\end{tabular}}}}%
    \put(0.73561617,0.33599565){\makebox(0,0)[lt]{\lineheight{1.25}\smash{\begin{tabular}[t]{l}$\lam_S$\end{tabular}}}}%
    \put(0.45597335,0.49545854){\makebox(0,0)[lt]{\lineheight{1.25}\smash{\begin{tabular}[t]{l}$\lam=\infty$\end{tabular}}}}%
    \put(0.3635523,0.25636317){\makebox(0,0)[lt]{\lineheight{1.25}\smash{\begin{tabular}[t]{l}$\lam=\lam_S$\end{tabular}}}}%
    \put(0.91120775,0.25573756){\makebox(0,0)[lt]{\lineheight{1.25}\smash{\begin{tabular}[t]{l}$\lam_S$\end{tabular}}}}%
  \end{picture}%
\endgroup%
}
  \caption{Moduli space $\M_{0,2}^\qq$ of twice-quilted disks with two boundary inputs. The figure shows the configurations at the codimension two corners, and also the end-points of the moduli space $\M_{0,2}^{\qq,\lam_S}$ with fixed ratio of radii.}
  \label{fig:pent-2q}
\end{figure}

Coherent domain-dependent perturbations for twice-quilted disks are defined in
a similar way to quilted disks, see \cite[Definition
5.9]{cw:flips}. That is, given coherent perturbations $\ul \Pe_0$,
$\ul \Pe_1$, $\ul \Pe_2$ on $X$, and coherent perturbation morphisms
$\ul \Pe_{01}$, $\ul \Pe_{12}$, perturbation data $\ul \Pe_{012}$ for
twice quilted disks are required to interpolate between these.  In
particular, if in a stratum, (Cutting edges) produces an unquilted disk
type $\Gamma$ of shading $k$ resp. a once-quilted disk type $\Gamma$
of shading $k$ and $k+1$, then the domain-dependent perturbation on
these strata is defined by the pullback of $(\Pe_k)_\Gamma$
resp. $(\Pe_{k,k+1})_\Gamma$.

The following is the main result of the Section. 

\begin{theorem} \label{thm:twicequilted} {\rm (\ainfty homotopies via
    twice-quilted disks)} Given coherent perturbations 
  $\ul{\Pe}_{01},\ul{\Pe}_{12},\ul{\Pe}_{02}$ defining morphisms
\[\phi_{ij}:
CF(L,\ul{\Pe}_i) \to CF(L,\ul{\Pe}_j), \quad 0 \leq i
< j \leq 2 \] 
and coherent perturbation data $\ul \Pe_{012}$ for
twice-quilted disks, there is a convergent \ainfty homotopy
\[\cT^{02} \in \Hom(\phi_{02}, \phi_{01} \circ \phi_{12}).\]
%
\end{theorem} 
\begin{remark}
  {\rm(Convergent \ainfty homotopies)} Let $A_0$, $A_1$ be \ainfty
  algebras with Novikov coefficients, and let
  $\F_0, \F_1 : A_0 \to A_1$ be \ainfty morphisms An \ainfty homotopy
  $\cT \in \Hom(\F_0,\F_1)$ from $\F_0$ to $\F_1$ is said to be
  \em{convergent} \index{Convergent!\ainfty homotopy} if $\cT^0$ has
  positive $q$-valuation.  For a convergent \ainfty algebra $A$, if a
  convergent morphism $\F:A \to A$ is convergent homotopy equivalent
  to the identity morphism, then the map on the space of Maurer-Cartan
  solutions $MC(\F) : MC(A) \to MC(A)$, and the map on cohomology
  $H(\F) : H(A,b) \to H(A,\F(b))$ for any $b \in MC(A)$ are
  isomorphisms modulo gauge transformations, see \cite[Lemma 5.2,
  5.3]{cw:flips}.
\end{remark}
\begin{proof}
  [Proof of Theorem \ref{thm:twicequilted}]
  We first introduce the necessary moduli spaces of holomorphic twice-quilted
  disks. 
  For a generic perturbation $\ul \Pe_{012}$ that extends perturbation data $\ul \Pe_i$, $i=0,1,2$, and perturbation morphisms
  $\ul \Pe_{01}$, $\ul \Pe_{12}$, $\ul \Pe_{02}$, and $\lam$ lying in a comeager subset $\R_{\geq 0}^{reg} \subset [0,\infty]$, for any type $\Gamma$ of twice-quilted disks, and any collection $\ul x$ of labels on the leaves,  
%
  \begin{itemize}
  \item 
 the moduli
    space $\M^{\qq,\lam}_\Gamma(X, L, \ul \Pe_{012}, \ul x)_d$ of twice quilted
    disks of type $\Gamma$ with ratio of radii is $\lam$ and expected dimension $d \leq 1$, and
  \item the moduli space $\M^{\qq}_\Gamma(X, L, \ul \Pe_{012}, \ul x)_d$ of twice
    quilted disks with the ratio of radii unprescribed with expected dimension $d \leq 0$
  \end{itemize}
  are transversally cut out. 
  For any $\lambda \in \R_{\geq 0}^{reg}$ as above,
counts of rigid twice-quilted holomorphic disks with ratio of
  radii $\lambda$ and index zero defines an
  \ainfty morphism
  \[ \phi_{02}^\lambda = ((\phi_{02}^\lambda)^d)_d,\qquad
    (\phi_{02}^\lambda)^d: CF(L, \ul \Pe_0)^{\otimes d} \to CF(L, \ul
    \Pe_2). \]
 For $\lam=0$,
  the \ainfty morphism $\phi_{02}^0$ corresponds to $\phi_{02}$, and for $\lam=\infty$, 
  $\phi_{02}^\infty$ corresponds to $\phi_{12} \circ \phi_{01}$.

  The homotopy between $\phi_{02}^0$ and $\phi_{02}^\infty$ is defined by combining homotopies constructed from
  small variations of ratio.  We aim to define a pre-natural
  transformation
  $\cT^{02}=((\cT^{02})^d)_{d \geq 0} \in \Hom(\phi_{02}^0,
  \phi_{02}^\infty)$ which will later be shown to be a homotopy.
  Given an energy bound $E>0$, we first define $\cT^{02}$ modulo terms involving $q^{\geq E}$. 
  We may divide $[0,\infty]$ into a finite number of intervals
  \begin{equation}
    \label{eq:def-lami}
  [\lambda_i, \lambda_{i+1}], \quad i=1,\dots, k,  
  \end{equation}
  so that there is at
  most one rigid twice-quilted disk with ratio $\lambda$ and
  symplectic area $\leq E$ (which, we recall from \eqref{eq:mult}, is
  proportional to the number of interior markings $d_\black$) in each
  such interval $[\lambda_i, \lambda_{i+1}]$. We define $\cT_{02}^{\lambda_i,\lambda_{i+1},\leq E}$
  by counting twice-quilted disks with ratio in the specified
  interval,
  \begin{multline} \label{twiceq} (\cT_{02}^{\lambda_i,
      \lam_{i+1},\leq E})^{d(\white)}: CF(L,\ul \Pe_0)^{\otimes
      d(\white)} \to CF(L, \ul \Pe_2) \\ (x_1,\ldots,x_{d(\white)})
    \mapsto \sum_{x_0,u \in \M^{\qq,\lam}_\Gamma(X,L,\ul \Pe, \ul x)_0}
    w(u) x_0,
  \end{multline}
  where $\ul x=(x_0,x_1,\dots,x_{d(\white)})$ and the sum is over
  rigid combinatorial types $\Gamma$ of twice-quilted disks with
  symplectic area $\leq E$.  We will prove in Lemma \ref{lem:homotopy}
  below that for any $d(\white) \geq 0$
  and any tuple of leaf labels $(x_1,\dots,x_{d(\white)}) \in \hat \cI(L)^{d(\white)}$, 
  \[ (\phi_{02}^{\lambda_i} -
    \phi_{02}^{\lambda_{i+1}})^{d(\white)}(x_1,\ldots,x_{d(\white)})
    =(m^1_{\cQ} \cT_{02}^{\lambda_i,\lam_{i+1},\leq E})^{d(\white)}(x_1,\ldots,x_{d(\white)}) \mod q^E.  \]
Composition using \eqref{eq:homotopy-compose} produces a homotopy
\[ \TT_{02}^{\leq E} := \TT_{02}^{\lambda_{k-1},\lambda_k,\leq E}\circ (\ldots \circ (\TT_{02}^{\lambda_{1},\lambda_{2}, \leq E}\circ 
\TT_{02}^{\lambda_{0},\lambda_1, \leq E} ) )\]
%
between $\phi_{12} \circ \phi_{01}$ and $\phi_{02}$ modulo terms
involving powers $q^{\geq E}$, that is, 
\[\phi_{02} - (\phi_{12} \circ \phi_{01}) = \TT_{02}^{\leq E} \mod q^E.\]
Taking the limit $E \to \infty$ defines a
homotopy
\[\TT_{02} := \lim_{E \to \infty} \TT_{02}^{\leq E} \] 
between $\phi_{02}$ and $\phi_{12} \circ \phi_{01}$.
The limit exists because for any $E_1>E$, $\TT_{02}^{\leq E}=\TT_{02}^{\leq E_1} \on{mod} q^E$, and therefore terms in $\TT_{02}$ with valuation $<E$ are given by $\TT_{02}^{\leq E}$. 
Convergence,
that is, that $\T^0_{02}(1)$ has coefficients with positive
$q$-valuation, holds since any treed disk with no incoming
  edges must, by stability, have a disk component with interior
  markings, and therefore has positive area. 
\end{proof} 

\begin{lemma}\label{lem:homotopy}
  For an area bound $E>0$, let the interval $[\lam_i,\lam_{i+1}]$ be as in \eqref{eq:def-lami}.
  Then,
  \begin{equation}
    \label{eq:homotopy-id}
  \phi_{02}^{\lambda_i,\leq E} - \phi_{02}^{\lambda_{i+1},\leq E} =
  m^1_{\cQ}(\cT_{02}^{\lambda_i,\lambda_{i+1},\leq E}) \mod q^E.   
  \end{equation}
\end{lemma}
Configurations consisting of multiple twice-quilted disks attached to
an unquilted (dark shaded) disk by broken edges are not transversally
cut out, and they occur in the boundary of one-dimensional moduli
spaces. Indeed, our current definition of holomorphic twice quilted
disks require the ratios of radii (the surface version \eqref{eq:lamS}
or the treed version \eqref{eq:lamT}) to be equal for all twice
quilted components. This requirement may be viewed as a fiber product,
and the resulting moduli space is not transversally cut out.

In order to obtain transversality, the fiber products
\eqref{eq:lamS}, \eqref{eq:lamT} 
involved in the definition of the twice-quilted disks 
must be perturbed, using \em{delay functions}, as
in Seidel \cite{se:bo} and Ma'u-Wehrheim-Woodward \cite{ainfty}.
\label{ainftyrefp}
In a twice-quilted disk type, 
denote by
\[\Edge_{\white,-}^2(\Gamma) \subset \Edge_{\white,-}(\Gamma)\]
the subset of boundary edges that lie below the inner quilting circle. In terms of Figure \ref{fig:oneleaf}, edges in $\Edge_{\white,-}^2(\Gamma)$ connect vertices with the darkest shading.
A \em{delay function} is a collection of maps 
\[\tau_{d(\black),d(\white)}=(\tau_\Gamma)_\Gamma, \quad \tau_\Gamma=(\tau_{\Gamma, e})_{e \in \Edge_{\white,-}^2(\Gamma)}: \ol \M^{\qq,\univ}_\Gamma \to \R^{\Edge_{\white,-}^2(\Gamma)},  \]
where $\Gamma$ ranges over all rigid twice-quilted disk types with $d(\black)$ interior leaves and $d(\white)$ non-root boundary leaves, and $\tau_\Gamma$ satisfies the following:
\begin{enumerate}
\item {\rm(Locality)} The map $\tau_\Gamma$ depends only on the edge lengths
  of edges in $\Edge_{\white,-}^2(\Gamma)$,
\item {\rm(Zero edges)} If for a curve $C \in \ol \M^{\qq,\univ}_\Gamma$ the edge length of $e \in \Edge_{\white,-}^2(\Gamma)$ is zero, then $\tau_{\Gamma,e}(C)=0$.
\item {\rm(Continuity)} If $\Gamma'$ is a stratum containing an edge $e$ of length zero, and $\Gamma_0$ and $\Gamma_1$ are open stratum in $\M_{d(\black),d(\white)}^{\qq,\univ}$ obtained by collapsing the edge $e$ and making $\ell(e)$ non-zero respectively, then,
  \[\tau_{\Gamma_0}|\M^{\qq,\univ}_\Gamma=\tau_{\Gamma_1}|\M^{\qq,\univ}_\Gamma.\]
\item {\rm(Infinite edges)} \label{part:infedge}
  For any rigid twice quilted disk type $\Gamma'$ which has $\leq d(\white)$ boundary leaves and $\leq d(\black)$ interior leaves, with at least one of the inequalities being strict, 
  there exist delay maps
  \[\tau_{\Gamma,\Gamma'}=(\tau_{\Gamma,\Gamma',e})_e : \M^{\qq,\univ}_{\Gamma'} \to \R^{\Edge_{\white,-}^2(\Gamma')}  \] 
such that the following holds: 
If $\M^{\qq,\univ}_\Gamma$ is a stratum with a collection of edges of infinite length, that separate $\Gamma$ into an unquilted type $\Gamma_0$ with the darkest shading, and a set $\{\Gamma_1,\dots,\Gamma_k\}$ of twice-quilted disk types connected to $\Gamma_0$ by edges $e_1, \dots,e_k$, so that $\M^{\qq,\univ}_\Gamma$ is a product
\[\M^{\qq,\univ}_\Gamma \xrightarrow{(\pi_0, \dots,\pi_k)} \M_{\Gamma_0} \times \prod_{i=1}^k\M^{\qq,\univ}_{\Gamma_i},\]

then for any edge $e \in \Edge_{\white,-}^2(\Gamma_i)$, $i=1,\dots,k$,  or $e \in \Edge_\white(\Gamma_0)$ and $i=0$, 
\begin{equation}
  \label{eq:tau-cut-edge}
  \tau_{\Gamma,\Gamma_i,e} \circ \pi_i = \tau_{\Gamma,e}.  
\end{equation}
Furthermore, for any $i=1,\dots,k$,
\begin{equation}
  \label{eq:const-on-i}
\tau_{\Gamma,e_i}=\text{constant} \quad \text{on } \ol \M_\Gamma^{\qq,\univ}.  
\end{equation}

\item {\rm(Increasing condition)}
  Let $\gamma_1$, $\gamma_2$ be paths in 
  $\Gamma$ starting from the root vertex $v_0$ and such that neither of them is contained in the other. If $\gamma_1<\gamma_2$ (in the sense that if, for $i=1,2$, $\gamma_i$ is contained in the path from the root vertex $v_0$ to a non-root leaf $e_i$, then $e_1$ occurs before $e_2$), then
  \begin{equation}
    \label{eq:tau-inc}
    \sum_{e \in \gamma_1} \tau_{\Gamma, e} < \sum_{e \in \gamma_2} \tau_{\Gamma, e}.
  \end{equation}
\end{enumerate}

We now define 
\em{delayed twice-quilted disks} where the fiber product is replaced by a perturbed version.
Define
\[\lam_{\tau_\Gamma} : \M^{\qq,\univ}_\Gamma \to [0,\infty]^{\Ver_{12}(\Gamma)}, \quad \lam_{\tau_\Gamma}(C)_v:=(\lam_\Gamma(C))_v + \sum_{e \in P(v_0,v)}(\tau_\Gamma(C))_e, \]
where $P(v_0,v)$ is the path in $\Gamma$ from the root vertex $v_0$ to $v$. 
An unbalanced twice-quilted disk $C$ 
is a \em{$\tau$-delayed twice-quilted disk} if $\lam_{\tau_\Gamma}(C)_v$ is equal for all vertices $v \in \Ver^{12}(C)$ containing an inner quilting circle $Q_2$,
and thus, the moduli space of $\tau$-delayed twice-quilted disks of type $\Gamma$ is
\[\M^{\qq,\lam_0}_{\Gamma,\tau}=\lam_{\tau_\Gamma}^{-1}(\Delta),\]
equipped with a \em{delayed ratio of radius} map
\begin{equation}
  \label{eq:psi-delayed}
  \psi:\M^{\qq,\lam_0}_{\Gamma,\tau} \to \R, \quad C \mapsto \lam_{\tau_\Gamma}(v),  
\end{equation}
where $v \in \Ver^{12}(\Gamma)$ can be any disk vertex containing an inner quilting circle. Delayed twice-quilted disks which have specific values of the delayed ratio of radii are denoted by 
\[\M^{\qq,\lam_0}_{\Gamma,\tau}:=\psi^{-1}(\lam), \quad \M^{\qq,[\lam_i,\lam_{i+1}]}_{\Gamma,\tau}:=\psi^{-1}([\lam_i,\lam_{i+1}]).\]

\begin{proof}
  [Proof of Lemma \ref{lem:homotopy}] To define delayed holomorphic
  twice-quilted disks, we choose perturbations close to the ones used
  to define ordinary holomorphic twice-quilted disks.  Given
  perturbation data $\ul \Pe_{012}$ for holomorphic twice quilted
  disks, we take the delay functions to be small enough and define
  coherent perturbations
  \[\ul \Pe_{012}^\tau=(\Pe_{012, \Gamma}^\tau)_\Gamma\]
  so that the moduli spaces
  \[\M^{\qq, [\lam_i,\lam_{i+1}]}_{\Gamma,\tau}(X,L,\ul \Pe_{012})_d, \quad
    \M^{\qq, \lam_i}_{\Gamma,\tau}(X,L,\ul \Pe_{012})_d, \quad \M^{\qq,
      \lam_{i+1}}_{\Gamma,\tau}(X,L,\ul \Pe_{012})_d\]
  with expected dimension $d \leq 1$ are transversally cut out, and
  there are cobordisms of zero-dimensional moduli spaces
  \begin{equation}
    \label{eq:cobord1}
    \M^{\qq, \lam_i}_\Gamma(X,L,\ul \Pe_{012})_0 \sim \M^{\qq, \lam_i}_{\Gamma,\tau}(X,L,\ul \Pe_{012}^\tau)_0, \quad
    \M^{\qq, \lam_{i+1}}_\Gamma(X,L,\ul \Pe_{012})_0 \sim \M^{\qq, \lam_{i+1}}_{\Gamma,\tau}(X,L,\ul \Pe_{012}^\tau)_0,
  \end{equation}
  and
  \begin{equation}
    \label{eq:cobord2}
    \M^{\qq, [\lam_i, \lam_{i+1}]}_{\Gamma'}(X,L,\ul \Pe_{012})_0 \sim \M^{\qq, [\lam_i,\lam_{i+1}]}_{\Gamma',\tau}(X,L,\ul \Pe_{012}^\tau)_0
  \end{equation}
  for all types $\Gamma' \leq \Gamma$.

  We finally prove the identity required by the Lemma. For any
  $d(\white)$, we consider input labels $(x_1,\dots,x_{d(\white)})$
  and output label $x_0$, and we aim to prove
  \begin{multline}
    \label{eq:homot-x}
    (\phi_{02}^{\lambda_i,\leq E} - \phi_{02}^{\lambda_{i+1},\leq E})^{d(\white)}(x_1,\dots,x_{d(\white)}) =\\
    (m^1_{\cQ}(\cT_{02}^{\lambda_i,\lambda_{i+1}}))^{\leq E,
      d(\white)}(x_1,\dots,x_{d(\white)}).
  \end{multline}
  We will show that \eqref{eq:homot-x} 
  is obtained by summing the
  contributions from boundaries of one-dimensional moduli spaces
  $\ol \M^{\qq, [\lam_i,\lam_{i+1}]}_{\Gamma,\tau}(X,L,\ul
  \Pe_{012}^\tau, \ul x)_1$ where $\Gamma$ ranges over rigid types of
  twice-quilted holomorphic disks with $d(\white)$ boundary inputs and
  $\leq \frac {E}{k}$ interior leaves as follows: A configuration
  occurring in the boundary of
  \[\ol \M^{\qq, [\lam_i,\lam_{i+1}]}_{\Gamma,\tau}(X,L,\ul
    \Pe_{012}^\tau, \ul x)_1\]
  is
  \begin{enumerate}
  \item\label{part:tq1} either a twice-quilted disk in
    $\M^{\qq, \lam_i}_{\Gamma,\tau}(X,L,\ul x)_0$ or 
    $\M^{\qq, \lam_{i+1}}_{\Gamma,\tau}(X,L,\ul x)_0$; or
  \item \label{part:tq2} it has an unquilted disk with the lightest shading breaking off (as
    in (A), (B) in Figure \ref{fig:bdry-q2}); or 
  \item \label{part:tq3} it has a collection of twice-quilted disks
    breaking off (as in (C), (D) in Figure \ref{fig:bdry-q2}).
  \end{enumerate}
  \begin{figure}[ht]
    \centering \scalebox{.8}{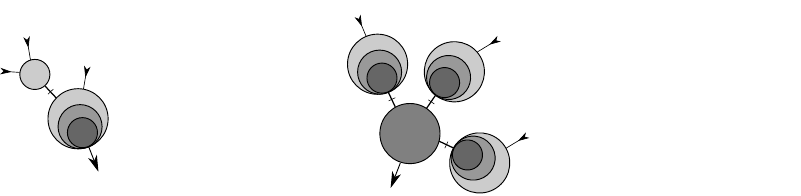}
    \caption{Configurations occurring in codimension one boundaries of
      moduli spaces of delayed twice-quilted disks with a fixed
      (generalized) ratio of radii. In (A) and (B), an unquilted disk
      breaks off outside the outer quilting circle. In (C) and (D), a
      collection of twice quilted disks break off.}
    \label{fig:bdry-q2}
  \end{figure}
  \noindent A count of configurations of type \eqref{part:tq1}, by the cobordism
  in \eqref{eq:cobord1}, is equal to the left hand side of
  \eqref{eq:homot-x}. 
   A configuration $u=(u_0,u_1)$ of type
   \eqref{part:tq2},
   consists of an unquilted $\ul \Pe_0$-holomorphic disk $u_0$ and a 
    twice-quilted disk $u_1$ whose ratio of radius
  $\lam$ lies in $(\lam_i,\lam_{i+1})$, and thus, $u_1$ lies in
  $\M^{\qq, [\lam_i,\lam_{i+1}]}_{\Gamma',\tau}(X,L,\ul x)_0$ for some
  type $\Gamma' \leq \Gamma$.  By the cobordism \eqref{eq:cobord2}, the count of configurations
  of type \eqref{part:tq2} is equal to the sum of terms of the form 
  \[\cT_{02}^{\lam_i,\lam_{i+1},\leq E}(\dots,m^e_{CF(L,\ul
      \Pe_0)}(\dots),\dots),\]
 occurring in 
 $(m^1_{\cQ} \cT_{02}^{\lam_i, \lam_{i+1}})^{\leq E,d(\white)}$, which is the right hand side of \eqref{eq:homot-x}.
 Finally, a configuration
  $u=(u_0,\dots, u_r)$ of type \eqref{part:tq3} consists of an
  unquilted $\ul \Pe_2$-holomorphic disk, and a collection
  $u_1,\dots, u_r$ of twice-quilted disks, with $u_i$ ($i \geq 1$)
  connected to $u_0$ by a broken edge $e_i$.
  Let
  \[\Lam_0:= \psi(u)\]
  be the delayed ratio of radius (defined in \eqref{eq:psi-delayed}). 
  For degree reasons,
  because of the fiber product with the diagonal, exactly one of the 
  twice-quilted disks in the set $\{u_1,\dots,u_r\}$, say the $k$-th, lies in the moduli space of
  expected dimension zero, while the rest have index one. The interval
  $[\lam_i, \lam_{i+1}]$ was chosen so that there is at most one
  twice-quilted disk in the zero-dimensional moduli space
  $\M^{\qq, [\lam_i,\lam_{i+1}], \leq E}(X,L)_0$ with at most
  $d(\white)$ inputs. That is, $u_k$ is the only element in this set, and its $\tau$-delayed ratio
  is $\psi(u_k)=\Lam_0+\tau_{\Gamma,e_k}(C_u)$, where $C_u$ is the domain curve of $u$. 
  For any $j \neq k$,
  \[\psi(u_j)=\Lam_0 +\tau_{\Gamma,e_j}(C_u),\] 
  $u_j$ has index $1$, and lies in the one-dimensional moduli space
  \[\ol \M(u_j):=\M^{\qq, [\lam_i,\lam_{i+1}]}_{\Gamma_j,\tau}(X,L,\ul
    \Pe_{012}^\tau, \ul x_j)_1\]
  for some input labels $\ul x_j$. We then have the following Claim.
  \begin{claim}
    \label{claim:wherebdry}
    For $j<k$, $\ol \M(u_j)$ does not have end-points in $\psi^{-1}(\lam_i, \Lam_0 + \tau_{\Gamma,e_j}(C_u))$
    and for $j<k$, $\ol \M(u_j)$ does not have end-points in $\psi^{-1}(\Lam_0 + \tau_{\Gamma,e_j}(C_u),\lam_{i+1})$.
  \end{claim}
\noindent  Assuming Claim \ref{claim:wherebdry}, we finish the proof of the Lemma. The Claim implies that for $j<k$, 
  there is a cobordism
  \begin{equation}
    \label{eq:cobord-j1}
  \psi^{-1}(\Lam_0 + \tau_{\Gamma,e_j}(C_u)) \cap \M(u_j)_1 \sim  \M^{\qq, \lam_i}_{\Gamma_j,\tau}(X,L,\ul \Pe_{012}^\tau,\ul x_j)_0,  
  \end{equation}
  since the disjoint union of both the spaces in \eqref{eq:cobord-j1} form the boundary of 
  the one-dimensional moduli space $\ol \M(u_j)_1 \cap \{\psi \geq (\Lam_0 + \tau_{\Gamma,e_j}(C_u))\}$. 
  Similarly, for $j>k$, there is a cobordism
  \begin{equation}
    \label{eq:cobord-j2}
  \psi^{-1}(\Lam_0 + \tau_{\Gamma,e_j}(C_u)) \cap \M(u_j)_1 \sim  \M^{\qq, \lam_{i+1}}_{\Gamma_j,\tau}(X,L,\ul \Pe_{012}^\tau)_0.
  \end{equation}
%
  In the case when $j<k$ resp. $j>k$, the cobordism \eqref{eq:cobord-j1}
  resp. \eqref{eq:cobord-j2} combined with the cobordism in
  \eqref{eq:cobord1} implies that 
  the signed count of maps $u_j$, with 
  inputs $\ul x_j^{in}$ and delayed ratio of radius given by $\psi(u_j)=\Lam_0+\tau_{\Gamma,e_j}(C_u)$,  
  is equal to
  \[\phi_{02}^{\lam_i}(\ul x_j^{in}) \quad \text{resp.} \quad  \phi_{02}^{\lam_{i+1}}(\ul x_j^{in}).\]
  We conclude that the count of configurations
  $(u_0,\dots,u_r)$
  occurring in the boundary
  of
  $\ol \M^{\qq, [\lam_i,\lam_{i+1}]}_{\Gamma,\tau}(X,L,\ul
  \Pe_{012}^\tau, \ul x)_1$ that are of type \eqref{part:tq3}, and for
  which the $k$-th twice-quilted disk $u_k$ has index $0$ is equal to the count of terms
  \[m^r_{CF(L,\ul \Pe_2)}(\underbrace{\phi_{02}^{\lam_i}(\dots), \dots, \phi_{02}^{\lam_i}(\dots)}_{k-1}, \cT^{02}(\dots), \underbrace{\phi_{02}^{\lam_{i+1}}(\dots), \dots, \phi_{02}^{\lam_{i+1}}(\dots)}_{r-k}) \mod q^E, \]
  occurring in
  $(m^1_{\cQ} \cT^{\lam_i,\lam_{i+1}})^{\leq E, d(\white)}$, which is the right hand side of \eqref{eq:homot-x}, finishing the proof of the Lemma.  It remains to prove Claim \ref{claim:wherebdry}.

  \begin{subproof}[Proof of Claim \ref{claim:wherebdry}]
  We consider the
  case $j<k$, the proof of the other case, $j>k$, being similar.  A map
  $v$ occurring in the boundary of the moduli space $\ol \M(u_j)_1$ in
  $\psi^{-1}(\lam_i, \Lam_0 + \tau_{\Gamma,e_j}(C_u))$ is of one of the two
  forms:
  \begin{itemize}
  \item Either $v$ has a collection of twice-quilted disks breaking
    off, that is, $v=(v_0,\dots,v_{r'})$, where $v_0$ is an unquilted
    disk with darkest shading, the others are $\tau$-delayed
    twice-quilted disks with one of them, say $v_\ell$, having index
    $0$. Since $u_k$ is the only delayed twice-quilted disk of index
    $0$ in $\psi^{-1}([\lam_i,\lam_{i+1}])$, we have that $v_\ell$ is
    the same as $u_k$. In particular,
    $\psi(v_\ell)=\Lam_0 + \tau_{\Gamma,e_k}(C_u)$. Suppose the root vertex
    of $v_k$ is connected to the root vertex of $v$ by a path
    $\gamma$.  By our assumption that $\psi(v) \in (\lam_i, \Lam_0 + \tau_{\Gamma,e_j})$, we have 
    \begin{equation}
      \label{eq:contra-C}
    \Lam_0 + \tau_{\Gamma,e_j}(C_u)>\psi(v)=\Lam_0 + \tau_{\Gamma,e_k}(C) - \sum_{e \in \gamma}\tau_{\Gamma,e}(C),   
    \end{equation}    
    where 
    the curve $C$ is obtained by replacing the subdomain of $u_j$ in $C_u$ with the domain of $v$. 
    Note that \eqref{eq:tau-cut-edge} implies $\tau_{\Gamma, e_k}(C)=\tau_{\Gamma, e_k}(C_u)$, and 
    \eqref{eq:const-on-i} implies $\tau_{\Gamma, e_j}(C)=\tau_{\Gamma, e_j}(C_u)$. Therefore,  \eqref{eq:contra-C} violates the increasing condition \eqref{eq:tau-inc} at the curve $C$.    
  \item Otherwise, $v$ has an unquilted disk bubbling off, that is $v=(v_0,v_1)$ where $v_0$ is twice-quilted and $v_1$ is unquilted. Since $u_k$ is the only delayed twice-quilted disk
    of index $0$ in $\psi^{-1}([\lam_i,\lam_{i+1}])$, we have
    $\psi(u_k)=\psi(v_1)=\Lam_0 + \tau_{\Gamma,e_k}(C_u)$, which
    cannot 
    lie in the interval
    $(\lam_i, \Lam_0+\tau_{\Gamma,e_j}(C_u))$, since the increasing condition \eqref{eq:tau-inc} implies
    that $\tau_{\Gamma,e_j}(C_u)<\tau_{\Gamma,e_k}(C_u)$.
  \end{itemize}
  This proves Claim \ref{claim:wherebdry}, and hence, also proves the  Lemma.
    \end{subproof}
  \end{proof}

  \begin{corollary}\label{cor:hequiv}
    For any two coherent convergent collections of perturbation data
  $\ul{\Pe}_0, \ul{\Pe}_1$ the Fukaya algebras $CF(L,\ul{\Pe}_0)$
  and $CF(L,\ul{\Pe}_1)$ are homotopy equivalent via 
  unital convergent \ainfty morphisms
  \[\phi: CF(L, \ul \Pe_0) \to CF(L, \ul \Pe_1), \quad \psi: CF(L, \ul \Pe_1) \to CF(L, \ul \Pe_0) \]
  and unital convergent homotopies $h$, $g$ satisfying
  \[\phi \circ \psi - \Id = m^1_\cQ(h), \quad \psi \circ \phi - \Id = m^1_\cQ(g).\]
\end{corollary} 

\begin{proof}
  The \ainfty morphisms $\phi$, $\psi$ are as constructed in the proof
  of Proposition \ref{prop:samedegree}, and the morphisms are unital
  and convergent.  We apply Theorem \ref{thm:twicequilted} taking
  $\Pe_2=\Pe_0$, $\phi_{01}=\phi$, $\phi_{12}=\psi$, $\phi_{02}=\Id$,
  and obtain a unital convergent homotopy $h$.  The homotopy $g$ is obtained by
  switching the role of
  $\Pe_0$ and $\Pe_1$. 
  \end{proof}

\section{Homotopy equivalence: unbroken to
  broken}\label{sec:homo-u2b}
In this section, we show that the Fukaya algebra of the broken
manifold $\XX$ is homotopy equivalent to the unbroken Fukaya algebra
on the manifold $X$.  A rough outline of the proof is as follows. An \ainfty morphism from
an unbroken Fukaya algebra $CF(L, \ul \Pe^\nu)$, where $\ul \Pe^\nu$ is a perturbation system on the neck-stretched manifold $X^\nu$, 
to a broken one
$CF^\br(L , \ul \Pe^\infty)$ is defined as a limit of \ainfty morphisms
$\phi_n:CF(L, \ul \Pe^\nu) \to CF(L, \ul \Pe^{\nu +n})$ as $n \to \infty$. The
\ainfty morphism $\phi_n$ is defined via a count of quilted
holomorphic disks in $(X,L)$. The existence of the limit relies on the
bijection between moduli spaces of broken and unbroken disks obtained
from the compactness and gluing theorems (Theorem
\ref{thm:cpt-breaking} and Theorem \ref{thm:gluing}).

We use a breaking perturbation datum on neck-stretched manifolds, so
that the bijection between moduli spaces of broken and unbroken disks
holds.  We recall from Definition \ref{def:breakpert} that given a
perturbation datum $\ul{\Pe}^\infty$ on the broken manifold $\XX$, by
gluing on the neck we obtain a perturbation datum
$\rho_\nu(\ul{\Pe}^\infty)$ on any $X^\nu$.  By Proposition
\ref{prop:stab-broken} there is a stabilizing pair $(\JJ,\DD)$
consisting of a cylindrical almost complex structure $\JJ$ on $\XX$
and a $\JJ$-holomorphic cylindrical stabilizing divisor $\DD$, for
which the glued family $(J^\nu, D^\nu)$ is a stabilizing pair on the
neck-stretched manifold $X^\nu$ for all $\nu$.  For the perturbation
datum $\ul{\Pe}^\infty$ we use $\DD$ as the stabilizing divisor and
$\JJ$ as the background almost complex structure.

 \begin{proposition}
   \label{prop:breakingpert}
   Let $\ul{\Pe}^\infty$ be a regular perturbation datum on $\XX$. For any
   $E_0>0$, there exists $\nu_0(E_0)$ such that the following holds.
   \begin{enumerate}
   \item
     \label{part:breaking1}
     There is a bijection between the zero-dimensional moduli spaces
     of rigid broken maps and rigid unbroken maps:
     \[\tM^\br(\XX,L, \ul{\Pe}^\infty)_0^{<E_0} \simeq
     \M(X^\nu,L,\rho_\nu(\ul{\Pe}^\infty))_0^{<E_0}, \quad \nu
     \geq \nu_0(E_0).\]
   Furthermore, the moduli spaces of rigid quilted and twice-quilted
   disks relating $X^{\nu}$ and $X^{\nu+1}$ with energy at most $E_0$
   are empty for $\nu \geq \nu_0(E_0)$.
 \item \label{part:breaking2} For any $\nu \in \Z_{>0}$, there exists
   a regular perturbation datum $\ul{\Pe}^\nu$, and a perturbation
   morphism $\ul{\Pe}^{\nu,\nu+1}$ extending $\ul{\Pe}^\nu$ and
   $\ul{\Pe}^{\nu+1}$ such that for all $E_0>0$ and for all
   $\nu \geq \nu_0(E_0)$, the (possibly curved) \ainfty homotopy
   equivalence
induced by $\ul{\Pe}^{\nu,\nu+1}$
     \[\phi_\nu^{\nu+1}: CF(L,\ul{\Pe}^\nu) \to CF(L,    \ul{\Pe}^{\nu+1})\]
   is the identity map 
   modulo $q^{E_0}$. Similarly, there is a (possibly curved)  \ainfty
   homotopy equivalence 
    \[\psi_{\nu+1}^\nu: CF(L,\ul \Pe^{\nu+1}) \to CF(L,\ul
      \Pe^{\nu})\]
that is  the identity map 
modulo $q^{E_0}$.  
%
Furthermore, the homotopy $h_\nu$ between 
$\phi_\nu^{\nu+1} \circ \psi_{\nu+1}^\nu$ and the identity,
and the homotopy $g_\nu$ between 
%
$ \psi_{\nu+1}^\nu \circ \phi_\nu^{\nu+1}$ and the identity
may be taken to be zero modulo $q^{E_0}$.
  \end{enumerate}
 \end{proposition}
 
 \begin{proof}
   The proof of bijection of moduli spaces is a consequence of the
   compactness and gluing theorems -- Theorem \ref{thm:cpt-breaking}
   and Theorem \ref{thm:gluing}.  By the gluing Theorem
   \ref{thm:gluing}, any regular $\ul{\Pe}^\infty$-disk $u: C \to \XX$
   can be glued to yield regular disks $u_\nu: C_\nu \to X^\nu$.
   Conversely, any sequence $(u_\nu)_\nu$ of maps with area $\leq E_0$
   converges to a broken disk $u_\infty$.  By surjectivity of gluing
   in Theorem \ref{thm:gluing}, for large enough $\nu$, $u_\nu$ is
   contained in the image of the gluing map for $u_\infty$.  Since the
   moduli space $\tM(\XX,L)^{\leq E_0}_0$ has a finite number of
   points, the constant $\nu_0(E_0)$ can be chosen as the maximum
   obtained from gluing each of the broken maps.
   
   To prove the statements about quilted disks and part
   \eqref{part:breaking2}, we construct regular perturbations and
   perturbation morphisms on neck-stretched manifolds.  A regular
   perturbation datum is constructed by extending the breaking
   perturbation datum.  For any $E_0>0$ and $\nu > \nu_0(E_0)$, the
   perturbation datum is defined as
   \begin{equation}
     \label{eq:pert-low-area}
     \Pe^\nu_\Gamma := \rho_\nu(\Pe^\infty_\Gamma)=(J^\nu_\Gamma, F^\infty_\Gamma)  
   \end{equation}
   if $E(\Gamma) \leq E_0$.  For the other strata, a regular
   perturbation $\Pe^\nu_\Gamma$ can be chosen using the
   transversality result Theorem \ref{thm:transversality} for all
   integers $\nu$.

   For strata with low enough area, perturbation morphisms are
   constructed using the breaking perturbation data. For
   $\nu \in \Z_+$, $\Gamma$ satisfying $\nu \geq \nu_0(E(\Gamma))$,
   the perturbation morphism is defined using the recipe in Remark
   \ref{rem:jtpert} as
   \[\Pe^{\nu,\nu+1}_\Gamma=(J_\Gamma^{\nu,\nu+1},F^\infty_{\Gamma'})\]
   where $F^\infty_{\Gamma'}$ is part of the perturbation data
   $\ul \Pe^\infty = (\ul J^\infty, \ul F^\infty)$ for broken disks,
   $\Gamma'$ is the type of broken disk obtained by forgetting the
   quilting, and for any $z \in \S_\Gamma$,
   %
   $J_\Gamma^{\nu,\nu+1}(z)$ is the neck-stretched domain-dependent
   almost complex structure from \eqref{eq:pert-low-area} with
   neck-length parameter $\nu + d \circ \delta_\Gamma(z)$.  The
   perturbation morphisms are extended to higher strata so that they
   are coherent and regular.

   Perturbations for twice-quilted disks
   are defined similarly: On low area strata, we use the
   neck-stretched domain-dependent almost complex structure from
   \eqref{eq:pert-low-area} with neck length parameter given by the
   domain-dependent function $\nu + d \circ
   \delta_\Gamma(z)$. Here $\delta_\Gamma$ is a version of
   the ``distance from the quilting circle'' for twice-quilted disks from \cite[Definition
   5.9]{cw:flips}.

   We claim that for an area bound $E_0$, for sufficiently large
   $\nu \geq \nu_0(E_0)$, moduli spaces of holomorphic quilted and
   twice-quilted disks are empty, with the exception of constant
   quilted holomorphic disks with a single input and output with the
   same label (as in Figure \ref{trivq}). More precisely, we show that
   the moduli space of non-constant quilted
   $\ul \Pe^{\nu, \nu+1}$-disks and twice-quilted disks
   contributing to the homotopies $g_\nu$, $h_\nu$ with area
   $\leq E_0$ are empty if $\nu \geq \nu_0(E_0)$.
   %
   %
   %
   Suppose, by way of contradiction, that there is a sequence $u_\nu$
   of quilted $\ul{\Pe}^{\nu,\nu+1}$-disks with bounded area.  Any
   surface component of $u_\nu$ is $J^{\nu'}$-holomorphic for some
   $\nu' \in [\nu,\nu+1]$. Therefore, after passing to a subsequence,
   the surface part of the map $u_\nu$ converge to a limit
   $J^\infty$-holomorphic broken map $u$.  For large enough $\nu$, the
   Maslov index $I(u_\nu)$ is $\nu$-independent, and therefore the
   number of disk inputs is bounded. After passing to a subsequence,
   we can assume that the disk input/output tuple $\ul x$ and the type
   $\Gamma$ is the same for all the quilted disks $u_\nu$. Since the
   \ainfty morphism counts quilted holomorphic disks with index $0$,
   we may assume that the index of the quilted disks is zero:
   $i^\quilt(\Gamma,\ul x)=0$.  The sequence $u_\nu$ converges to a
   $\ul{\Pe}^\infty$-unquilted broken disk $u$ of type $\Gamma_\tau$
   and input/output labels given by $\ul x$.  Here $\Gamma_\tau$ is
   the type of a broken treed holomorphic map for which collapsing the
   tropical edges yields a type $\Gamma'$ of a treed holomorphic map,
   which is the same as the type obtained by forgetting the quilting
   in $\Gamma$.  The unquilted map $u$ is stable in all cases except
   the one contributing to the identity term, that is, when there is a
   single input and output, and the map is constant on the surface and
   tree part of the domain. In all other cases, the index of the
   unquilted disk is one less (see \eqref{eq:expdim-q}), and we have
   $i(\Gamma',\ul x)=-1$.  Since collapsing tropical edges does not
   affect the dimension, we have $i(\Gamma_\tau,\ul x)=-1$.  The
   existence of such a disk $u$ is a contradiction, since
   $\ul{\Pe}^\infty$ is a regular perturbation.  The argument for
   twice-quilted disks is similar.
   
   Finally, the \ainfty morphisms $\phi_\nu^{\nu+1}, \psi_{\nu+1}^\nu$
   are defined by counting $\ul{\Pe}^{\nu,\nu+1}$-holomorphic quilted
   holomorphic disks.  The previous paragraph shows that if
   $\nu \geq \nu_0(E_0)$,
   %
   \[(\phi_\nu^{\nu+1})^1=\Id \mod q^{E_0}, \quad  (\phi_\nu^{\nu+1})^d=0 \mod q^{E_0}, d \neq 1,\]
   and similarly for $\psi_{\nu+1}^\nu$.
   The homotopies $h_\nu$, $g_\nu$ are defined by counts of
   twice-quilted disks, and therefore, they vanish modulo $q^{E_0}$ if
   $\nu \geq \nu_0(E_0)$.
 \end{proof}

 \begin{corollary} \label{mdconv} Let $CF(\ul \Pe^\nu,L)$ and
   $CF_\br(\ul \Pe^\infty,L)$ be the unbroken and broken Fukaya
   algebras defined by the perturbation data $\ul \Pe^\nu$,
   $\ul \Pe^\infty$ from Proposition \ref{prop:breakingpert}.  The
   structure maps $m^d_\nu$ defining $CF(\ul \Pe^\nu,L)$ converge as
   $\nu \to \infty$ to the structure maps $m^d_\br$ defining
   $CF_\br(\ul \Pe^\infty,L)$, in the sense that for any $E_0$ there
   exists $\nu(E_0)$ so that $m^d_\nu$ agrees with $m^d_\br$ up to
   terms divisible by $q^{E_0}$ for $\nu > \nu(E_0)$.
 \end{corollary}

\begin{proof}[Completion of the second proof of Theorem \ref{thm:yields-br}]
  It suffices to prove that for any $E_0>0$ that
  homotopy-associativity holds up to order $q^{E_0}$.  This follows
  from Corollary \ref{mdconv} and Theorem \ref{thm:yields-unbr}, since
  the broken and unbroken composition maps are equal to up arbitrary
  order, for sufficiently large neck length.
\end{proof}

 \begin{proposition}\label{prop:unbreak-break}
   For any $\nu_0$, let $\ul \Pe_{\nu_0}$ be the perturbation defined in
   Proposition \ref{prop:breakingpert}.  Then
   there are  convergent strictly  unital \ainfty morphisms 
   \[\phi:CF(L,\ul \Pe^{\nu_0}) \to CF_\br(L,\ul
     \Pe^\infty), \quad \psi:CF_\br(\ul \Pe^\infty) \to
     CF(L,\ul \Pe^{\nu_0})\]
   such that $\psi \circ \phi$ and $\phi \circ \psi$ are \ainfty
   homotopy equivalent by a convergent \ainfty homotopy.
 \end{proposition}

 \begin{proof}
   The \ainfty morphisms $\phi$ and $\psi$ required by the Proposition
   are defined to be the limits of the compositions of the \ainfty
   morphisms from Proposition \ref{prop:breakingpert}
   \eqref{part:breaking2}, namely,
  \[ \phi_{\nu}^{\nu+1}: CF(L,\ul{\Pe}^\nu) \to
  CF(L,\ul{\Pe}^{\nu+1}), \quad \psi_{\nu+1}^{\nu}:
  CF(L,\ul{\Pe}^{\nu+1}) \to CF(L,\ul{\Pe}^\nu). \]
For any $E_0$, if $\nu \geq \nu'(E_0)$, the \ainfty morphisms
$\psi_{\nu}^{\nu+1}$ and $\psi_{\nu+1}^{\nu}$  are equal to $\Id$ modulo $q^{E_0}$.  Therefore, there exist limits of the successive compositions:
Let
\[ \phi_{n} := \phi_{\nu_0+n-1}^{\nu_0 + n} \circ \ldots \circ
  \phi_{\nu_0 + 1}^{\nu_0+2} \circ \phi_{\nu_0}^{\nu_0+1} :
  CF(L,\ul{\Pe}^{\nu_0}) \to CF(L,\ul{\Pe}^{\nu_0+n }) \]
  for any choice of $\nu_0$. The limit
  \[ \phi = \lim_{n \to \infty} \phi_{n}:
  CF(L,\ul{\Pe}^{\nu_0}) \to \lim_{n \to \infty} CF(L,\ul{\Pe}^{\nu_0+n})= 
  CF_\br(L,\ul{\Pe}^{\infty}) \]
exists, and similarly, for
\[ \psi_{n} := \psi_{\nu_0+1}^{\nu_0} \circ \ldots \circ
   \psi_{\nu_0+n}^{\nu_0+n-1} :
  CF(L,\ul{\Pe}^{\nu_0+n}) \to CF(L,\ul{\Pe}^{\nu_0}), \]
the limit 
  \[  \psi = \lim_{n \to \infty} \psi_{n}: CF_\br(L,\ul{\Pe}^{\infty}) \to CF(L,\ul{\Pe}^{\nu_0})  \]
  exists. 
Since the composition of strictly unital morphisms is strictly unital,
the composition $\psi$ is strictly unital mod terms divisible by $q^E$
for any $E$, hence strictly unital. The limit map $\psi$ resp. $\phi$
is an \ainfty morphism whose domain resp. target is $CF_\br(L, \ul \Pe^\infty)$ 
because the composition maps in $CF(L,\ul{\Pe}^\nu)$ converge to
$CF_\br(L,\ul{\Pe}^\infty)$.  The composition maps of
$CF(L,\ul{\Pe}^\nu)$ converge to those of
$CF_\br(L,\ul \Pe^\infty)$ by the bijection of the moduli spaces of
disks in Proposition \ref{prop:breakingpert}
\eqref{part:breaking1}. Indeed, the bijection preserves the orientation
$\eps(u)$, area $A(u)$, and holonomy $\Hol_L([\partial u])$ since gluing
preserves these quantities.

We claim that $\phi$ and $\psi$ are homotopy equivalences.  First, we inductively construct
homotopies  $h_n$, $g_n$ satisfying
  \[\phi_n \circ \psi_n - \Id = m_{\cQ}^1(h_n), 
  \quad \psi_n \circ \phi_n - \Id = m_{\cQ}^1(g_n). \]
Starting from $h_n$, the homotopy $h_{n+1}$ is defined as follows:
By Proposition \ref{prop:breakingpert}
\eqref{part:breaking2},
there is a homotopy $H_{n+1}$ satisfying

\[(\phi_{\nu_0 +n}^{\nu_0+n+1} \circ \psi_{\nu_0+n+1}^{\nu_0+n}) - \Id=m^1_\cQ(H_{n+1})\]
that vanishes modulo $q^{E(n)}$ for some $E(n)>0$. Define $H_{n+1}':=L_{\phi_{\nu_0 +n}^{\nu_0+n+1}}(R_{\psi_{\nu_0+n+1}^{\nu_0+n}}(H_{n+1}))$ by composing $H_{n+1}$ with \ainfty morphisms as in \eqref{eq:lcomp}, \eqref{eq:rcomp}. Then, $H_{n+1}'$ is a homotopy from $\phi_{n+1} \circ \psi_{n+1}$ to $\phi_n \circ \psi_n$ that also vanishes  modulo $q^{E(n)}$. The required homotopy
is then defined as 
\[h_{n+1}:=H_{n+1}' \circ h_n,\]
where the composition of homotopies is as in \eqref{eq:homotopy-compose}, and 
 it has the property that
\[h_{n+1}=h_n \mod q^{E(n)}.\]
Since $E(n) \to \infty$ as $n \to \infty$  by Proposition \ref{prop:breakingpert}
\eqref{part:breaking2}, the limit $h:=\lim_{n \to \infty} h_n$ is well-defined. The homotopies $g_n$ are defined
similarly, and the limit \(g:=\lim_{n \to \infty}g_n\) is well-defined.
%
Since the \ainfty algebra $CF_\br(L,\ul \Pe^\infty)$ is a limit
  $\lim_{\nu \to \infty}CF(L,\ul{\Pe}^\nu)$ (in the sense that composition maps converge),
  and the \ainfty morphism $\psi_n$ resp. $\phi_n$ limits to $\psi$ resp. $\phi$, we conclude
  \[\phi \circ \psi - \Id = m_{\cQ}^1(h), 
  \quad \psi \circ \phi - \Id = m_{\cQ}^1(g). \]
The \ainfty homotopies $h,g$ are convergent
 since each of them is the limit of convergent homotopies.
\end{proof}

\section{Independence of perturbations in broken manifolds}
\label{sec:brokenind}

In this Section, we show that broken Fukaya algebras are independent
of perturbations up to \ainfty homotopy equivalence,
where the perturbations are not required to be gluable.
For the proof of Theorem \ref{thm:bfuk} (equivalently Proposition \ref{prop:unbreak-break}) we used $X$-cylindrical almost complex structures(Definition \ref{def:brokenJ}\eqref{part:brokenJ-cyl}) on broken manifolds, which were gluable. As a consequence of the result of this Section,  broken Fukaya algebras can be defined using tamed almost complex structures on broken symplectic manifolds, called $\XX$-cylindrical almost complex
structures (Definition \ref{def:omxxcyl}). 
Therefore, in applications, we need not distinguish between the spaces $X_P$ 
and $X_P^\om$, and the superscript $\om$ may
be dropped.

In order to interpolate between perturbations taking values in
$X$-cylindrical and $\XX$-cylindrical almost complex structures, we
consider a generalized kind of almost complex structures called
\em{$\cG$-cylindrical}. The underlying collection of inner products,
denoted by
\[g:=\{g_Q^P: \t_Q \times \t_Q \to \R:Q \subset P, \enspace P \in
  \PP^0\}\]
satisfies the \hyperref[item:restrict]{(Restriction)} condition, and
the collection of connection one-forms $(\alpha_Q^P)_{Q \subset P}$
satisfies the consistency condition \eqref{eq:consis} with respect to
$g$.

As in the unbroken case, independence of broken Fukaya algebras of
perturbation data is proved by constructing an \ainfty morphism by
counts of quilted broken holomorphic disks which we now define.  Given
regular perturbation data $\ul \Pe_0$, $\ul \Pe_1$ on $\XX$ that
take values in $\cG$-cylindrical almost complex structures, a
perturbation morphism $\ul \Pe^{01}$ connecting $\ul \Pe_0$ and
$\ul \Pe_1$ is as in the unbroken case in Definition \ref{def:pert-morph}
with the following differences:   The domain-dependent almost complex structures are on $\XX$, and $\ul \Pe^{01}$ is adapted to a path of stabilizing pairs
$(\JJ_t, \DD_t)_{t \in [0,1]}$ on $\XX$.

\begin{definition} {\rm(Quilted broken holomorphic disks)}
  Given a perturbation morphism
  $\ul \Pe^{01}=(\Pe^{01}_\Gamma)_\Gamma$ on $\XX$, a \em{quilted
    broken holomorphic disk} consists of a quilted treed disk $C$ of
  type $\Gamma$ (Definition \ref{def:qdisk}), a tropical structure
  $\cT$ on $\Gamma$, and a $\Pe^{01}_\Gamma$-holomorphic map
  $u:C \to \XX$ with tropical structure $\cT$, boundary mapping to the
  Lagrangian submanifold $L \subset X_{P_0}$, and the map being
  adapted to the family of broken divisors $\{\DD_t\}_t$ as in
  Definition \ref{def:hol-quilt}.
\end{definition}

The following Proposition is analogous to Proposition
\ref{prop:samedegree}, and the proof of that Theorem carries over.
\begin{proposition}
  \label{prop:pd3} Given regular perturbation data
  $\ul \Pe^0$, $\ul \Pe^1$, there is a comeager set
  of regular perturbation morphisms extending $\ul \Pe^0$,
  $\ul \Pe^1$. Any such perturbation morphism induces a convergent unital \ainfty
  morphism
  %
  \[ \phi=(\phi^r)_{r \ge 0} : CF^\br(L,\ul \Pe^0) \to
    CF^\br(L,\ul \Pe^1)\]
%
defined by counts of quilted broken holomorphic disks. The \ainfty morphism $\phi$
is a \ainfty homotopy equivalence
  and has zero-th composition map $\phi_0(1)$ with positive
  $q$-valuation, that is
  %
  $\phi^0(1) \in \Lam_{>0}\bran{\hat \cI(L)}$.
  %
\end{proposition}

\chapter{Potentials for semi-Fano toric surfaces}

In this concluding Chapter,
we return to fill in the technical gaps in the 
computations of disk potentials in semi-Fano toric manifolds in
Chapter  \ref{chap:apps}. The perturbations used to compute the disk
potentials in the cubic surface and the second Hirzebruch surface
depend not  just on the domain treed curve, 
but also on the tropical graph.  We rigorously define this new kind
of perturbation, called \em{split perturbation}, and show that the
broken Fukaya algebras it produces are homotopy equivalent to the ones
we have defined so far.  We also prove a technical result used in the
computations of disk potentials in semi-Fano toric manifolds in
Chapter  \ref{chap:apps}, namely that, the multiple cut of a semi-Fano toric
manifold, under some hypothesis, has the same disk potential as the
original smooth manifold.

\section{Split perturbations} 
\label{sec:decouple}

A split perturbation depends not just on the domain
curve but also on an accompanying tropical structure on it, recorded
by the \em{base tropical graph}.
The domain type for maps for split perturbations, denoted by $\tGam$, 
may have additional vertices than those in the base tropical graph
$\Gamma$. However, for generic perturbations, for maps of index
$\leq 1$, the domain type $\tGam$ is equal to the base tropical graph $\Gamma$.
\begin{definition} \label{def:basedcurve} {\rm(Based curves)} A
  \em{curve with base type} consists of 
  \begin{enumerate}
    \index{Disk! Treed disk with base type}
  \item a stable treed disk $C$ of type $\tGam$,
  \item a \em{base tropical graph} $\bGam$,
    \index{Base tropical graph}
    \index{Tropical graph!Base tropical graph}
  \item and an edge collapse morphism $\kappa:\tGam \to \bGam$ that
    necessarily collapses all disk edges $e \in \Edge_\white(\tGam)$
    and treed segments $T_e$, $e \in \Edge_\white(\tGam)$, and
    possibly some interior edges. (The map $\kappa$ is an edge
    collapse of graphs, and not of tropical graphs, because $\tGam$
    does not have a tropical structure.)
  \end{enumerate}
  The \em{type} of such a curve consists of the datum
  \index{Combinatorial type! of a curve with base} \index{Based!curve
    type}
  \index{Moduli space!of treed disks with base $\M_{\tGam,\Gamma}$}
  $\tGam \xrightarrow{\kappa} \bGam$ together with a root vertex 
   $v_\white\in \Ver(\Gamma)$. 
   The \em{ root vertex}
    $v_\white\in \Ver(\Gamma)$ of the base tropical graph is the vertex 
    containing the disk
    components of $\tGam$, that is, $\kappa(v)=v_{\white}$ for all
    $v \in \Ver_\white(\tGam)$.
    
  The moduli space of stable treed curves of type $\tGam$ with base type
  $\bGam$ is denoted by $\M_{\tGam,\Gamma}$, or simply $\M_{\tGam}$
  when the context allows it.  The topology on the space of based
  curves  is the standard topology on the space of stable curves with the
  additional axiom that the base tropical graph is preserved in the
  limit.  This ends the Definition.
\end{definition}

%
Morphsisms on treed disk types (Definition \ref{def:pertops}) extend naturally to morphisms of treed disk types with base.
\begin{definition} {\rm (Based graph morphisms)}
  \label{def:bgm}
  The following morphisms are defined on types of curves with a base :
  \index{Based!graph morphisms}
  \index{Collapsing an edge!for a based curve type}
    \index{Making an edge length finite/non-zero!for a based curve type}
  \begin{enumerate}
  \item
    A based curve type $\tGam_1 \xrightarrow{\kappa_1} \Gamma$ is obtained by \em{(Collapsing edges)}
    resp. \em{(Making an edge length/weight
      finite/non-zero)} in the based curve type $\tGam_0 \xrightarrow{\kappa_0} \Gamma$ if there is a (Collapsing edges) resp. (Making an edge length/weight
    finite/non-zero) morphism $\tGam_0 \xrightarrow{\tilde\kappa} \tGam_1$ of treed curve types (see Definition \ref{def:coherent}) and $\tilde \kappa \circ  \kappa_0=\kappa_1$.
Note that the (Collapsing edges) morphism does not collapse an edge $e$ in $\tGam_0$ that is present in the base graph $\Gamma$.
\item \index{Cutting an edge!for a based curve type}
  A type with base
  $\tGam \xrightarrow{\kappa} \bGam$ is obtained by \em{(Cutting an
  edge)} in $\tGam' \xrightarrow{\kappa'} \bGam'$ if
  \begin{itemize}
  \item $\tGam$ is obtained by cutting an
    $e=(v_+,v_-) \in \Edge_{\white,-}(\tGam')$ containing a breaking,
  \item $\bGam'$ is obtained by identifying the vertices
    $\kappa(v_+)$, $\kappa(v_-)$ in $\bGam$.
  \end{itemize}
  \end{enumerate}
\end{definition}

\index{Background almost complex structure!for a split perturbation}
In a split perturbation datum (Definition \ref{def:basedpert}), the
stabilizing divisor and background almost complex structure can be
chosen to be different for curve components corresponding to different
vertices of the base tropical graph.  However,
the set of stabilizing pairs
ranging over all based curve types are required to satisfy the following coherence
condition, whose necessity arises from the coherence condition on split perturbation data. 
\begin{definition}{\rm(Coherence for stabilizing pairs)} 
  \label{def:trg-iso}
  \begin{enumerate}
  \item {\rm(Inclusion of base tropical graphs)} An \em{inclusion} 
    $i:\Gamma_0 \to \Gamma_1$ between base tropical graphs
    is a map of graphs that
    preserves the root vertex, edge directions, and polytope assignments on vertices.
  \item {\rm(Coherence condition)}
    A collection of stabilizing pair assignments 
    \[(\ul \JJ, \ul \DD)=(\JJ_v, \DD_v)_{v \in \Ver(\Gamma),\Gamma}\]
    for all vertices $v$ in all base tropical graphs $\Gamma$ is \em{coherent}, if
    for any inclusion $i:\Gamma_0 \to \Gamma_1$ of base tropical graphs,
    \[(\JJ_{i(v)}, \DD_{i(v)})=(\JJ_v, \DD_v) \quad \forall v \in \Ver(\Gamma_0).\]
  \end{enumerate}
\end{definition}
The coherence condition in Definition \ref{def:trg-iso} implies that 
the stabilizing pair for a root vertex is the same
for all base tropical graphs $\Gamma$.
Furthermore, 
$(\ul \JJ, \ul \DD)$ satisfies the above coherence condition for stabilizing
pairs if the condition is satisfied for inclusions of base tropical graphs
whose root vertex has a single edge incident on it.

\begin{definition} \label{def:basedpert}
  {\rm(Split perturbation datum)}
  \index{Perturbation! Split perturbation datum}
  A domain-dependent perturbation $\Pe_{\tGam,\Gamma}$ for a 
  based curve type $(\tGam,\Gamma)$ consists of
   a stabilizing pair $(\JJ_v, \DD_v)$ corresponding to each vertex $v$ of the base tropical graph $\Gamma$, and  perturbations 
  \[J_{\tGam,\Gamma}:\S_\tGam \to \J^\cyl_{\om_\XX, \tau}(\XX), \quad F_{\tGam,\Gamma} : \T_\tGam \to C^\infty(L,\R) \]
  such that on $\S_{\tGam,\kappa^{-1}(v)} \subset \S_\tGam$, namely the components corresponding to
  $\kappa^{-1}(v)$, $\J_{\tGam,\Gamma}$ is adapted to $(\JJ_v, \DD_v)$. Here, $\J^\cyl_{\om_\XX, \tau}(\XX)$ is the space of $\om_\XX$-tamed cylindrical almost complex structures on $\XX$ (Definition \ref{def:omxxcyl} \eqref{part:omxxcyl2}).
  %
  %
  A \em{coherent split perturbation datum}
  adapted to a coherent collection $(\ul \JJ, \ul \DD)$ of stabilizing pairs 
  is a collection
  \[\ul \Pe=(\Pe_{(\tGam,\Gamma)})_{(\tGam,\Gamma)}, \]
  consisting of a perturbation datum
  \[\Pe_{(\tGam,\Gamma)}=(J_{(\tGam,\Gamma)}, F_{(\tGam,\Gamma)}), \quad J_{\tGam,\Gamma}:\S_\tGam \to \J^\cyl_{\om_\XX, \tau}(\XX), \quad F_{\tGam,\Gamma} : \T_\tGam \to C^\infty(L,\R) \]
  for each  based curve type $(\tGam,\Gamma)$, that is coherent under the based graph morphisms of (Collapsing edges), (Making an edge length or weight finite/non-zero) and (Cutting edges) from Definition \ref{def:bgm}, and the (Locality axiom) from Definition \ref{def:coherent}. 
\end{definition}

The coherence condition for split perturbation data necessitates the coherence condition for stabilizing pairs as follows: Let $\tGam \xrightarrow{\kappa} \Gamma$ be a based curve type, and suppose cutting a broken boundary edge $e$ in $\tGam$ produces types whose base tropical graphs are $\Gamma_0$, $\Gamma_1$. Coherence of perturbations under cutting the edge $e$ requires that stabilizing pairs satisfy the coherence condition corresponding to the inclusions $\Gamma_0 \to \Gamma$ and $\Gamma_1 \to \Gamma$.

\begin{remark}
  An important difference between ordinary perturbations (from
  Definition \ref{def:coherent}) on $\XX$ and the split perturbations
  in the above Definition \ref{def:basedpert} is that the split
  perturbation datum $\Pe_{\tGam,\Gamma}$ is not required to be
  coherent under collapsing edges that are present in the base
  tropical graph $\Gamma$. Indeed, for a based graph
  $\tGam \xrightarrow{\kappa} \Gamma$, the
  \hyperref[item:collapsingedgesmorphism]{(Collapsing edges)} morphism
  is only applicable on edges $e \in \Edge(\tGam) \bs \Edge(\Gamma)$.
\end{remark}

\begin{definition}\label{def:basemap}
  {\rm(Broken maps with base)}
  \index{Based! Broken map with base}
  \index{Map! Broken map with base}
  Let $\ul \Pe$ be a coherent split perturbation datum.
  \begin{enumerate}
  \item \label{part:basemap1}  A \em{$\ul \Pe$-holomorphic broken map with base} consists of
    \begin{itemize}
    \item a based type $\kappa:\tGam \to \Gamma$,
    \item and a broken map $u$ whose domain is of type $\tGam$ and $u$
    is $\Pe_{\tGam,\Gamma}$-holomorphic in the ordinary sense (as in
    Definition \ref{def:pdisks}).
    Furthermore, if
    $\tGam_\tr$ is the tropical graph of $u$ then $\kappa$ factors
    through $\tGam_\tr$ as
    \[\kappa=\kappa_1 \circ \tr, \quad \tGam \xrightarrow{\tr}
      \tGam_\tr \xrightarrow{\kappa_1} \Gamma,\]
    where $\kappa_1$ is a tropical edge collapse.
  \end{itemize}
\item {\rm(Type)} \label{part:basemap2} The \em{type} of an 
  broken map with base consists of the tropical graph $\tGam$, the
  tropical edge collapse map $\tGam \to \bGam$, and the homology and
  tangency data for the map $u$ (as in the type of a broken map, see
  Definition \ref{def:type-broken}). Whenever it is possible, the base
  tropical type $\bGam$ is suppressed in the
  notation. \index{Combinatorial type! of broken maps with
    base}
    \index{Tropical symmetry! for a broken map with base}
    \index{Based! broken map type}
  \item {\rm(Tropical symmetry)} For a type $\tGam$ of a broken map
    with base, the tropical symmetry group $T_\trop(\tGam)$ is the
    symmetry group obtained by viewing $\tGam$ as a type of broken map
    (by Definition \ref{def:tsym}) and forgetting the base type
    $\Gamma$.
  \item {\rm(Rigidity)} \label{part:basemap3} \index{Rigid! broken map
      with base} The type $\tGam \to \bGam$ of a broken map with base
    is \em{rigid} if $\tGam$ is rigid as a type of broken map.  Thus
    for a rigid type, the base tropical graph $\Gamma$ is rigid, the
    tropical graph $\tGam_\tr$ of the map is equal to $\Gamma$, and
    the morphism $\tGam \to \bGam$ only collapses treed components and
    boundary nodes.
  \end{enumerate}
  
\end{definition}

Split perturbations can be used to define broken Fukaya algebras.
Given a coherent split perturbation datum $\ul \Pe$ on $\XX$,  define the  Fukaya algebra 
as the graded vector space
\[CF_\br(L,\ul \Pe):=CF^{\on{geom}}(L) \oplus \Lam_{\geq 0} x^{\greyt}[1]
  \oplus \Lam_{\geq 0} x^{\whitet}\]
  \index{Fukaya algebra!Broken, with split perturbations $CF_\br(L,\ul \Pe)$}
with composition maps
\begin{equation}\label{eq:mnbase}
  m^{d}_\br(x_1,\dots,x_{d})= \sum_{x_0,u \in {\tM_{\tGam,\Gamma}(\ul \Pe,\ul \eta, \ul{x})_0}} w_s(u) x_0,
\end{equation}
where $u$ ranges over all rigid types $(\tGam,\Gamma)$ of broken maps
with base that have $d$ inputs (see Definition \ref{def:basemap}
\eqref{part:basemap3} for rigidity), and
\begin{equation}
  \label{eq:wtwsubase}
  w_s(u):=(-1)^{\heartsuit}(d_\black(\Gam)!)^{-1}  \Hol_L([\partial u]) \eps(u)
  q^{A(u)},
\end{equation}
where the symbols in \eqref{eq:wtwsubase} are as in \eqref{eq:wtwu},
and in particular the orientation is computed in the same way as for
broken map without base, see Remark \ref{rem:orientmap}.  As in the
case of broken maps (without base), for a one-dimensional component of
the moduli space, the configurations with a boundary edge of length
zero constitute a fake boundary, whereas those with a broken boundary
edge constitute the true boundary of the moduli space.  Setting the
counts of the maps occurring in the true boundary to zero yields the
\ainfty associativity relations, and we conclude that $CF(L,\ul \Pe)$
is an \ainfty algebra.
 
Broken Fukaya algebras defined by counts of broken maps with base are
independent of the choice of perturbation up to homotopy equivalence.
The proof is similar to the proof of independence of perturbations for
unbroken Fukaya algebras shown in Proposition \ref{prop:samedegree}.
Assuming that the stabilizing divisors have the same degree, the proof
is by constructing an \ainfty morphism by counts of quilted broken
holomorphic disks with base.  These are holomorphic disks with base,
with the additional datum of a quilting circle on disks.  Split
perturbations differ in the sense that they are not coherent under
collapsing the edges in the base tropical graph. The extension is
straightforward because quilting phenomena take place on disk
components, and all disk components are collapsed into a single vertex in the base tropical graph. 
We obtain the following result, which is an extension of Proposition
\ref{prop:pd3}.

\begin{proposition}\label{prop:samedegree-base}
  {\rm(Independence of perturbations, split version)} Suppose
  $\ul \Pe^0$, $\ul \Pe^1$ are regular split perturbation data that
  are defined using coherent collections of stabilizing pairs $(\ul \JJ^0,\ul \DD^0)$ and $(\ul \JJ^1, \ul \DD^1)$,
  which are connected by a path of a coherent collection of stabilizing pairs
  $\{(\ul \JJ^t,\ul \DD^t)\}_{t \in [0,1]}$. 
  There exists a coherent perturbation
  datum $\ul \Pe^{01}$ which induces a convergent unital \ainfty
  morphism (as in \eqref{eq:unital-morph})
  \[\phi_{01} : CF_\br(L,\ul \Pe^0) \to CF_\br(L,\ul \Pe^1) \]
  which is a homotopy equivalence. 
\end{proposition}

In fact, Fukaya algebras defined using ordinary perturbation data are
\ainfty homotopy equivalent to those defined using split perturbation
data:

\begin{proposition}
  Let $\ul \Pe$ be a regular coherent perturbation datum for the
  broken manifold $(\XX,L)$ and let
  $\ul \Pe'=(\ul \Pe'_\Gamma)_\Gamma$ be a collection of regular
  coherent split perturbation data for all rigid tropical graphs
  $\Gamma$.
Then the broken Fukaya algebras $CF_\br(L,\ul \Pe)$, $CF_\br(L,\ul \Pe')$ are homotopy equivalent.
\end{proposition}

\begin{proof}
  The result follows from the homotopy equivalence of any two broken
  Fukaya algebras defined by split perturbations, and the fact that a
  non-split perturbation may be regarded as a split perturbation in
  the following way: Given a perturbation $\ul \Pe$ without base, we
  define a split perturbation $\ul \Pe'$ by forgetting the base
  tropical graph. The resulting curve counts defining the composition
  maps are the same for $\ul \Pe$ and $\ul \Pe'$. Indeed, a
  $\ul \Pe$-holomorphic map of type $\Gamma$ corresponds to a unique
  $\ul \Pe'$-holomorphic map of type $\Gamma \xrightarrow{\tr} \Gamma_\tr$
  where $\tr$ is the tropicalization map; the uniqueness follows from
  the fact that since $\Gamma$ is rigid, there is no other tropical
  graph that can be obtained by edge collapses.
\end{proof}

\section{Potentials for semi-Fano toric surfaces}

Recall that in Section \ref{sec:cubic-intro} we counted disks in a
cubic surface by deforming the symplectic form on the cubic surface to
that of a \em{semi-Fano} toric surface.  \index{Semi-Fano toric surface}
A complex manifold $X$ is semi-Fano if the first Chern
number $c_1(TX|S)$ is non-negative on any holomorphic curve
$S \subset X$.  In this section, we show that the disk potential of a
toric Lagrangian in a semi-Fano toric surface is well-defined if we
use a perturbation datum for which the almost complex structure is
close enough to a divisor-preserving almost complex structure -- these
are almost complex structures for which the torus-invariant divisors
are holomorphic. We also show that under some homological hypotheses,
the potential is preserved by a multiple cut on a semi-Fano toric
surface. These results were used in 
 Sections \ref{sec:a1} and 
\ref{sec:cubic-intro}
to obtain the disk potential by counting broken
disks.

\begin{definition}
  Let $X$ be a toric symplectic manifold.  A tamed almost complex
  structure $J$ on $X$ is called \em{divisor-preserving} if all toric
  divisors $Y$ of $X$ are $J$-holomorphic, that is, $J(TY)=TY$.
\end{definition}

%
We may associate a disk potential $W \in \Lam_{\geq 0}$ to a non-regular almost complex structure $J_0$,
if $W$ is the disk potential for any regular perturbation taking values in a small enough neighborhood of $J_0$.
\begin{definition}\label{def:nonregpot}
  {\rm(Potential for a non-regular almost complex structure)} Suppose
  $J_0$ is a compatible almost complex structure on a symplectic
  manifold $(M,\om)$ and $L \subset M$
  is
  a Lagrangian with a brane
  structure.  Let $\ul F=(F_\Gamma)_\Gamma$ be a coherent regular
  domain-dependent Morse function on $L$. The manifold $(X,J_0,L)$
  with Morse datum $\ul F$ has potential
  \begin{equation}
    \label{eq:J0pot}
    W_0(J_0) \in \Lam_{>0}  
  \end{equation}
  if the following holds: Given an area level $E_0$, there exists
  $\eps>0$ such that for any coherent regular perturbation
  $\ul \Pe=(\ul J, \ul F)$ such that $\ul J$ maps to an
  $\eps$-neighborhood of $J_0$,
  \[m^0_{CF(L,\ul \Pe)}(1)= W_0 1^\blackt \mod q^{E_0}. \]
\end{definition}

\begin{proposition}{\rm(Potential on semi-Fano toric manifolds)}
  \label{prop:semiFanopot}
  Let $X$ be a semi-Fano symplectic toric manifold and $L \subset X$
  is a toric Lagrangian.  Let $J_0$ be a divisor-preserving almost
  complex structure, and let $\ul F$ be a generic coherent
  domain-dependent Morse function on $L$.  The potential for the
  manifold $(X,J_0,L)$ is well-defined in the sense of Definition
  \ref{def:nonregpot}.
\end{proposition}
\begin{proof}
As a first step, we define the element 
$m^0_{J_0}(1)  \in \Lam_{\geq 0} \bran{\Crit(F)}$ for the
  divisor-preserving almost complex structure $J_0$.
  \begin{claim}\label{claim:m0}
    There is an element $m^0_{J_0}(1) \in \Lam_{\geq 0} \bran{\Crit(F)}$ such
    that for any area level $E_0$ there exists $\eps>0$ such that for
    any coherent regular perturbation $\ul \Pe$ that is $\eps$-close
    to $J_0$,
    \[m^0_{CF(L,\ul \Pe)}(1)= m^0_{J_0}(1)  \mod q^{E_0}.\]
  \end{claim}
  \begin{subproof}
  To prove the Claim, let us assume the opposite, that is,  there is
  a sequence of coherent regular perturbations $\ul \Pe^\nu$ that
  converge uniformly to the constant $J_0$, and such that
  \[m^0_{CF(L,\ul \Pe^\nu)}(1) \neq m^0_{CF(L,\ul \Pe^{\nu+1})}(1) \mod
    q^{E_0}\]
  for all $\nu$. Choose a regular perturbation morphism
  \[\ul \Pe^{\nu,\nu+1}=(\ul J^{\nu,\nu+1}, \ul F)\]
  with end points $\ul \Pe^\nu$, $\ul \Pe^{\nu+1}$, such that the
  sequence $\ul J^{\nu,\nu+1}$ converges uniformly to the constant
  $J_0$ as $\nu \to \infty$. Since the potentials are unequal for
  $\ul \Pe^\nu$ and $\ul \Pe^{\nu+1}$, there is a
  $\ul \Pe^{\nu,\nu+1}$-quilted disk $u_\nu$ with no inputs, having
  index zero, and area at most $E_0$, which contains a non-constant
  quilted surface component $S_v$.  Indeed, the boundary of the
  one-dimensional part of the moduli space of quilted disks with no
  inputs consists of quilted disks with a broken edge $e$ (that is,
  $\ell(e)=\infty$), and the only configurations where the map is
  constant on the quilted disk are those counted in
  $m^0_{CF(L,P^\nu)}(1)$ and $m^0_{CF(L,P^{\nu+1})}(1)$.

  We now rule out the appearance of quilted disks by a   dimension argument.   The quilted disk $u_\nu$ from the previous paragraph is of one of the forms shown in Figure \ref{fig:qbroken0}.  A subsequence of the sequence $u_\nu$ converges
  to a $J_0$-holomorphic treed disk $u_\infty$ of index $-1$. The
  limit $u_\infty$ also has a broken edge, and let $u_\infty^i$,
  $i=0,1,\cdots$, be the maps obtained by cutting the broken edges in
  $u_\infty$. One of these maps, say $u_\infty^j$, has index $-1$ and
  the map on the surface part is non-constant.  The surface component
  in $u_\infty^j$ is $J_0$-holomorphic, and so, it consists of
  \begin{itemize}
  \item disks that have positive isolated intersections with the toric
    divisors, and so, have Maslov index at least $2$;
  \item and spheres that have $c_1 \geq 0$ by the semi-Fano condition.
  \end{itemize}
  Consequently, if $u_\infty^j$ does not have inputs, the index
  $i(u_\infty^j) \geq 0$, which rules out configurations of the form
  (b) in Figure \ref{fig:qbroken0} for large enough $\nu$.  Suppose an
  input leaf of $u_\infty^j$ is the output of a treed disk
  $u_\infty^k$ as in Figure \ref{fig:qbroken0} (a).  By the same
  reasoning used for $u_\infty^j$ the surface part of $u_\infty^k$ has
  Maslov index at least two, and so, the output is necessarily the maximum
  point $x^\blackt$. Therefore, in Figure \ref{fig:qbroken0} (a), any
  input to the $(-1)$-index disk $u_\infty^{j}$ is $x^\blackt$, and
  consequently, the disk $u_\infty^j$ does not exist for degree
  reasons since the surface component has Maslov index at least $2$.
  Therefore the configuration (a) is also ruled out for large
  $\nu$. Thus, Claim \ref{claim:m0} is proved. 
\end{subproof}

  Finally,  we show that $m^0_{J_0}(1)$ is a multiple of $x^\blackt$: For
  any sequence of regular perturbations
  $\ul \Pe^\nu=(\ul J^\nu, \ul F)$ that uniformly converge to
  $(J_0,\ul F)$, and a sequence of treed disks $\ul \Pe_\nu$ disks
  $u_\nu$ with no inputs, of index zero, and uniformly bounded area,
  a subsequence converges to a $(J_0,\ul F)$-disk $u_\infty$ of index
  $0$.  By the semi-Fano condition the Maslov index of $u_\infty$ is
  at least $2$, 
  and since there are no inputs, a dimension count dictates
  that $x^\blackt$ is the only possible output.
\end{proof}

\begin{figure}[t]
 {\tiny \centering\scalebox{.8}{
\begingroup%
  \makeatletter%
  \providecommand\color[2][]{%
    \errmessage{(Inkscape) Color is used for the text in Inkscape, but the package 'color.sty' is not loaded}%
    \renewcommand\color[2][]{}%
  }%
  \providecommand\transparent[1]{%
    \errmessage{(Inkscape) Transparency is used (non-zero) for the text in Inkscape, but the package 'transparent.sty' is not loaded}%
    \renewcommand\transparent[1]{}%
  }%
  \providecommand\rotatebox[2]{#2}%
  \newcommand*\fsize{\dimexpr\f@size pt\relax}%
  \newcommand*\lineheight[1]{\fontsize{\fsize}{#1\fsize}\selectfont}%
  \ifx\svgwidth\undefined%
    \setlength{\unitlength}{197.87285614bp}%
    \ifx\svgscale\undefined%
      \relax%
    \else%
      \setlength{\unitlength}{\unitlength * \real{\svgscale}}%
    \fi%
  \else%
    \setlength{\unitlength}{\svgwidth}%
  \fi%
  \global\let\svgwidth\undefined%
  \global\let\svgscale\undefined%
  \makeatother%
  \begin{picture}(1,0.25806653)%
    \lineheight{1}%
    \setlength\tabcolsep{0pt}%
    \put(0,0){\includegraphics[width=\unitlength,page=1]{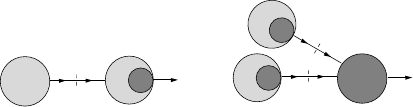}}%
    \put(0.13220209,0.18869623){\color[rgb]{0,0,0}\makebox(0,0)[lt]{\lineheight{1.25}\smash{\begin{tabular}[t]{l}(a)\end{tabular}}}}%
    \put(0.83611734,0.1995256){\color[rgb]{0,0,0}\makebox(0,0)[lt]{\lineheight{1.25}\smash{\begin{tabular}[t]{l}(b)\end{tabular}}}}%
  \end{picture}%
\endgroup%
}}
   \caption{Codimension one strata in a moduli space of quilted disks with no inputs.}
  \label{fig:qbroken0}
\end{figure}

\begin{proposition}
  For a semi-Fano toric symplectic manifold, the potential $W(J_0)$
  does not depend on the choice of divisor-preserving almost complex
  structure $J_0$.
\end{proposition}
\begin{proof}
  [Outline of proof]  Assume the opposite, so that there is a path $\{J_t\}_{t \in [0,1]}$ of divisor-preserving almost complex structures on which the terms in the potential with area level
  $\leq E_0$ are not $t$-independent.  Then, there is a sequence of
  quilted disks with area $\leq E_0$ that converge to a
  $J_t$-holomorphic disk of index $-1$ for some $t \in [0,1]$.  The
  existence of such a disk is ruled out as in the proof of Proposition
  \ref{prop:semiFanopot}.
\end{proof}

The following Proposition gives a sufficient condition under which the potential is not altered by multiple cutting of a semi-Fano toric surface.

\begin{proposition}\label{prop:semiFano-break}{\rm(When breaking does not alter the potential)}
  Let $X$ be a semi-Fano toric surface, $\PP$ a toric multiple cut,
  and $L \subset \ol X_{P_0}$ a toric Lagrangian in a component
  $\ol X_{P_0}$ of the broken manifold $\XX_{\PP}$.  Let $\JJ_0$ be a
  broken divisor-preserving tamed almost complex structure on
  $\XX$. Suppose for any broken $\JJ_0$-holomorphic disk $u$ in $\XX$,
  $I(u_{\glue}) \geq 2$.  Then the potential for $(\XX,\JJ_0,L)$ is
  well-defined and is the same as the toric potential for $X$ equipped
  with a divisor preserving almost complex structure.
\end{proposition}
\begin{proof}
We argue by contradiction, using Gromov compactness.  
  Let $J_0^\nu$ be the divisor preserving almost complex structure
  obtained from $\JJ_0$ by gluing the cylindrical ends of $\XX$.
  Following the proof of Proposition \ref{prop:semiFanopot}, if the
  Proposition were not true, there would exist
  \begin{itemize}
  \item a sequence of regular perturbations $\ul \Pe_\nu$ on $X^\nu$
    that converge to the constant broken almost complex structure
    $J_0$,
  \item perturbation morphisms $\ul \Pe^{\nu,\nu+1}$ that also
    converge to $J_0$,
  \item  a sequence $u_\nu$ of
    $\ul \Pe^{\nu,\nu+1}$-quilted disks of index zero, uniformly
    bounded area, each having a broken edge $e$ (that is, $\ell(e)=\infty$), and a quilted surface component on which the map is non-constant.
  \end{itemize}
  Consequently, after passing to a subsequence, the sequence
 $\{u_\nu\}_\nu$ converges to a $\JJ_0$-holomorphic broken treed disk
  $u_\infty$ of index $-1$.  Indeed, this is the only case to
  consider, since in the proof of Proposition \ref{prop:semiFanopot}
  we have already ruled out quilted disks of the above form that 
  converge to a $J^\nu_0$-holomorphic disk for a finite $\nu$.  The
  quilted disks in the sequence $\nu$ are of one of the two forms
  shown in Figure \ref{fig:qbroken0}. Copying the arguments from the proof of
  Proposition \ref{prop:semiFanopot}, each of the cases is ruled out
  using the fact that for a broken $\JJ_0$-holomorphic disk $u$ 
  cannot
  have
  index $i^\br(u)=-1$ as follows: By Proposition \ref{prop:expdim} \eqref{part:glue-dim}, the glued type has the same index, that is, 
  $i^\br(u)=i(u_{\on{glue}})=-1$, which is not possible, since by the hypothesis, 
 the Maslov
  index $I(u_{\on{glue}})$ is at least two.
\end{proof}

\begin{figure}[h]
  \centering \scalebox{.8}{
\begingroup%
  \makeatletter%
  \providecommand\color[2][]{%
    \errmessage{(Inkscape) Color is used for the text in Inkscape, but the package 'color.sty' is not loaded}%
    \renewcommand\color[2][]{}%
  }%
  \providecommand\transparent[1]{%
    \errmessage{(Inkscape) Transparency is used (non-zero) for the text in Inkscape, but the package 'transparent.sty' is not loaded}%
    \renewcommand\transparent[1]{}%
  }%
  \providecommand\rotatebox[2]{#2}%
  \newcommand*\fsize{\dimexpr\f@size pt\relax}%
  \newcommand*\lineheight[1]{\fontsize{\fsize}{#1\fsize}\selectfont}%
  \ifx\svgwidth\undefined%
    \setlength{\unitlength}{114.92474533bp}%
    \ifx\svgscale\undefined%
      \relax%
    \else%
      \setlength{\unitlength}{\unitlength * \real{\svgscale}}%
    \fi%
  \else%
    \setlength{\unitlength}{\svgwidth}%
  \fi%
  \global\let\svgwidth\undefined%
  \global\let\svgscale\undefined%
  \makeatother%
  \begin{picture}(1,0.36744514)%
    \lineheight{1}%
    \setlength\tabcolsep{0pt}%
    \put(0,0){\includegraphics[width=\unitlength,page=1]{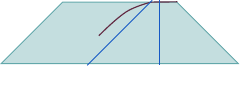}}%
    \put(0.32657437,0.01616046){\color[rgb]{0.08627451,0.34901961,0.77254902}\makebox(0,0)[lt]{\lineheight{1.25}\smash{\begin{tabular}[t]{l}$Y_1$\end{tabular}}}}%
    \put(0.63248108,0.02839636){\color[rgb]{0.08627451,0.34901961,0.77254902}\makebox(0,0)[lt]{\lineheight{1.25}\smash{\begin{tabular}[t]{l}$Y_2$\end{tabular}}}}%
    \put(0,0){\includegraphics[width=\unitlength,page=2]{h2-fail.pdf}}%
  \end{picture}%
\endgroup%
}
  \caption{A broken disk with Maslov index $0$ in the second Hirzebruch surface with the multiple cut $\PP_1$ from Example \ref{ex:h2}.}
  \label{fig:h2-fail}
\end{figure}

\begin{remark}
  \label{rem:ctrh2}
  {\rm(A counter-example)} In Example \ref{ex:h2} we saw that there is
  a multiple cut $\PP_2$ on the second Hirzebruch surface $X:=F_2$ for
  which the potential of the broken manifold $\XX_{\PP_2}$ differs
  from that of the unbroken one $X$.  The broken manifold
  $\XX_{\PP_2}$ does not satisfy the hypothesis of Proposition
  \ref{prop:semiFano-break} because there is a broken disk $u$ (shown
  in Figure \ref{fig:h2-fail}) whose glued Maslov index $I(u_\glue)$
  is $0$. The map $u$ contains a disk of Maslov index $4$ and two
  spheres with self-intersection $-1$.
\end{remark}

In the following result, we verify that the multiple cut $\PP$ on the
cubic surface given in Figure \ref{fig:stretchingcubic} satisfies the
hypothesis of Proposition \ref{prop:semiFano-break}, and use it to
conclude that the multiple cut preserves the potential of the toric
Lagrangian in the cubic surface.

\begin{figure}[t]
\begin{center}
\scalebox{.8}{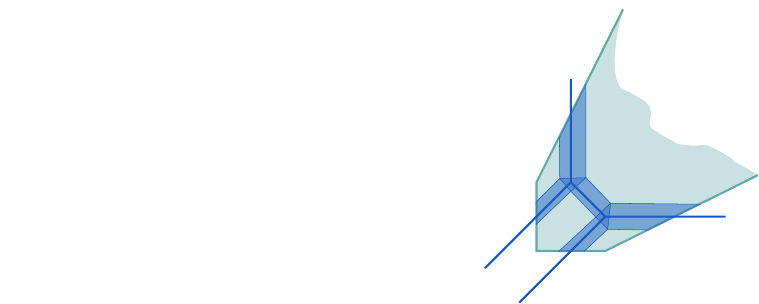} \end{center}
  \caption{Left: Multiple cut $\PP$ on the cubic surface $X$. Right: Components of the broken cubic surface $\XX$.}
  \label{fig:stretchingcubic}
\end{figure}

\begin{proposition}\label{prop:cubic0disks}
  In the multiply cut cubic surface $\XX$ as in Figure
  \ref{fig:stretchingcubic} equipped with a divisor-preserving
  domain-dependent almost complex structure $\JJ_0$ and a toric
  Lagrangian $L \subset \ol X_{P_0}$, for any $\JJ_0$-holomorphic
  broken disk $u$, the glued Maslov index $I(u_{\glue})$ is at least
  $2$.  Consequently, the multiple cut of Figure
  \ref{fig:stretchingcubic} preserves the disk potential of the toric
  Lagrangian $L$.
\end{proposition}

\begin{proof}
  The proof is combinatorial and proceeds by analyzing various cases.
  We first consider disks whose tropical nodes do not map to orbifold
  singularities, and prove that their glued Maslov index is at least
  $2$.  For a rigid broken map $u : C \to \XX$ of type $\Gamma$, the
  Maslov index of the glued type $\Gamma_\glue$ is
  \begin{equation}
    \label{eq:mas-formula}
    I(\Gamma_\glue)=\sum_{v \in \Ver(\Gamma)}\ol I(u_v), \quad \ol
    I(u_v):=I(u_v) - 
    \sum_{e \in \Edge_{\trop}(\Gamma): e \ni v}
    2\mu_{e,v},   
  \end{equation}
  where $\mu_{e,v}$ is the sum of the intersection multiplicities with
  all the relative divisors at the lift of the node $w_e$ on $C_v$,
  and $I(u_v)$ is the Maslov index of the disk/sphere $u_v$.  Indeed,
  \eqref{eq:mas-formula} follows from \eqref{eq:maslovglue} by
  observing that $\ol I(u_v)=I_\adj(u_v) - 2|\Edge_\trop(\Gamma_v)|$
  where $\Edge_\trop(\Gamma_v)$ is the set of tropical edges incident
  on $v$ (see also the related Remark \ref{rem:on-adj}).  Via a
  case-by-case analysis of the components of the map, we will show
  that $I(\Gamma_\glue) \geq 2$. In the sequel, we say that
  $v \in \Ver(\Gamma)$ is a \em{descendent} of $w \in \Ver(\Gamma)$ if
  there is an edge between $v$ and $w$, and $v$ is further from the
  root than $w$.

  \textsc{ Case 1}: If the map $u_v$ does not map to a toric divisor
  of $\ol \XC_{P(v)}$ then $\ol I(u_v) \geq 0$. Indeed, in this case
  $\ol I(u_v)$ is equal to the number of intersections of the map
  $u_v$ with non-boundary toric divisors counted with multiplicity,
  which is a non-negative integer. An inspection of the multiply cut
  cubic surface $\XX$ tells us that this case covers all maps $u_v$
  where $\codim (P(v))>0$.

  \textsc{ Case 2}: Suppose $P(v)$ is top-dimensional,
  $P(v) \neq P_0$, and $u_v$ maps to a non-boundary toric divisor
  $Y \subset \ol X_{P(v)}$. Then $Y$ is a $(-1)$-curve and one of the
  ends of $Y$ intersects a non-boundary toric divisor
  $Y_1 \subset \ol X_{P(v)}$. Therefore, $\ol I(u_v)=0$.
    
  \textsc{ Case 3}: Suppose $P(v)=P_0$ and $u_v$ maps to a long
  divisor, say $D_1'$. In this case $\ol I(u_v)$ is negative, but we
  show that its contribution is cancelled by positive contributions
  from some descendent vertices of $v$.  For the moment, we assume
  $u_v$ is a simple cover of $D_1'$. Then $v$ has a descendent $v_1$
  such that
  \begin{center}
    (Case A) $P(v_1)=P_1$ or $P_{01}$ \quad or \quad (Case B)
    $P(v_1)=P_4$ or $P_{04}$.
  \end{center}
  Since Cases (A) and (B) are symmetric, we only consider (A).  If
  $P(v_1)=P_1$, $u_{v_1}$ intersects either $D_1''$ or $E_1'$, both of
  which are non-relative toric divisors of $\ol X_{P_1}$; and if
  $P(v_1)=P_{01}$, $u_{v_1}$ has an intersection with the thickening
  of $D_1' \cap D_1''$, which is a non-boundary toric divisor of
  $\ol \XC_{P_{01}}$. In both cases $I(u_{v_1})=2$, which cancels the
  negative contribution $I(u_v)=-2$. In case $u_v$ is a $k$-cover of
  $D_1'$, then a similar cancellation argument applies, using
  descendent vertices of $v$.

  So far, we have shown that the sum of $\ol I(u_v)$ over all vertices
  is non-negative. We now show that the sum is positive.  Consider a
  disk component $v \in \Ver_\white(\Gamma)$ on which the map is
  non-constant.  The disk $u_v$ either intersects a long divisor, in
  which case $\ol I(u_v) \geq 2$; otherwise it intersects the relative
  divisor, say $\ol X_{P_{01}}$. In the latter case, there is a
  descendent vertex $v_1$ of $v$ for which $\ol I(u_{v_1}) \geq 2$
  using the same argument as in Case 3. This finishes the proof in the
  case when the tropical nodes of the broken disk $u$ do not map to
  orbifold singularities.

  In the case when a tropical node on a component $u_v$ maps to an
  orbifold singularity in $\ol \XX_{P(v)}$, we consider a toric
  resolution $\widetilde \XX_v$ of $\ol \XX_{P(v)}$ and a lift
  $\tilde u_v : C_v \to \widetilde \XX_v$. The formula
  \eqref{eq:mas-formula} is applicable if the Maslov index $I(u_v)$
  and the sum of intersection multiplicities $\mu_{e,v}$ are replaced
  by the corresponding quantities for the lift $\tilde u_v$. The other
  arguments in the proof now carry over. In fact, the components $u_v$
  mapping to orbifold points are covered by Case 1, since in the
  multiply cut cubic surface, toric divisors in $\XB_{P(v)}$ do not
  contain orbifold singularities in their closure.

  We have shown that in the multiply cut cubic surface $\XX$, the
  glued Maslov index of any broken disk is at least $2$. Therefore, we
  may apply Proposition \ref{prop:semiFano-break} on $\XX$ and
  conclude that the potential of the toric Lagrangian $L$ in the cubic
  surface is preserved under the multiple cut $\PP$.  This finishes
  the proof of Proposition \ref{prop:cubic0disks}.
\end{proof}

\begin{remark}\label{rem:otherP0}
  A curious reader may wonder whether the potential has the same terms
  if the Lagrangian is in some other top-dimensional piece of the
  multiply cut cubic surface.  The potential has the same terms as the
  unbroken case when $L$ is a toric Lagrangian in any top-dimensional
  piece $\ol X_{P_i}$, $i=0,\dots,9$ in the broken manifold. Indeed,
  the proof of Proposition \ref{prop:cubic0disks} can be replicated,
  and there are no disks of Maslov index less than $2$ in each of the
  cases. However, if $L \subset \ol X_{P_i}$, $i \neq 0$, the disks
  cannot
  be counted in a straightforward way, since some of the
  broken disks contain components that are multiple covers of
  $(-1)$-spheres.  For example, for the disk $u$ in Figure
  \ref{fig:wall1} the gluing $u_\glue$ has homology class
  $\delta_{E_5}$. The component $u_{v_3}$ in $u$ is homologous to a
  double cover $2[D_1']$; the count of such a configuration is $-\hh$
  by the Graber-Pandharipande multiple cover formula \cite{bryan:lgw}
  (and for which we provide an alternate proof in \cite{vw:torus}
  without using virtual localization). Therefore, after
  accounting for the multiplicity of the edge $e$, the map $u$
  contributes $-1$ to the count.  There is another broken disk $u'$ in
  the same homology class, where the sphere $u_{v_3}$ is replaced by
  two spheres $u_{v_3}$, $u_{v_3'}$, each of homology class $[D_1']$
  (see Figure \ref{fig:wall2}); there are two broken maps of this type
  obtained by interchanging the sub-trees attached at the vertices
  $v_3$, $v_3'$, adding $+2$ to the curve count. Together, the two
  types of maps give a contribution of $+1$, which is the contribution
  of a disk of class $\delta_{E_5}$ in the unbroken case.
\end{remark}
\begin{figure}[h]
  \centering\scalebox{.8}{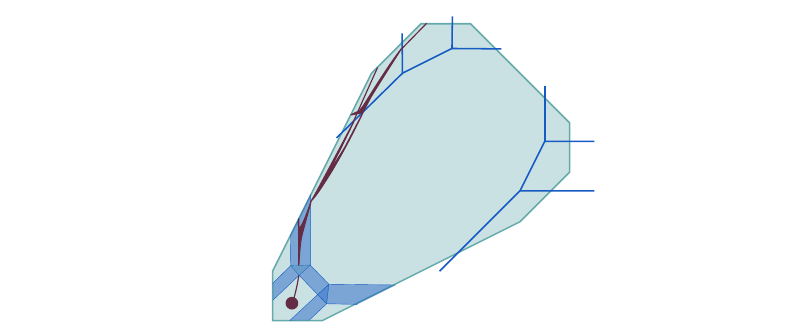}
  \caption{A broken map $u$ whose gluing is in the class $\delta_{E_5}$, and which contains a multiple cover. The node $(v_1,v_2)$ (whose $v_1$-end maps to an orbifold singularity) is similar to the node in Example \ref{ex:adj-mas} \eqref{part:adj-mas3}.}
  \label{fig:wall1}
\end{figure}

\begin{figure}[h]
  \centering\scalebox{.8}{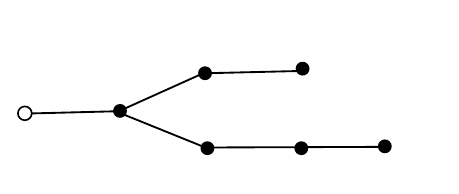}
  \caption{The graph of a broken map $u'$ whose gluing is in the class
    $\delta_{E_5}$, but which does not have any component that is
    homologous to a multiple cover.}
  \label{fig:wall2}
\end{figure}



\printindex
\end{document}